\documentclass{article}
\usepackage{amsmath}
\usepackage{amssymb}
\usepackage{amsthm}
\usepackage{graphicx}
\usepackage{cite}
\usepackage[arrow,curve,matrix,arc]{xy}
\usepackage{makeidx}
\usepackage[refpage,intoc]{nomencl}
\makenomenclature
\makeindex
\DeclareFontFamily{U}{rsfs}{} \DeclareFontShape{U}{rsfs}{n}{it}{<->
rsfs10}{} \DeclareSymbolFont{mscr}{U}{rsfs}{n}{it}
\DeclareSymbolFontAlphabet{\scr}{mscr}
\def\mathscr{\scr}
\begin{document}
\theoremstyle{plain}
\newtheorem{thm}{Theorem}[section]
\newtheorem{lem}[thm]{Lemma}
\newtheorem{prop}[thm]{Proposition}
\newtheorem{cor}[thm]{Corollary}
\newtheorem{quest}[thm]{Question}
\newtheorem{conj}[thm]{Conjecture}
\theoremstyle{definition}
\newtheorem{dfn}[thm]{Definition}
\newtheorem{ass}[thm]{Assumption}
\newtheorem{ex}[thm]{Example}
\newtheorem{rem}[thm]{Remark}
\numberwithin{equation}{section}
\renewcommand{\nomname}{Glossary of Notation}
\renewcommand{\pagedeclaration}[1]{, #1}
\setlength{\nomitemsep}{0pt}
\def\e#1\e{\begin{equation}#1\end{equation}}
\def\ea#1\ea{\begin{align}#1\end{align}}
\def\eq#1{{\rm(\ref{#1})}}
\newcommand{\Ker}{\mathop{\rm Ker}}
\newcommand{\End}{\mathop{\rm End}\nolimits}
\newcommand{\Coker}{\mathop{\rm Coker}}
\renewcommand{\Im}{\mathop{\rm Im}}
\newcommand{\rank}{\mathop{\rm rank}}
\newcommand{\hd}{\mathop{\rm hd}}
\newcommand{\Hilb}{\mathop{\rm Hilb}\nolimits}
\newcommand{\coh}{\mathop{\rm coh}\nolimits}
\newcommand{\cs}{{\rm cs}}
\newcommand{\Con}{{\sst{\rm Con}}}
\newcommand{\id}{\mathop{\rm id}\nolimits}
\newcommand{\bdim}{{\mathbin{\bf dim}\kern.1em}}
\newcommand{\Crit}{\mathop{\rm Crit}}
\newcommand{\Stab}{\mathop{\rm Stab}\nolimits}
\newcommand{\stab}{\mathop{\mathfrak{stab}}\nolimits}
\newcommand{\Quot}{\mathop{\rm Quot}}
\newcommand{\CF}{\mathop{\rm CF}\nolimits}
\newcommand{\GL}{\mathop{\rm GL}}
\newcommand{\SO}{\mathop{\rm SO}}
\newcommand{\SU}{\mathop{\rm SU}}
\newcommand{\SL}{\mathop{\rm SL}}
\newcommand{\U}{\mathop{\ts\rm U}}
\newcommand{\Gr}{\mathop{\rm Gr}}
\newcommand{\Mo}{\mathop{\text{M\"o}}}
\newcommand{\Ch}{\mathop{\rm Ch}}
\newcommand{\Tr}{\mathop{\rm Tr}}
\newcommand{\Eu}{\mathop{\rm Eu}\nolimits}
\newcommand{\ch}{\mathop{\rm ch}\nolimits}
\newcommand{\num}{{\rm num}}
\newcommand{\vir}{{\rm vir}}
\newcommand{\stp}{{\rm stp}}
\newcommand{\fr}{{\rm fr}}
\newcommand{\stf}{{\rm stf}}
\newcommand{\rsi}{{\rm si}}
\newcommand{\na}{{\rm na}}
\newcommand{\stk}{{\rm stk}}
\newcommand{\rss}{{\rm ss}}
\newcommand{\st}{{\rm st}}
\newcommand{\all}{{\rm all}}
\newcommand{\rk}{{\rm rk}}
\newcommand{\vi}{{\rm vi}}
\newcommand{\an}{{\rm an}}
\newcommand{\cha}{\mathop{\rm char}}
\newcommand{\ind}{{\rm \kern.05em ind}}
\newcommand{\Exp}{\mathop{\rm Exp}}
\newcommand{\Aut}{\mathop{\rm Aut}}
\newcommand{\Per}{\mathop{\rm Per}}
\newcommand{\cone}{\mathop{\rm cone}\nolimits}
\newcommand{\At}{\mathop{\rm At}\nolimits}
\newcommand{\Hom}{\mathop{\rm Hom}\nolimits}
\newcommand{\Iso}{\mathop{\rm Iso}\nolimits}
\newcommand{\Ext}{\mathop{\rm Ext}\nolimits}
\newcommand{\Spec}{\mathop{\rm Spec}\nolimits}
\newcommand{\LCF}{\mathop{\rm LCF}\nolimits}
\newcommand{\SF}{\mathop{\rm SF}\nolimits}
\newcommand{\SFa}{\mathop{\rm SF}\nolimits_{\rm al}}
\newcommand{\SFai}{\mathop{\rm SF}\nolimits_{\rm al}^{\rm ind}}
\newcommand{\uSF}{\mathop{\smash{\underline{\rm SF\!}\,}}\nolimits}
\newcommand{\uSFi}{\mathop{\smash{\underline{\rm SF\!}\,}}\nolimits^{\rm ind}}
\newcommand{\oSF}{\mathop{\bar{\rm SF}}\nolimits}
\newcommand{\oSFa}{\mathop{\bar{\rm SF}}\nolimits_{\rm al}}
\newcommand{\oSFai}{{\ts\bar{\rm SF}{}_{\rm al}^{\rm ind}}}
\newcommand{\uoSF}{\mathop{\bar{\underline{\rm SF\!}\,}}\nolimits}
\newcommand{\uoSFa}{\mathop{\bar{\underline{\rm SF\!}\,}}\nolimits_{\rm al}}
\newcommand{\uoSFi}{\mathop{\bar{\underline{\rm SF\!}\,}}\nolimits_{\rm al}^{\rm
ind}}
\newcommand{\LSF}{\mathop{\rm LSF}\nolimits}
\newcommand{\LSFa}{\mathop{\rm LSF}\nolimits_{\rm al}}
\newcommand{\uLSF}{\mathop{\smash{\underline{\rm LSF\!}\,}}\nolimits}
\newcommand{\dLSF}{{\dot{\rm LSF}}\kern-.1em\mathop{}\nolimits}
\newcommand{\doLSF}{{\dot{\bar{\rm LSF}}}\kern-.1em\mathop{}\nolimits}
\newcommand{\duoLSF}{{\dot{\bar{\underline{\rm LSF\!}\,}}}\kern-.1em\mathop{}
\nolimits}
\newcommand{\ouLSF}{{\bar{\underline{\rm LSF\!}\,}}\kern-.1em\mathop{}
\nolimits}
\newcommand{\dLSFi}{{\dot{\rm LSF}}\kern-.1em\mathop{}\nolimits^{\rm ind}}
\newcommand{\dLSFa}{{\dot{\rm LSF}}\kern-.1em\mathop{}\nolimits_{\rm al}}
\newcommand{\doLSFa}{{\dot{\bar{\rm LSF}}}\kern-.1em\mathop{}\nolimits_{\rm al}}
\newcommand{\dLSFai}{{\dot{\rm LSF}}\kern-.1em\mathop{}\nolimits^{\rm ind}_{\rm
al}}
\newcommand{\duLSF}{{\dot{\underline{\rm LSF\!}\,}}\kern-.1em\mathop{}\nolimits}
\newcommand{\duLSFi}{{\dot{\underline{\rm LSF\!}\,}}\kern-.1em\mathop{}
\nolimits^{\rm ind}}
\newcommand{\oLSF}{\mathop{\bar{\rm LSF}}\nolimits}
\newcommand{\oLSFa}{\mathop{\bar{\rm LSF}}\nolimits_{\rm al}}
\newcommand{\oLSFai}{\mathop{\bar{\rm LSF}}\nolimits_{\rm al}^{\rm ind}}
\newcommand{\uoLSF}{\mathop{\bar{\underline{\rm LSF\!}\,}}\nolimits}
\newcommand{\uoLSFa}{\mathop{\bar{\underline{\rm LSF\!}\,}}\nolimits_{\rm al}}
\newcommand{\uoLSFi}{\mathop{\bar{\underline{\rm LSF\!}\,}}\nolimits_{\rm
al}^{\rm ind}}
\newcommand{\modCQ}{\text{\rm mod-$\C Q$}}
\newcommand{\modCQI}{\text{\rm mod-$\C Q/I$}}
\newcommand{\modKQ}{\text{\rm mod-$\K Q$}}
\newcommand{\modKQI}{\text{\rm mod-$\K Q/I$}}
\newcommand{\projKQ}{\text{\rm proj-$\K Q$}}
\newcommand{\projKQI}{\text{\rm proj-$\K Q/I$}}
\newcommand{\Obj}{\mathop{\rm Obj\kern .1em}\nolimits}
\newcommand{\fObj}{\mathop{\mathfrak{Obj}\kern .05em}\nolimits}
\newcommand{\Sch}{\mathop{\rm Sch}\nolimits}
\newcommand{\bs}{\boldsymbol}
\renewcommand{\ge}{\geqslant}
\renewcommand{\le}{\leqslant\nobreak}
\newcommand{\sA}{{\mathbin{\mathscr A}}}
\newcommand{\sAs}{{\mathscr A}\kern-1.5pt{}_{\rm si}}
\newcommand{\sB}{{\mathbin{\mathscr B}}}
\newcommand{\sE}{{\mathbin{\mathscr E}}}
\newcommand{\sG}{{\mathbin{\mathscr G}}}
\newcommand{\sS}{{\mathbin{\mathscr S}\kern-3pt{}}}
\newcommand{\bA}{{\mathbin{\mathbb A}}}
\newcommand{\bG}{{\mathbin{\mathbb G}}}
\newcommand{\A}{{\mathbin{\mathcal A}}}
\newcommand{\B}{{\mathbin{\mathcal B}}}
\newcommand{\fE}{{\mathbin{\mathfrak E}}}
\newcommand{\fF}{{\mathbin{\mathfrak F}}}
\newcommand{\fG}{{\mathbin{\mathfrak G}}}
\newcommand{\fH}{{\mathbin{\mathfrak H}}}
\newcommand{\R}{{\mathbin{\mathbb R}}}
\newcommand{\Z}{{\mathbin{\mathbb Z}}}
\newcommand{\cQ}{{\mathbin{\mathcal Q}}}
\newcommand{\Q}{{\mathbin{\mathbb Q}}}
\newcommand{\fN}{{\mathbin{\mathfrak N}}}
\newcommand{\N}{{\mathbin{\mathbb N}}}
\newcommand{\fU}{{\mathbin{\mathfrak{U}\kern .05em}}}
\newcommand{\fV}{{\mathbin{\mathfrak V}}}
\newcommand{\C}{{\mathbin{\mathbb C}}}
\newcommand{\K}{{\mathbin{\mathbb K}}}
\newcommand{\CP}{{\mathbin{\mathbb{CP}}}}
\newcommand{\RP}{{\mathbin{\mathbb{RP}}}}
\newcommand{\M}{{\mathbin{\mathcal M}}}
\newcommand{\fM}{{\mathbin{\mathfrak M}}}
\newcommand{\g}{{\mathbin{\mathfrak g}}}
\newcommand{\m}{{\mathbin{\mathfrak m}}}
\newcommand{\oM}{{\mathbin{\smash{\,\,\overline{\!\!\mathcal M\!}\,}}}}
\newcommand{\cO}{{\mathcal O}}
\newcommand{\fR}{{\mathbin{\mathfrak R}}}
\newcommand{\fS}{{\mathbin{\mathfrak S}}}
\newcommand{\fExact}{\mathop{\mathfrak{Exact}\kern .05em}\nolimits}
\newcommand{\Vect}{{\mathbin{\mathcal{V}ect}}}
\newcommand{\fVect}{{\mathbin{\mathfrak{Vect}}}}
\newcommand{\Hol}{\mathbin{\mathcal{H}ol}}
\newcommand{\fHol}{{\mathbin{\mathfrak{Hol}}}}
\newcommand{\fW}{{\mathbin{\mathfrak W}}}
\newcommand{\SHom}{\mathbin{\mathcal{H}om}}
\newcommand{\SExt}{\mathcal{E}xt}
\newcommand{\dtensor}{\stackrel{L}{\otimes}}
\newcommand{\tr}{\mathop{\rm tr}\nolimits}
\newcommand{\Exal}{\mathop{\rm Exal}\nolimits}
\newcommand{\al}{\alpha}
\newcommand{\be}{\beta}
\newcommand{\ga}{\gamma}
\newcommand{\de}{\delta}
\newcommand{\bde}{\bar\delta}
\newcommand{\bdss}{\bar\delta_{\rm ss}}
\newcommand{\io}{\iota}
\newcommand{\ep}{\epsilon}
\newcommand{\bep}{\bar\epsilon}
\newcommand{\la}{\lambda}
\newcommand{\ka}{\kappa}
\newcommand{\ze}{\zeta}
\newcommand{\up}{\upsilon}
\newcommand{\vp}{\varphi}
\newcommand{\si}{\sigma}
\newcommand{\om}{\omega}
\newcommand{\De}{\Delta}
\newcommand{\La}{\Lambda}
\newcommand{\Om}{\Omega}
\newcommand{\Up}{\Upsilon}
\newcommand{\Ga}{\Gamma}
\newcommand{\Si}{\Sigma}
\newcommand{\Th}{\Theta}
\newcommand{\pd}{\partial}
\newcommand{\db}{{\bar\partial}}
\newcommand{\ts}{\textstyle}
\newcommand{\sst}{\scriptscriptstyle}
\newcommand{\w}{\wedge}
\newcommand{\sm}{\setminus}
\newcommand{\bu}{\bullet}
\newcommand{\op}{\oplus}
\newcommand{\ot}{\otimes}
\newcommand{\ov}{\overline}
\newcommand{\bigop}{\bigoplus}
\newcommand{\bigot}{\bigotimes}
\newcommand{\iy}{\infty}
\newcommand{\es}{\emptyset}
\newcommand{\ra}{\rightarrow}
\newcommand{\Ra}{\Rightarrow}
\newcommand{\ab}{\allowbreak}
\newcommand{\longra}{\longrightarrow}
\newcommand{\hookra}{\hookrightarrow}
\newcommand{\lt}{\ltimes}
\newcommand{\el}{{\mathbin{\ell\kern .08em}}}
\newcommand{\ci}{\circ}
\newcommand{\ti}{\tilde}
\newcommand{\gr}{\grave}
\newcommand{\rd}{{\rm d}}
\newcommand{\ha}{{\ts\frac{1}{2}}}
\def\md#1{\vert #1 \vert}
\def\nm#1{\Vert #1 \Vert}
\def\bmd#1{\big\vert #1 \big\vert}
\def\bnm#1{\big\Vert #1 \big\Vert}
\def\ms#1{\vert #1 \vert^2}
\title{A theory of generalized Donaldson--Thomas invariants}
\author{Dominic Joyce\footnote{Address: The Mathematical Institute,
24-29 St. Giles, Oxford, OX1 3LB, U.K., E-mail: \tt
joyce@maths.ox.ac.uk}{}\ \ and Yinan Song\footnote{E-mail: \tt
yinansong@gmail.com}}
\date{}
\maketitle

\vskip -.7cm

\vskip -.7cm

\begin{abstract}
{\it Donaldson--Thomas invariants\/} $DT^\al(\tau)$ are integers
which `count' $\tau$-stable coherent sheaves with Chern character
$\al$ on a Calabi--Yau 3-fold $X$, where $\tau$ denotes Gieseker
stability for some ample line bundle on $X$. They are unchanged
under deformations of $X$. The conventional definition works only
for classes $\al$ containing no strictly $\tau$-semistable sheaves.
Behrend showed that $DT^\al(\tau)$ can be written as a weighted
Euler characteristic $\chi\bigl(\M_\st^\al(\tau),
\nu_{\M_\st^\al(\tau)}\bigr)$ of the stable moduli scheme
$\M_\st^\al(\tau)$ by a constructible function
$\nu_{\M_\st^\al(\tau)}$ we call the `Behrend function'.

This book studies {\it generalized Donaldson--Thomas invariants\/}
$\bar{DT}{}^\al(\tau)$. They are rational numbers which `count' both
$\tau$-stable and $\tau$-semistable coherent sheaves with Chern
character $\al$ on $X$; strictly $\tau$-semistable sheaves must be
counted with complicated rational weights. The
$\bar{DT}{}^\al(\tau)$ are defined for all classes $\al$, and are
equal to $DT^\al(\tau)$ when it is defined. They are unchanged under
deformations of $X$, and transform by a wall-crossing formula under
change of stability condition~$\tau$.

To prove all this we study the local structure of the moduli stack
$\fM$ of coherent sheaves on $X$. We show that an atlas for $\fM$
may be written locally as $\Crit(f)$ for $f:U\ra\C$ holomorphic and
$U$ smooth, and use this to deduce identities on the Behrend
function $\nu_\fM$. We compute our invariants $\bar{DT}{}^\al(\tau)$
in examples, and make a conjecture about their integrality
properties. We also extend the theory to abelian categories
$\modCQI$ of representations of a quiver $Q$ with relations $I$
coming from a superpotential $W$ on $Q$, and connect our ideas with
Szendr\H oi's noncommutative Donaldson--Thomas invariants, and work
by Reineke and others on invariants counting quiver representations.
Our book is closely related to Kontsevich and Soibelman's
independent paper~\cite{KoSo1}.
\end{abstract}

\vskip -.4cm

\setcounter{tocdepth}{2}
\tableofcontents
\nomenclature[coh(X)]{$\coh(X)$}{abelian category of coherent
sheaves on a scheme $X$\nomnorefpage}
\nomenclature[Db(coh(X))]{$D^b(\coh(X))$ or $D^b(X)$}{bounded
derived category of coherent sheaves on $X$\nomnorefpage}
\nomenclature[E,F,G]{$\fE,\fF,\fG,\fH,\ldots$}{Artin
$\K$-stacks\nomnorefpage}
\nomenclature[OX]{$\cO_X$}{structure sheaf of a scheme
$X$\nomnorefpage}
\nomenclature[K]{$\K$}{base field, usually algebraically
closed\nomnorefpage}
\nomenclature[Crit(f)]{$\Crit(f)$}{critical locus of a holomorphic
function $f$, as a complex analytic space\nomnorefpage}
\nomenclature[Hpq(X)]{$H^{p,q}(X)$}{Dolbeault cohomology groups of a
K\"ahler manifold $X$\nomnorefpage}
\nomenclature[P(V)]{${\mathbb P}(V)$}{projective space of a vector
space $V$\nomnorefpage}
\nomenclature[L\pi]{$L\pi^*$}{left derived pullback functor $D(Y)\ra
D(X)$ of a scheme morphism $\pi:X\ra Y$. See Huybrechts \cite[\S
3.3]{Huyb}.\nomnorefpage}
\nomenclature[L\pi b]{$\pi^{"!}$}{functor $D(Y)\ra D(X)$ mapping
${\cal E}\mapsto Lf^*({\cal E})\ot\om_\pi[\dim\pi]$ for a smooth
scheme morphism $\pi:X\ra Y$. See Huybrechts \cite[\S
3.3]{Huyb}.\nomnorefpage}
\nomenclature[R\pi]{$R\pi_*$}{right derived pushforward functor
$D(X)\ra D(Y)$ of a scheme morphism $\pi:X\ra Y$. See Huybrechts
\cite[\S 3.3]{Huyb}.\nomnorefpage}
\nomenclature[RHom]{$R\SHom$}{derived sheaf Hom functor. If ${\cal
E},{\cal F}\in D(X)$ then $R\SHom({\cal E},{\cal F})\in D(X)$. See
Huybrechts \cite[\S 3.3]{Huyb}.\nomnorefpage}
\index{sheaf!perverse|see{perverse sheaf}}
\index{sheaf!coherent|see{coherent sheaf}}
\index{coherent sheaf!derived category of|see{derived category}}
\index{triangulated category|seealso{derived category}}
\section{Introduction}
\label{dt1}

Let $X$ be a Calabi--Yau 3-fold\index{Calabi--Yau 3-fold} over the
complex numbers $\C$, and $\cO_X(1)$ a very ample line bundle on
$X$. Our definition of Calabi--Yau 3-fold requires $X$ to be
projective, with $H^1(\cO_X)=0$. Write $\coh(X)$ for the abelian
category of coherent sheaves on $X$, and $K^\num(\coh(X))$ for the
numerical Grothendieck group\index{Grothendieck group!numerical} of
$\coh(X)$. We use $\tau$ to denote Gieseker stability of coherent
sheaves with respect to $\cO_X(1)$. If $E$ is a coherent sheaf on
$X$ then the class $[E]\in K^\num(\coh(X))$ is in effect the Chern
character $\ch(E)$ of $E$ in~$H^{\rm even}(X;\Q)$.

For $\al\in K^\num(\coh(X))$ we form the coarse moduli schemes
$\M_\rss^\al(\tau),\M_\st^\al(\tau)$ of $\tau$-(semi)stable sheaves
$E$ with $[E]\!=\!\al$. Then $\M_\rss^\al(\tau)$ is a projective
$\C$-scheme whose points correspond to
S-equivalence\index{S-equivalence} classes of $\tau$-semistable
sheaves, and $\M_\st^\al(\tau)$ is an open subscheme of
$\M_\rss^\al(\tau)$ whose points correspond to isomorphism classes
of $\tau$-stable sheaves.

For Chern characters $\al$ with
$\M_\rss^\al(\tau)=\M_\st^\al(\tau)$, following Donaldson and Thomas
\cite[\S 3]{DoTh}, Thomas \cite{Thom} constructed a symmetric
obstruction theory on $\M_\st^\al(\tau)$ and defined the {\it
Donaldson--Thomas invariant\/} to be the virtual
class\index{Donaldson--Thomas invariants!original $DT^\al(\tau)$}
\e
DT^\al(\tau)=\ts\int_{[\M_\st^\al(\tau)]^\vir}1\in\Z,
\label{dt1eq1}
\e
an integer which `counts' $\tau$-semistable sheaves in class $\al$.
Thomas' main result \cite[\S 3]{Thom} is that $DT^\al(\tau)$ is
unchanged under deformations of the underlying Calabi--Yau 3-fold
$X$. Later, Behrend \cite{Behr} showed that Donaldson--Thomas
invariants can be written as a weighted Euler characteristic
\e
DT^\al(\tau)=\chi\bigl(\M_\st^\al(\tau),\nu_{\M_\st^\al(\tau)}\bigr),
\label{dt1eq2}
\e
where $\nu_{\M_\st^\al(\tau)}$ is the {\it Behrend
function},\index{Behrend function} a constructible function on
$\M_\st^\al(\tau)$ depending only on $\M_\st^\al(\tau)$ as a
$\C$-scheme. (Here, and throughout, Euler characteristics are taken
with respect to cohomology with compact support.)

Conventional Donaldson--Thomas invariants $DT^\al(\tau)$ are only
defined for classes $\al$ with $\M_\rss^\al(\tau)=\M_\st^\al(\tau)$,
that is, when there are no strictly $\tau$-semistable sheaves. Also,
although $DT^\al(\tau)$ depends on the stability condition $\tau$,
that is, on the choice of very ample line bundle $\cO_X(1)$ on $X$,
this dependence was not understood until now. The main goal of this
book is to address these two issues.

For a Calabi--Yau 3-fold $X$ over $\C$ we will define {\it
generalized Donaldson--Thomas invariants\/}\index{Donaldson--Thomas
invariants!generalized $\bar{DT}{}^\al(\tau)$}
$\bar{DT}{}^\al(\tau)\in\Q$ for all $\al\in K^\num(\coh(X))$, which
`count' $\tau$-semistable sheaves in class $\al$. These have the
following important properties:
\begin{itemize}
\setlength{\itemsep}{0pt}
\setlength{\parsep}{0pt}
\item $\bar{DT}{}^\al(\tau)\in\Q$ is unchanged by deformations of the
Calabi--Yau 3-fold~$X$.
\item If $\M_\rss^\al(\tau)=\M_\st^\al(\tau)$ then $\bar{DT}{}^\al(\tau)$
lies in $\Z$ and equals the conventional Donaldson--Thomas invariant
$DT^\al(\tau)$ defined by Thomas~\cite{Thom}.
\item If $\M_\rss^\al(\tau)\ne\M_\st^\al(\tau)$ then conventional
Donaldson--Thomas invariants $DT^\al(\tau)$ are not defined for
class $\al$. Our generalized invariant $\bar{DT}{}^\al(\tau)$
may lie in $\Q$ because strictly semistable sheaves $E$ make
(complicated) $\Q$-valued contributions to
$\bar{DT}{}^\al(\tau)$. We can write the $\bar{DT}{}^\al(\tau)$
in terms of other invariants $\hat{DT}{}^\al(\tau)$ which, in
the case of `generic' $\tau$, we conjecture to be
integer-valued.
\item If $\tau,\ti\tau$ are two stability conditions on $\coh(X)$,
there is an explicit change of stability condition formula giving
$\bar{DT}{}^\al(\ti\tau)$ in terms of the $\bar{DT}{}^\be(\tau)$.
\end{itemize}
These invariants are a continuation of the first author's
programme~\cite{Joyc1,Joyc2,Joyc3,Joyc4,Joyc5,Joyc6,Joyc7}.

We begin with three sections of background. Sections
\ref{dt2}--\ref{dt3} explain material on constructible functions,
stack functions, Ringel--Hall algebras, counting invariants for
Calabi--Yau 3-folds, and wall-crossing, from the first author's
series \cite{Joyc1,Joyc2,Joyc3,Joyc4,Joyc5,Joyc6, Joyc7}. Section
\ref{dt4} explains Behrend's approach \cite{Behr} to
Donaldson--Thomas invariants as Euler characteristics of moduli
schemes weighted by the Behrend function, as in \eq{dt1eq2}. We
include some new material here, and explain the connection between
Behrend functions and the theory of perverse sheaves and vanishing
cycles. Our main results are given in \S\ref{dt5}, including the
definition of generalized Donaldson--Thomas invariants
$\bar{DT}{}^\al(\tau) \in\Q$, their deformation-invariance, and
wall-crossing formulae under change of stability condition~$\tau$.

Sections \ref{dt6} and \ref{dt7} give many examples, applications,
and generalizations of the theory, with \S\ref{dt6} focussing on
coherent sheaves on (possibly noncompact) Calabi--Yau 3-folds, and
\S\ref{dt7} on representations of quivers with superpotentials, in
connection with work by many authors on 3-Calabi--Yau categories,
noncommutative Donaldson--Thomas invariants, and so on. One new
piece of theory is that in \S\ref{dt62}, motivated by ideas in
Kontsevich and Soibelman \cite[\S 2.5 \& \S 7.1]{KoSo1}, we define
{\it BPS invariants\/} $\hat{DT}{}^\al(\tau)$ by the formula\index{BPS
invariants}\index{Donaldson--Thomas invariants!BPS invariants
$\hat{DT}{}^\al(\tau)$}
\begin{equation*}
\bar{DT}{}^\al(\tau)=\sum_{m\ge 1,\; m\mid\al}\frac{1}{m^2}\,
\hat{DT}{}^{\al/m}(\tau).
\end{equation*}
These are supposed to count `BPS states' in some String
Theory\index{String Theory} sense, and we conjecture that for `generic'
stability conditions $\tau$ we have $\hat{DT}{}^\al(\tau)\in\Z$ for
all $\al$. An analogue of this conjecture for invariants
$\hat{DT}{}^{\bs d}_Q(\mu)$ counting representations of a quiver $Q$
without relations is proved in~\S\ref{dt76}.

Sections \ref{dt8}--\ref{dt13} give the proofs of the main results
stated in \S\ref{dt5}, and we imagine many readers will not need to
look at these. In the rest of this introduction we survey
\S\ref{dt2}--\S\ref{dt7}. Section \ref{dt1.1} very briefly sketches
the set-up of \cite{Joyc1,Joyc2,Joyc3,Joyc4,Joyc5,Joyc6, Joyc7},
which will be explained in \S\ref{dt2}--\S\ref{dt3}. Section
\ref{dt1.2} discusses {\it Behrend functions\/} from \S\ref{dt4},
\S\ref{dt1.3} outlines the main results in \S\ref{dt5}, and
\S\ref{dt1.4}--\S\ref{dt1.5} summarize the applications and
generalizations in \S\ref{dt6}--\S\ref{dt7}. Finally, \S\ref{dt1.6}
explains the relations between our work and the recent paper of
Kontsevich and Soibelman \cite{KoSo1}, which is summarized in
\cite{KoSo2}. This book is surveyed in~\cite{Joyc8}.

In \S\ref{dt4}--\S\ref{dt7} we give problems for further research,
as Questions or Conjectures.

\subsection[Brief sketch of background from
$\text{\cite{Joyc1,Joyc2,Joyc3,Joyc4,Joyc5,Joyc6,Joyc7}}$]{Brief
sketch of background from \cite{Joyc1,Joyc2,Joyc3,Joyc4,Joyc5,Joyc6,Joyc7}}
\label{dt1.1}

We recall a few important ideas from \cite{Joyc1,Joyc2,Joyc3,Joyc4,
Joyc5,Joyc6,Joyc7}, which will be explained at greater length in
\S\ref{dt2}--\S\ref{dt3}. We work not with coarse moduli
schemes,\index{coarse moduli scheme}\index{moduli scheme!coarse} as in
\cite{Thom}, but with {\it Artin stacks}. Let $X$ be a Calabi--Yau
3-fold over $\C$, and write $\fM$ for the moduli stack of all
coherent sheaves $E$ on $X$. It is an Artin $\C$-stack.

The ring of `stack functions' $\SF(\fM)$ in \cite{Joyc2} is
basically the Grothendieck group $K_0(\mathop{\rm Sta}_\fM)$ of the
2-category $\mathop{\rm Sta}_\fM$ of stacks over $\fM$. That is,
$\SF(\fM)$ is generated by isomorphism classes $[(\fR,\rho)]$ of
representable 1-morphisms $\rho:\fR\ra\fM$ for $\fR$ a finite type
Artin $\C$-stack, with the relation
\begin{equation*}
[(\fR,\rho)]=[(\fS,\rho\vert_\fS)]+[(\fR\sm\fS,\rho\vert_{\fR\sm\fS})]
\end{equation*}
when $\fS$ is a closed $\C$-substack of $\fR$. But there is more to
stack functions than this. In \cite{Joyc2} we study different kinds
of stack function spaces with other choices of generators and
relations, and operations on these spaces. These include projections
$\Pi^\vi_n:\SF(\fM)\ra\SF(\fM)$ to stack functions of `virtual rank
$n$', which act on $[(\fR,\rho)]$ by modifying $\fR$ depending on
its stabilizer groups.

In \cite[\S 5.2]{Joyc4} we define a Ringel--Hall type
algebra\index{Ringel--Hall algebra} $\SFa(\fM)$ of stack
functions\index{stack function!with algebra stabilizers} `with algebra
stabilizers' on $\fM$, with an associative, non-commutative
multiplication $*$. In \cite[\S 5.2]{Joyc4} we define a Lie
subalgebra $\SFai(\fM)$ of stack functions `supported on virtual
indecomposables'.\index{virtual indecomposable}\index{stack
function!supported on virtual indecomposables} In \cite[\S
6.5]{Joyc4} we define an explicit Lie algebra $L(X)$ to be the
$\Q$-vector space with basis of symbols $\la^\al$ for $\al\in
K^\num(\coh(X))$, with Lie bracket
\e
[\la^\al,\la^\be]=\bar\chi(\al,\be)\la^{\al+\be},
\label{dt1eq3}
\e
for $\al,\be\in K^\num(\coh(X))$, where $\bar\chi(\,,\,)$ is the
Euler form\index{Euler form} on $K^\num(\coh(X))$. As $X$ is a
Calabi--Yau 3-fold, $\bar\chi$ is antisymmetric, so \eq{dt1eq3}
satisfies the Jacobi identity and makes $L(X)$ into an
infinite-dimensional Lie algebra over~$\Q$.

Then in \cite[\S 6.6]{Joyc4} we define a {\it Lie algebra
morphism\/} $\Psi:\SFai(\fM)\ra L(X)$, as in \S\ref{dt34} below.
Roughly speaking this is of the form
\e
\Psi(f)=\ts\sum_{\al\in K^\num(\coh(X))}\chi^{\rm stk}
\bigl(f\vert_{\fM^\al}\bigr)\la^{\al},
\label{dt1eq4}
\e
where $f=\sum_{i=1}^mc_i[(\fR_i,\rho_i)]$ is a stack function on
$M$, and $\fM^\al$ is the substack in $\fM$ of sheaves $E$ with
class $\al$, and $\chi^{\rm stk}$ is a kind of stack-theoretic Euler
characteristic. But in fact the definition of $\Psi$, and the proof
that $\Psi$ is a Lie algebra morphism, are highly nontrivial, and
use many ideas from \cite{Joyc1,Joyc2,Joyc4}, including those of
`virtual rank' and `virtual indecomposable'.\index{virtual
indecomposable} The problem is that the obvious definition of
$\chi^{\rm stk}$ usually involves dividing by zero, so defining
\eq{dt1eq4} in a way that makes sense is quite subtle. The proof
that $\Psi$ is a Lie algebra morphism uses Serre duality and the
assumption that $X$ is a Calabi--Yau 3-fold.

Now let $\tau$ be a stability condition on $\coh(X)$, such as
Gieseker stability. Then we have open, finite type substacks
$\fM_\rss^\al(\tau),\fM_\st^\al(\tau)$ in $\fM$ of
$\tau$-(semi)stable sheaves $E$ in class $\al$, for all $\al\in
K^\num(\coh(X))$. Write $\bar\de_\rss^\al(\tau)$ for the
characteristic function of $\fM_\rss^\al(\tau)$, in the sense of
stack functions \cite{Joyc2}. Then
$\bar\de_\rss^\al(\tau)\in\SFa(\fM)$. In \cite[\S 8]{Joyc5}, we
define elements $\bar\ep^\al(\tau)$ in $\SFa(\fM)$ by
\e
\bar\ep^\al(\tau)= \!\!\!\!\!\!\!
\sum_{\begin{subarray}{l}n\ge 1,\;\al_1,\ldots,\al_n\in
K^\num(\coh(X)):\\
\al_1+\cdots+\al_n=\al,\; \tau(\al_i)=\tau(\al),\text{ all
$i$}\end{subarray}} \!\!\!
\frac{(-1)^{n-1}}{n}\,\,\bar\de_\rss^{\al_1}(\tau)*\bar
\de_\rss^{\al_2}(\tau)* \cdots*\bar\de_\rss^{\al_n}(\tau),
\label{dt1eq5}
\e
where $*$ is the Ringel--Hall multiplication in $\SFa(\fM)$. Then
\cite[Th.~8.7]{Joyc5} shows that $\bar\ep^\al(\tau)$ lies in the Lie
subalgebra $\SFai(\fM)$, a nontrivial result.

Thus we can apply the Lie algebra morphism $\Psi$ to
$\bar\ep^\al(\tau)$. In \cite[\S 6.6]{Joyc6} we define invariants
$J^\al(\tau)\in\Q$ for all $\al\in K^\num(\coh(X))$ by
\e
\Psi\bigl(\bar\ep^\al(\tau)\bigr)=J^\al(\tau)\la^\al.
\label{dt1eq6}
\e
These $J^\al(\tau)$ are rational numbers `counting'
$\tau$-semistable sheaves $E$ in class $\al$. When
$\M_\rss^\al(\tau)=\M_\st^\al(\tau)$ we have
$J^\al(\tau)=\chi(\M_\st^\al(\tau))$, that is, $J^\al(\tau)$ is the
na\"\i ve Euler characteristic of the moduli space
$\M_\st^\al(\tau)$. This is {\it not\/} weighted by the Behrend
function $\nu_{\M_\st^\al(\tau)}$, and so in general does not
coincide with the Donaldson--Thomas invariant $DT^\al(\tau)$
in~\eq{dt1eq3}.

As the $J^\al(\tau)$ do not include Behrend functions, they do not
count semistable sheaves with multiplicity, and so they will not in
general be unchanged under deformations of the underlying
Calabi--Yau 3-fold, as Donaldson--Thomas invariants are. However,
the $J^\al(\tau)$ do have very good properties under change of
stability condition. In \cite{Joyc6} we show that if $\tau,\ti\tau$
are two stability conditions on $\coh(X)$, then we can write
$\bar\ep^\al(\ti\tau)$ in terms of a (complicated) explicit formula
involving the $\bar\ep^\be(\tau)$ for $\be\in K^\num(\coh(X))$ and
the Lie bracket in~$\SFai(\fM)$.

Applying the Lie algebra morphism $\Psi$ shows that
$J^\al(\ti\tau)\la^\al$ may be written in terms of the
$J^\be(\tau)\la^\be$ and the Lie bracket in $L(X)$, and hence
\cite[Th.~6.28]{Joyc6} we get an explicit transformation law for the
$J^\al(\tau)$ under change of stability condition. In \cite{Joyc7}
we show how to encode invariants $J^\al(\tau)$ satisfying a
transformation law in generating functions on a complex manifold of
stability conditions, which are both holomorphic and continuous,
despite the discontinuous wall-crossing behaviour of the
$J^\al(\tau)$. This concludes our sketch
of~\cite{Joyc1,Joyc2,Joyc3,Joyc4,Joyc5,Joyc6,Joyc7}.

\subsection[Behrend functions of schemes and stacks, from
$\text{\S\ref{dt4}}$]{Behrend functions of schemes and stacks, from
\S\ref{dt4}}
\label{dt1.2}\index{Behrend function|(}

Let $X$ be a $\C$-scheme or Artin $\C$-stack, locally of finite
type, and $X(\C)$ the set of $\C$-points of $X$. The {\it Behrend
function\/} $\nu_X:X(\C)\ra\Z$ is a $\Z$-valued locally
constructible function on $X$, in the sense of \cite{Joyc1}. They
were introduced by Behrend \cite{Behr} for finite type $\C$-schemes
$X$; the generalization to Artin stacks in \S\ref{dt41} is new but
straightforward. Behrend functions are also defined for complex
analytic spaces $X_\an$, and the Behrend function of a $\C$-scheme
$X$ coincides with that of the underlying complex analytic
space~$X_\an$.

A good way to think of Behrend functions is as {\it multiplicity
functions}. If $X$ is a finite type $\C$-scheme then the Euler
characteristic $\chi(X)$ `counts' points without multiplicity, so
that each point of $X(\C)$ contributes 1 to $\chi(X)$. If $X^{\rm
red}$ is the underlying reduced $\C$-scheme then $X^{\rm
red}(\C)=X(\C)$, so $\chi(X^{\rm red})=\chi(X)$, and $\chi(X)$ does
not see non-reduced behaviour in $X$. However, the weighted Euler
characteristic $\chi(X,\nu_X)$ `counts' each $x\in X(\C)$ weighted
by its multiplicity $\nu_X(x)$. The Behrend function $\nu_X$ detects
non-reduced behaviour, so in general $\chi(X,\nu_X)\ne\chi(X^{\rm
red},\nu_{X^{\rm red}})$. For example, let $X$ be the $k$-fold point
$\Spec\bigl(\C[z]/(z^k)\bigr)$ for $k\ge 1$. Then $X(\C)$ is a
single point $x$ with $\nu_X(x)=k$, so $\chi(X)=1=\chi(X^{\rm
red},\nu_{X^{\rm red}})$, but~$\chi(X,\nu_X)=k$.

An important moral of \cite{Behr} is that (at least in moduli
problems with symmetric obstruction theories, such as
Donaldson--Thomas theory) it is better to `count' points in a moduli
scheme $\M$ by the weighted Euler characteristic $\chi(\M,\nu_\M)$
than by the unweighted Euler characteristic $\chi(\M)$. One reason
is that $\chi(\M,\nu_\M)$ often gives answers unchanged under
deformations of the underlying geometry, but $\chi(\M)$ does not.
For example, consider the family of $\C$-schemes
$X_t=\Spec\bigl(\C[z]/(z^2-t^2)\bigr)$ for $t\in\C$. Then $X_t$ is
two reduced points $\pm t$ for $t\ne 0$, and a double point when
$t=0$. So as above we find that $\chi(X_t,\nu_{X_t})=2$ for all $t$,
which is deformation-invariant, but $\chi(X_t)$ is 2 for $t\ne 0$
and 1 for $t=0$, which is not deformation-invariant.

Here are some important properties of Behrend functions:
\begin{itemize}
\setlength{\itemsep}{0pt}
\setlength{\parsep}{0pt}
\item[(i)] If $X$ is a smooth Artin $\C$-stack of dimension $n\in\Z$
then~$\nu_X\equiv(-1)^n$.
\item[(ii)] If\/ $\vp:X\ra Y$ is a smooth 1-morphism of Artin
$\C$-stacks of relative dimension $n\in\Z$
then~$\nu_X\equiv(-1)^nf^*(\nu_Y)$.
\item[(iii)] Suppose $X$ is a proper $\C$-scheme equipped with
a symmetric obstruction theory, and $[X]^\vir$ is the
corresponding virtual class. Then\index{virtual class}
\e
\ts\int_{[X]^\vir}1=\chi(X,\nu_X)\in\Z.
\label{dt1eq7}
\e
\item[(iv)] Let $U$ be a complex manifold and $f:U\ra\C$ a holomorphic
function, and define $X$ to be the complex analytic space
$\Crit(f)\subseteq U$. Then the Behrend function $\nu_X$ of $X$ is
given by
\e
\nu_X(x)=(-1)^{\dim U}\bigl(1-\chi(MF_f(x))\bigr) \qquad\text{for
$x\in X$,}
\label{dt1eq8}
\e
where $MF_f(x)$ is the {\it Milnor fibre\/} of $f$ at
$x$.\index{Milnor fibre}
\end{itemize}
Equation \eq{dt1eq7} explains the equivalence of the two expressions
for $DT^\al(\tau)$ in \eq{dt1eq1} and \eq{dt1eq2} above. The Milnor
fibre description \eq{dt1eq8} of Behrend functions will be crucial
in proving the Behrend function identities
\eq{dt1eq10}--\eq{dt1eq11} below.\index{Behrend function|)}

\subsection[Summary of the main results in $\text{\S\ref{dt5}}$]{Summary
of the main results in \S\ref{dt5}}
\label{dt1.3}

The basic idea behind this whole book is that we should insert the
Behrend function $\nu_\fM$ of the moduli stack $\fM$ of coherent
sheaves in $X$ as a weight in the programme of
\cite{Joyc1,Joyc2,Joyc3,Joyc4,Joyc5,Joyc6,Joyc7} summarized in
\S\ref{dt1.1}. Thus we will obtain weighted versions $\ti\Psi$ of
the Lie algebra morphism $\Psi$ of \eq{dt1eq4}, and
$\bar{DT}{}^\al(\tau)$ of the counting invariant $J^\al(\tau)\in\Q$
in \eq{dt1eq6}. Here is how this is worked out in~\S\ref{dt5}.

We define a modification $\ti L(X)$ of the Lie algebra $L(X)$ above,
the $\Q$-vector space with basis of symbols $\ti \la^\al$ for
$\al\in K^\num(\coh(X))$, with Lie bracket
\begin{equation*}
[\ti\la^\al,\ti\la^\be]=(-1)^{\bar\chi(\al,\be)}
\bar\chi(\al,\be)\ti \la^{\al+\be},
\end{equation*}
which is \eq{dt1eq3} with a sign change. Then we define a {\it Lie
algebra morphism\/} $\ti\Psi:\SFai(\fM)\ra\ti L(X)$. Roughly
speaking this is of the form
\e
\ti\Psi(f)=\ts\sum_{\al\in K^\num(\coh(X))}\chi^{\rm stk}
\bigl(f\vert_{\fM^\al},\nu_{\fM}\bigr)\ti \la^{\al},
\label{dt1eq9}
\e
that is, in \eq{dt1eq4} we replace the stack-theoretic Euler
characteristic $\chi^{\rm stk}$ with a stack-theoretic Euler
characteristic weighted by the Behrend function $\nu_{\fM}$.

The proof that $\ti\Psi$ is a Lie algebra morphism combines the
proof in \cite{Joyc4} that $\Psi$ is a Lie algebra morphism with the
two Behrend function identities\index{Behrend function!identities}
\begin{gather}
\nu_{\fM}(E_1\op E_2)=(-1)^{\bar\chi([E_1],[E_2])}
\nu_{\fM}(E_1)\nu_{\fM}(E_2),
\label{dt1eq10}\\
\begin{split}
\int_{\begin{subarray}{l}[\la]\in\mathbb{P}(\Ext^1(E_2,E_1)):\\
\la\; \Leftrightarrow\; 0\ra E_1\ra F\ra E_2\ra
0\end{subarray}}\!\!\!\!\!\!\nu_{\fM}(F) \rd\chi -
\int_{\begin{subarray}{l}[\la']\in\mathbb{P}(\Ext^1(E_1,E_2)):\\
\la'\; \Leftrightarrow\; 0\ra E_2\ra F'\ra E_1\ra
0\end{subarray}}\!\!\!\!\!\!\nu_{\fM}(F')\rd\chi \\
=\bigl(\dim\Ext^1(E_2,E_1)-\dim\Ext^1(E_1,E_2)\bigr)
\nu_{\fM}(E_1\op E_2),
\end{split}
\label{dt1eq11}
\end{gather}
which will be proved in Theorem \ref{dt5thm4}. Here in \eq{dt1eq11}
the correspondence between $[\la]\in\mathbb{P}(\Ext^1(E_2,E_1))$ and
$F\in\coh(X)$ is that $[\la]\in\mathbb{P} (\Ext^1(E_2,E_1))$ lifts
to some $0\ne\la\in\Ext^1(E_2,E_1)$, which corresponds to a short
exact sequence $0\ra E_1\ra F\ra E_2\ra 0$ in $\coh(X)$. The
function $[\la]\mapsto \nu_{\fM}(F)$ is a constructible function
$\mathbb{P}(\Ext^1(E_2,E_1))\ra\Z$, and the integrals in
\eq{dt1eq11} are integrals of constructible functions using Euler
characteristic as measure, as in~\cite{Joyc1}.

Proving \eq{dt1eq10}--\eq{dt1eq11} requires a deep understanding of
the local structure of the moduli stack $\fM$, which is of interest
in itself. First we show in \S\ref{dt8} using a composition of
Seidel--Thomas twists by $\cO_X(-n)$ for $n\gg 0$ that $\fM$ is
locally 1-isomorphic to the moduli stack $\fVect$ of vector bundles
on $X$. Then we prove in \S\ref{dt9} that near $[E]\in\fVect(\C)$,
an atlas for $\fVect$ can be written locally in the complex analytic
topology in the form $\Crit(f)$ for $f:U\ra\C$ a holomorphic
function on an open set $U$ in $\Ext^1(E,E)$. These $U,f$ are {\it
not algebraic}, they are constructed using gauge theory on the
complex vector bundle $E$ over $X$ and transcendental methods.
Finally, we deduce \eq{dt1eq10}--\eq{dt1eq11} in \S\ref{dt10} using
the Milnor fibre expression \eq{dt1eq8} for Behrend functions
applied to these~$U,f$.

We then define {\it generalized Donaldson--Thomas invariants\/}
$\bar{DT}{}^\al(\tau)\in\Q$ by\index{Donaldson--Thomas
invariants!generalized $\bar{DT}{}^\al(\tau)$}
\e
\ti\Psi\bigl(\bar\ep^\al(\tau)\bigr)=-\bar{DT}{}^\al(\tau)\ti
\la^\al,
\label{dt1eq12}
\e
as in \eq{dt1eq6}. When $\M_\rss^\al(\tau)=\M_\st^\al(\tau)$ we have
$\bar\ep^\al(\tau)=\bar\de_\rss^\al(\tau)$, and \eq{dt1eq9} gives
\e
\ti\Psi\bigl(\bar\ep^\al(\tau)\bigr)=\chi^{\rm stk}\bigl(
\fM_\st^\al (\tau),\nu_{\fM_\st^\al(\tau)}\bigr)\ti\la^\al.
\label{dt1eq13}
\e
The projection $\pi:\fM_\st^\al(\tau)\ra\M_\st^\al(\tau)$ from the
moduli stack to the coarse moduli scheme\index{coarse moduli
scheme}\index{moduli scheme!coarse} is smooth of dimension $-1$, so
$\nu_{\fM_\st^\al(\tau)}=-\pi^*(\nu_{\M_\st^\al(\tau)})$ by (ii) in
\S\ref{dt1.2}, and comparing \eq{dt1eq2}, \eq{dt1eq12}, \eq{dt1eq13}
shows that $\bar{DT}{}^\al(\tau)=DT^\al(\tau)$. But our new
invariants $\bar{DT}{}^\al(\tau)$ are also defined for $\al$ with
$\M_\rss^\al(\tau)\ne\M_\st^\al(\tau)$, when conventional
Donaldson--Thomas invariants $DT^\al(\tau)$ are not defined.

Write $C(\coh(X))=\bigl\{[E]\in K^\num(\coh(X)):0\ne
E\in\coh(X)\bigr\}$ for the `positive cone' of classes in
$K^\num(\coh(X))$ of nonzero objects in $\coh(X)$. Then
$\M_\rss^\al(\tau)=\M_\st^\al(\tau)=\es$ if $\al\in
K^\num(\coh(X))\sm C(\coh(X))$, so $\bar{DT}{}^\al(\tau)=0$. By
convention the zero sheaf is not (semi)stable, so
$\M_\rss^0(\tau)=\es$ and~$\bar{DT}{}^0(\tau)=0$.

Since $\ti\Psi$ is a Lie algebra morphism, the change of stability
condition formula for the $\bar\ep^\al(\tau)$ in \cite{Joyc6}
implies a formula for the elements $-\bar{DT}{}^\al(\tau)\ti
\la^\al$ in $\ti L(X)$, and thus a transformation law for the
invariants $\bar{DT}{}^\al(\tau)$, of the form
\e
\begin{aligned}
&\bar{DT}{}^\al(\ti\tau)=\\
&\!\!\!\sum_{\substack{\text{iso.}\\ \text{classes}\\
\text{of finite}\\ \text{sets $I$}}}
\sum_{\substack{\ka:I\ra C(\coh(X)):\\ \sum_{i\in I}\ka(i)=\al}}\,
\sum_{\begin{subarray}{l} \text{connected,}\\
\text{simply-}\\ \text{connected}\\ \text{digraphs $\Ga$,}\\
\text{vertices $I$}\end{subarray}}
\begin{aligned}[t]
(-1)^{\md{I}-1} V(I,\Ga,\ka;\tau,\ti\tau)\cdot
\prod\nolimits_{i\in I} \bar{DT}{}^{\ka(i)}(\tau)&\\
\cdot (-1)^{\frac{1}{2}\sum_{i,j\in
I}\md{\bar\chi(\ka(i),\ka(j))}}\cdot\! \prod\limits_{\text{edges
\smash{$\mathop{\bu}\limits^{\sst i}\ra\mathop{\bu}\limits^{\sst
j}$} in $\Ga$}\!\!\!\!\!\!\!\!\!\!\!\!\!\!\!\!\!\!\!\!\!\!\!
\!\!\!\!\!\!\!\!} \bar\chi(\ka(i),\ka(j))&,\!\!
\end{aligned}
\end{aligned}
\label{dt1eq14}
\e
where $\bar\chi$ is the Euler form on $K^\num(\coh(X))$, and
$V(I,\Ga,\ka;\tau,\ti\tau)\in\Q$ are combinatorial coefficients
defined in \S\ref{dt35}.

To study our new invariants $\bar{DT}{}^\al(\tau)$, we find it
helpful to introduce another family of invariants
$PI^{\al,n}(\tau')$,\index{stable pair invariants
$PI^{\al,n}(\tau')$} similar to Pandharipande--Thomas invariants
\cite{PaTh}.\index{Pandharipande--Thomas invariants} Let $n\gg 0$ be
fixed. A {\it stable pair\/}\index{stable pair} is a nonzero
morphism $s:\cO_{X}(-n)\ra E$ in $\coh(X)$ such that $E$ is
$\tau$-semistable, and if $\Im s\subset E'\subset E$ with $E'\ne E$
then $\tau([E'])<\tau([E])$. For $\al\in K^\num(\coh(X))$ and $n\gg
0$, the moduli space $\M_\stp^{\al,n}(\tau')$ of stable pairs
$s:\cO_X(-n)\ra X$ with $[E]=\al$ is a fine moduli
scheme,\index{fine moduli scheme}\index{moduli scheme!fine} which is
proper and has a symmetric obstruction theory.\index{symmetric
obstruction theory}\index{obstruction theory!symmetric} We define
\e
\ts PI^{\al,n}(\tau')=\int_{[\M_\stp^{\al,n}(\tau')]^\vir}1=
\chi\bigl( \M_\stp^{\al,n}(\tau'),\nu_{\M_\stp^{\al,n}(\tau')}
\bigr)\in\Z,
\label{dt1eq15}
\e
where the second equality follows from \eq{dt1eq7}. By a similar
proof to that for Donaldson--Thomas invariants in \cite{Thom}, we
find that $PI^{\al,n}(\tau')$ is unchanged under deformations of the
underlying Calabi--Yau 3-fold~$X$.

By a wall-crossing\index{wall-crossing formula} proof similar to that
for \eq{dt1eq14}, we show that $PI^{\al,n}(\tau')$ can be written in
terms of the $\bar{DT}{}^\be(\tau)$ by
\e
PI^{\al,n}(\tau')=\!\!\!\!\!\!\!\!\!\!\!\!\!\!\!
\sum_{\begin{subarray}{l} \al_1,\ldots,\al_l\in
C(\coh(X)),\\ l\ge 1:\; \al_1
+\cdots+\al_l=\al,\\
\tau(\al_i)=\tau(\al),\text{ all\/ $i$}
\end{subarray} \!\!\!\!\!\!\!\!\! }
\begin{aligned}[t] \frac{(-1)^l}{l!} &\prod_{i=1}^{l}\bigl[
(-1)^{\bar\chi([\cO_X(-n)]-\al_1-\cdots-\al_{i-1},\al_i)} \\
&\bar\chi\bigl([\cO_{X}(-n)]\!-\!\al_1\!-\!\cdots\!-\!\al_{i-1},\al_i
\bigr) \bar{DT}{}^{\al_i}(\tau)\bigr],\!\!\!\!\!\!\!\!\!\!\!\!\!\!
\end{aligned}
\label{dt1eq16}
\e
for $n\gg 0$. Dividing the sum in \eq{dt1eq16} into cases $l=1$ and
$l\ge 1$ gives
\e
PI^{\al,n}(\tau')=(-1)^{P(n)-1}P(n)\bar{DT}{}^\al(\tau)
+\bigl\{\text{terms in $\ts\prod_{i=1}^l\bar{DT}{}^{\al_i}(\tau)$,
$l\ge 2$}\bigr\},
\label{dt1eq17}
\e
where $P(n)=\bar\chi([\cO_X(-n)],\al)$ is the Hilbert polynomial of
$\al$, so that $P(n)>0$ for $n\gg 0$. As $PI^{\al,n}(\tau')$ is
deformation-invariant, we deduce from \eq{dt1eq17} by induction on
$\rank\al$ with $\dim\al$ fixed that $\bar{DT}{}^\al(\tau)$ is also
deformation-invariant.

The pair invariants $PI^{\al,n}(\tau')$ are a useful tool for
computing the $\bar{DT}{}^\al(\tau)$ in examples in \S\ref{dt6}. The
method is to describe the moduli spaces $\M_\stp^{\al,n}(\tau')$
explicitly, and then use \eq{dt1eq15} to compute
$PI^{\al,n}(\tau')$, and \eq{dt1eq16} to deduce the values of
$\bar{DT}{}^\al(\tau)$. Our point of view is that the
$\bar{DT}{}^\al(\tau)$ are of primary interest, and the
$PI^{\al,n}(\tau')$ are secondary invariants, of less interest in
themselves.

\subsection[Examples and applications in $\text{\S\ref{dt6}}$]{Examples
and applications in \S\ref{dt6}}
\label{dt1.4}

In \S\ref{dt6} we compute the invariants $\bar{DT}{}^\al(\tau)$ and
$PI^{\al,n}(\tau')$ in several examples. One basic example is this:
suppose that $E$ is a rigid, $\tau$-stable sheaf in class $\al$ in
$K^\num(\coh(X))$, and that the only $\tau$-semistable sheaf in
class $m\al$ up to isomorphism is $mE=\op^mE$, for all $m\ge 1$.
Then we show that
\e
\bar{DT}{}^{m\al}(\tau)=\frac{1}{m^2} \quad\text{for all $m\ge 1$.}
\label{dt1eq18}
\e
Thus the $\bar{DT}{}^\al(\tau)$ can lie in $\Q\sm\Z$. We think of
\eq{dt1eq18} as a `multiple cover formula', similar to the well
known Aspinwall--Morrison computation for a Calabi--Yau 3-fold $X$,
that a rigid embedded $\CP^1$ in class $\al\!\in\!H_2(X;\Z)$
contributes $1/m^3$ to the genus zero Gromov--Witten invariant
$GW_{0,0}(m\al)$ of $X$ in class $m\al$ for all~$m\!\ge\! 1$.

In Gromov--Witten theory, one defines {\it Gopakumar--Vafa
invariants\/}\index{Gopakumar--Vafa invariants} $GV_g(\al)$ which
are $\Q$-linear combinations of Gromov--Witten
invariants,\index{Gromov--Witten invariants} and are conjectured to
be integers, so that they `count' curves in $X$ in a more meaningful
way. For a Calabi--Yau 3-fold in genus $g=0$ these satisfy
\begin{equation*}
GW_{0,0}(\al)=\sum_{m\ge 1,\; m\mid\al}\frac{1}{m^3}\, GV_0(\al/m),
\end{equation*}
where the factor $1/m^3$ is the Aspinwall--Morrison contribution. In
a similar way, and following Kontsevich and Soibelman \cite[\S 2.5
\& \S 7.1]{KoSo1}, we define {\it BPS invariants}
$\hat{DT}{}^\al(\tau)$ to satisfy
\e
\bar{DT}{}^\al(\tau)=\sum_{m\ge 1,\; m\mid\al}\frac{1}{m^2}\,
\hat{DT}{}^{\al/m}(\tau),
\label{dt1eq19}
\e
where the factor $1/m^2$ comes from \eq{dt1eq18}. The inverse of
\eq{dt1eq19} is
\begin{equation*}
\hat{DT}{}^\al(\tau)=\sum_{m\ge 1,\; m\mid\al}\frac{\Mo(m)}{m^2}\,
\bar{DT}{}^{\al/m}(\tau),
\end{equation*}
where $\Mo(m)$ is the M\"obius function\index{Mobius
function@M\"obius function} from elementary number theory. We have
$\hat{DT}{}^\al(\tau)=DT^\al(\tau)$ when
$\M_\rss^\al(\tau)=\M_\st^\al(\tau)$, so the BPS invariants are also
generalizations of Donaldson--Thomas invariants.

A {\it stability condition\/} $(\tau,T,\le)$,\index{stability condition}
or $\tau$ for short, on $\coh(X)$ is a totally ordered set $(T,\le)$
and a map $\tau:C(\coh(X))\ra T$ such that if $\al,\be,\ga\in
C(\coh(X))$ with $\be=\al+\ga$ then
$\tau(\al)\!<\!\tau(\be)\!<\!\tau(\ga)$ or
$\tau(\al)\!=\!\tau(\be)\!=\!\tau(\ga)$ or
$\tau(\al)\!>\!\tau(\be)\!>\!\tau(\ga)$. We call a stability
condition $\tau$ {\it generic\/} if for all $\al,\be\in C(\coh(X))$
with $\tau(\al)=\tau(\be)$ we have $\bar\chi(\al,\be)=0$, where
$\bar\chi$ is the Euler form of $\coh(X)$. We conjecture that if
$\tau$ is generic, then $\hat{DT}{}^\al(\tau)\in\Z$ for all $\al\in
C(\coh(X))$. We give evidence for this conjecture, and in
\S\ref{dt76} we prove the analogous result for invariants
$\hat{DT}{}^{\bs d}_Q(\mu)$ counting representations of a quiver $Q$
without relations.

In the situations in \S\ref{dt6}--\S\ref{dt7} in which we can
compute invariants explicitly, we usually find that the values of
$PI^{\al,n}(\tau')$ are complicated (often involving generating
functions with a MacMahon function\index{MacMahon function} type
factorization), the values of $\bar{DT}{}^\al(\tau)$ are simpler,
and the values of $\hat{DT}{}^\al(\tau)$ are simplest of all. For
example, for dimension zero sheaves, if $p=[\cO_x]\in
K^\num(\coh(X))$ is the class of a point sheaf, and $\chi(X)$ is the
Euler characteristic of the Calabi--Yau 3-fold $X$, we have
\begin{gather*}
\ts 1+\sum_{d\ge 1} PI^{dp,n}(\tau')s^d=\prod_{k\ge 1}\bigl(1-
(-s)^k\bigr){}^{-k\,\chi(X)},\\
\ts\bar{DT}{}^{dp}(\tau)=-\chi(X)\sum_{l\ge 1, \; l \mid
d}\,\frac{1}{l^2},\quad\text{and}\quad
\hat{DT}{}^{dp}(\tau)=-\chi(X),\quad\text{all $d\ge 1$.}
\end{gather*}

\subsection[Extension to quivers with superpotentials in
$\text{\S\ref{dt7}}$]{Extension to quivers with superpotentials in
\S\ref{dt7}}
\label{dt1.5}\index{quiver!with superpotential|(}\index{quiver|(}

Section \ref{dt7} studies an analogue of Donaldson--Thomas theory in
which the abelian category $\coh(X)$ of coherent sheaves on a
Calabi--Yau 3-fold is replaced by the abelian category $\modCQI$ of
representations of a quiver with relations
$(Q,I)$,\index{quiver!with relations} where the relations $I$ are
defined using a {\it superpotential\/} $W$\index{superpotential} on
the quiver $Q$. This builds on the work of many authors; we mention
in particular Ginzburg \cite{Ginz}, Hanany et al.\
\cite{FHKVW,HHV,HaKe,HaVe}, Nagao and Nakajima
\cite{Naga,NaNa,Naka}, Reineke et al.\
\cite{EnRe,MoRe,Rein1,Rein2,Rein3}, Szendr\H oi \cite{Szen}, and
Young and Bryan~\cite{Youn,YoBr}.

Categories $\modCQI$ coming from a quiver $Q$ with superpotential
$W$ share two important properties with categories $\coh(X)$ for $X$
a Calabi--Yau 3-fold:
\begin{itemize}
\setlength{\itemsep}{0pt}
\setlength{\parsep}{0pt}
\item[(a)] The moduli stack $\fM_{Q,I}$ of
objects in $\modCQI$ can locally be written in terms of
$\Crit(f)$ for $f:U\ra\C$ holomorphic and $U$ smooth.
\item[(b)] For all $D,E$ in $\modCQI$ we have
\begin{align*}
\bar\chi\bigl(\bdim D,\bdim E\bigr)=
\,&\bigl(\dim\Hom(D,E)-\dim\Ext^1(D,E) \bigr)-\\
&\bigl(\dim\Hom(E,D)-\dim\Ext^1(E,D)\bigr),
\end{align*}
where $\bar\chi:\Z^{Q_0}\times\Z^{Q_0}$ is an explicit
antisymmetric biadditive form on the group of dimension vectors
for~$Q$.
\end{itemize}
Using these we can extend most of \S\ref{dt1.3} to $\modCQI$: the
Behrend function identities \eq{dt1eq10}--\eq{dt1eq11}, the Lie
algebra $\ti L(X)$ and Lie algebra morphism $\ti\Psi$, the
invariants $\bar{DT}{}^\al(\tau),PI^{\al,n}(\tau')$ and their
transformation laws \eq{dt1eq14} and \eq{dt1eq16}. We allow the case
$W\equiv 0$, so that $\modCQI=\modCQ$, the representations of a
quiver $Q$ without relations.

One aspect of the Calabi--Yau 3-fold case which does not extend is
that in $\coh(X)$ the moduli schemes $\M_\rss^\al(\tau)$ and
$\M_\stp^{\al,n}(\tau')$ are {\it proper}, but the analogues in
$\modCQI$ are not. Properness is essential for forming virtual
cycles and proving deformation-invariance of $\bar{DT}{}^\al
(\tau),PI^{\al,n}(\tau')$. Therefore, the quiver analogues of
$\bar{DT}{}^\al(\tau),PI^{\al,n}(\tau')$ will in general not be
invariant under deformations of the superpotential~$W$.

It is an interesting question why such categories $\modCQI$ should
be good analogues of $\coh(X)$ for $X$ a Calabi--Yau 3-fold. In some
important classes of examples $Q,W$, such as those coming from the
{\it McKay correspondence\/}\index{McKay correspondence} for $\C^3/G$
for finite $G\subset\SL(3,\C)$, or from a {\it consistent brane
tiling},\index{brane tiling} the abelian category $\modCQI$ is 3-{\it
Calabi--Yau},\index{abelian category!3-Calabi--Yau} that is, Serre
duality\index{Serre duality} holds in dimension 3, so that
$\Ext^i(D,E)\cong\Ext^{3-i}(E,D)^*$ for all $D,E$ in $\modCQI$. In
the general case, $\modCQI$ can be embedded as the heart of a
t-structure in a 3-Calabi--Yau triangulated category~$\cal T$.

It turns out that our new Donaldson--Thomas type invariants for
quivers $\smash{\bar{DT}{}^{\bs
d}_{Q,I}(\mu)},\ab\smash{\hat{DT}{}^{\bs d}_{Q,I}(\mu)}$ have not
really been considered, but the quiver analogues of pair invariants
$PI^{\al,n}(\tau')$, which we write as $NDT_{Q,I}^{\bs d,\bs
e}(\mu')$, are known in the literature as {\it noncommutative
Donaldson--Thomas invariants},\index{Donaldson--Thomas
invariants!noncommutative} and studied in
\cite{EnRe,MoRe,Rein1,Rein2,Rein3,Szen,Youn,YoBr}. We prove that the
analogue of \eq{dt1eq16} holds:
\e
NDT^{\bs d,\bs e}_{Q,I}(\mu')=\!\!\!\!\!\!\!\!\!\!\!\!\!\!\!
\sum_{\begin{subarray}{l} \bs d_1,\ldots,\bs d_l\in
C(\modCQI),\\ l\ge 1:\; \bs d_1
+\cdots+\bs d_l=\bs d,\\
\mu(\bs d_i)=\mu(\bs d),\text{ all\/ $i$}
\end{subarray} \!\!\!\!\!\!\!\!\! }
\begin{aligned}[t] \frac{(-1)^l}{l!} &\prod_{i=1}^{l}\bigl[
(-1)^{\bs e\cdot\bs d_i-\bar\chi(\bs d_1+\cdots+\bs
d_{i-1},\bs d_i)} \\
&\bigl(\bs e\cdot\bs d_i-\bar\chi(\bs d_1\!+\!\cdots\!+\!\bs
d_{i-1},\bs d_i)\bigr)\bar{DT}{}^{\bs
d_i}_{Q,I}(\mu)\bigr].\!\!\!\!\!\!
\end{aligned}
\label{dt1eq20}
\e

We use computations of $\smash{NDT^{\bs d,\bs e}_{Q,I}(\mu')}$ in
examples by Szendr\H oi \cite{Szen} and Young and Bryan \cite{YoBr},
and equation \eq{dt1eq20} to deduce values of
$\smash{\bar{DT}{}^{\bs d}_{Q,I}(\mu)}$ and hence
$\smash{\hat{DT}{}^{\bs d}_{Q,I}(\mu)}$ in examples. We find the
$\smash{NDT^{\bs d,\bs e}_{Q,I}(\mu')}$ are complicated, the
$\smash{\bar{DT}{}^{\bs d}_{Q,I}(\mu)}$ simpler, and the
$\smash{\hat{DT}{}^{\bs d}_{Q,I}(\mu)}$ are very simple; this
suggests that the $\smash{\hat{DT}{}^{\bs d}_{Q,I}(\mu)}$ may be
more useful invariants than the $\smash{NDT^{\bs d,\bs
e}_{Q,I}(\mu')}$, a better tool for understanding what is really
going on in these examples.

For quivers $Q$ without relations (that is, with superpotential
$W\equiv 0$) and for generic slope stability conditions $\mu$ on
$\modCQ$, we prove using work of Reineke \cite{Rein1,Rein3} that the
quiver BPS invariants $\hat{DT}{}^{\bs d}_Q(\mu)$ are
integers.\index{quiver!with superpotential|)}\index{quiver|)}

\subsection[Relation to the work of Kontsevich and Soibelman
$\text{\cite{KoSo1}}$]{Relation to the work of Kontsevich and
Soibelman \cite{KoSo1}}
\label{dt1.6}

The recent paper of Kontsevich and Soibelman \cite{KoSo1},
summarized in \cite{KoSo2}, has significant overlaps with this book,
and with the previously published series
\cite{Joyc1,Joyc2,Joyc3,Joyc4,Joyc5,Joyc6,Joyc7}. Kontsevich and
Soibelman also study generalizations of Donaldson--Thomas
invariants, but they are more ambitious than us, and work in a more
general context --- they consider derived categories of coherent
sheaves, Bridgeland stability conditions, and general motivic
invariants, whereas we work only with abelian categories of coherent
sheaves, Gieseker stability, and the Euler characteristic.

The large majority of the research in this book was done
independently of \cite{KoSo1}. After the appearance of Behrend's
seminal paper \cite{Behr} in 2005, it was clear to the first author
that Behrend's approach should be integrated with \cite{Joyc1,Joyc2,
Joyc3,Joyc4,Joyc5,Joyc6,Joyc7} to extend Donaldson--Thomas theory,
much along the lines of the present book. Within a few months the
first author applied for an EPSRC grant to do this, and started work
on the project with the second author in October 2006.

When we first received an early version of \cite{KoSo1} in April
2008, we understood the material of \S\ref{dt53}--\S\ref{dt54} below
and many of the examples in \S\ref{dt6}, and had written
\S\ref{dt12} as a preprint, and we knew we had to prove the Behrend
function identities \eq{dt1eq10}--\eq{dt1eq11}, but for some months
we were unable to do so. Our eventual solution of the problem, in
\S\ref{dt51}--\S\ref{dt52}, was rather different to the
Kontsevich--Soibelman formal power series approach in~\cite[\S 4.4
\& \S 6.3]{KoSo1}.

There are three main places in this book in which we have made
important use of ideas from Kontsevich and Soibelman \cite{KoSo1},
which we would like to acknowledge with thanks. The first is that in
the proof of \eq{dt1eq10}--\eq{dt1eq11} in \S\ref{dt10} we localize
by the action of $\bigl\{\id_{E_1}+\la\id_{E_2}:\la\in\U(1)\bigr\}$
on $\Ext^1(E_1\op E_2, E_1\op E_2)$, which is an idea we got from
\cite[Conj.~4, \S 4.4]{KoSo1}. The second is that in \S\ref{dt62}
one should define BPS invariants $\hat{DT}{}^\al(\tau)$, and they
should be integers for generic $\tau$, which came from \cite[\S 2.5
\& \S 7.1]{KoSo1}. The third is that in \S\ref{dt7} one should
consider Donaldson--Thomas theory for $\modCQI$ coming from a quiver
$Q$ with arbitrary minimal superpotential $W$, rather than only
those for which $\modCQI$ is 3-Calabi--Yau, which came in part
from~\cite[Th.~9, \S 8.1]{KoSo1}.

Having said all this, we should make it clear that the parallels
between large parts of \cite{Joyc1,Joyc2,Joyc3,Joyc4,Joyc5,
Joyc6,Joyc7} and this book on the one hand, and large parts of
\cite[\S\S 1,2,4,6 \& 7]{KoSo1} on the other, are really very close
indeed. Some comparisons:
\begin{itemize}
\setlength{\itemsep}{0pt}
\setlength{\parsep}{0pt}
\item `Motivic functions in the equivariant setting' \cite[\S
4.2]{KoSo1} are basically the stack functions of~\cite{Joyc2}.
\item The `motivic Hall algebra' $H({\cal C})$ \cite[\S 6.1]{KoSo1}
is a triangulated category version of Ringel--Hall
algebras\index{Ringel--Hall algebra} of stack functions
$\SF(\fM_\A)$ in \cite[\S 5]{Joyc4}.
\item The `motivic quantum torus' ${\cal R}_\Ga$ in \cite[\S 6.2]{KoSo1}
is basically the algebra $A(\A,\La,\chi)$ in \cite[\S 6.2]{Joyc4}.
\item The Lie algebra ${\mathfrak g}_\Ga$ of \cite[\S 1.4]{KoSo1}
is our $\ti L(X)$ in \S\ref{dt1.3}.
\item The algebra morphism $\Phi:H({\cal C})\ra{\cal R}_\Ga$ in
\cite[Th.~8]{KoSo1} is similar to the algebra morphism
$\Phi^\La:\SF(\fM_\A)\ra A(\A,\La,\chi)$ in \cite[\S
6.2]{Joyc4}, and our Lie algebra morphism $\ti\Psi$ in
\S\ref{dt53} should be some kind of limit of their~$\Phi$.
\item Once their algebra morphism $\Phi$ and our Lie algebra
morphism $\ti\Psi$ are constructed, we both follow the method of
\cite{Joyc6} exactly to define Donaldson--Thomas invariants and
prove wall-crossing formulae for them.
\item Our $\bar{DT}{}^\al(\tau)$ and $\hat{DT}{}^\al(\tau)$ in
\S\ref{dt53}, \S\ref{dt62} should correspond to their
`quasi-classical invariants' $-a(\al)$ and $\Om(\al)$ in
\cite[\S 2.5 \& \S 7.1]{KoSo1}, respectively.
\end{itemize}

Some differences between our programme and that of~\cite{KoSo1}:
\begin{itemize}
\setlength{\itemsep}{0pt}
\setlength{\parsep}{0pt}
\item Nearly every major result in \cite{KoSo1} depends explicitly or
implicitly on conjectures, whereas by being less ambitious, we can
give complete proofs.
\item Kontsevich and Soibelman also tackle issues to do with
triangulated categories, such as including $\Ext^i(D,E)$ for
$i<0$, which we do not touch.
\item Although our wall-crossing formulae are both proved using the
method of \cite{Joyc6}, we express them differently. Our
formulae are written in terms of combinatorial coefficients
$S,U(\al_1,\ldots,\al_n;\tau,\ti\tau)$ and
$V(I,\Ga,\ka;\tau,\ti\tau)$, as in \S\ref{dt33} and
\S\ref{dt35}. These are not easy to work with; see \S\ref{dt133}
for a computation of $U(\al_1,\ldots,\al_n;\tau,\ti\tau)$ in an
example.

By contrast, Kontsevich and Soibelman write their wall-crossing
formulae in terms of products in a pro-nilpotent Lie group
$G_V$. This seems to be an important idea, and may be a more
useful point of view than ours. See Reineke \cite{Rein3} for a
proof of an integrality conjecture \cite[Conj.~1]{KoSo1} on
factorizations in $G_V$, which is probably related to our
Theorem~\ref{dt7thm6}.
\item We prove the Behrend function identities\index{Behrend function!identities}
\eq{dt1eq10}--\eq{dt1eq11} by first showing that near a point $[E]$
the moduli stack $\fM$ can be written in terms of $\Crit(f)$ for
$f:U\ra\C$ holomorphic and $U$ open in $\Ext^1(E,E)$. The proof uses
gauge theory and transcendental methods, and works only over~$\C$.

Their parallel passages \cite[\S 4.4 \& \S 6.3]{KoSo1} work over
a field $\K$ of characteristic zero, and say that the formal
completion $\hat\fM_{[E]}$ of $\fM$ at $[E]$ can be written in
terms of $\Crit(f)$ for $f$ a formal power series on
$\smash{\Ext^1(E,E)}$, with no convergence criteria. Their
analogue of \eq{dt1eq10}--\eq{dt1eq11}, \cite[Conj.~4]{KoSo1},
concerns the `motivic Milnor fibre' of the formal power series
$f$.
\item In \cite{Joyc2,Joyc4,Joyc5,Joyc6} the first author put a lot
of effort into the difficult idea of `virtual rank',\index{virtual
rank} the projections $\Pi^\vi_n$ on stack functions, the Lie
algebra $\SFai(\fM)$ of stack functions `supported on virtual
indecomposables',\index{virtual indecomposable} and the proof
\cite[Th.~8.7]{Joyc5} that $\bar\ep^\al(\tau)$ in \eq{dt1eq5}
lies in $\SFai(\fM)$. This is very important for us, as our Lie
algebra morphism $\ti\Psi$ in \eq{dt1eq9} is defined only on
$\SFai(\fM)$, so $\bar{DT}{}^\al(\tau)$ in \eq{dt1eq12} is only
defined because~$\bar\ep^\al(\tau)\in\SFai(\fM)$.

Most of this has no analogue in \cite{KoSo1}, but they come up
against the problem this technology was designed to solve in
\cite[\S 7]{KoSo1}. Roughly speaking, they first define
Donaldson--Thomas invariants $\bar{DT}{}_{\rm vP}^\al(\tau)$
over virtual Poincar\'e polynomials, which are rational
functions in $t$. They then want to specialize to $t=-1$ to get
Donaldson--Thomas invariants over Euler characteristics, which
should coincide with our $\bar{DT}{}^\al(\tau)$. But this is
only possible if $\bar{DT}{}_{\rm vP}^\al(\tau)$ has no pole at
$t=-1$, which they assume in their `absence of poles
conjectures' in \cite[\S 7]{KoSo1}. The fact that
$\bar\ep^\al(\tau)$ lies in $\SFai(\fM)$ should be the key to
proving these conjectures.
\end{itemize}

\bigskip

\bigskip

\noindent{\bf Acknowledgements.} We would like especially to thank
Tom Bridgeland and Richard Thomas for lots of help with various
parts of this project, and also to thank Kai Behrend, Jim Bryan,
Daniel Fox, Spiro Karigiannis, Sheldon Katz, Bernhard Keller,
Alastair King, Martijn Kool, Alan Lauder, Davesh Maulik, Sven
Meinhardt, Tommaso Pacini, J\"org Sch\"urmann, Ed Segal, Yan
Soibelman, Bal\'{a}zs Szendr\H{o}i, and Yukinobu Toda for useful
conversations. We would also like to thank two referees for service
above and beyond the call of duty. This research was supported by
EPSRC grant~EP/D077990/1.

\section{Constructible functions and stack functions}
\label{dt2}

We begin with some background material on Artin stacks,
constructible functions, stack functions, and motivic invariants,
drawn mostly from~\cite{Joyc1,Joyc2}.

\subsection{Artin stacks and (locally) constructible functions}
\label{dt21}\index{constructible function|(}

{\it Artin stacks\/}\index{Artin stack} are a class of geometric spaces,
generalizing schemes and algebraic spaces. For a good introduction
to Artin stacks see G\'omez \cite{Gome}, and for a thorough
treatment see Laumon and Moret-Bailly \cite{LaMo}. We make the
convention that all Artin stacks in this book are {\it locally of
finite type}, and substacks are {\it locally closed}. We work
throughout over an algebraically closed field $\K$. For the parts of
the story involving constructible functions, or needing a
well-behaved notion of Euler characteristic, $\K$ must have
characteristic zero.\index{field $\K$}\index{field $\K$!characteristic
zero}\index{field $\K$!algebraically closed}\index{field $\K$!positive
characteristic}

Artin $\K$-stacks form a 2-{\it category}.\index{2-category} That is, we
have {\it objects} which are $\K$-stacks $\fF,\fG$, and also two
kinds of morphisms, 1-{\it morphisms}\index{1-morphism}
$\phi,\psi:\fF\ra\fG$ between $\K$-stacks, and 2-{\it
morphisms}\index{2-morphism} $A:\phi\ra\psi$ between 1-morphisms.

\begin{dfn} Let $\K$ be an algebraically closed field, and $\fF$ a
$\K$-stack. Write $\fF(\K)$\nomenclature[F(K)]{$\fF(\K)$}{set of $\K$-points of
an Artin $\K$-stack $\fF$} for the set of 2-isomorphism classes $[x]$
of 1-morphisms $x:\Spec\K\ra\fF$. Elements of $\fF(\K)$ are called
$\K$-{\it points}, or {\it geometric points}, of $\fF$. If
$\phi:\fF\ra\fG$ is a 1-morphism then composition with $\phi$
induces a map of sets~$\phi_*:\fF(\K)\ra\fG(\K)$.

For a 1-morphism $x:\Spec\K\ra\fF$, the {\it stabilizer
group}\index{Artin stack!stabilizer group}
$\Iso_\fF(x)$\nomenclature[Iso]{$\Iso_\fF(x)$}{stabilizer group of an Artin
stack $\fF$ at the point $x$} is the group of 2-morphisms $x\ra x$.
When $\fF$ is an Artin $\K$-stack, $\Iso_\fF(x)$ is an {\it
algebraic $\K$-group}. We say that $\fF$ {\it has affine geometric
stabilizers}\index{Artin stack!affine geometric stabilizers} if
$\Iso_\fF(x)$ is an affine algebraic $\K$-group for all
1-morphisms~$x:\Spec\K\ra\fF$.

As an algebraic $\K$-group up to isomorphism, $\Iso_\fF(x)$ depends
only on the isomorphism class $[x]\in\fF(\K)$ of $x$ in
$\Hom(\Spec\K,\fF)$. If $\phi:\fF\ra\fG$ is a 1-morphism,
composition induces a morphism of algebraic $\K$-groups
$\phi_*:\Iso_\fF([x])\ra\Iso_\fG\bigr(\phi_*([x])\bigr)$,
for~$[x]\in\fF(\K)$.
\label{dt2def1}
\end{dfn}

Next we discuss {\it constructible functions} on $\K$-stacks,
following \cite{Joyc1}.\index{constructible function!definition}

\begin{dfn} Let $\K$ be an algebraically closed field of
characteristic zero, and $\fF$ an Artin $\K$-stack. We call
$C\subseteq\fF(\K)$ {\it constructible}\index{constructible set} if
$C=\bigcup_{i\in I} \fF_i(\K)$, where $\{\fF_i:i\in I\}$ is a finite
collection of finite type Artin $\K$-substacks $\fF_i$ of $\fF$. We
call $S\subseteq\fF(\K)$ {\it locally constructible}\index{locally
constructible set} if $S\cap C$ is constructible for all
constructible~$C\subseteq\fF(\K)$.

A function $f:\fF(\K)\ra\Q$ is called {\it constructible} if
$f(\fF(\K))$ is finite and $f^{-1}(c)$ is a constructible set in
$\fF(\K)$ for each $c\in f(\fF(\K))\sm\{0\}$. A function
$f:\fF(\K)\ra\Q$ is called {\it locally constructible}\index{locally
constructible function}\index{constructible function!locally} if
$f\cdot\de_C$ is constructible for all constructible
$C\subseteq\fF(\K)$, where $\de_C$ is the characteristic function of
$C$. Write $\CF(\fF)$\nomenclature[CF(F)]{$\CF(\fF)$}{$\Q$-vector space of
constructible functions on a stack $\fF$} and
$\LCF(\fF)$\nomenclature[LCF(F)]{$\LCF(\fF)$}{$\Q$-vector space of locally
constructible functions on a stack $\fF$} for the $\Q$-vector spaces
of $\Q$-valued constructible and locally constructible functions
on~$\fF$.
\label{dt2def2}
\end{dfn}

Following \cite[Def.s~4.8, 5.1 \& 5.5]{Joyc1} we define {\it
pushforwards} and {\it pullbacks} of constructible functions along
1-morphisms.

\begin{dfn} Let $\K$ have characteristic zero, and $\fF$ be an
Artin $\K$-stack with affine geometric stabilizers and
$C\subseteq\fF(\K)$ be constructible. Then \cite[Def.~4.8]{Joyc1}
defines the {\it na\"\i ve Euler characteristic}\index{na\"\i ve Euler
characteristic}\index{Euler characteristic!na\"\i ve}
$\chi^\na(C)$\nomenclature[\chi na(C)]{$\chi^\na(C)$}{na\"\i ve Euler
characteristic of a constructible set $C$ in a stack} of $C$. It is
called {\it na\"\i ve} as it takes no account of stabilizer groups.
For $f\in\CF(\fF)$, define $\chi^\na(\fF,f)$\nomenclature[\chi
na(F,f)]{$\chi^\na(\fF,f)$}{na\"\i ve Euler characteristic of an
Artin stack $\fF$ weighted by a constructible function $f$} in $\Q$
by
\begin{equation*}
\chi^\na(\fF,f)=\ts\sum_{c\in f(\fF(\K))\sm\{0\}}c\,\chi^\na
\bigl(f^{-1}(c)\bigr).
\end{equation*}

Let $\fF,\fG$ be Artin $\K$-stacks with affine geometric
stabilizers, and $\phi:\fF\ra\fG$ a 1-morphism. For $f\in\CF(\fF)$,
define $\CF^\na(\phi)f:\fG(\K)\ra\Q$
by\nomenclature[CFna(phi)]{$\CF^\na(\phi)$}{na\"\i ve pushforward of
constructible functions along 1-morphism $\phi$}
\begin{equation*}
\CF^\na(\phi)f(y)=\chi^\na\bigl(\fF,f\cdot
\de_{\phi_*^{-1}(y)}\bigr) \quad\text{for $y\in\fG(\K)$,}
\end{equation*}
where $\de_{\smash{\phi_*^{-1}(y)}}$ is the characteristic function
of $\phi_*^{-1}(\{y\})\subseteq\fG(\K)$ on $\fG(\K)$. Then
$\CF^\na(\phi):\CF(\fF)\ra\CF(\fG)$ is a $\Q$-linear map called the
{\it na\"\i ve pushforward}.\index{na\"\i ve
pushforward}\index{pushforward!na\"\i ve}

Now suppose $\phi$ is {\it
representable}.\index{1-morphism!representable} Then for any
$x\in\fF(\K)$ we have an injective morphism
$\phi_*:\Iso_\fF(x)\ra\Iso_\fG\bigl(\phi_*(x)\bigr)$ of affine
algebraic $\K$-groups. The image $\phi_*\bigl(\Iso_\fF(x)\bigr)$ is
an affine algebraic $\K$-group closed in $\Iso_\fG\bigl(
\phi_*(x)\bigr)$, so the quotient $\Iso_\fG\bigl(\phi_*(x)\bigr)
/\phi_*\bigl(\Iso_\fF(x)\bigr)$ exists as a quasiprojective
$\K$-variety. Define a function $m_\phi:\fF(\K)\ra\Z$ by
$m_\phi(x)=\chi\bigl(\Iso_\fG(\phi_*(x))/\phi_*(\Iso_\fF(x))\bigr)$
for $x\in\fF(\K)$. For $f\in\CF(\fF)$, define
$\CF^\stk(\phi)f:\fG(\K)\ra\Q$ by
\begin{equation*}
\CF^\stk(\phi)f(y)=\chi^\na\bigl(\fF,m_\phi\cdot f\cdot
\de_{\phi_*^{-1}(y)}\bigr) \quad\text{for $y\in\fG(\K)$.}
\end{equation*}
An alternative definition is
\begin{equation*}
\CF^\stk(\phi)f(y)=\chi\bigl(\fF\times_{\phi,\fG,y}\Spec\K,\pi_\fF^*(f)\bigr)
\quad\text{for $y\in\fG(\K)$,}
\end{equation*}
where $\fF\times_{\phi,\fG,y}\Spec\K$ is a $\K$-scheme (or algebraic
space) as $\phi$ is representable, and $\chi(\cdots)$ is the Euler
characteristic of this $\K$-scheme weighted by $\pi_\fF^*(f)$. These
two definitions are equivalent as the projection
$\pi_1:\fF\times_{\phi,\fG,y}\Spec\K\ra\fF$ induces a map on
$\K$-points
$(\pi_1)_*:(\fF\times_{\phi,\fG,y}\Spec\K)(\K)\ra\phi_*^{-1}(y)\subset
\fF(\K)$, and the fibre of $(\pi_1)_*$ over $x\in\phi_*^{-1}(y)$ is
$\bigl(\Iso_\fG(\phi_*(x))/\phi_*(\Iso_\fF(x))\bigr)(\K)$, with
Euler characteristic $m_\phi(x)$. Then $\CF^\stk(\phi):\CF(\fF)\ra
\CF(\fG)$ is a $\Q$-linear map called the {\it stack
pushforward}.\index{stack
pushforward}\index{pushforward!stack}\nomenclature[CFstk(phi)]{$\CF^\stk(\phi)$}{stack
pushforward of constructible functions along representable $\phi$}
If $\fF,\fG$ are $\K$-schemes then $\CF^\na(\phi),\CF^\stk(\phi)$
coincide, and we write them both as $\CF(\phi):\CF(\fF)\ra\CF(\fG)$.

Let $\theta:\fF\ra\fG$ be a finite type
1-morphism.\index{1-morphism!finite type} If $C\subseteq \fG(\K)$ is
constructible then so is $\theta_*^{-1}(C)\subseteq\fF(\K)$. It
follows that if $f\in\CF(\fG)$ then $f\ci\theta_*$ lies in
$\CF(\fF)$. Define the {\it pullback\/}
$\theta^*:\CF(\fG)\ra\CF(\fF)$ by $\theta^*(f)= f\ci\theta_*$. It is
a linear map.
\label{dt2def3}
\end{dfn}

Here \cite[Th.s~4.9, 5.4, 5.6 \& Def.~5.5]{Joyc1} are some
properties of these.

\begin{thm} Let\/ $\K$ have characteristic zero,\index{field
$\K$!characteristic zero} $\fE,\fF,\fG,\fH$ be Artin $\K$-stacks
with affine geometric stabilizers, and\/ $\be:\fF\ra\fG,$
$\ga:\fG\ra\fH$ be $1$-morphisms. Then
\ea
\CF^\na(\ga\ci\be)&=\CF^\na(\ga)\ci\CF^\na(\be):\CF(\fF)\ra\CF(\fH),
\label{dt2eq1}\\
\CF^\stk(\ga\ci\be)&=\CF^\stk(\ga)\ci\CF^\stk(\be):\CF(\fF)\ra\CF(\fH),
\label{dt2eq2}\\
(\ga\ci\be)^*&=\be^*\ci\ga^*:\CF(\fH)\ra\CF(\fF),
\label{dt2eq3}
\ea
supposing $\be,\ga$ representable in {\rm\eq{dt2eq2},} and of finite
type in \eq{dt2eq3}. If
\e
\begin{gathered}
\xymatrix{
\fE \ar[r]_\eta \ar[d]^\theta & \fG \ar[d]_\psi \\
\fF \ar[r]^\phi & \fH }
\end{gathered}
\quad
\begin{gathered}
\text{is a Cartesian square\index{Cartesian square} with}\\
\text{$\eta,\phi$ representable and}\\
\text{$\theta,\psi$ of finite type, then}\\
\text{the following commutes:}
\end{gathered}
\quad
\begin{gathered}
\xymatrix@C=35pt{
\CF(\fE) \ar[r]_{\CF^\stk(\eta)} & \CF(\fG) \\
\CF(\fF) \ar[r]^{\CF^\stk(\phi)} \ar[u]_{\theta^*} & \CF(\fH).
\ar[u]^{\psi^*} }
\end{gathered}
\label{dt2eq4}
\e
\label{dt2thm1}
\end{thm}

As discussed in \cite[\S 3.3]{Joyc1}, equation \eq{dt2eq2} is {\it
false\/} for $\K$ of positive characteristic, so constructible
function methods tend to fail in positive
characteristic.\index{constructible function|)}\index{field $\K$!positive
characteristic}\index{constructible function!in positive characteristic}

\subsection{Stack functions}
\label{dt22}\index{stack function|(}

{\it Stack functions\/} are a universal generalization of
constructible functions introduced in \cite[\S 3]{Joyc2}. Here
\cite[Def.~3.1]{Joyc2} is the basic definition.

\begin{dfn} Let $\K$ be an algebraically closed field, and $\fF$ be
an Artin $\K$-stack with affine geometric stabilizers. Consider
pairs $(\fR,\rho)$, where $\fR$ is a finite type Artin $\K$-stack
with affine geometric stabilizers and $\rho:\fR\ra\fF$ is a
1-morphism. We call two pairs $(\fR,\rho)$, $(\fR',\rho')$ {\it
equivalent\/} if there exists a 1-isomorphism $\io:\fR\ra\fR'$ such
that $\rho'\ci\io$ and $\rho$ are 2-isomorphic 1-morphisms
$\fR\ra\fF$. Write $[(\fR,\rho)]$ for the equivalence class of
$(\fR,\rho)$. If $(\fR,\rho)$ is such a pair and $\fS$ is a closed
$\K$-substack of $\fR$ then $(\fS,\rho\vert_\fS)$,
$(\fR\sm\fS,\rho\vert_{\fR\sm\fS})$ are pairs of the same kind.

Define $\uSF(\fF)$\nomenclature[SFa(F)]{$\uSF(\fF)$}{vector space of `stack
functions' on an Artin stack $\fF$} to be the $\Q$-vector space
generated by equivalence classes $[(\fR,\rho)]$ as above, with for
each closed $\K$-substack $\fS$ of $\fR$ a relation
\e
[(\fR,\rho)]=[(\fS,\rho\vert_\fS)]+[(\fR\sm\fS,\rho\vert_{\fR\sm\fS})].
\label{dt2eq5}
\e
Define $\SF(\fF)$\nomenclature[SFb(F)]{$\SF(\fF)$}{vector space of `stack
functions' on an Artin stack $\fF$, defined using representable
1-morphisms} to be the $\Q$-vector space generated by $[(\fR,\rho)]$
with $\rho$ representable, with the same relations \eq{dt2eq5}.
Then~$\SF(\fF)\subseteq\uSF(\fF)$.
\label{dt2def4}
\end{dfn}\index{stack function!definition}

Elements of $\uSF(\fF)$ will be called {\it stack functions}. We
write stack functions either as letters $f,g,\ldots,$ or explicitly
as sums $\sum_{i=1}^mc_i[(\fR_i,\rho_i)]$. If $[(\fR,\rho)]$ is a
generator of $\uSF(\fF)$ and $\fR^{\rm red}$ is the reduced substack
of $\fR$ then $\fR^{\rm red}$ is a closed substack of $\fR$ and the
complement $\fR\sm\fR^{\rm red}$ is empty. Hence \eq{dt2eq5} implies
that
\begin{equation*}
[(\fR,\rho)]=[(\fR^{\rm red},\rho\vert_{\fR^{\rm red}})].
\end{equation*}
Thus, the relations \eq{dt2eq5} destroy all information on
nilpotence in the stack structure of $\fR$. In
\cite[Def.~3.2]{Joyc2} we relate $\CF(\fF)$ and~$\SF(\fF)$.

\begin{dfn} Let $\fF$ be an Artin $\K$-stack with affine
geometric stabilizers, and $C\subseteq\fF(\K)$ be constructible.
Then $C=\coprod_{i=1}^n\fR_i(\K)$, for $\fR_1,\ldots,\fR_n$ finite
type $\K$-substacks of $\fF$. Let $\rho_i:\fR_i\ra\fF$ be the
inclusion 1-morphism. Then $[(\fR_i,\rho_i)]\in\SF(\fF)$. Define
$\bde_C=\ts\sum_{i=1}^n[(\fR_i,\rho_i)]\in\SF(\fF)$. We think of
this stack function as the analogue of the characteristic function
$\de_C\in\CF(\fF)$ of $C$. When $\K$ has characteristic zero, define
a $\Q$-linear map $\io_\fF:\CF(\fF)\ra\SF(\fF)$ by
$\io_\fF(f)=\ts\sum_{0\ne c\in f(\fF(\K))}c\cdot\bde_{f^{-1}(c)}$.
Define $\Q$-linear $\pi_\fF^\stk:\SF(\fF)\ra\CF(\fF)$ by
\begin{equation*}
\pi_\fF^\stk\bigl(\ts\sum_{i=1}^nc_i[(\fR_i,\rho_i)]\bigr)=
\ts\sum_{i=1}^nc_i\CF^\stk(\rho_i)1_{\fR_i},
\end{equation*}
where $1_{\fR_i}$ is the function 1 in $\CF(\fR_i)$. Then
\cite[Prop.~3.3]{Joyc2} shows $\pi_\fF^\stk\ci\io_\fF$ is the
identity on $\CF(\fF)$. Thus, $\io_\fF$ is injective and
$\pi_\fF^\stk$ is surjective. In general $\io_\fF$ is far from
surjective, and $\uSF,\SF(\fF)$ are much larger than~$\CF(\fF)$.
\label{dt2def5}
\end{dfn}

The operations on constructible functions in \S\ref{dt21} extend to
stack functions.

\begin{dfn} Define {\it multiplication} `$\,\cdot\,$' on $\uSF(\fF)$ by
\e
[(\fR,\rho)]\cdot[(\fS,\si)]=[(\fR\times_{\rho,\fF,\si}\fS,\rho\ci\pi_\fR)].
\label{dt2eq6}
\e
This extends to a $\Q$-bilinear product $\uSF(\fF)\times\uSF(\fF)\ra
\uSF(\fF)$ which is commutative and associative, and $\SF(\fF)$ is
closed under `$\,\cdot\,$'. Let $\phi:\fF\!\ra\!\fG$ be a 1-morphism
of Artin $\K$-stacks with affine geometric stabilizers. Define the
{\it pushforward\/} $\phi_*:\uSF(\fF)\!\ra\!\uSF(\fG)$~by
\begin{equation*}
\phi_*:\ts\sum_{i=1}^mc_i[(\fR_i,\rho_i)]\longmapsto
\ts\sum_{i=1}^mc_i[(\fR_i,\phi\ci\rho_i)].
\end{equation*}
If $\phi$ is representable then $\phi_*$ maps $\SF(\fF)\!\ra\!
\SF(\fG)$. For $\phi$ of finite type, define {\it pullbacks}
$\phi^*:\uSF(\fG)\!\ra\!\uSF(\fF)$,
$\phi^*:\SF(\fG)\!\ra\!\SF(\fF)$~by
\e
\phi^*:\ts\sum_{i=1}^mc_i[(\fR_i,\rho_i)]\longmapsto
\ts\sum_{i=1}^mc_i[(\fR_i\times_{\rho_i,\fG,\phi}\fF,\pi_\fF)].
\label{dt2eq7}
\e
The {\it tensor product\/}
$\ot\!:\!\uSF(\fF)\!\times\!\uSF(\fG)\!\ra \!\uSF(\fF\!\times\!\fG)$
or $\SF(\fF)\!\times\!\SF(\fG)\!\ra\! \SF(\fF\!\times\!\fG)$~is
\e
\bigl(\ts\sum_{i=1}^mc_i[(\fR_i,\rho_i)]\bigr)\!\ot\!
\bigl(\ts\sum_{j=1}^nd_j[(\fS_j,\si_j)]\bigr)\!=\!\ts
\sum_{i,j}c_id_j[(\fR_i\!\times\!\fS_j,\rho_i\!\times\!\si_j)].
\label{dt2eq8}
\e
\label{dt2def6}
\end{dfn}

Here \cite[Th.~3.5]{Joyc2} is the analogue of Theorem~\ref{dt2thm1}.

\begin{thm} Let\/ $\fE,\fF,\fG,\fH$ be Artin $\K$-stacks with
affine geometric stabilizers, and\/ $\be:\fF\ra\fG,$ $\ga:\fG\ra\fH$
be $1$-morphisms. Then
\begin{align*}
(\ga\!\ci\!\be)_*\!&=\!\ga_*\!\ci\!\be_*:\uSF(\fF)\!\ra\!\uSF(\fH),&
(\ga\!\ci\!\be)_*\!&=\!\ga_*\!\ci\!\be_*:\SF(\fF)\!\ra\!\SF(\fH),\\
(\ga\!\ci\!\be)^*\!&=\!\be^*\!\ci\!\ga^*:\uSF(\fH)\!\ra\!\uSF(\fF),&
(\ga\!\ci\!\be)^*\!&\!=\!\be^*\!\ci\!\ga^*:\SF(\fH)\!\ra\!\SF(\fF),
\end{align*}
for $\be,\ga$ representable in the second equation, and of finite
type in the third and fourth. If\/ $f,g\in\uSF(\fG)$ and\/ $\be$ is
finite type then $\be^*(f\cdot g)=\be^*(f)\cdot\be^*(g)$. If
\begin{equation*}
\begin{gathered}
\xymatrix@R=15pt{
\fE \ar[r]_\eta \ar[d]^{\,\theta} & \fG \ar[d]_{\psi\,} \\
\fF \ar[r]^\phi & \fH }
\end{gathered}
\quad
\begin{gathered}
\text{is a Cartesian square with}\\
\text{$\theta,\psi$ of finite type, then}\\
\text{the following commutes:}
\end{gathered}
\quad
\begin{gathered}
\xymatrix@C=35pt@R=10pt{
\uSF(\fE) \ar[r]_{\eta_*} & \uSF(\fG) \\
\uSF(\fF) \ar[r]^{\phi_*} \ar[u]_{\,\theta^*} & \uSF(\fH).
\ar[u]^{\psi^*\,} }
\end{gathered}
\end{equation*}
The same applies for $\SF(\fE),\ldots,\SF(\fH)$ if\/ $\eta,\phi$ are
representable.
\label{dt2thm2}
\end{thm}

In \cite[Prop.~3.7 \& Th.~3.8]{Joyc2} we relate pushforwards and
pullbacks of stack and constructible functions
using~$\io_\fF,\pi_\fF^\stk$.

\begin{thm} Let\/ $\K$ have characteristic zero, $\fF,\fG$ be
Artin $\K$-stacks with affine geometric stabilizers, and\/
$\phi:\fF\ra\fG$ be a $1$-morphism. Then
\begin{itemize}
\setlength{\itemsep}{0pt}
\setlength{\parsep}{0pt}
\item[{\rm(a)}] $\phi^*\!\ci\!\io_\fG\!=\!\io_\fF\!\ci\!\phi^*:
\CF(\fG)\!\ra\!\SF(\fF)$ if\/ $\phi$ is of finite type;
\item[{\rm(b)}] $\pi^\stk_\fG\ci\phi_*=\CF^\stk(\phi)\ci\pi_\fF^\stk:
\SF(\fF)\ra\CF(\fG)$ if\/ $\phi$ is representable; and
\item[{\rm(c)}] $\pi^\stk_\fF\ci\phi^*=\phi^*\ci\pi_\fG^\stk:
\SF(\fG)\ra\CF(\fF)$ if\/ $\phi$ is of finite type.
\end{itemize}
\label{dt2thm3}
\end{thm}

In \cite[\S 3]{Joyc2} we extend all the material on $\uSF,\SF(\fF)$
to {\it local stack functions} $\uLSF,\LSF(\fF)$, the analogues of
locally constructible functions. The main differences are in which
1-morphisms must be of finite type.\index{local stack function}\index{stack
function!local}\nomenclature[LSFb(F)]{$\LSF(\fF)$}{vector space of `local stack
functions' on an Artin stack $\fF$, defined using representable
1-morphisms}\nomenclature[LSFa(F)]{$\uLSF(\fF)$}{vector space of `local stack
functions' on an Artin stack $\fF$}

\subsection{Operators $\Pi^\mu$ and projections $\Pi^\vi_n$}
\label{dt23}

We will need the following standard notation and facts about
algebraic $\K$-groups and tori, which can be found in Borel
\cite{Bore}. Throughout $\K$ is an algebraically closed field and
$G$ is an affine algebraic $\K$-group.\index{algebraic $\K$-group}
\begin{itemize}
\setlength{\itemsep}{0pt}
\setlength{\parsep}{0pt}
\item Write $\bG_m$ for $\K\sm\{0\}$ as a $\K$-group under
multiplication.\nomenclature[Gm]{$\bG_m$}{the algebraic $\K$-group
$\K\sm\{0\}$}
\item By a {\it torus} we mean an algebraic $\K$-group isomorphic
to $\bG_m^k$ for some $k\ge 0$. A {\it subtorus} of $G$ means a
$\K$-subgroup of $G$ which is a torus.
\item A {\it maximal torus}\index{maximal torus}\index{algebraic
$\K$-group!maximal torus} in $G$ is a subtorus
$T^G$\nomenclature[TG]{$T^G$}{maximal torus in an algebraic $\K$-group $G$}
contained in no larger subtorus $T$ in $G$. All maximal tori in
$G$ are conjugate by Borel \cite[Cor.~IV.11.3]{Bore}. The {\it
rank\/}\index{rank} $\rk\,G$ is the dimension of any maximal torus.
A maximal torus in $\GL(k,\K)$ is the subgroup $\bG_m^k$ of
diagonal matrices.
\item Let $T$ be a torus and $H$ a closed $\K$-subgroup of $T$.
Then $H$ is isomorphic to $\bG_m^k\times K$ for some $k\ge 0$
and finite abelian group~$K$.
\item If $S$ is a subset of $T^G$, define the {\it
centralizer}\index{centralizer} of $S$ in $G$ to be $C_G(S)=\{\ga\in
G:\ga s=s\ga$ $\forall s\in S\}$,\nomenclature[CGS]{$C_G(S)$}{centralizer
of a subset $S$ in a group $G$} and the {\it
normalizer}\index{normalizer} of $S$ in $G$ to be $N_G(S)=\{\ga\in
G:\ga^{-1}S\ga=S\}$.\nomenclature[NGS]{$N_G(S)$}{normalizer of a subset $S$
in a group $G$} They are closed $\K$-subgroups of $G$ containing
$T^G$, and $C_G(S)$ is normal in~$N_G(S)$.
\item The quotient group $W(G,T^G)=N_G(T^G)/C_G(T^G)$ is called the
{\it Weyl group}\index{Weyl group} of
$G$.\nomenclature[W(G,TG)]{$W(G,T^G)$}{Weyl group of algebraic $\K$-group
$G$} As in \cite[IV.11.19]{Bore} it is a finite group, which
acts on~$T^G$.
\item Define the {\it centre}\index{algebraic
$\K$-group!centre} of $G$ to be $C(G)=\{\ga\in G:\ga\de=\de\ga$
$\forall\de\in G\}$.\nomenclature[C(G)]{$C(G)$}{centre of an algebraic
$\K$-group $G$} It is a closed $\K$-subgroup of~$G$.
\item An algebraic $\K$-group $G$ is called {\it
special\/}\index{algebraic $\K$-group!special}\index{special algebraic
$\K$-group} if every principal $G$-bundle locally trivial in the
\'etale topology\index{etale topology@\'etale topology} is also
locally trivial in the Zariski topology.\index{Zariski topology}
Properties of special $\K$-groups can be found in \cite[\S\S
1.4, 1.5 \& 5.5]{Chev} and \cite[\S 2.1]{Joyc2}. Special
$\K$-groups are always affine and connected. Products of special
groups are special.
\item $\bG_m^k$ and $\GL(k,\K)$ are special for all $k\ge 0$.
\end{itemize}

Now we define some linear maps~$\Pi^\mu:\uSF(\fF)\ra\uSF(\fF)$.

\begin{dfn} A {\it weight function\/} $\mu$ is a map
\begin{equation*}
\mu:\bigl\{\text{$\K$-groups $\bG_m^k\!\times\!K$, $k\!\ge\!0$, $K$
finite abelian, up to isomorphism}\bigr\}\!\longra\!\Q.
\end{equation*}
For any Artin $\K$-stack $\fF$ with affine geometric stabilizers, we
will define linear maps $\Pi^\mu:\uSF(\fF)\ra\uSF(\fF)$ and
$\Pi^\mu:\SF(\fF)\ra\SF(\fF)$. Now $\uSF(\fF)$ is generated by
$[(\fR,\rho)]$ with $\fR$ 1-isomorphic to a quotient $[X/G]$, for
$X$ a quasiprojective $\K$-variety and $G$ a special algebraic
$\K$-group, with maximal torus $T^G$.

Let ${\cal S}(T^G)$ be the set of subsets of $T^G$ defined by
Boolean operations upon closed $\K$-subgroups $L$ of $T^G$. Given a
weight function $\mu$ as above, define a measure $\rd\mu:{\cal
S}(T^G)\ra\Q$ to be additive upon disjoint unions of sets in ${\cal
S}(T^G)$, and to satisfy $\rd\mu(L)=\mu(L)$ for all algebraic
$\K$-subgroups $L$ of $T^G$. Define
\e
\begin{split}
&\Pi^\mu\bigl([(\fR,\rho)]\bigr)=\\
&\int_{t\in T^G}\frac{\md{\{w\in W(G,T^G):w\cdot
t=t\}}}{\md{W(G,T^G)}}\,\bigl[\bigl([X^{\{t\}}/
C_G(\{t\})],\rho\ci\io^{\{t\}}\bigr)\bigr]\rd\mu.
\end{split}
\label{dt2eq9}
\e
Here $X^{\{t\}}$ is the subvariety of $X$ fixed by $t$, and
$\io^{\{t\}}:[X^{\{t\}}/ C_G(\{t\})]\ra[X/G]$ is the obvious
1-morphism of Artin stacks.

The integrand in \eq{dt2eq9}, regarded as a function of $t\in T^G$,
is a constructible function taking only finitely many values. The
level sets of the function lie in ${\cal S}(T^G)$, so they are
measurable w.r.t.\ $\rd\mu$, and the integral is well-defined.
\label{dt2def7}
\end{dfn}

If $\fR$ has abelian stabilizer groups, then $\Pi^\mu\bigl(
[(\fR,\rho)]\bigr)$ simply weights each point $r$ of $\fR$ by
$\mu(\Iso_\fR(r))$. But if $\fR$ has nonabelian stabilizer groups,
then $\Pi^\mu\bigl([(\fR,\rho)]\bigr)$ replaces each point $r$ with
stabilizer group $G$ by a $\Q$-linear combination of points with
stabilizer groups $C_G(\{t\})$ for $t\in T^G$, where the
$\Q$-coefficients depend on the values of $\mu$ on subgroups of
$T^G$. Then \cite[Th.s~5.11 \& 5.12]{Joyc2} shows:

\begin{thm} In the situation above, $\Pi^\mu\bigl([(\fR,\rho)]
\bigr)$ is independent of the choices of\/ $X,G,T^G$ and\/
$1$-isomorphism $\fR\cong[X/G],$ and\/ $\Pi^\mu$ extends to unique
linear maps $\Pi^\mu:\uSF(\fF)\ra\uSF(\fF)$
and\/~$\Pi^\mu:\SF(\fF)\ra\SF(\fF)$.
\label{dt2thm4}
\end{thm}

\begin{thm} {\rm(a)} $\Pi^1$ defined using $\mu\equiv 1$ is the
identity on~$\uSF(\fF)$.
\begin{itemize}
\setlength{\itemsep}{0pt}
\setlength{\parsep}{0pt}
\item[{\rm(b)}] If\/ $\phi:\fF\ra\fG$ is a $1$-morphism of
Artin $\K$-stacks with affine geometric stabilizers
then~$\Pi^\mu\ci\phi_*=\phi_*\ci\Pi^\mu:\uSF(\fF)\ra\uSF(\fG)$.
\item[{\rm(c)}] If\/ $\mu_1,\mu_2$ are weight functions as in
Definition {\rm\ref{dt2def7}} then $\mu_1\mu_2$ is also a weight
function and\/~$\Pi^{\mu_2}\ci\Pi^{\mu_1}=\Pi^{\mu_1}\ci\Pi^{\mu_2}=
\Pi^{\mu_1\mu_2}$.
\end{itemize}
\label{dt2thm5}
\end{thm}

\begin{dfn} For $n\ge 0$, define
$\Pi^\vi_n$\nomenclature[\Pi]{$\Pi^\vi_n$}{projection to stack
functions with `virtual rank $n$'} to be the operator $\Pi^{\mu_n}$
defined with weight $\mu_n$ given by $\mu_n([H])=1$ if $\dim H=n$
and $\mu_n([H])=0$ otherwise, for all $\K$-groups
$H\cong\bG_m^k\times K$ with $K$ a finite abelian group.
\label{dt2def8}
\end{dfn}

Here \cite[Prop.~5.14]{Joyc2} are some properties of the
$\Pi^\vi_n$.

\begin{prop} In the situation above, we have:
\begin{itemize}
\setlength{\itemsep}{0pt}
\setlength{\parsep}{0pt}
\item[{\rm(i)}] $(\Pi^\vi_n)^2=\Pi^\vi_n,$ so that\/ $\Pi^\vi_n$ is
a projection, and\/ $\Pi^\vi_m\ci\Pi^\vi_n=0$ for~$m\ne n$.
\item[{\rm(ii)}] For all\/ $f\in\uSF(\fF)$ we have $f=\sum_{n\ge 0}
\Pi^\vi_n(f),$ where the sum makes sense as $\Pi^\vi_n(f)=0$
for~$n\gg 0$.
\item[{\rm(iii)}] If\/ $\phi:\fF\ra\fG$ is a $1$-morphism of
Artin $\K$-stacks with affine geometric stabilizers
then~$\Pi^\vi_n\ci\phi_*=\phi_*\ci\Pi^\vi_n:\uSF(\fF)\ra\uSF(\fG)$.
\item[{\rm(iv)}] If\/ $f\in\uSF(\fF),$ $g\in\uSF(\fG)$
then~$\Pi^\vi_n(f\ot
g)=\sum_{m=0}^n\Pi^\vi_m(f)\ot\Pi^\vi_{n-m}(g)$.
\end{itemize}
\label{dt2prop1}
\end{prop}

Very roughly speaking, $\Pi^\vi_n$ projects
$[(\fR,\rho)]\in\uSF(\fF)$ to $[(\fR_n,\rho)]$, where $\fR_n$ is the
$\K$-substack of points $r\in\fR(\K)$ whose stabilizer groups
$\Iso_\fR(r)$ have rank $n$, that is, maximal torus $\bG_m^n$.
Unfortunately, it is more complicated than this. The right notion is
not the actual rank of stabilizer groups, but the {\it virtual
rank}.\index{virtual rank} We treat $r\in\fR(\K)$ with nonabelian
stabilizer group $G=\Iso_\fR(r)$ as a linear combination of points
with `virtual ranks' in the range $\rk\,C(G)\le n\le\rk\,G$.
Effectively this {\it abelianizes stabilizer groups}, that is, using
virtual rank we can treat $\fR$ as though its stabilizer groups were
all abelian, essentially tori~$\bG_m^n$.

\subsection{Stack function spaces
${\bar{\text{\underline{\rm SF\!}\,}},\bar{\rm SF}({\mathfrak
F},\chi,{\mathbb Q})}$}
\label{dt24}

In \cite[\S 4]{Joyc2} we extend {\it motivic} invariants\index{motivic
invariant} of quasiprojective $\K$-varieties, such as Euler
characteristics, virtual Poincar\'e polynomials, and virtual Hodge
polynomials, to Artin stacks. Then in \cite[\S 4--\S 6]{Joyc2} we
define several different classes of stack function spaces `twisted
by motivic invariants'. This is a rather long, complicated story,
which we will not explain. Instead, we will discuss only the spaces
$\uoSF,\oSF(\fF,\chi,\Q)$ `twisted by the Euler characteristic'
which we need later.

Throughout this section $\K$ is an algebraically closed field of
characteristic zero. We continue to use the notation on algebraic
$\K$-groups in \S\ref{dt23}. Here is some more notation,
\cite[Def.s~5.5 \& 5.16]{Joyc2}.

\begin{dfn} Let $G$ be an affine algebraic $\K$-group with maximal
torus $T^G$. If $S\subset T^G$ then $Q=T^G\cap C(C_G(S))$ is a
closed $\K$-subgroup of $T^G$ containing $S$. As $S\subseteq Q$ we
have $C_G(Q)\subseteq C_G(S)$. But $Q$ commutes with $C_G(S)$, so
$C_G(S)\subseteq C_G(Q)$. Thus $C_G(S)=C_G(Q)$. So $Q=T^G\cap
C(C_G(Q))$, and $Q$ and $C_G(Q)$ determine each other, given
$G,T^G$. Define $\cQ(G,T^G)$\nomenclature[Q(G,TG)]{$\cQ(G,T^G)$}{a set of
subtori of the maximal torus $T^G$ of a $\K$-group $G$} to be the
set of closed $\K$-subgroups $Q$ of $T^G$ such that~$Q=T^G\cap
C(C_G(Q))$.

In \cite[Lem.~5.6]{Joyc2} we show that $\cQ(G,T^G)$ is finite and
closed under intersections, with maximal element $T^G$ and minimal
element~$Q_{\rm min}=T^G\cap C(G)$.

An affine algebraic $\K$-group $G$ is called {\it very
special\/}\index{algebraic $\K$-group!very special} if $C_G(Q)$ and $Q$
are special for all $Q\in\cQ(G,T^G)$, for any maximal torus $T^G$ in
$G$. Then $G$ is special, as $G=C_G(Q_{\rm min})$. In \cite[Ex.~5.7
\& Def.~5.16]{Joyc2} we compute $\cQ(G,T^G)$ for $G=\GL(k,\K)$, and
deduce that $\GL(k,\K)$ is very special.
\label{dt2def9}
\end{dfn}

We can now define the spaces $\uoSF,\oSF(\fF,\chi,\Q)$, \cite[Def.s
5.17 \& 6.8]{Joyc2}.

\begin{dfn} Let $\fF$ be an Artin $\K$-stack with affine
geometric stabilizers. Consider pairs $(\fR,\rho)$, where $\fR$ is a
finite type Artin $\K$-stack with affine geometric stabilizers and
$\rho:\fR\ra\fF$ is a 1-morphism, with equivalence of pairs as in
Definition \ref{dt2def4}. Define
$\uoSF(\fF,\chi,\Q)$\nomenclature[SFc(F,\chi)]{$\uoSF(\fF,\chi,\Q)$}{vector
space of `stack functions' on an Artin stack $\fF$ with extra
relations involving the Euler characteristic} to be the $\Q$-vector
space generated by equivalence classes $[(\fR,\rho)]$ as above, with
the following relations:
\begin{itemize}
\setlength{\itemsep}{0pt}
\setlength{\parsep}{0pt}
\item[(i)] Given $[(\fR,\rho)]$ as above and $\fS$ a closed $\K$-substack
of $\fR$ we have $[(\fR,\rho)]=[(\fS,\rho\vert_\fS)]+[(\fR\sm\fS,
\rho\vert_{\fR\sm\fS})]$, as in~\eq{dt2eq5}.
\item[(ii)] Let $\fR$ be a finite type Artin $\K$-stack with
affine geometric stabilizers, $U$ a quasiprojective
$\K$-variety, $\pi_\fR:\fR\times U\ra\fR$ the natural
projection, and $\rho:\fR\ra\fF$ a 1-morphism. Then~$[(\fR\times
U,\rho\ci\pi_\fR)] =\chi([U])[(\fR,\rho)]$.

Here $\chi(U)\in\Z$ is the Euler characteristic of $U$. It is a {\it
motivic invariant\/} of $\K$-schemes, that is,
$\chi(U)=\chi(V)+\chi(U\sm V)$ for $V\subset U$ closed.
\item[(iii)] Given $[(\fR,\rho)]$ as above and a 1-isomorphism
$\fR\cong[X/G]$ for $X$ a quasiprojective $\K$-variety and $G$ a
very special algebraic $\K$-group acting on $X$ with maximal torus
$T^G$, we have
\e
[(\fR,\rho)]=\ts\sum_{Q\in\cQ(G,T^G)}F(G,T^G,Q)
\bigl[\bigl([X/Q],\rho\ci\io^Q\bigr)\bigr],
\label{dt2eq10}
\e
where $\io^Q:[X/Q]\ra\fR\cong[X/G]$ is the natural projection
1-morphism.
\end{itemize}
Here $F(G,T^G,Q)\in\Q$\nomenclature[F(G,T^G,Q)]{$F(G,T^G,Q)$}{rational
coefficients used in the definition of stack function spaces
$\uoSF,\oSF(\fF,\chi,\Q)$} are a system of rational coefficients
with a complicated definition in \cite[\S 6.2]{Joyc2}, which we will
not repeat. In \cite[\S 6.2]{Joyc2} we derive an inductive formula
for computing them when~$G=\GL(k,\K)$.

Similarly, define
$\oSF(\fF,\chi,\Q)$\nomenclature[SFd(F,\chi)]{$\oSF(\fF,\chi,\Q)$}{vector space
of `stack functions' on an Artin stack $\fF$ with extra relations,
defined using representable 1-morphisms} to be the $\Q$-vector space
generated by $[(\fR,\rho)]$ with $\rho$ representable, and relations
(i)--(iii) as above. Then $\oSF(\fF,\chi,\Q)\subset
\uoSF(\fF,\chi,\Q)$. Define projections
$\bar\Pi^{\chi,\Q}_\fF:\uSF(\fF) \ra\uoSF(\fF,\chi,\Q)$ and
$\SF(\fF) \ra\oSF(\fF,\chi,\Q)$ by
$\bar\Pi^{\chi,\Q}_\fF:\ts\sum_{i\in I}c_i[(\fR_i,\rho_i)]\mapsto
\ts\sum_{i\in I}c_i[(\fR_i,\rho_i)]$.

Define {\it multiplication\/} `$\,\cdot\,$', {\it pushforwards\/}
$\phi_*$, {\it pullbacks\/} $\phi^*$, and {\it tensor products}
$\ot$ on the spaces $\uoSF,\oSF(*,\chi,\Q)$ as in Definition
\ref{dt2def6}, and {\it projections\/} $\Pi^\vi_n$ as in
\S\ref{dt23}. The important point is that \eq{dt2eq6}--\eq{dt2eq9}
are compatible with the relations defining $\uoSF,\oSF(*,\chi,\Q)$,
or they would not be well-defined. This is proved in \cite[Th.s 5.19
\& 6.9]{Joyc2}, and depends on deep properties of the~$F(G,T^G,Q)$.
\label{dt2def10}
\end{dfn}

Here \cite[Prop.s 5.21 \& 5.22 \& \S 6.3]{Joyc2} is a useful way to
represent these spaces.

\begin{prop} $\uoSF,\oSF(\fF,\chi,\Q)$ are spanned over $\Q$ by
elements $[(U\times[\Spec\K/T],\rho)],$ for\/ $U$ a quasiprojective
$\K$-variety and\/ $T$ an algebraic $\K$-group isomorphic to
$\bG_m^k\times K$ for $k\ge 0$ and\/ $K$ finite abelian.

Suppose $\sum_{i\in I}c_i[(U_i\times[\Spec\K/T_i],\rho_i)]=0$ in
$\uoSF(\fF,\chi,\Q)$ or $\oSF(\fF,\chi,\Q),$ where $I$ is finite
set, $c_i\in\Q,$ $U_i$ is a quasiprojective $\K$-variety, and\/
$T_i$ is an algebraic $\K$-group isomorphic to $\bG_m^{k_i}\times
K_i$ for $k_i\ge 0$ and\/ $K_i$ finite abelian, with\/ $T_i\not\cong
T_j$ for $i\ne j$. Then $c_j[(U_j\times[\Spec\K/T_j],\rho_j)]=0$ for
all\/~$j\in I$.
\label{dt2prop2}
\end{prop}

In this representation, the operators $\Pi^\vi_n$ of \S\ref{dt23}
are easy to define: we have
\begin{equation*}
\Pi^\vi_n\bigl([(U\times[\Spec\K/T],\rho)]\bigr)=
\begin{cases} [(U\times[\Spec\K/T],\rho)], &\dim T=n,\\
0, & \text{otherwise.}\end{cases}
\end{equation*}
Proposition \ref{dt2prop2} says that a general element
$[(\fR,\rho)]$ of $\uoSF,\oSF(\fF,\chi,\Q)$, whose stabilizer groups
$\Iso_\fR(x)$ for $x\in\fR(\K)$ are arbitrary affine algebraic
$\K$-groups, may be written as a $\Q$-linear combination of elements
$[(U\times[\Spec\K/T],\rho)]$ whose stabilizer groups $T$ are of the
form $\bG_m^k\times K$ for $k\ge 0$ and $K$ finite abelian. That is,
{\it by working in\/ $\uoSF,\oSF(\fF,\chi,\Q),$ we can treat all
stabilizer groups as if they are abelian}. Furthermore, although
$\uoSF,\oSF(\fF,\chi,\Q)$ forget information about nonabelian
stabilizer groups, they do remember the difference between abelian
stabilizer groups of the form $\bG_m^k\times K$ for finite~$K$.

In \cite[Prop.~6.11]{Joyc2} we completely
describe~$\uoSF,\oSF(\Spec\K,\chi,\Q)$.

\begin{prop} Define a commutative $\Q$-algebra $\La$ with basis
isomorphism classes $[T]$ of\/ $\K$-groups\/ $T$ of the form\/
$\bG_m^k\times K,$ for $k\ge 0$ and\/ $K$ finite abelian, with
multiplication $[T]\cdot[T']=[T\times T']$. Define
$i_\La:\La\ra\uoSF (\Spec\K,\chi,\Q)$ by
$\sum_ic_i[T_i]\mapsto\sum_ic_i [[\Spec\K/T_i]]$. Then $i_\La$ is an
algebra isomorphism. It restricts to an
isomorphism~$i_\La:\Q[\{1\}]\ra \oSF(\Spec\K,\chi,\Q)\cong\Q$.
\label{dt2prop3}
\end{prop}

Proposition \ref{dt2prop3} shows that the relations Definition
\ref{dt2def10}(i)--(iii) are well chosen, and in particular, the
coefficients $F(G,T^G,Q)$ in \eq{dt2eq10} have some beautiful
properties. If the $F(G,T^G,Q)$ were just some random numbers, one
might expect relation (iii) to be so strong that
$\uoSF(\fF,\chi,\Q)$ would be small, or even zero, for all $\fF$.
But $\uoSF(\Spec\K,\chi,\Q)$ is large, and easily
understood.\index{stack function|)}

\section[Background material from
$\text{\cite{Joyc3,Joyc4,Joyc5,Joyc6}}$]{Background material from
\cite{Joyc3,Joyc4,Joyc5,Joyc6}}
\label{dt3}

Next we review material from the first author's series of papers
\cite{Joyc3,Joyc4,Joyc5,Joyc6}.

\subsection{Ringel--Hall algebras of an abelian category}\index{abelian
category|(}\index{Ringel--Hall algebra|(}\index{abelian
category!Ringel--Hall algebra|(}
\label{dt31}

Let $\A$ be a $\K$-linear abelian category. We define the {\it
Grothendieck group\/} $K_0(\A)$, the {\it Euler form\/} $\bar\chi$,
and the {\it numerical Grothendieck group\/}~$K^\num(\A)$.

\begin{dfn} Let $\A$ be an abelian category. The {\it Grothendieck
group\/}\index{Grothendieck
group}\nomenclature[K0(A)]{$K_0(\A)$}{Grothendieck group of an
abelian category $\A$} $K_0(\A)$ is the abelian group generated by
all isomorphism classes $[E]$ of objects $E$ in $\A$, with the
relations $[E]+[G]=[F]$ for each short exact sequence $0\ra E\ra
F\ra G\ra 0$ in $\A$. In many interesting cases such as
$\A=\coh(X)$, the Grothendieck group $K_0(\A)$ is very large, and it
is useful to replace it by a smaller group. Suppose $\A$ is
$\K$-linear for some algebraically closed field $\K$, and that
$\Ext^*(E,F)$ is finite-dimensional over $\K$ for all $E,F\in\A$.
The {\it Euler form\/}\index{Euler
form}\nomenclature[\chi]{$\bar\chi$}{Euler form of an abelian
category} $\bar\chi:K_0(\A)\times K_0(\A)\ra\Z$ is a biadditive map
satisfying
\e
\bar\chi\bigl([E],[F]\bigr)=\ts\sum_{i\ge 0}(-1)^i\dim\Ext^i(E,F)
\label{dt3eq1}
\e
for all $E,F\in\A$. We use the notation $\bar\chi$ rather than
$\chi$ for the Euler form, because $\chi$ will be used often to mean
Euler characteristic or weighted Euler characteristic. The {\it
numerical Grothendieck group\/}\index{Grothendieck
group!numerical}\nomenclature[Knum(A)]{$K^\num(\A)$}{numerical
Grothendieck group of an abelian category $\A$} $K^\num(\A)$ is the
quotient of $K_0(\A)$ by the (two-sided) kernel of $\bar\chi$, that
is, $K^\num(\A)=K_0(\A)/I$ where $I=\bigl\{\al\in
K_0(\A):\bar\chi(\al,\be)=\bar\chi(\be,\al)=0$ for all $\be\in
K_0(\A)\bigr\}$. Then $\bar\chi$ on $K_0(\A)$ descends to a
biadditive Euler form $\bar\chi:K^\num(\A)\times K^\num(\A)\ra\Z$.

If $\A$ is 3-Calabi--Yau\index{abelian category!3-Calabi--Yau} then
$\bar\chi$ is antisymmetric, so the left and right kernels of
$\bar\chi$ on $K_0(\A)$ coincide, and $\bar\chi$ on $K^\num(\A)$ is
nondegenerate. If $\A=\coh(X)$ for $X$ a smooth projective
$\K$-scheme of dimension $m$ then Serre duality implies that
$\bar\chi(E,F)=(-1)^m \bar\chi(F,E\ot K_X)$. Thus, again, the left
and right kernels of $\bar\chi$ on $K_0(\coh(X))$ are the same, and
$\bar\chi$ on $K^\num(\coh(X))$ is nondegenerate.
\label{dt3def1}
\end{dfn}

Our goal is to associate a {\it Ringel--Hall algebra\/}
$\SFa(\fM_\A)$ to $\A$. To do this we will need to be able to do
algebraic geometry in $\A$, in particular, to form moduli
$\K$-stacks of objects and exact sequences in $\A$ and 1-morphisms
between them. This requires some extra data, described
in~\cite[Assumptions 7.1 \& 8.1]{Joyc3}.

\begin{ass} Let $\K$ be an algebraically closed field and $\A$
a $\K$-linear abelian category with $\Ext^i(E,F)$ finite-dimensional
$\K$-vector spaces for all $E,F$ in $\A$ and $i\ge 0$. Let $K(\A)$
be the quotient of the Grothendieck group $K_0(\A)$ by some fixed
subgroup. Usually we will take
$K(\A)=K^\num(\A)$,\nomenclature[K(A)]{$K(\A)$}{a quotient of the Grothendieck
group $K_0(\A)$ of an abelian category $\A$. Often
$K(\A)=K^\num(\A)$} the numerical Grothendieck group from Definition
\ref{dt3def1}. Suppose that if $E\in\A$ with $[E]=0$ in $K(\A)$ then
$E\cong 0$. From \S\ref{dt32} we will also assume $\A$ is {\it
noetherian}.

To define moduli stacks of objects or configurations in $\A$, we
need some {\it extra data}, to tell us about algebraic families of
objects and morphisms in $\A$, parametrized by a base scheme $U$. We
encode this extra data as a {\it stack in exact categories\/}
$\fF_\A$ on the {\it category of\/ $\K$-schemes\/} $\Sch_\K$, made
into a {\it site\/} with the {\it \'etale topology}.\index{etale
topology@\'etale topology} The $\K,\A,K(\A),\fF_\A$ must satisfy
some complex additional conditions \cite[Assumptions 7.1 \&
8.1]{Joyc3}, which we do not give.
\label{dt3ass}
\end{ass}

Examples of data satisfying Assumption \ref{dt3ass} are given in
\cite[\S 9--\S 10]{Joyc3}. These include $\A=\coh(X)$, the abelian
category of coherent sheaves on a smooth projective $\K$-scheme $X$,
with $K(\A)=K^\num(\coh(X))$, and $\A=\text{mod-}\K Q/I$, the
abelian category of $\K$-representations of a quiver
$Q=(Q_0,Q_1,b,e)$ with relations $I$, with $K(\A)=\Z^{Q_0}$, the
lattice of dimension vectors for~$Q$.

Suppose Assumption \ref{dt3ass} holds. We will use the following
notation:
\begin{itemize}
\setlength{\itemsep}{0pt}
\setlength{\parsep}{0pt}
\item Define the `positive cone'\index{positive
cone}\nomenclature[C(A)]{$C(\A)$}{the positive cone in $K(\A)$} $C(\A)$ in
$K(\A)$ to be
\e
C(\A)=\bigl\{[E]\in K(\A):0\not\cong E\in\A\bigr\}\subset K(\A).
\label{dt3eq2}
\e
\item Write $\fM_\A$\nomenclature[MaA]{$\fM_\A$ or $\fM$}{moduli stack of
objects in an abelian category $\A$} for the moduli stack of
objects in $\A$. It is an Artin $\K$-stack, locally of finite
type. Elements of $\fM_\A(\K)$ correspond to isomorphism classes
$[E]$ of objects $E$ in $\A$, and the stabilizer group
$\Iso_{\fM_\A}([E])$ in $\fM_\A$ is isomorphic as an algebraic
$\K$-group to the automorphism group~$\Aut(E)$.
\item For $\al\in C(\A)$, write
$\fM^\al_\A$\nomenclature[MaAa]{$\fM^\al_\A$ or $\fM^\al$}{moduli stack of
objects in $\A$ with class $\al\in K(\A)$} for the substack of
objects $E\in\A$ in class $\al$ in $K(\A)$. It is an open and
closed $\K$-substack of $\fM_\A$.
\item Write $\fExact_\A$\nomenclature[ExactA]{$\fExact_\A$}{moduli stack
of short exact sequences in an abelian category $\A$} for the
moduli stack of short exact sequences $0\ra E_1\ra E_2\ra E_3\ra
0$ in $\A$. It is an Artin $\K$-stack, locally of finite type.
\item For $j=1,2,3$ write $\pi_j:\fExact_\A\ra\fM_\A$ for the
1-morphism of Artin stacks projecting $0\ra E_1\ra E_2\ra E_3\ra
0$ to $E_j$. Then $\pi_2$ is {\it
representable},\index{1-morphism!representable} and
$\pi_1\times\pi_3:\fExact_\A\ra \fM_\A\times\fM_\A$ is of {\it
finite type}.\index{1-morphism!finite type}
\end{itemize}

In \cite{Joyc4} we define {\it Ringel--Hall algebras}, using stack
functions.

\begin{dfn} Suppose Assumption \ref{dt3ass} holds. Define bilinear
operations $*$ on the stack function spaces\index{stack function}
$\uSF,\SF(\fM_\A)$ and $\uoSF,\oSF(\fM_\A,\chi,\Q)$ by
\e
f*g=(\pi_2)_*\bigl((\pi_1\times\pi_3)^*(f\ot g)\bigr),
\label{dt3eq3}
\e
using pushforwards, pullbacks and tensor products in Definition
\ref{dt2def6}. They are well-defined as $\pi_2$ is representable,
and $\pi_1\times\pi_3$ is of finite type. By \cite[Th.~5.2]{Joyc4}
this
* is {\it associative}, and makes $\uSF,\SF(\fM_\A)$,
$\uoSF,\oSF(\fM_\A,\chi,\Q)$ into noncommutative $\Q$-algebras, with
identity $\bde_{[0]}$, where $[0]\in\fM_\A$ is the zero object. We
call them {\it Ringel--Hall algebras}, as they are a version of the
Ringel--Hall method for defining algebras from abelian categories.
The natural inclusions and projections $\bar\Pi^{\chi,\Q}_{\fM_\A}$
between these spaces are algebra morphisms.

As these algebras are inconveniently large for some purposes, in
\cite[Def.~5.5]{Joyc4} we define subalgebras $\SFa(\fM_\A),
\oSFa(\fM_\A,\chi,\Q)$ using the algebra structure on stabilizer
groups in $\fM_\A$. Suppose $[(\fR,\rho)]$ is a generator of
$\SF(\fM_\A)$. Let $r\in\fR(\K)$ with $\rho_*(r)=[E]\in\fM_\A(\K)$,
for some $E\in\A$. Then $\rho$ induces a morphism of stabilizer
$\K$-groups $\rho_*:\Iso_\fR(r)\ra\Iso_{\fM_\A}([E])\cong\Aut(E)$.
As $\rho$ is representable this is {\it injective}, and induces an
isomorphism of $\Iso_\fR(r)$ with a $\K$-subgroup of $\Aut(E)$. Now
$\Aut(E)=\End(E)^\times$ is the $\K$-group of invertible elements in
a {\it finite-dimensional\/ $\K$-algebra} $\End(E)=\Hom(E,E)$. We
say that $[(\fR,\rho)]$ {\it has algebra stabilizers\/}\index{stack
function!with algebra stabilizers} if whenever $r\in\fR(\K)$ with
$\rho_*(r)=[E]$, the $\K$-subgroup $\rho_*\bigl(\Iso_\fR(r)\bigr)$
in $\Aut(E)$ is the $\K$-group $A^\times$ of invertible elements in
a $\K$-subalgebra $A$ in $\End(E)$. Write
$\SFa(\fM_\A),\oSFa(\fM_\A,\chi,\Q)$ for the subspaces of
$\SF(\fM_\A),\oSF(\fM_\A,\chi,\Q)$ spanned over $\Q$ by
$[(\fR,\rho)]$ with algebra stabilizers. Then
\cite[Prop.~5.7]{Joyc4} shows that
$\SFa(\fM_\A),\oSFa(\fM_\A,\chi,\Q)$ are {\it subalgebras\/} of the
Ringel--Hall algebras~$\SF(\fM_\A),\oSF(\fM_\A,\chi,\Q)$.

Now \cite[Cor.~5.10]{Joyc4} shows that $\SFa(\fM_\A),\oSFa
(\fM_\A,\chi,\Q)$ are closed under the operators $\Pi^\vi_n$ on
$\SF(\fM_\A),\oSF(\fM_\A,\chi,\Q)$ defined in \S\ref{dt23}. In
\cite[Def.~5.14]{Joyc4} we define
$\SFai(\fM_\A),\oSFai(\fM_\A,\chi,\Q)$ to be the subspaces of $f$ in
$\SFa(\fM_\A)$ and $\oSFa (\fM_\A,\chi,\Q)$ with $\Pi^\vi_1(f)=f$.
We think of $\SFai(\fM_\A),\oSFai(\fM_\A,\chi,\Q)$ as stack
functions {\it `supported on virtual indecomposables'}.\index{stack
function!supported on virtual indecomposables}\index{virtual
indecomposable} This is because if $E\in\A$ then $\rk\Aut(E)$ is the
number of indecomposable factors of $E$, that is, $\rk\Aut(E)=r$ if
$E\cong E_1\op\cdots\op E_r$ with $E_i$ nonzero and indecomposable
in $\A$. But $\Pi^\vi_1$ projects to stack functions with `virtual
rank' 1, and thus with `one virtual indecomposable factor'.

In \cite[Th.~5.18]{Joyc4} we show $\SFai(\fM_\A),\oSFai
(\fM_\A,\chi,\Q)$ are closed under the Lie bracket $[f,g]=f*g-g*f$
on $\SFa(\fM_\A),\ab\oSFa(\fM_\A,\chi,\Q)$. Thus,
$\SFai\!(\fM_\A),\ab\oSFai (\fM_\A,\chi,\Q)$ are {\it Lie
subalgebras} of $\SFa(\fM_\A),\oSFa(\fM_\A,\chi,\Q)$. The projection
$\bar\Pi^{\chi,\Q}_{\fM_\A}:\SFai\!(\fM_\A)\ra \oSFai
(\fM_\A,\chi,\Q)$ is a Lie algebra morphism.
\label{dt3def2}
\end{dfn}

As in \cite[Cor.~5.11]{Joyc4}, the first part of Proposition
\ref{dt2prop2} simplifies to give:

\begin{prop} $\oSFa(\fM_\A,\chi,\Q)$ is spanned over $\Q$ by
elements of the form $[(U\times[\Spec\K/\bG_m^k],\rho)]$ with
algebra stabilizers, for\/ $U$ a quasiprojective $\K$-variety and\/
$k\ge 0$. Also $\oSFai(\fM_\A,\chi,\Q)$ is spanned over $\Q$ by
$[(U\times[\Spec\K/\bG_m],\rho)]$ with algebra stabilizers, for\/
$U$ a quasiprojective $\K$-variety.
\label{dt3prop1}
\end{prop}\index{Ringel--Hall algebra|)}\index{abelian
category!Ringel--Hall algebra|)}

\subsection{(Weak) stability conditions on $\A$}
\label{dt32}\index{stability condition|(}

Next we discuss material in \cite{Joyc5} on {\it stability
conditions}.

\begin{dfn} Let $\A$ be an abelian category, $K(\A)$ be the
quotient of $K_0(\A)$ by some fixed subgroup, and $C(\A)$ as in
\eq{dt3eq2}. Suppose $(T,\le)$ is a totally ordered set, and
$\tau:C(\A)\ra T$ a map. We call $(\tau,T,\le)$\nomenclature[\tau
T]{$(\tau,T,\le)$}{stability condition on an abelian category} a
{\it stability condition} on $\A$ if whenever $\al,\be,\ga\in C(\A)$
with $\be=\al+\ga$ then either
$\tau(\al)\!<\!\tau(\be)\!<\!\tau(\ga)$, or
$\tau(\al)\!>\!\tau(\be)\!>\!\tau(\ga)$, or
$\tau(\al)\!=\!\tau(\be)\!=\!\tau(\ga)$. We call $(\tau,T,\le)$ a
{\it weak stability condition} on $\A$ if whenever $\al,\be, \ga\in
C(\A)$ with $\be=\al+\ga$ then either $\tau(\al)\!\le\!
\tau(\be)\!\le\!\tau(\ga)$, or $\tau(\al)\!\ge\!\tau(\be)\!\ge
\!\tau(\ga)$.\index{stability condition!weak}

For such $(\tau,T,\le)$, we say that a nonzero object $E$ in $\A$ is
\begin{itemize}
\setlength{\itemsep}{0pt}
\setlength{\parsep}{0pt}
\item[(i)] $\tau$-{\it semistable} if for all $S\subset E$ with
$S\not\cong 0,E$ we have $\tau([S])\le\tau([E/S])$;
\item[(ii)] $\tau$-{\it stable} if for all $S\subset E$ with
$S\not\cong 0,E$ we have $\tau([S])<\tau([E/S])$; and
\item[(iii)] $\tau$-{\it unstable} if it is not $\tau$-semistable.
\end{itemize}\index{semistable@$\tau$-semistable}
\index{stable@$\tau$-stable}\index{unstable@$\tau$-unstable}

Given a weak stability condition $(\tau,T,\le)$ on $\A$, we say that
$\A$ is $\tau$-{\it artinian}\index{artinian@$\tau$-artinian}\index{abelian
category!artinian@$\tau$-artinian} if there exist no infinite chains
of subobjects $\cdots\!\subset\!A_2\!\subset\! A_1\!\subset\!X$ in
$\A$ with $A_{n+1}\!\ne\!A_n$ and $\tau([A_{n+1}])\!\ge\!
\tau([A_n/A_{n+1}])$ for all $n$.
\label{dt3def3}
\end{dfn}

In \cite[Th.~4.4]{Joyc5} we prove the existence of {\it
Harder--Narasimhan filtrations}.\index{Harder--Narasimhan filtration}

\begin{prop} Let\/ $(\tau,T,\le)$ be a weak stability condition
on an abelian category $\A$. Suppose $\A$ is
noetherian\index{noetherian}\index{abelian category!noetherian} and\/
$\tau$-artinian. Then each\/ $E\in\A$ admits a unique filtration
$0\!=\!E_0\!\subset\!\cdots\!\subset\!E_n\!=\!E$ for $n\ge 0$, such
that\/ $S_k\!=\!E_k/E_{k-1}$ is $\tau$-semistable for
$k=1,\ldots,n$, and\/~$\tau([S_1])>\tau([S_2])>\cdots>\tau([S_n])$.
\label{dt3prop2}
\end{prop}

We define {\it permissible\/} (weak) stability conditions, a
condition needed to get well-behaved invariants `counting'
$\tau$-(semi)stable objects in~\cite{Joyc6}.\index{stability
condition!permissible}

\begin{dfn} Suppose Assumption \ref{dt3ass} holds for
$\K,\A,K(\A)$, so that the moduli stack $\fM_\A$ of objects in $\A$
is an Artin $\K$-stack, with substacks $\fM^\al_\A$ for $\al\in
C(\A)$. Suppose too that $\A$ is {\it noetherian}. Let
$(\tau,T,\le)$ be a weak stability condition on $\A$. For $\al\in
C(\A)$, write $\fM_\rss^\al(\tau),\fM_\st^\al
(\tau)$\nomenclature[Mass]{$\fM^\al_\rss(\tau)$}{moduli stack of
$\tau$-semistable objects in class
$\al$}\nomenclature[Mast]{$\fM^\al_\st(\tau)$}{moduli stack of
$\tau$-stable objects in class $\al$} for the moduli substacks of
$\tau$-(semi)stable $E\in\A$ with class $[E]=\al$ in $K(\A)$. As in
\cite[\S 4.2]{Joyc5}, $\fM_\rss^\al(\tau),\fM_\st^\al (\tau)$ are
{\it open\/} $\K$-substacks of $\fM^\al_\A$. We call $(\tau,T,\le)$
{\it permissible} if:
\begin{itemize}
\setlength{\itemsep}{0pt}
\setlength{\parsep}{0pt}
\item[(a)] $\A$ is $\tau$-artinian, in the sense of Definition
\ref{dt3def3}, and
\item[(b)] $\fM_\rss^\al(\tau)$ is a {\it finite type\/} substack of
$\fM^\al_\A$ for all~$\al\in C(\A)$.
\end{itemize}
\label{dt3def4}
\end{dfn}

Here (b) is necessary if `counting' $\tau$-(semi)stables in class
$\al$ is to yield a finite answer. We will be interested in two
classes of examples of permissible (weak) stability conditions on
coherent sheaves, Gieseker stability and $\mu$-stability.

\begin{ex} Let $\K$ be an algebraically closed field, $X$ a smooth
projective $\K$-scheme of dimension $m$, and $\A=\coh(X)$ the
coherent sheaves\index{coherent sheaf} on $X$. Then \cite[\S 9]{Joyc3}
defines data satisfying Assumption \ref{dt3ass}, with
$K(\A)=K^\num(\coh(X))$. It is a finite rank lattice, that
is,~$K(\A)\cong\Z^l$.

Define $G$ to be the set of monic rational polynomials in~$t$:
\begin{equation*}
G=\bigl\{p(t)=t^d+a_{d-1}t^{d-1}+\cdots+a_0:d=0,1,\ldots,\;\>
a_0,\ldots,a_{d-1}\in\Q\bigr\}.
\end{equation*}
Define a total order `$\le$' on $G$ by $p\le p'$ for $p,p'\in G$ if
either
\begin{itemize}
\setlength{\itemsep}{0pt}
\setlength{\parsep}{0pt}
\item[(a)] $\deg p>\deg p'$, or
\item[(b)] $\deg p=\deg p'$ and $p(t)\le p'(t)$ for all $t\gg 0$.
\end{itemize}
We write $p<q$ if $p\le q$ and $p\ne q$. Note that $\deg p>\deg p'$
in (a) implies that $p(t)>p'(t)$ for all $t\gg 0$, which is the
opposite to $p(t)\le p'(t)$ for $t\gg 0$ in (b), and not what you
might expect, but it is necessary for Definition \ref{dt3def3} to
hold. The effect of (a) is that $\tau$-semistable sheaves are
automatically {\it pure}, because if $0\ne S\subset E$ with $\dim
S<\dim E$ then $S$ destabilizes~$E$.

Fix a very ample line bundle
$\cO_X(1)$\nomenclature[OX(1)]{$\cO_X(1)$}{very ample line bundle on
a scheme $X$} on $X$. For $E\in\coh(X)$, the {\it Hilbert
polynomial\/}\index{Hilbert polynomial} $P_E$ is the unique
polynomial in $\Q[t]$ such that $P_E(n)=\dim H^0(E(n))$ for all
$n\gg 0$. Equivalently, $P_E(n)=\bar\chi\bigl([\cO_X(-n)],[E]\bigr)$
for all $n\in\Z$. Thus, $P_E$ depends only on the class $\al\in
K^\num(\coh(X))$ of $E$, and we may write $P_\al$ instead of $P_E$.
Define $\tau:C(\coh(X))\ra G$ by $\tau(\al)=P_\al/r_\al$, where
$P_\al$ is the Hilbert polynomial of $\al$, and $r_\al$ is the
leading coefficient of $P_\al$, which must be positive. Then as in
\cite[Ex.~4.16]{Joyc5}, $(\tau,G,\le)$\nomenclature[\tau
G]{$(\tau,G,\le)$}{the stability condition of Gieseker stability on
$\coh(X)$} is a {\it permissible stability
condition\/}\index{stability condition!permissible}\index{stability
condition!on coherent sheaves} on $\coh(X)$. It is called {\it
Gieseker stability},\index{Gieseker stability}\index{stability
condition!Gieseker} and $\tau$-(semi)stable sheaves are called {\it
Gieseker (semi)stable}. Gieseker stability is studied in~\cite[\S
1.2]{HuLe2}.

For the case of Gieseker stability, as well as the moduli stacks
$\fM_\rss^\al(\tau),\fM_\st^\al(\tau)$ of $\tau$-(semi)stable
sheaves $E$ with class $[E]=\al$, later we will also use the
notation
$\M_\rss^\al(\tau),\M_\st^\al(\tau)$\nomenclature[Mass]{$\M_\rss^\al(\tau)$}{coarse
moduli scheme of $\tau$-semistable objects in class
$\al$}\nomenclature[Mast]{$\M_\st^\al(\tau)$}{coarse moduli scheme
of $\tau$-stable objects in class $\al$} for the {\it coarse moduli
schemes\/}\index{coarse moduli scheme}\index{moduli scheme!coarse}
of $\tau$-(semi)stable sheaves $E$ with class $[E]=\al$ in
$K^\num(\coh(X))$. By \cite[Th.~4.3.4]{HuLe2}, $\M_\rss^\al(\tau)$
is a projective $\K$-scheme whose $\K$-points correspond to
S-equivalence\index{S-equivalence} classes of Gieseker semistable
sheaves in class $\al$, and $\M_\st^\al(\tau)$ is an open
$\K$-subscheme whose $\K$-points correspond to isomorphism classes
of Gieseker stable sheaves.
\label{dt3ex1}
\end{ex}

\begin{ex} In the situation of Example \ref{dt3ex1}, define
\begin{equation*}
M=\bigl\{p(t)=t^d+a_{d-1}t^{d-1}:d=0,1,\ldots,\;\> a_{d-1}\in\Q\;\>
a_{-1}=0\bigr\}\subset G
\end{equation*}
and restrict the total order $\le$ on $G$ to $M$. Define
$\mu:C(\coh(X))\ra M$ by $\mu(\al)=t^d+a_{d-1}t^{d-1}$ when
$\tau(\al)=P_\al/r_\al=t^d+a_{d-1}t^{d-1}+\cdots+a_0$, that is,
$\mu(\al)$ is the truncation of the polynomial $\tau(\al)$ in
Example \ref{dt3ex1} at its second term. Then as in
\cite[Ex.~4.17]{Joyc5}, $(\mu,M,\le)$\nomenclature[\mu]{$(\mu,M,\le)$}{the weak
stability condition of $\mu$-stability on $\coh(X)$} is a {\it
permissible weak stability condition\/} on $\coh(X)$. It is called
$\mu$-{\it stability},\index{$\mu$-stability}\index{stability
condition!$\mu$-stability}\index{m-stability@$\mu$-stability}\index{stability
condition!on coherent sheaves} and is studied in~\cite[\S
1.6]{HuLe2}.
\label{dt3ex2}
\end{ex}

In \cite[\S 8]{Joyc5} we define interesting stack functions
$\bdss^\al(\tau),\bep^\al(\tau)$ in $\SFa(\fM_\A)$.

\begin{dfn} Let $\K,\A,K(\A)$ satisfy Assumption \ref{dt3ass}, and
$(\tau,T,\le)$ be a permissible weak stability condition on $\A$.
Define stack functions\index{stack function}
$\bdss^\al(\tau)=\bde_{\fM_\rss^\al(\tau)}$\nomenclature[\delta]{$\bdss^\al(\tau)$}{element
of the Ringel--Hall algebra $\SFa(\fM)$ that `counts'
$\tau$-semistable objects in class $\al$} in $\SFa(\fM_\A)$ for
$\al\in C(\A)$. That is, $\bdss^\al(\tau)$ is the characteristic
function, in the sense of Definition \ref{dt2def5}, of the moduli
substack $\fM_\rss^\al(\tau)$ of $\tau$-semistable sheaves in
$\fM_\A$. In \cite[Def.~8.1]{Joyc5} we define elements
$\bep^\al(\tau)$ in $\SFa(\fM_\A)$
by\nomenclature[\epsilon]{$\bar\ep^\al(\tau)$}{element of the
Ringel--Hall Lie algebra $\SFai(\fM)$ that `counts'
$\tau$-semistable objects in class $\al$}
\ea
\bep^\al(\tau)&= \!\!\!\!\!\!\!
\sum_{\begin{subarray}{l}n\ge 1,\;\al_1,\ldots,\al_n\in C(\A):\\
\al_1+\cdots+\al_n=\al,\; \tau(\al_i)=\tau(\al),\text{ all
$i$}\end{subarray}} \!\!\!\!\!\!
\frac{(-1)^{n-1}}{n}\,\,\bdss^{\al_1}(\tau)*\bdss^{\al_2}(\tau)
*\cdots*\bdss^{\al_n}(\tau),
\label{dt3eq4}
\intertext{where $*$ is the Ringel--Hall multiplication in
$\SFa(\fM_\A)$. Then \cite[Th.~8.2]{Joyc5} proves} \bdss^\al(\tau)&=
\!\!\!\!\!\!\!
\sum_{\begin{subarray}{l}n\ge 1,\;\al_1,\ldots,\al_n\in C(\A):\\
\al_1+\cdots+\al_n=\al,\; \tau(\al_i)=\tau(\al),\text{ all
$i$}\end{subarray}} \!\!\!
\frac{1}{n!}\,\,\bep^{\al_1}(\tau)*\bep^{\al_2}(\tau)*
\cdots*\bep^{\al_n}(\tau).
\label{dt3eq5}
\ea
There are only finitely many nonzero terms in
\eq{dt3eq4}--\eq{dt3eq5}, because as the family of $\tau$-semistable
sheaves in class $\al$ is bounded, there are only finitely ways to
write $\al=\al_1+\cdots+\al_n$ with $\tau$-semistable sheaves in
class $\al_i$ for all~$i$.
\label{dt3def5}
\end{dfn}

Here is a way to interpret \eq{dt3eq4} and \eq{dt3eq5} informally in
terms of $\log$ and $\exp$: working in a completed version
$\widehat{\SFa}(\fM_\A)$ of the algebra $\SFa(\fM_\A)$, so that
appropriate classes of infinite sums make sense, for fixed $t\in T$
we have
\ea
\sum_{\al\in C(\A):\tau(\al)=t}\bep^\al(\tau)&=\log
\raisebox{-5pt}{\begin{Large}$\displaystyle\Bigl[$\end{Large}}
\bde_0+\sum_{\al\in C(\A):\tau(\al)=t}\bde_\rss^\al(\tau)
\raisebox{-5pt}{\begin{Large}$\displaystyle\Bigr]$\end{Large}},
\label{dt3eq6}\\
\bde_0+\sum_{\al\in C(\A):\tau(\al)=t}\bde_\rss^\al(\tau)&=\exp
\raisebox{-5pt}{\begin{Large}$\displaystyle\Bigl[$\end{Large}}
\sum_{\al\in C(\A):\tau(\al)=t}\bep^\al(\tau)
\raisebox{-5pt}{\begin{Large}$\displaystyle\Bigr]$\end{Large}},
\label{dt3eq7}
\ea
where $\bde_0$ is the identity 1 in $\widehat{\SFa}(\fM_\A)$. For
$\al\in C(\A)$ and $t=\tau(\al)$, using the power series
$\log(1+x)=\sum_{n\ge 1}\frac{(-1)^{n-1}}{n}\,x^n$ and
$\exp(x)=1+\sum_{n\ge 1}\frac{1}{n!}\,x^n$ we see that
\eq{dt3eq4}--\eq{dt3eq5} are the restrictions of
\eq{dt3eq6}--\eq{dt3eq7} to $\fM_\A^\al$. This makes clear why
\eq{dt3eq4} and \eq{dt3eq5} are inverse, since $\log$ and $\exp$ are
inverse. Thus, knowing the $\bep^\al(\tau)$ is equivalent to knowing
the $\bdss^\al(\tau)$.

If $\fM_\rss^\al(\tau)=\fM_\st^\al(\tau)$ then
$\bep^\al(\tau)=\bdss^\al(\tau)$. The difference between
$\bep^\al(\tau)$ and $\bdss^\al(\tau)$ is that $\bep^\al(\tau)$
`counts' strictly semistable sheaves in a special, complicated way.
Here \cite[Th.~8.7]{Joyc5} is an important property of the
$\bep^\al(\tau)$, which does not hold for the $\bdss^\al(\tau)$. The
proof is highly nontrivial, using the full power of the
configurations formalism of~\cite{Joyc3,Joyc4,Joyc5,Joyc6}.

\begin{thm} $\bep^\al(\tau)$ lies in the Lie subalgebra
$\SFai(\fM_\A)$ in~$\SFa(\fM_\A)$.
\label{dt3thm1}
\end{thm}

\subsection{Changing stability conditions and algebra identities}
\label{dt33}

In \cite{Joyc6} we prove transformation laws for the
$\bdss^\al(\tau),\bep^\al(\tau)$ under change of stability
condition. These involve combinatorial coefficients
$S(*;\tau,\ti\tau)\in\Z$ and $U(*;\tau, \ti\tau)\in\Q$ defined in
\cite[\S 4.1]{Joyc6}. We have changed some notation
from~\cite{Joyc6}.

\begin{dfn} Let $\A,K(\A)$ satisfy Assumption \ref{dt3ass}, and
$(\tau,T,\le),(\ti\tau,\ti T,\le)$ be weak stability conditions on
$\A$. We say that $(\ti\tau,\ti T,\le)$ {\it dominates}
$(\tau,T,\le)$ if $\tau(\al)\le\tau(\be)$ implies
$\ti\tau(\al)\le\ti\tau(\be)$ for all~$\al,\be\in
C(\A)$.\index{stability condition!$(\tilde\tau,\tilde T,\leqslant)$
dominates $(\tau,T,\leqslant)$}

Let $n\ge 1$ and $\al_1,\ldots,\al_n\in C(\A)$. If for all
$i=1,\ldots,n-1$ we have either
\begin{itemize}
\setlength{\itemsep}{0pt}
\setlength{\parsep}{0pt}
\item[(a)] $\tau(\al_i)\le\tau(\al_{i+1})$ and
$\ti\tau(\al_1+\cdots+\al_i)>\ti\tau(\al_{i+1}+\cdots+\al_n)$ or
\item[(b)] $\tau(\al_i)>\tau(\al_{i+1})$ and~
$\ti\tau(\al_1+\cdots+\al_i)\le\ti\tau(\al_{i+1}+\cdots+\al_n)$,
\end{itemize}
then define $S(\al_1,\ldots,\al_n;\tau,\ti\tau)=(-1)^r$,
\nomenclature[S(\alpha)]{$S(\al_1,\ldots,\al_n;\tau,\ti\tau)$}{combinatorial
coefficient used in wall-crossing formulae} where $r$ is the number
of $i=1,\ldots,n-1$ satisfying (a). Otherwise define
$S(\al_1,\ldots,\al_n;\tau,\ti\tau)=0$. Now
define\nomenclature[U(\alpha)]{$U(\al_1,\ldots,\al_n;\tau,\ti\tau)$}{combinatorial
coefficient used in wall-crossing formulae}
\ea
&U(\al_1,\ldots,\al_n;\tau,\ti\tau)=
\label{dt3eq8}\\
&\sum_{\begin{subarray}{l} \phantom{wiggle}\\
1\le l\le m\le n,\;\> 0=a_0<a_1<\cdots<a_m=n,\;\>
0=b_0<b_1<\cdots<b_l=m:\\
\text{Define $\be_1,\ldots,\be_m\in C(\A)$ by
$\be_i=\al_{a_{i-1}+1}+\cdots+\al_{a_i}$.}\\
\text{Define $\ga_1,\ldots,\ga_l\in C(\A)$ by
$\ga_i=\be_{b_{i-1}+1}+\cdots+\be_{b_i}$.}\\
\text{Then $\tau(\be_i)=\tau(\al_j)$, $i=1,\ldots,m$,
$a_{i-1}<j\le a_i$,}\\
\text{and $\ti\tau(\ga_i)=\ti\tau(\al_1+\cdots+\al_n)$,
$i=1,\ldots,l$}
\end{subarray}
\!\!\!\!\!\!\!\!\!\!\!\!\!\!\!\!\!\!\!\!\!\!\!\!\!\!\!\!\!\!\!\!\!
\!\!\!\!\!\!\!\!\!\!\!\!\!\!\!\!\!\!\!\!\!\!\!\!\!\!\!\!\!\!\!\!\!
\!\!\!\!\!\!\!\!\!\!\!\!\!\!\!\!\!\!\!\!}
\begin{aligned}[t]
\frac{(-1)^{l-1}}{l}\cdot\prod\nolimits_{i=1}^lS(\be_{b_{i-1}+1},
\be_{b_{i-1}+2},\ldots,\be_{b_i}; \tau,\ti\tau)&\\
\cdot\prod_{i=1}^m\frac{1}{(a_i-a_{i-1})!}&\,.
\end{aligned}
\nonumber
\ea
\label{dt3def6}
\end{dfn}

Then in \cite[\S 5]{Joyc6} we derive wall-crossing
formulae\index{wall-crossing formula} for the
$\bdss^\al(\tau),\bep^\al(\tau)$ under change of stability condition
from $(\tau,T,\le)$ to~$(\ti\tau,\ti T,\le)$:

\begin{thm} Let Assumption {\rm\ref{dt3ass}} hold, and\/ $(\tau,T,\le),
(\ti\tau,\ti T,\le),(\hat\tau,\hat T,\le)$ be permissible weak
stability conditions on $\A$ with\/ $(\hat\tau,\hat T,\le)$
dominating $(\tau,T,\le)$ and\/ $(\ti\tau,\ti T,\le)$. Then for
all\/ $\al\in C(\A)$ we have
\begin{gather}
\begin{gathered}
\bdss^\al(\ti\tau)= \!\!\!\!\!\!\!
\sum_{\begin{subarray}{l}n\ge 1,\;\al_1,\ldots,\al_n\in
C(\A):\\ \al_1+\cdots+\al_n=\al\end{subarray}} \!\!\!\!\!\!\!
\begin{aligned}[t]
S(\al_1,&\ldots,\al_n;\tau,\ti\tau)\cdot\\
&\bdss^{\al_1}(\tau)*\bdss^{\al_2}(\tau)*\cdots*
\bdss^{\al_n}(\tau),
\end{aligned}
\end{gathered}
\label{dt3eq9}\\
\begin{gathered}
\bep^\al(\ti\tau)= \!\!\!\!\!\!\!
\sum_{\begin{subarray}{l}n\ge 1,\;\al_1,\ldots,\al_n\in
C(\A):\\ \al_1+\cdots+\al_n=\al\end{subarray}} \!\!\!\!\!\!\!
\begin{aligned}[t]
U(\al_1,&\ldots,\al_n;\tau,\ti\tau)\cdot\\
&\bep^{\al_1}(\tau)*\bep^{\al_2}(\tau)*\cdots* \bep^{\al_n}(\tau),
\end{aligned}
\end{gathered}
\label{dt3eq10}
\end{gather}
where there are only finitely many nonzero terms in
\eq{dt3eq9}--\eq{dt3eq10}.
\label{dt3thm2}
\end{thm}

Here the third stability condition $(\hat\tau,\hat T,\le)$ may be
thought of as lying on a `wall' separating $(\tau,T,\le)$ and
$(\ti\tau,\ti T,\le)$ in the space of stability conditions. Here is
how to prove \eq{dt3eq9}--\eq{dt3eq10}. If $E\in\A$ then by
Proposition \ref{dt3prop2} there is a unique
Harder--Narasimhan\index{Harder--Narasimhan filtration} filtration
$0\!=\!E_0\!\subset\!\cdots\!\subset\!E_n\!=\!E$ with
$S_k=E_k/E_{k-1}$ $\tau$-semistable and
$\tau([S_1])>\cdots>\tau([S_n])$. As $\hat\tau$ dominates $\tau$,
one can show $E$ is $\hat\tau$-semistable if and only if
$\hat\tau([S_1])=\cdots=\hat\tau([S_n])$. It easily follows that
\e
\bdss^\al(\hat\tau)= \!\!\!\!\!\!\!
\sum_{\begin{subarray}{l}n\ge 1,\;\al_1,\ldots,\al_n\in
C(\A):\al_1+\cdots+\al_n=\al,\\
\tau(\al_1)>\cdots>\tau(\al_n),\;
\hat\tau(\al_1)=\cdots=\hat\tau(\al_n)\end{subarray}}
\!\!\!\!\!\!\!\!\!\! \bdss^{\al_1}(\tau)*\bdss^{\al_2}(\tau)*\cdots*
\bdss^{\al_n}(\tau).
\label{dt3eq11}
\e

By a similar argument to \eq{dt3eq4}--\eq{dt3eq7} but using the
inverse functions $x\mapsto x/(1-x)$ and $x\mapsto x/(1+x)$ rather
than $\log,\exp$, we find the inverse of \eq{dt3eq11} is
\e
\bdss^\al(\tau)= \!\!\!\!\!\!\!
\sum_{\begin{subarray}{l}n\ge 1,\;\al_1,\ldots,\al_n\in
C(\A):\al_1+\cdots+\al_n=\al,\\
\tau(\al_1)>\cdots>\tau(\al_n),\;
\hat\tau(\al_1)=\cdots=\hat\tau(\al_n)
\end{subarray}} \!\!\!\!\!\!\!\!\!\!
(-1)^{n-1}\bdss^{\al_1}(\hat\tau)*\bdss^{\al_2}(\hat\tau)*\cdots*
\bdss^{\al_n}(\ti\tau).
\label{dt3eq12}
\e
Substituting \eq{dt3eq11} into \eq{dt3eq12} with $\ti\tau$ in place
of $\tau$ gives \eq{dt3eq9}, and \eq{dt3eq10} then follows from
\eq{dt3eq4}, \eq{dt3eq5} and \eq{dt3eq9}. From this proof we can see
that over each point of $\fM_\A$ there are only finitely many
nonzero terms in \eq{dt3eq9}--\eq{dt3eq10}, and also that every term
in \eq{dt3eq9}--\eq{dt3eq10} is supported on the open substack
$\fM_\rss^\al(\hat\tau)$ in $\fM_\A$. Since $\hat\tau$ is assumed to
be permissible, $\fM_\rss^\al(\hat\tau)$ is of finite type, and
therefore there are only finitely many nonzero terms in
\eq{dt3eq9}--\eq{dt3eq10}. In \cite[Th.~5.4]{Joyc6} we prove:

\begin{thm} Equation \eq{dt3eq10} may be rewritten as an equation in
$\SFai(\fM_\A)$ using the Lie bracket\/ $[\,,\,]$ on
$\SFai(\fM_\A),$ rather than as an equation in $\SFa(\fM_\A)$ using
the Ringel--Hall\index{Ringel--Hall algebra} product\/ $*$. That is, we
may rewrite \eq{dt3eq10} in the form
\e
\begin{gathered}
\bep^\al(\ti\tau)= \!\!\!\!\!\!\!
\sum_{\begin{subarray}{l}n\ge 1,\;\al_1,\ldots,\al_n\in
C(\A):\\ \al_1+\cdots+\al_n=\al\end{subarray}} \!\!\!\!\!\!\!
\begin{aligned}[t]
\ti U(\al_1,&\ldots,\al_n;\tau,\ti\tau)\,\cdot\\
&[[\cdots[[\bep^{\al_1}(\tau),\bep^{\al_2}(\tau)],\bep^{\al_3}(\tau)],
\ldots],\bep^{\al_n}(\tau)],
\end{aligned}
\end{gathered}
\label{dt3eq13}
\e
for some system of combinatorial coefficients\/ $\ti U(\al_1,\ldots,
\al_n;\tau,\ti\tau)\in\Q,$ with only finitely many nonzero terms.
\label{dt3thm3}
\end{thm}

There is an irritating technical problem in \cite[\S 5]{Joyc6} in
changing between stability conditions on $\coh(X)$ when $\dim X\ge
3$. Suppose $(\tau,T,\le),(\ti\tau,\ti T,\le)$ are two (weak)
stability conditions on $\coh(X)$ of Gieseker or $\mu$-stability
type, as in Examples \ref{dt3ex1} and \ref{dt3ex2}, defined using
different ample line bundles $\cO_X(1),\ti\cO_X(1)$. Then the first
author was not able to show that the changes between $(\tau,T,\le)$
and $(\ti\tau,\ti T,\le)$ are {\it globally finite}.\index{stability
condition!changes globally finite} That is, we prove
\eq{dt3eq9}--\eq{dt3eq10} hold in the local stack function
spaces\index{stack function!local} $\LSF(\fM_{\coh(X)})$, but we do
not know there are only finitely many nonzero terms in
\eq{dt3eq9}--\eq{dt3eq10}, although the first author believes this
is true. Instead, as in \cite[\S 5.1]{Joyc6}, we can show that we
can interpolate between any two stability conditions on $X$ of
Gieseker or $\mu$-stability type by a finite sequence of stability
conditions, such that between successive stability conditions in the
sequence the changes are globally finite, and Theorem \ref{dt3thm2}
applies.\index{stability condition|)}

\subsection{Calabi--Yau 3-folds and Lie algebra morphisms}
\label{dt34}
\index{Calabi--Yau 3-fold|(}

We now specialize to the case when $\A=\coh(X)$ for $X$ a
Calabi--Yau 3-fold, and explain some results of \cite[\S 6.6]{Joyc4}
and \cite[\S 6.5]{Joyc6}. We restrict to $\K$ of characteristic
zero\index{field $\K$!characteristic zero} so that Euler characteristics
over $\K$ are well-behaved.

\begin{dfn} Let $\K$ be an algebraically closed field of
characteristic zero. A {\it Calabi--Yau\/ $3$-fold\/} is a smooth
projective 3-fold $X$ over $\K$, with trivial canonical
bundle\index{canonical bundle}
$K_X$.\nomenclature[KX]{$K_X$}{canonical bundle of a smooth scheme
$X$} From \S\ref{dt5} onwards we will also assume that
$H^1(\cO_X)=0$, but this is not needed for the results of
\cite{Joyc3,Joyc4,Joyc5,Joyc6}. Take $\A$ to be $\coh(X)$ and
$K(\coh(X))$ to be $K^\num(\coh(X))$. As in Definition \ref{dt3def1}
we have the Euler form $\bar\chi:K(\coh(X))\times
K(\coh(X))\!\ra\!\Z$ in \eq{dt3eq1}. As $X$ is a Calabi--Yau 3-fold,
Serre duality\index{Serre duality} gives
$\Ext^i(F,E)\ab\cong\Ext^{3-i}(E,F)^*$, so
$\dim\Ext^i(F,E)=\dim\Ext^{3-i}(E,F)$ for all $E,F\in\coh(X)$.
Therefore $\bar\chi$ is also given by
\e
\begin{split}
\bar\chi\bigl([E],[F]\bigr)=\,&\bigl(\dim\Hom(E,F)-\dim\Ext^1(E,F)
\bigr)-\\
&\bigl(\dim\Hom(F,E)-\dim\Ext^1(F,E)\bigr).
\end{split}
\label{dt3eq14}
\e
Thus the Euler form $\bar\chi$ on $K(\coh(X))$ is {\it
antisymmetric}.

In \cite[\S 6.5]{Joyc4} we define an explicit Lie algebra $L(X)$ as
follows: $L(X)$ is the $\Q$-vector space with basis of symbols
$\la^\al$ for $\al\in K(\coh(X))$, with Lie
bracket\nomenclature[L(X)]{$L(X)$}{Lie algebra depending on a Calabi--Yau
3-fold $X$}\nomenclature[\lambda]{$\lambda^\al$}{basis element of Lie algebra
$L(X)$}
\e
[\la^\al,\la^\be]=\bar\chi(\al,\be)\la^{\al+\be},
\label{dt3eq15}
\e
for $\al,\be\in K(\coh(X))$. As $\bar\chi$ is antisymmetric,
\eq{dt3eq15} satisfies the Jacobi identity and makes $L(X)$ into an
infinite-dimensional Lie algebra over $\Q$. (We have changed
notation: in \cite{Joyc4}, $L(X),\la^\al$ are written $C^{\rm
ind}(\coh(X),\Q,\ha\bar\chi),c^\al$.)

Define a $\Q$-linear map
$\Psi^{\chi,\Q}:\oSFai(\fM_{\coh(X)},\chi,\Q)\ra L(X)$ by
\e
\Psi^{\chi,\Q}(f)=\ts\sum_{\al\in K(\coh(X))}\ga^\al \la^{\al},
\label{dt3eq16}
\e
where $\ga^\al\in\Q$ is defined as follows. Proposition
\ref{dt3prop1} says $\oSFai(\fM_{\coh(X)},\chi,\Q)$ is spanned by
elements $[(U\times[\Spec\K/\bG_m],\rho)]$. We may write
\e
f\vert_{\fM_{\coh(X)}^\al}=\ts\sum_{i=1}^n\de_i[(U_i\times[\Spec\K/
\bG_m],\rho_i)],
\label{dt3eq17}
\e
where $\de_i\in\Q$ and $U_i$ is a quasiprojective $\K$-variety. We
set
\e
\ga^\al=\ts\sum_{i=1}^n\de_i\chi(U_i).
\label{dt3eq18}
\e
This is independent of the choices in \eq{dt3eq17}. Now define
$\Psi:\SFai(\fM_{\coh(X)})\ra L(X)$ by~$\Psi=\Psi^{\chi,\Q}\ci\bar
\Pi^{\chi,\Q}_{\fM_{\coh(X)}}$.\nomenclature[\Psi]{$\Psi$}{Lie algebra morphism
$\SFai(\fM)\ra L(X)$}\nomenclature[\Psi a \chi]{$\Psi^{\chi,\Q}$}{Lie algebra
morphism $\oSFai(\fM,\chi,\Q)\ra L(X)$}
\label{dt3def7}
\end{dfn}

In \cite[Th.~6.12]{Joyc4}, using equation \eq{dt3eq14}, we prove:

\begin{thm} $\Psi:\SFai(\fM_{\coh(X)})\ra L(X)$ and\/ $\Psi^{\chi,\Q}:
\oSFai(\fM_{\coh(X)},\chi,\Q)\ab\ra L(X)$ are Lie algebra morphisms.
\label{dt3thm4}
\end{thm}

Our next example may help readers to understand why this is true.

\begin{ex} Suppose $E,F$ are simple sheaves\index{sheaf!simple} on $X$,
with $[E]=\al$ and $[F]=\be$ in $K(\coh(X))$. Consider the stack
functions\index{stack function} $\bde_E,\bde_F$ in
$\oSFa(\fM_{\coh(X)},\chi,\Q)$, the characteristic functions of the
points $E,F$ in $\fM_{\coh(X)}$. Since $\Aut(E)=\Aut(F)=\bG_m$ as
$E,F$ are simple, we may write
\e
\bde_E=\bigl[([\Spec\K/\bG_m],e)\bigr]\quad\text{and}\quad
\bde_F=\bigl[([\Spec\K/\bG_m],f)\bigr],
\label{dt3eq19}
\e
where the 1-morphisms $e,f:[\Spec\K/\bG_m]\ra \fM_{\coh(X)}$
correspond to $E,F$. Thus $\bde_E,\bde_F$ have virtual rank 1, and
lie in $\oSFai(\fM_{\coh(X)},\chi,\Q)$. We will prove explicitly
that $\Psi^{\chi,\Q}\bigl([\bde_E,\bde_F]\bigr)=\bigl[\Psi^{\chi,\Q}
(\bde_E),\Psi^{\chi,\Q}(\bde_F)\bigr]$, as we need for
$\Psi^{\chi,\Q}$ to be a Lie algebra morphism. From \eq{dt3eq19} and
$[E]=\al,[F]=\be$ we have $\Psi^{\chi,\Q} (\bde_E)=\la^\al$ and
$\Psi^{\chi,\Q}(\bde_F)=\la^\be$, so that
$\bigl[\Psi^{\chi,\Q}(\bde_E),\Psi^{\chi,\Q}(\bde_F)\bigr]
=\bar\chi(\al,\be)\la^{\al+\be}$ by~\eq{dt3eq15}.

Now in $\oSFai(\fM_{\coh(X)},\chi,\Q)$ we have
\ea
\bde_E&*\bde_F=\bigl[([\Spec\K/\bG_m],e)\bigr]*
\bigl[([\Spec\K/\bG_m],f)\bigr]
\nonumber\\
&=\bigl[\bigl([\Spec\K/\bG_m^2]\times_{e\times f,\fM_{\coh(X)}\times
\fM_{\coh(X)},\pi_1\times\pi_3}\fExact_{\coh(X)},\pi_2\ci
\pi_{\fExact_{\coh(X)}}\bigr)\bigr]
\nonumber\\
&=\bigl[\bigl([\Ext^1(F,E)/(\bG_m^2\lt\Hom(F,E))],\rho_1\bigr)\bigr]
\nonumber\\
&=\bigl[\bigl([\Spec\K/(\bG_m^2\lt\Hom(F,E))],\rho_2\bigr)\bigr]
\label{dt3eq20}\\
&\qquad+\bigl[\bigl({\mathbb P}(\Ext^1(F,E))
\times[\Spec\K/(\bG_m\times\Hom(F,E))],\rho_3\bigr)\bigr]
\nonumber\\
&=\bigl[\bigl([\Spec\K/\bG_m^2],\rho_4\bigr)\bigr]-\dim\Hom(F,E)
\bigl[\bigl([\Spec\K/\bG_m],\rho_5\bigr)\bigr]
\nonumber\\
&\qquad+\bigl[\bigl({\mathbb P}(\Ext^1(F,E))
\times[\Spec\K/\bG_m],\rho_6\bigr)\bigr], \nonumber
\ea
where $\rho_i$ are 1-morphisms to $\fM_{\coh(X)}^{\al+\be}$, and the
group law on $\bG_m^2\lt\Hom(F,E)$ is $(\la,\mu,\phi)\cdot
(\la',\mu',\phi')\!=\!(\la\la',\mu\mu',\la\phi'\!+\!\mu'\phi)$ for
$\la,\la',\mu,\mu'$ in $\bG_m$ and $\phi,\phi'$ in $\Hom(F,E)$, and
$\bG_m^2\lt\Hom(F,E)$ acts on $\Ext^1(F,E)$
by~$(\la,\mu,\phi):\ep\mapsto\la\mu^{-1}\ep$.

Here in the first step of \eq{dt3eq20} we use \eq{dt3eq19}, in the
second \eq{dt3eq3}, and in the third that the fibre of
$\pi_1\times\pi_3$ over $E,F$ is $[\Ext^1(F,E)/\Hom(F,E)]$. In the
fourth step of \eq{dt3eq20} we use relation Definition
\ref{dt2def10}(i) in $\oSFai(\fM_{\coh(X)},\chi,\Q)$ to cut
$[\Ext^1(F,E)/\bG_m^2\lt\Hom(F,E)]$ into two pieces
$[\{0\}/(\bG_m^2\lt\Hom(F,E))]$ and
$[(\Ext^1(F,E)\!\sm\!\{0\})/(\bG_m^2\lt\Hom(F,E))]$, where for the
second $\bG_m^2\lt\Hom(F,E)$ acts by dilation on
$\Ext^1(F,E)\!\sm\!\{0\}$, turning it into ${\mathbb
P}(\Ext^1(F,E))$, and the stabilizer of each point is
$\bG_m\!\times\!\Hom(F,E)$. In the fifth step of \eq{dt3eq20} we use
relation Definition \ref{dt2def10}(iii) to rewrite in terms of
quotients by tori $\bG_m,\bG_m^2$; the term in $\dim\Hom(F,E)$ is
there as the coefficient
$F\bigl(\bG_m^2\lt\Hom(F,E),\bG_m^2,\bG_m\bigr)$ in \eq{dt2eq10} is
$-\dim\Hom(F,E)$ (see equation \eq{dt11eq13} in \S\ref{dt11} for
this computation).

In the same way we show that
\e
\begin{split}
\bde_F*\bde_E=&=\bigl[\bigl([\Spec\K/\bG_m^2],\rho_4\bigr)\bigr]
-\dim\Hom(E,F)\bigl[\bigl([\Spec\K/\bG_m],\rho_5\bigr)\bigr]\\
&\qquad+\bigl[\bigl({\mathbb P}(\Ext^1(E,F))
\times[\Spec\K/\bG_m],\rho_7\bigr)\bigr],
\end{split}
\label{dt3eq21}
\e
where the terms $\bigl[\bigl([\Spec\K/\bG_m^2],\rho_4\bigr)\bigr]$
and $\bigl[\bigl([\Spec\K/\bG_m],\rho_5\bigr)\bigr]$ in
\eq{dt3eq20}--\eq{dt3eq21} are the same, mapping to $E\op F$. So
subtracting \eq{dt3eq21} from \eq{dt3eq20} yields
\begin{align*}
&[\bde_E,\bde_F]=\bigl(\dim\Hom(E,F)-\dim\Hom(F,E)\bigr)\bigl[
\bigl([\Spec\K/\bG_m],\rho_5\bigr)\bigr]\\
&+\!\bigl[\bigl({\mathbb P}(\Ext^1(F,E))
\!\times\![\Spec\K/\bG_m],\rho_6\bigr)\bigr]
\!-\!\bigl[\bigl({\mathbb P}(\Ext^1(E,F))
\!\times\![\Spec\K/\bG_m],\rho_7\bigr)\bigr].
\end{align*}
Applying $\Psi^{\chi,\Q}$ thus yields
\begin{align*}
&\Psi^{\chi,\Q}\bigl([\bde_E,\bde_F]\bigr)\\
&=\bigl(\dim\Hom(E,F)\!-\!\dim\Hom(F,E)\!+\!\dim\Ext^1(F,E)
\!-\!\dim\Ext^1(E,F)\bigr)\la^{\al+\be}\\
&=\bar\chi\bigl([E],[F]\bigr)\la^{\al+\be}=
\bar\chi(\al,\be)\la^{\al+\be}=
\bigl[\Psi^{\chi,\Q}(\bde_E),\Psi^{\chi,\Q}(\bde_F)\bigr],
\end{align*}
by equation \eq{dt3eq14} and $\chi\bigl({\mathbb
P}(\Ext^1(E,F))\bigr)=\dim\Ext^1(E,F)$.
\label{dt3ex3}
\end{ex}\index{abelian category|)}

\subsection{Invariants $J^\al(\tau)$ and transformation laws}
\label{dt35}

We continue in the situation of \S\ref{dt34}, with $\K$ of
characteristic zero and $X$ a Calabi--Yau 3-fold over $\K$. Let
$(\tau,T,\le)$ be a permissible weak stability
condition\index{stability condition!permissible} on $\coh(X)$, for
instance, Gieseker stability or $\mu$-stability w.r.t.\ some ample
line bundle $\cO_X(1)$ on $X$, as in Example \ref{dt3ex1} or
\ref{dt3ex2}. In \cite[\S 6.6]{Joyc6} we define invariants
$J^\al(\tau)\in\Q$ for all $\al\in C(\coh(X))$
by\nomenclature[Ja(t)]{$J^\al(\tau)$}{invariant counting
$\tau$-semistable sheaves in class $\al$ on a Calabi--Yau 3-fold,
introduced in \cite{Joyc6}}
\e
\Psi\bigl(\bep^\al(\tau)\bigr)=J^\al(\tau)\la^\al.
\label{dt3eq22}
\e
This is valid by Theorem \ref{dt3thm1}. These $J^\al(\tau)$ are
rational numbers `counting' $\tau$-semistable sheaves $E$ in class
$\al$. When $\M_\rss^\al(\tau)=\M_\st^\al(\tau)$ we have
$J^\al(\tau)=\chi(\M_\st^\al(\tau))$, that is, $J^\al(\tau)$ is the
Euler characteristic of the moduli space $\M_\st^\al(\tau)$. As we
explain in \S\ref{dt4}, this is {\it not\/} weighted by the Behrend
function $\nu_{\M_\st^\al(\tau)}$, and is not the Donaldson--Thomas
invariant $DT^\al(\tau)$. Also, the $J^\al(\tau)$ are in general
{\it not\/} unchanged under deformations of $X$, as we show in
Example \ref{dt6ex8} below.

Now suppose $(\tau,T,\le), (\ti\tau,\ti T,\le),(\hat\tau,\hat
T,\le)$ are as in Theorem \ref{dt3thm2}, so that equation
\eq{dt3eq10} holds, and is equivalent to a Lie algebra equation
\eq{dt3eq13} as in Theorem \ref{dt3thm3}. Therefore we may apply the
Lie algebra morphism $\Psi$ to equation \eq{dt3eq13}. In fact we
prefer to work with equation \eq{dt3eq10}, since the coefficients
$\ti U(\al_1,\ldots,\al_n;\tau,\ti\tau)$ in \eq{dt3eq13} are
difficult to write down. So we express it as an equation in the {\it
universal enveloping algebra\/}\index{universal enveloping algebra}
$U(L(X))$.\nomenclature[ULX]{$U(L(X))$}{universal enveloping algebra of Lie
algebra $L(X)$} This gives
\e
\begin{gathered}
J^\al(\ti\tau)\la^\al= \!\!\!\!\!\!\!
\sum_{\begin{subarray}{l}n\ge 1,\;\al_1,\ldots,\al_n\in
C(\coh(X)):\\
\al_1+\cdots+\al_n=\al\end{subarray}\!\!\!\!\!\!\!\!\!\!\!\!\!\!\!
\!\!\!\!\!\!} \!\!\!\!\!\!\!\!
\begin{aligned}[t]
U(\al_1,\ldots,\al_n;\tau,\ti\tau)\,\cdot\, &\ts\prod_{i=1}^n
J^{\al_i}(\tau)\,\cdot\\
& \la^{\al_1}\star\la^{\al_2}\star\cdots\star\la^{\al_n},
\end{aligned}
\end{gathered}
\label{dt3eq23}
\e
where $\star$ is the product in~$U(L(X))$.

Now in \cite[\S 6.5]{Joyc4}, an explicit description is given of the
universal enveloping algebra $U(L(X))$ (the notation used for
$U(L(X))$ in \cite{Joyc4} is $C(\coh(X),\Q,\ha\bar\chi)$). There is
an explicit basis given for $U(L(X))$ in terms of symbols
$\la_{[I,\ka]}$, and multiplication $\star$ in $U(L(X))$ is given in
terms of the $\la_{[I,\ka]}$ as a sum over graphs. Here $I$ is a
finite set, $\ka$ maps $I\ra C(\coh(X))$, and when $\md{I}=1$, so
that $I=\{i\}$, we have $\la_{[I,\ka]}=\la^{\ka(i)}$. Then
\cite[eq.~(127)]{Joyc6} gives an expression for
$\la^{\al_1}\star\cdots\star \la^{\al_n}$ in $U(L(X))$, in terms of
sums over {\it directed graphs\/} ({\it digraphs\/}):\index{digraph}
\begin{gather}
\la^{\al_1}\star\cdots\star\la^{\al_n}=\text{ terms in
$\la_{[I,\ka]}$, $\md{I}>1$, }
\label{dt3eq24}\\
+\raisebox{-6pt}{\begin{Large}$\displaystyle\biggl[$\end{Large}}
\frac{1}{2^{n-1}}\!\!\!\!\!
\sum_{\substack{\text{connected, simply-connected digraphs
$\Ga$:}\\
\text{vertices $\{1,\ldots,n\}$, edge $\mathop{\bu} \limits^{\sst
i}\ra\mathop{\bu}\limits^{\sst j}$ implies $i<j$}}} \,\,\,
\prod_{\substack{\text{edges}\\
\text{$\mathop{\bu}\limits^{\sst i}\ra\mathop{\bu}\limits^{\sst
j}$}\\ \text{in $\Ga$}}}\bar\chi(\al_i,\al_j)
\raisebox{-6pt}{\begin{Large}$\displaystyle\biggr]$\end{Large}}
\la^{\al_1+\cdots+\al_n}. \nonumber
\end{gather}

Substitute \eq{dt3eq24} into \eq{dt3eq23}. The terms in
$\la_{[I,\ka]}$ for $\md{I}>1$ all cancel, as \eq{dt3eq23} lies in
$L(X)\subset U(L(X))$. So equating coefficients of $\la^\al$ yields
\e
\begin{gathered}
J^\al(\ti\tau)=\!\!\!\!\!\!
\sum_{\begin{subarray}{l}n\ge 1,\;\al_1,\ldots,\al_n\in
C(\coh(X)):\\ \al_1+\cdots+\al_n=\al\end{subarray}}\,\,\,\,
\sum_{\begin{subarray}{l}\text{connected, simply-connected digraphs $\Ga$:}\\
\text{vertices $\{1,\ldots,n\}$, edge $\mathop{\bu} \limits^{\sst
i}\ra\mathop{\bu}\limits^{\sst j}$ implies $i<j$}\end{subarray}}\\
\frac{1}{2^{n-1}}\, U(\al_1,\ldots,\al_n;\tau,\ti\tau) \!\!\!\!\!
\prod_{\text{edges $\mathop{\bu}\limits^{\sst
i}\ra\mathop{\bu}\limits^{\sst j}$ in $\Ga$}}\!\!\!\!\!
\bar\chi(\al_i,\al_j) \prod_{i=1}^nJ^{\al_i}(\tau).
\end{gathered}
\label{dt3eq25}
\e

Following \cite[Def.~6.27]{Joyc6}, we define combinatorial
coefficients~$V(I,\Ga,\ka;\tau,\ti\tau)$:

\begin{dfn} In the situation above, suppose $\Ga$ is a connected,
simply-connected digraph with finite vertex set $I$, where
$\md{I}=n$, and $\ka:I\ra C(\coh(X))$ is a map. Define
$V(I,\Ga,\ka;\tau,\ti\tau)\in\Q$
by\index{digraph}\nomenclature[VI\Gamma]{$V(I,\Ga,\ka;\tau,\ti\tau)$}{combinatorial
coefficient used in wall-crossing formulae}
\e
V(I,\Ga,\ka;\tau,\ti\tau)=\frac{1}{2^{n-1}n!}
\!\!\sum_{\substack{\text{orderings $i_1,\ldots,i_n$ of $I$:}\\
\text{edge $\mathop{\bu} \limits^{\sst
i_a}\ra\mathop{\bu}\limits^{\sst i_b}$ in $\Ga$ implies $a<b$}}
\!\!\!\!\!\!\!\!\!\!\!\!\!\!\!\!\!\!\!\!\!\!\!\!\!\!\!\!\!\!\!\!\!\!\!
} \!\!\!\!\! U(\ka(i_1),\ka(i_2),\ldots,\ka(i_n);\tau,\ti\tau).
\label{dt3eq26}
\e
\label{dt3def8}
\end{dfn}

Then as in \cite[Th.~6.28]{Joyc6}, using \eq{dt3eq26} to rewrite
\eq{dt3eq25} yields a transformation law for the $J^\al(\tau)$ under
change of stability condition:
\e
\begin{gathered}
J^\al(\ti\tau)\!=\!\!\!\!
\sum_{\substack{\text{iso.}\\ \text{classes}\\
\text{of finite}\\ \text{sets $I$}}}\,\,
\sum_{\substack{\ka:I\ra C(\coh(X)):\\ \sum_{i\in I}\ka(i)=\al}}\,\,
\sum_{\begin{subarray}{l} \text{connected,}\\
\text{simply-connected}\\ \text{digraphs $\Ga$,}\\
\text{vertices $I$}\end{subarray}\!\!\!\!\!\!\!\!\!\!\!\!\!\!\!\!\!
\!\!\!\!\!\!\!\!\!\!} V(I,\Ga,\ka;\tau,\ti\tau)
\begin{aligned}[t]
&\cdot\prod\limits_{\text{edges \smash{$\mathop{\bu}\limits^{\sst
i}\ra\mathop{\bu}\limits^{\sst j}$} in
$\Ga$}\!\!\!\!\!\!\!\!\!\!\!\!\!\!\!\!\!\!\!\!\!\!\!
\!\!\!\!\!\!\!\!\!\!\!\!\!\!} \bar\chi(\ka(i),\ka(j))\\
&\cdot\prod\nolimits_{i\in I}J^{\ka(i)}(\tau).
\end{aligned}
\end{gathered}
\label{dt3eq27}
\e
As in \cite[Rem.~6.29]{Joyc6}, $V(I,\Ga,\ka;\tau,\ti\tau)$ depends
on the orientation of $\Ga$ only up to sign: changing the directions
of $k$ edges multiplies $V(I,\Ga,\ka;\tau,\ti\tau)$ by $(-1)^k$.
Since $\bar\chi$ is antisymmetric, it follows that
$V(I,\Ga,\ka;\tau,\ti\tau)\cdot\prod_{\smash{\mathop{\bu}\limits^{\sst
i}\ra\mathop{\bu}\limits^{\sst j}}}\bar\chi(\ka(i),\ka(j))$ in
\eq{dt3eq27} is independent of the orientation
of~$\Ga$.\index{Calabi--Yau 3-fold|)}

\section[Behrend functions and Donaldson--Thomas theory]{Behrend
functions and Donaldson--Thomas \\ theory}
\label{dt4}\index{Behrend function|(}

We now discuss {\it Behrend functions\/} of schemes and stacks, and
their application to Donaldson--Thomas invariants. Our primary
source is Behrend's paper \cite{Behr}. But Behrend considers only
$\C$-schemes and Deligne--Mumford $\C$-stacks, whereas we treat
Artin stacks, and discuss which parts of the theory work over other
algebraically closed fields $\K$. Some of our results, such as
Theorem \ref{dt4thm4} below, appear to be new. Also, in \S\ref{dt45}
we give an exact cohomological description of the numerical
Grothendieck group $K^\num(\coh(X))$ of a Calabi--Yau 3-fold~$X$.

We have not tried to be brief; instead, we have tried to make
\S\ref{dt41}--\S\ref{dt44} a helpful reference on Behrend functions,
by collecting ideas and material which may be useful in the future.
Section \ref{dt44}, and most of \S\ref{dt42}, will not be used in
this book. We include in \S\ref{dt42} a discussion of {\it perverse
sheaves\/} and {\it vanishing cycles}, since they seem to be
connected to Behrend functions at a deep level, but we expect many
of our readers may not be familiar with them.

\subsection{The definition of Behrend functions}
\label{dt41}

\begin{dfn} Let $\K$ be an algebraically closed field of
characteristic zero, and $X$ a finite type $\K$-scheme. Write
$Z_*(X)$\nomenclature[Z(X)]{$Z_*(X)$}{group of algebraic cycles on a scheme
$X$} for the group of {\it algebraic cycles\/} on $X$, as in Fulton
\cite{Fult}. Suppose $X\hookra M$ is an embedding of $X$ as a closed
subscheme of a smooth $\K$-scheme $M$. Let $C_XM$ be the {\it normal
cone\/} of $X$ in $M$, as in \cite[p.~73]{Fult}, and $\pi:C_XM\ra X$
the projection. As in \cite[\S 1.1]{Behr}, define a cycle
${\mathfrak c}_{X/M}\in Z_*(X)$ by
\begin{equation*}
{\mathfrak c}_{X/M}=\ts\sum_{C'}(-1)^{\dim\pi(C')}{\rm
mult}(C')\pi(C'),
\end{equation*}
where the sum is over all irreducible components $C'$ of $C_XM$.

It turns out that ${\mathfrak c}_{X/M}$ depends only on $X$, and not
on the embedding $X\hookra M$. Behrend \cite[Prop.~1.1]{Behr} proves
that given a finite type $\K$-scheme $X$, there exists a unique
cycle ${\mathfrak c}_X\in Z_*(X)$, such that for any \'etale map
$\vp:U\ra X$ for a $\K$-scheme $U$ and any closed embedding
$U\hookra M$ into a smooth $\K$-scheme $M$, we have
$\vp^*({\mathfrak c}_X)={\mathfrak c}_{U/M}$ in $Z_*(U)$. If $X$ is
a subscheme of a smooth $M$ we take $U=X$ and get ${\mathfrak
c}_X={\mathfrak c}_{X/M}$. Behrend calls ${\mathfrak c}_X$ the {\it
signed support of the intrinsic normal cone}, or the {\it
distinguished cycle} of~$X$.

Write $\CF_\Z(X)$ for the group of $\Z$-valued constructible
functions\index{constructible function} on $X$. The {\it local Euler
obstruction\/} is a group isomorphism
$\Eu:Z_*(X)\ra\CF_\Z(X)$.\nomenclature[Eu]{$\Eu$}{the `local Euler
obstruction', an isomorphism $Z_*(X)\ra\CF_\Z(X)$} It was first
defined by MacPherson \cite{MacP} when $\K=\C$, using complex
analysis, but Kennedy \cite{Kenn} provides an alternative algebraic
definition which works over any algebraically closed field $\K$ of
characteristic zero. If $V$ is a prime cycle on $X$, the
constructible function $\Eu(V)$ is given by
\begin{equation*}
\ts\Eu(V):x\longmapsto\int_{\mu^{-1}(x)}c(\ti T)\cap
s(\mu^{-1}(x),\ti V),
\end{equation*}
where $\mu:\ti V\ra V$ is the Nash blowup of $V$, $\ti T$ the dual
of the universal quotient bundle, $c$ the total Chern class and $s$
the Segre class of the normal cone to a closed immersion. Kennedy
\cite[Lem.~4]{Kenn} proves that $\Eu(V)$ is
constructible.\index{constructible function} For each finite type
$\K$-scheme $X$, define the {\it Behrend function} $\nu_X$ in
$\CF(X)$ by $\nu_X=\Eu({\mathfrak c}_X)$, as in Behrend~\cite[\S
1.2]{Behr}.\nomenclature[\nu]{$\nu_X$}{Behrend function of a scheme or stack
$X$}\index{Behrend function!definition}
\label{dt4def1}
\end{dfn}

In the case $\K=\C$, using MacPherson's complex analytic definition
of the local Euler obstruction \cite{MacP}, the definition of
$\nu_X$ makes sense in the framework of complex analytic geometry,
and so Behrend functions can be defined for {\it complex analytic
spaces\/} $X_\an$.\index{complex analytic space} Informally, we have a
commutative diagram:\index{Behrend function!algebraic|(}\index{Behrend
function!analytic|(}
\begin{equation*}
\xymatrix@C=10pt@R=10pt{
{\begin{subarray}{l} \ts\text{$\C$-subvariety $X$ in} \\
\ts\text{smooth $\C$-variety $M$} \end{subarray}}
\ar[r] \ar[d] &
{\begin{subarray}{l} \ts \text{Algebraic cycle} \\
\ts \text{${\mathfrak c}_{X/M}$ in $Z_*(X)$}
\end{subarray}} \ar[r]_(0.37){\Eu} \ar[d] & {\begin{subarray}{l}
\ts\text{Algebraic Behrend function}\\
\ts\text{$\nu_X=\Eu({\mathfrak c}_{X/M})$ in $\CF_\Z(X)$}
\end{subarray}} \ar[d] \\
{\begin{subarray}{l} \ts\text{complex analytic space $\!\ti X$} \\
\ts\text{in complex manifold $\ti M$} \end{subarray}}
\ar[r] &
{\begin{subarray}{l} \ts \text{Analytic cycle} \\ \ts
\text{${\mathfrak c}_{\ti X/\ti M}\!$ in $Z_*^\an(\ti X)$}
\end{subarray}}  \ar[r]^(0.37){\Eu} &
{\begin{subarray}{l}
\ts\text{Analytic Behrend function}\\
\ts\text{$\nu_{\ti X}\!=\!\Eu({\mathfrak c}_{\ti X/\ti M})$
in $\CF_\Z^\an(\ti X)$,}
\end{subarray}}}
\end{equation*}
where the columns pass from $\C$-algebraic
varieties/cycles/constructible functions to the underlying complex
analytic spaces/cycles/constructible functions. Thus we deduce:

\begin{prop}{\bf(a)} If\/ $\K$ is an algebraically closed field of
characteristic zero, and\/ $X$ is a finite type $\K$-scheme, then
the Behrend function $\nu_X$ is a well-defined\/ $\Z$-valued
constructible function on $X,$ in the Zariski topology.\index{Zariski
topology}

\noindent{\bf(b)} If\/ $Y$ is a complex analytic space then the
Behrend function $\nu_Y$ is a well-defined\/ $\Z$-valued locally
constructible function on $Y,$ in the analytic topology.

\noindent{\bf(c)} If $X$ is a finite type $\C$-scheme, with
underlying complex analytic space $X_\an,$ then the algebraic
Behrend function $\nu_X$ in {\bf(a)} and the analytic Behrend
function $\nu_{\smash{X_\an}}$ in {\bf(b)} coincide. In particular,
$\nu_X$ depends only on the complex analytic space $X_\an$
underlying $X,$ locally in the analytic topology.
\label{dt4prop1}
\end{prop}\index{Behrend function!algebraic|)}\index{Behrend
function!analytic|)}

Here are some important properties of Behrend functions. They are
proved by Behrend \cite[\S 1.2 \& Prop.~1.5]{Behr} when $\K=\C$, but
his proof is valid for general~$\K$.

\begin{thm} Let\/ $\K$ be an algebraically closed field of
characteristic zero, and\/ $X,Y$ be finite type $\K$-schemes. Then:
\begin{itemize}
\setlength{\itemsep}{0pt}
\setlength{\parsep}{0pt}
\item[{\rm(i)}] If\/ $X$ is smooth of dimension\/ $n$
then~$\nu_X\equiv(-1)^n$.
\item[{\rm(ii)}] If\/ $\vp:X\!\ra\! Y$ is smooth with
relative dimension $n$ then\/~$\nu_X\!\equiv\!(-1)^n\vp^*(\nu_Y)$.
\item[{\rm(iii)}] $\nu_{X\times Y}\equiv\nu_X\boxdot\nu_Y,$
where $(\nu_X\boxdot\nu_Y)(x,y)=\nu_X(x)\nu_Y(y)$.
\end{itemize}
\label{dt4thm1}
\end{thm}

We can extend the definition of Behrend functions to $\K$-schemes,
algebraic $\K$-spaces, and Artin $\K$-stacks, locally of finite
type.\index{Behrend function!of Artin stack}

\begin{prop} Let\/ $\K$ be an algebraically closed field of
characteristic zero, and\/ $X$ be a $\K$-scheme, algebraic
$\K$-space, or Artin $\K$-stack, locally of finite type. Then there
is a well-defined \begin{bfseries}Behrend function\end{bfseries}
$\nu_X,$ a $\Z$-valued locally constructible
function\index{constructible function!locally} on $X,$ which is
characterized uniquely by the property that if\/ $W$ is a finite
type $\K$-scheme and $\vp:W\ra X$ is a $1$-morphism of Artin stacks
that is smooth of relative dimension $n$ then
$\vp^*(\nu_X)=(-1)^n\nu_W$ in~$\CF(W)$.
\label{dt4prop2}
\end{prop}

\begin{proof} As Artin $\K$-stacks include $\K$-schemes and
algebraic $\K$-spaces, it is enough to do the Artin stack case.
Suppose $X$ is an Artin $\K$-stack, locally of finite type. Let
$x\in X(\K)$. Then by the existence of atlases for $X$, and as $X$
is locally of finite type, there exists a finite type $\K$-scheme
$W$ and a 1-morphism $\vp:W\ra X$ smooth of relative dimension $n$,
with $x=\vp_*(w)$ for some $w\in W(\K)$. We wish to
define~$\nu_X(x)=(-1)^n\nu_W(w)$.

To show this is well-defined, suppose $W',\vp',n',w'$ are
alternative choices for $W,\vp,n,w$. Consider the fibre product
$Y=W\times_{\vp,X,\vp'}W'$. This is a finite type $\K$-scheme, as
$W,W'$ are. Let $\pi_1:Y\ra W$ and $\pi_2:Y\ra W'$ be the
projections to the factors of the fibre product. Then $\pi_1,\pi_2$
are morphisms of $\K$-schemes, and $\pi_1$ is smooth of relative
dimension $n'$ as $\vp'$ is, and $\pi_2$ is smooth of relative
dimension $n$ as $\vp$ is. Hence Theorem \ref{dt4thm1}(ii) gives
\e
(-1)^{n'}\pi_1^*(\nu_W)\equiv\nu_Y\equiv(-1)^n\pi_2^*(\nu_{W'}).
\label{dt4eq1}
\e
Since $\vp_*(w)=x=\vp'_*(w')$, the fibre of $\pi_1\times\pi_2:Y\ra
W\times W'$ over $(w,w')$ is isomorphic as a $\K$-scheme to the
stabilizer group $\Iso_X(x)$, and so is nonempty. Thus there exists
$y\in Y(\K)$ with $(\pi_1)_*(y)=w$ and $(\pi_2)_*(y)=w'$. Equation
\eq{dt4eq1} thus gives $(-1)^{n'}\nu_W(w)=
\nu_Y(y)=(-1)^n\nu_{W'}(w')$, so that $(-1)^n\nu_W(w)=
(-1)^{n'}\nu_{W'}(w')$. Hence $\nu_X(x)$ is well-defined.

Therefore there exists a unique function $\nu_X:X(\K)\ra\Z$ with the
property in the proposition. It remains only to show that $\nu_X$ is
locally constructible. For $\vp,W,n$ as above,
$\vp^*(\nu_X)=(-1)^n\nu_W$ and $\nu_W$ constructible imply that
$\nu_X$ is constructible on the constructible set
$\vp_*(W(\K))\subseteq X(\K)$. But any constructible subset $S$ of
$X(\K)$ can be covered by finitely many such subsets $\vp_*(W(\K))$,
so $\nu_X\vert_S$ is constructible, and thus $\nu_X$ is locally
constructible.
\end{proof}

It is then easy to deduce:

\begin{cor} Theorem {\rm\ref{dt4thm1}} also holds for Artin\/
$\K$-stacks $X,Y$ locally of finite type.
\label{dt4cor1}
\end{cor}

\subsection{Milnor fibres and vanishing cycles}
\label{dt42}\index{Milnor fibre|(}

We define {\it Milnor fibres\/} for holomorphic functions on complex
analytic spaces.

\begin{dfn} Let $U$ be a complex analytic space, $f:U\ra\C$ a
holomorphic function, and $x\in U$. Let $d(\,,\,)$ be a metric on
$U$ near $x$ induced by a local embedding of $U$ in some $\C^N$. For
$\de,\ep>0$, consider the holomorphic map
\begin{equation*}
\Phi_{f,x}:\bigl\{y\in U:d(x,y)\!<\!\de,\;
0\!<\!\md{f(y)\!-\!f(x)}\!<\!\ep\bigr\}\longra
\bigl\{z\in\C:0\!<\!\md{z}\!<\!\ep\bigr\}
\end{equation*}
given by $\Phi_{f,x}(y)=f(y)-f(x)$. Milnor \cite{Miln}, extended by
L\^e \cite{Le}, shows that $\Phi_{f,x}$ is a locally trivial
topological fibration provided $0<\ep\ll\de\ll 1$. The {\it Milnor
fibre\/} $MF_f(x)$\nomenclature[MFf(x)]{$MF_f(x)$}{Milnor fibre of a
holomorphic function $f$ at point $x$} is the fibre of $\Phi_{f,x}$.
It is independent of the choice of $0<\ep\ll\de\ll 1$ up to
homeomorphism, or up to diffeomorphism for smooth~$U$.
\label{dt4def2}
\end{dfn}

The next theorem is due to Parusi\'nski and Pragacz \cite{PaPr}, as
in~\cite[\S 1.2]{Behr}.

\begin{thm} Let\/ $U$ be a complex manifold and\/ $f:U\ra\C$ a
holomorphic function, and define $X$ to be the complex analytic
space $\Crit(f)\subseteq U$. Then the
Behrend function $\nu_X$ of\/ $X$ is given by
\e
\nu_X(x)=(-1)^{\dim U}\bigl(1-\chi(MF_f(x))\bigr) \qquad\text{for
$x\in X$.}
\label{dt4eq2}
\e
\label{dt4thm2}
\end{thm}

These ideas on Milnor fibres have a deep and powerful generalization
in the theory of {\it perverse sheaves\/}\index{perverse sheaf|(} and
{\it vanishing cycles}.\index{vanishing cycle} We now sketch a few of
the basics of the theory. It works both in the algebraic and complex
analytic contexts, but we will explain only the complex analytic
setting. A survey paper on the subject is Massey \cite{Mass}, and
three books are Kashiwara and Schapira \cite{KaSc}, Dimca
\cite{Dimc}, and Sch\"urmann \cite{Schur}. Over the field $\C$,
Saito's theory of {\it mixed Hodge modules\/}\index{mixed Hodge module}
\cite{Sait} provides a generalization of the theory of perverse
sheaves with more structure, which may also be a context in which to
generalize Donaldson--Thomas theory, but we will not discuss this.

What follows will not be needed to understand the rest of the book
--- the only result in this discussion we will use later is Theorem
\ref{dt4thm4}, which makes sense using only the definitions of
\S\ref{dt41}. We include this material both for completeness, as it
underlies the theory of Behrend functions, and also to point out to
readers in Donaldson--Thomas theory that future developments in the
subject, particularly in the direction of motivic Donaldson--Thomas
invariants and motivic Milnor fibres envisaged by Kontsevich and
Soibelman \cite{KoSo1}, will probably be framed in terms of perverse
sheaves and vanishing cycles.

\begin{dfn} Let $X$ be a complex analytic space. Consider sheaves of
$\Q$-modules $\cal C$ on $X$. Note that these are {\it not\/}
coherent sheaves, which are sheaves of $\cO_X$-modules. A sheaf
$\cal C$ is called {\it
constructible\/}\index{sheaf!constructible}\index{constructible
sheaf} if there is a locally finite stratification $X=\bigcup_{j\in
J}X_j$ of $X$ in the complex analytic topology, such that ${\cal
C}\vert_{X_j}$ is a $\Q$-local system for all $j\in J$, and all the
stalks ${\cal C}_x$ for $x\in X$ are finite-dimensional $\Q$-vector
spaces. A complex ${\cal C}^\bu$ of sheaves of $\Q$-modules on $X$
is called {\it constructible\/} if all its cohomology sheaves
$H^i({\cal C}^\bu)$ for $i\in\Z$ are constructible.

Write $D^b_\Con(X)$\nomenclature[DbCon(X)]{$D^b_\Con(X)$}{bounded
derived category of constructible complexes on $X$} for the bounded
derived category of constructible complexes on $X$. It is a
triangulated category. By \cite[Th.~4.1.5]{Dimc}, $D^b_\Con(X)$ is
closed under Grothendieck's ``six operations on
sheaves''\index{sheaf!Grothendieck's six operations}
$R\vp_*,R\vp_!,\vp^*,\vp^!,{\cal
RH}om,\smash{\mathop{\otimes}\limits^{\sst L}}$. The {\it perverse
sheaves\/} on $X$ are a particular abelian subcategory
$\Per(X)$\nomenclature[Per(X)]{$\Per(X)$}{abelian category of
perverse sheaves on $X$} in $D^b_\Con(X)$, which is the heart of a
t-structure on $D^b_\Con(X)$. So perverse sheaves are actually
complexes of sheaves, not sheaves, on $X$. The category $\Per(X)$ is
noetherian\index{noetherian}\index{abelian category!noetherian} and
locally artinian, and is artinian\index{artinian}\index{abelian
category!artinian} if $X$ is of finite type, so every perverse sheaf
has (locally) a unique filtration whose quotients are simple
perverse sheaves; and the simple perverse sheaves can be described
completely in terms of irreducible local systems on irreducible
subvarieties in~$X$.
\label{dt4def3}
\end{dfn}

Next we explain {\it nearby cycles\/} and {\it vanishing
cycles}.\index{vanishing cycle} Let $X$ be a complex analytic space, and
$f:X\ra\C$ a holomorphic function. Define $X_0=f^{-1}(0)$, as a
complex analytic space, and $X^*=X\sm X_0$. Consider the commutative
diagram
\begin{equation*}
\xymatrix@R=10pt@C=30pt{ X_0 \ar[r]_i \ar[d]^f  & X \ar[d]^f  & X^*
\ar[l]^j \ar[d]^f  & \widetilde{X^*} \ar[l]^p \ar@/_.7pc/[ll]_\pi
\ar[d]^{\ti f} \\
\{0\} \ar[r] & \C  & \C^* \ar[l]  & \widetilde{\C^*}. \ar[l]_\rho}
\end{equation*}
Here $i:X_0\ra X$, $j:X^*\ra X$ are the inclusions,
$\rho:\widetilde{\C^*}\ra\C^*$ is the universal cover of
$\C^*=\C\sm\{0\}$, and $\widetilde{X^*}=X^*\times_{f,\C^*,\rho}
\widetilde{\C^*}$ the corresponding cover of $X^*$, with covering
map $p:\widetilde{X^*}\ra X^*$, and $\pi=j\ci p$. The {\it nearby
cycle functor\/}\index{nearby cycle functor} $\psi_f:D^b_\Con(X)\ra
D^b_\Con(X_0)$ is $\psi_f=i^*R\pi_*\pi^*$.\nomenclature[\psi
f]{$\psi_f$}{nearby cycle functor on derived category of
constructible sheaves}

There is a natural transformation $\Xi:i^* \Rightarrow \psi_f$
between the functors $i^*,\psi_f:D^b_\Con(X)\ra D^b_\Con(X_0)$. The
{\it vanishing cycle functor\/}\index{vanishing cycle!functor}\nomenclature[\phi
f]{$\phi_f$}{vanishing \kern-.1em cycle \kern-.1em functor
\kern-.1em on \kern-.1em derived \kern-.1em category \kern-.1em of
\kern-.1em constructible \kern-.1em sheaves} $\phi_f:D^b_\Con(X)\ra
D^b_\Con(X_0)$ is a functor such that for every ${\cal C}^\bu$ in
$D^b_\Con(X)$ we have a distinguished triangle
\e
\smash{\xymatrix@C=40pt{i^*({\cal C}^\bu) \ar[r]^{\Xi({\cal C}^\bu)}
& \psi_f({\cal C}^\bu) \ar[r] & \phi_f({\cal C}^\bu) \ar[r]^{[+1]} &
i^*({\cal C}^\bu)}}
\label{dt4eq3}
\e
in $D^b_\Con(X_0)$. So roughly speaking $\phi_f$ is the cone on
$\Xi$, but this is not a good definition as cones are not unique up
to canonical isomorphism. The shifted functors
$\psi_f[-1],\phi_f[-1]$ take perverse sheaves to perverse sheaves.

As $i^*,\psi_f,\phi_f$ are exact, they induce morphisms on the
Grothendieck groups
\begin{equation*}
(i^*)_*,(\psi_f)_*,(\phi_f)_*:K_0(D^b_\Con(X))\longra
K_0(D^b_\Con(X_0)),
\end{equation*}
with $(\psi_f)_*=(i^*)_*+(\phi_f)_*$ by \eq{dt4eq3}. Note that
$K_0(D^b_\Con(X))=K_0(\Per(X))$ and $K_0(D^b_\Con(X_0))=
K_0(\Per(X_0))$, and for $X$ of finite type $K_0(\Per(X))$ is
spanned by isomorphism classes of simple perverse
sheaves,\index{perverse sheaf!simple} which have a nice
description~\cite[Th.~5.2.12]{Dimc}.

Write $\CF_\Z^\an(X)$\nomenclature[CFZan(X)]{$\CF_\Z^\an(X)$}{group of
$\Z$-valued analytically constructible functions on $X$} for the
group of $\Z$-valued analytically constructible
functions\index{constructible function!analytic} on $X$. Define a map
$\chi_X: \Obj(D^b_\Con(X))\ra\CF_\Z^\an(X)$ by taking Euler
characteristics of the cohomology of stalks of complexes, given by
\begin{equation*}
\chi_X({\cal C}^\bu):x\longmapsto \ts\sum_{k\in\Z}(-1)^k\dim{\cal
H}^k({\cal C}^\bu)_x.
\end{equation*}
Since distinguished triangles in $D^b_\Con(X)$ give long exact
sequences on cohomology of stalks ${\cal H}^k(-)_x$, this $\chi_X$
is additive over distinguished triangles, and so descends to a group
morphism $\chi_X:K_0(D^b_\Con(X))\ra \CF_\Z^\an(X)$.

These maps $\chi_X:\Obj(D^b_\Con(X))\ra\CF_\Z^\an(X)$ and $\chi_X:
K_0(D^b_\Con(X))\ra \CF_\Z^\an(X)$ are surjective, since
$\CF_\Z^\an(X)$ is spanned by the characteristic functions of closed
analytic cycles $Y$ in $X$, and each such $Y$ lifts to a perverse
sheaf in $D^b_\Con(X)$. In category-theoretic terms, $X\mapsto
D^b_\Con(X)$ is a functor $D^b_\Con$ from complex analytic spaces to
triangulated categories, and $X\mapsto \CF_\Z^\an(X)$ is a functor
$\CF_\Z^\an$ from complex analytic spaces to abelian groups, and
$X\mapsto\chi_X$ is a natural transformation $\chi$ from $D^b_\Con$
to~$\CF_\Z^\an$.

As in Sch\"urmann \cite[\S 2.3]{Schur}, the operations
$R\vp_*,R\vp_!,\vp^*,\vp^!,{\cal RH}om$, and $\mathop
{\otimes}\limits^{\sst L}$  on $D^b_\Con(X)$ all have analogues on
constructible functions, which commute with the maps $\chi_X$. So,
for example, if $\vp:X\ra Y$ is a morphism of complex analytic
spaces, pullback of complexes $\vp^*$ corresponds to pullback of
constructible functions in \S\ref{dt21}, that is, we have a
commutative diagram
\begin{equation*}
\xymatrix@R=10pt@C=50pt{ D^b_\Con(Y) \ar[d]^{\chi_Y} \ar[r]_{\vp^*}
& D^b_\Con(X) \ar[d]_{\chi_X} \\ \CF_\Z^\an(Y) \ar[r]^{\vp^*} &
\CF_\Z^\an(X).}
\end{equation*}
Similarly, if $\vp$ is {\it proper\/} then $R\vp_*$ on complexes
corresponds to pushforward of constructible functions $\CF(\vp)$ in
\S\ref{dt21}, that is, we have a commutative diagram
\e
\begin{gathered}
\xymatrix@R=10pt@C=50pt{ D^b_\Con(X) \ar[d]^{\chi_X} \ar[r]_{R\vp_*}
& D^b_\Con(Y) \ar[d]_{\chi_Y} \\ \CF_\Z^\an(X) \ar[r]^{\CF(\vp)} &
\CF_\Z^\an(Y).}
\label{dt4eq4}
\end{gathered}
\e
Also $\smash{\mathop{\otimes}\limits^{\sst L}}$ corresponds to
multiplication of constructible functions.

The functors $\psi_f,\phi_f$ above have analogues $\Psi_f,\Phi_f$ on
constructible functions defined by Verdier \cite[Prop.s 3.4 \&
4.1]{Verd}. For $X,f,X_0$ as above, there is a unique morphism
$\Psi_f:\CF_\Z^\an(X)\ra \CF_\Z^\an(X_0)$\nomenclature[\Psi z]{$\Psi_f$}{nearby
cycle functor on constructible functions} such that
\e
\Psi_f(1_Z):x\longmapsto \begin{cases}
\chi\bigl(MF_{f\vert_Z}(x)\bigr), & x\in X_0\cap Z, \\ 0, & x\in
X_0\sm Z,
\end{cases}
\label{dt4eq5}
\e
whenever $Z$ is a closed complex analytic subspace of $X$, and
$1_Z\in\CF_\Z^\an(X)$ is given by $1_Z(x)=1$ if $x\in Z$ and
$1_Z(x)=0$ if $x\notin Z$. We set $\Phi_f=\Psi_f-i^*$,\nomenclature[\Phi
f]{$\Phi_f$}{vanishing cycle functor on constructible functions}
where $i:X_0\ra X$ is the inclusion. Then we have commutative
diagrams
\e
\begin{gathered}
\xymatrix@R=10pt@C=20pt{ D^b_\Con(X) \ar[d]^{\chi_X}
\ar[rr]_{\psi_f} && D^b_\Con(X_0) \ar[d]_{\chi_{X_0}} & D^b_\Con(X)
\ar[d]^{\chi_X} \ar[rr]_{\phi_f} && D^b_\Con(X_0) \ar[d]_{\chi_{X_0}} \\
\CF_\Z^\an(X) \ar[rr]^{\Psi_f} && \CF_\Z^\an(X_0), & \CF_\Z^\an(X)
\ar[rr]^{\Phi_f} && \CF_\Z^\an(X_0).}
\end{gathered}
\label{dt4eq6}
\e

Now let $U$ be a complex manifold of dimension $n$, and $f:U\ra\C$ a
holomorphic function. The critical locus $X=\Crit(f)$ is a complex
analytic subspace of $U$, and $f$ is locally constant on $X$, so
locally $X\subseteq f^{-1}(c)$ for some $c\in\C$. Suppose $X$ is
contained in $f^{-1}(0)=U_0$. Write $\underline{\Q}$ for the
constant sheaf with fibre $\Q$ on $U$, regarded as an element of
$D^b_\Con(U)$. As $U$ is smooth of dimension $n$, the shift
$\underline{\Q}[n]$ is a simple perverse sheaf on $U$. Since
$\psi_f[-1],\phi_f[-1]$ take perverse sheaves to perverse sheaves,
it follows that $\psi_f[-1]\bigl(\underline{\Q}[n]\bigr)=
\psi_f\bigl(\underline{\Q}[n-1]\bigr)$ and
$\phi_f[-1]\bigl(\underline{\Q}[n]\bigr)=
\phi_f\bigl(\underline{\Q}[n-1]\bigr)$ are perverse sheaves on
$U_0$. We call these the {\it perverse sheaves of nearby cycles and
vanishing cycles},\index{vanishing cycle!perverse sheaf}\index{perverse
sheaf!of vanishing cycles} respectively.

We will compute $\chi_{U_0}\bigl(\phi_f(\underline{\Q}[n-1])\bigr)$.
We have
\begin{align*}
\chi_{U_0}\bigl(\phi_f(\underline{\Q}[n-1])\bigr)&\equiv
\bigl(\Phi_f\ci\chi_U(\underline{\Q}[n-1])\bigr)\equiv
(-1)^{n-1}\bigl(\Phi_f\ci\chi_U(\underline{\Q})\bigr)\\
&\equiv (-1)^{n-1}\bigl(\Phi_f(1_U)\bigr)\equiv
(-1)^{n-1}\bigl(\Psi_f(1_U)-i^*(1_U)\bigr)\\
&\equiv (-1)^{n-1}\bigl(\Psi_f(1_U)-1_{U_0}\bigr) \equiv
(-1)^n\bigl(1_{U_0}-\Psi_f(1_U)\bigr),
\end{align*}
using \eq{dt4eq6} commutative in the first step,
$\chi_U\ci[+1]=-\chi_U$ in the second, $\chi_U(\underline{\Q})=1_U$
in the third and $\Phi_f=\Psi_f-i^*$ in the fourth. So \eq{dt4eq5}
gives
\e
\chi_{U_0}\bigl(\phi_f(\underline{\Q}[n-1])\bigr):x\longmapsto
(-1)^n\bigl(1-\chi(MF_f(x))\bigr)\qquad\text{for $x\in U_0$.}
\label{dt4eq7}
\e
If $x\in U_0\sm X$ then $MF_f(x)$ is an open ball, so
$\chi_{U_0}\bigl(\phi_f(\underline{\Q}[n-1])\bigr)(x)=0$ by
\eq{dt4eq7}, and if $x\in X$ then
$\chi_{U_0}\bigl(\phi_f(\underline{\Q}[n-1])\bigr)(x)= \nu_X(x)$ by
\eq{dt4eq7} and Theorem \ref{dt4thm2}. Thus we have proved:

\begin{thm} Let\/ $U$ be a complex manifold of dimension $n,$ and\/
$f:U\ra\C$ a holomorphic function with\/ $X=\Crit(f)$ contained in
$U_0=f^{-1}(\{0\})$. Then the perverse sheaf of vanishing
cycles\index{vanishing cycle!perverse sheaf}\index{perverse sheaf!of
vanishing cycles} $\phi_f(\underline{\Q}[n-1])$ on $U_0$ is
supported on $X,$ and
\e
\chi_{U_0}\bigl(\phi_f(\underline{\Q}[n-1])\bigr)(x)=\begin{cases}
\nu_X(x), & x\in X, \\ 0, & x\in U_0\sm X, \end{cases}
\label{dt4eq8}
\e
where $\nu_X$ is the Behrend function of the complex analytic
space~$X$.
\label{dt4thm3}
\end{thm}

Behrend \cite[eq.~(5)]{Behr} gives equation \eq{dt4eq8} with an
extra sign $(-1)^{n-1}$, since he omits the shift $[n-1]$ in
$\underline{\Q}[n-1]$, which makes $\phi_f(\underline{\Q}[n-1])$ a
perverse sheaf. Theorem \ref{dt4thm3} may be important for future
work in Donaldson--Thomas theory, as it suggests that we should try
to lift from constructible functions\index{constructible function} to
perverse sheaves, or mixed Hodge modules \cite{Sait}, or some
similar setting.

This bridge between perverse sheaves and vanishing cycles on one
hand, and Milnor fibres and Behrend functions on the other, is also
useful because we can take known results on the perverse sheaf side,
and translate them into properties of Milnor fibres by applying the
surjective functors $\chi_X$. Here is one such result. For
constructible complexes, the functors $\psi_f,\phi_f$ commute with
proper pushdowns \cite[Prop.~4.2.11]{Dimc}. Applying $\chi_X$
yields:

\begin{prop} Let\/ $X,Y$ be complex analytic spaces, $\vp:Y\ra X$ a
proper morphism, and\/ $f:X\ra\C$ a holomorphic function. Set\/
$g=f\ci\vp,$ and write $X_0=f^{-1}(0)$ and\/ $Y_0=g^{-1}(0)$. Then
the following diagrams commute:
\e
\begin{gathered}
\xymatrix@R=10pt@C=17pt{ \CF_\Z^\an(Y) \ar[d]^{\Psi_g}
\ar[rr]_{\CF(\vp)} && \CF_\Z^\an(X) \ar[d]_{\Psi_f} & \CF_\Z^\an(Y)
\ar[d]^{\Phi_g} \ar[rr]_{\CF(\vp)} && \CF_\Z^\an(X) \ar[d]_{\Phi_f}\\
\CF_\Z^\an(Y_0) \ar[rr]^{\CF(\vp)} && \CF_\Z^\an(X_0), &
\CF_\Z^\an(Y_0) \ar[rr]^{\CF(\vp)} && \CF_\Z^\an(X_0). }
\end{gathered}
\label{dt4eq9}
\e
\label{dt4prop3}
\end{prop}\index{perverse sheaf|)}

We use this to prove a property of Milnor fibres that we will need
later. The authors would like to thank J\"org Sch\"urmann for
suggesting the simple proof of Theorem \ref{dt4thm4} below using
Proposition \ref{dt4prop3}, which replaces a longer proof using
Lagrangian cycles in an earlier version of this book.

\begin{thm} Let\/ $U$ be a complex manifold, $f:U\ra\C$ a
holomorphic function, $V$ a closed, embedded complex submanifold
of\/ $U,$ and\/ $v\in V\cap\Crit(f)$. Define $\ti U$ to be the
blowup of\/ $U$ along\/ $V,$ with blow-up map $\pi:\ti U\ra U,$ and
set\/ $\ti f=f\ci\pi:\ti U\ra\C$. Then $\pi^{-1}(v)={\mathbb
P}(T_vU/T_vV)$ is contained in $\Crit(\ti f),$ and
\e
\begin{split}
\chi\bigl(MF_f(v)\bigr)=\,&\int_{w\in {\mathbb
P}(T_vU/T_vV)}\chi\bigl(MF_{\ti f}(w)\bigr)\,\rd\chi\\
&\quad+\bigl(1-\dim U+\dim V\bigr)\chi\bigl(MF_{f\vert_V}(v)\bigr).
\end{split}
\label{dt4eq10}
\e
Here $w\mapsto\chi(MF_{\ti f}(w))$ is a constructible
function\index{constructible function} on ${\mathbb P}(T_vU/T_vV),$ and
the integral in\/ \eq{dt4eq10} is the Euler characteristic of\/
${\mathbb P}(T_vU/T_vV)$ weighted by this.
\label{dt4thm4}
\end{thm}

\begin{proof} Let $U,V,\ti U,v$ be as in the theorem. It is
immediate that $\pi^{-1}(v)={\mathbb P}(T_vU/T_vV)\subseteq\Crit(\ti
f)$. Replacing $f$ by $f-f(v)$ if necessary, we can suppose
$f(v)=0$. Applying Proposition \ref{dt4prop3} with $U,\ti
U,\pi,f,\ti f$ in place of $X,Y,\vp,f,g$ to the function $1_{\ti U}$
on $\ti U$ shows that
\e
\CF(\pi)\ci\Psi_{\ti f}(1_{\ti U})=\Psi_f\ci\CF(\pi)1_{\ti U}.
\label{dt4eq11}
\e
We evaluate \eq{dt4eq11} at $v\in V$. Since $\pi^{-1}(v)={\mathbb
P}(T_vU/T_vV)\subset\ti V$, we have
\e
\bigl(\CF(\pi)\!\ci\!\Psi_{\ti f}(1_{\ti U})\bigr)(v)\!=\!\int_{w\in
{\mathbb P}(T_vU/T_vV)}\!\!\!\!\!\!\Psi_{\ti
f}(w)\,\rd\chi=\!\int_{w\in {\mathbb
P}(T_vU/T_vV)}\!\!\!\!\!\!\!\!\!\!\!\!\!\!\!\!\!\! \chi\bigl(MF_{\ti
f}(w)\bigr)\,\rd\chi,
\label{dt4eq12}
\e
by \eq{dt4eq5}. The fibre $\pi^{-1}(u)$ of $\pi:\ti U\ra U$ is one
point over $u\in U\sm V$, with $\chi\bigl(\pi^{-1}(u))=1$, and a
projective space ${\mathbb P}(T_uU/T_uV)$ for $u\in V$, with
$\chi\bigl(\pi^{-1}(u))=\dim U-\dim V$. It follows that
$\CF(\pi)1_{\ti U}$ is 1 at $u\in U\sm V$ and $\dim U-\dim V$ at
$u\in V$, giving
\e
\CF(\pi)1_{\ti U}=1_U+(\dim U-\dim V-1)1_V.
\label{dt4eq13}
\e
Applying $\Psi_f$ to \eq{dt4eq13} and using \eq{dt4eq5} to evaluate
it at $v$ gives
\e
\bigl(\Psi_f\ci\CF(\pi)1_{\ti U}\bigr)(v)=\chi\bigl(MF_f(v)\bigr)
+\bigl(\dim U-\dim V-1\bigr)\chi\bigl(MF_{f\vert_V}(v)\bigr).
\label{dt4eq14}
\e
Equation \eq{dt4eq10} now follows from \eq{dt4eq11}, \eq{dt4eq12}
and~\eq{dt4eq14}.
\end{proof}\index{Milnor fibre|)}

\subsection{Donaldson--Thomas invariants of Calabi--Yau 3-folds}
\label{dt43}\index{Calabi--Yau 3-fold|(}\index{Donaldson--Thomas
invariants!original $DT^\al(\tau)$|(}

{\it Donaldson--Thomas invariants\/} $DT^\al(\tau)$ were defined by
Richard Thomas \cite{Thom}, following a proposal of Donaldson and
Thomas~\cite[\S 3]{DoTh}.\nomenclature[DTa]{$DT^\al(\tau)$}{original
Donaldson--Thomas invariants defined in \cite{Thom}}

\begin{dfn} Let $\K$ be an algebraically closed field of
characteristic zero. As in \S\ref{dt34}, a {\it Calabi--Yau\/
$3$-fold\/} is a smooth projective 3-fold $X$ over $\K$, with
trivial canonical bundle\index{canonical bundle} $K_X$. Fix a very
ample line bundle $\cO_X(1)$ on $X$, and let $(\tau,G,\le)$ be
Gieseker stability on $\coh(X)$ w.r.t.\ $\cO_X(1)$, as in Example
\ref{dt3ex1}. For $\al\in K^\num(\coh(X))$, write
$\M_\rss^\al(\tau),\M_\st^\al(\tau)$ for the coarse moduli
schemes\index{coarse moduli scheme}\index{moduli scheme!coarse} of
$\tau$-(semi)stable sheaves $E$ with class $[E]=\al$. Then
$\M_\rss^\al(\tau)$ is a projective $\K$-scheme, and
$\M_\st^\al(\tau)$ an open subscheme.

Thomas \cite{Thom} constructs a symmetric obstruction
theory\index{symmetric obstruction theory|(}\index{obstruction
theory!symmetric|(} on $\M_\st^\al(\tau)$. Suppose that
$\M_\rss^\al(\tau)= \M_\st^\al(\tau)$. Then $\M_\st^\al(\tau)$ is
proper, so using the obstruction theory Behrend and Fantechi
\cite{BeFa1} define a virtual class\index{virtual class}
$[\M_\st^\al(\tau)]^\vir\in
A_0(\M_\st^\al(\tau))$.\nomenclature[Mzvir]{$[\M]^\vir$}{virtual
cycle of a proper moduli scheme $\M$, defined using an obstruction
theory on $\M$} The {\it Donaldson--Thomas invariant\/} \cite{Thom}
is defined to be
\e
DT^\al(\tau)=\ts\int_{[\M_\st^\al(\tau)]^\vir}1.
\label{dt4eq15}
\e
Note that $DT^\al(\tau)$ {\it is defined only when\/ $\M_\rss^\al
(\tau)=\M_\st^\al(\tau),$ that is, there are no strictly semistable
sheaves\/ $E$ in class\/} $\al$. One of our main goals is to extend
the definition to all $\al\in K^\num(\coh(X))$.
\label{dt4def4}
\end{dfn}

In fact Thomas did not define invariants $DT^\al(\tau)$ counting
sheaves with fixed class $\al\in K^\num(\coh(X))$, but coarser
invariants $DT^P(\tau)$ counting sheaves with fixed Hilbert
polynomial\index{Hilbert polynomial} $P(t)\in\Q[t]$. Since
$\M_\rss^P(\tau)= \coprod_{\al:P_\al=P}\M_\rss^\al(\tau)$, the
relationship with our version $DT^\al(\tau)$ is
\begin{equation*}
DT^P(\tau)=\ts\sum_{\al\in K^\num(\coh(X)):P_\al=P} DT^\al(\tau),
\end{equation*}
with only finitely many nonzero terms in the sum. Thomas' main
result \cite[\S 3]{Thom}, which works over an arbitrary
algebraically closed base field $\K$, is that

\begin{thm} For each Hilbert polynomial $P(t),$ the invariant\/
$DT^P(\tau)$ is unchanged by continuous deformations of the
underlying Calabi--Yau $3$-fold~$X$.
\label{dt4thm5}
\end{thm}

The same proof shows that $DT^\al(\tau)$ for $\al\in
K^\num(\coh(X))$ is deformation-invariant, {\it provided\/} we know
that the group $K^\num(\coh(X))$ is deformation-invariant, so that
this statement makes sense. This issue will be discussed in
\S\ref{dt45}. We show that when $\K=\C$ we can describe
$K^\num(\coh(X))$ in terms of cohomology groups
$H^*(X;\Z),H^*(X;\Q)$, so that $K^\num(\coh(X))$ is manifestly
deformation-invariant, and therefore $DT^\al(\tau)$ is also
deformation-invariant.

Here is a property of Behrend functions which is crucial for
Donaldson--Thomas theory. It is proved by Behrend
\cite[Th.~4.18]{Behr} when $\K=\C$, but his proof is valid for
general~$\K$.

\begin{thm} Let\/ $\K$ be an algebraically closed field of
characteristic zero, $X$ a proper $\K$-scheme with a symmetric
obstruction theory, and\/ $[X]^\vir\in A_0(X)$ the corresponding
virtual class from Behrend and Fantechi\/ {\rm\cite{BeFa1}}. Then
\begin{equation*}
\ts\int_{[X]^\vir}1=\chi(X,\nu_X)\in\Z,
\end{equation*}
where $\chi(X,\nu_X)=\int_{X(\K)}\nu_X\,\rd\chi$ is the Euler
characteristic of\/ $X$ weighted by the Behrend function $\nu_X$
of\/ $X$. In particular, $\int_{[X]^\vir}1$ depends only on the\/
$\K$-scheme structure of\/ $X,$ not on the choice of symmetric
obstruction theory.
\label{dt4thm6}
\end{thm}

Theorem \ref{dt4thm6} implies that $DT^\al(\tau)$ in \eq{dt4eq15} is
given by
\e
DT^\al(\tau)=\chi\bigl(\M_\st^\al(\tau),\nu_{\M_\st^\al(\tau)}\bigr).
\label{dt4eq16}
\e
There is a big difference between the two equations \eq{dt4eq15} and
\eq{dt4eq16} defining Donaldson--Thomas invariants. Equation
\eq{dt4eq15} is non-local, and non-motivic, and makes sense only if
$\M_\st^\al(\tau)$ is a proper $\K$-scheme. But \eq{dt4eq16} is
local, and (in a sense) motivic, and makes sense for arbitrary
finite type $\K$-schemes $\M_\st^\al(\tau)$. In fact, one could take
\eq{dt4eq16} to be the definition of Donaldson--Thomas invariants
even when $\M_\rss^\al(\tau)\ne\M_\st^\al(\tau)$, but we will argue
in \S\ref{dt65} that this is not a good idea, as then $DT^\al(\tau)$
would not be unchanged under deformations of~$X$.

Equation \eq{dt4eq16} was the inspiration for this book. It shows
that Donaldson--Thomas invariants $DT^\al(\tau)$ can be written as
{\it motivic\/} invariants,\index{motivic invariant} like those studied
in \cite{Joyc3,Joyc4,Joyc5,Joyc6,Joyc7}, and so it raises the
possibility that we can extend the results of
\cite{Joyc3,Joyc4,Joyc5,Joyc6,Joyc7} to Donaldson--Thomas invariants
by including Behrend functions as weights.\index{Calabi--Yau
3-fold|)}\index{Donaldson--Thomas invariants!original $DT^\al(\tau)$|)}

\subsection{Behrend functions and almost closed 1-forms}
\label{dt44}

The material of \S\ref{dt42}--\S\ref{dt43} raises an obvious
question. Given a proper moduli space $\M$ with a symmetric
obstruction theory, such as a moduli space of sheaves $\M_\st^\al
(\tau)$ on a Calabi--Yau 3-fold when $\M_\st^\al(\tau)=\M_\rss^\al
(\tau)$, we have $\int_{[\M]^\vir}1=\chi(\M,\nu_\M)$ by Theorem
\ref{dt4thm6}. If we could write $\M$ as $\Crit(f)$ for $f:U\ra\C$ a
holomorphic function on a complex manifold $U$, we could use the
results of \S\ref{dt42} to study the Behrend function $\nu_\M$.
However, as Behrend says~\cite[p.~5]{Behr}:
\begin{quotation}
`We do not know if every scheme admitting a symmetric obstruction
theory can locally be written as the critical locus of a regular
function on a smooth scheme. This limits the usefulness of the above
formula for $\nu_X(x)$ in terms of the Milnor fibre.'
\end{quotation}

Later we will prove using transcendental complex analytic methods
that when $\K=\C$, moduli spaces $\M_\st^\al(\tau)$ on a Calabi--Yau
3-fold can indeed be written as $\Crit(f)$ for $f$ holomorphic on a
complex manifold $U$, and so we can apply \S\ref{dt42} to prove
identities on Behrend functions \eq{dt5eq2}--\eq{dt5eq3}. But here
we sketch an alternative approach due to Behrend \cite{Behr}, which
could perhaps be used to give a strictly algebraic proof of the same
identities.

\begin{dfn} Let $\K$ be an algebraically closed field, and $M$ a
smooth $\K$-scheme. Let $\om$ be a 1-form on $M$, that is, $\om\in
H^0(T^*M)$. We call $\om$ {\it almost closed\/}\index{almost closed
1-form|(} if $\rd\om$ is a section of $I_\om\cdot\La^2T^*M$, where
$I_\om$ is the ideal sheaf of the zero locus $\om^{-1}(0)$ of $\om$.
Equivalently, $\rd\om\vert_{\om^{-1}(0)}$ is zero as a section of
$\La^2T^*M\vert_{\om^{-1}(0)}$. In (\'etale) local coordinates
$(z_1,\ldots,z_n)$ on $M$, if $\om=f_1\rd z_1+\cdots+f_n\rd z_n$,
then $\om$ is almost closed provided
\begin{equation*}
\frac{\pd f_j}{\pd z_k}\equiv\frac{\pd f_k}{\pd z_j} \;\>\mod
(f_1,\ldots,f_n).
\end{equation*}
\label{dt4def5}
\end{dfn}

Behrend \cite[Prop.~3.14]{Behr} proves the following, by a proof
valid for general~$\K$:

\begin{prop} Let\/ $\K$ be an algebraically closed field, and\/ $X$
a $\K$-scheme with a symmetric obstruction theory. Then $X$ may be
covered by Zariski open sets $Y\subseteq X$ such that there exists a
smooth\/ $\K$-scheme $M,$ an almost closed\/ $1$-form $\om$ on $M,$
and an isomorphism of\/ $\K$-schemes\/~$Y\cong\om^{-1}(0)$.
\label{dt4prop4}
\end{prop}\index{symmetric obstruction
theory|)}\index{obstruction theory!symmetric|)}

If we knew the almost closed 1-form $\om$ was closed, then locally
$\om=\rd f$ for $f:M\ra\K$ regular, and $X\cong\Crit(f)$ as we want.
Restricting to $\K=\C$, Behrend \cite[Prop.~4.22]{Behr} gives an
expression for the Behrend function of the zero locus of an almost
closed 1-form as a linking number. He states it in the complex
algebraic case, but his proof is also valid in the complex analytic
case.

\begin{prop} Let\/ $M$ be a complex manifold and\/ $\om$ an almost
closed holomorphic $(1,0)$-form on $M,$ and let\/ $X=\om^{-1}(0)$ as
a complex analytic subspace of\/ $M$. Fix\/ $x\in X,$ choose
holomorphic coordinates $(z_1,\ldots,z_n)$ on $X$ near $x$ with\/
$z_1(x)=\cdots=z_n(x)=0,$ and let\/ $(z_1,\ldots,z_n,w_1,\ldots,
w_n)$ be the induced coordinates on $T^*M,$ with\/
$(z_1,\ldots,w_n)$ representing the $1$-form $w_1\rd
z_1+\cdots+w_n\rd z_n$ at\/ $(z_1,\ldots,z_n),$ so that we identify
$T^*M$ near $x$ with\/~$\C^{2n}$.

Then for all\/ $\eta\in\C$ and\/ $\ep\in\R$ with\/
$0<\md{\eta}\ll\ep\ll 1$ we have
\e
\nu_X(x)=L_{{\cal S}_\ep}\bigl(\Ga_{\eta^{-1}\om}\cap{\cal S}_\ep,
\De\cap {\cal S}_\ep\bigr),
\label{dt4eq17}
\e
where ${\cal S}_\ep\!=\!\bigl\{(z_1,\ldots,w_n)\!\in\!\C^{2n}:
\ms{z_1}\!+\!\cdots\!+\!\ms{w_n}\!=\!\ep^2\bigr\}$ is the sphere of
radius $\ep$ in $\C^{2n},$ and\/ $\Ga_{\eta^{-1}\om}$ the graph of\/
$\eta^{-1}\om$ regarded locally as a complex submanifold of\/
$\C^{2n},$ and
$\De=\bigl\{(z_1,\ldots,w_n)\!\in\!\C^{2n}:w_j\!=\!\bar z_j,$
$j\!=\!1,\ldots,n\bigr\},$ and\/ $L_{{\cal S}_\ep}(\,,\,)$ the
linking number of two disjoint, closed, oriented\/
$(n\!-\!1)$-submanifolds in~${\cal S}_\ep$.
\label{dt4prop5}
\end{prop}

Here are some questions which seem interesting. If the answer to (a)
is yes, it suggests the possibility of an alternative proof of our
Behrend function identities \eq{dt5eq2}--\eq{dt5eq3} using algebraic
almost closed 1-forms as in Proposition \ref{dt4prop4}, rather than
using transcendental complex analytic methods.

\begin{quest} Let\/ $M$ be a complex manifold, $\om$ an almost
closed holomorphic $(1,0)$-form on $M,$ and\/ $X=\om^{-1}(0)$ as a
complex analytic subspace of\/~$M$.
\smallskip

\noindent {\bf(a)} Can one prove results for Behrend functions
$\nu_X$ analogous to those one can prove for Behrend functions of\/
$\Crit(f)$ for $f:M\ra\C$ holomorphic, using Proposition
{\rm\ref{dt4prop5}?} For instance, is the analogue of Theorem
{\rm\ref{dt4thm4}} true with $\rd f$ replaced by an almost closed\/
$1$-form $\om,$ and\/ $\rd\ti f$ replaced
by~$\pi^*(\om)?$\index{almost closed 1-form|)}
\smallskip

\noindent{\bf(b)} Can one define a natural perverse sheaf\/ $\cal P$
supported on $X,$ with $\chi_X({\cal P})=\nu_X,$ such that\/ ${\cal
P}\cong\phi_f(\underline{\Q}[n-1])$ when\/ $\om=\rd f$ for
$f:M\ra\C$ holomorphic?
\smallskip

\noindent{\bf(c)} If the answer to {\bf(a)} or {\bf(b)} is yes, are
there generalizations to the algebraic setting, which work say
over\/ $\K$ algebraically closed of characteristic zero?\index{field
$\K$!characteristic zero}
\label{dt4quest}
\end{quest}

One can also ask Question \ref{dt4quest}(b) for Saito's mixed Hodge
modules~\cite{Sait}.\index{Behrend function|)}\index{mixed Hodge module}

\subsection{Characterizing $K^\num(\coh(X))$ for Calabi--Yau 3-folds}
\label{dt45}\index{Grothendieck group!numerical|(}\index{cohomology|(}

Let $X$ be a Calabi--Yau 3-fold over $\C$, with $H^1(\cO_X)=0$. We
will now give an exact description of the numerical Grothendieck
group $K^\num(\coh(X))$ in terms of the cohomology $H^{\rm
even}(X,\Q)$. A corollary of this is that $K^\num(\coh(X))$ is
unchanged by small deformations of the complex structure of $X$.
This is necessary for our claim in \S\ref{dt54} that the
$\bar{DT}{}^\al(\tau)$ for $\al\in K^\num(\coh(X))$ are
deformation-invariant to make sense.

To do this we will use the {\it Chern character},\index{Chern character}
as in Hartshorne \cite[App.~A]{Hart2} or Fulton \cite{Fult}. For
each $E\in\coh(X)$ we have the rank $r(E)\in H^0(X;\Z)$ and the
Chern classes $c_i(E)\in H^{2i}(X;\Z)$ for $i=1,2,3$. It is useful
to organize these into the Chern character
$\ch(E)$\nomenclature[ch(E)]{$\ch(E)$}{Chern character of a coherent sheaf $E$
in $H^{\rm even}(X,\Q)$}\nomenclature[chi(E)]{$\ch_i(E)$}{$i^{\rm th}$
component of Chern character of $E$ in $H^{2i}(X,\Q)$} in $H^{\rm
even}(X,\Q)$,\nomenclature[Heven(X,Q)]{$H^{\rm even}(X;\Q)$}{even cohomology of
a complex manifold $X$} where
$\ch(E)=\ch_0(E)+\ch_1(E)+\ch_2(E)+\ch_3(E)$ with $\ch_i(E)\in
H^{2i}(X;\Q)$, with
\e
\begin{gathered}
\ch_0(E)=r(E),\quad \ch_1(E)=c_1(E),\quad
\ch_2(E)=\ts\frac{1}{2}\bigl(c_1(E)^2-2c_2(E)\bigr),\\
\ch_3(E)=\ts\frac{1}{6}\bigl(c_1(E)^3-3c_1(E)c_2(E)+3c_3(E)\bigr).
\end{gathered}
\label{dt4eq19}
\e
Here we use the natural morphism $H^{\rm even}(X;\Z)\ra H^{\rm
even}(X;\Q)$ to make $r(E),\ab c_i(E)$ into elements of $H^{\rm
even}(X;\Q)$. The kernel of this morphism is the {\it
torsion\/}\index{cohomology!torsion} of $H^{\rm even}(X;\Z)$, the
subgroup of elements of finite order. From now on we will neglect
torsion in $H^{\rm even}(X;\Z)$, so by an abuse of notation, when we
say that an element $\la_i$ of $H^{2i}(X;\Q)$ lies in
$H^{2i}(X;\Z)$, we really mean that $\la_i$ lies in the image of
$H^{2i}(X;\Z)$ in~$H^{2i}(X;\Q)$.

By the Hirzebruch--Riemann--Roch Theorem\index{Hirzebruch--Riemann--Roch
Theorem} \cite[Th.~A.4.1]{Hart2}, the Euler form\index{Euler form} on
coherent sheaves $E,F$ is given in terms of their Chern characters
by
\e
\bar\chi\bigl([E],[F]\bigr)=\deg\bigl(\ch(E)^\vee\cdot\ch(F)\cdot{\rm
td}(TX)\bigr){}_3,
\label{dt4eq20}
\e
where ${\rm td}(TX)$\nomenclature[td(TX)]{${\rm td}(TX)$}{Todd class of $TX$ in
$H^{\rm even}(X,\Q)$} is the {\it Todd class\/}\index{Todd class} of
$TX$, which is $1+\frac{1}{12}c_2(TX)$ as $X$ is a Calabi--Yau
3-fold, and
$(\la_0,\la_1,\la_2,\la_3)^\vee=(\la_0,-\la_1,\la_2,-\la_3)$,
writing elements of $H^{\rm even}(X;\Q)$ as $(\la_0,\ldots,\la_3)$
with~$\la_i\in H^{2i}(X;\Q)$.

The Chern character is additive over short exact sequences. That is,
if $0\ra E\ra F\ra G\ra 0$ is exact in $\coh(X)$ then
$\ch(F)=\ch(E)+\ch(G)$. Hence $\ch$ induces a group morphism
$\ch:K_0(\coh(X))\ra H^{\rm even}(X;\Q)$. We have
$K^\num(\coh(X))=K_0(\coh(X))/I$, where $I$ is the kernel of
$\bar\chi$ on $K_0(\coh(X))$. Equation \eq{dt4eq20} implies that
$\Ker\ch\subseteq I$. Theorem \ref{dt4thm7} will identify the image
of $\ch$ in $H^{\rm even}(X;\Q)$. This image spans $H^{\rm
even}(X;\Q)$ over $\Q$, so as the pairing
$(\al,\be)\mapsto\deg\bigl(\al^\vee\cdot\be \cdot{\rm
td}(TX)\bigr){}_3$ is nondegenerate on $H^{\rm even}(X;\Q)$, it is
nondegenerate on the image of $\ch$, so~$\Ker\ch\!=\!I$.

Hence $\ch$ induces an {\it injective\/} morphism
$\ch:K^\num(\coh(X))\hookra H^{\rm even}(X;\Q)$, and we may regard
$K^\num(\coh(X))$ as a subgroup of $H^{\rm even}(X;\Q)$. (Actually,
this is true for any smooth projective $\C$-scheme $X$.) Our next
theorem identifies the image of~$\ch$.\index{Grothendieck
group!numerical!and cohomology}

\begin{thm} Let\/ $X$ be a Calabi--Yau $3$-fold over $\C$ with\/
$H^1(\cO_X)\!=\!0$. Define\nomenclature[\Lambda
X]{$\La_X$}{sublattice of $H^{\rm even}(X,\Q)$ for Calabi--Yau
3-fold $X$}
\e
\begin{split}
\La_X=\ts\bigl\{(\la_0,\la_1,\la_2,\la_3)\in H^{\rm
even}(X;\Q):\la_0\in H^0(X;\Z),\;\>
\la_1\in H^2(X;\Z)&,\\
\la_2-\ha\la_1^2\in H^4(X;\Z),\;\> \la_3+\ts\frac{1}{12}\la_1
c_2(TX)\in H^6(X;\Z)\bigr\}&.
\end{split}
\label{dt4eq21}
\e
Then\/ $\La_X$ is a subgroup of\/ $H^{\rm even}(X;\Q),$ a lattice of
rank\/ $\sum_{i=0}^3b^{2i}(X),$ and the Chern character gives a
group isomorphism~$\ch:K^\num(\coh(X))\!\ra\!\La_X$.

Therefore the numerical Grothendieck group\/ $K^\num(\coh(X))$
depends only on the underlying topological space of $X$ up to
homotopy, and so\/ $K^\num(\coh(X))$ is unchanged by deformations of
the complex structure of\/~$X$.
\label{dt4thm7}
\end{thm}

\begin{proof} Suppose for simplicity that $X$ is connected, so that
we have natural isomorphisms $H^6(X;\Q)\cong\Q$ and
$H^6(X;\Z)\cong\Z$. If it is not connected, we can run the argument
below for each connected component of~$X$.

To show $\La_X$ is a subgroup of $H^{\rm even}(X;\Q)$, we must check
it is closed under addition and inverses. The only issue is that the
condition $\la_2-\ha\la_1^2\in H^4(X;\Z)$ is not linear in $\la_1$.
If $(\la_0,\ldots,\la_3), (\la_0',\ldots,\la_3')\in\La_X$ then
\begin{equation*}
(\la_2+\la_2')-\ha(\la_1+\la_1')^2=\bigl[\la_2-\ha\la_1^2\bigr]
+\bigl[\la'_2-\ha(\la'_1)^2\bigr]+\bigl[\la_1\la_1'\bigr],
\end{equation*}
and the right hand side is the sum of three terms in $H^4(X;\Z)$. So
$(\la_0+\la_0',\ldots,\la_3+\la_3')\in\La_X$. Also
\begin{equation*}
(-\la_2)-\ha(-\la_1)^2=-\bigl[\la_2-\ha\la_1^2\bigr]-
\bigl[\la_1^2\bigr],
\end{equation*}
with the right hand the sum of two terms in $H^4(X;\Z)$. So
$(-\la_0,\ldots,-\la_3)\in\La_X$, and $\La_X$ is a subgroup of
$H^{\rm even}(X;\Q)$. We have
\begin{equation*}
\ts\frac{1}{6}\,H^{\rm even}(X;\Z)/\text{torsion}\subseteq\La_X\subseteq
H^{\rm even}(X;\Z)/\text{torsion}\subseteq H^{\rm even}(X;\Q),
\end{equation*}
so $\La_X$ is a lattice of rank $\sum_{i=0}^3b^{2i}(X)$ as $H^{\rm
even}(X;\Z)/\text{torsion}$ is.

Next we show that $\ch\bigl(K^\num(\coh(X))\bigr)\subseteq\La_X$. As
$\La_X$ is a subgroup, it is enough to show that $\ch(E)\in\La_X$
for any $E\in\coh(X)$. Set $\ch(E)=(\la_0,\ldots,\la_3)$. Then
\eq{dt4eq19} gives $\la_0=r(E)$, $\la_1=c_1(E)$,
$\la_2=\frac{1}{2}\bigl(c_1(E)^2-2c_2(E)\bigr)$ and
$\la_3=\frac{1}{6}\bigl(c_1(E)^3-3c_1(E)c_2(E)+3c_3(E)\bigr)$, with
$r(E)\in H^0(X;\Z)$ and $c_i(E)\!\in\!H^{2i}(X;\Z)$. The conditions
$\la_0\!\in\! H^0(X;\Z)$ and $\la_1\!\in\! H^2(X;\Z)$ are immediate,
and $\la_2\!-\!\ha\la_1^2\!=\!-\!c_2(E)$ which lies in $H^4(X;\Z)$.
For the final condition,
\begin{align*}
\deg\bigl(\la_3&+\ts\frac{1}{12}\la_1 c_2(TX)\bigr)\\
&=\ts\deg\bigl(\frac{1}{6}\bigl(c_1(E)^3-3c_1(E)c_2(E)+3c_3(E)\bigr)
+\ts\frac{1}{12}c_1(E)c_2(TX)\bigr)\\
&=\ts\deg\bigl((1,0,0,0)\cdot \bigl(r(E),c_1(E),
\frac{1}{2}(c_1(E)^2-2c_2(E)),\\
&\qquad\qquad\ts
\frac{1}{6}(c_1(E)^3-3c_1(E)c_2(E)+3c_3(E))\bigr)\cdot
(1,0,\frac{1}{12}c_2(TX),0)\bigr){}_3\\
&=\deg\bigl(\ch(\cO)^\vee\cdot\ch(E)\cdot{\rm td}(TX)\bigr){}_3
=\bar\chi\bigl([\cO_X],[F]\bigr)\in\Z,
\end{align*}
using \eq{dt4eq20} in the last line. Then
$\deg\bigl(\la_3+\ts\frac{1}{12}\la_1 c_2(TX)\bigr)\in\Z$ implies
$\la_3+\ts\frac{1}{12}\la_1 c_2(TX)\in H^6(X;\Z)\cong\Z$. Hence
$\ch(E)=(\la_0,\ldots,\la_3)\in\La_X$, as we want.

As $X$ is a Calabi--Yau 3-fold over $\C$ with $H^1(\cO_X)=0$ we have
$H^{2,0}(X)=H^{0,2}(X)=0$, so $H^{1,1}(X)=H^2(X;\C)$. Therefore
every $\be\in H^2(X;\Z)$ is $c_1(L_\be)$ for some holomorphic line
bundle $L_\be$, with
\e
\ch(L_\be)=\ts\bigl(1,\be,\ha\be^2,\frac{1}{6}\be^3\bigr),
\qquad\text{for any $\be\in H^2(X;\Z)$.}
\label{dt4eq22}
\e
Pick $x\in X$, and let $\cO_x$ be the skyscraper sheaf at $x$. Then
\e
\ch(\cO_x)=(0,0,0,1),
\label{dt4eq23}
\e
identifying $H^6(X;\Q)\cong\Q$ and $H^6(X;\Z)\cong\Z$ in the natural
way. Suppose $\Si$ is an reduced algebraic curve in $X$, with
homology class $[\Si]\in H_2(X;\Z)\cong H^4(X;\Z)$. Then the
structure sheaf $\cO_\Si$ in $\coh(X)$ has
\e
\ch(\cO_\Si)=\bigl(0,0,[\Si],k\bigr)\quad\text{for some $k\in\Z$.}
\label{dt4eq24}
\e

Now $\ch\bigl(K^\num(\coh(X))\bigr)$ is a subgroup of $\La_X$ which
contains \eq{dt4eq22}--\eq{dt4eq24}. We claim that the elements
\eq{dt4eq22}--\eq{dt4eq24} over all $\be,\Si$ generate $\La_X$,
which forces $\ch\bigl(K^\num(\coh(X))\bigr)=\La_X$ and proves the
theorem. This depends on a deep fact: Voisin \cite[Th.~2]{Vois}
proves the Hodge Conjecture over $\Z$ for Calabi--Yau 3-folds $X$
over $\C$ with $H^1(\cO_X)=0$. In this case, the Hodge Conjecture
over $\Z$ is equivalent to the statement that $H_2(X;\Z)\cong
H^4(X;\Z)$ is generated as a group by classes $[\Si]$ of algebraic
curves $\Si$ in $X$. It follows that \eq{dt4eq23} and \eq{dt4eq24}
taken over all $\Si$ generate the subgroup of
$(0,0,\la_2,\la_3)\in\La_X$ with $\la_2\in H^4(X;\Z)$ and $\la_3\in
H^6(X;\Z)$. Together with \eq{dt4eq22} for all $\be\in H^2(X;\Z)$,
these generate~$\La_X$.
\end{proof}

\begin{rem}{\bf(a)} Our proof used the Hodge Conjecture\index{Hodge
Conjecture} over $\Z$ for Calabi--Yau 3-folds, proved by Voisin
\cite{Vois}. But the Hodge Conjecture over $\Z$ is false in general,
so the theorem may not generalize to other classes of varieties.
\smallskip

\noindent{\bf(b)} In fact $\La_X$ is a subring of $H^{\rm
even}(X;\Q)$. Also, $K_0(\coh(X)),K^\num(\coh(X))$ naturally have
the structure of rings, with multiplication `$\,\cdot\,$'
characterized by $[E]\cdot[F]=[E\ot F]$ for $E,F$ locally free. As
$\ch(E\ot F)=\ch(E)\ch(F)$ for $E,F$ locally free, it follows that
$\ch:K^\num(\coh(X))\ra\La_X$ is a ring isomorphism. But we will
make no use of these ring structures.
\smallskip

\noindent{\bf(c)} If $X$ is a Calabi--Yau 3-fold over $\C$ but
$H^1(\cO_X)\ne 0$ then $\ch\bigl(K^\num(\coh(X))\bigr)$ can be a
proper subgroup of $\La_X$, and this subgroup can change under
deformations of $X$. Thus $K^\num(\coh(X))$ need not be
deformation-invariant up to isomorphism, as its rank can jump under
deformation.

To see this, note that if $H^1(\cO_X)\ne 0$ then $H^{2,0}(X)\ne 0$,
so $H^{1,1}(X)$ is a proper subspace of $H^2(X;\C)$, which can vary
as we deform $X$, and the intersection $H^{1,1}(X)\cap H^2(X;\Z)$
can change under deformation. Let $\be\in H^2(X;\Z)$. Then
$\be=c_1(L)$ for some holomorphic line bundle $L$ on $X$ if and only
if $\be\in H^{1,1}(X)$, and then $\ch(L)=\bigl(1,\be,\ha\be^2,
\frac{1}{6}\be^3\bigr)$. One can show that $\bigl(1,\be,\ha\be^2,
\frac{1}{6}\be^3\bigr)$ lies in $\ch\bigl(K^\num(\coh(X))\bigr)$ if
and only if $\be\in H^{1,1}(X)$.
\smallskip

\noindent{\bf(d)} Let $B$ be a $\C$-scheme, and $X_b$ for $b\in
B(\C)$ be a family of Calabi--Yau 3-folds with $H^1(\cO_{X_b})=0$.
That is, we have a smooth $\C$-scheme morphism $\pi:X\ra B$, with
fibres $X_b$ for $b\in B(\C)$. Then we can form $K^\num(\coh(X_b))$
for each $b\in B(\C)$. If $B$ is connected, then for $b_0,b_1\in
B(\C)$, we can choose a continuous path $\ga:[0,1]\ra B(\C)$ joining
$b_0$ and $b_1$. The family $X_{\ga(t)}$ for $t\in[0,1]$ defines a
homotopy $X_{b_0}\,{\buildrel\sim\over\longra}\,X_{b_1}$, and so
induces isomorphisms $H^{\rm even}(X_{b_0};\Q)\cong H^{\rm
even}(X_{b_1};\Q)$, $\La_{X_{b_0}}\cong\La_{X_{b_1}}$,
and~$K^\num(\coh(X_{b_0}))\cong K^\num(\coh(X_{b_1}))$.

However, if $B$ is not {\it simply-connected\/} this isomorphism
$K^\num(\coh(X_{b_0}))\cong K^\num(\coh(X_{b_1}))$ can depend on the
homotopy class of the path $\ga$. The groups $K^\num(\coh(X_b))$ for
$b\in B(\C)$ form a {\it local system\/} on $B(\C)$, so that the
fibres are all isomorphic, but going round loops in $B(\C)$ can
induce nontrivial automorphisms of $K^\num(\coh(X_b))$, through an
action of $\pi_1(B(\C))$ on $K^\num(\coh(X_b))$. This phenomenon is
called {\it monodromy}.\index{monodromy|(} We study it in Theorem
\ref{dt4thm8} below.

Thus, the statement in Theorem \ref{dt4thm7} that $K^\num(\coh(X))$
is unchanged by deformations of the complex structure of $X$ should
be treated with caution: it is true up to isomorphism, but in
general it is only true up to canonical isomorphism if we restrict
to a simply-connected family of deformations.
\smallskip

\noindent{\bf(e)} It may be possible to extend Theorem \ref{dt4thm7}
to work over an algebraically closed base field $\K$ of
characteristic zero by replacing $H^*(X;\Q)$ by the {\it algebraic
de Rham cohomology\/}\index{algebraic de Rham cohomology} $H^*_{\rm
dR}(X)$ of Hartshorne \cite{Hart1}. For $X$ a smooth projective
$\K$-scheme, $H^*_{\rm dR}(X)$ is a finite-dimensional vector space
over $\K$. There is a Chern character map
$\ch:K^\num(\coh(X))\hookra H^{\rm even}_{\rm dR}(X)$. In \cite[\S
4]{Hart1}, Hartshorne considers how $H^*_{\rm dR}(X_t)$ varies in
families $X_t:t\in T$, and defines a Gauss--Manin connection, which
makes sense of $H^*_{\rm dR}(X_t)$ being locally constant in~$t$.
\label{dt4rem}
\end{rem}

We now study monodromy phenomena for $K^\num(\coh(X_u))$ in families
of smooth $\K$-schemes $X\ra U$, as in Remark \ref{dt4rem}(d). We
find that we can always eliminate such monodromy by passing to a
finite cover $\ti U$ of $U$. This will be used in \S\ref{dt54} and
\S\ref{dt12} to prove deformation-invariance of
the~$\bar{DT}{}^\al(\tau),PI^{\al,n}(\tau')$.

\begin{thm} Let\/ $\K$ be an algebraically closed field, $\vp:X\ra
U$ a smooth projective morphism of\/ $\K$-schemes with\/ $U$
connected, and\/ $\cO_X(1)$ a relative very ample line on $X,$ so
that for each\/ $u\in U(\K),$ the fibre $X_u$ of\/ $\vp$ is a smooth
projective $\K$-scheme with very ample line bundle\/ $\cO_{X_u}(1)$.
Suppose the numerical Grothendieck groups $K^\num(\coh(X_u))$ are
locally constant in $U(\K),$ so that\/ $u\mapsto K^\num(\coh(X_u))$
is a local system of abelian groups on\/~$U$.

Fix a base point\/ $v\in U(\K),$ and let\/ $\Ga$ be the group of
automorphisms of\/ $K^\num(\coh(X_v))$ generated by monodromy round
loops in $U$. Then $\Ga$ is a finite group. There exists a finite
\'etale cover $\pi:\ti U\ra U$ of degree $\md{\Ga},$ with\/ $\ti U$
a connected $\K$-scheme, such that writing $\ti X=X\times_U\ti U$
and $\ti\vp:\ti X\ra\ti U$ for the natural projection, with fibre
$\ti X_{\ti u}$ at $\ti u\in\ti U(\K),$ then $K^\num(\coh(\ti X_{\ti
u}))$ for all $\ti u\in\ti U(\K)$ are all globally canonically
isomorphic to $K^\num(\coh(X_v))$. That is, the local system $\ti
u\mapsto K^\num(\coh(\ti X_{\ti u}))$ on $\ti U$ is trivial.
\label{dt4thm8}
\end{thm}

\begin{proof} As $K^\num(\coh(X_v))$ is a finite rank lattice, we
can choose $E_1,\ldots,E_n\in\coh(X_v)$ such that
$[E_1],\ldots,[E_n]$ generate $K^\num(\coh(X_v))$. Write $\tau$ for
Gieseker stability on $\coh(X_v)$ with respect to $\cO_{X_v}(1)$.
Then as in \S\ref{dt32} each $E_i$ has a Harder--Narasimhan
filtration\index{Harder--Narasimhan filtration} with
$\tau$-semistable factors $S_{ij}$. Let $\al_1,\ldots,\al_k\in
K^\num(\coh(X_v))$ be the classes of the $S_{ij}$ for all $i,j$.
Then $\al_1,\ldots,\al_k$ generate $K^\num(\coh(X_v))$ as an abelian
group, and the coarse moduli scheme $\M_\rss^{\al_i}(\tau)_v$ of
$\tau$-semistable sheaves on $X_v$ in class $\al_i$ is nonempty for
$i=1,\ldots,k$. Let $P_i$ be the Hilbert polynomial of $\al_i$ for
$i=1,\ldots,k$.

Let $\ga\in\Ga$, and consider the images $\ga\cdot\al_i\in
K^\num(\coh(X_v))$ for $i=1,\ldots,k$. As we assume $\cO_X(1)$ is
globally defined on $U$ and does not change under monodromy, it
follows that the Hilbert polynomials of classes $\al\in
K^\num(\coh(X_v))$ do not change under monodromy. Hence
$\ga\cdot\al_i$ has Hilbert polynomial $P_i$. Also, the condition
that $\M_\rss^\al(\tau)_u\ne\es$ for $u\in U(\K)$ and $\al\in
K^\num(\coh(X_u))$ is an open and closed condition in $(u,\al)$, so
as $\M_\rss^{\al_i}(\tau)_v\ne\es$ we have
$\M_\rss^{\ga\cdot\al_i}(\tau)_v\ne\es$.

For each $i=1,\ldots,k$, the family of $\tau$-semistable sheaves on
$X_v$ with Hilbert polynomial $P_i$ is bounded, and therefore
realizes only finitely many classes $\be_i^1,\ldots,\be_i^{n_i}$ in
$K^\num(\coh(X_v))$. It follows that for each $\ga\in\Ga$ we have
$\ga\cdot\al_i\in\{\be_i^1,\ldots,\be_i^{n_i}\}$. So there are at
most $n_1\cdots n_k$ possibilities for $(\ga\cdot\al_1,\ldots,
\ga\cdot\al_k)$. But $(\ga\cdot\al_1,\ldots,\ga\cdot\al_k)$
determines $\ga$ as $\al_1,\ldots,\al_k$ generate
$K^\num(\coh(X_v))$. Hence $\md{\Ga}\le n_1\cdots n_k$, and $\Ga$ is
finite.

We can now construct an \'etale cover $\pi:\ti U\ra U$ which is a
principal $\Ga$-bundle, and so has degree $\md{\Ga}$, such that the
$\K$-points of $\ti U$ are pairs $(u,\io)$ where $u\in U(\K)$ and
$\io:K^\num(\coh(X_u))\ra K^\num(\coh(X_v))$ is an isomorphism
induced by parallel transport along some path from $u$ to $v$, which
is possible as $U$ is connected, and $\Ga$ acts freely on $\ti
U(\K)$ by $\ga:(u,\io)\mapsto (u,\ga\ci\io)$, so that the
$\Ga$-orbits correspond to points $u\in U(\K)$. Then for $\ti
u=(u,\io)$ we have $\ti X_{\ti u}=X_u$, with canonical
isomorphism~$\io:K^\num(\coh(\ti X_{\ti u}))\ra K^\num(\coh(X_v))$.
\end{proof}\index{Grothendieck group!numerical|)}\index{monodromy|)}
\index{cohomology|)}

\section{Statements of main results}
\label{dt5}\index{Calabi--Yau 3-fold|(}

Let $X$ be a Calabi--Yau 3-fold over the complex numbers $\C$, and
$\cO_X(1)$ a very ample line bundle over $X$. For the rest of the
book, our definition of Calabi--Yau 3-fold includes the assumption
that $H^1(\cO_X)=0$, except where we explicitly allow otherwise.
Remarks \ref{dt5rem1} and \ref{dt5rem2} discuss the reasons for
assuming $\K=\C$ and $H^1(\cO_X)=0$. Write $\coh(X)$ for the abelian
category of coherent sheaves on $X$, and $K(\coh(X))$ for the
numerical Grothendieck group of $\coh(X)$. Let $(\tau,G,\le)$ be the
stability condition on $\coh(X)$ of Gieseker stability with respect
to $\cO_X(1)$, as in Example \ref{dt3ex1}. If $E$ is a coherent
sheaf on $X$ then the class $[E]\in K(\coh(X))$ is in effect the
Chern character ch$(E)$ of~$E$.

Write $\fM$ for the moduli stack of coherent sheaves $E$ on $X$. It
is an Artin $\C$-stack, locally of finite type. For $\al\in
K(\coh(X))$, write $\fM^\al$ for the open and closed substack of $E$
with $[E]=\al$ in $K(\coh(X))$. (In \S\ref{dt3} we used the notation
$\fM_{\coh(X)},\fM^\al_{\coh(X)}$ for $\fM,\fM^\al$, but we now drop
the subscript $\coh(X)$ for brevity). Write $\fM_\rss^\al(\tau),
\fM_\st^\al(\tau)$ for the substacks of $\tau$-(semi)stable sheaves
$E$ in class $[E]=\al$, which are finite type open substacks of
$\fM^\al$. Write $\M_\rss^\al(\tau),\M_\st^\al(\tau)$ for the coarse
moduli schemes\index{coarse moduli scheme}\index{moduli
scheme!coarse} of $\tau$-(semi)stable sheaves $E$ with $[E]=\al$.
Then $\M_\rss^\al(\tau)$ is a projective $\C$-scheme whose points
correspond to S-equivalence\index{S-equivalence} classes of
$\tau$-semistable sheaves, and $\M_\st^\al(\tau)$ is an open
subscheme of $\M_\rss^\al(\tau)$ whose points correspond to
isomorphism classes of $\tau$-stable sheaves.

We divide our main results into four sections
\S\ref{dt51}--\S\ref{dt54}. Section \ref{dt51} studies local
properties of the moduli stack $\fM$ of coherent sheaves on $X$. We
first show that $\fM$ is Zariski locally isomorphic to the moduli
stack $\fVect$ of algebraic vector bundles on $X$. Then we use gauge
theory on complex vector bundles and transcendental complex analytic
methods to show that an atlas for $\fM$ may be written locally in
the complex analytic topology as $\Crit(f)$ for $f:U\ra\C$ a
holomorphic function on a complex manifold $U$. The proofs of
Theorems \ref{dt5thm1}, \ref{dt5thm2}, and \ref{dt5thm3} in
\S\ref{dt51} are postponed to~\S\ref{dt8}--\S\ref{dt9}.

Section \ref{dt52} uses the results of \S\ref{dt51} and the Milnor
fibre description of Behrend functions in \S\ref{dt42} to prove two
identities \eq{dt5eq2}--\eq{dt5eq3} for the Behrend function
$\nu_\fM$ of the moduli stack $\fM$. The proof of Theorem
\ref{dt5thm4} in \S\ref{dt52} is given in \S\ref{dt10}. Section
\ref{dt53}, the central part of our book, constructs a Lie algebra
morphism $\ti\Psi:\SFai(\fM)\ra\ti L(X)$, which modifies $\Psi$ in
\S\ref{dt34} by inserting the Behrend function $\nu_\fM$ as a
weight. Then we use $\ti\Psi$ to define generalized
Donaldson--Thomas invariants $\bar{DT}{}^\al(\tau)$, and show they
satisfy a transformation law under change of stability condition
$\tau$. Theorem \ref{dt5thm5} in \S\ref{dt53} is proved
in~\S\ref{dt11}.

Section \ref{dt54} shows that our new invariants
$\bar{DT}{}^\al(\tau)$ are unchanged under deformations of the
underlying Calabi--Yau 3-fold $X$. We do this by first defining
auxiliary invariants $PI^{\al,n}(\tau')$ counting `stable pairs'
$s:\cO_X(-n)\ra E$ for $E\in\coh(X)$ and $n\gg 0$, similar to
Pandharipande--Thomas invariants \cite{PaTh}. We show the moduli
space of stable pairs $\M_\stp^{\al,n}(\tau')$ is a projective
scheme with a symmetric obstruction theory, and deduce that
$PI^{\al,n}(\tau')$ is unchanged under deformations of $X$. We prove
a formula for $PI^{\al,n}(\tau')$ in terms of the
$\bar{DT}{}^\be(\tau)$, and use this to deduce that
$\bar{DT}{}^\al(\tau)$ is deformation-invariant. The proofs of
Theorems \ref{dt5thm7}, \ref{dt5thm8}, \ref{dt5thm9}, and
\ref{dt5thm10} in \S\ref{dt54} are postponed
to~\S\ref{dt12}--\S\ref{dt13}.

\begin{rem} We will use the assumption that the base field $\K=\C$
for the Calabi--Yau 3-fold $X$ in three main ways in the rest of the
book:\index{field $\K$|(}
\begin{itemize}
\setlength{\itemsep}{0pt}
\setlength{\parsep}{0pt}
\item[(a)] Theorems \ref{dt5thm2} and \ref{dt5thm3} in
\S\ref{dt51} are proved using gauge theory and transcendental
complex analytic methods, and work only over $\K=\C$. These are
used to prove the Behrend function identities
\eq{dt5eq2}--\eq{dt5eq3} in \S\ref{dt52}, which are vital for
much of \S\ref{dt53}--\S\ref{dt7}, including the wall crossing
formula \eq{dt5eq14} for the $\bar{DT}{}^\al(\tau)$, and the
relation \eq{dt5eq17} between
$PI^{\al,n}(\tau'),\bar{DT}{}^\al(\tau)$. In examples we often
compute $PI^{\al,n}(\tau')$ and then use \eq{dt5eq17} to
find~$\bar{DT}{}^\al(\tau)$.
\item[(b)] As in \S\ref{dt45}, when $\K=\C$ the Chern
character\index{Chern character} embeds $K^\num(\coh(X))$ in $H^{\rm
even}(X;\Q)$, and we use this to show $K^\num(\coh(X))$ is
unchanged under deformations of $X$. This is important for the
results in \S\ref{dt54} that $\bar{DT}{}^\al(\tau)$ and
$PI^{\al,n}(\tau')$ for $\al\in K^\num(\coh(X))$ are invariant
under deformations of $X$ even to make sense. Also, in
\S\ref{dt6} we use this embedding in $H^{\rm even}(X;\Q)$ as a
convenient way of describing classes in~$K^\num(\coh(X))$.
\item[(c)] Our notion of {\it compactly
embeddable\/}\index{compactly embeddable} in \S\ref{dt67} is complex
analytic and does not make sense for general~$\K$.
\end{itemize}

We now discuss the extent to which the results can be extended to
other fields $\K$. Thomas' original definition \eq{dt4eq15} of
$DT^\al(\tau)$, and our definition \eq{dt5eq15} of the pair
invariants $PI^{\al,n}(\tau')$, are both valid over general
algebraically closed fields $\K$. Apart from problem (b) with
$K^\num(\coh(X))$ above, the proofs of deformation-invariance of
$DT^\al(\tau),PI^{\al,n}(\tau')$ are also valid over general~$\K$.

As noted after Theorem \ref{dt2thm1}, constructible
functions\index{constructible function!in positive characteristic}
methods fail for $\K$ of positive characteristic.\index{field
$\K$!positive characteristic} Because of this, the alternative
descriptions \eq{dt4eq16}, \eq{dt5eq16} for
$DT^\al(\tau),PI^{\al,n}(\tau')$ as weighted Euler characteristics,
and the definition of $\bar{DT}{}^\al(\tau)$ in \S\ref{dt53}, are
valid for algebraically closed fields $\K$ of characteristic
zero.\index{field $\K$!characteristic zero}

The authors believe that the Behrend function identities
\eq{dt5eq2}--\eq{dt5eq3} should hold over algebraically closed
fields $\K$ of characteristic zero; Question \ref{dt5quest3}(a)
suggests a starting point for a purely algebraic proof of
\eq{dt5eq2}--\eq{dt5eq3}. This would resolve (a) above, and probably
also (c), because we only need the notion of `compactly embeddable'
as our complex analytic proof of \eq{dt5eq2}--\eq{dt5eq3} requires
$X$ compact; an algebraic proof of \eq{dt5eq2}--\eq{dt5eq3} would
presumably also work for compactly supported sheaves on a
noncompact~$X$.

For (b), one approach valid over general $\K$ which is
deformation-invariant is to count sheaves with fixed Hilbert
polynomial, as in Thomas \cite{Thom}, rather than with fixed class
in $K^\num(\coh(X))$. It seems likely that $K^\num(\coh(X))$ is
deformation-invariant for more general $\K$, so there may not be a
problem.
\label{dt5rem1}
\end{rem}\index{field $\K$|)}

\begin{rem} We will use the assumption $H^1(\cO_X)=0$ for the
Calabi--Yau 3-fold $X$ in four different ways in the rest of the
book:
\begin{itemize}
\setlength{\itemsep}{0pt}
\setlength{\parsep}{0pt}
\item[(i)] Theorem \ref{dt5thm1} shows that the moduli stack
$\fM$ of coherent sheaves is locally isomorphic to the moduli
stack $\fVect$ of vector bundles. The proof uses Seidel--Thomas
twists by $\cO_X(-n)$, and is only valid if $\cO_X(-n)$ is a
spherical object in $D^b(\coh(X))$, that is, if $H^1(\cO_X)=0$.
Theorem \ref{dt5thm1} is needed to show that the Behrend
function identities \eq{dt5eq2}--\eq{dt5eq3} hold on $\fM$ as
well as on $\fVect$, and \eq{dt5eq2}--\eq{dt5eq3} are essential
for most of~\S\ref{dt53}--\S\ref{dt54}.
\item[(ii)] If $H^1(\cO_X)=0$ then as in \S\ref{dt45}
$K^\num(\coh(X))$ is unchanged under deformations of $X$, which
makes sense of the idea that $\bar{DT}{}^\al(\tau)$ is
deformation-invariant for~$\al\in K^\num(\coh(X))$.
\item[(iii)] In \S\ref{dt63} we use that if $H^1(\cO_X)=0$ then
Hilbert schemes\index{Hilbert scheme} $\Hilb^d(X)$ are open
subschemes of moduli schemes of sheaves on~$X$.
\item[(iv)] In \S\ref{dt64} and \S\ref{dt66} we use that if
$H^1(\cO_X)=0$ then every $\be\in H^2(X;\Z)$ is $c_1(L)$ for
some line bundle $L$ and the map $E\mapsto E\ot L$ for
$E\in\coh(X)$ to deduce symmetries of
the~$\bar{DT}{}^\al(\tau)$.
\end{itemize}
Of these, in (i) Theorem \ref{dt5thm1} is false if $H^1(\cO_X)\ne
0$, but nonetheless the authors expect \eq{dt5eq2}--\eq{dt5eq3} will
be true if $H^1(\cO_X)\ne 0$, and this is the important thing for
most of our theory. Question \ref{dt5quest3}(a) suggests a route
towards a purely algebraic proof of \eq{dt5eq2}--\eq{dt5eq3}, which
is likely not to require~$H^1(\cO_X)=0$.

For (ii), if $H^1(\cO_X)\ne 0$ then in \S\ref{dt45} the Chern
character induces a map $\ch:K^\num(\coh(X))\ra\La_X$ which is
injective but may not be surjective, where $\La_X\subset H^{\rm
even}(X;\Q)$ is the lattice in \eq{dt4eq21}. Let us identify
$K^\num(\coh(X))$ with its image under $\ch$ in $\La_X$, and then
extend the definition of $\bar{DT}{}^\al(\tau)$ to all $\al\in\La_X$
by setting $\bar{DT}{}^\al(\tau)=0$ for $\al\in\La_X\sm\ch\bigl(
K^\num(\coh(X))\bigr)$. Then $\bar{DT}{}^\al(\tau)$ is defined for
$\al\in\La_X$, where $\La_X$ is deformation-invariant, and we claim
that $\bar{DT}{}^\al(\tau)$ will then be deformation-invariant. Part
(iii) is false if~$H^1(\cO_X)\ne 0$.

Generalizing our theory to the case $H^1(\cO_X)\ne 0$ is not very
interesting anyway, as most invariants $\bar{DT}{}^\al(\tau)$ (as we
have defined them) will be automatically zero. If $\dim
H^1(\cO_X)=g>0$ then there is a $T^{2g}$ family of flat line bundles
on $X$ up to isomorphism, which is a group under $\ot$, with
identity $\cO_X$. Suppose $\al\in K^\num(\coh(X))$ with
$\rank\al>0$, so that $\tau$-semistable sheaves in class $\al$ are
torsion-free. Then $E\mapsto E\ot L$ for $E\in\M_\rss^\al(\tau)$ and
$L\in T^{2g}$ defines an action of $T^{2g}$ on $\M_\rss^\al(\tau)$
which is essentially free. As $\bar{DT}{}^\al(\tau)$ is a weighted
Euler characteristic of $\M_\rss^\al(\tau)$, and each orbit of
$T^{2g}$ in $\M_\rss^\al(\tau)$ is a copy of $T^{2g}$ with Euler
characteristic zero, it follows that $\bar{DT}{}^\al(\tau)=0$
when~$\rank\al>0$.

One solution to this is to consider sheaves with {\it fixed
determinant},\index{sheaf!with fixed determinant}\index{coherent
sheaf!with fixed determinant} as in Thomas \cite{Thom}. But this is
not nicely compatible with viewing $\coh(X)$ as an abelian category,
or with our treatment of wall-crossing formulae. In (iv) above, if
$H^1(\cO_X)\ne 0$ and $\be\in H^2(X;\Z)$ is not of the form $c_1(L)$
for a holomorphic line bundle $L$, then for any $\al\in
K^\num(\coh(X))$ which would be moved by the symmetry of $H^{\rm
even}(X;\Q)$ corresponding to $\be$ we must have
$\bar{DT}{}^\al(\tau)=0$ as above. Thus (iv) should still hold when
$H^1(\cO_X)\ne 0$, but for trivial reasons.
\label{dt5rem2}
\end{rem}

\subsection{Local description of the moduli of coherent sheaves}
\label{dt51}

We begin by recalling some facts about moduli spaces and moduli
stacks of coherent sheaves and vector bundles over smooth projective
$\K$-schemes, to establish notation. Let $\K$ be an algebraically
closed field, and $X$ a smooth projective $\K$-scheme of dimension
$m$. In Theorem \ref{dt5thm1} we will take $X$ to be a Calabi--Yau
$m$-fold over general $\K$, and from Theorem \ref{dt5thm2} onwards
we will restrict to $\K=\C$, $m=3$ and $X$ a Calabi--Yau 3-fold,
except for Theorems \ref{dt5thm7}, \ref{dt5thm8} and \ref{dt5thm9},
which work over general~$\K$.

Some good references are Hartshorne \cite[\S II.5]{Hart2} on
coherent sheaves, Huybrechts and Lehn \cite{HuLe2} on coherent
sheaves and moduli schemes, and Laumon and Moret-Bailly \cite{LaMo}
on algebraic spaces\index{algebraic space} and stacks. When we say a
coherent sheaf $E$ is {\it simple\/}\index{coherent sheaf!simple|(} we
mean that $\Hom(E,E)\cong\C$. (Beware that some authors use `simple'
with the different meaning `has no nontrivial subobjects'. An
alternative word for `simple' in our sense would be {\it
Schurian}.)\index{coherent sheaf!Schurian} By an {\it algebraic vector
bundle\/}\index{vector bundle!algebraic} we mean a locally free coherent
sheaf on $X$ of rank $l\ge 0$. (See Hartshorne
\cite[Ex.~II.5.18]{Hart2} for an alternative definition of vector
bundles $E$ as a morphism of schemes $\pi:E\ra X$ with fibre $\K^l$
and with extra structure, and an explanation of why these are in
1--1 correspondence with locally free sheaves.)

Write $\fM$ and $\fVect$\nomenclature[Vect]{$\fVect$}{moduli stack of algebraic
vector bundles} for the moduli stacks of coherent sheaves and
algebraic vector bundles on $X$, respectively. By Laumon and
Moret-Bailly \cite[\S\S 2.4.4, 3.4.4 \& 4.6.2]{LaMo} using results
of Grothendieck \cite{Grot1,Grot2}, they are Artin $\K$-stacks,
locally of finite type, and $\fVect$ is an open $\K$-substack of
$\fM$. Write $\M_\rsi$\nomenclature[Mcasi]{$\M_\rsi$}{coarse moduli space of
simple coherent sheaves} and
$\Vect_\rsi$\nomenclature[Vectsi]{$\Vect_\rsi$}{coarse moduli space of simple
algebraic vector bundles} for the coarse moduli spaces of simple
coherent sheaves and simple algebraic vector bundles. By Altman and
Kleiman \cite[Th.~7.4]{AlKl} they are algebraic
$\K$-spaces,\index{algebraic space} locally of finite type. Also
$\Vect_\rsi$ is an open subspace of~$\M_\rsi$.

Here $\fM,\fVect,\M_\rsi,\Vect_\rsi$ being {\it locally of finite
type\/}\index{Artin stack!locally of finite type} means roughly only
that they are locally finite-dimensional; in general
$\fM,\ldots,\Vect_\rsi$ are neither proper (essentially, compact),
nor separated (essentially, Hausdorff), nor of finite type (finite
type is necessary for invariants such as Euler characteristics to be
well-defined; for the purposes of this book, `finite type' means
something like `measurable'). These results tell us little about the
global geometry of $\fM,\ldots,\Vect_\rsi$. But we do have some
understanding of their local geometry.

For the moduli stack $\fM$ of coherent sheaves on $X$, writing
$\fM(\K)$ for the set of geometric points of $\fM$, as in Definition
\ref{dt2def1}, elements of $\fM(\K)$ are just isomorphism classes
$[E]$ of coherent sheaves $E$ on $X$. Fix some such $E$. Then the
stabilizer group $\Iso_{\fM}([E])$ in $\fM$ is isomorphic as an
algebraic $\K$-group to the automorphism group $\Aut(E),$ and the
Zariski tangent space $T_{[E]}\fM$ to $\fM$ at $[E]$ is isomorphic
to $\Ext^1(E,E)$, and the action of $\Iso_{\fM}([E])$ on
$T_{[E]}\fM$ corresponds to the action of $\Aut(E)$ on $\Ext^1(E,E)$
by $\ga:\ep\mapsto\ga\ci\ep\ci\ga^{-1}$.

Since $\Iso_{\fM}([E])$ is the group of invertible elements in the
finite-dimensional $\K$-algebra $\Aut(E)$, it is an affine
$\K$-group. Hence the Artin $\K$-stack $\fM$ has {\it affine
geometric stabilizers},\index{Artin stack!affine geometric
stabilizers} in the sense of Definition \ref{dt2def1}. If $S$ is a
$\K$-scheme, then 1-morphisms $\phi:S\ra\fM$ are just families of
coherent sheaves on $X$ parametrized by $S$, that is, they are
coherent sheaves $E_S$ on $X\times S$ flat over $S$. A 1-morphism
$\phi:S\ra\fM$ is an {\it atlas\/}\index{Artin stack!atlas} for some
open substack $\fU\subset\fM$, if and only if $E_S$ is a {\it versal
family\/}\index{versal family} of sheaves such
that~$\bigl\{[E_s]:s\in S(\K)\bigr\}=\fU(\K)\subseteq \fM(\K)$.

For the algebraic $\K$-spaces $\M_\rsi,\Vect_\rsi$, elements of
$\M_\rsi(\K),\Vect_\rsi(\K)$ are isomorphism classes $[E]$ of simple
coherent sheaves or vector bundles $E$. When $\K=\C$,
$\M_\rsi,\Vect_\rsi$ are {\it complex algebraic spaces},\index{complex
algebraic space} and so $\M_\rsi(\C),\Vect_\rsi(\C)$ have the
induced structure of {\it complex analytic spaces}.\index{complex
analytic space|(} Direct constructions of $\Vect_\rsi(\C)$ as a
complex analytic space parametrizing complex analytic holomorphic
vector bundles are given by L\"ubke and Okonek \cite{LuOk} and
Kosarew and Okonek \cite{KoOk}. Miyajima \cite[Th.~3]{Miya} shows
that these complex analytic space structures on $\Vect_\rsi(\C)$
coming from the algebraic side \cite{AlKl} and the analytic side
\cite{KoOk,LuOk} are equivalent.

Our first result works for Calabi--Yau $m$-folds $X$ of any
dimension $m\ge 1$, and over any algebraically closed field $\K$. It
is proved in \S\ref{dt8}. The authors are grateful to Tom Bridgeland
for suggesting the approach used to prove Theorem~\ref{dt5thm1}.

\begin{thm} Let\/ $\K$ be an algebraically closed field, and\/
$X$ a projective Calabi--Yau $m$-fold over $\K$ for $m\ge 1,$ with\/
$H^i(\cO_X)=0$ for $0<i<m,$ and\/ $\fM,\fVect,\M_\rsi,\Vect_\rsi$ be
as above. Let\/ $\mathfrak U$ be an open, finite type substack of\/
$\fM$. Then there exists an open substack $\mathfrak V$ in\/
$\fVect,$ and a $1$-isomorphism $\vp:{\mathfrak U}\ra{\mathfrak V}$
of Artin $\K$-stacks. Similarly, let\/ $U$ be an open, finite type
subspace of\/ $\M_\rsi$. Then there exists an open subspace $V$ in\/
$\Vect_\rsi$ and an isomorphism $\psi:U\ra V$ of algebraic\/
$\K$-spaces. That is, $\fM,\M_\rsi$ are locally isomorphic to
$\fVect,\Vect_\rsi,$ in the Zariski topology.\index{Zariski
topology} The isomorphisms $\vp,\psi$ are constructed as the
composition of\/ $m$ Seidel--Thomas twists by $\cO_X(-n)$ for\/
$n\gg 0$ and a shift\/~$[-m]$.
\label{dt5thm1}
\end{thm}

We now restrict to Calabi--Yau 3-folds over $\C$, which includes the
assumption $H^1(\cO_X)=0$. Our next two results, Theorems
\ref{dt5thm2} and \ref{dt5thm3}, are proved in \S\ref{dt9}. Roughly,
they say that moduli spaces of coherent sheaves on Calabi--Yau
3-folds over $\C$ can be written locally in the form $\Crit(f)$, for
$f$ a holomorphic function on a complex manifold. This is a partial
answer to the question of Behrend quoted at the beginning of
\S\ref{dt44}. Because of Theorems \ref{dt5thm2} and \ref{dt5thm3},
we can use the Milnor fibre\index{Milnor fibre} formula for the
Behrend function of $\Crit(f)$ in \S\ref{dt42} to study the Behrend
function $\nu_\fM$, and this will be vital in proving Theorem
\ref{dt5thm4}.

\begin{thm} Let\/ $X$ be a Calabi--Yau $3$-fold over\/ $\C,$ and\/
$\M_\rsi$ the coarse moduli space of simple coherent sheaves on\/
$X,$ so that\/ $\M_\rsi(\C)$ is the set of isomorphism classes $[E]$
of simple coherent sheaves $E$ on $X,$ and is a complex analytic
space. Then for each\/ $[E]\in\M_\rsi(\C)$ there exists a
finite-dimensional complex manifold\/ $U,$ a holomorphic function
$f:U\ra\C,$ and a point\/ $u\in U$ with\/ $f(u)=\rd f\vert_u=0,$
such that\/ $\M_\rsi(\C)$ near\/ $[E]$ is locally isomorphic as a
complex analytic space to\/ $\Crit(f)$ near\/ $u$. We can take\/ $U$
to be an open neighbourhood of\/ $u=0$ in the finite-dimensional
complex vector space~$\Ext^1(E,E)$.
\label{dt5thm2}
\end{thm}\index{coherent sheaf!simple|)}

Our next result generalizes Theorem \ref{dt5thm2} from simple to
arbitrary coherent sheaves, and from algebraic spaces to Artin
stacks.

\begin{thm} Let\/ $X$ be a Calabi--Yau $3$-fold over\/ $\C,$ and\/
$\fM$ the moduli stack of coherent sheaves on\/ $X$. Suppose\/ $E$
is a coherent sheaf on\/ $X,$ so that\/ $[E]\in\fM(\C)$. Let\/ $G$
be a maximal compact subgroup\index{maximal compact subgroup} in
$\Aut(E),$ and\/ $G^{\sst\C}$ its complexification. Then\/
$G^{\sst\C}$ is an algebraic $\C$-subgroup of\/ $\Aut(E),$ a maximal
reductive subgroup,\index{maximal reductive subgroup} and\/
$G^{\sst\C}=\Aut(E)$ if and only if\/ $\Aut(E)$ is reductive.

There exists a quasiprojective $\C$-scheme $S,$ an action of\/
$G^{\sst\C}$ on $S,$ a point\/ $s\in S(\C)$ fixed by $G^{\sst\C},$
and a $1$-morphism of Artin $\C$-stacks $\Phi:[S/G^{\sst\C}]\ra\fM,$
which is smooth of relative dimension $\dim\Aut(E)-\dim G^{\sst\C},$
where $[S/G^{\sst\C}]$ is the quotient stack, such that\/
$\Phi(s\,G^{\sst\C})=[E],$ the induced morphism on stabilizer groups
$\Phi_*:\Iso_{[S/G^{\sst\C}]}(s\,G^{\sst\C})\ra\Iso_{\fM}([E])$ is
the natural morphism $G^{\sst\C}\hookra\Aut(E)\cong\Iso_{\fM}([E]),$
and\/ $\rd\Phi\vert_{s\,G^{\sst\C}}:T_sS\cong T_{s\,G^{\sst\C}}
[S/G^{\sst\C}]\ra T_{[E]}\fM\cong \Ext^1(E,E)$ is an isomorphism.
Furthermore, $S$ parametrizes a formally versal family $(S,{\cal
D})$ of coherent sheaves on $X,$ equivariant under the action of\/
$G^{\sst\C}$ on $S,$ with fibre\/ ${\cal D}_s\cong E$ at\/ $s$. If\/
$\Aut(E)$ is reductive then $\Phi$ is \'etale.

Write $S_\an$ for the complex analytic space underlying the
$\C$-scheme $S$. Then there exists an open neighbourhood\/ $U$ of\/
$0$ in\/ $\Ext^1(E,E)$ in the analytic topology, a holomorphic
function $f:U\ra\C$ with\/ $f(0)=\rd f\vert_0=0,$ an open
neighbourhood\/ $V$ of\/ $s$ in $S_\an,$ and an isomorphism of
complex analytic spaces $\Xi:\Crit(f)\ra V,$ such that\/ $\Xi(0)=s$
and\/ $\rd\Xi\vert_0:T_0\Crit(f)\ra T_sV$ is the inverse of\/
$\rd\Phi\vert_{s\,G^{\sst\C}}:T_sS\ra\Ext^1(E,E)$. Moreover we can
choose $U,f,V$ to be $G^{\sst\C}$-invariant, and\/ $\Xi$ to be
$G^{\sst\C}$-equivariant.
\label{dt5thm3}
\end{thm}

Here the first paragraph is immediate, and the second has a
straightforward proof in \S\ref{dt93}, similar to parts of the proof
of Luna's Etale Slice Theorem\index{Luna's Etale Slice Theorem}
\cite{Luna}; the case in which $\Aut(E)$ is reductive, so that
$G^{\sst\C}=\Aut(E)$ and $\Phi$ is \'etale, is a fairly direct
consequence of the Etale Slice Theorem. The third paragraph is what
takes the hard work in the proof. Composing the projection
$\pi:S\ra[S/G^{\sst\C}]$ with $\Phi$ gives a smooth 1-morphism
$\Phi\ci\pi:S\ra\fM$, which is locally an atlas for $\fM$ near
$[E]$. Thus, Theorem \ref{dt5thm3} says that we can write an atlas
for $\fM$ in the form $\Crit(f)$, locally in the analytic topology,
where $f:U\ra\C$ is a holomorphic function on a complex manifold.

By Theorem \ref{dt5thm1}, it suffices to prove Theorems
\ref{dt5thm2} and \ref{dt5thm3} with $\Vect_\rsi,\fVect$ in place of
$\M_\rsi,\fM$. We do this using gauge theory, motivated by an idea
of Donaldson and Thomas \cite[\S 3]{DoTh}, \cite[\S 2]{Thom}. Let
$E\ra X$ be a fixed complex (not holomorphic) vector bundle over
$X$. Write $\sA$ for the infinite-dimensional affine space of smooth
semiconnections ($\db$-operators) on $E$, and $\sAs$ for the open
subset of simple semiconnections, and $\sG$ for the
infinite-dimensional Lie group of smooth gauge transformations of
$E$. Note that we do not assume semiconnections are integrable. Then
$\sG$ acts on $\sA$ and $\sAs$, and $\sB=\sA/\sG$ is the space of
gauge-equivalence classes of semiconnections on~$E$.

The subspace $\sB_\rsi=\sAs/\sG$ of simple semiconnections should be
an infinite-dimensional complex manifold. Inside $\sB_\rsi$ is the
subspace ${\mathscr V}_\rsi$ of integrable simple semiconnections,
which should be a finite-dimensional complex analytic space. Now the
moduli scheme $\Vect_\rsi$ of simple complex algebraic vector
bundles has an underlying complex analytic space $\Vect_\rsi(\C)$;
the idea is that ${\mathscr V}_\rsi$ is naturally isomorphic as a
complex analytic space to the open subset of $\Vect_\rsi(\C)$ of
algebraic vector bundles with underlying complex vector bundle~$E$.

We fix $\db_E$ in $\sA$ coming from a holomorphic vector bundle
structure on $E$. Then points in $\sA$ are of the form $\db_E+A$ for
$A\in C^\iy\bigl(\End(E)\ot_\C \La^{0,1}T^*X\bigr)$, and $\db_E+A$
makes $E$ into a holomorphic vector bundle if $F_A^{0,2}=\db_EA+A\w
A$ is zero in $\smash{C^\iy\bigl(\End(E)\ot_\C
\La^{0,2}T^*X\bigr)}$. Thus, the moduli space of holomorphic vector
bundle structures on $E$ is isomorphic to $\{\db_E+A\in\sA:
F_A^{0,2}=0\}/\sG$. Thomas observes that when $X$ is a Calabi--Yau
3-fold, there is a natural holomorphic function $CS:\sA\ra\C$ called
the {\it holomorphic Chern--Simons functional}, invariant under
$\sG$ up to addition of constants, such that
$\{\db_E+A\in\sA:F_A^{0,2}=0\}$ is the critical locus of $CS$. Thus,
${\mathscr V}_\rsi$, and hence $(\Vect_\rsi)(\C)$, is (informally)
locally the set of critical points of a holomorphic function $CS$ on
an infinite-dimensional complex manifold~$\sB_\rsi$.

In the proof of Theorem \ref{dt5thm2} in \S\ref{dt9}, when $\db_E$
is simple, we show using results of Miyajima \cite{Miya} that there
is a finite-dimensional complex submanifold $Q_\ep$ of $\sA$
containing $\db_E$, such that $\Vect_\rsi(\C)$ near $[(E,\db_E)]$ is
isomorphic as a complex analytic space to $\Crit(CS\vert_{Q_\ep})$
near $\db_E$, where $CS\vert_{Q_\ep}: Q_\ep\ra\C$ is a holomorphic
function on the finite-dimensional complex manifold $Q_\ep$. We also
show $Q_\ep$ is biholomorphic to an open neighbourhood $U$ of 0
in~$\Ext^1(E,E)$.

In the proof of Theorem \ref{dt5thm3} in \S\ref{dt9}, without
assuming $\db_E$ simple, we show that a local atlas $S$ for $\fVect$
near $[(E,\db_E)]$ is isomorphic as a complex analytic space to
$\Crit(CS\vert_{Q_\ep})$ near $\db_E$. The new issues in Theorem
\ref{dt5thm3} concern to what extent we can take $S,Q_\ep$ and
$CS\vert_{Q_\ep}:Q_\ep\ra\C$ to be invariant under $\Aut(E,\db_E)$.
In fact, in Theorem \ref{dt5thm3} we would have preferred to take
$S,U,V,f$ invariant under the full group $\Aut(E)$, rather than just
under the maximal reductive subgroup $G^{\sst\C}$. But we expect
this is not possible.

On the algebraic geometry side, the choice of $S,\Phi,\cal D$ in the
second paragraph of Theorem \ref{dt5thm3}, to construct $S$ we use
ideas from the proof of Luna's Etale Slice Theorem\index{Luna's
Etale Slice Theorem} \cite{Luna}, which works only for reductive
groups, so we can make $S$ invariant under at most a maximal
reductive subgroup $G^{\sst\C}$ in $\Aut(E)$. On the gauge theory
side, constructing $Q_\ep$ involves a {\it slice\/}
$\sS_E=\{\db_E+A:\db_E^*A=0\}$ to the action of $\sG$ in $\sA$ at
$\db_E\in\sA$, where $\db_E^*$ is defined using choices of Hermitian
metrics $h_X,h_E$ on $X$ and $E$. In general we cannot make $\sS_E$
invariant under $\Aut(E,\db_E)$. The best we can do is to choose
$h_E$ invariant under a maximal compact subgroup $G$ of
$\Aut(E,\db_E)$. Then $\sS_E$ is invariant under $G$, and hence
under $G^{\sst\C}$ as $\sS_E$ is a closed complex submanifold.

We can improve the group-invariance in Theorem \ref{dt5thm3} if we
restrict to moduli stacks of {\it semistable\/} sheaves.\index{coherent
sheaf!semistable}

\begin{cor} Let\/ $X$ be a Calabi--Yau $3$-fold over\/ $\C$. Write
$\tau$ for Gieseker stability of coherent sheaves on $X$ w.r.t.\
some ample line bundle $\cO_X(1),$ and\/ $\fM_\rss^\al(\tau)$ for
the moduli stack of\/ $\tau$-semistable sheaves with Chern character
$\al$. It is an open Artin\/ $\C$-substack of\/~$\fM$.

Then for each\/ $[E]\in\fM_\rss^\al(\tau)(\C),$ there exists an
affine $\C$-scheme $S$ with associated complex analytic space
$S_\an,$ a point $s\in S_\an,$ a reductive affine algebraic
$\C$-group $H$ acting on $S,$ an \'etale morphism
$\Phi:[S/H]\ra\fM_\rss^\al(\tau)$ mapping $H\cdot s\mapsto [E],$ a
finite-dimensional complex manifold\/ $U$ with a holomorphic action
of\/ $H,$ an $H$-invariant holomorphic function $f:U\ra\C,$ an
$H$-invariant open neighbourhood $V$ of\/ $s$ in $S_\an$ in the
analytic topology, and an $H$-equivariant isomorphism of complex
analytic spaces~$\Xi:\Crit(f)\ra V$.
\label{dt5cor1}
\end{cor}

\begin{proof} Let $[E]\in\fM_\rss^\al(\tau)(\C)$. Then by properties
of Gieseker stability, $E$ has a Jordan--H\"older decomposition into
pairwise non-isomorphic stable factors $E_1,\ab\ldots,\ab E_k$ with
multiplicities $m_1,\ldots,m_k$ respectively, and $E$ is an
arbitrarily small deformation of $E'=m_1E_1\op\cdots\op m_kE_k$. We
have $\Hom(E_i,E_j)=0$ if $i\ne j$ and $\Hom(E_i,E_i)=\C$. Thus
$\Aut(E')\cong\prod_{i=1}^k\GL(m_i,\C)$, which is the
complexification of its maximal compact subgroup
$\prod_{i=1}^k\U(m_i)$. Applying Theorem \ref{dt5thm3} to $E'$ with
$G=\prod_{i=1}^k\U(m_i)$ and $G^{\sst\C}=\Aut(E')$ gives
$S,H=G^{\sst\C},\Phi,U,f,V,\ab\Xi$. Since $E$ is an arbitrarily
small deformation of $E'$ and $\Phi$ is \'etale with $\Phi_*:[H\cdot
0]\mapsto[E']$, $[E]$ lies in the image under $\Phi_*$ of any open
neighbourhood of $[H\cdot 0]$ in $[S/H](\C)$, and thus $[E]$ lies in
the image of any $H$-invariant open neighbourhood $V$ of $0$ in
$S_\an$, in the analytic topology. Hence there exists $s\in
V\subseteq S_\an$ with $\Phi(H\cdot s)=[E]$. The corollary follows.
\end{proof}\index{complex analytic space|)}

We can connect the last three results with the ideas on {\it
perverse sheaves and vanishing cycles\/} sketched in \S\ref{dt42}.
The first author would like to thank Kai Behrend, Jim Bryan and
Bal\'{a}zs Szendr\H oi for explaining the following ideas. Theorem
\ref{dt5thm2} proves that the complex algebraic space $\M_\rsi$ may
be written locally in the complex analytic topology as $\Crit(f)$,
for $f:U\ra\C$ holomorphic and $U$ a complex manifold. Therefore
Theorem \ref{dt4thm3} shows that locally in the complex analytic
topology, there is a {\it perverse sheaf of vanishing
cycles\/}\index{vanishing cycle!perverse sheaf}\index{perverse sheaf!of
vanishing cycles} $\phi_f(\underline{\Q}[\dim U-1])$ supported on
$\Crit(f)\cong\M_\rsi$, which projects to $\nu_{\M_\rsi}$ under
$\chi_{U_0}$. So it is natural to ask whether we can glue these to
get a global perverse sheaf on~$\M_\rsi$:

\begin{quest}{\bf(a)} Let\/ $X$ be a Calabi--Yau\/ $3$-fold over\/
$\C,$ and write\/ $\M_\rsi$ for the coarse moduli space of simple
coherent sheaves on\/ $X$. Does there exist a natural perverse
sheaf\/ $\cal P$ on $\M_\rsi,$ with\/ $\chi_{\M_\rsi}({\cal
P})=\nu_{\M_\rsi},$ which is locally isomorphic to
$\phi_f(\underline{\Q}[\dim U-1])$ for $f,U$ as in Theorem
{\rm\ref{dt5thm2}?}
\smallskip

\noindent{\bf(b)} Is there also some Artin stack version of\/ $\cal
P$ in\/ {\bf(a)} for the moduli stack\/ $\fM,$ locally isomorphic to
$\phi_f(\underline{\Q}[\dim U-1])$ for $f,U$ as in Theorem
{\rm\ref{dt5thm3}?}
\label{dt5quest1}
\end{quest}

The authors have no particular view on whether the answer is yes or
no. One can also ask Question \ref{dt5quest1} for Saito's mixed
Hodge modules~\cite{Sait}.\index{mixed Hodge module}

\begin{rem}{\bf(i)} Question \ref{dt5quest1}(a) could be tested by
calculation in examples, such as the Hilbert scheme\index{Hilbert
scheme} of $n$ points on $X$. Partial results in this case can be
found in Dimca and Szendr\H oi \cite{DiSz} and Behrend, Bryan and
Szendr\H oi \cite{BBS}, see in particular~\cite[Rem.~3.2]{BBS}.
\smallskip

\noindent{\bf(ii)} If the answer to Question \ref{dt5quest1}(a) is
yes, it would provide a way of {\it categorifying\/} (conventional)
Donaldson--Thomas invariants $DT^\al(\tau)$. That is, if $\al\in
K(\coh(X))$ with $\M_\rss^\al(\tau)=\M_\st^\al(\tau)$, as in
\S\ref{dt43}, then we can restrict $\cal P$ in Question
\ref{dt5quest1}(a) to a perverse sheaf on the open, proper subscheme
$\M_\st^\al(\tau)$ in $\M_\rsi$, and form the {\it
hypercohomology\/}\index{perverse sheaf!hypercohomology} ${\mathbb
H}^*\bigl(\M_\st^\al(\tau);{\cal
P}\vert_{\M_\st^\al(\tau)}\bigr)$,\nomenclature[HzP]{${\mathbb
H}^*(\M;{\cal P})$}{hypercohomology of a perverse sheaf $\cal P$ on
a scheme $\M$} which is a finite-dimensional graded $\Q$-vector
space. Then
\e
\begin{split}
\sum_{k\in\Z}(-1)^k\dim{\mathbb H}^k\bigl(\M_\st^\al(\tau);{\cal
P}\vert_{\M_\st^\al(\tau)}\bigr)&=\chi\bigl(\M_\st^\al(\tau),
\chi_{\M_\rsi}({\cal P})\vert_{\M_\st^\al(\tau)}\bigr)\\[-6pt]
=\chi\bigl(\M_\st^\al(\tau),\nu_{\M_\rsi}\vert_{\M_\st^\al(\tau)}\bigr)
&=\chi\bigl(\M_\st^\al(\tau),\nu_{\M_\st^\al(\tau)}\bigr)=DT^\al(\tau),
\end{split}
\label{dt5eq1}
\e
where the first equality in \eq{dt5eq1} holds because we have a
commutative diagram
\begin{equation*}
\xymatrix@C=50pt@R=10pt{ D^b_\Con(\M_\st^\al(\tau)) \ar[r]_{R\pi_*}
\ar[d]^{\chi_{\M_\st^\al(\tau)}} &
D^b_\Con(\Spec\C) \ar[d]_{\chi_{\Spec\C}} \\
\CF^\an_\Z(\M_\st^\al(\tau)) \ar[r]^{\CF(\pi)} &
\CF^\an_\Z(\Spec\C).}
\end{equation*}
by \eq{dt4eq4}, where $\pi:\M_\st^\al(\tau)\ra\Spec\C$ is the
projection, which is proper as $\M_\st^\al(\tau)$ is proper, and the
last equality in \eq{dt5eq1} holds by~\eq{dt4eq16}.

Thus, ${\mathbb H}^*\bigl(\M_\st^\al(\tau);{\cal
P}\vert_{\M_\st^\al(\tau)}\bigr)$ would be a natural cohomology
group of $\M_\st^\al(\tau)$ whose Euler characteristic is the
Donaldson--Thomas invariant by \eq{dt5eq1}; the Poin\-car\'e
polynomial of ${\mathbb H}^*\bigl(\M_\st^\al(\tau);{\cal
P}\vert_{\M_\st^\al(\tau)}\bigr)$ would be a lift of $DT^\al(\tau)$
to $\Z[t,t^{-1}]$, which might also be interesting.
\smallskip

\noindent{\bf(iii)} If the answers to Question
\ref{dt5quest1}(a),(b) are no, at least locally in the Zariski
topology,\index{Zariski topology} this might be bad news for the
programme of Kontsevich--Soibelman \cite{KoSo1} to extend
Donaldson--Thomas invariants of Calabi--Yau 3-folds to other motivic
invariants.\index{motivic invariant} Kontsevich and Soibelman wish to
associate a `motivic Milnor fibre'\index{Milnor fibre!motivic} to each
point of $\fM$. The question of how these vary in families over the
base $\fM$ is important, but not really addressed in \cite{KoSo1}.
It appears to the authors to be a similar issue to whether one can
glue perverse sheaves above; indeed, $\cal P$ in Question
\ref{dt5quest1} may be some kind of cohomology pushforward of the
Kontsevich--Soibelman family of motivic Milnor fibres, if this
exists.
\label{dt5rem3}
\end{rem}

The last three results use transcendental complex analysis, and so
work only over $\C$. It is an important question whether analogous
results can be proved using strictly algebraic methods, and over
fields $\K$ other than $\C$. Observe that above we locally write
$\M_\rsi$ as $\Crit(f)$ for $f:U\ra\C$, that is, we write $\M_\rsi$
as the zeroes $(\rd f)^{-1}(0)$ of a closed 1-form $\rd f$ on a
smooth complex manifold $U$. A promising way to generalize Theorems
\ref{dt5thm2}--\ref{dt5thm3} to the algebraic context is to replace
$\rd f$ by an {\it almost closed\/ $1$-form\/} $\om$,\index{almost
closed 1-form} in the sense of~\S\ref{dt44}.

Results of Thomas \cite{Thom} imply that the coarse moduli space of
simple coherent sheaves $\M_\rsi$ on $X$ carries a symmetric
obstruction theory,\index{symmetric obstruction theory}\index{obstruction
theory!symmetric} and thus Proposition \ref{dt4prop4} shows that
$\M_\rsi$ is locally isomorphic to the zeroes of an almost closed
1-form $\om$ on a smooth variety $U$. Etale locally near
$[E]\in\M_\rsi(\K)$ we can take $U$ to be $\Ext^1(E,E)$. Thus we
deduce:

\begin{prop} Let\/ $\K$ be an algebraically closed field and\/ $X$
a Calabi--Yau $3$-fold over\/ $\K,$ and write\/ $\M_\rsi$ for the
coarse moduli space of simple coherent sheaves on\/ $X,$ which is an
algebraic $\K$-space. Then for each point\/ $[E]\in\M_\rsi(\K)$
there exists a Zariski open subset\/ $U$ in the affine $\K$-space
$\Ext^1(E,E)$ with\/ $0\in U(\K),$ an algebraic almost closed
$1$-form\/ $\om$ on $U$ with\/ $\om\vert_0=\pd\om\vert_0=0,$ and an
\'etale morphism $\xi:\om^{-1}(0)\ra\M_\rsi$ with\/
$\xi(0)=[E]\in\M_\rsi(\K)$ and\/ $\rd\xi\vert_0:T_0(\om^{-1}(0))=
\Ext^1(E,E)\ra T_{[E]}\M_\rsi$ the natural isomorphism, where
$\om^{-1}(0)$ is the $\K$-subscheme of\/ $U$ on which\/~$\om\equiv
0$.
\label{dt5prop1}
\end{prop}

This is an analogue of Theorem \ref{dt5thm2}, with $\C$ replaced by
any algebraically closed $\K$, the complex analytic topology
replaced by the \'etale topology,\index{etale topology@\'etale
topology} and the closed 1-form $\rd f$ replaced by the almost
closed 1-form $\om$. We can ask whether there is a corresponding
algebraic analogue of Theorem~\ref{dt5thm3}.

\begin{quest} Let\/ $\K$ be an algebraically closed field and\/ $X$
a Calabi--Yau $3$-fold over\/ $\K,$ and write\/ $\fM$ for the moduli
stack of coherent sheaves on\/~$X$.
\begin{itemize}
\setlength{\itemsep}{0pt}
\setlength{\parsep}{0pt}
\item[{\bf(a)}] For each\/ $[E]\in\fM(\K),$ does there
exist a Zariski open subset\/ $U$ in the affine $\K$-space
$\Ext^1(E,E)$ with\/ $0\in U(\K),$ an algebraic almost closed
$1$-form\/ $\om$ on $U$ with\/ $\om\vert_0=\pd\om\vert_0=0,$ and
a $1$-morphism $\xi:\om^{-1}(0)\ra\fM$ smooth of relative
dimension $\dim\Aut(E),$ with\/ $\xi(0)=[E]\in\fM(\K)$ and\/
$\rd\xi\vert_0:T_0(\om^{-1}(0))=\Ext^1(E,E)\ra T_{[E]}\fM$ the
natural isomorphism?
\item[{\bf(b)}] In\/ {\bf(a)}{\rm,} let\/ $G$ be a maximal reductive
subgroup of\/ $\Aut(E),$ acting on $\Ext^1(E,E)$ by
$\ga:\ep\!\mapsto\!\ga\!\ci\!\ep\!\ci\!\ga^{-1}$. Can we take
$U,\om,\xi$ to be $G$-invariant?
\end{itemize}
\label{dt5quest2}
\end{quest}

\subsection{Identities on Behrend functions of moduli stacks}
\label{dt52}
\index{Behrend function|(}

We use the results of \S\ref{dt51} to study the {\it Behrend
function\/} $\nu_{\fM}$ of the moduli stack $\fM$ of coherent
sheaves on $X$, as in \S\ref{dt4}. Our next theorem is proved
in~\S\ref{dt10}.

\begin{thm} Let\/ $X$ be a Calabi--Yau $3$-fold over\/ $\C,$ and\/
$\fM$ the moduli stack of coherent sheaves on\/ $X$. The
\begin{bfseries}Behrend function\end{bfseries}\/ $\nu_{\fM}:
\fM(\C)\ra\Z$ is a natural locally constructible
function\index{constructible function!locally} on $\fM$. For all\/
$E_1,E_2\in\coh(X),$ it satisfies:\index{Behrend function!identities}
\begin{gather}
\nu_{\fM}(E_1\op E_2)=(-1)^{\bar\chi([E_1],[E_2])}
\nu_{\fM}(E_1)\nu_{\fM}(E_2),
\label{dt5eq2}\\
\begin{split}
\int_{\begin{subarray}{l}[\la]\in\mathbb{P}(\Ext^1(E_2,E_1)):\\
\la\; \Leftrightarrow\; 0\ra E_1\ra F\ra E_2\ra
0\end{subarray}}\!\!\!\!\!\!\nu_{\fM}(F)\,\rd\chi -
\int_{\begin{subarray}{l}[\ti\la]\in\mathbb{P}(\Ext^1(E_1,E_2)):\\
\ti\la\; \Leftrightarrow\; 0\ra E_2\ra\ti F\ra E_1\ra
0\end{subarray}}\!\!\!\!\!\!\nu_{\fM}(\ti F)\,\rd\chi \\
=\bigl(\dim\Ext^1(E_2,E_1)-\dim\Ext^1(E_1,E_2)\bigr)
\nu_{\fM}(E_1\op E_2).
\end{split}
\label{dt5eq3}
\end{gather}
Here\/ $\bar\chi([E_1],[E_2])$ in \eq{dt5eq2} is defined in
{\rm\eq{dt3eq1},} and in \eq{dt5eq3} the correspondence between\/
$[\la]\in\mathbb{P}(\Ext^1(E_2,E_1))$ and\/ $F\in\coh(X)$ is that\/
$[\la]\in\mathbb{P}(\Ext^1(E_2,E_1))$ lifts to some\/
$0\ne\la\in\Ext^1(E_2,E_1),$ which corresponds to a short exact
sequence\/ $0\ra E_1\ra F\ra E_2\ra 0$ in\/ $\coh(X)$ in the usual
way. The function $[\la]\mapsto\nu_{\fM}(F)$ is a constructible
function\/ $\mathbb{P}(\Ext^1(E_2,E_1))\ra\Z,$ and the integrals in
\eq{dt5eq3} are integrals of constructible functions using the Euler
characteristic as measure.
\label{dt5thm4}
\end{thm}

We will prove Theorem \ref{dt5thm4} using Theorem \ref{dt5thm3} and
the Milnor fibre\index{Milnor fibre} description of Behrend
functions from \S\ref{dt42}. We apply Theorem \ref{dt5thm3} to
$E=E_1\op E_2$, and we take the maximal compact subgroup $G$ of
$\Aut(E)$ to contain the subgroup $\bigl\{\id_{E_1}+\la\id_{E_2}:
\la\in\U(1)\bigr\}$, so that $G^{\sst\C}$ contains
$\bigl\{\id_{E_1}+\la\id_{E_2}:\la\in\bG_m\bigr\}$. Equations
\eq{dt5eq2} and \eq{dt5eq3} are proved by a kind of localization
using this $\bG_m$-action on~$\Ext^1(E_1\op E_2,E_1\op E_2)$.

Equations \eq{dt5eq2}--\eq{dt5eq3} are related to a conjecture of
Kontsevich and Soibelman \cite[Conj.~4]{KoSo1} and its application
in \cite[\S 6.3]{KoSo1}, and could probably be deduced from it. The
authors got the idea of proving \eq{dt5eq2}--\eq{dt5eq3} by
localization using the $\bG_m$-action on $\Ext^1(E_1\op E_2, E_1\op
E_2)$ from \cite{KoSo1}. However, Kontsevich and Soibelman approach
\cite[Conj.~4]{KoSo1} via formal power series and non-Archimedean
geometry. Instead, in Theorem \ref{dt5thm3} we in effect first prove
that we can choose the formal power series to be convergent, and
then use ordinary differential geometry and Milnor fibres.

Note that our proof of Theorem \ref{dt5thm4} is {\it not wholly
algebro-geometric\/} --- it uses gauge theory,\index{gauge theory} and
transcendental complex analytic geometry methods. Therefore this
method will not suffice to prove the parallel conjectures in
Kontsevich and Soibelman \cite[Conj.~4]{KoSo1}, which are supposed
to hold for general fields $\K$ as well as $\C$, and for general
motivic invariants of algebraic $\K$-schemes as well as for the
topological Euler characteristic.

\begin{quest}{\bf(a)} Suppose the answers to Questions\/
{\rm\ref{dt4quest}(a)} and\/ {\rm\ref{dt5quest2}} are both yes. Can
one use these to give an alternative, strictly algebraic proof of
Theorem {\rm\ref{dt5thm4}} using \begin{bfseries}almost closed\/
$1$-forms\end{bfseries}\index{almost closed 1-form} as in
\S{\rm\ref{dt44},} either over\/ $\K=\C$ using the linking number
expression for Behrend functions in\/ {\rm\eq{dt4eq17},} or over
general algebraically closed\/ $\K$ of characteristic zero by some
other means?
\smallskip

\noindent{\bf(b)} Might the ideas of\/ {\bf(a)} provide an approach
to proving {\rm\cite[Conj.~4]{KoSo1}} without using formal power
series methods?
\smallskip

\noindent{\bf(c)} Can one extend Theorem\/ {\rm\ref{dt5thm4}} from
the abelian category\/ $\coh(X)$ to the derived
category\index{derived category} $D^b(X),$ say to all objects\/
$E_1\!\op\!E_2$ in $D(X)$ with\/ $\Ext^{<0}(E_1\!\op\!
E_2,E_1\!\op\! E_2)\!=\!0$?
\label{dt5quest3}
\end{quest}\index{Behrend function|)}

\subsection[A Lie algebra morphism, and generalized D--T invariants]{A
Lie algebra morphism $\ti\Psi:\SFai(\fM)\ra\ti L(X),$ and \\
generalized Donaldson--Thomas invariants $\bar{DT}{}^\al(\tau)$}
\label{dt53}

In \S\ref{dt34} we defined an explicit Lie algebra $L(X)$ and Lie
algebra morphisms $\Psi:\SFai(\fM)\ra L(X)$ and
$\Psi^{\chi,\Q}:\oSFai (\fM,\chi,\Q)\ra L(X)$. We now define
modified versions $\ti L(X),\ti\Psi,\ti\Psi^{\chi,\Q}$, with
$\ti\Psi,\ti\Psi^{\chi,\Q}$ weighted by the Behrend function
$\nu_\fM$ of $\fM$. We continue to use the notation
of~\S\ref{dt2}--\S\ref{dt4}.

\begin{dfn} Define a Lie algebra $\ti L(X)$ to be the $\Q$-vector
space with basis of symbols $\ti \la^\al$ for $\al\in
K^\num(\coh(X))$, with Lie bracket\nomenclature[L(X)a]{$\ti L(X)$}{Lie algebra
depending on a Calabi--Yau 3-fold $X$, variant of $L(X)$}\nomenclature[\lambda
a]{$\ti\lambda^\al$}{basis element of Lie algebra $\ti L(X)$}
\e
[\ti\la^\al,\ti \la^\be]=(-1)^{\bar\chi(\al,\be)}\bar\chi(\al,\be)
\ti\la^{\al+\be},
\label{dt5eq4}
\e
which is \eq{dt3eq15} with a sign change. As $\bar\chi$ is
antisymmetric, \eq{dt5eq4} satisfies the Jacobi identity, and makes
$\ti L(X)$ into an infinite-dimensional Lie algebra over~$\Q$.

Define a $\Q$-linear map $\ti\Psi^{\chi,\Q}:\oSFai
(\fM,\chi,\Q)\ra\ti L(X)$ by\nomenclature[\Psi d]{$\ti\Psi^{\chi,\Q}$}{Lie
algebra morphism $\oSFai(\fM,\chi,\Q)\ra \ti L(X)$}
\begin{equation*}
\ti\Psi^{\chi,\Q}(f)=\ts\sum_{\al\in K^\num(\coh(X))}\ga^\al\ti
\la^{\al},
\end{equation*}
as in \eq{dt3eq16}, where $\ga^\al\in\Q$ is defined as follows.
Write $f\vert_{\fM^\al}$ in terms of $\de_i,U_i,\rho_i$ as in
\eq{dt3eq17}, and set
\e
\ga^\al=\ts\sum_{i=1}^n\de_i\chi\bigl(U_i,\rho_i^*(\nu_\fM)\bigr),
\label{dt5eq5}
\e
where $\rho_i^*(\nu_\fM)$ is the pullback of the Behrend function
$\nu_\fM$ to a constructible function\index{constructible function}
on $U_i\times[\Spec\C/\bG_m]$, or equivalently on $U_i$, and
$\chi\bigl(U_i,\rho_i^*(\nu_\fM)\bigr)$ is the Euler characteristic
of $U_i$ weighted by $\rho_i^*(\nu_\fM)$. One can show that the map
from \eq{dt3eq17} to \eq{dt5eq5} is compatible with the relations in
$\oSFai(\fM^\al,\chi,\Q)$, and so $\ti\Psi^{\chi,\Q}$ is
well-defined. Define $\ti\Psi:\SFai(\fM)\ra\ti L(X)$
by~$\ti\Psi=\ti\Psi^{\chi,\Q}\ci\bar\Pi^{\chi,\Q}_\fM$.\nomenclature[\Psi
c]{$\ti\Psi$}{Lie algebra morphism $\SFai(\fM)\ra\ti L(X)$}

Here is an alternative way to write $\ti\Psi^{\chi,\Q},\ti\Psi$
using constructible functions. Define a $\Q$-linear map
$\Pi_{\CF}:\oSFai(\fM,\chi,\Q)\ra\CF(\fM)$ by\nomenclature[\Pi
z]{$\Pi_{\CF}$}{projection $\oSFai(\fM,\chi,\Q)\ra\CF(\fM)$}
\begin{equation*}
\Pi_{\CF}: \ts\sum_{i=1}^n\de_i[(U_i\times[\Spec\C/\bG_m],\rho_i)]
\longmapsto \ts\sum_{i=1}^n\de_i\CF^\na(\rho_i)(1_{U_i}),
\end{equation*}
where by Proposition \ref{dt3prop1} any element of
$\oSFai(\fM,\chi,\Q)$ can be written as
$\sum_{i=1}^n\de_i[(U_i\times[\Spec\C/\bG_m],\rho_i)]$ for
$\de_i\in\Q$, $U_i$ a quasiprojective $\C$-variety, and
$[(U_i\times[\Spec\C/\bG_m],\ab \rho_i)]$ with algebra
stabilizers,\index{stack function!with algebra stabilizers} and
$1_{U_i}\in\CF(U_i)$ is the function 1, and $\CF^\na(\rho_i)$ is as
in Definition \ref{dt2def3}. Then we have
\e
\begin{split}
\ti\Psi^{\chi,\Q}(f)&=\ts\sum_{\al\in K^\num(\coh(X))}\chi^\na
\bigl(\fM^\al,(\Pi_{\CF}(f)\cdot\nu_\fM)\vert_{\fM^\al}\bigr)\,
\ti\la^{\al},\\
\ti\Psi(f)&=\ts\sum_{\al\in K^\num(\coh(X))}\chi^\na
\bigl(\fM^\al,(\Pi_{\CF}\ci\bar\Pi^{\chi,\Q}_\fM(f)\cdot\nu_\fM)
\vert_{\fM^\al}\bigr)\,\ti\la^{\al}.
\end{split}
\label{dt5eq6}
\e
\label{dt5def1}
\end{dfn}

Our Lie algebra $\ti L(X)$ is essentially the same as the Lie
algebra ${\mathfrak g}_\Ga$ of Kontsevich and Soibelman \cite[\S
1.4]{KoSo1}. They also define a Lie algebra ${\mathfrak g}_V$ which
is a completion of a Lie subalgebra of ${\mathfrak g}_\Ga$, and a
pro-nilpotent Lie group $G_V$ with Lie algebra ${\mathfrak g}_V$.
Kontsevich and Soibelman express wall-crossing for Donaldson--Thomas
type invariants in terms of multiplication in the Lie group $G_V$,
whereas we do it in terms of Lie brackets in the Lie algebra $\ti
L(X)$. Applying $\exp: {\mathfrak g}_V\ra G_V$ should transform our
approach to that of Kontsevich and Soibelman.

The reason for the sign change between \eq{dt3eq15} and \eq{dt5eq4}
is the signs involved in Behrend functions, in particular, the
$(-1)^n$ in Theorem \ref{dt4thm1}(ii), which is responsible for the
factor $(-1)^{\bar\chi([E_1],[E_2])}$ in~\eq{dt5eq2}.

Here is the analogue of Theorem \ref{dt3thm4}. It is proved
in~\S\ref{dt11}.

\begin{thm} $\ti\Psi:\SFai(\fM)\ra\ti L(X)$ and\/
$\ti\Psi^{\chi,\Q}:\oSFai(\fM,\chi,\Q)\ra\ti L(X)$ are Lie algebra
morphisms.
\label{dt5thm5}
\end{thm}

Theorem \ref{dt5thm5} should be compared with Kontsevich and
Soibelman \cite[\S 6.3]{KoSo1}, which gives a conjectural
construction of an {\it algebra morphism\/} $\Phi:\SF(\fM)\ra {\cal
R}_{\smash{K(\coh(X))}}$, where ${\cal R}_{\smash{K(\coh(X))}}$ is a
certain explicit algebra. We expect our $\ti\Psi$ should be obtained
from their $\Phi$ by restricting to $\SFai(\fM)$, and obtaining $\ti
L(X)$ from a Lie subalgebra of ${\cal R}_{K(\coh(X))}$ by taking a
limit, the limit corresponding to specializing from virtual
Poincar\'e polynomials or more general motivic invariants of
$\C$-varieties to Euler characteristics.

We can now define generalized Donaldson--Thomas invariants.
\index{Donaldson--Thomas invariants!generalized
$\bar{DT}{}^\al(\tau)$|(}

\begin{dfn} Let $X$ be a projective Calabi--Yau 3-fold over $\C$,
let $\cO_X(1)$ be a very ample line bundle on $X$, and let
$(\tau,G,\le)$ be Gieseker stability\index{Gieseker stability} and
$(\mu,M,\le)$ be
$\mu$-stability\index{$\mu$-stability}\index{stability
condition!$\mu$-stability}\index{m-stability@$\mu$-stability} on
$\coh(X)$ w.r.t.\ $\cO_X(1)$, as in Examples \ref{dt3ex1} and
\ref{dt3ex2}. As in \eq{dt3eq22}, define {\it generalized
Donaldson--Thomas invariants\/}\index{Donaldson--Thomas
invariants!generalized, definition} $\bar{DT}{}^\al(\tau)\in\Q$ and
$\bar{DT}{}^\al(\mu)\in\Q$ for all $\al\in C(\coh(X))$
by\nomenclature[DTb]{$\bar{DT}{}^\al(\tau)$}{generalized
Donaldson--Thomas
invariants}\nomenclature[DTd]{$\bar{DT}{}^\al(\mu)$}{generalized
Donaldson--Thomas invariants for $\mu$-stability}
\e
\ti\Psi\bigl(\bep^\al(\tau)\bigr)=-\bar{DT}{}^\al(\tau)\ti
\la^\al\qquad\text{and}\qquad
\ti\Psi\bigl(\bep^\al(\mu)\bigr)=-\bar{DT}{}^\al(\mu)\ti \la^\al.
\label{dt5eq7}
\e
Here $\bep^\al(\tau),\bep^\al(\mu)$ are defined in \eq{dt3eq4}, and
lie in $\SFai(\fM)$ by Theorem \ref{dt3thm1}, so
$\bar{DT}{}^\al(\tau),\bar{DT}{}^\al(\mu)$ are well-defined. The
signs in \eq{dt5eq7} will be explained after Proposition
\ref{dt5prop2}. Equation \eq{dt5eq6} implies that an alternative
expression is
\e
\begin{split}
\bar{DT}{}^\al(\tau)&=-\chi^\na\bigl(\fM_\rss^\al(\tau),\Pi_{\CF}
\ci\bar\Pi^{\chi,\Q}_\fM(\bep^\al(\tau))\cdot\nu_\fM\bigr),\\
\bar{DT}{}^\al(\mu)&=-\chi^\na\bigl(\fM_\rss^\al(\mu),\Pi_{\CF}
\ci\bar\Pi^{\chi,\Q}_\fM(\bep^\al(\mu))\cdot\nu_\fM\bigr).
\end{split}
\label{dt5eq8}
\e

For the case of Gieseker stability\index{Gieseker stability}
$(\tau,G,\le)$, we have a projective coarse moduli
scheme\index{coarse moduli scheme}\index{moduli scheme!coarse}
$\M_\rss^\al(\tau)$. Write
$\pi:\fM_\rss^\al(\tau)\ra\M_\rss^\al(\tau)$ for the natural
projection. Then by functoriality of the na\"\i ve pushforward
\eq{dt2eq1}, we can rewrite the first line of \eq{dt5eq8} as a
weighted Euler characteristic of~$\M_\rss^\al(\tau)$:
\e
\bar{DT}{}^\al(\tau)=-\chi\bigl(\M_\rss^\al(\tau),
\CF^\na(\pi)\bigl[\Pi_{\CF}\ci\bar\Pi^{\chi,\Q}_\fM
(\bep^\al(\tau))\cdot\nu_\fM\bigr]\bigr).
\label{dt5eq9}
\e
The constructible functions\index{constructible function}
$-\Pi_{\CF}\ci\bar\Pi^{\chi,\Q}_\fM (\bep^\al(\tau))\cdot\nu_\fM$ on
$\fM_\rss^\al(\tau)$ in \eq{dt5eq8}, and
$-\CF^\na(\pi)[\Pi_{\CF}\ci\bar\Pi^{\chi,\Q}_\fM
(\bep^\al(\tau))\cdot\nu_\fM]$ on $\M_\rss^\al(\tau)$ in
\eq{dt5eq9}, are the contributions to $\bar{DT}{}^\al(\tau)$ from
each $\tau$-semistable, and each S-equivalence\index{S-equivalence}
class of $\tau$-semistables (or
$\tau$-polystable),\index{polystable@$\tau$-polystable}
respectively. We will return to \eq{dt5eq9} in~\S\ref{dt62}.
\label{dt5def2}
\end{dfn}

\begin{rem} We show in Corollary \ref{dt5cor4} below that
$\bar{DT}{}^\al(\tau)$ {\it is unchanged under deformations of\/}
$X$. Our definition of $\bar{DT}{}^\al(\tau)$ is very complicated.
It counts sheaves using two kinds of weights: firstly, we define
$\bep^\al(\tau)$ from the $\bdss^\be(\tau)$ by \eq{dt3eq4}, with
$\Q$-valued weights $(-1)^{n-1}/n$, and then we apply the Lie
algebra morphism $\ti\Psi$, which takes Euler characteristics
weighted by the $\Z$-valued Behrend function $\nu_\fM$. Furthermore,
to compute $\ti\Psi(\bep^\al(\tau))$ we must first write
$\bep^\al(\tau)$ in the form \eq{dt3eq17} using Proposition
\ref{dt3prop1}, and this uses relation Definition
\ref{dt2def10}(iii) involving coefficients~$F(G,T^G,Q)\in\Q$.

In \S\ref{dt65} we will show in an example that all this complexity
is {\it really necessary\/} to make $\bar{DT}{}^\al(\tau)$
deformation-invariant. In particular, we will show that strictly
$\tau$-semistable sheaves must be counted with {\it non-integral\/}
weights, and also that the obvious definition $DT^\al(\tau)=
\chi(\M_\st^\al(\tau),\nu_{\M_\st^\al(\tau)})$ from \eq{dt4eq16} is
not deformation-invariant
when~$\M_\rss^\al(\tau)\ne\M_\st^\al(\tau)$.
\label{dt5rem4}
\end{rem}

Suppose that $\M_\rss^\al(\tau)=\M_\st^\al(\tau)$, that is, there
are no strictly $\tau$-semistable sheaves in class $\al$. Then the
only nonzero term in \eq{dt3eq4} is $n=1$ and $\al_1=\al$, so
\e
\bep^\al(\tau)=\bar\de_\rss^\al(\tau)=
\bde_{\fM_\st^\al(\tau)}=[(\fM_\st^\al(\tau),\io)],
\label{dt5eq10}
\e
where $\io:\fM_\st^\al(\tau)\ra\fM$ is the inclusion 1-morphism.
Write $\pi:\fM_\st^\al(\tau)\ra\M_\st^\al(\tau)$ for the projection
from $\fM_\st^\al(\tau)$ to its coarse moduli scheme\index{coarse moduli
scheme}\index{moduli scheme!coarse} $\M_\st^\al(\tau)$. Then
\begin{align*}
\bar{DT}{}^\al(\tau)&=-\chi^\na\bigl(\fM_\st^\al(\tau),\io^*(\nu_\fM)
\bigr)=-\chi^\na\bigl(\fM_\st^\al(\tau),\nu_{\fM_\st^\al(\tau)}\bigr)\\
&=\chi^\na\bigl(\fM_\st^\al(\tau),\pi^*(\nu_{\M_\st^\al(\tau)})\bigr)
=\chi\bigl(\M_\st^\al(\tau),\nu_{\M_\st^\al(\tau)}\bigr)=DT^\al(\tau),
\end{align*}
using Definition \ref{dt5def1} and \eq{dt5eq7} in the first step,
$\fM_\st^\al(\tau)$ open in $\fM$ in the second, $\pi$ smooth of
relative dimension $-1$ and Corollary \ref{dt4cor1} to deduce
$\pi^*(\nu_{\M_\st^\al(\tau)})\equiv -\nu_{\fM_\st^\al(\tau)}$ in
the third, $\pi_*:\fM_\st^\al(\tau)(\C)\ra\M_\st^\al(\tau)(\C)$ an
isomorphism of constructible sets in the fourth, and \eq{dt4eq16} in
the fifth. Thus we have proved:

\begin{prop} If\/ $\M_\rss^\al(\tau)=\M_\st^\al(\tau)$ then
$\bar{DT}{}^\al(\tau)=DT^\al(\tau)$. That is, our new generalized
Donaldson--Thomas invariants $\bar{DT}{}^\al(\tau)$ are equal to the
original Donaldson--Thomas invariants $DT^\al(\tau)$ whenever the
$DT^\al(\tau)$ are defined.
\label{dt5prop2}
\end{prop}

We include the minus signs in \eq{dt5eq7} to cancel that in
$\pi^*\bigl(\nu_{\M_\st^\al(\tau)}\bigr)=- \nu_{\fM_\st^\al(\tau)}$.
Omitting the signs in \eq{dt5eq7} would have given
$\bar{DT}{}^\al(\tau)=-DT^\al(\tau)$ above.

We can now repeat the argument of \S\ref{dt35} to deduce
transformation laws for generalized Donaldson--Thomas invariants
under change of stability condition. Suppose
$(\tau,T,\le),(\ti\tau,\ti T,\le),(\hat\tau,\hat T,\le)$ are as in
Theorem \ref{dt3thm2} for $\A=\coh(X)$. Then as in \S\ref{dt32}
equation \eq{dt3eq10} holds, and by Theorem \ref{dt3thm3} it is
equivalent to a Lie algebra equation \eq{dt3eq13} in $\SFai(\fM)$.
Thus we may apply the Lie algebra morphism $\ti\Psi$ to transform
\eq{dt3eq13} (or equivalently \eq{dt3eq10}) into an identity in the
Lie algebra $\ti L(X)$, and use \eq{dt5eq7} to write this in terms
of generalized Donaldson--Thomas invariants. As for \eq{dt3eq23},
this gives an equation in the universal enveloping algebra~$U(\ti
L(X))$:\index{universal enveloping algebra}
\e
\begin{gathered}
\bar{DT}{}^\al(\ti\tau)\ti \la^\al= \!\!\!\!\!\!\!
\sum_{\begin{subarray}{l}n\ge 1,\;\al_1,\ldots,\al_n\in
C(\coh(X)):\\ \al_1+\cdots+\al_n=\al\end{subarray}
\!\!\!\!\!\!\!\!\!\!\!\!\!\!\!\!\!\!\!\!\!} \!\!\!\!\!\!\!\!\!
\begin{aligned}[t]
U(\al_1,\ldots,\al_n;\tau,\ti\tau)\cdot &\ts
(-1)^{n-1}\prod_{i=1}^n\bar{DT}{}^{\al_i}(\tau)\cdot\\
& \ti \la^{\al_1}\star \ti \la^{\al_2}\star\cdots\star\ti
\la^{\al_n}.
\end{aligned}
\end{gathered}
\label{dt5eq11}
\e
As in \cite[\S 6.5]{Joyc4}, we describe $U(\ti L(X))$ explicitly,
and the analogue of \eq{dt3eq24} is\index{wall-crossing formula}
\begin{gather}
\ti \la^{\al_1}\star\cdots\star\ti \la^{\al_n}=\text{ terms in $\ti
\la_{[I,\ka]}$, $\md{I}>1$, }
\label{dt5eq12}\\
+\raisebox{-6pt}{\begin{Large}$\displaystyle\biggl[$\end{Large}}
\frac{(-1)^{\sum_{1\le i<j\le n}\bar\chi(\al_i,\al_j)}
}{2^{n-1}}\!\!\!\!\!
\raisebox{4pt}{$\displaystyle\sum_{\begin{subarray}{l}
\text{connected, simply-connected}\\
\text{digraphs $\Ga$: vertices $\{1,\ldots,n\}$,}\\
\text{edge $\mathop{\bu} \limits^{\sst
i}\ra\mathop{\bu}\limits^{\sst j}$ implies $i<j$}\end{subarray}}
\,\,
\prod_{\begin{subarray}{l}\text{edges}\\
\text{$\mathop{\bu}\limits^{\sst i}\ra\mathop{\bu}\limits^{\sst
j}$}\\ \text{in $\Ga$}\end{subarray}}\bar\chi(\al_i,\al_j)$}
\raisebox{-6pt}{\begin{Large}$\displaystyle\biggr]$\end{Large}}
\ti \la^{\al_1+\cdots+\al_n}. \nonumber
\end{gather}

Substitute \eq{dt5eq12} into \eq{dt5eq11}. As for \eq{dt3eq25},
equating coefficients of $\ti \la^\al$ yields
\begin{gather}
\bar{DT}{}^\al(\ti\tau)=\!\!\!\!\!\!
\sum_{\begin{subarray}{l}\al_1,\ldots,\al_n\in C(\coh(X)):\\
n\ge 1,\;\al_1+\cdots+\al_n=\al\end{subarray}}\,\,\,\,
\sum_{\begin{subarray}{l}\text{connected, simply-connected digraphs $\Ga$:}\\
\text{vertices $\{1,\ldots,n\}$, edge $\mathop{\bu} \limits^{\sst
i}\ra\mathop{\bu}\limits^{\sst j}$ implies $i<j$}\end{subarray}}
\label{dt5eq13}\\
\frac{(-1)^{n-1+\sum_{1\le i<j\le
n}\bar\chi(\al_i,\al_j)}}{2^{n-1}}\,
U(\al_1,\ldots,\al_n;\tau,\ti\tau)\prod_{\text{edges
$\mathop{\bu}\limits^{\sst i}\ra\mathop{\bu}\limits^{\sst j}$ in
$\Ga$}\!\!\!\!\!\!\!\!\!\!\!\!\!\!\!\!\!\!\!\!\!\!\! \!\!\!\!\!\!}
\bar\chi(\al_i,\al_j) \prod_{i=1}^n\bar{DT}{}^{\al_i}(\tau).
\nonumber
\end{gather}
Using the coefficients $V(I,\Ga,\ka;\tau,\ti\tau)$ of Definition
\ref{dt3def8} to rewrite \eq{dt5eq13}, we obtain an analogue of
\eq{dt3eq27}, as in~\cite[Th.~6.28]{Joyc6}:

\begin{thm} Under the assumptions above, for all\/ $\al\in
C(\coh(X))$ we have
\e
\begin{aligned}
&\bar{DT}{}^\al(\ti\tau)=\\
&\!\!\!\sum_{\substack{\text{iso.}\\ \text{classes}\\
\text{of finite}\\ \text{sets $I$}}}\!
\sum_{\substack{\ka:I\ra C(\coh(X)):\\ \sum_{i\in I}\ka(i)=\al}}\,
\sum_{\begin{subarray}{l} \text{connected,}\\
\text{simply-}\\ \text{connected}\\ \text{digraphs $\Ga$,}\\
\text{vertices $I$}\end{subarray}}\!\!
\begin{aligned}[t]
(-1)^{\md{I}-1} V(I,\Ga,\ka;\tau,\ti\tau)\cdot
\prod\nolimits_{i\in I} \bar{DT}{}^{\ka(i)}(\tau)&\\
\cdot (-1)^{\frac{1}{2}\sum_{i,j\in
I}\md{\bar\chi(\ka(i),\ka(j))}}\cdot\! \prod\limits_{\text{edges
\smash{$\mathop{\bu}\limits^{\sst i}\ra\mathop{\bu}\limits^{\sst
j}$} in $\Ga$}\!\!\!\!\!\!\!\!\!\!\!\!\!\!\!\!\!\!\!\!\!\!\!
\!\!\!\!\!\!\!\!} \bar\chi(\ka(i),\ka(j))&,\!\!
\end{aligned}
\end{aligned}
\label{dt5eq14}
\e
with only finitely many nonzero terms.
\label{dt5thm6}
\end{thm}

As we explained at the end of \S\ref{dt33}, for technical reasons
the authors do not know whether the changes between every two weak
stability conditions of Gieseker or $\mu$-stability type on
$\coh(X)$ are globally finite,\index{stability condition!changes
globally finite|(} so we cannot apply Theorem \ref{dt5thm6}
directly. But as in \cite[\S 5.1]{Joyc6}, we can interpolate between
any two such stability conditions on $X$ of Gieseker stability or
$\mu$-stability type by a finite sequence of stability conditions,
such that between successive stability conditions in the sequence
the changes are globally finite. Thus we deduce:\index{stability
condition!changes globally finite|)}

\begin{cor} Let\/ $(\tau,T,\le),(\ti\tau,\ti T,\le)$ be two
permissible weak stability conditions on $\coh(X)$ of Gieseker or
$\mu$-stability type, as in Examples {\rm\ref{dt3ex1}} and\/
{\rm\ref{dt3ex2}}. Then the $\bar{DT}{}^\al(\tau)$ for all\/ $\al\in
C(\coh(X))$ completely determine the $\bar{DT}{}^\al(\ti\tau)$ for
all\/ $\al\in C(\coh(X)),$ and vice versa, through finitely many
applications of the transformation law~\eq{dt5eq14}.
\label{dt5cor2}
\end{cor}\index{Donaldson--Thomas invariants!generalized $\bar{DT}{}^\al(\tau)$|)}

\subsection[Invariants counting stable pairs, and
deformation-invariance]{Invariants $PI^{\al,n}(\tau')$ counting
stable pairs, and \\ deformation-invariance of the
$\bar{DT}{}^\al(\tau)$}
\label{dt54}\index{stable pair|(}

Next we define {\it stable pairs\/} on $X$. Our next three results,
Theorems \ref{dt5thm7}, \ref{dt5thm8} and \ref{dt5thm9}, will be
proved in \S\ref{dt12}. They work for $X$ a Calabi--Yau 3-fold over
a general algebraically closed field $\K$,\index{field $\K$} without
assuming~$H^1(\cO_X)=0$.

\begin{dfn} Let $\K$ be an algebraically closed field, and $X$ a
Calabi--Yau 3-fold over $\K$, which may have $H^1(\cO_X)\ne 0$.
Choose a very ample line bundle $\cO_X(1)$ on $X$, and write
$(\tau,G,\le)$ for Gieseker stability w.r.t.\ $\cO_X(1)$, as in
Example~\ref{dt3ex1}.

Fix $n\gg 0$ in $\Z$. A {\it pair\/} is a nonzero morphism of
sheaves $s:\cO_{X}(-n)\ra E$, where $E$ is a nonzero sheaf. A {\it
morphism\/} between two pairs $s:\cO_{X}(-n)\ra E$ and
$t:\cO_{X}(-n)\ra F$ is a morphism of $\cO_X$-modules $f:E\ra F$,
with $f\ci s=t$. A pair $s:\cO_X(-n)\ra E$ is called {\it stable\/}
if:
\begin{itemize}
\setlength{\itemsep}{0pt}
\setlength{\parsep}{0pt}
\item[(i)] $\tau([E'])\le\tau([E])$ for all subsheaves $E'$ of
$E$ with $0\neq E'\neq E$; and
\item[(ii)] If also $s$ factors through $E'$,
then~$\tau([E'])<\tau([E])$.
\end{itemize}
Note that (i) implies that if $s:\cO_X(-n)\ra E$ is stable then $E$
is $\tau$-semistable. The {\it class\/} of a pair $s:\cO_{X}(-n)\ra
E$ is the numerical class $[E]$ in~$K^\num(\coh(X))$.

We have no notion of semistable pairs. We will use $\tau'$ {\it to
denote stability of pairs}, defined using $\cO_X(1)$. Note that
pairs do not form an abelian category, so $\tau'$ is not a (weak)
stability condition on an abelian category in the sense of
\S\ref{dt32}. However, in \S\ref{dt131} we will define an auxiliary
abelian category $\B_p$ and relate stability of pairs $\tau'$ to a
weak stability condition $(\ti\tau,\ti T,\le)$ on~$\B_p$.
\label{dt5def3}
\end{dfn}

\begin{dfn} Use the notation of Definition \ref{dt5def3}. Let $T$
be a $\K$-scheme, and write $\pi_X:X\times T\ra X$ for the
projection. A $T$-{\it family of stable pairs with class\/} $\al$ in
$K^\num(\coh(X))$ is a morphism of $\cO_{X\times T}$-modules $s:
\pi_X^*(\cO_X(-n))\ra E$, where $E$ is flat over $T$, and when
restricting to $\K$-points $t\in T(\K)$, $s_t:\cO_{X}(-n)\ra E_t$ is
a stable pair, with $[E_t]=\al$. Note that since $E$ is flat over
$T$, the class $[E_t]$ in $K^\num(\coh(X))$ is locally constant on
$T$, so requiring $[E_t]=\al$ for all $t\in T(\K)$ is an open
condition on such families.

Two $T$-families of stable pairs $s_1:\pi_X^*(\cO_X(-n))\ra E_1$,
$s_2:\pi_X^*(\cO_X(-n))\ra E_2$ are called {\it isomorphic\/} if
there exists an isomorphism $f:E_1\ra E_2$, such that the following
diagram commutes:
\begin{equation*}
\xymatrix@R=5pt@C=50pt{
\pi_X^*(\cO_X(-n)) \ar[r]_(0.6){s_1} \ar@{=}[d] &E_1 \ar[d]^(0.45){f} \\
\pi_X^*(\cO_X(-n)) \ar[r]^(0.6){s_2} &E_2. }
\end{equation*}
The {\it moduli functor of stable pairs with class\/} $\al$:
\begin{equation*}
\xymatrix{ {\mathbb M}_\stp^{\al,n}(\tau'): \Sch_\K \ar[r]
&{\mathop{\bf Sets}}}
\end{equation*}
is defined to be the functor that takes a $\K$-scheme $T$ to the set
of isomorphism classes of $T$-families of stable pairs with
class~$\al$.
\label{dt5def4}
\end{dfn}

In \S\ref{dt121} we will use results of Le Potier to prove:

\begin{thm} The moduli functor ${\mathbb M}_\stp^{\al,n}(\tau')$ is
represented by a projective $\K$-scheme\/~$\M_\stp^{\al,n}(\tau')$.
\label{dt5thm7}
\end{thm}

Broadly following similar proofs by Pandharipande and Thomas
\cite[\S 2]{PaTh} and Huybrechts and Thomas \cite[\S 4]{HuTh}, in
\S\ref{dt125}--\S\ref{dt127} we prove:

\begin{thm} If\/ $n$ is sufficiently large then the projective
$\K$-scheme $\M_\stp^{\al,n}(\tau')$ has a symmetric obstruction
theory.\index{symmetric obstruction theory}\index{obstruction
theory!symmetric}
\label{dt5thm8}
\end{thm}

Here $n$ is sufficiently large if all $\tau$-semistable sheaves $E$
in class $\al$ are $n$-regular. Using this symmetric obstruction
theory, Behrend and Fantechi \cite{BeFa1} construct a canonical Chow
class\index{Chow homology} $[\M_\stp^{\al,n}(\tau')]^\vir\in
A_{*}(\M_\stp^{\al,n} (\tau'))$. It lies in degree zero since the
obstruction theory is symmetric. Since $\M_\stp^{\al,n}(\tau')$ is
proper, there is a degree map on
$\smash{A_0\bigl(\M_\stp^{\al,n}(\tau')\bigr)}$. We define an
invariant counting stable pairs of class $(\al,n)$ to be the degree
of this virtual fundamental class.

\begin{dfn} In the situation above, if $\al\in K^\num(\coh(X))$
and $n\gg 0$ is sufficiently large, define {\it stable pair
invariants\/} $PI^{\al,n}(\tau')$\index{stable pair invariants
$PI^{\al,n}(\tau')$} in $\Z$
by\nomenclature[PIa(t)]{$PI^{\al,n}(\tau')$}{invariants counting
`stable pairs' $s:\cO_X(-n)\ra E$}
\e
\ts PI^{\al,n}(\tau')=\int_{[\M_\stp^{\al,n}(\tau')]^\vir}1.
\label{dt5eq15}
\e
Theorem \ref{dt4thm6} implies that when $\K$ has characteristic
zero, the stable pair invariants may also be written
\e
PI^{\al,n}(\tau')=\chi\bigl(\M_\stp^{\al,n}(\tau'),
\nu_{\M_\stp^{\al,n}(\tau')}\bigr).
\label{dt5eq16}
\e
\label{dt5def5}
\end{dfn}

Our invariants $PI^{\al,n}(\tau')$ were inspired by Pandharipande
and Thomas \cite{PaTh}, who use invariants counting pairs to study
curve counting in Calabi--Yau 3-folds. Observe an important
difference between Donaldson--Thomas and stable pair invariants:
$DT^\al(\tau)$ is defined only for classes $\al\in K^\num(\coh(X))$
with $\M_\rss^\al(\tau)= \M_\st^\al(\tau)$, but $PI^{\al,n}(\tau')$
is defined for all $\al\in K^\num(\coh(X))$ and all $n\gg 0$. We
wish to show the $PI^{\al,n}(\tau')$ for $\al\in K^\num(\coh(X))$
are unchanged under deformations of the underlying Calabi--Yau
3-fold $X$. For this to make sense, $K^\num(\coh(X))$ should also be
unchanged under deformations of $X$, so we put this in as an
assumption. Our next theorem will be proved in~\S\ref{dt128}.

\begin{thm} Let\/ $\K$ be an algebraically closed field, $U$ a
connected algebraic $\K$-variety, and\/ $X\stackrel{\vp}{\longra}U$
be a family of Calabi--Yau $3$-folds, so that for each\/ $u\in
U(\K)$ the fibre $X_u=X\times_{\vp,U,u}\Spec\K$ of\/ $\vp$ over $u$
is a Calabi--Yau $3$-fold over $\K,$ which may have\/
$H^1(\cO_{X_u})\ne 0$. Let\/ $\cO_X(1)$ be a relative very ample
line bundle for $X\stackrel{\vp}{\longra}U,$ and write\/
$\cO_{X_u}(1)$ for $\cO_X(1)\vert_{X_u}$. Suppose that the numerical
Grothendieck groups $K^\num(\coh(X_u))$ for $u\in U(\K)$ are
canonically isomorphic locally in $U(\K),$ and write $K(\coh(X))$
for this group $K^\num(\coh(X_u))$ up to canonical isomorphism. Then
the stable pair invariants\/ $PI^{\al,n}(\tau')_u$ of\/
$X_u,\cO_{X_u}(1)$ for $\al\in K(\coh(X))$ and\/ $n\gg 0$ are
independent of\/~$u\in U(\K)$.
\label{dt5thm9}
\end{thm}

Here is what we mean by saying that the numerical Grothendieck
groups $K^\num(\coh(X_u))$ for $u\in U(\K)$ are canonically
isomorphic {\it locally in\/} $U(\K)$: we do not require canonical
isomorphisms $K^\num(\coh(X_u))\cong K(\coh(X))$ for all $u\in
U(\K)$ (this would be canonically isomorphic {\it globally in\/}
$U(\K)$). Instead, we mean that the groups $K^\num(\coh(X_u))$ for
$u\in U(\K)$ form a {\it local system of abelian groups\/} over
$U(\K)$, with fibre $K(\coh(X))$.

When $\K=\C$, this means that in simply-connected regions of $U(\C)$
in the complex analytic topology the $K^\num(\coh(X_u))$ are all
canonically isomorphic, and isomorphic to $K(\coh(X))$. But around
loops in $U(\C)$, this isomorphism with $K(\coh(X))$ can change by
{\it monodromy},\index{monodromy} by an automorphism
$\mu:K(\coh(X))\ra K(\coh(X))$ of $K(\coh(X))$, as in Remark
\ref{dt4rem}(d). In Theorem \ref{dt4thm8} we showed that the group
of such monodromies $\mu$ is finite, and we can make it trivial by
passing to a finite cover $\ti U$ of $U$. If we worked instead with
invariants $PI^{P,n}(\tau')$ counting pairs $s:\cO_{X}(-n)\ra E$ in
which $E$ has fixed Hilbert polynomial\index{Hilbert polynomial}
$P$, rather than fixed class $\al\in K^\num(\coh(X))$, as in Thomas'
original definition of Donaldson--Thomas invariants \cite{Thom},
then we could drop the assumption on $K^\num(\coh(X_u))$ in
Theorem~\ref{dt5thm9}.

In Theorem \ref{dt4thm7} we showed that when $\K=\C$ and
$H^1(\cO_X)=0$ the numerical Grothendieck group $K^\num(\coh(X))$ is
unchanged under small deformations of $X$ up to canonical
isomorphism. So Theorem \ref{dt5thm9} yields:

\begin{cor} Let\/ $X$ be a Calabi--Yau $3$-fold over $\K=\C,$ with\/
$H^1(\cO_X)=0$. Then the pair invariants  $PI^{\al,n}(\tau')$ are
unchanged by continuous deformations of the complex structure
of\/~$X$.
\label{dt5cor3}
\end{cor}

The following result, proved in \S\ref{dt13}, expresses the pair
invariants $PI^{\al,n}(\tau')$ above in terms of the generalized
Donaldson--Thomas invariants $\bar{DT}{}^\be(\tau)$ of \S\ref{dt53}.
Although the theorem makes sense for general algebraically closed
fields $\K$, our proof works only for $\K=\C$, since it involves a
version of the Behrend function identities\index{Behrend
function!identities} \eq{dt5eq2}--\eq{dt5eq3}, which are proved
using complex analytic methods.

\begin{thm} Let\/ $\K=\C$. Then for\/ $\al\in C(\coh(X))$ and\/
$n\gg 0$ we have
\e
PI^{\al,n}(\tau')=\!\!\!\!\!\!\!\!\!\!\!\!\!
\sum_{\begin{subarray}{l} \al_1,\ldots,\al_l\in
C(\coh(X)),\\ l\ge 1:\; \al_1
+\cdots+\al_l=\al,\\
\tau(\al_i)=\tau(\al),\text{ all\/ $i$}
\end{subarray} \!\!\!\!\!\!\!\!\! }
\begin{aligned}[t] \frac{(-1)^l}{l!} &\prod_{i=1}^{l}\bigl[
(-1)^{\bar\chi([\cO_X(-n)]-\al_1-\cdots-\al_{i-1},\al_i)} \\
&\bar\chi\bigl([\cO_{X}(-n)]\!-\!\al_1\!-\!\cdots\!-\!\al_{i-1},\al_i
\bigr) \bar{DT}{}^{\al_i}(\tau)\bigr],\!\!\!\!\!\!\!\!\!\!\!\!
\end{aligned}
\label{dt5eq17}
\e
where there are only finitely many nonzero terms in the sum.
\label{dt5thm10}
\end{thm}

As we will see in \S\ref{dt6}, equation \eq{dt5eq17} is useful for
computing invariants $\bar{DT}{}^\al(\tau)$ in examples. We also use
it to deduce the $\bar{DT}{}^\al(\tau)$ are deformation-invariant.

\begin{cor} The generalized Donaldson--Thomas invariants
$\bar{DT}{}^\al(\tau)$ of\/ {\rm\S\ref{dt53}} are unchanged under
continuous deformations of the underlying Calabi--Yau\/
$3$-fold\/~$X$.\index{Donaldson--Thomas
invariants!deformation-invariance}
\label{dt5cor4}
\end{cor}

\begin{proof} Let $\al\in C(\coh(X))$ have dimension
$\dim\al=d=0,1,2$ or 3. Then the Hilbert polynomial $P_\al$ is of
the form $P_\al(t)=\frac{k}{d!}t^d+a_{d-1}t^{d-1}+\cdots+a_0$ for
$k$ a positive integer and $a_{d-1},\ldots,a_0\in\Q$. Fix $d$, and
suppose by induction on $K\ge 0$ that $\bar{DT}{}^\al(\tau)$ is
deformation-invariant for all $\al\in C(\coh(X))$ with $\dim\al=d$
and $P_\al(t)=\frac{k}{d!}t^d+\cdots+a_0$ for $k\le K$. This is
vacuous for~$K=0$.

Let $\al\in C(\coh(X))$ with $\dim\al=d$ and
$P_\al(t)=\frac{K+1}{d!}t^d+\cdots+a_0$. We rewrite \eq{dt5eq17} by
splitting into terms $l=1$ and $l\ge 2$ as
\e
\begin{split}
&(-1)^{\bar\chi([\cO_X(-n)],\al)}\bar\chi\bigl([\cO_{X}(-n)],\al\bigr)
\bar{DT}{}^{\al}(\tau)=-PI^{\al,n}(\tau')\\
&\quad +\sum_{\begin{subarray}{l} \al_1,\ldots,\al_l\in
C(\coh(X)),\\ l\ge 2:\; \al_1+\cdots+\al_l=\al,\\
\tau(\al_i)=\tau(\al),\text{ all\/ $i$}
\end{subarray} \!\!\!\!\!\!\!\!\!\!\!\! }
\begin{aligned}[t] \frac{(-1)^l}{l!} &\prod_{i=1}^{l}\bigl[
(-1)^{\bar\chi([\cO_X(-n)]-\al_1-\cdots-\al_{i-1},\al_i)} \\
&\bar\chi\bigl([\cO_{X}(-n)]\!-\!\al_1\!-\!\cdots\!-\!\al_{i-1},\al_i
\bigr)\bar{DT}{}^{\al_i}(\tau)\bigr].\!\!\!\!\!\!\!\!\!\!\!\!\!\!
\end{aligned}
\end{split}
\label{dt5eq18}
\e
Here $\bar\chi([\cO_{X}(-n)],\al)>0$ for $n\gg 0$, so the
coefficient of $\bar{DT}{}^{\al}(\tau)$ on the left hand side of
\eq{dt5eq18} is nonzero. On the right hand side, $PI^{\al,n}(\tau')$
is unchanged under deformations of $X$ by Corollary~\ref{dt5cor3}.

For terms $l\ge 2$, $\al_1,\ldots,\al_l\in C(\coh(X))$ with
$\al_1+\cdots+\al_l=\al$ and $\tau(\al_i)=\tau(\al)$ in
\eq{dt5eq18}, we have $\dim\al_i=d$ and
$P_{\al_i}(t)=\frac{k_i}{d!}t^d+\cdots+a_0$, where $k_1,\ldots,k_l$
are positive integers with $k_1+\cdots+k_l=K+1$. Thus $k_i\le K$ for
each $i$, and $\bar{DT}{}^{\al_i}(\tau)$ is deformation-invariant by
the inductive hypothesis. Therefore everything on the right hand
side of \eq{dt5eq18} is deformation-invariant, so
$\bar{DT}{}^{\al}(\tau)$ is deformation-invariant. This proves the
inductive step.
\end{proof}

In many interesting cases the terms $\bar\chi(\al_i,\al_j)$ in
\eq{dt5eq17} are automatically zero. Then \eq{dt5eq17} simplifies,
and we can encode it in a generating function equation. The proof of
the next proposition is immediate. Note that there is a problem with
choosing $n$ in \eq{dt5eq20}, as \eq{dt5eq19} only holds for $n\gg
0$ depending on $\al$, but \eq{dt5eq20} involves one fixed $n$ but
infinitely many $\al$. We can regard the initial term 1 in
\eq{dt5eq20} as $PI^{\al,n}(\tau')\,q^\al$ for $\al=0$. In
Conjecture \ref{dt6conj1} we will call $(\tau,T,\le)$ {\it
generic\/} if $\bar\chi(\be,\ga)=0$ for all $\be,\ga$
with~$\tau(\be)=\tau(\ga)$.

\begin{prop} In the situation above, with $(\tau,T,\le)$ a weak
stability condition on $\coh(X),$ suppose $t\in T$ is such that\/
$\bar\chi(\be,\ga)=0$ for all\/ $\be,\ga\in C(\coh(X))$ with\/
$\tau(\be)=\tau(\ga)=t$. Then for all $\al\in C(\coh(X))$ with
$\tau(\al)=t$ and\/ $n\gg 0$ depending on $\al,$ equation
\eq{dt5eq17} becomes
\e
\begin{gathered}
PI^{\al,n}(\tau')=\!\!\!\!\!\!\!\!\!\!\!\!\!\!\!
\sum_{\begin{subarray}{l} \al_1,\ldots,\al_l\in
C(\coh(X)),\\ l\ge 1:\; \al_1
+\cdots+\al_l=\al,\\
\tau(\al_i)=t,\text{ all\/ $i$}
\end{subarray} \!\!\!\!\!\!\!\! }
\begin{aligned}[t] \frac{(-1)^l}{l!}\prod_{i=1}^{l}\bigl[
(-1)^{\bar\chi([\cO_X(-n)],\al_i)}
\bar\chi\bigl([\cO_{X}(-n)],\al_i\bigr)&\\[-5pt]
\bar{DT}{}^{\al_i}(\tau)&\bigr].\!\!\!\!\!\!\!\!\!\!\!\!\!\!
\end{aligned}
\end{gathered}
\label{dt5eq19}
\e
Ignore for the moment the fact that\/ \eq{dt5eq19} only holds for
$n\gg 0$ depending on $\al$. Then \eq{dt5eq19} can be encoded as the
$q^\al$ term in the formal power series
\e
\begin{split}
&1+\sum_{\al\in C(\coh(X)):\;\tau(\al)=t
\!\!\!\!\!\!\!\!\!\!\!\!\!\!\!\!\!\!\!\!\!\!\!\!\!\!}
PI^{\al,n}(\tau')q^\al=\\
&\quad\exp
\raisebox{-4pt}{\begin{Large}$\displaystyle\Bigl[$\end{Large}}
-\sum_{\al\in C(\coh(X)):\;\tau(\al)=t
\!\!\!\!\!\!\!\!\!\!\!\!\!\!\!\!\!\!\!\!\!\!\!\!\!\!\!\!\!\!\!\!\!\!}
(-1)^{\bar\chi([\cO_X(-n)],\al)}
\bar\chi\bigl([\cO_{X}(-n)],\al\bigr)\bar{DT}{}^\al(\tau)q^\al
\raisebox{-4pt}{\begin{Large}$\displaystyle\Bigr]$\end{Large}},
\end{split}
\label{dt5eq20}
\e
where $q^\al$ for $\al\in C(\coh(X))$ are formal symbols satisfying
$q^\al\cdot q^\be=q^{\al+\be}$.
\label{dt5prop3}
\end{prop}

Now Theorem \ref{dt5thm10} relates the invariants
$PI^{\al,n}(\tau')$ and $\bar{DT}{}^\be(\tau)$, which can both be
written in terms of Euler characteristics weighted by Behrend
functions.\index{Behrend function} There is an analogue in which we
simply omit the Behrend functions. Omitting the Behrend function
$\nu_{\smash{\M_\stp^{\al,n} (\tau')}}$ in the expression
\eq{dt5eq16} for $PI^{\al,n}(\tau')$ shows that the unweighted
analogue of $PI^{\al,n}(\tau')$ is
$\smash{\chi\bigl(\M_\stp^{\al,n}(\tau')\bigr)}$. Comparing
\eq{dt3eq22} and \eq{dt5eq7} shows that (up to sign) the unweighted
analogue of $\bar{DT}{}^\be(\tau)$ is the invariant $J^\be(\tau)$ of
\S\ref{dt35}. The proof of Theorem \ref{dt5thm10} in \S\ref{dt13}
involves a Lie algebra morphism $\ti\Psi{}^{\B_p}$ in \S\ref{dt134};
for the unweighted case we must replace this by a Lie algebra
morphism $\Psi^{\B_p}$ which is related to $\ti\Psi{}^{\B_p}$ in the
same way that $\Psi$ in \S\ref{dt34} is related to $\ti\Psi$ in
\S\ref{dt53}, and maps to a Lie algebra $L(\B_p)$ with the sign
omitted in \eq{dt13eq29}. In this way we obtain the following
unweighted version of Theorem~\ref{dt5thm10}:

\begin{thm} For\/ $\al\in C(\coh(X))$ and\/ $n\gg 0$ we have
\e
\chi\bigl(\M_\stp^{\al,n}(\tau')\bigr)=\!\!\!\!\!\!\!\!\!\!
\sum_{\begin{subarray}{l} \al_1,\ldots,\al_l\in
C(\coh(X)),\\ l\ge 1:\; \al_1
+\cdots+\al_l=\al,\\
\tau(\al_i)=\tau(\al),\text{ all\/ $i$}
\end{subarray}  }
\begin{aligned}[t] \frac{1}{l!} \prod_{i=1}^{l}\bigl[\bar\chi
\bigl([\cO_{X}(-n)]\!-\!\al_1\!-\!\cdots\!-\!\al_{i-1},\al_i\bigr)&\\[-6pt]
\cdot\,\, J^{\al_i}(\tau)\bigr]&,
\end{aligned}
\label{dt5eq21}
\e
for $J^{\al_i}(\tau)$ as in {\rm\S\ref{dt35},} with only finitely
many nonzero terms in the sum.
\label{dt5thm11}
\end{thm}\index{stable pair|)}

\section{Examples, applications, and generalizations}
\label{dt6}\index{Donaldson--Thomas invariants!computation in examples|(}

We now give many worked examples of the theory of \S\ref{dt5}, and
some consequences and further developments. This section considers
Donaldson--Thomas theory in $\coh(X)$, for $X$ a Calabi--Yau 3-fold
over $\C$. As in \S\ref{dt5}, our definition of Calabi--Yau 3-fold
includes the assumption that $H^1(\cO_X)=0$, as discussed in
Remark~\ref{dt5rem2}.

Section \ref{dt7} will discuss Donaldson--Thomas theory in
categories of quiver representations $\modCQI$ coming from a
superpotential $W$ on $Q$, which is a fertile source of easily
computable examples.

\subsection{Computing $PI^{\al,n}(\tau')$, $\bar{DT}{}^\al(\tau)$
and $J^\al(\tau)$ in examples}
\label{dt61}

Here are a series of simple situations in which we can calculate
contributions to the invariants $PI^{\al,n}(\tau')$ and
$\bar{DT}{}^\al(\tau)$ of \S\ref{dt5}, and $J^\al(\tau)$
of~\S\ref{dt35}.

\begin{ex} Let $X$ be a Calabi--Yau 3-fold over $\C$ equipped
with a very ample line bundle $\cO_X(1)$. Suppose $\al\in
C(\coh(X))$, and that $E\in\coh(X)$ with $[E]=\al$ is $\tau$-stable
and rigid, so that $\Ext^1(E,E)=0$. Then $mE={\buildrel
{\!\ulcorner\,\text{$m$ copies }\,\urcorner\!} \over {E\op\cdots \op
E}}$ for $m\ge 2$ is a strictly $\tau$-semistable sheaf of class
$m\al$, which is also rigid. Hence $\{[mE]\}$ is a connected
component of $\M_\rss^{m\al}(\tau)$, and $\pi^{-1}([mE])$ is a
connected component of $\M_\stp^{m\al,n}(\tau')$ for $m\ge 1$, where
$\pi:\M_\stp^{m\al,n}(\tau')\ra\M_\rss^{m\al}(\tau)$ is the
projection from a stable pair $s:\cO(-n)\ra E'$ to the S-equivalence
class $[E']$ of the underlying $\tau$-semistable sheaf $E'$. Suppose
for simplicity that $mE$ is the only $\tau$-semistable sheaf of
class $m\al$; alternatively, we can consider the following as
computing the contribution to $PI^{m\al,n}(\tau')$ from stable
pairs~$s:\cO(-n)\ra mE$.

A pair $s:\cO(-n)\ra mE$ may be regarded as $m$ elements
$s^1,\ldots,s^m$ of $H^0(E(n))\ab\cong\C^{P_\al(n)}$ for $n\gg 0$,
where $P_\al$ is the Hilbert polynomial\index{Hilbert polynomial} of
$E$. Such a pair turns out to be stable if and only if
$s^1,\ldots,s^m$ are linearly independent in $H^0(E(n))$. Two such
pairs are equivalent if they are identified under the action of
$\Aut(mE)\cong\GL(m,\C)$, acting in the obvious way on
$(s^1,\ldots,s^m)$. Thus, equivalence classes of stable
pairs\index{stable pair} correspond to linear subspaces of dimension
$m$ in $H^0(E(n))$, so the moduli space
$\smash{\M_\stp^{m\al,n}(\tau')}$ is isomorphic as a $\C$-scheme to
the Grassmannian $\Gr(\C^m,\C^{P_\al(n)})$. This is smooth of
dimension $m(P_\al(n)-m)$, so that
$\nu_{\M_\stp^{m\al,n}(\tau')}\equiv (-1)^{m(P_\al(n)-m)}$ by
Theorem \ref{dt4thm1}(i). Also $\Gr(\C^m,\C^{P_\al(n)})$ has Euler
characteristic the binomial coefficient
$\smash{\binom{P_\al(n)}{m}}$. Therefore \eq{dt5eq16} gives
\e
PI^{m\al,n}(\tau')=\ts (-1)^{m(P_\al(n)-m)}\binom{P_\al(n)}{m}.
\label{dt6eq1}
\e
\label{dt6ex1}
\end{ex}

We can use equations \eq{dt5eq17} and \eq{dt5eq21} to compute the
generalized Donald\-son--Thomas invariants $\bar{DT}{}^{m\al}(\tau)$
and invariants $J^{m\al}(\tau)$ in Example~\ref{dt6ex1}.

\begin{ex} Work in the situation of Example \ref{dt6ex1}, and
assume that $mE$ is the only $\tau$-semistable sheaf of class $m\al$
for all $m\ge 1$, up to isomorphism. Consider \eq{dt5eq17} with
$m\al$ in place of $\al$. If $\al_1,\ldots,\al_l$ give a nonzero
term on the right hand side of \eq{dt5eq17} then
$m\al=\al_1+\cdots+\al_l$, and $\bar{DT}{}^{\al_i}(\tau)\ne 0$, so
there exists a $\tau$-semistable $E_i$ in class $\al_i$. Thus
$E_1\op\cdots\op E_l$ lies in class $m\al$, and is $\tau$-semistable
as $\tau(\al_i)=\tau(\al)$ for all $i$. Hence $E_1\op\cdots\op
E_l\cong mE$, which implies that $E_i\cong k_iE$ for some
$k_1,\ldots,k_l\ge 1$ with $k_1+\cdots+k_l=m$, and~$\al_i=k_i\al$.

Setting $\al_i=k_i\al$, we see that $\bar\chi(\al_j,\al_i)=0$ and
$\bar\chi([\cO_{X}(-n)],\al_i)=k_iP_\al(n)$, where $P_\al$ is the
Hilbert polynomial of $E$. Thus in \eq{dt5eq17} we have
$\bar\chi([\cO_X(-n)]-\al_1-\cdots-\al_{i-1},\al_i)=k_iP_\al(n)$.
Combining \eq{dt6eq1}, and \eq{dt5eq17} with these substitutions,
and cancelling a factor of $(-1)^{mP_\al(n)}$ on both sides, yields
\e
(-1)^m\binom{P_\al(n)}{m}=
\sum_{\begin{subarray}{l}
l,k_1,\ldots,k_l\ge 1:\\
k_1+\cdots+k_l=m\end{subarray}}
\begin{aligned}[t] \frac{(-1)^l}{l!} &\prod_{i=1}^{l}
k_iP_\al(n)\bar{DT}{}^{k_i\al}(\tau).
\end{aligned}
\label{dt6eq2}
\e
Regarding each side as a polynomial in $P_\al(n)$ and taking the
linear term in $P_\al(n)$ we see that
\e
\bar{DT}{}^{m\al}(\tau)=\frac{1}{m^2} \quad\text{for all $m\ge 1$.}
\label{dt6eq3}
\e
Setting $\bar{DT}{}^{k_i\al}(\tau)=1/k_i^2$, we see that \eq{dt6eq2}
is the $x^m$ term in the power series expansion of the identity
\begin{equation*}
(1-x)^{P_\al(n)}=\exp\bigl[\ts -P_\al(n)\sum_{k=1}^\iy
x^k/k\,\bigr].
\end{equation*}
This provides a consistency check for \eq{dt5eq17} in this example:
there exist unique values for $\bar{DT}{}^{k\al}(\tau)$ for
$k=1,2,\ldots$ such that \eq{dt6eq2} holds for all~$n,m$.

In the same way, by \eq{dt5eq21} the analogue of \eq{dt6eq2} is
\begin{equation*}
\binom{P_\al(n)}{m}=
\sum_{\begin{subarray}{l}
l,k_1,\ldots,k_l\ge 1:\\
k_1+\cdots+k_l=m\end{subarray}}
\begin{aligned}[t] \frac{1}{l!} &\prod_{i=1}^{l}
k_iP_\al(n)J^{k_i\al}(\tau).
\end{aligned}
\end{equation*}
Taking the linear term in $P_\al(n)$ on both sides gives
\e
J^{m\al}(\tau)=\frac{(-1)^{m-1}}{m^2} \quad\text{for all $m\ge 1$.}
\label{dt6eq4}
\e
\label{dt6ex2}
\end{ex}

From \eq{dt6eq3}--\eq{dt6eq4} we see that

\begin{cor} The invariants\/ $\bar{DT}{}^\al(\tau),J^\al(\tau)\in\Q$
need not be integers.
\label{dt6cor1}
\end{cor}

\begin{ex} Work in the situation of Example \ref{dt6ex1}, but
suppose now that $E_1,\ldots,E_l$ are rigid, pairwise non-isomorphic
$\tau$-stable coherent sheaves with $[E_i]=\al_i\in C(\coh(X))$,
where $\al_1,\ldots,\al_l$ are distinct with $\tau(\al_1)=\cdots=
\tau(\al_l)=\tau(\al)$ for $\al=\al_1+\cdots+\al_l$, and suppose
$E=E_1\op\cdots\op E_l$ is the only $\tau$-semistable sheaf in class
$\al\in C(\coh(X))$, up to isomorphism. Then by properties of
(semi)stable sheaves we have $\Hom(E_i,E_j)=0$ for all $i\ne j$, so
$\Hom(E,E)=\bigop_{i=1}^l\Hom(E_i,E_i)$, and
$\Aut(E)=\prod_{i=1}^l\Aut(E_i)\cong\bG_m^l$. A pair $s:\cO(-n)\ra
E$ is an $l$-tuple $(s_1,\ldots,s_l)$ with~$s_i\in H^0(E_i(n))\cong
\C^{P_{\al_i}(n)}$.

The condition for $s:\cO(-n)\ra E$ to be stable is $s_i\ne 0$ for
$i=1,\ldots,l$. Thus $\M_\stp^{\al,n}(\tau')$ is the quotient of
$\prod_{i=1}^l\bigl(H^0(E_i(n))\sm\{0\}\bigr)$ by
$\Aut(E)\cong\bG_m^l$, so that $\M_\stp^{\al,n}(\tau')\cong
\prod_{i=1}^l\CP^{P_{\al_i}(n)-1}$ as a smooth $\C$-scheme, where
$E_i$ has Hilbert polynomial $P_{\al_i}$. This has Euler
characteristic $\prod_{i=1}^lP_{\al_i}(n)$ and dimension
$\sum_{i=1}^l(P_{\al_i}(n)-1)$, so that
$\nu_{\M_\stp^{\al,n}(\tau')} \equiv(-1)^{\sum_{i=1}^l
(P_{\al_i}(n)-1)}$. So \eq{dt5eq16} gives
\e
PI^{\al,n}(\tau')=\ts (-1)^{\sum_{i=1}^l
(P_{\al_i}(n)-1)}\prod_{i=1}^lP_{\al_i}(n).
\label{dt6eq5}
\e
\label{dt6ex3}
\end{ex}

\begin{ex} We work in the situation of Example \ref{dt6ex3}. Let
$i,j=1,\ldots,l$ with $i\ne j$. Since $E_i,E_j$ are nonisomorphic
$\tau$-stable sheaves with $\tau([E_i])=\tau([E_j])$ we have
$\Hom(E_i,E_j)=\Hom(E_j,E_i)=0$. As by assumption $E=E_1\op\cdots\op
E_l$ is the only $\tau$-semistable sheaf in class $\al$, we also
have $\Ext^1(E_i,E_j)=\Ext^1(E_j,E_i)=0$, since if
$\Ext^1(E_i,E_j)\ne 0$ we would have a nontrivial extension $0\ra
E_j\ra F\ra E_i\ra 0$, and then $F\op\bigop_{k\ne i,j}E_k$ would be
a $\tau$-semistable sheaf in class $\al$ not isomorphic to $E$. So
by \eq{dt3eq14} we have~$\bar\chi([E_i],[E_j])=\bar\chi(\al_i,
\al_j)=0$.

If $\es\ne I\subseteq\{1,\ldots,l\}$, using \eq{dt3eq4} and
$\Hom(E_i,E_j)\!=\!0\!=\!\Ext^1(E_i,E_j)$ gives
\e
\begin{aligned}
\bep^{\sum_{i\in I}\al_i}(\tau)&=\sum_{\begin{subarray}{l}
I=I_1\amalg\cdots\amalg I_n:\\
n\ge 1,\; I_j\ne\es\end{subarray}}
\frac{(-1)^{n-1}}{n}\,\,\bdss^{\sum_{i\in I_1}\al_i}(\tau)
*\cdots*\bdss^{\sum_{i\in I_n}\al_i}(\tau)\\
&=\sum_{\begin{subarray}{l}
I=I_1\amalg\cdots\amalg I_n:\\
n\ge 1,\; I_j\ne\es\end{subarray}}
\frac{(-1)^{n-1}}{n}\,\,\bde_{[\bigop_{i\in I_1}E_i]}*\cdots*
\bde_{[\bigop_{i\in I_n}E_i]}\\
&=\raisebox{-6pt}{\begin{Large}$\displaystyle\biggl[$\end{Large}}
\sum_{\begin{subarray}{l}
I=I_1\amalg\cdots\amalg I_n:\\
n\ge 1,\; I_j\ne\es\end{subarray}\!\!\!\!} \frac{(-1)^{n-1}}{n}
\raisebox{-6pt}{\begin{Large}$\displaystyle\biggr]$\end{Large}}
\bde_{[\bigop_{i\in I}E_i]}=\begin{cases}\bde_{[E_i]}, & I=\{i\}, \\
0, & \md{I}\ge 2,\end{cases}
\end{aligned}
\label{dt6eq6}
\e
where the combinatorial sum $[\cdots]$ in the last line is evaluated
as in the proof of \cite[Th.~7.8]{Joyc5}. It follows from
\eq{dt5eq7} that $\bar{DT}{}^{\al_i}(\tau)=1$ for all
$i=1,\ldots,l$, and $\bar{DT}{}^{\sum_{i\in I}\al_i}(\tau)=0$ for
all subsets $I\subseteq\{1,\ldots,l\}$ with $\md{I}\ge 2$.
Substituting these values into \eq{dt5eq17}, the only nonzero terms
come from splitting $\al=\al_{\si(1)}+\al_{\si(2)}+\cdots
+\al_{\si(l)}$, where $\si:\{1,\ldots,l\}\ra\{1,\ldots,l\}$ is a
permutation of $\{1,\ldots,l\}$. This term contributes
\begin{align*}
&\frac{(-1)^l}{l!} \prod_{i=1}^{l}\begin{aligned}[h]\bigl[
&(-1)^{\bar\chi([\cO_X(-n)]-\al_{\si(1)}-\cdots-\al_{\si(i-1)},
\al_{\si(i)})}\\
&\bar\chi\bigl([\cO_{X}(-n)]\!-\!\al_{\si(1)}\!-\!\cdots\!-
\!\al_{\si(i-1)},\al_{\si(i)}\bigr)\cdot 1\bigr]\end{aligned} \\
&=\frac{(-1)^l}{l!} \prod_{i=1}^{l}\bigl[
(-1)^{P_{\al_{\si(i)}}(n)}P_{\al_{\si(i)}}(n)\bigr]=
\frac{1}{l!}\cdot (-1)^{\sum_{i=1}^l
(P_{\al_i}(n)-1)}\prod_{i=1}^lP_{\al_i}(n)
\end{align*}
to the r.h.s.\ of \eq{dt5eq17}. As there are $l!$ permutations
$\si$, summing these contributions in \eq{dt5eq17} gives
\eq{dt6eq5}. A similar computation with \eq{dt5eq21} shows that
$J^{\al_i}(\tau)=1$ and $J^{\sum_{i\in I}\al_i}(\tau)=0$ when
$\md{I}\ge 2$. Thus we see that if $E_1,\ldots,E_l$ are pairwise
nonisomorphic $\tau$-stable sheaves for $l\ge 2$ with $[E_i]=\al_i$,
and $\tau(\al_i)=\tau(\al_j)$ and $\Ext^1(E_i,E_j)=0$, then the
$\tau$-semistable sheaf $E_1\op\cdots\op E_l$ contributes zero to
$\bar{DT}{}^{\al_1+\cdots+\al_l}(\tau)$
and~$J^{\al_1+\cdots+\al_l}(\tau)$.
\label{dt6ex4}
\end{ex}

\begin{ex} We combine Examples \ref{dt6ex1} and \ref{dt6ex3}.
Suppose $E_1,\ldots,E_l$ are rigid, pairwise non-isomorphic stable
coherent sheaves, where $E_i$ has Hilbert polynomial $P_{\al_i}$,
that $m_1,\ldots,m_l\ge 1$, and that $E=m_1E_1\op\cdots\op m_lE_l$
is the only semistable sheaf in class $\al\in C(\coh(X))$, up to
isomorphism.

Then a pair $s:\cO(-n)\ra E$ is a collection of $s_i^j\in
H^0(E_i(n))\cong\C^{P_{\al_i}(n)}$ for $i=1,\ldots,l$ and
$j=1,\ldots,m_i$, and is stable if and only if $s_i^1,\ldots,
s_i^{m_i}$ are linearly independent in $H^0(E_i(n))$ for all
$i=1,\ldots,l$. The automorphism group $\Aut(E)\cong
\prod_{i=1}^l\GL(m_i,\C)$ acts upon the set of such stable pairs,
and taking the quotient shows that the moduli space
$\M_\stp^{\al,n}(\tau')$ is isomorphic to the product of
Grassmannians $\prod_{i=1}^l\Gr(\C^{m_i},\C^{P_{\al_i}(n)})$. Hence
\e
PI^{\al,n}(\tau')=\ts
\prod_{i=1}^l(-1)^{m_i(P_{\al_i}(n)-m_i)}\binom{P_{\al_i}(n)}{m_i}.
\label{dt6eq7}
\e
Equation \eq{dt6eq7} includes both \eq{dt6eq1} as the case $l=1$
with $P_1=P$, $m_1=m$ and $\al$ in place of $m\al$, and \eq{dt6eq5}
as the case~$m_1=\cdots=m_l=1$.
\label{dt6ex5}
\end{ex}

\begin{ex} We combine Examples \ref{dt6ex2} and \ref{dt6ex4}. Work
in the situation of Example \ref{dt6ex5}. Then
$\bar\chi(\al_i,\al_j)=0$ for $i\ne j$. Suppose for simplicity that
$\al_1,\ldots,\al_l$ are linearly independent over $\Z$ in
$K^\num(\coh(X))$, and that $m_1E_1\op\cdots\op m_lE_l$ is the only
$\tau$-semistable sheaf in class $m_1\al_1+\cdots+m_l\al_l$ for all
$m_1,\ldots,m_l\ge 0$. We claim that $\bar{DT}{}^{m_i\al_i}(\tau)=
1/m_i^2$ and $J^{m_i\al_i}(\tau)=(-1)^{m_i-1}/m_i^2$ for all
$i=1,\ldots,l$ and $m_i=1,2,\ldots$, and $\bar{DT}{}^{m_1\al_1+
\cdots+m_l\al_l}(\tau)=J^{m_1\al_1+\cdots+m_l\al_l}(\tau)=0$
whenever at least two $m_i$ are positive. The latter holds as
$\bep^{m_1\al_1+\cdots+m_l\al_l}(\tau)=0$ whenever at least two
$m_i$ are positive, as in \eq{dt6eq6}. It is not difficult to show,
as in Examples \ref{dt6ex2} and \ref{dt6ex4}, that substituting
these values into the r.h.s.\ of \eq{dt5eq17} gives~\eq{dt6eq7}.
\label{dt6ex6}
\end{ex}

\begin{ex} Suppose now that $E_1,E_2$ are rigid $\tau$-stable
sheaves in classes $\al_1,\al_2$ in $C(\coh(X))$ with
$\al_1\ne\al_2$ and $\tau(\al_1)=\tau(\al_2)=\tau(\al)$, where
$\al=\al_1+\al_2$. Suppose too that $\Ext^1(E_1,E_2)=0$ and
$\Ext^1(E_2,E_1)\cong\C^d$. We have $\Hom(E_1,E_2)=\Hom(E_2,E_1)=0$,
as $E_1,E_2$ are nonisomorphic $\tau$-stable sheaves with
$\tau([E_1])=\tau([E_2])$. So by \eq{dt3eq14} we
have~$\bar\chi(\al_1,\al_2)=d$.

As $E_1,E_2$ are rigid we have $\Ext^1(E_1,E_1)=\Ext^1(E_2,E_2)=0$.
Hence $\Ext^1(E_1\op E_2,E_1\op E_2)=\Ext^1(E_2,E_1)\cong\C^d$. Now
$\Ext^1(E_1\op E_2,E_1\op E_2)$ parametrizes infinitesimal
deformations of $E_1\op E_2$. All deformations in $\Ext^1(E_2,E_1)$
are realized by sheaves $F$ in exact sequences $0\ra E_1\ra F\ra
E_2\ra 0$. Therefore as $\Ext^1(E_1\op E_2,E_1\op E_2)=\Ext^1(E_2,
E_1)$, all deformations of $E_1\op E_2$ are unobstructed, and the
moduli stack of deformations of $E_1\op E_2$ is the quotient stack
$\bigl[\Ext^1(E_1\op E_2,E_1\op E_2)/\Aut(E_1\op
E_2)\bigr]\cong[\C^d/\bG_m^2]$, where $\bG_m^2$ acts on $\C^d$ by
$(\la,\mu):v\mapsto \la\mu^{-1}v$.

Suppose now that the only $\tau$-semistable sheaf up to isomorphism
in class $\al_1$ is $E_1$, and the only in class $\al_2$ is $E_2$,
and the only in class $\al_1+\al_2$ are extensions $F$ in $0\ra
E_1\ra F\ra E_2\ra 0$. Then we have $\fM_\rss^{\al_1}(\tau)\cong
[\Spec\C/\bG_m]\cong\fM_\rss^{\al_2}(\tau)$, and
$\fM_\rss^{\al_1+\al_2}(\tau)\cong[\C^d/\bG_m^2]$. These are smooth
of dimensions $-1,-1,d-2$ respectively, and
$\fM_\rss^{\al_1+\al_2}(\tau)$ is the non-separated disjoint union
of a projective space $\CP^{d-1}$ with stabilizer groups $\bG_m$,
and a point with stabilizer group~$\bG_m^2$.

The moduli space $\M_\stp^{\al_1+\al_2,n}(\tau')$ has points
$s:\cO_X(-n)\ra F_\ep$, for $0\ra E_1\ra F_\ep\ra E_2\ra 0$ exact.
Here $F_\ep$ corresponds to some $\ep\in\Ext^1(E_2,E_1)$, and $s\in
H^0(F_\ep(n))$, where the exact sequence $0\ra E_1\ra F_\ep\ra
E_2\ra 0$ and $E_1,F_\ep,E_2$ $n$-regular give an exact sequence
\begin{equation*}
0\longra H^0(E_1(n))\longra H^0(F_\ep(n))\longra H^0(E_2(n))\longra 0.
\end{equation*}
Globally over $\ep\in\Ext^1(E_2,E_1)$ we can (noncanonically) split
this short exact sequence and identify $H^0(F_\ep(n))\cong
H^0(E_1(n))\op H^0(E_2(n))$, so $s\in H^0(F_\ep(n))$ is identified
with~$(s_1,s_2)\in H^0(E_1(n))\op
H^0(E_2(n))\cong\C^{P_{\al_1}(n)}\op \C^{P_{\al_2}(n)}$.

The condition that $s:\cO_X(-n)\ra F_\ep$ is a stable pair turns out
to be that either $\ep\ne 0$ and $s_2\ne 0$, or $\ep=0$ and
$s_1,s_2\ne 0$. The equivalence relation on triples $(s_1,s_2,\ep)$
is that $(s_1,s_2,\ep)\sim(\la s_1,\mu s_2,\la\mu^{-1}\ep)$, for
$\la\in\Aut(E_1)\cong\bG_m$ and $\mu\in\Aut(E_2)\cong\bG_m$. This
proves that
\begin{align*}
\M_\stp^{\al_1+\al_2,n}(\tau')\cong\bigl\{ (s_1,s_2,\ep)\in
\C^{P_{\al_1}(n)}\op \C^{P_{\al_2}(n)}\op \C^d: \text{$\ep\ne 0$ and
$s_2\ne 0$,}&\\
\text{or $\ep=0$ and $s_1,s_2\ne 0$}\bigr\}/\bG_m^2&.
\end{align*}

Therefore $\M_\stp^{\al_1+\al_2,n}(\tau')$ is a smooth projective
variety of dimension $P_{\al_1}(n)\!+\!P_{\al_2}(n)\!+\!d\!-\!2$, so
$\nu_{\M_\stp^{\al_1+\al_2,n}(\tau')}\!=\!(-1)^{P_{\al_1}(n)
+P_{\al_2}(n)+d-2}$. We cut $\M_\stp^{\al_1+\al_2,n}(\tau')$ into
the disjoint union of two pieces (a) points with $\ep=0$, and (b)
points with $\ep\ne 0$. Piece (a) is isomorphic to
$\CP^{P_{\al_1}(n)-1}\times\CP^{P_{\al_2}(n)-1}$, and has Euler
characteristic $P_{\al_1}(n)P_{\al_2}(n)$. Piece (b) is a vector
bundle over $\CP^{P_{\al_2}(n)-1}\times\CP^{d-1}$ with fibre
$\C^{P_{\al_1}(n)}$, and has Euler characteristic $P_{\al_2}(n)d$.
Hence $\M_\stp^{\al_1+\al_2,n}(\tau')$ has Euler characteristic
$\bigl(P_{\al_1}(n)+d\bigr)P_{\al_2}(n)$, and \eq{dt5eq16} yields
\e
PI^{\al_1+\al_2,n}(\tau')=(-1)^{P_{\al_1}(n)+P_{\al_2}(n)+d-2}
\bigl(P_{\al_1}(n)+d\bigr)P_{\al_2}(n).
\label{dt6eq8}
\e

The expression \eq{dt5eq17} for $PI^{\al_1+\al_2,n}(\tau')$ yields
\e
\begin{split}
PI^{\al_1+\al_2,n}(\tau')=-(-1)^{P_{\al_1}(n)+P_{\al_2}(n)}\bigl(
P_{\al_1}(n)+P_{\al_2}(n)\bigr)\bar{DT}{}^{\al_1+\al_2}(\tau)&\\
+\ha(-1)^{P_{\al_1}(n)}P_{\al_1}(n)(-1)^{P_{\al_2}(n)-d}\bigl(P_{\al_2}(n)
-d\bigr)\bar{DT}{}^{\al_1}(\tau)\bar{DT}{}^{\al_2}(\tau)&\\
+\ha(-1)^{P_{\al_2}(n)}P_{\al_2}(n)(-1)^{P_{\al_1}(n)+d}\bigl(P_{\al_1}(n)
+d\bigr)\bar{DT}{}^{\al_2}(\tau)\bar{DT}{}^{\al_1}(\tau)&,
\end{split}
\label{dt6eq9}
\e
where the three terms on the right correspond to splitting $\al$
into $\al=\al$ with $l=1$, into $\al=\al_1+\al_2$ with $l=2$, and
into $\al=\al_2+\al_1$ with $l=2$ respectively. We have
$\bar{DT}{}^{\al_i}(\tau)=1$ by Example \ref{dt6ex2}. So comparing
\eq{dt6eq8} and \eq{dt6eq9} shows that
$\bar{DT}{}^{\al_1+\al_2}(\tau)=(-1)^{d-1}d/2$, and
similarly~$J^{\al_1+\al_2}(\tau)=d/2$.
\label{dt6ex7}
\end{ex}

Here is a more complicated example illustrating non-smooth moduli
spaces, nontrivial Behrend functions, and failure of
deformation-invariance of the~$J^\al(\tau)$.

\begin{ex} Let $X_t$ for $t\in\C$ be a smooth family of Calabi--Yau
3-folds over $\C$, equipped with a smooth family of very ample line
bundles $\cO_{X_t}(1)$; note that our definition of Calabi--Yau
3-fold requires that $H^1(\cO_{X_t})=0$. Then by Theorem
\ref{dt4thm7} the numerical Grothendieck groups\index{Grothendieck
group!numerical} $K^\num(\coh(X_t))$ for $t\in\C$ are all
canonically isomorphic, so we identify them with
$K^\num(\coh(X_0))$. Suppose $\al\in K^\num(\coh(X_0))$, and that
\begin{equation*}
\fM_\st^\al(\tau)_t=\fM_\rss^\al(\tau)_t\cong
\Spec\bigl(\C[z]/(z^2-t^2)\bigr)\times[\Spec\C/\bG_m]
\end{equation*}
for all $t\in\C$, where the subscript $t$ means the moduli space for
$X_t$. That is, $\fM_\rss^\al(\tau)_t$ for $t\ne 0$ is the disjoint
union of two points $[\Spec\C/\bG_m]$ at $z=t$ and $z=-t$, which
correspond to rigid, stable sheaves $E_+,E_-$ with $[E_\pm]=\al$.
But $\fM_\rss^\al(\tau)_0$ is $\Spec\bigl(\C[z]/(z^2)\bigr)
\times[\Spec\C/\bG_m]$. This contains only one stable sheaf $E_0$,
whose moduli space is a double point. That is, $E_0$ has one
infinitesimal deformation, so that $\Ext^1(E_0,E_0)=\C$, but this
deformation is obstructed to second order. So the picture is that as
$t\ra 0$, the two distinct rigid stable sheaves $E_+,E_-$ come
together, and at $t=0$ they are replaced by one stable, non-rigid
sheaf $E_0$ with an infinitesimal deformation.

First consider the invariants $\bar{DT}{}^{\al}(\tau)_t$ and
$J^\al(\tau)_t$. Since $\fM_\st^\al(\tau)_t=\fM_\rss^\al(\tau)_t$ we
have $\bep^\al(\tau)_t=\bde_{\fM_\rss^\al(\tau)_t}$. When $t\ne 0$,
$\fM_\rss^\al(\tau)_t\cong[\Spec\C/\bG_m]\amalg [\Spec\C/\bG_m]$ is
smooth of dimension $-1$, so $\nu_{\fM_\rss^\al(\tau)_t}\equiv -1$.
It follows that $\ti\Psi\bigl(\bep^\al(\tau)_t\bigr)=-2\ti \la^\al$
in the notation of \S\ref{dt53}, so $\bar{DT}{}^{\al}(\tau)_t=2$ by
\eq{dt5eq7}. Similarly, $\Psi\bigl(\bep^\al(\tau)_t\bigr)=2\la^\al$
in the notation of \S\ref{dt34}, so $J^{\al}(\tau)_t=2$
by~\eq{dt3eq22}.

When $t=0$, $\fM_\rss^\al(\tau)_0$ is not smooth. As
$\Spec\bigl(\C[z]/(z^2)\bigr)=\Crit(\frac{1}{3}z^3)$, the Milnor
fibre of $\frac{1}{3}z^3$ is 3 points, and $\dim\C=1$, we have
$\nu_{\Spec(\C[z]/(z^2))}\equiv 2$ by Theorem \ref{dt4thm2}, so
$\nu_{\fM_\rss^\al(\tau)_0}=-2$ by Theorem \ref{dt4thm1}(i) and
Corollary \ref{dt4cor1}. Thus, as $\fM_\rss^\al(\tau)_0$ is a single
point with Behrend function $-2$ we have
$\ti\Psi\bigl(\bep^\al(\tau)_0\bigr)=-2\ti \la^\al$, so
$\bar{DT}{}^{\al}(\tau)_0=2$, but
$\Psi\bigl(\bep^\al(\tau)_0\bigr)=\la^\al$, so $J^{\al}(\tau)_0=1$.
To summarize,
\e
\bar{DT}{}^{\al}(\tau)_t=2, \quad\text{all $t$, and}\quad
J^{\al}(\tau)_t=\begin{cases} 2, & t\ne 0, \\ 1, & t=0.\end{cases}
\label{dt6eq10}
\e

Now let us assume that the only $\tau$-semistable sheaves in class
$2\al$ are those with stable factors in class $\al$. Thus, when
$t\ne 0$ the $\tau$-semistable sheaves in class $\al$ are $E_+\op
E_+$, and $E_-\op E_-$, and $E_+\op E_-$. Example \ref{dt6ex2} when
$m=2$ implies that $E_+\op E_+$, and $E_-\op E_-$ each contribute
$\frac{1}{4}$ to $\bar{DT}{}^{2\al}(\tau)_t$ and $-\frac{1}{4}$ to
$J^{2\al}(\tau)_t$, and Example \ref{dt6ex4} shows that $E_+\op E_-$
contributes 0 to both. Therefore $\bar{DT}{}^{2\al}(\tau)_t=\ha$
and~$J^{2\al}(\tau)_t=-\ha$.

When $t=0$, as $\Ext^1(E_0,E_0)=\C$, there is one nontrivial
extension $F$ in $0\ra E_0\ra F\ra E_0\ra 0$. Hence
$\fM_\rss^{2\al}(\tau)_0(\C)$ consists of two points $[E_0\op E_0]$
and $[F]$. Since $\Aut(E_0\op E_0)\cong\GL(2,\C)$ is the
complexification of its maximal subgroup $\U(2)$, Theorem
\ref{dt5thm3} implies that we may write $\fM_\rss^{2\al}(\tau)_0$
locally near $[E_0\op E_0]$ in the complex analytic topology as
$\Crit(f)/\Aut(E_0\op E_0)$, where $U\subseteq\Ext^1(E_0\op
E_0,E_0\op E_0)$ is an $\Aut(E_0\op E_0)$-invariant open set, and
$f:U\ra\C$ is an $\Aut(E_0\op E_0)$-invariant holomorphic function.
As $\Ext^1(E_0,E_0) \cong\C$, we may identify $\Ext^1(E_0\op
E_0,E_0\op E_0)$ with $2\times 2$ complex matrices
$A=\bigl(\begin{smallmatrix} a & b \\ c & d\end{smallmatrix}\bigr)$,
with $\Aut(E_0\op E_0)\cong\GL(2,\C)$ acting by conjugation.

Since $f$ is a conjugation-invariant holomorphic function, it must
be a function of $\Tr(A)$ and $\det(A)$. But when we restrict to
diagonal matrices $\bigl(\begin{smallmatrix} a & 0 \\ 0 &
d\end{smallmatrix}\bigr)$, $f$ must reduce to the potential defining
$\fM^\al_0\times\fM^\al_0$ at $(E_0,E_0)$. As
$\Spec\bigl(\C[z]/(z^2)\bigr)=\Crit(\frac{1}{3}z^3)$, we want
$f\bigl(\begin{smallmatrix} a & 0 \\ 0 & d\end{smallmatrix}\bigr)=
\frac{1}{3}a^3+\frac{1}{3}d^3$. But $f$ is a function of $\Tr(A)$
and $\det(A)$, so we see that
\begin{equation*}
f\begin{pmatrix} a & b \\ c & d\end{pmatrix}=\frac{1}{3}(\Tr
A)^3-\Tr A\det A= \frac{1}{3}(a^3+d^3)+(a+d)bc.
\end{equation*}
We can then take $U=\Ext^1(E_0\op E_0,E_0\op E_0)$, and we see that
$\fM_\rss^{2\al}(\tau)_0\cong[\Crit(f)/\GL(2,\C)]$ as an Artin
stack.

Now $\Crit(f)$ consists of two $\GL(2,\C)$-orbits, the point 0 which
corresponds to $E_0\op E_0$, and the orbit of
$\bigl(\begin{smallmatrix} 0 & 1 \\ 0 & 0\end{smallmatrix}\bigr)$
which corresponds to $F$. As $f$ is a homogeneous polynomial, the
Milnor fibre\index{Milnor fibre!examples} $MF_f(0)$ is diffeomorphic
to $f^{-1}(1)$. One can show that $\chi\bigl(f^{-1}(1)\bigr)=-3$, so
$\nu_{\Crit(f)}(0)=4$ by \eq{dt4eq2}, and as the projection
$\Crit(f)\ra[\Crit f/\GL(2,\C)]$ is smooth of relative dimension 4,
we deduce that $\nu_{\fM_\rss^{2\al}(\tau)_0}(E_0\op E_0)=4$ by
Corollary \ref{dt4cor1}. This also follows from
$\nu_{\fM_\rss^\al(\tau)_0}(E_0)=-2$ and equation \eq{dt5eq2}. The
orbit of $\bigl(\begin{smallmatrix} 0 & 1
\\ 0 & 0\end{smallmatrix}\bigr)$ in $\Crit(f)$ is smooth of
dimension 2, so Theorem \ref{dt4thm1}(i) gives
$\nu_{\Crit(f)}\bigl(\begin{smallmatrix} 0 & 1 \\ 0 &
0\end{smallmatrix}\bigr)=1$,
and~$\nu_{\fM_\rss^{2\al}(\tau)_0}(F)=1$.

Using the definition \eq{dt3eq4} of $\bep^{2\al}(\tau)$ and the
relations in $\oSFai(\fM,\chi,\Q)$ in \S\ref{dt24}, reasoning as in
the proof of Theorem \ref{dt5thm5} in \S\ref{dt11} we can show that
\e
\bar\Pi^{\chi,\Q}_\fM\bigl(\bep^{2\al}(\tau)_0\bigr)=\ts
-\frac{1}{4}\bigl[([\Spec\C/\bG_m],\rho_{E_0\op E_0})\bigr]+
\frac{1}{2}\bigl[([\Spec\C/\bG_m],\rho_{F})\bigr],
\label{dt6eq11}
\e
where $\rho_{E_0\op E_0},\rho_{F}$ map $[\Spec\C/\bG_m]$ to $E_0\op
E_0$ and $F$ respectively. So Definition \ref{dt5def2} gives
\e
\bar{DT}{}^{2\al}(\tau)_0\!=-\!\bigl(\ts -\frac{1}{4}
\nu_{\fM_\rss^{2\al}(\tau)_0}(E_0\op E_0)\!+\!\frac{1}{2}
\nu_{\fM_\rss^{2\al}(\tau)_0}(F)\bigr)\!=\!\frac{1}{4}\cdot 4
\!-\!\frac{1}{2}\cdot 1\!=\!\frac{1}{2}.
\label{dt6eq12}
\e
Similarly $J^{2\al}(\tau)_0=-\frac{1}{4}+\frac{1}{2}=\frac{1}{4}$.
To summarize,
\e
\bar{DT}{}^{2\al}(\tau)_t=\ha, \quad\text{all $t$, and}\quad
J^{2\al}(\tau)_t=\begin{cases} -\ha, & t\ne 0, \\
\phantom{-}\frac{1}{4}, & t=0.\end{cases}
\label{dt6eq13}
\e
\label{dt6ex8}
\end{ex}

Equations \eq{dt6eq10}, \eq{dt6eq13} illustrate the fact that the
$\bar{DT}{}^\al(\tau)$ are deformation-invariant, as in Corollary
\ref{dt5cor4}, but the $J^\al(\tau)$ of \S\ref{dt35} are not.

\subsection{Integrality properties of the $\bar{DT}{}^\al(\tau)$}
\label{dt62}\index{Donaldson--Thomas invariants!integrality properties|(}

This subsection is based on ideas in Kontsevich and Soibelman
\cite[\S 2.5 \& \S 7.1]{KoSo1}. Example \ref{dt6ex2} shows that
given a rigid $\tau$-stable sheaf $E$ in class $\al$, the sheaves
$mE$ contribute $1/m^2$ to $\bar{DT}{}^{m\al}(\tau)$ for all $m\ge
1$. We can regard this as a kind of `multiple cover formula',
analogous to the well known Aspinwall--Morrison computation for a
Calabi--Yau 3-fold $X$ that a rigid embedded $\CP^1$ in class
$\al\in H_2(X;\Z)$ contributes $1/m^3$ to the genus zero
Gromov--Witten invariant\index{Gromov--Witten invariants} of $X$ in
class $m\al$ for all $m\ge 1$. So we can define new invariants
$\hat{DT}{}^\al(\tau)$ which subtract out these contributions from
$mE$ for~$m>1$.

\begin{dfn} Let $X$ be a projective Calabi--Yau 3-fold over $\C$,
let $\cO_X(1)$ be a very ample line bundle on $X$, and let
$(\tau,T,\le)$ be a weak stability condition on $\coh(X)$ of
Gieseker or $\mu$-stability type. Then Definition \ref{dt5def2}
defines generalized Donaldson--Thomas invariants
$\bar{DT}{}^\al(\tau)\in\Q$ for $\al\in C(\coh(X))$.

Let us define new invariants $\hat{DT}{}^\al(\tau)$ for $\al\in
C(\coh(X))$ to satisfy\nomenclature[DTc]{$\hat{DT}{}^\al(\tau)$}{BPS invariants
of a Calabi--Yau 3-fold}
\e
\bar{DT}{}^\al(\tau)=\sum_{m\ge 1,\; m\mid\al}\frac{1}{m^2}\,
\hat{DT}{}^{\al/m}(\tau).
\label{dt6eq14}
\e
We can invert \eq{dt6eq14} explicitly to write
$\hat{DT}{}^\al(\tau)$ in terms of the $\bar{DT}{}^{\al/m}(\tau)$.
The {\it M\"obius function\/}\nomenclature[Mo(m)]{$\Mo(m)$}{M\"obius
function}\index{Mobius function@M\"obius function} $\Mo:\N\ra\{-1,0,1\}$
in elementary number theory and combinatorics is given by
$\Mo(n)=(-1)^d$ if $n=1,2,\ldots$ is square-free and has $d$ prime
factors, and $\Mo(n)=0$ if $n$ is not square-free. Then the {\it
M\"obius inversion formula\/}\index{Mobius inversion formula@M\"obius
inversion formula} says that if $f,g:\N\ra\Q$ are functions with
$g(n)=\sum_{m\mid n}f(n/m)$ for $n=1,2,\ldots$ then
$f(n)=\sum_{m\mid n}\Mo(m)g(n/m)$ for $n=1,2,\ldots$. Suppose
$\be\in C(\coh(X))$ is primitive. Applying the M\"obius inversion
formula with $f(n)=n^2\smash{\hat{DT}{}^{n\be}(\tau)}$ and
$g(n)=n^2\bar{DT}{}^{n\be}(\tau)$, we find the inverse of
\eq{dt6eq14} is
\e
\hat{DT}{}^\al(\tau)=\sum_{m\ge 1,\; m\mid\al}\frac{\Mo(m)}{m^2}\,
\bar{DT}{}^{\al/m}(\tau).
\label{dt6eq15}
\e

We take \eq{dt6eq15} to be the {\it definition\/} of
$\hat{DT}{}^\al(\tau)$, and then reversing the argument shows that
\eq{dt6eq14} holds. The $\hat{DT}{}^\al(\tau)$ are our analogues of
invariants $\Om(\al)$ discussed in \cite[\S 2.5 \& \S 7.1]{KoSo1}.
We call $\hat{DT}{}^\al(\tau)$ the {\it BPS invariants\/}\index{BPS
invariants}\index{Donaldson--Thomas invariants!BPS invariants
$\hat{DT}{}^\al(\tau)$} of $X$, since Kontsevich and Soibelman
suggest that their $\Om(\al)$ count BPS states. The coefficients
$1/m^2$ in \eq{dt6eq14} are related to the appearance of {\it
dilogarithms\/}\index{dilogarithm} in Kontsevich and Soibelman
\cite[\S 2.5]{KoSo1}. The
dilogarithm\nomenclature[Li2]{$Li_2(t)$}{dilogarithm function} is
$Li_2(t)=\sum_{m\ge 1}t^m/m^2$ for $\md{t}<1$, and the inverse
function for $Li_2$ near $t=0$ is $Li_2^{-1}(t)=\sum_{m\ge
1}\Mo(m)t^m/m^2$ for $\md{t}<1$, with power series coefficients
$\Mo(m)/m^2$ as in~\eq{dt6eq15}.
\label{dt6def1}
\end{dfn}

If $\M_\rss^\al(\tau)=\M_\st^\al(\tau)$ then $\M_\rss^{\al/m}(\tau)
=\es$ for all $m\ge 2$ dividing $\al$, since if $[E]\in
\M_\rss^{\al/m}(\tau)$ then $[mE]\in \M_\rss^\al(\tau)\sm
\M_\st^\al(\tau)$. So $\bar{DT}{}^{\al/m}(\tau)=0$, and hence
\eq{dt6eq15} and Proposition \ref{dt5prop2} give:

\begin{prop} If\/ $\M_\rss^\al(\tau)=\M_\st^\al(\tau)$ then
$\hat{DT}{}^\al(\tau)=DT^\al(\tau)$.
\label{dt6prop1}
\end{prop}

Thus the $\hat{DT}{}^\al(\tau)$ are also generalizations of
Donaldson--Thomas invariants $DT^\al(\tau)$. Using \eq{dt6eq14} we
evaluate the $\hat{DT}{}^\al(\tau)$ in each of the examples
of~\S\ref{dt61}:
\begin{itemize}
\setlength{\itemsep}{0pt}
\setlength{\parsep}{0pt}
\item In Examples \ref{dt6ex1}--\ref{dt6ex2} we have
$\hat{DT}{}^{\al}(\tau)=1$ and $\hat{DT}{}^{m\al}(\tau)=0$ for
all $m>1$. Thus, a rigid stable sheaf $E$ and its `multiple
covers' $mE$ for $m\ge 2$ contribute 1 to
$\hat{DT}{}^{\al}(\tau)$ and 0 to $\hat{DT}{}^{m\al}(\tau)$ for
$m\ge 2$. The point of \eq{dt6eq14} was to achieve this, as it
suggests that the $\hat{DT}{}^\al(\tau)$ are a more meaningful
way to `count' stable sheaves.
\item In Examples \ref{dt6ex3}--\ref{dt6ex4}
$\hat{DT}{}^{\al_i}(\tau)=1$ for all $i=1,\ldots,l$, and
$\hat{DT}{}^{\sum_{i\in I}\al_i}(\tau)=0$ for all subsets
$I\subseteq\{1,\ldots,l\}$ with $\md{I}\ge 2$.
\item In Examples \ref{dt6ex5}--\ref{dt6ex6}
$\hat{DT}{}^{m_1\al_1+\cdots+m_l\al_l}(\tau)=1$ if $m_i=1$ for some
$i=1,\ldots,l$ and $m_j=0$ for $i\ne j$, and
$\hat{DT}{}^{m_1\al_1+\cdots+m_l\al_l}(\tau)=0$ otherwise.
\item In Example \ref{dt6ex7}
$\hat{DT}{}^{\al_1+\al_2}(\tau)=(-1)^{d-1}d/2$, where
$\bar\chi(\al_1,\al_2)=d$. Note that
$\hat{DT}{}^{\al_1+\al_2}(\tau)\notin\Z$ when $d$ is odd.
\item In Example \ref{dt6ex8}
$\hat{DT}{}^\al(\tau)_t=2$ and $\hat{DT}{}^{2\al}(\tau)_t=0$.
\end{itemize}

Here is our version of a conjecture by Kontsevich and
Soibelman~\cite[Conj.~6]{KoSo1}.

\begin{conj} Let\/ $X$ be a Calabi--Yau $3$-fold over $\C,$ and\/
$(\tau,T,\le)$ a weak stability condition on $\coh(X)$ of Gieseker
or $\mu$-stability type. Call\/ $(\tau,T,\le)$
\begin{bfseries}generic\end{bfseries}\index{stability condition!generic|(}
if for all\/ $\al,\be\in C(\coh(X))$ with\/ $\tau(\al)=\tau(\be)$ we
have\/~$\bar\chi(\al,\be)=0$.

If\/ $(\tau,T,\le)$ is generic, then $\hat{DT}{}^\al(\tau)\in\Z$ for
all\/ $\al\in C(\coh(X))$.
\label{dt6conj1}
\end{conj}

Kontsevich and Soibelman deal with Bridgeland stability conditions
on derived categories, and their notion of {\it generic\/} stability
condition is stronger than ours: they require that
$\tau(\al)=\tau(\be)$ implies $\al,\be$ are linearly dependent in
$\Z$. But we believe $\bar\chi(\al,\be)=0$ is sufficient. Note that
Conjecture \ref{dt6conj1} holds in the examples above: the only case
in which $\hat{DT}{}^\al(\tau)\notin\Z$ is Example \ref{dt6ex7} when
$d$ is odd, and then $(\tau,T,\le)$ is not generic, as
$\tau(\al_1)=\tau(\al_2)$ but~$\bar\chi(\al_1,\al_2)=d\ne 0$.

Suppose now that $(\tau,T,\le)$ is a {\it stability condition}, such
as Gieseker stability, rather than a weak stability condition. This
is necessary for decomposition of $\tau$-semistables into
$\tau$-stables to be well-behaved, as in \cite[Th.~4.5]{Joyc5}. Then
as in \eq{dt5eq9} we can write $\bar{DT}{}^\al(\tau)$ as the Euler
characteristic of the {\it coarse moduli scheme\/}\index{coarse
moduli scheme}\index{moduli scheme!coarse} $\M_\rss^\al(\tau)$
weighted by a constructible function.\index{constructible function}
(The existence of coarse moduli schemes $\M_\rss^\al(\tau)$ is known
in the two cases we consider in this book, Gieseker stability for
coherent sheaves and slope stability for quiver representations. For
the argument below to work, we do not need $\M_\rss^\al(\tau)$ to be
a scheme, but only a constructible set, the quotient of
$\fM_\rss^\al(\tau)(\C)$ by a constructible equivalence relation,
and this should always be true.)

We will write $\hat{DT}{}^\al(\tau)$ as a weighted Euler
characteristic of $\M_\rss^\al(\tau)$ in the same way. For $m\ge 1$,
define a 1-morphism $P_m:\fM\ra\fM$
\nomenclature[Pm]{$P_m$}{1-morphism $\fM\ra\fM$ taking $E\mapsto
mE=E\op\cdots\op E$} by $P_m:[E]\mapsto [mE]$ for $E\in\coh(X)$,
where $mE={\buildrel {\!\ulcorner\,\text{$m$ copies }\,\urcorner\!}
\over {E\op\cdots \op E}}$. Then from equations \eq{dt2eq1},
\eq{dt5eq9} and \eq{dt6eq15}, for $\al\in C(\coh(X))$ we deduce that
\e
\begin{gathered}
\hat{DT}{}^\al(\tau)=\chi\bigl(\M_\rss^\al(\tau),F^\al(\tau)
\bigr), \qquad\text{where}\\
\!\!\!F^\al(\tau)\!=\!-\!\!\sum_{\!\!\!\!\!\!\!m\ge 1,\;
m\mid\al\!\!\!\!\!\!} \frac{\Mo(m)}{m^2}
\CF^\na(\pi)\bigl[\CF^\na(P_m)\!\ci\!\Pi_{\CF}\!\ci\!
\bar\Pi^{\chi,\Q}_\fM(\bep^{\al/m}(\tau))\cdot\nu_\fM\bigr],\!\!
\end{gathered}
\label{dt6eq16}
\e
and $\pi:\fM_\rss^\al(\tau)\ra\M_\rss^\al(\tau)$ is the projection
to the coarse moduli
scheme.\nomenclature[F\alpha(t)]{$F^\al(\tau)$}{function in
$\CF(\M_\rss^\al(\tau))$ with $\chi(\M_\rss^\al(\tau),F^\al(\tau))
=\hat{DT}{}^\al(\tau)$} The following conjecture implies Conjecture
\ref{dt6conj1}, at least for stability conditions rather than weak
stability conditions.

\begin{conj} Let\/ $X$ be a Calabi--Yau\/ $3$-fold over $\C,$ and\/
$(\tau,G,\le)$ a generic Gieseker stability condition. Then the
functions\/ $F^\al(\tau)\in\CF(\M_\rss^\al(\tau))$ of\/ \eq{dt6eq16}
are\/ $\Z$-valued for all\/~$\al\in C(\coh(X))$.
\label{dt6conj2}
\end{conj}

That is, the contributions to $\hat{DT}{}^\al(\tau)$ from each
S-equivalence\index{S-equivalence} class of $\tau$-semi\-stables (or
each $\tau$-polystable)\index{polystable@$\tau$-polystable} are
integral. In \S\ref{dt76} we will prove versions of Conjectures
\ref{dt6conj1} and \ref{dt6conj2} for Donaldson--Thomas type
invariants $\hat{DT}{}^{\bs d}_Q(\mu)$ for quivers without
relations. By analogy with Question \ref{dt5quest1}(a), we can ask:

\begin{quest} Suppose Conjecture {\rm\ref{dt6conj2}} is true. For
generic $(\tau,G,\le),$ does there exist a natural perverse
sheaf\/\index{perverse sheaf} $\cal Q$ on $\M_\rss^\al(\tau)$ with\/
$\chi_{\M_\rss^\al(\tau)}({\cal Q})\equiv F^\al(\tau)?$
\label{dt6quest1}
\end{quest}

Such a perverse sheaf $\cal Q$ would be interesting as it would
provide a `categorification'\index{Donaldson--Thomas
invariants!categorification} of the BPS invariants
$\hat{DT}{}^\al(\tau)$, and help explain their integrality.

We can also ask whether the unweighted invariants $J^\al(\tau)$ of
\S\ref{dt35} also have similar integrality properties to those
suggested in Conjectures \ref{dt6conj1} and \ref{dt6conj2}. The
answer is no. Following the argument above but using \eq{dt6eq4}
rather than \eq{dt6eq3}, one would expect that the correct analogue
of \eq{dt6eq14} is
\begin{equation*}
J^\al(\tau)=\sum\nolimits_{m\ge 1,\;
m\mid\al}\frac{(-1)^{m-1}}{m^2}\, \hat J{}^{\al/m}(\tau).
\end{equation*}
But then in Example \ref{dt6ex8}, from \eq{dt6eq10} and \eq{dt6eq13}
we see that $\hat J{}^{2\al}(\tau)_0=\ha$, so the $\hat
J{}^\al(\tau)$ need not be integers even for a generic stability
condition. In fact, using \eq{dt6eq4} in Example \ref{dt6ex2} and
\eq{dt6eq10} and \eq{dt6eq13} when $t=0$ in Example \ref{dt6ex8},
one can show that there is no universal formula with
$c_1,c_2,\ldots\in\Q$ and $c_1=1$ defining
\begin{equation*}
\hat J^\al(\tau)=\ts\sum_{m\ge 1,\; m\mid\al}c_m\,J^{\al/m}(\tau),
\end{equation*}
such that $\hat J^\al(\tau)\in\Z$ for all generic $(\tau,T,\le)$ and
$\al\in C(\coh(X))$. One conclusion (at least if you believe
Conjecture \ref{dt6conj1}) is that counting sheaves weighted by the
Behrend function is essential to ensure good integrality
properties.\index{Donaldson--Thomas invariants!integrality
properties|)}\index{stability condition!generic|)}

\subsection{Counting dimension zero sheaves}
\label{dt63}

Let $X$ be a Calabi--Yau 3-fold over $\C$ with $H^1(\cO_X)=0$, let
$\cO_X(1)$ be a very ample line bundle on $X$, and $(\tau,G,\le)$
the associated Gieseker stability condition on $\coh(X)$, as in
Example \ref{dt3ex1}. For $x\in X(\C)$, write $\cO_{x}$ for the
skyscraper sheaf at $x$, and define $p=[\cO_x]$ in
$K^\num(\coh(X))$, the `point class' on $X$. Then $p$ is independent
of the choice of $x$ in $X(\C)$, as $X$ is connected.

For $d\ge 0$, {\it the Hilbert scheme $\Hilb^dX$ of\/ $d$ points
on\/}\index{Hilbert scheme}\nomenclature[Hilb(X)]{$\Hilb^d(X)$}{Hilbert scheme of
$d$ points on $X$} $X$ parametrizes 0-dimensional subschemes $S$ of
$X$ of length $d$. It is a projective $\C$-scheme, which is singular
for $d\ge 4$, and for $d\gg 0$ has many irreducible components. The
virtual count of $\Hilb^dX$ may be written as a weighted Euler
characteristic $\chi\bigl(\Hilb^dX,\nu_{\Hilb^dX}\bigr)$ as in
\S\ref{dt43}. Values for these virtual counts were conjectured by
Maulik et al.\ \cite[Conj.~1]{MNOP1}, and different proofs are given
by Behrend and Fantechi \cite[Th.~4.12]{BeFa2}, Li
\cite[Th.~0.2]{Li}, and Levine and Pandharipande~\cite[\S
14.1]{LePa}.

\begin{thm} $\sum_{d=0}^\iy\chi\bigl(\Hilb^dX,\nu_{\Hilb^dX}\bigr)
\,s^d\!=\!M(-s)^{\chi(X)},$ where $\chi(X)$ is the Euler
characteristic of\/ $X,$ and\/ $M(q)\!=\!\prod_{k\ge 1}(1\!-\!
q^k)^{-k}$ the MacMahon function.\index{MacMahon function}
\label{dt6thm1}
\end{thm}

We will compute the generalized Donaldson--Thomas invariants
$\bar{DT}{}^{dp}(\tau)$ counting dimension 0 sheaves in class $dp$
for $d\ge 1$. Our calculation is parallel to Kontsevich and
Soibelman \cite[\S 6.5]{KoSo1}. First consider the pair invariants
$PI^{dp,n}(\tau')$. These count stable pairs $s:\cO_X(-n)\ra E$ for
$E\in\coh(X)$ with $[E]=dp$. The condition for $s:\cO_X(-n)\ra E$ to
be a stable pair is simply that $s$ is surjective. Tensoring by
$\cO_X(n)$ gives a morphism $s(n):\cO_X\ra E(n)$. But $E(n)\cong E$
as $E$ has dimension 0. Thus, tensoring by $\cO_X(n)$ gives an
isomorphism $\M_\stp^{dp,n}(\tau')\cong\M_\stp^{dp,0}(\tau')$, so
that $\M_\stp^{dp,n}(\tau')$ and $PI^{dp,n}(\tau')$ are independent
of $n$. Furthermore, $\smash{\M_\stp^{dp,0}(\tau')}$ parametrizes
surjective $s:\cO_X\ra E$ for $E\in\coh(X)$ with $[E]=dp$, which are
just points of $\Hilb^dX$. Therefore
$\M_\stp^{dp,n}(\tau')\cong\Hilb^dX$, and
\e
PI^{dp,n}(\tau')=\chi\bigl(\Hilb^dX,\nu_{\Hilb^dX}\bigr),
\quad\text{for all $n\in\Z$ and $d\ge 0$.}
\label{dt6eq17}
\e

We have $\tau(dp)=1$ in $G$, and any $\be,\ga\in C(\coh(X))$ with
$\tau(\be)=\tau(\ga)=1$ are of the form $\be=dp$, $\ga=ep$, so that
$\bar\chi(\be,\ga)=0$. Therefore Proposition \ref{dt5prop3} applies
with $t=1$ in $G$. So from Theorem \ref{dt6thm1} and \eq{dt6eq17} we
see that
\e
\begin{split}
&M(-s)^{\chi(X)}= 1+\sum_{d\ge 1} PI^{dp,n}(\tau')s^d=\\
&\quad\exp
\raisebox{-4pt}{\begin{Large}$\displaystyle\Bigl[$\end{Large}}
-\sum_{d\ge 1}(-1)^{\bar\chi([\cO_X(-n)],dp)}
\bar\chi\bigl([\cO_X(-n)],dp\bigr)\bar{DT}{}^{dp}(\tau)s^d
\raisebox{-4pt}{\begin{Large}$\displaystyle\Bigr]$\end{Large}}.
\end{split}
\label{dt6eq18}
\e
Here we have replaced the sums over $\al\in C(\coh(X))$ with
$\tau(\al)=1$ by a sum over $dp$ for $d\ge 1$, and used the formal
variable $s$ in place of $q^p$ in \eq{dt5eq20}, so that $q^{dp}$ is
replaced by~$s^d$.

Now $\bar\chi([\cO_X(-n)],p)\!=\!\sum_i(-1)^i\dim
H^i(\cO_x(n))\!=\!1$, so $\bar\chi([\cO_X(-n)],dp)\!=\!d$.
Substituting this into \eq{dt6eq18}, taking logs, and using
$M(q)\!=\!\prod_{k\ge 1}(1\!-\!q^k)^{-k}$ yields
\begin{equation*}
-\sum_{d\ge 1}(-1)^dd\,\bar{DT}{}^{dp}(\tau)s^d\!=\!\chi(X)\!
\sum_{k\ge 1} (-k)\log \bigl(1\!-\!(-s)^k\bigr)\!=\!\chi(X)\!
\sum_{k,l\ge 1}\frac{k}{l}(-s)^{kl}.
\end{equation*}
Equating coefficients of $s^d$ yields after a short calculation
\e
\bar{DT}{}^{dp}(\tau)=-\chi(X)\sum_{l\ge 1, \; l \mid
d}\frac{1}{l^2}.
\label{dt6eq19}
\e
So from \eq{dt6eq14} we deduce that
\e
\hat{DT}{}^{dp}(\tau)=-\chi(X),\quad\text{all $d\ge 1$.}
\label{dt6eq20}
\e
This is a satisfying result, and confirms Conjecture \ref{dt6conj1}
for dimension 0 sheaves.

\subsection{Counting dimension one sheaves}
\label{dt64}

Let $X$ be a Calabi--Yau 3-fold over $\C$, and $\cO_X(1)$ a very
ample line bundle on $X$. From \S\ref{dt45} the Chern
character\index{Chern character} $\ch$ identifies $K^\num(\coh(X))$
with a particular lattice $\La_X$ in $H^{\rm even}(X;\Q)$, so we may
write $\al\in K^\num(\coh(X))$ as $(\al_0,\al_2,\al_4,\al_6)$ with
$\al_{2j}\!\in\!H^{2j}(X;\Q)$. If $E\!\ra\!X$ is a vector bundle
with $[E]\!=\!\al$ then~$\al_0\!=\!\rank E\!\in\!\Z$.

Let us consider invariants $\bar{DT}{}^\al(\tau),\hat{DT}{}^\al
(\tau)$ counting pure sheaves $E$ of dimension 1 on $X$, that is,
sheaves $E$ supported on curves $C$ in $X$. If $[E]=\al$ then
$\al_0=\al_2=0$ for dimensional reasons, so we may write
$\al=(0,0,\be,k)$. By \cite[\S A.4]{Hart2} we have $\be=-c_2(E)$ and
$k=\ha c_3(E)$, where $c_2(E)\in H^4(X;\Z)\cong H_2(X;\Z)$ and
$c_3(E)\in H^6(X;\Z)\cong\Z$ are Chern classes of $E$. Write
$\de=c_1(\cO_X(1))$. If $[E]=\al$ then $[E(n)]=\exp(n\de)\al=
(0,0,\be,k+n\be\cup\de)$. Hence by the Hirzebruch--Riemann--Roch
Theorem\index{Hirzebruch--Riemann--Roch Theorem}
\cite[Th.~A.4.1]{Hart2} we have
\e
\begin{split}
\chi(E(n))&=\deg\bigl(\ch(E(n)\cdot{\rm td}(TX)\bigr)_3\\
&=\deg\bigl( (0,0,\be,k+n\be\cup\de)\cdot (1,0,*,*)\bigr)_3=
k+n\be\cup\de,
\end{split}
\label{dt6eq21}
\e
using $c_1(X)=0$ to simplify ${\rm td}(TX)$. So the Hilbert
polynomial of $E$ is
\e
P_{(0,0,\be,k)}(t)=(\be\cup\de)\,t+k.
\label{dt6eq22}
\e
Note that $\be\cup\de,k\in\Z$ as $P_{(0,0,\be,k)}$ maps $\Z\ra\Z$.
Note too that for dimension 1 sheaves, Gieseker stability in Example
\ref{dt3ex1} and $\mu$-stability in Example \ref{dt3ex2} coincide,
since truncating a degree 1 polynomial at its second term has no
effect.

Here are some properties of the $\bar{DT}{}^\al(\tau),\hat{DT}{}^\al
(\tau)$ in dimension 1. Part (a) may be new, and answers a question
of Sheldon Katz in \cite[\S 3.2]{Katz}. The proof of (c) uses
$H^1(\cO_X)=0$, which we assume for Calabi--Yau 3-folds in this
section.

\begin{thm} Let\/ $X$ be a Calabi--Yau $3$-fold over $\C,$ and\/
$(\tau,T,\le)$ a weak stability condition on $\coh(X)$ of Gieseker
or $\mu$-stability type. Consider invariants
$\bar{DT}{}^{(0,0,\be,k)}(\tau),\hat{DT}{}^{(0,0,\be,k)}(\tau)$ for
$0\ne\be\in H^4(X;\Z)$ and\/ $k\in\Z$ counting $\tau$-semistable
dimension $1$ sheaves in $X$. Then
\begin{itemize}
\setlength{\itemsep}{0pt}
\setlength{\parsep}{0pt}
\item[{\rm(a)}] $\bar{DT}{}^{(0,0,\be,k)}(\tau),\hat{DT}{}^{(0,0,
\be,k)}(\tau)$ are independent of the choice of\/~$(\tau,T,\le)$.
\item[{\rm(b)}] Assume Conjecture {\rm\ref{dt6conj1}} holds.
Then\/~$\hat{DT}{}^{(0,0,\be,k)}(\tau)\in\Z$.
\item[{\rm(c)}] For any\/ $l\in \be\cup H^2(X;\Z)\subseteq\Z$ we
have $\bar{DT}{}^{(0,0,\be,k)}(\tau)=\bar{DT}{}^{(0,0,\be,k+l)}
(\tau)$ and\/ $\hat{DT}{}^{(0,0,\be,k)}(\tau)
=\hat{DT}{}^{(0,0,\be,k+l)}(\tau)$.
\end{itemize}
\label{dt6thm2}
\end{thm}

\begin{proof} For (a), let $(\tau,T,\le),(\ti\tau,\ti T,\le)$ be
two weak stability conditions on $\coh(X)$ of Gieseker or
$\mu$-stability type. Then Corollary \ref{dt5cor2} shows that we may
write $\bar{DT}{}^{(0,0,\be,k)}(\ti\tau)$ in terms of the
$\bar{DT}{}^{(0,0,\be',k')}(\tau)$ by finitely many applications of
the transformation law \eq{dt5eq14}. Now each of these changes of
stability condition involves only sheaves in the abelian subcategory
$\coh_{\le 1}(X)$ of sheaves in $\coh(X)$ with dimension $\le 1$.
However, the Euler form\index{Euler form} $\bar\chi$ vanishes on
$K_0(\coh_{\le 1}(X))$ for dimensional reasons. Each term
$I,\ka,\Ga$ in \eq{dt5eq14} has $\md{I}-1$ factors
$\bar\chi\bigl(\ka(i),\ka(j)\bigr)$ so all terms with $\md{I}>1$ are
zero as $\bar\chi\equiv 0$ on this part of $\coh(X)$. So
\eq{dt5eq14} reduces to $\bar{DT}{}^\al(\ti\tau)=
\bar{DT}{}^\al(\tau)$. Therefore each of the finitely many
applications of \eq{dt5eq14} leaves the
$\bar{DT}{}^{(0,0,\be',k')}(\tau)$ unchanged, proving~(a).

For (b), note that any stability condition $(\tau,T,\le)$ on
$\coh(X)$ is generic on $\coh_{\le 1}(X)$, since $\bar\chi=0$ on
$K_0(\coh_{\le 1}(X))$. Alternatively, one can show that if
$\smash{\ti\cO_X(1)}$ is a sufficiently general very ample line
bundle on $X$ the resulting Gieseker stability condition
$(\ti\tau,G,\le)$ is generic on all of $\coh(X)$, and then apply
(a). Either way, Conjecture \ref{dt6conj1} implies
that~$\hat{DT}{}^{(0,0,\be,k)}(\tau)\in\Z$.

For (c), let $L\ra X$ be a line bundle, let $(\tau,T,\le)$ be a weak
stability condition on $\coh(X)$ of Gieseker or $\mu$-stability
type, and define another weak stability condition $(\ti\tau,T,\le)$
on $\coh(X)$ by $\ti\tau([E])=\tau([E\ot L^{-1}])$. There is an
automorphism $F^L:\coh(X)\ra\coh(X)$ acting as $F^L:E\mapsto E\ot L$
on objects. It induces a 1-isomorphism $F^L_*:\fM\ra\fM$. Also $E$
is $\tau$-semistable if and only if $E\ot L$ is
$\ti\tau$-semistable, so $F^L_*$ maps $\fM_\rss^\al(\tau)\ra
\fM_\rss^{F^L_*(\al)}(\ti\tau)$, where
$F^L_*(\al)=\exp(c_1(L))\cdot\al$ in~$H^{\rm even}(X;\Q)$.

Clearly we have $\hat{DT}{}^\al(\tau)=\hat{DT}{}^{F^L_*(\al)}
(\ti\tau)$ for all $\al\in C(\coh(X))$. When $\al=(0,0,\be,k)$ we
have $F^L_*(\al)=(0,0,\be,k+\be\cup c_1(L))$. Thus
\begin{equation*}
\bar{DT}{}^{(0,0,\be,k)}(\tau)=\bar{DT}{}^{(0,0,\be,k+\be\cup
c_1(L))}(\ti\tau)=\bar{DT}{}^{(0,0,\be,k+\be\cup c_1(L))}(\tau),
\end{equation*}
by (a). Since $H^1(\cO_X)=0$, $c_1(L)$ can take any value in
$H^2(X;\Z)$, and so $\be\cup c_1(L)$ can take any value $l\in
\be\cup H^2(X;\Z)$, proving the first part of (c). The second part
follows by~\eq{dt6eq15}.
\end{proof}

We will compute contributions to $\bar{DT}{}^{(0,0,\be,k)}
(\tau),\hat{DT}{}^{(0,0,\be,k)}(\tau)$ from sheaves supported on
nice curves $C$ in $X$. We begin with a rigid rational curve. The
proof of the next proposition is based on Hosono et
al.~\cite[Prop.~4.3]{HST}.

\begin{prop} Let\/ $X$ be a Calabi--Yau $3$-fold over\/ $\C,$ and\/
$(\tau,T,\le)$ a weak stability condition on\/ $\coh(X)$ of Gieseker
or $\mu$-stability type. Suppose $i:\CP^1\ra X$ is an embedding,
and\/ $i(\CP^1)$ has normal bundle $\cO_{\CP^1}(-1)\op
\cO_{\CP^1}(-1)$. Then the only $\tau$-semistable dimension $1$
sheaves supported set-theoretically on $i(\CP^1)$ in $X$ are
$i_*\bigl(m\cO_{\CP^1}(k)\bigr)$ for $m\ge 1$ and\/ $k\in\Z,$ and
these sheaves are rigid in~$\coh(X)$.
\label{dt6prop2}
\end{prop}

\begin{proof} Let $\be\in H^4(X;\Z)$ be Poincar\'e dual to
$[i(\CP^1)]\in H_2(X;\Z)$. Suppose $E$ is a pure dimension 1 sheaf
supported on $i(\CP^1)$ in $X$. Then $[E]=(0,0,m\be,k)$ in
$K^\num(\coh(X))\subset H^{\rm even}(X;\Q)$, where $m\ge 1$ is the
multiplicity of $E$ at a generic point of $i(\CP^1)$. Any subsheaf
$0\ne E'\subset E$ has $[E']=(0,0,m'\be,k')$ for $1\le m'\le m$ and~
$k'\in\Z$. Let $(\tau,T,\le)$ be defined using an ample line bundle
$\cO_X(1)$ with $c_1(L)=\de$. Then by \eq{dt6eq22}, the Hilbert
polynomials\index{Hilbert polynomial} of $E$ and $E'$ are
$m(\be\cup\de)\,t+k$ and $m'(\be\cup\de)\,t+k'$,
where~$\be\cup\de>0$.

By Example \ref{dt3ex1} or \ref{dt3ex2}, $E$ is $\tau$-semistable if
and only if, for all $0\ne E'\subset E$, when $[E']=(0,0,m'\be,k')$,
we have $k'/m'(\be\cup\de)\le k/m(\be\cup\de)$, that is, $k'/m\le
k/m$. Note that this condition is {\it independent of the choice of
stability condition\/} $(\tau,T,\le)$. This is a stronger analogue
of Theorem \ref{dt6thm2}(a): if $\Si\subset X$ is an irreducible
curve in $X$, then the moduli stacks $\fM_\rss^\al(\tau)_\Si$ of
$\tau$-semistable sheaves supported on $\Si$ are independent
of~$(\tau,T,\le)$.

Suppose $E$ is $\tau$-semistable and dimension 1 with
$[E]=(0,0,m\be,k)$. Locally in the \'etale or complex analytic
topology near $i(\CP^1)$ in $X$ we can find a line bundle $L$ such
that $i^*(L)\cong \cO_{\CP^1}(1)$ (this holds as the obstructions to
finding such an $L$ lie in
$\Ext^2\bigl(\cO_{\CP^1}(1),\cO_{\CP^1}(1) \bigr)$, which is zero as
$2>\dim\CP^1$). Then $[E\ot L^n]=(0,0,m\be,k+mn)$ for $n\in\Z$, and
$E\ot L^n$ is $\tau$-semistable by the proof of Theorem
\ref{dt6thm2}(c) and $\fM_\rss^\al(\tau)_{i(\CP^1)}$ independent of
$\tau$ above. Let $n\in\Z$ be unique such that $d=k+mn$ lies in
$\{1,2,\ldots,m\}$. Then $[E\ot L^n]=(0,0,m\be,d)$, so $\chi(E\ot
L^n)=P_{[E\ot L^n]}(0)=d$ by \eq{dt6eq22}. But $H^i(E\ot L^n)=0$ for
$i>1$ as $E\ot L^n$ is supported in dimension 1, so $\dim H^0(E\ot
L^n)\ge\dim H^0(E\ot L^n)-\dim H^1(E\ot L^n)=d>0$, and we can choose
$0\ne s\in H^0(E\ot L^n)$.

Thus we have a nonzero morphism $s:\cO_X\ra E\ot L^n$ in $\coh(X)$.
Write $J$ for the image and $K$ for the kernel of $s$. Then $0\ne
J\subset E\ot L^n$ and $K\subset\cO_X$. As $E\ot L^n$ is pure of
dimension 1, $J$ is pure of dimension 1. Let $I$ be the ideal sheaf
of $i(\CP^1)$. Since supp$(J)=i(\CP^1)$ which is reducible and
reduced we see that $K\subset I\subset\cO_X$. Consider the two cases
(a) $K=I$ and (b) $K\ne I$. In case (a) we have
$J=\cO_{i(\CP^1)}=i_*(\cO_{\CP^1}(0))$, which has class
$[J]=(0,0,\be,1)$. Since $E\ot L^n$ is $\tau$-semistable with $[E\ot
L^n]=(0,0,m\be,d)$ and $0\ne J\subset E\ot L^n$, this implies that
$1\le d/m$. But $d=1,2,\ldots,m$ by choice of $n$, so this
forces~$d=m$.

In case (b), there is a unique $l\ge 1$ such that $I^{l+1}\subset K$
and $I^l\not\subset K$. Then $K+I^l/K$ is a nontrivial subsheaf of
$\cO_X/L\cong J$, and so $K+I^l/K\subset E\ot L^n$. But
$I^l/I^{l+1}$ is the $l^{\rm th}$ symmetric power of the conormal
bundle of $i(\CP^1)$ in $X$, so that $I^l/I^{l+1}\cong
i_*\bigl((l+1)\cO_{\CP^1}(l)\bigr)$. As $I^{l+1}\subset K$ there is
a surjection $I^l/I^{l+1}\ra K+I^l/K$. Let
$[K+I^l/K]=(0,0,m'\be,k')$. Since
$\bigl[i_*\bigl((l+1)\cO_{\CP^1}(l)\bigr)\bigr]=\bigl(0,0,
(l+1)\be,(l+1)^2\bigr)$ and $K+I^l/K$ is a quotient sheaf of
$i_*\bigl((l+1) \cO_{\CP^1}(l)\bigr)$ which is $\tau$-semistable, we
deduce that $l+1\le k'/m'$. But $K+I^l/K\subset E\ot L^n$, so $E\ot
L^n$ $\tau$-semistable implies $k'/m'\le d/m$. Hence $l+1\le
k'/m'\le d/m$, a contradiction as $l\ge 1$ and~$d\le m$.

Thus, case (b) does not happen, and in case (a) we must have $d=m$,
and $E\ot L^n$ has a subsheaf $J\cong i_*(\cO_{\CP^1}(0))$. As
$\tau(J)=\tau(E\ot L^n)=t+1/(\be\cup\de)$, the quotient $(E\ot
L^n)/J$ is also $\tau$-semistable with $[(E\ot
L^n)/J]=(0,0,(m-1)\be,m-1)$. By induction on $m$ we now see that
$E\ot L^n$ has a filtration $0=F_0\subset F_1\subset\cdots\subset
F_m=E\ot L^n$ with $F_i/F_{i-1}\cong i_*(\cO_{\CP^1}(0))$ for
$i=1,\ldots,m$. As the normal bundle is $\cO_{\CP^1}(-1)\op
\cO_{\CP^1}(-1)$ the curve $i(\CP^1)$ is rigid in $X$, which implies
that $\Ext^1\bigl(i_*(\cO_{\CP^1}(0)), i_*(\cO_{\CP^1}(0))\bigr)=0$.
It follows by induction on $m$ that $E\ot L^n\cong
i_*(m\cO_{\CP^1}(0))$, and also that $E\ot L^n$ is rigid. Tensoring
by $L^{-n}$ shows that $E\cong i_*\bigl(m\cO_{\CP^1}(-n)\bigr)$ and
$E$ is rigid. This completes the proof.
\end{proof}

Combining Proposition \ref{dt6prop2} with Examples \ref{dt6ex1} and
\ref{dt6ex2}, and taking $E$ in Example \ref{dt6ex1} to be
$i_*\bigl(\cO_{\CP^1}(k)\bigr)$ for $k\in\Z$, we deduce:

\begin{prop} Let\/ $X$ be a Calabi--Yau $3$-fold over\/ $\C,$ and\/
$(\tau,T,\le)$ a weak stability condition on $\coh(X)$ of Gieseker
or $\mu$-stability type. Suppose $i:\CP^1\ra X$ is an embedding,
and\/ $i(\CP^1)$ has normal bundle $\cO_{\CP^1}(-1)\op
\cO_{\CP^1}(-1)$. Let\/ $\be\in H^4(X;\Z)$ be Poincar\'e dual to\/
$[i(\CP^1)]\in H_2(X;\Z)$.

If\/ $m\ge 1$ and\/ $m\mid k$ then $\tau$-semistable sheaves
supported on $i(\CP^1)$ contribute $1/m^2$ to
$\bar{DT}{}^{(0,0,m\be,k)}(\tau),$ and contribute $1$ if\/ $m=1$
and\/ $0$ if\/ $m>1$ to $\hat{DT}{}^{(0,0,m\be,k)}(\tau)$. If\/
$m\ge 1$ and\/ $m\nmid k$ there are no $\tau$-semistable sheaves in
class $(0,0,m\be,k)$ supported on $i(\CP^1),$ so no contribution to
$\bar{DT},\hat{DT}{}^{(0,0,m\be,k)}(\tau)$.
\label{dt6prop3}
\end{prop}

For higher genus curves the contributions are zero. Note that we do
not need $i(\Si)$ to be rigid, the contributions are local via
weighted Euler characteristics.

\begin{prop} Let\/ $X$ be a Calabi--Yau $3$-fold over\/ $\C,$ and\/
$(\tau,T,\le)$ a weak stability condition on $\coh(X)$ of Gieseker
or $\mu$-stability type. Suppose $\Si$ is a connected, nonsingular
Riemann surface of genus $g\ge 1$ and\/ $i:\Si\ra X$ is an
embedding. Let\/ $\be\in H^4(X;\Z)$ be Poincar\'e dual to\/
$[i(\Si)]\!\in\! H_2(X;\Z)$. Then $\tau$-semistable dimension $1$
sheaves supported on $i(\Si)$ contribute $0$ to both\/
$\bar{DT}{}^{(0,0,m\be,k)}(\tau)$ and\/
$\hat{DT}{}^{(0,0,m\be,k)}(\tau)$ for all $m\ge 1$ and\/~$k\in\Z$.
\label{dt6prop4}
\end{prop}

\begin{proof} The family of line bundles $L_t\ra\Si$ with $c_1(L)=0$
form a group $T^{2g}$ under $\ot$. As $i$ is an embedding, locally
near $i(\Si)$ in $X$ we can find a family of line bundles $\ti L_t$
for $t\in T^{2g}$ which form a group under $\ot$, with $i^*(\ti
L_t)=L_t$. Write $\fM_\rss^{(0,0,m\be,k)}(\tau)_{i(\Si)}$ for the
substack of $\fM_\rss^{(0,0,m\be,k)}(\tau)$ supported on $i(\Si)$.
Then $t:E\mapsto E\ot\ti L_t$ defines an action of $T^{2g}$ on
$\fM_\rss^{(0,0,m\be,k)}(\tau)_{i(\Si)}$. For $m\ge 1$, the
stabilizer groups of this action are finite. So
$\fM_\rss^{(0,0,m\be,k)}(\tau)_{i(\Si)}(\C)$ is fibred by orbits of
$T^{2g}$ isomorphic to $T^{2g}/K\cong T^{2g}$ for $K$ finite.

As the $T^{2g}$-action extends to an open neighbourhood of
$\fM_\rss^{(0,0,m\be,k)}(\tau)_{i(\Si)}$ in
$\fM_\rss^{(0,0,m\be,k)}(\tau)$, the restriction of the Behrend
function of $\fM_\rss^{(0,0,m\be,k)} (\tau)$ to
$\fM_\rss^{(0,0,m\be,k)}(\tau)_{i(\Si)}$ is $T^{2g}$-invariant. Now
the contribution to $\bar{DT}{}^{(0,0,m\be,k)}(\tau)$ from sheaves
supported on $i(\Si)$ is the Euler characteristic of
$\fM_\rss^{(0,0,m\be,k)}(\tau)_{i(\Si)}$ weighted by a constructible
function\index{constructible function} built from
$\bep^{(0,0,m\be,k)}(\tau)$ and the Behrend function
$\nu_{\fM_\rss^{(0,0,m\be,k)}(\tau)}$, as in \S\ref{dt53}. This
constructible function is $T^{2g}$-invariant, as
$\bep^{(0,0,m\be,k)}(\tau),\nu_{\smash{\fM_\rss^{(0,0,m\be,k)}(\tau)}}$
are. But $\chi(T^{2g})=0$ as $g\ge 1$, so each $T^{2g}$-orbit
$T^{2g}/K\cong T^{2g}$ contributes zero to the weighted Euler
characteristic. Thus sheaves supported on $i(\Si)$ contribute 0 to
$\bar{DT}{}^{(0,0,m\be,k)}(\tau)$ for all $m,k$, and so contribute 0
to $\hat{DT}{}^{(0,0,m\be,k)}(\tau)$ by~\eq{dt6eq15}.
\end{proof}

Let $X$ be a Calabi--Yau 3-fold over $\C$, and for $\ga\in
H_2(X;\Z)$ write $GW_0(\ga)\in\Q$ for the genus zero Gromov--Witten
invariants of $X$. Then the {\it genus zero Gopakumar--Vafa
invariants\/} $GV_0(\ga)$ may be defined by the
formula\nomenclature[GVg(a)]{$GV_g(\al)$}{Gopakumar--Vafa
invariants}\nomenclature[GWg(a)]{$GW_{g,m}(\al)$}{Gromov--Witten
invariants}\index{Gopakumar--Vafa invariants}\index{Gromov--Witten
invariants}
\begin{equation*}
GW_0(\ga)=\sum_{m\mid\ga}\frac{1}{m^3}\,GV_0(\ga/m).
\end{equation*}
A priori we have $GV_0(\ga)\in\Q$, but Gopakumar and Vafa
\cite{GoVa} conjecture that the $GV_0(\ga)$ are integers, and count
something meaningful in String Theory.\index{String Theory}

Katz \cite{Katz} considers the moduli spaces
$\M_\rss^{(0,0,\be,1)}(\tau)$ when $k=1$, where $\be\in H^4(X;\Z)$
is Poincar\'e dual to $\ga$. Then
$\M_\rss^{(0,0,\be,1)}(\tau)=\M_\st^{(0,0,\be,1)}(\tau)$ as
$(\be,1)$ is primitive, so $\bar{DT}{}^{(0,0,\be,1)}(\tau)=
\hat{DT}{}^{(0,0,\be,1)}(\tau)=DT^{(0,0,\be,1)}(\tau)$ by
Propositions \ref{dt5prop2} and \ref{dt6prop1}. Katz
\cite[Conj.~2.3]{Katz} conjectures that
$GV_0(\ga)=DT^{(0,0,\be,1)}(\tau)$; this had also earlier been
conjectured by Hosono, Saito and Takahashi \cite[Conj.~3.2]{HST}. We
can now extend their conjecture to all~$k\in\Z$.

\begin{conj} Let\/ $X$ be a Calabi--Yau $3$-fold over\/ $\C,$ and\/
$(\tau,T,\le)$ a weak stability condition on $\coh(X)$ of Gieseker
or $\mu$-stability type. Then for $\ga\in H_2(X;\Z)$ with\/ $\be\in
H^4(X;\Z)$ Poincar\'e dual to $\ga$ and all\/ $k\in\Z$ we have
$\hat{DT}{}^{(0,0,\be,k)}(\tau)=GV_0(\ga)$. In particular,
$\hat{DT}{}^{(0,0,\be,k)}(\tau)$ is independent of\/~$k,\tau$.
\label{dt6conj3}
\end{conj}

For evidence for this, see \cite{Katz} for the case $k=1$, and note
also that Theorem \ref{dt6thm2}(a) shows $\hat{DT}{}^{(0,0,\be,k)}
(\tau)$ is independent of $\tau$, Theorem \ref{dt6thm2}(b) suggests
$\hat{DT}{}^{(0,0,\be,k)}(\tau)\in\Z$, and Propositions
\ref{dt6prop3} and \ref{dt6prop4} show that the contributions to
$\hat{DT}{}^{(0,0,\be,k)}(\tau)$ from rigid rational curves and
embedded higher genus curves are as expected, and independent of
$k$. Also Theorem \ref{dt6thm2}(c) implies that
$\hat{DT}{}^{(0,0,\be,k)}(\tau)$ is periodic in $k$, which supports
the idea that it is independent of $k$. The first author would like
to thank Sheldon Katz and Davesh Maulik for conversations about
Conjecture~\ref{dt6conj3}.

\begin{rem} There are other ways to count curves using
Donaldson--Thomas theory than counting dimension 1 sheaves. The
(ordinary) Donaldson--Thomas invariants $DT^{(1,0,\be,k)}(\tau)$ for
$\be\in H^4(X;\Z)$ and $k\in\Z$ `count' ideal sheaves of subschemes
$S$ of $X$ with $\dim S\le 1$, and the celebrated {\it MNOP
Conjecture\/}\index{MNOP Conjecture} \cite{MNOP1,MNOP2} expresses
$DT^{(1,0,\be,k)}(\tau)$ in terms of the Gromov--Witten invariants
$GW_g(\ga)$ of $X$ for all genera $g\ge 0$, or equivalently in terms
of the Gopakumar--Vafa invariants $GV_g(\ga)$ of $X$ for all $g\ge
0$. {\it Pandharipande--Thomas
invariants\/}\index{Pandharipande--Thomas invariants}
$PT_{n,\be}$\nomenclature[PTn\beta]{$PT_{n,\be}$}{Pandharipande--Thomas
invariants} in \cite{PaTh} count pairs $s:\cO_X\ra E$ for $E$ a pure
dimension 1 sheaf, like our $PI^{\al,n}(\tau)$ but with a different
stability condition, and these also have conjectural equivalences
\cite[\S 3]{PaTh} with $DT^{(1,0,\be,k)}(\tau)$, $GW_g(\ga)$
and~$GV_g(\ga)$.

We will not discuss these further in this book. However, we note
that the results of this book should lead to advances in the theory
of these curve counting invariants, and the relations between them.
In particular, our wall-crossing formula Theorem \ref{dt5thm6}
should be used to prove the correspondence between Donaldson--Thomas
invariants $DT^{(1,0,\be,k)}(\tau)$ and Pandharipande--Thomas
invariants $PT_{n,\be}$. Recent papers by Toda \cite{Toda} and
Stoppa and Thomas \cite {StTh} prove a version of this for
invariants without Behrend functions as weights, and using our
methods to include Behrend functions should yield the proof.
Bridgeland \cite{Brid2} also proves the correspondence assuming
conjectures in~\cite{KoSo1}.
\label{dt6rem1}
\end{rem}

\subsection{Why it all has to be so complicated: an example}
\label{dt65}\index{Donaldson--Thomas invariants!integrality properties}

Our definitions of $\bar{DT}{}^\al(\tau)$ and $\hat{DT}{}^\al(\tau)$
are very complicated. They count sheaves using two kinds of weights:
firstly, we define $\bep^\al(\tau)$ from the $\bdss^\be(\tau)$ by
\eq{dt3eq4}, with rational weights $(-1)^{n-1}/n$, and then we apply
the Lie algebra morphism $\ti\Psi$ of \S\ref{dt53}, which takes
Euler characteristics weighted by the $\Z$-valued Behrend function
$\nu_\fM$. Some readers may have wondered whether all this
complexity is really necessary. For instance, following \eq{dt4eq16}
when $\M_\rss^\al(\tau)=\M_\st^\al(\tau)$, we could simply have
defined $DT^\al(\tau)$ for all $\al\in C(\coh(X))$ by
\e
DT^\al(\tau)=\chi\bigl(\M_\st^\al(\tau),\nu_{\M_\st^\al(\tau)}\bigr).
\label{dt6eq23}
\e

We will now show, by carefully studying an example of dimension 1
sheaves supported on two rigid $\CP^1$'s in $X$ which cross under
deformation, that to get invariants unchanged under deformations of
$X$, the extra layer of complexity with the $\bep^\al(\tau)$ and
rational weights really is necessary. Our example will show that
\eq{dt6eq23} is {\it not\/} deformation-invariant when
$\M_\rss^\al(\tau)\ne\M_\st^\al(\tau)$, and the same holds if we
replace $\M_\st^\al(\tau)$ by $\M_\rss^\al(\tau)$ or
$\fM_\rss^\al(\tau)$; also, we will see that to get a
deformation-invariant answer, it can be necessary to count strictly
$\tau$-semistable sheaves with {\it rational, non-integral\/}
weights, so we do need the~$\bep^\al(\tau)$.

For $\ep>0$ write $\De_\ep=\bigl\{t\in\C:\md{t}<\ep \bigr\}$. Let
$X_t$ for $t\in\De_\ep$ be a smooth family of Calabi--Yau 3-folds
over $\C$, equipped with a family of very ample line bundles
$\cO_{X_t}(1)$. Identify $H^*(X_t;\Q)\cong H^*(X_0;\Q)$,
$H_*(X_t;\Z)\cong H_*(X_0;\Z)$ for all $t$. Suppose $i_t:\CP^1\ra
X_t$ and $j_t:\CP^1\ra X_t$ are two smooth families of embeddings
for $t\in\De_\ep$, and $i_t(\CP^1),j_t(\CP^1)$ have normal bundle
$\cO_{\CP^1}(-1)\op \cO_{\CP^1}(-1)$ for all $t$. Suppose that
$i_t(\CP^1)\cap j_t(\CP^1)=\es$ for $t\ne 0$, and that
$i_0(\CP^1),j_0(\CP^1)$ intersect in a single point $x\in X_0$, with
$T_x\bigl(i_0(\CP^1)\bigr)\cap T_x\bigl(j_0(\CP^1)\bigr)=0$
in~$T_xX_0$.

Now $i_0(\CP^1)\cup j_0(\CP^1)$ is a nodal $\CP^1$ in $X_0$, so we
can regard it as the image of a genus 0 stable map $k_0:\Si_0\ra
X_0$ from a prestable curve $\Si_0=\CP^1\cup_{x}\CP^1$, in the sense
of Gromov--Witten theory. As we have prescribed the normal bundles
and intersection of $i_0(\CP^1),j_0(\CP^1)$, we can show that
$k_0:\Si_0\ra X_0$ is a {\it rigid\/} stable map, and so it persists
as a stable map under small deformations of $X_0$. Thus, making
$\ep>0$ smaller if necessary, for $t\in\De_\ep$ there is a
continuous family of genus 0 stable maps $k_t:\Si_t\ra X_t$. Now
$k_t(\Si_t)$ cannot be reducible for small $t\ne 0$, since the
irreducible components would have to be $i_t(\CP^1),j_t(\CP^1)$, but
these do not intersect. So, making $\ep>0$ smaller if necessary, we
can suppose $\Si_t\cong\CP^1$, and $k_t$ is an embedding, and
$k_t(\CP^1)$ has normal bundle $\cO_{\CP^1}(-1)\op \cO_{\CP^1}(-1)$,
for all~$0\ne t\in\De_\ep$.

Let $\be,\ga\in H^4(X_0;\Z)$ be Poincar\'e dual to
$[i_0(\CP^1)],[j_0(\CP^1)]$ in $H_2(X_0;\Z)$. Suppose $\be,\ga$ are
linearly independent over $\Z$. Let $\de=c_1(\cO_{X_0}(1))$ in
$H^2(X_0;\Z)$. Set $c_\be=\be\cup\de$ and $c_\ga=\ga\cup\de$ and
$c_{\be+\ga}=c_\be+c_\ga$, so that $c_\be,c_\ga,c_{\be+\ga}\in\N$.
Write classes $\al\in K^\num(\coh(X_0))$ as
$(\al_0,\al_2,\al_4,\al_6)$ as in \S\ref{dt64}. We will consider
$\tau$-semistable sheaves $E$ on $X_t$ in classes $(0,0,\be,k)$,
$(0,0,\ga,l)$ and $(0,0,\be+\ga,m)$ for $k,l,m\in\Z$ and
$t\in\De_\ep$. Suppose for simplicity that all such sheaves are
supported on $i_t(\CP^1)\cup j_t(\CP^1)\cup k_t(\Si_t)$;
alternatively, we can consider the following as computing the
contributions to $\bar{DT}{}^{(0,0,\be,k)}(\tau)_t,\ldots,
\bar{DT}{}^{(0,0,\be+\ga,m)}(\tau)_t$ from sheaves supported
on~$i_t(\CP^1)\cup j_t(\CP^1)\cup k_t(\Si_t)$.

Here is a way to model all this explicitly in a family of compact
Calabi--Yau 3-folds. Let $\CP^2\times\CP^2$ have homogeneous
coordinates $\bigl([x_0,x_1,x_2],[y_0,y_1,y_2]\bigr)$, write $\bs
x=(x_0,x_1,x_2)$, $\bs y=(y_0,y_1,y_2)$, and let $X_t$ be the
bicubic $F_t(\bs x,\bs y)=0$ in $\CP^2\times\CP^2$, with very ample
line bundle $\cO_{X_t}(1)=\cO_{\CP^1\times\CP^1}(1,1)\vert_{X_t}$,
where
\begin{align*}
F_t(\bs x,\bs y)=x_0^2&x_1y_0^2y_1+x_0^3y_0^2y_2+ x_0^2x_2y_0^3
-tx_0^3y_0^3\\
&+x_1x_2P_{1,3}(\bs x,\bs y)+x_2y_2P'_{2,2}(\bs x,\bs
y)+y_1y_2P''_{3,1}(\bs x,\bs y),
\end{align*}
with $P_{1,3},P_{2,2}',P''_{3,1}$ homogeneous polynomials of the
given bidegrees.

Define $i_t,j_t,k_t:\CP^1\ra X_t$ by
$i_t:[u,v]\mapsto\bigl([u,v,0],[1,0,t]\bigr)$ and
$j_t:[u,v]\mapsto\bigl([1,0,t],[u,v,0]\bigr)$ for all $t$, and
$k_t:[u,v]\mapsto\bigl([u,v,0],[v,tu,0]\bigr)$ for $t\ne 0$. Then
the conditions above hold for $P_{1,3},P_{2,2}',P''_{3,1}$ generic
and $\ep>0$ small. Consider first the moduli spaces
$\fM_\rss^{(0,0,\be,k)}(\tau)_t, \fM_\rss^{(0,0,\ga,l)}(\tau)_t$
over $X_t$. These are $\tau$-semistable sheaves supported on
$i_t(\CP^1),j_t(\CP^1)$, so by Proposition \ref{dt6prop2} we see
that the only $\tau$-semistable sheaves in classes $(0,0,\be,k)$ and
$(0,0,\ga,l)$ are $E_t(k)=(i_t)_*(\cO_{\CP^1}(k-1))$ and
$F_t(l)=(j_t)_*(\cO_{\CP^1}(l-1))$ respectively, and both are
$\tau$-stable and rigid.

The Hilbert polynomials of $E_t(k)$ and $F_t(l)$ are
$P_{(0,0,\be,k)}(t)=c_\be\,t+k$ and $P_{(0,0,\ga,l)}(t)=c_\ga\,t+l$
by \eq{dt6eq22}, so we have
\e
\tau\bigl(\bigl[E_t(k)\bigr]\bigr)=t+k/c_\be,\;\>
\tau\bigl(\bigl[F_t(l)\bigr]\bigr)=t+l/c_\ga.
\label{dt6eq24}
\e
Therefore the sheaf $E_t(k)\op F_t(l)$ in class $(0,0,\be+\ga,k+l)$
is $\tau$-semistable if and only if~$k\,c_\ga=l\,c_\be$.

We can now describe $\fM_\rss^{(0,0,\be+\ga,m)}(\tau)_t$ for $t\ne
0$. For all $m\in\Z$, we have a rigid $\tau$-stable sheaf
$G_t(m)=(k_t)_*(\cO_{\CP^1}(m-1))$ in class $(0,0,\be+\ga,m)$
supported on $k_t(\Si_t)$, which contributes $[\Spec\C/\bG_m]$ to
$\fM_\rss^{(0,0,\be+\ga,m)}(\tau)_t$. In addition, if there exist
$k,l\in\Z$ with $k+l=m$ and $k\,c_\ga=l\,c_\be$, then $E_t(k)\op
F_t(l)$ is a rigid, strictly $\tau$-semistable sheaf in class
$(0,0,\be+\ga,m)$ supported on $i_t(\CP^1)\amalg j_t(\CP^1)$, which
contributes $[\Spec\C/\bG_m^2]$ to
$\fM_\rss^{(0,0,\be+\ga,m)}(\tau)_t$. We have
$k=m\,c_\be/c_{\be+\ga}$, $l=m\,c_\ga/c_{\be+\ga}$, which lie in
$\Z$ if and only if $c_{\be+\ga}\mid m\,c_\be$. These are all the
$\tau$-semistable sheaves in class $(0,0,\be+\ga,m)$. Thus we see
that
\begin{gather}
\begin{gathered}
\text{$t\ne 0$, $c_{\be+\ga}\nmid m\,c_\be$ imply}\;\>
\fM_\rss^{(0,0,\be+\ga,m)}(\tau)_t\cong [\Spec\C/\bG_m]\\
\text{and}\quad
\M_\rss^{(0,0,\be+\ga,m)}(\tau)_t=\M_\st^{(0,0,\be+\ga,m)}(\tau)_t
\cong\Spec\C,
\end{gathered}
\label{dt6eq25}\\
\begin{gathered}
\text{$t\!\ne\!0$, $c_{\be+\ga}\mid m\,c_\be$ imply}\;
\fM_\rss^{(0,0,\be+\ga,m)}(\tau)_t\!\cong\![\Spec\C/\bG_m]\!\amalg\!
[\Spec\C/\bG_m^2],\\
 \M_\rss^{(0,0,\be+\ga,m)}(\tau)_t\cong
\Spec\C\amalg\Spec\C, \;\>\text{and}\;\>
\M_\st^{(0,0,\be+\ga,m)}(\tau)_t\cong\Spec\C.
\end{gathered}
\label{dt6eq26}
\end{gather}

Now consider $\fM_\rss^{(0,0,\be+\ga,m)}(\tau)_0$ when $t=0$.
Writing $\cO_x$ for the structure sheaf of intersection point of
$i_0(\CP^1)$ and $j_0(\CP^1)$, we have exact sequences
\e
\begin{gathered}
\xymatrix@C=20pt@R=5pt{0 \ar[r] & E_0(k) \ar[r] & E_0(k+1)
\ar[r]^(0.6){\pi_x} & \cO_x \ar[r] &
0,\\
0 \ar[r] & F_0(l) \ar[r] & F_0(l+1) \ar[r]^(0.6){\pi_x} & \cO_x
\ar[r] & 0. }
\end{gathered}
\label{dt6eq27}
\e
Define $G_0(k,l)$ to be the kernel of the morphism in $\coh(X_0)$
\begin{equation*}
\pi_x\op\pi_x: E_0(k+1)\op F_0(l+1)\longra \cO_x.
\end{equation*}
Since $\bigl[E_0(k+1)\bigr]=(0,0,\be, k+1)$,
$\bigl[F_0(l+1)\bigr]=(0,0,\be,l+1)$ and $[\cO_x]=(0,0,0,1)$ and
each $\pi_x$ is surjective we have~$[G_0(k,l)]=(0,0,\be+\ga,k+l+1)$.

From \eq{dt6eq27} we see that we have non-split exact sequences
\e
\begin{gathered}
\xymatrix@C=20pt@R=3pt{0 \ar[r] & E_0(k) \ar[r] &
G_0(k,l) \ar[r] & F_0(l+1) \ar[r] & 0,\\
0 \ar[r] & F_0(l) \ar[r] & G_0(k,l) \ar[r] & E_0(k+1) \ar[r] & 0.}
\end{gathered}
\label{dt6eq28}
\e
By \eq{dt6eq24}, the first sequence of \eq{dt6eq28} destabilizes
$G_0(k,l)$ if $k/c_\be>(l+1)/c_\ga$, and the second sequence
destabilizes $G_0(k,l)$ if $l/c_\ga>(k+1)/c_\be$. The sequences
\eq{dt6eq28} are sufficient to test the $\tau$-(semi)stability of
$G_0(k,l)$. It follows that $G_0(k,l)$ is $\tau$-semistable if
$k/c_\be\le (l+1)/c_\ga$ and $l/c_\ga\le (k+1)/c_\be$, and
$G_0(k,l)$ is $\tau$-stable if $k/c_\be<(l+1)/c_\ga$
and~$l/c_\ga<(k+1)/c_\be$.

Now fix $m\in\Z$. It is easy to show from these inequalities that if
$c_{\be+\ga}\nmid m\,c_\be$ there is exactly one choice of
$k,l\in\Z$ with $k+l+1=m$ and $G_0(k,l)$ $\tau$-semistable in class
$(0,0,\be+\ga,m)$, and in fact this $G_0(k,l)$ is $\tau$-stable. And
if $c_{\be+\ga}\mid m\,c_\be$ then setting $k=m\,c_\be/c_{\be+\ga}$,
$l=m\,c_\ga/ c_{\be+\ga}$ in $\Z$ we find that $G_0(k-1,l)$ and
$G_0(k,l-1)$ are both strictly $\tau$-semistable in class
$(0,0,\be+\ga,m)$, and in addition $E_0(k)\op F_0(l)$ is strictly
$\tau$-semistable in class $(0,0,\be+\ga,m)$. These are all the
$\tau$-semistables in class $(0,0,\be+\ga,m)$. So
\begin{gather}
\begin{gathered}
\text{$c_{\be+\ga}\nmid m\,c_\be$ implies}\;\>
\fM_\rss^{(0,0,\be+\ga,m)}(\tau)_0\cong [\Spec\C/\bG_m]\\
\text{and}\quad
\M_\rss^{(0,0,\be+\ga,m)}(\tau)_0=\M_\st^{(0,0,\be+\ga,m)}(\tau)_0
\cong\Spec\C,
\end{gathered}
\label{dt6eq29}\\
\begin{gathered}
\text{$c_{\be+\ga}\mid m\,c_\be$ implies}\;\>
\M_\st^{(0,0,\be+\ga,m)}(\tau)_0=\es\;\>\text{and}\\
\fM_\rss^{(0,0,\be+\ga,m)}(\tau)_0(\C)=\bigl\{G_0(k-1,l),G_0(k,l-1),
E_0(k)\op F_0(l)\bigr\}.
\end{gathered}
\label{dt6eq30}
\end{gather}

Next we describe the stack structure on
$\fM_\rss^{(0,0,\be+\ga,m)}(\tau)_0$ when $c_{\be+\ga}\mid
m\,c_\be$. As $E_0(k),F_0(l)$ are rigid, we have
\e
\begin{split}
\Ext^1\bigl(E_0(k)&\op F_0(l),E_0(k)\op F_0(l)\bigr) \\
&=\Ext^1\bigl(E_0(k),F_0(l)\bigr)\op\Ext^1\bigl(
F_0(l),E_0(k)\bigr)\cong\C\op\C,
\end{split}
\label{dt6eq31}
\e
where a nonzero element of $\Ext^1\bigl(E_0(k),F_0(l)\bigr)$
corresponds to $G_0(k-1,l)$, and a nonzero element of
$\Ext^1\bigl(F_0(l),E_0(k)\bigr)$ corresponds to $G_0(k,l-1)$, by
\eq{dt6eq28}. Let $(y,z)$ be coordinates on $\C\op\C$ in
\eq{dt6eq31}. Then $\Aut\bigl(E_0(k)\op F_0(l)\bigr)\cong\bG_m^2$
acts on $\C\op\C$ by $(\la,\mu):(y,z)\mapsto
(\la\mu^{-1}y,\la^{-1}\mu z)$.

By Theorem \ref{dt5thm3}, $\fM_\rss^{(0,0,\be+\ga,m)}(\tau)_0$ is
locally isomorphic as an Artin stack near $E_0(k)\op F_0(l)$ to
$[\Crit(f)/\bG_m^2]$, where $U\subseteq \C\op\C$ is a
$\bG_m^2$-invariant analytic open neighbourhood of 0, and $f:U\ra\C$
is a $\bG_m^2$-invariant holomorphic function. Since $f$ is
$\bG_m^2$-invariant, it must be a function of $yz$. Now $(y,0)$ for
$y\ne 0$ in $\C\op \C$ represents $G_0(k-1,l)$, which is rigid; also
$(0,z)$ for $z\ne 0$ represents $G_0(k,l-1)$, which is rigid.
Therefore $\bigl\{(y,0):0\ne y\in\C\bigr\}$ and $\bigl\{(0,z):0\ne
z\in\C\bigr\}$ must be smooth open sets in $\Crit(f)$, so that $f$
is nondegenerate quadratic normal to them, to leading order.

It follows that we may take $U=\C\op\C$ and $f(y,z)=y^2z^2$, giving
\e
\fM_\rss^{(0,0,\be+\ga,m)}(\tau)_0\cong
\bigl[\Crit(y^2z^2)/\bG_m^2\bigr].
\label{dt6eq32}
\e
There are three $\bG_m^2$ orbits in $\Crit(y^2z^2)$:
$\bigl\{(y,0):0\ne y\in\C\bigr\}$ corresponding to $G_0(k-1,l)$, and
$\bigl\{(0,z):0\ne z\in\C\bigr\}$ corresponding to $G_0(k,l-1)$, and
$(0,0)$ corresponding to $E_0(k)\op F_0(l)$. The Milnor
fibres\index{Milnor fibre!examples} of $y^2z^2$ at $(1,0)$ and $(0,1)$
are both two discs, with Euler characteristic 2. Since $y^2z^2$ is
homogeneous, the Milnor fibre of $y^2z^2$ at $(0,0)$ is
diffeomorphic to $\bigl\{(y,z)\in\C^2:y^2z^2=1\bigr\}$, which is the
disjoint union of two copies of $\C\sm\{0\}$, with Euler
characteristic zero. By results in \S\ref{dt4} we deduce that
\e
\begin{gathered}
\nu_{\fM_\rss^{(0,0,\be+\ga,m)}(\tau)_0}\bigl(G_0(k-1,l)\bigr)=
\nu_{\fM_\rss^{(0,0,\be+\ga,m)}(\tau)_0}\bigl(G_0(k,l-1)\bigr)=-1\\
\text{and}\qquad
\nu_{\fM_\rss^{(0,0,\be+\ga,m)}(\tau)_0}\bigl(E_0(k)\op
F_0(l)\bigr)=1.
\end{gathered}
\label{dt6eq33}
\e
From \eq{dt6eq32} we find that the coarse moduli space
$\M_\rss^{(0,0,\be+\ga,m)}(\tau)_0$ is
\e
\M_\rss^{(0,0,\be+\ga,m)}(\tau)_0\cong\Spec\C.
\label{dt6eq34}
\e

As in equation \eq{dt6eq11} of Example \ref{dt6ex8} we find that
\begin{align*}
\bar\Pi^{\chi,\Q}_\fM\bigl(\bep^{(0,0,\be+\ga,m)}(\tau)_0\bigr)&=
\ts\frac{1}{2}\bigl[([\Spec\C/\bG_m],\rho_{G_0(k-1,l)})\bigr]\\
&+\ts\frac{1}{2}\bigl[([\Spec\C/\bG_m],\rho_{G_0(k,l-1)})\bigr],
\end{align*}
where $\rho_{G_0(k-1,l)},\rho_{G_0(k,l-1)}$ map $[\Spec\C/\bG_m]$ to
$G_0(k-1,l),G_0(k,l-1)$. Note that $\bep^{(0,0,\be+\ga,m)}(\tau)_0$
is zero over $E_0(k)\op F_0(l)$. As for \eq{dt6eq12} we have
\e
\begin{split}
\bar{DT}{}^{(0,0,\be+\ga,m)}(\tau)_0&=\ts \frac{1}{2}
\nu_{\fM_\rss^{(0,0,\be+\ga,m)}(\tau)_0}\bigl(G_0(k-1,l)\bigr)\\
&+\ts\frac{1}{2}\nu_{\fM_\rss^{(0,0,\be+\ga,m)}(\tau)_0}\bigl(
G_0(k,l-1)\bigr)=\frac{1}{2}\cdot 1+\frac{1}{2}\cdot 1=1.
\end{split}
\label{dt6eq35}
\e
Thus $G_0(k-1,l)$ and $G_0(k,l-1)$ each contribute $\ha$
to~$\bar{DT}{}^{(0,0,\be+\ga,m)}(\tau)_0$.

From equations
\eq{dt6eq25},\eq{dt6eq26},\eq{dt6eq29},\eq{dt6eq30},\eq{dt6eq33},\eq{dt6eq34}
and \eq{dt6eq35} we deduce:
\ea
\bar{DT}{}^{(0,0,\be+\ga,m)}(\tau)_t&=1,\quad\text{all $t\in\De_\ep$
and $m\in\Z$,}
\nonumber\\
\chi\bigl(\M_\st^{(0,0,\be+\ga,m)}(\tau)_t,
\nu_{\M_\st^{(0,0,\be+\ga,m)}(\tau)_t}\bigr)&=
\begin{cases}
1, & \text{$t\ne 0$ or $c_{\be+\ga}\nmid m\,c_\be$,} \\
0, & \text{$t=0$ and $c_{\be+\ga}\mid m\,c_\be$,}
\end{cases}
\label{dt6eq36}\\
\chi\bigl(\M_\rss^{(0,0,\be+\ga,m)}(\tau)_t,
\nu_{\M_\rss^{(0,0,\be+\ga,m)}(\tau)_t}\bigr)&=
\begin{cases}
1, & \text{$t\ne 0$ or $c_{\be+\ga}\nmid m\,c_\be$,} \\
2, & \text{$t=0$ and $c_{\be+\ga}\mid m\,c_\be$,}
\end{cases}
\label{dt6eq37}\\
\chi^\na\bigl(\fM_\rss^{(0,0,\be+\ga,m)}(\tau)_t,
\nu_{\fM_\rss^{(0,0,\be+\ga,m)}(\tau)_t}\bigr)&=
\begin{cases}
-1,\! & \text{$t=0$ or $c_{\be+\ga}\nmid m\,c_\be$,} \\
\phantom{-}0,\! & \text{$t\!\ne\! 0$ and $c_{\be+\ga}\!\mid\!
m\,c_\be$.}
\end{cases}
\label{dt6eq38}
\ea
Equations \eq{dt6eq36}--\eq{dt6eq38} imply:

\begin{cor} Let\/ $X$ be a Calabi--Yau $3$-fold over\/ $\C$ and\/
$\al\!\in\! K^\num(\coh(X))$ with\/ $\M_\rss^\al(\tau)\ne
\M_\st^\al(\tau)$. Then none of\/ $\chi\bigl(\M_\st^\al(\tau),
\nu_{\M_\st^\al(\tau)}\bigr),$ $\chi\bigl(\M_\rss^\al(\tau),
\nu_{\M_\rss^\al(\tau)}\bigr)$ or\/ $\chi^\na\bigl(\fM_\rss^\al
(\tau),\nu_{\fM_\rss^\al(\tau)}\bigr)$ need be unchanged under
deformations of\/~$X$.
\label{dt6cor2}
\end{cor}

We can also use these calculations to justify the necessity of
rational weights in the $\bep^\al(\tau)$ in our definition of
$\bar{DT}{}^\al(\tau)$. Let $c_{\be+\ga}\mid m\,c_\be$. Then when
$t\ne 0$, we have one stable, rigid sheaf $G_t(m)$ in class
$(0,0,\be+\ga,m)$, which is counted with weight 1 in
$\bar{DT}{}^{(0,0,\be+\ga,m)}(\tau)_t$. But when $t=0$, $G_t(m)$ is
replaced by two strictly $\tau$-semistable sheaves $G_0(k-1,l)$ and
$G_0(k,l-1)$, which are counted with weight $\ha$ in
$\bar{DT}{}^{(0,0,\be+\ga,m)}(\tau)_0$. By symmetry between
$G_0(k-1,l),G_0(k,l-1)$, to get deformation-invariance it is
necessary that they are each counted with weight $\ha$, which means
that {\it we must allow non-integral weights for strictly\/
$\tau$-semistables in our counting scheme to get a
deformation-invariant answer}.

Also, we cannot tell that $G_0(k-1,l),G_0(k,l-1)$ should have weight
$\ha$ just from the stack
$\smash{\fM_\rss^{(0,0,\be+\ga,m)}(\tau)_0}$, as they are rigid with
stabilizer group $\bG_m$, and look just like $\tau$-stables. The
strict $\tau$-semistability of $G_0(k-1,l),G_0(k,l-1)$ is measured
by the fact that
$\bdss^{(0,0,\ga,l)}(\tau)*\bdss^{(0,0,\be,k)}(\tau),
\bdss^{(0,0,\be,k)}(\tau)*\bdss^{(0,0,\ga,l)}(\tau)$ are nonzero
over $G_0(k-1,l),G_0(k,l-1)$ respectively. But
$\bdss^{(0,0,\ga,l)}(\tau)*\bdss^{(0,0,\be,k)}(\tau)$ and
$\bdss^{(0,0,\be,k)}(\tau)*\bdss^{(0,0,\ga,l)}(\tau)$ occur with
coefficient $-\ha$ in the expression \eq{dt3eq4} for
$\bep^{(0,0,\be+\ga,m)}(\tau)$. This suggests that {\it using\/
$\bep^\al(\tau)$ or something like it is necessary to make\/
$\bar{DT}{}^\al(\tau)$ deformation-invariant}.

\subsection{$\mu$-stability and invariants $\bar{DT}{}^\al(\mu)$}
\label{dt66}\index{$\mu$-stability|(}\index{stability
condition!$\mu$-stability|(}\index{m-stability@$\mu$-stability|(}

So far we have mostly discussed invariants $\bar{DT}{}^\al(\tau)$,
where $(\tau,G,\le)$ is Gieseker stability w.r.t.\ a very ample line
bundle $\cO_X(1)$, as in Example \ref{dt3ex1}. We can also consider
$\bar{DT}{}^\al(\mu)$, where $(\mu,M,\le)$ is $\mu$-{\it
stability\/} w.r.t.\ $\cO_X(1)$, as in Example \ref{dt3ex2}. We now
prove some simple but nontrivial facts about
the~$\bar{DT}{}^\al(\mu)$.

First note that as $(\mu,M,\le)$ is a truncation of $(\tau,G,\le)$,
we have $\tau(\be)\le\tau(\ga)$ implies $\mu(\be)\le\mu(\ga)$ for
$\be,\ga\in C(\coh(X))$, and so $(\mu,M,\le)$ {\it
dominates\/}\index{stability condition!$(\tilde\tau,\tilde T,\leqslant)$
dominates $(\tau,T,\leqslant)$} $(\tau,G,\le)$ in the sense of
Definition \ref{dt3def6}. In Theorem \ref{dt3thm2} we can use
$(\hat\tau,\hat T,\le)=(\mu,M,\le)$ as the dominating weak stability
condition to write $\bep^\al(\mu)$ in terms of the $\bep^\be(\tau)$
and vice versa, and then Theorem \ref{dt5thm6} writes
$\bar{DT}{}^\al(\mu)$ in terms of the $\bar{DT}{}^\be(\tau)$ and
vice versa. Since the Gieseker stability invariants
$\bar{DT}{}^\be(\tau)$ are deformation-invariant by Corollary
\ref{dt5cor4}, we deduce:

\begin{cor} The $\mu$-stability invariants $\bar{DT}{}^\al(\mu)$
are unchanged under continuous deformations of the underlying
Calabi--Yau $3$-fold\/~$X$.
\label{dt6cor3}
\end{cor}

The next well-known lemma says that for torsion-free sheaves,
$\mu$-stability is unchanged by tensoring by a line bundle. It holds
on any smooth projective scheme $X$, not just Calabi--Yau 3-folds,
and works because $\mu([E\ot L])-\mu([E])$ is independent of $E$
when $\dim E=\dim X$. The corresponding results are not true for
Gieseker stability, nor for $\mu$-stability for non-torsion-free
sheaves.

\begin{lem} Let\/ $E$ be a nonzero torsion-free sheaf on $X,$ and\/
$L$ a line bundle. Then $E\ot L$ is $\mu$-semistable if and only
if\/ $E$ is $\mu$-semistable.
\label{dt6lem}
\end{lem}

Now $E\mapsto E\ot L$ induces an automorphism of the abelian
category $\coh(X)$, which acts on $K^\num(\coh(X))\subset H^{\rm
even}(X;\Q)$ by $\al\mapsto \al\exp(\ga)$. For torsion-free sheaves,
this automorphism takes $\mu$-semistables to $\mu$-semistables, and
so maps $\bdss^\al(\mu)$ to $\bdss^{\al\exp(\ga)}(\mu)$ and
$\bep^\al(\mu)$ to $\bep^{\al\exp(\ga)}(\mu)$ for $\rank(\al)>0$.
Applying $\ti\Psi$ as in \S\ref{dt53}, we see that
$\bar{DT}{}^{\al\exp(\ga)}(\mu)=\bar{DT}{}^\al(\mu)$. Since we
assume $H^1(\cO_X)=0$, every $\ga\in H^2(X;\Z)$ is $c_1(L)$ for some
line bundle $L$. Thus we deduce:

\begin{thm} Let\/ $X$ be a Calabi--Yau $3$-fold over\/ $\C$ and\/
$(\mu,M,\le)$ be $\mu$-stability with respect to a very ample line
bundle $\cO_X(1)$ on $X,$ as in Example {\rm\ref{dt3ex2}}. Write
elements $\al$ of\/ $K^\num(\coh(X))\subset H^{\rm even}(X;\Q)$ as
$(\al_0,\al_2,\al_4,\al_6),$ as in\/ {\rm\S\ref{dt64}}. Then for
all\/ $\al\in C(\coh(X))$ with\/ $\al_0>0$ and all\/ $\ga\in
H^2(X;\Z)$ we have\/~$\bar{DT}{}^{\al\exp(\ga)}(\mu)
=\bar{DT}{}^\al(\mu)$.
\label{dt6thm3}
\end{thm}

Theorem \ref{dt6thm3} encodes a big symmetry group of generalized
Donaldson--Thomas invariants $\bar{DT}{}^\al(\mu)$ in positive rank,
which would be much more complicated to write down for Gieseker
stability.\index{$\mu$-stability|)}\index{stability
condition!$\mu$-stability|)}\index{m-stability@$\mu$-stability|)}

\subsection{Extension to noncompact Calabi--Yau 3-folds}
\label{dt67}\index{Calabi--Yau 3-fold!noncompact|(}

So far we have considered only {\it compact\/} Calabi--Yau 3-folds,
and indeed our convention is that Calabi--Yau 3-folds are by
definition compact, unless we explicitly say that they are
noncompact. Suppose $X$ is a {\it noncompact Calabi--Yau\/
$3$-fold\/} over $\C$, by which we mean a smooth quasiprojective
3-fold over $\C$, with trivial canonical bundle $K_X$. (We will
impose further conditions on $X$ shortly.) Then the abelian category
$\coh(X)$ of coherent sheaves on $X$ is badly behaved, from our
point of view -- for instance, groups $\Hom(E,F)$ for
$E,F\in\coh(X)$ may be infinite-dimensional, so the Euler
form\index{Euler form} $\bar\chi$ on $\coh(X)$ may not be defined.

However, the abelian category
$\coh_\cs(X)$\nomenclature[cohcs(X)]{$\coh_\cs(X)$}{abelian category
of compactly supported coherent sheaves on $X$} of {\it
compactly-supported\/}\index{coherent sheaf!compactly-supported|(}
coherent shea\-ves on $X$ is well-behaved: $\Ext^i(E,F)$ is
finite-dimensional for $E,F\in\coh_\cs(X)$ and satisfies Serre
duality\index{Serre duality} $\Ext^i(F,E)\cong\Ext^{3-i}(E,F)^*$, so
$\coh_\cs(X)$ has a well-defined Euler form. If $X$ has no compact
connected components then $\coh_\cs(X)$ consists of torsion sheaves,
supported in dimension 0,1 or 2.

We propose that a good generalization of Donaldson--Thomas theory to
noncompact Calabi--Yau 3-folds is to define invariants counting
sheaves in $\coh_\cs(X)$. Note that this is {\it not\/} the route
that has been taken by other authors such as Szendr\H oi \cite[\S
2.8]{Szen}, who instead consider invariants counting {\it ideal
sheaves\/}\index{coherent sheaf!ideal sheaf}\index{ideal sheaf} $I$
of compact subschemes of $X$. Such $I$ are not compactly-supported,
but are isomorphic to $\cO_X$ outside a compact subset of~$X$.

Going through the theory of \S\ref{dt4}--\S\ref{dt5}, we find that
the assumption that $X$ is compact (proper, or projective) is used
in three important ways:
\begin{itemize}
\setlength{\itemsep}{0pt}
\setlength{\parsep}{0pt}
\item[(a)] The \kern -.2em Euler \kern -.2em form \kern -.2em
$\bar\chi$ \kern -.2em on \kern -.2em $K_0(\coh(X))$ \kern -.2em
is \kern -.2em undefined \kern -.2em for \kern -.2em noncompact
\kern -.2em $X$ \kern -.2em as \kern -.2em $\Ext^i(E,F)$ may be
infinite-dimensional, so $K^\num(\coh(X))$ is undefined. For
$\coh_\cs(X)$, $K^\num(\coh_\cs(X))$ is well defined, but may
not be isomorphic to the image of $\ch:K_0(\coh_\cs(X))\ra
H^{\rm even}_\cs(X;\Q)$. Hilbert polynomials $P_E$ of
$E\in\coh_\cs(X)$ need not factor through the class $[E]$ in
$K^\num(\coh_\cs(X))$ for $X$ noncompact.
\item[(b)] For noncompact $X$, Theorem \ref{dt5thm1} in
\S\ref{dt51} fails because nonzero vector bundles on $X$ are not
compactly-supported. But Theorems \ref{dt5thm2} and \ref{dt5thm3}
depend on Theorem \ref{dt5thm1}, and Theorem \ref{dt5thm4} in
\S\ref{dt52} depends on Theorem \ref{dt5thm3}, and most of the rest
of \S\ref{dt5}--\S\ref{dt6} depends on Theorem \ref{dt5thm4}.
\item[(c)] For noncompact $X$, the moduli schemes $\M_\rss^\al(\tau)$
and $\M_\stp^{\al,n}(\tau')$ of \S\ref{dt43} and \S\ref{dt54}
need not be proper. This means that the virtual cycle
definitions of $DT^\al(\tau)$ in \eq{dt4eq15} when
$\M_\rss^\al(\tau)= \M_\st^\al(\tau)$, and of
$PI^{\al,n}(\tau')$ in \eq{dt5eq15}, are not valid. The weighted
Euler characteristic expressions \eq{dt4eq16}, \eq{dt5eq16} for
$DT^\al(\tau)$ and $PI^{\al,n}(\tau')$ still make sense. But the
proofs that $DT^\al(\tau),\bar{DT}{}^\al(\tau),
PI^{\al,n}(\tau')$ are unchanged by
deformations\index{Donaldson--Thomas
invariants!deformation-invariance} of $X$ no longer work, as
they are based on the virtual cycle\index{virtual cycle}
definitions \eq{dt4eq15},\eq{dt5eq15}.
\end{itemize}

Here is how we deal with these issues. For (a), with $X$ noncompact,
note that although $\coh(X)$ may not have a well-defined Euler form,
there is an Euler pairing $\bar\chi:K_0(\coh(X))\times
K_0(\coh_\cs(X))\ra\Z$. Under the Chern character\index{Chern
character} maps $\ch: K_0(\coh(X))\ra H^{\rm even}(X;\Q)$ and
$\ch_\cs: K_0(\coh_\cs(X))\ra H^{\rm
even}_\cs(X;\Q)$,\nomenclature[chcs(E)]{$\ch_\cs(E)$}{Chern
character of a compactly-supported coherent sheaf $E$} this
$\bar\chi$ is mapped to the pairing $H^{\rm even}(X;\Q)\times H^{\rm
even}_\cs(X;\Q)\ra\Q$\nomenclature[Hevf]{$H^{\rm
even}_\cs(X;\Q)$}{compactly-supported even cohomology of $X$} given
by $(\al,\be)\mapsto\deg\bigl(\al^\vee\cdot\be\cdot{\rm
td}(TX)\bigr){}_3$, which is nondegenerate by the invertibility
of~${\rm td}(TX)=1+\frac{1}{12}\,c_2(X)$ and \index{Poincar\'e
duality}Poincar\'e duality~$H^{2k}(X;\Q)\cong H^{6-2k}_\cs(X;\Q)^*$.

In Assumption \ref{dt3ass}, with $\A=\coh_\cs(X)$, we should take
$K(\coh_\cs(X))$ to be the quotient of $K_0(\coh_\cs(X))$ by the
kernel in $K_0(\coh_\cs(X))$ of the Euler pairing
$\bar\chi:K_0(\coh(X))\times K_0(\coh_\cs(X))\ra\Z$. This is {\it
not\/} the same as the numerical Grothendieck
group\index{Grothendieck group!numerical} $K^\num(\coh_\cs(X))$,
which is the quotient of $K_0(\coh_\cs(X))$ by the kernel of
$\bar\chi:K_0(\coh_\cs(X))\times K_0(\coh_\cs(X))\ra\Z$; in general
$K^\num(\coh_\cs(X))$ is a quotient of $K(\coh_\cs(X))$, but they
may not be equal. In Example \ref{dt6ex10} below we will have
$K(\coh_\cs(X))\cong\Z^2$ but $K^\num(\coh_\cs(X))=0$. As the
pairing above is nondegenerate, this $K(\coh_\cs(X))$ is naturally
identified with the image of the {\it compactly-supported Chern
character\/} $\ch_\cs:K_0(\coh_\cs(X))\ra H^{\rm even}_\cs(X;\Q)$.
This can be defined by combining the `localized Chern character'
$\ch_Z^X$ of Fulton \cite[\S 18.1 \& p.~368]{Fult} with the
compactly supported Chow groups\index{Chow homology} and homology in
\cite[Ex.s 10.2.8 \& 19.1.12]{Fult}, using the limiting process over
all compact subschemes $Z\subset X$ in~\cite[Ex.~10.2.8]{Fult}.

Then if $E\in\coh(X)$ and $F\in\coh_\cs(X)$, the Euler form
$\bar\chi(E,F)$ depends only on $E$ and $[F]$ in $K(\coh_\cs(X))$.
In particular, given a very ample line bundle $\cO_X(1)$ on $X$, the
Hilbert polynomial\index{Hilbert polynomial}
$P_F(n)=\bar\chi\bigl(\cO_X(-n),F\bigr)$ of $F$ depends only on the
class $[F]$ in $K(\coh_\cs(X))$. Since $\cO_X(-n)$ is not
compactly-supported, in general $P_F$ does not depend only on $[F]$
in $K^\num(\coh_\cs(X))$, as in Example~\ref{dt6ex10}.

This is important for two reasons. Firstly, equation \eq{dt5eq17} in
\S\ref{dt54} involves $\bar\chi\bigl([\cO_X(-n)],\al\bigr)$ for
$\al\in K(\coh(X))$, and if $K(\coh_\cs(X))=K^\num(\coh_\cs(X))$
then $\bar\chi\bigl([\cO_X(-n)],\al\bigr)$ would not be well-defined
for $\al\in K(\coh_\cs(X))$, and Theorem \ref{dt5thm10} would fail.
Secondly, the proof that moduli spaces $\M_\st^\al(\tau),\ab
\M_\rss^\al(\tau),\ab\fM_\st^\al (\tau),\ab\fM_\rss^\al(\tau)$ of
$\tau$-(semi)stable sheaves $E$ in class $\al$ in $K(\coh(X))$ are
of {\it finite type\/} depends on the fact that $\al$ determines the
Hilbert polynomial of $E$. If we took
$K(\coh_\cs(X))=K^\num(\coh_\cs(X))$, this would not be true, the
moduli spaces might not be of finite type, and then weighted Euler
characteristic expressions such as \eq{dt4eq16}, \eq{dt5eq16} would
not make sense.

For (b), we will show in Theorem \ref{dt6thm4} that under extra
assumptions on $X$ we can deduce Theorems \ref{dt5thm2} and
\ref{dt5thm3} for noncompact $X$ from the compact case. This is
enough to generalize \S\ref{dt52}--\S\ref{dt53} and parts of
\S\ref{dt54} to the noncompact case. For (c), we should accept that
for noncompact $X$, the invariants $\bar{DT}{}^\al(\tau),
\hat{DT}{}^\al(\tau),PI^{\al,n}(\tau')$ may not be
deformation-invariant,\index{Donaldson--Thomas
invariants!deformation-invariance} as they can change when the
(compact) support of a sheaf `goes to infinity' in $X$ as we deform
the complex structure of $X$. Here is an example.

\begin{ex} For $t\in\C$, define
\e
Y_t=\bigl\{(z_1,z_2,z_3,z_4)\in\C^4:(1-t^2z_1)^2+z_2^2+z_3^2
+z_4^2=0\bigr\}.
\label{dt6eq39}
\e
Then $Y_0$ is nonsingular, and $Y_t$ for $t\ne 0$ has one singular
point at $(t^{-2},0,0,0)$, an ordinary double point, which has two
small resolutions. Because \eq{dt6eq39} involves $t^2$ rather than
$t$ we may choose one of these small resolutions continuously in $t$
for all $t\ne 0$ to get a smooth family of noncompact Calabi--Yau
3-folds $X_t$ for $t\in\C$, which is also smooth over~$t=0$.

As a complex 3-fold, $X_t$ for $t\ne 0$ is the total space of
$\cO(-1)\op\cO(-1)\ra\CP^1$, and $X_0$ is $\C\times Q$, where $Q$ is
a smooth quadric in $\C^3$. All the $X_t$ are diffeomorphic to
$\R^2\times T^*{\cal S}^2$, and we can regard them as a smooth
family of complex structures $J_t$ for $t\in\C$ on the fixed
6-manifold $\R^2\times T^*{\cal S}^2$. Then $X_t$ for $t\ne 0$
contains a curve $\Si_t\cong\CP^1$, the fibre over $(t^{-2},0,0,0)$
in $Y_t$. As $t\ra 0$, this $\Si_t$ goes to infinity in $X_t$, and
$X_0$ contains no compact curves, as it is affine.

From Example \ref{dt6ex10} below it follows that for $t\ne 0$ there
are nonzero invariants $\bar{DT}{}^\al(\tau)_t$ with $\dim\al=1$
counting dimension 1 sheaves supported on $\Si_t$ in $X_t$. But when
$t=0$ there are no compactly-supported dimension 1 sheaves on $X_0$,
as there are no curves on which they could be supported, so
$\bar{DT}{}^\al(\tau)_0=0$, and $\bar{DT}{}^\al(\tau)$ is not
deformation invariant.
\label{dt6ex9}
\end{ex}

Here is the extra condition we need to extend Theorems
\ref{dt5thm2}--\ref{dt5thm3} to~$\coh_\cs(X)$.

\begin{dfn} Let $X$ be a noncompact Calabi--Yau 3-fold over $\C$. We
call $X$ {\it compactly embeddable\/}\index{compactly embeddable|(}
if whenever $K\subset X$ is a compact subset, in the analytic
topology, there exists an open neighbourhood $U$ of $K$ in $X$ in
the analytic topology, a compact Calabi--Yau 3-fold $Y$ over $\C$
with $H^1(\cO_Y)=0$, an open subset $V$ of $Y$ in the analytic
topology, and an isomorphism of complex manifolds~$\phi:U\ra V$.
\label{dt6def2}
\end{dfn}

\begin{thm} Let\/ $X$ be a noncompact Calabi--Yau\/ $3$-fold over\/
$\C,$ and suppose $X$ is compactly embeddable. Then Theorems
{\rm\ref{dt5thm2}} and\/ {\rm\ref{dt5thm3}} hold in\/~$\coh_\cs(X)$.
\label{dt6thm4}
\end{thm}

\begin{proof} Write $\fM^X$ for the moduli stack of
compactly-supported coherent shea\-ves on $X$ and $\M_\rsi^X$ for
the complex algebraic space of simple compactly-supported coherent
sheaves\index{coherent sheaf!complex analytic} on $X$. For each
compactly-supported (algebraic) coherent sheaf $E$ on $X$ there is
an underlying compactly-supported complex analytic coherent sheaf
$E_\an$, and by Serre \cite{Serr} this map $E\mapsto E_\an$ is an
equivalence of categories.

Let $E\in\coh_\cs(X)$, so that $[E]\in\fM^X(\C)$, or
$[E]\in\M_\rsi^X(\C)$ if $E$ is simple. Then $K=\mathop{\rm supp}E$
is a compact subset of $X$. Let $U,Y,V$ be as in Definition
\ref{dt6def2} for this $K$, and write $\fM^Y$ for the moduli stack
of coherent sheaves on $Y$, and $\M_\rsi^Y$ for the complex
algebraic space of simple coherent sheaves on $X$. Then
$E_\an\vert_U$ is a complex analytic coherent sheaf on $U\subset X$,
so $\phi_*(E_\an)$ is a complex analytic coherent sheaf on $V\subset
Y$, which we extend by zero to get a complex analytic coherent sheaf
$F_\an$ on $Y$, and this is associated to a unique (algebraic)
coherent sheaf $F$ on $Y$ by \cite{Serr}, with $[F]\in\fM^Y(\C)$,
and $[F]\in\M_\rsi^Y(\C)$ if $F$ (or equivalently $E$) is simple.

For Theorem \ref{dt5thm2}, let $E$ be simple, and write $W^X$ for
the subset of $[E']\in\M_\rsi^X(\C)$ with $E'$ supported on $U$, and
$W^Y$ for the subset of $[F']\in\M_\rsi^Y(\C)$ with $F'$ supported
on $V$. Then $W^X,W^Y$ are open neighbourhoods of $[E],[F]$ in
$\M_\rsi^X(\C),\M_\rsi^Y(\C)$ in the complex analytic topology, and
there is a unique map $\phi_*:W^X\ra W^Y$ with $\phi_*([E'])=[F']$
if $\phi_*(E'_\an)\cong F'_\an$. Since $\phi$ is an isomorphism of
complex manifolds, it is easy to see that $\phi_*$ is an isomorphism
of complex analytic spaces.

By Theorem \ref{dt5thm2}, $W^Y$ near $[F]$ is locally isomorphic to
$\Crit(f)$ as a complex analytic space, for $f:U\ra\C$ holomorphic
and $U\subset\Ext^1(F,F)$ open. Since $W^X\cong W^Y$ as complex
analytic spaces and $\Ext^1(E,E)\cong\Ext^1(E_\an,E_\an)
\cong\Ext^1(F_\an,F_\an)\cong\Ext^1(F,F)$ by \cite{Serr}, Theorem
\ref{dt5thm2} for $\coh_\cs(X)$ follows.

For Theorem \ref{dt5thm3}, let $S^X,\Phi^X$ be as in the second
paragraph of Theorem \ref{dt5thm3} for $\fM^X,E$ in $\coh_\cs(X)$,
and $S^Y,\Phi^Y$ for $\fM^Y,F$ on $\coh(Y)$. Then as in Proposition
\ref{dt9prop6}(b) in the sheaf case there are formally versal
families\index{versal family} $(S^X,{\cal D}^X)$ of compactly-supported
coherent sheaves on $X$ with ${\cal D}_0^X\cong E$, and $(S^Y,{\cal
D}^Y)$ of coherent sheaves on $Y$ with ${\cal D}_0^Y\cong F$. The
corresponding families $(S^X(\C),{\cal D}_\an^X)$, $(S^Y(\C),{\cal
D}_\an^Y)$ of complex analytic coherent sheaves are versal. Let
$W^X,W^Y$ be the subsets of $S^X(\C),S^Y(\C)$ representing sheaves
supported on $U,V$. Then $W^X,W^Y$ are open neighbourhoods of 0 in
$S^X(\C),S^Y(\C)$, in the analytic topology.

Since $\phi:U\ra V$ is an isomorphism of complex manifolds, $\phi_*$
takes versal families of complex analytic sheaves on $U$ to versal
families of complex analytic sheaves on $V$. Therefore
$\bigl(W^X,\phi_*({\cal D}_\an^X\vert_{W^X})\bigr)$ and $(W^Y,{\cal
D}_\an^Y\vert_{W^Y})$ are both versal families of complex analytic
coherent sheaves on $V$ with $\phi_*({\cal D}_\an^X\vert_{W^X})_0
\cong F_\an\cong({\cal D}_\an^Y\vert_{W^Y})_0$. We can now argue as
in Proposition \ref{dt9prop7} using the fact that $T_0W^X\cong
\Ext^1(E_\an,E_\an)\cong \Ext^1(F_\an,F_\an)\cong T_0W^Y$ that $W_X$
near 0 is isomorphic as a complex analytic space to $W^Y$ near 0.
Theorem \ref{dt5thm3} for $X$ then follows from Theorem
\ref{dt5thm3} for~$Y$.
\end{proof}

\begin{quest} Let\/ $X$ be a noncompact Calabi--Yau\/ $3$-fold
over\/ $\C$. Can you prove Theorems {\rm\ref{dt5thm2}} and\/
{\rm\ref{dt5thm3}} hold in\/ $\coh_\cs(X)$ without assuming\/ $X$ is
compactly embeddable?
\label{dt6quest2}
\end{quest}

All of \S\ref{dt52}--\S\ref{dt53} now extends immediately to
$\coh_\cs(X)$ for $X$ a compactly embeddable noncompact Calabi--Yau
3-fold: the Behrend function identities\index{Behrend
function!identities} \eq{dt5eq2}--\eq{dt5eq3}, the Lie algebra
morphisms $\ti\Psi,\ti\Psi{}^{\chi,\Q}$, the definition of
generalized Donaldson--Thomas invariants $\bar{DT}{}^\al(\tau)$ for
$\al\in K(\coh_\cs(X))$, and the transformation law \eq{dt5eq14}
under change of stability condition.

In \S\ref{dt54} the definition of stable pairs still works, and the
moduli scheme $\M_\stp^{\al,n}(\tau')$ is well-defined, but may not
be {\it proper}. So \eq{dt5eq15} does not make sense, and we take
the weighted Euler characteristic \eq{dt5eq16} to be the {\it
definition\/} of the pair invariants $PI^{\al,n}(\tau')$. The
deformation-invariance\index{Donaldson--Thomas
invariants!deformation-invariance} of
$\bar{DT}{}^\al(\tau),PI^{\al,n}(\tau')$ in Corollaries
\ref{dt5cor3} and \ref{dt5cor4} will not hold for $\coh_\cs(X)$ in
general, as Example \ref{dt6ex9} shows. But Theorem \ref{dt5thm10},
expressing the $PI^{\al,n}(\tau')$ in terms of the
$\bar{DT}{}^\be(\tau)$, is still valid, with proof essentially
unchanged; it does not matter that $\cO_X(-n)$ lies in $\coh(X)$
rather than $\coh_\cs(X)$.

As in \S\ref{dt62} we define BPS invariants
$\hat{DT}{}^\al(\tau)$\index{Donaldson--Thomas invariants!BPS invariants
$\hat{DT}{}^\al(\tau)$}\index{BPS invariants} for $\coh_\cs(X)$ from the
$\bar{DT}{}^\al(\tau)$, and conjecture they are integers for generic
$(\tau,T,\le)$. The results of \S\ref{dt63} computing invariants
counting dimension zero sheaves also hold in the noncompact case, as
the proof of Theorem \ref{dt6thm1} in \cite{BeFa2} does not need $X$
compact.

\begin{ex} Let $X$ be the noncompact Calabi--Yau 3-fold
$\cO(-1)\op\cO(-1)\ra\CP^1$, that is, the total space of the rank 2
vector bundle $\cO(-1)\op\cO(-1)$ over $\CP^1$. This is a very
familiar example from the Mathematics and String Theory\index{String
Theory} literature; it is a crepant resolution of the {\it
conifold\/}\index{conifold} $z_1^2+z_2^2+z_3^2+z_4^2=0$ in $\C^4$,
so it is often known as the {\it resolved
conifold}.\index{conifold!resolved}

Let $Y$ be any compact Calabi--Yau 3-fold containing a rational
curve $C\cong\CP^1$ with normal bundle $\cO(-1)\op\cO(-1)$; explicit
examples such as quintics are easy to find. Then $Y$ near $C$ is
isomorphic as a complex manifold to $X$ near the zero section. Since
any compact subset $K$ in $X$ can be mapped into any open
neighbourhood of the zero section in $X$ by a sufficiently small
dilation, it follows that $X$ is {\it compactly embeddable}, and our
theory applies for~$\coh_\cs(X)$.

We have $H^{2j}_\cs(X;\Q)=\Q$ for $j=2,3$ and $H^{2j}_\cs(X;\Q)=0$
otherwise. For $E\in\coh_\cs(X)$ we have $\ch_\cs(E)=\bigl(0,0,
\ch_2(E),\ch_3(E)\bigr)$, where $\ch_j(E)\in\Z\subset\Q=
H^{2j}_\cs(X;\Q)$ for $j=2,3$. Thus we can identify $K(\coh_\cs(X))$
with $\Z^2$ with coordinates $(a_2,a_3)$, where $[E]=(a_2,a_3)$ if
$\ch_j(E)=a_j$ for $j=2,3$. The class of a point sheaf $\cO_x$ for
$x\in X$ is $(0,1)$, and if $i:\CP^1\ra X$ is the zero section, the
class of $i_*(\cO_{\CP^1}(k))$ is $(1,1+k)$. The positive cone
$C(\coh_\cs(X))$ is
\e
C(\coh_\cs(X))=\bigl\{(a_2,a_3)\in\Z^2:\text{$a_2=0$ and $a_3>0$, or
$a_2>0$}\bigr\}.
\label{dt6eq40}
\e
The Euler form $\bar\chi$ on $\coh_\cs(X)$ is zero,
so~$K^\num(\coh_\cs(X))=0$.

Let $(\tau,G,\le)$ be Gieseker stability on $X$ with respect to the
ample line bundle $\pi^*(\cO_{\CP^1}(1))$. We can write down the
full Donaldson--Thomas and BPS invariants $\bar{DT}{}^\al(\tau),
\hat{DT}{}^\al(\tau)$ using the work of \S\ref{dt63}--\S\ref{dt64}.
We have
\ea
\bar{DT}{}^{(a_2,a_3)}(\tau)&=\begin{cases} \displaystyle
-2\sum_{m\ge 1, \; m \mid a_3}\frac{1}{m^2}, &
a_2=0,\; a_3\ge 1, \\
\displaystyle \frac{1}{a_2^2}, & a_2>0, \; a_2 \mid a_3, \\
0, & \text{otherwise,}
\end{cases}
\label{dt6eq41}\\
\hat{DT}{}^{(a_2,a_3)}(\tau)&=\begin{cases} -2, &
a_2=0,\; a_3\ge 1, \\
\phantom{-}1, & a_2=1,\\
\phantom{-}0, & \text{otherwise.}
\end{cases}
\label{dt6eq42}
\ea
Here the first lines of \eq{dt6eq41}--\eq{dt6eq42} count dimension 0
sheaves and are taken from \eq{dt6eq19}--\eq{dt6eq20}, noting that
$\chi(X)=2$, and the rest which count dimension 1 sheaves follow
from Proposition \ref{dt6prop3}. We will return to this example
in~\S\ref{dt752}.
\label{dt6ex10}
\end{ex}

It is easy to show that other important examples of noncompact
Calabi--Yau 3-folds such as $K_{\CP^2}$ and $K_{\CP^1\times\CP^1}$
are also compactly embeddable.\index{Calabi--Yau
3-fold!noncompact|)}\index{coherent
sheaf!compactly-supported|)}\index{compactly embeddable|)}

\subsection[Configuration operations and extended D--T
invariants]{Configuration operations and extended Donaldson--Thomas
invariants}
\label{dt68}\index{configurations|(}\index{Ringel--Hall
algebra!extra operations|(}

Let $X$ be a Calabi--Yau 3-fold over $\C$, and $(\tau,T,\le)$ a weak
stability condition on $\coh(X)$ of Gieseker or $\mu$-stability
type. In \S\ref{dt32} we explained how to construct elements
$\bdss^\al(\tau)$ in the algebra $\SFa(\fM)$ and $\bep^\al(\tau)$ in
the Lie algebra $\SFai(\fM)$ for $\al\in C(\coh(X))$. Then in
\S\ref{dt53} we defined a Lie algebra morphism
$\ti\Psi:\SFai(\fM)\ra\ti L(X)$, and applied $\ti\Psi$ to
$\bep^\al(\tau)$ to define~$\bar{DT}{}^\al(\tau)$.

Now the theory of \cite{Joyc3,Joyc4,Joyc5,Joyc6} is even more
complicated than was explained in \S\ref{dt3}. As well as the
Ringel--Hall product $*$ on $\SFa(\fM)$ and the Lie bracket
$[\,,\,]$ on $\SFai(\fM)$, in \cite[Def.~5.3]{Joyc4} using the idea
of `configurations' we define an infinite family of multilinear
operations
$P_{(I,\preceq)}$\nomenclature[PI\leqslant]{$P_{(I,\preceq)}$}{multilinear
operation on $\SFa(\fM)$ depending on a poset $(I,\preceq)$} on
$\SFa(\fM)$ depending on a finite partially ordered set
(poset)\index{poset} $(I,\preceq)$, with $*=P_{(\{1,2\},\le)}$. It
follows from \cite[Th.~5.17]{Joyc4} that certain linear combinations
of the $P_{(I,\preceq)}$ are multilinear operations on $\SFai(\fM)$,
with $[\,,\,]$ being the simplest of these.

Also, in \cite[\S 8]{Joyc5}, given $(\tau,T,\le)$ on $\coh(X)$ we
construct much larger families of interesting elements of
$\SFa(\fM)$ and $\SFai(\fM)$ than just the $\bdss^\al(\tau)$ and
$\bep^\al(\tau)$. In \cite[Def.~8.9]{Joyc5} we define a Lie
subalgebra $\bar{\cal L}^{\rm pa}_\tau$\nomenclature[Lpat]{$\bar{\cal L}^{\rm
pa}_\tau,\bar{\cal L}^{\rm to}_\tau$}{Lie subalgebras of
Ringel--Hall Lie algebra $\SFai(\fM)$} of $\SFai(\fM)$ which is
spanned by certain elements $\bs\si_*(I)\bar\de^{\rm b}_{\rm
si}(I,\preceq,\ka,\tau)$ of $\SFai(\fM)$, where $(I,\preceq)$ is a
finite, connected poset and $\ka:I\ra C(\coh(X))$ is a map. The
$\bep^\al(\tau)$ lie in $\bar{\cal L}^{\rm pa}_\tau$, and may be
written as finite $\Q$-linear combinations of $\bs\si_*(I)
\bar\de^{\rm b}_{\rm si}(I,\preceq,\ka,\tau)$, but the
$\bep^\al(\tau)$ do not generate $\bar{\cal L}^{\rm pa}_\tau$ as a
Lie algebra, they only generate a smaller Lie algebra $\bar{\cal
L}^{\rm to}_\tau$. The $\bep^\al(\tau)$ do generate $\bar{\cal
L}^{\rm pa}_\tau$ over the infinite family of multilinear operations
on $\SFai(\fM)$ defined from the $P_{(I,\preceq)}$. In \cite[\S
6.5]{Joyc6} we apply the Lie algebra morphism $\Psi:\SFai(\fM)\ra
L(X)$ of \S\ref{dt34} to the $\bs\si_*(I)\bar\de^{\rm b}_{\rm
si}(I,\preceq,\ka,\tau)$ to define invariants $J^{\rm b}_{\rm
si}(I,\preceq,\ka,\tau)\in\Q$, and prove they satisfy a
transformation law under change of stability condition. So replacing
$\Psi$ by $\ti\Psi$ we define:

\begin{dfn} In the situation above, define {\it extended
Donaldson--Thomas invariants\/}\index{Donaldson--Thomas
invariants!extended}\nomenclature[Jsi]{$\ti J{}^{\rm b}_{\rm
si}(I,\preceq,\ka,\tau)$}{extended Donaldson--Thomas invariants}
$\ti J{}^{\rm b}_{\rm si}(I,\preceq,\ka,\tau)\in\Q$, where
$(I,\preceq)$ is a finite, connected poset and $\ka:I\ra C(\coh(X))$
is a map, by
\e
\ti\Psi\bigl(\bs\si_*(I) \bar\de^{\rm b}_{\rm
si}(I,\preceq,\ka,\tau)\bigr)=\ti J{}^{\rm b}_{\rm
si}(I,\preceq,\ka,\tau)\,\ti\la^{\ka(I)},
\label{dt6eq43}
\e
where $\bs\si_*(I)\bar\de^{\rm b}_{\rm si}(I,\preceq,\ka,\tau)\in
\SFai(\fM)$ is as in~\cite[Def.~8.1]{Joyc5}.
\label{dt6def3}
\end{dfn}

Here are some good properties of the $\ti J{}^{\rm b}_{\rm
si}(I,\preceq,\ka,\tau)$:
\begin{itemize}
\setlength{\itemsep}{0pt}
\setlength{\parsep}{0pt}
\item $\bep^\al(\tau)$ may be written as a $\Q$-linear
combination of the $\bs\si_*(I)\bar\de^{\rm b}_{\rm
si}(I,\preceq, \ka,\tau)$. Thus comparing \eq{dt5eq7} and
\eq{dt6eq43} shows that $\bar{DT}{}^\al(\tau)$ is a $\Q$-linear
combination of the $\ti J{}^{\rm b}_{\rm
si}(I,\preceq,\ka,\tau)$.
\item $\bar{\cal L}^{\rm pa}_\tau$ is a Lie algebra spanned by the
$\bs\si_*(I)\bar\de^{\rm b}_{\rm si}(I,\preceq,\ka,\tau)$, and the
Lie bracket of two generators $\bs\si_*(I)\bar\de^{\rm b}_{\rm
si}(I,\preceq,\ka,\tau)$ may be written as an explicit $\Q$-linear
combination of other generators. So since $\ti\Psi$ is a Lie algebra
morphism, we can deduce many {\it multiplicative relations\/}
between the $\ti J{}^{\rm b}_{\rm si}(I,\preceq,\ka,\tau)$.
\item As for the $J^{\rm b}_{\rm si}(I,\preceq,\ka,\tau)$ in
\cite{Joyc6}, there is a known {\it wall-crossing
formula\/}\index{wall-crossing formula} for the $\ti J{}^{\rm
b}_{\rm si}(I,\preceq,\ka,\tau)$ under change of stability
condition.
\end{itemize}

Since the $\bar{DT}{}^\al(\tau)$ are deformation-invariant by
Corollary \ref{dt5cor4}, we can ask whether the $\ti J{}^{\rm
b}_{\rm si}(I,\preceq,\ka,\tau)$ are
deformation-invariant.\index{Donaldson--Thomas
invariants!deformation-invariance|(} Also, we can ask whether the
multilinear operations on $\SFai(\fM)$ above are taken by $\ti\Psi$
to multilinear operations on $\ti L(X)$ by $\ti\Psi$. The answer to
both is no, as we show by an example.

\begin{ex} Define a 1-morphism $\phi:\fM\times\fM\ra\fM$ by
$\phi(E,F)=E\op F$ on objects. In a similar way to the Ringel--Hall
product $*$ in \S\ref{dt31}, define a bilinear operation $\bu$ on
$\SFa(\fM)$ by $f\bu g=\phi_*(f\ot g)$. Then $\bu$ is commutative
and associative; in the notation of \cite[Def.~5.3]{Joyc4} we have
$\bu=P_{(\{1,2\},\triangleleft)}$, where $i\triangleleft j$ if
$i=j$. Define a bilinear operation $\diamond$ on $\SFa(\fM)$ by
$f\diamond g=f*g-f\bu g$. Then \cite[Th.~5.17]{Joyc4} implies that
if $f,g\in\SFai(\fM)$ then $f\diamond g\in\SFai(\fM)$, so $\diamond$
restricts to a bilinear operation on $\SFai(\fM)$. We have
$[f,g]=f\diamond g-g\diamond f$, since $\bu$ is commutative. The Lie
algebra $\bar{\cal L}^{\rm pa}_\tau$ above is closed
under~$\diamond$.

Now let us work in the situation of \S\ref{dt65}. Consider the
elements $\bep^{(0,0,\be,k)}(\tau)_t$, $\bep^{(0,0,\ga,l)}(\tau)_t$
and $\bep^{(0,0,\be,k)}(\tau)_t\diamond\bep^{(0,0,\ga,l)}(\tau)_t$
in $\SFai(\fM)_t$, for $t\in\De_\ep$, and their images under
$\ti\Psi$. We find that $\bep^{(0,0,\be,k)}(\tau)_t=\bde_{E_t(k)}$
and $\bep^{(0,0,\ga,l)}(\tau)_t=\bde_{F_t(l)}$, so
$\bep^{(0,0,\be,k)}(\tau)_t\bu\bep^{(0,0,\ga,l)}(\tau)_t=\bde_{E_t(k)\op
F_t(l)}$. But
$\bep^{(0,0,\be,k)}(\tau)_t*\bep^{(0,0,\ga,l)}(\tau)_t$ is
$\bde_{E_t(k)\op F_t(l)}$ when $t\ne 0$, and $\bde_{E_0(k)\op
F_t(l)}+\bde_{G_0(k,l-1)}$ when $t=0$. Hence
\begin{equation*}
\bep^{(0,0,\be,k)}(\tau)_t\diamond\bep^{(0,0,\ga,l)}(\tau)_t=
\begin{cases} 0, & t\ne 0, \\ \bde_{G_0(k,l-1)}, & t=0. \end{cases}
\end{equation*}
Since each of $E_t(k)$, $F_t(l)$ and $G_0(k,l-1)$ are simple and
rigid, we see that
\e
\begin{gathered}
\ti\Psi\bigl(\bep^{(0,0,\be,k)}(\tau)_t\bigr)=-\ti\la{}^{(0,0,\be,k)},\quad
\ti\Psi\bigl(\bep^{(0,0,\ga,l)}(\tau)_t\bigr)=-\ti\la{}^{(0,0,\ga,l)},\\
\text{and}\qquad \ti\Psi\bigl(\bep^{(0,0,\be,k)}(\tau)_t
\diamond\bep^{(0,0,\ga,l)}(\tau)_t\bigr)=
\begin{cases} 0, & t\ne 0, \\
-\ti\la{}^{(0,0,\be+\ga,k+l)}, & t=0. \end{cases}
\end{gathered}
\label{dt6eq44}
\e
\label{dt6ex11}
\end{ex}

Equation \eq{dt6eq44} tells us three important things. Firstly,
there cannot exist a deformation-invariant bilinear operation
$\raisebox{1.5pt}{\scriptsize$\blacklozenge$}$ on $\ti L(X)$ with
$\ti\Psi(f\diamond
g)=\ti\Psi(f)\kern.1em\raisebox{1.5pt}{\scriptsize
$\blacklozenge$}\kern.1em\ti\Psi(g)$ for all $f,g\in\SFai(\fM)_t$.
Thus, although $\ti\Psi$ is compatible with the Lie bracket
$[\,,\,]$ on $\SFai(\fM)$, it will not be nicely compatible with the
more general multilinear operations on $\SFai(\fM)$ defined using
the~$P_{(I,\preceq)}$.

Secondly, $\bep^{(0,0,\be,k)}(\tau)_t\diamond
\bep^{(0,0,\ga,l)}(\tau)_t$ is an element of $\bar{\cal L}^{\rm
pa}_\tau$, a $\Q$-linear combination of elements
$\bs\si_*(I)\bar\de^{\rm b}_{\rm si}(I,\preceq,\ka,\tau)$, and its
image under $\ti\Psi$ is a $\Q$-linear combination of $\ti J{}^{\rm
b}_{\rm si}(I,\preceq,\ka,\tau)$, multiplied by
$\ti\la{}^{(0,0,\be+\ga,k+l)}$. Equation \eq{dt6eq44} shows that
this $\Q$-linear combination of $\ti J{}^{\rm b}_{\rm
si}(I,\preceq,\ka,\tau)$ is not deformation-invariant, so some at
least of the extended Donaldson--Thomas invariants in Definition
\ref{dt6def3} are not deformation-invariant. Thirdly, the $\ti
J{}^{\rm b}_{\rm si}(I,\preceq,\ka,\tau)$ do in general include
extra information not encoded in the $\bar{DT}{}^\al(\tau)$, as if
they did not they would have to be deformation-invariant.

Here and in \S\ref{dt65} we have considered several ways of defining
invariants by counting sheaves weighted by the Behrend function
$\nu_\fM$, but which turn out not to be deformation-invariant. It
seems to the authors that deformation-invariance arises in
situations where you have proper moduli schemes with obstruction
theories, such as $\M_\stp^{\al,n}(\tau')$ in \S\ref{dt54}, and that
you should not expect deformation-invariance if you cannot find such
proper moduli schemes in the problem.

\begin{quest} Let\/ $X$ be a Calabi--Yau $3$-fold over $\C$. Are
there any $\Q$-linear combinations of extended Donaldson--Thomas
invariants $\ti J{}^{\rm b}_{\rm si}(I,\preceq,\ka,\tau)$ of\/ $X,$
which are unchanged by deformations of\/ $X$ for all\/ $X,$ but
which cannot be written in terms of the\/~$\bar{DT}{}^\al(\tau)?$
\label{dt6quest3}
\end{quest}\index{Calabi--Yau 3-fold|)}\index{Donaldson--Thomas
invariants!computation in
examples|)}\index{configurations|)}\index{Donaldson--Thomas
invariants!deformation-invariance|)}\index{Ringel--Hall algebra!extra
operations|)}

\section{Donaldson--Thomas theory for quivers with superpotentials}
\label{dt7}\index{quiver!with superpotential|(}\index{quiver|(}

The theory of \S\ref{dt5}--\S\ref{dt6} relied on three properties of
the abelian category $\coh(X)$ of coherent sheaves on a compact
Calabi--Yau 3-fold $X$:
\begin{itemize}
\setlength{\itemsep}{0pt}
\setlength{\parsep}{0pt}
\item[(a)] The moduli stack $\fM$ of objects in $\coh(X)$ can
locally be written in terms of $\Crit(f)$ for $f:U\ra\C$ holomorphic
and $U$ smooth, as in Theorem \ref{dt5thm3};
\item[(b)] For all $D,E\in\coh(X)$ we have
\begin{align*}
\bar\chi\bigl([D],[E]\bigr)=\,&\bigl(\dim\Hom(D,E)-\dim\Ext^1(D,E)
\bigr)-\\
&\bigl(\dim\Hom(E,D)-\dim\Ext^1(E,D)\bigr),
\end{align*}
where $\bar\chi:K(\coh(X))\times K(\coh(X))\ra\Z$ is biadditive
and antisymmetric. This is a consequence of Serre duality in
dimension 3, that is, $\Ext^i(D,E)\cong\Ext^{3-i}(E,D)^*$, but
we do not actually need Serre duality to hold; and
\item[(c)] We can form {\it proper\/} moduli schemes
$\M_\st^\al(\tau)$ when $\M_\rss^\al(\tau)=\M_\st^\al(\tau)$,
and $\M_\stp^{\al,n}(\tau')$ in general, which have symmetric
obstruction theories.
\end{itemize}
As in \S\ref{dt67}, for a noncompact Calabi--Yau 3-fold $X$,
properties (a),(b) hold for compactly-supported sheaves
$\coh_\cs(X)$, but the properness in (c) fails. Properness is
essential in proving $\bar{DT}{}^\al(\tau), PI^{\al,n}(\tau')$ are
{\it deformation-invariant\/} in~\S\ref{dt54}.

We will show that properties (a) and (b) also hold for $\C$-linear
abelian categories of representations $\modCQI$ of a quiver $Q$ with
relations $I$ coming from a superpotential $W$. So we can extend
much of \S\ref{dt5}--\S\ref{dt6} to these categories. As property
(c) does not hold, the Donaldson--Thomas type invariants we define
may not be unchanged under deformations of the underlying geometry
or algebra. Much work has already been done in this area, and we
will explain as we go along how our results relate to those in the
literature.

\subsection{Introduction to quivers}
\label{dt71}

Let $\K$ be an algebraically closed field of characteristic
zero.\index{field $\K$!characteristic zero} Here are the basic
definitions in quiver theory. Benson \cite[\S 4.1]{Bens} is a good
reference.

\begin{dfn} A {\it quiver\/}\index{quiver!definition} $Q$ is a
finite directed graph. That is, $Q$ is a quadruple $(Q_0,Q_1,h,t)$,
where $Q_0$ is a finite set of {\it vertices}, $Q_1$ is a finite set
of {\it edges}, and $h,t:Q_1\ra Q_0$ are maps giving the {\it
head\/} and {\it tail\/} of each
edge.\nomenclature[QI]{$(Q,I)$}{quiver with
relations}\nomenclature[KQ]{$\K Q$}{path algebra of a
quiver}\nomenclature[KQI]{$\K Q/I$}{algebra of a quiver with
relations $(Q,I)$}

The {\it path algebra}\index{quiver!path algebra} $\K Q$ is an
associative algebra over $\K$ with basis all {\it paths of length\/}
$k\ge 0$, that is, sequences of the form
\e
v_0\,{\buildrel e_1\over\longra}\, v_1\ra\cdots\ra
v_{k-1}\,{\buildrel e_k\over\longra}\,v_k,
\label{dt7eq1}
\e
where $v_0,\ldots,v_k\in Q_0$, $e_1,\ldots,e_k\in Q_1$,
$t(a_i)=v_{i-1}$ and $h(a_i)=v_i$. Multiplication is given by
composition of paths in reverse order.

For $n\ge 0$, write $\K Q_{(n)}$ for the vector subspace of $\K Q$
with basis all paths of length $k\ge n$. It is an ideal in $\K Q$. A
{\it quiver with relations\/} $(Q,I)$ is defined to be a quiver $Q$
together with a two-sided ideal $I$ in $\K Q$ with $I\subseteq\K
Q_{(2)}$. Then $\K Q/I$ is an associative $\K$-algebra.
\label{dt7def1}
\end{dfn}

We define {\it representations\/} of quivers, and of quivers with
relations.

\begin{dfn} Let $Q=(Q_0,Q_1,h,t)$ be a quiver. A {\it
representation}\index{quiver!representation} of $Q$ consists of
finite-dimensional $\K$-vector spaces $X_v$ for each $v\in Q_0$, and
linear maps $\rho_e:X_{t(e)}\ra X_{h(e)}$ for each $e\in Q_1$.
Representations of $Q$ are in 1-1 correspondence with {\it
finite-dimensional left\/ $\K Q$-modules} $(X,\rho)$, as follows.

Given $X_v,\rho_e$, define $X=\bigop_{v\in Q_0}X_v$, and a linear
$\rho:\K Q\ra\End(X)$ taking \eq{dt7eq1} to the linear map $X\ra X$
acting as $\rho_{e_k}\ci \rho_{e_{k-1}}\ci\cdots\ci\rho_{e_1}$ on
$X_{v_0}$, and 0 on $X_v$ for $v\ne v_0$. Then $(X,\rho)$ is a left
$\K Q$-module. Conversely, any such $(X,\rho)$ comes from a unique
representation of $Q$. If $(Q,I)$ is a quiver with relations, a {\it
representation of\/} $(Q,I)$ is a representation of $Q$ such that
the corresponding left $\K Q$-module $(X,\rho)$ has $\rho(I)=0$.

A {\it morphism of representations} $\phi:(X,\rho)\ra(Y,\si)$ is a
linear map $\phi:X\ra Y$ with $\phi\ci\rho(\ga)=\si(\ga)\ci\phi$ for
all $\ga\in\K Q$. Equivalently, $\phi$ defines linear maps
$\phi_v:X_v\ra Y_v$ for all $v\in Q_0$ with
$\phi_{h(e)}\ci\rho_e=\si_e\ci\phi_{t(e)}$ for all $e\in Q_1$. Write
$\modKQ,\modKQI$ for the categories of representations of $Q$ and
$(Q,I)$. They are $\K$-linear abelian categories, of finite
length.\nomenclature[modKQ]{$\modKQ$}{abelian category of representations of a
quiver $Q$}\nomenclature[modKQI]{$\modKQI$}{abelian category of representations
of a quiver with relations}

If $(X,\rho)$ is a representation of $Q$ or $(Q,I)$, the {\it
dimension vector}\index{quiver!dimension vector} $\bdim(X,\rho)$ of
$(X,\rho)$ in $\Z^{Q_0}_{\sst\ge 0}\subset\Z^{Q_0}$ is
$\bdim(X,\rho):v\mapsto\dim_\K X_v$. This induces surjective
morphisms $\bdim:K_0(\modKQ)$ or~$K_0(\modKQI)\ra \Z^{Q_0}$.

In \cite[\S 10]{Joyc3} we show that $\modKQ$ and $\modKQI$ satisfy
Assumption \ref{dt3ass}, where we choose the quotient group
$K(\modKQ)$ or $K(\modKQI)$ to be $\Z^{Q_0}$, using this morphism
$\bdim$. For quivers we will always take $K(\modKQI)$ to be
$\Z^{Q_0}$ rather than the numerical Grothendieck group
$K^\num(\modKQI)$; one reason is that in some interesting cases the
Euler form $\bar\chi$ on $K_0(\modKQI)$ is zero, so that
$K^\num(\modKQI)=0$, but $\Z^{Q_0}$ is nonzero.
\label{dt7def2}
\end{dfn}

\begin{rem} There is an analogy between quiver representations and
sheaves on noncompact schemes, as in \S\ref{dt67}. For $\K,Q$ or
$(Q,I)$ as above, we define $\projKQ$ and $\projKQI$ to be the exact
categories of finitely generated, projective, but not necessarily
finite-dimensional modules over $\K Q$ and $\K Q/I$. Here a
representation is {\it projective\/}\index{quiver!projective
representation} if it is a direct summand of a free
module.\nomenclature[projKQI]{$\projKQI$}{exact category of projective
representations of $(Q,I)$}\nomenclature[projKQ]{$\projKQ$}{exact category of
projective representations of a quiver $Q$}

If $E\in\projKQI$ and $F\in\modKQI$ then $\Ext^*(E,F)$ is
finite-dimensional, so we have a biadditive pairing
$\bar\chi:K_0(\projKQI)\times K_0(\modKQI)\ra\Z$ given by
$\bar\chi([E],[F])=\sum_{i\ge 0}(-1)^i\dim\Ext^i(E,F)$. The quotient
of $K_0(\projKQI)$ by the left kernel of $\bar\chi$, and the
quotient of $K_0(\modKQI)$ by the right kernel of $\bar\chi$, are
both naturally isomorphic to the dimension vectors $\Z^{Q_0}$.

We can think of $\modKQ,\modKQI$ as like the category of compactly
supported sheaves $\coh_\cs(X)$ for some smooth noncompact
$\K$-scheme $X$, and $\projKQ,\ab\projKQI$ as like the category
$\coh(X)$ of all coherent sheaves. The pairing
$\bar\chi:K_0(\projKQI)\times K_0(\modKQI)\ra\Z$ is like the pairing
$\bar\chi:K_0(\coh(X))\times K_0(\coh_\cs(X))\ra\Z$ in \S\ref{dt67}.
Thus, our choice of $K(\modKQ)=\Z^{Q_0}$ is directly analogous to
our definition of $K(\coh_\cs(X))$ in~\S\ref{dt67}.

As we will discuss briefly in \S\ref{dt75}, there are examples known
in which there are equivalences of derived categories
$D^b(\modKQ/I)\sim D^b(\coh_\cs(X))$ for $X$ a noncompact
Calabi--Yau 3-fold over $\K$. Then we also expect equivalences of
derived categories~$D^b(\projKQ/I)\sim D^b(\coh(X))$.
\label{dt7rem1}
\end{rem}

If $Q$ is a quiver, the moduli stack $\fM_Q$ of objects $(X,\rho)$
in $\modKQ$ is an Artin $\K$-stack. For $\bs d\in\Z_{\sst\ge
0}^{Q_0}$, the open substack $\fM^{\bs d}_Q$ of $(X,\rho)$ with
$\bdim(X,\rho)=\bs d$ has a very explicit description: as a quotient
$\K$-stack we have\nomenclature[McQ]{$\fM_Q$}{moduli stack of
representations of a quiver $Q$}
\e
\fM^{\bs d}_Q\cong\ts\bigl[\prod_{e\in Q_1}\Hom(\K^{\bs
d(t(e))},\K^{\bs d(h(e))})/\prod_{v\in Q_0}\GL(\bs d(v))\bigr].
\label{dt7eq2}
\e
If $(Q,I)$ is a quiver with relations, the moduli stack $\fM_{Q,I}$
of objects $(X,\rho)$ in $\modKQI$ is a substack of $\fM_Q$, and for
$\bs d\in\Z_{\sst\ge 0}^{Q_0}$ we may
write\nomenclature[McQI]{$\fM_{Q,I}$}{moduli stack of
representations of a quiver with relations}
\e
\fM^{\bs d}_{Q,I}\cong\ts\bigl[V_{Q,I}^{\bs d}/\prod_{v\in
Q_0}\GL(\bs d(v))\bigr],
\label{dt7eq3}
\e
where $V_{Q,I}^{\bs d}$ is a closed $\prod_{v\in Q_0}\GL(\bs
d(v))$-invariant $\K$-subscheme of $\prod_{e\in Q_1}\Hom\ab(\K^{\bs
d(t(e))},\ab\K^{\bs d(h(e))})$ defined using the relations~$I$.

Let $Q=(Q_0,Q_1,h,t)$ be a quiver, without relations. It is well
known that $\Ext^i(D,E)=0$ for all $D,E\in\modKQ$ and $i>1$, and
\e
\dim_\K\Hom(D,E)-\dim_\K\Ext^1(D,E)=\hat\chi(\bdim D,\bdim E),
\label{dt7eq4}
\e
where $\hat\chi:\Z^{Q_0}\times\Z^{Q_0}\ra\Z$ is the Euler form of
$\modKQ$, given by
\e
\hat\chi(\bs d,\bs e)=\ts\sum_{v\in Q_0}\bs d(v)\bs e(v)-\sum_{e\in
Q_1}\bs d(t(e))\bs e(h(e)).
\label{dt7eq5}
\e
Note that $\hat\chi$ need not be antisymmetric. Define
$\bar\chi:\Z^{Q_0}\times\Z^{Q_0}\ra\Z$ by
\e
\bar\chi(\bs d,\bs e)=\hat\chi(\bs d,\bs e)-\hat\chi(\bs e,\bs d)=
\ts\sum_{e\in Q_1}\bigl(\bs d(h(e))\bs e(t(e))- \bs d(t(e))\bs
e(h(e))\bigr).
\label{dt7eq6}
\e
Then $\bar\chi$ is antisymmetric, and as in (b) above, for all
$D,E\in\modKQ$ we have
\e
\begin{split}
\bar\chi\bigl(\bdim D,\bdim E\bigr)=
\,&\bigl(\dim\Hom(D,E)-\dim\Ext^1(D,E)\bigr)-\\
&\bigl(\dim\Hom(E,D)-\dim\Ext^1(E,D)\bigr).
\end{split}
\label{dt7eq7}
\e
This is the analogue of \eq{dt3eq14} for Calabi--Yau 3-folds,
property (b) at the beginning of \S\ref{dt7}. Theorem \ref{dt7thm1}
generalizes \eq{dt7eq7} to quivers with a superpotential.

We define a class of {\it stability conditions\/}\index{stability
condition!for quivers} on $\modKQI$,~\cite[Ex.~4.14]{Joyc7}.

\begin{ex} Let $(Q,I)$ be a quiver with relations, and take
$K(\modKQI)$ to be $\Z^{Q_0}$, as above. Then $C(\modKQI)=
\Z_{\sst\ge 0}^{Q_0}\sm\{0\}$. Let $c:Q_0\ra\R$ and
$r:Q_0\ra(0,\iy)$ be maps. Define $\mu:C(\modKQI)\ra\R$ by
\begin{equation*}
\mu(\bs d)=\frac{\sum_{v\in Q_0}c(v)\bs d(v)}{\sum_{v\in Q_0}r(v)\bs
d(v)}\,.
\end{equation*}
Note that $\sum_{v\in Q_0}r(v)\bs d(v)>0$ as $r(v)>0$ for all $v\in
Q_0$, and $\bs d(v)\ge 0$ for all $v$ with $\bs d(v)>0$ for some
$v$. Then \cite[Ex.~4.14]{Joyc7} shows that
$(\mu,\R,\le)$\nomenclature[\mu R]{$(\mu,\R,\le)$}{slope stability
condition on $\modKQ$ or $\modKQI$} is a {\it permissible stability
condition\/}\index{stability condition!permissible} on $\modKQI$
which we call {\it slope stability}. Write $\fM_\rss^{\bs d}(\mu)$
for the open $\K$-substack of $\mu$-semistable objects in class $\bs
d$ in~$\fM^{\bs d}_{Q,I}$.

A simple case is to take $c\equiv 0$ and $r\equiv 1$, so that
$\mu\equiv 0$. Then $(0,\R,\le)$ is a trivial stability condition on
$\modKQ$ or $\modKQI$, and every nonzero object in $\modKQ$ or
$\modKQI$ is 0-semistable, so that $\fM_\rss^{\bs d}(0)=\fM^{\bs
d}_{Q,I}$.
\label{dt7ex1}
\end{ex}

\subsection{Quivers with superpotentials, and 3-Calabi--Yau
categories}
\label{dt72}

We shall be interested in quivers with relations coming from a {\it
superpotential}. This is an idea which originated in Physics. Two
foundational mathematical papers on them are Ginzburg \cite{Ginz}
and Derksen, Weyman and Zelevinsky \cite{DWZ}. Again, $\K$ is an
algebraically closed field of characteristic zero
throughout.\index{field $\K$!characteristic zero}

\begin{dfn} Let $Q$ be a quiver. A {\it
superpotential\/}\index{quiver!with superpotential!definition} $W$ for
$Q$ over $\K$ is an element of $\K Q/[\K Q,\K Q]$. The cycles in $Q$
up to cyclic permutation form a basis for $\K Q/[\K Q,\K Q]$ over
$\K$, so we can think of $W$ as a finite $\K$-linear combination of
cycles up to cyclic permutation. Following \cite{KoSo1}, we call $W$
{\it minimal\/} if all cycles in $W$ have length at least 3. We will
consider only minimal superpotentials $W$.

Define $I$ to be the two-sided ideal in $\K Q$ generated by $\pd_eW$
for all edges $e\in Q_1$, where if $C$ is a cycle in $Q$, we define
$\pd_eC$ to be the sum over all occurrences of the edge $e$ in $C$
of the path obtained by cyclically permuting $C$ until $e$ is in
first position, and then deleting it. Since $W$ is minimal, $I$ lies
in $\K Q_{(2)}$, so that $(Q,I)$ is a quiver with relations.

We allow $W\equiv 0$, so that $I=0$, and $\modKQI=\modKQ$.
\label{dt7def3}
\end{dfn}

When $I$ comes from a superpotential $W$, we can improve the
description \eq{dt7eq3} of the moduli stacks $\fM^{\bs d}_{Q,I}$.
Define a $\prod_{v\in Q_0}\GL(\bs d(v))$-invariant polynomial
\begin{equation*}
\ts W^{\bs d}:\prod_{e\in Q_1}\Hom\bigl(\K^{\bs d(t(e))},\K^{\bs
d(h(e))}\bigr)\longra\K
\end{equation*}
as follows. Write $W$ as a finite sum $\sum_i\ga^iC^i$,where
$\ga^i\in\K$ and $C^i$ is a cycle $v_0^i\,{\buildrel
e_1^i\over\longra}\, v_1^i\ra\cdots\ra v_{k^i-1}^i\,{\buildrel
e_{k^i}^i\over\longra}\,v_{k^i}^i=v_0^i$ in $Q$. Set
\begin{equation*}
W^{\bs d}\bigl(A_e:e\in Q_1\bigr)=\ts\sum_i\ga^i
\Tr\bigl(A_{e_{k^i}^i}\ci A_{e_{k^i-1}^i}\ci\cdots \ci
A_{e_1^i}\bigr).
\end{equation*}
Then $V_{Q,I}^{\bs d}=\Crit(W^{\bs d})$ in \eq{dt7eq3}, so that
\e
\fM^{\bs d}_{Q,I}\cong\ts\bigl[\Crit(W^{\bs d})/\prod_{v\in
Q_0}\GL(\bs d(v))\bigr].
\label{dt7eq8}
\e
Equation \eq{dt7eq8} is an analogue of Theorem \ref{dt5thm3} for
categories $\modKQI$ coming from a superpotential $W$ on $Q$, and
gives property (a) at the beginning of~\S\ref{dt7}.

We now show that property (b) at the beginning of \S\ref{dt7} holds
for quivers with relations $(Q,I)$ coming from a minimal
superpotential $W$. Note that we do not impose any other condition
on $W$, and in particular, we do not require $\modKQI$ to be
3-Calabi--Yau. Also, $\bar\chi$ is in general not the Euler
form\index{Euler form} of the abelian category $\modKQI$. When $W\equiv
0$, so that $\modKQI=\modKQ$, Theorem \ref{dt7thm1} reduces to
equations \eq{dt7eq6}--\eq{dt7eq7}. We have not been able to find a
reference for Theorem \ref{dt7thm1} and it may be new, though it is
probably obvious to experts in the context of Remark \ref{dt7rem2}
below.

\begin{thm} Let\/ $Q=(Q_0,Q_1,h,t)$ be a quiver with relations $I$
coming from a minimal superpotential\/ $W$ on\/ $Q$ over\/ $\K$.
Define $\bar\chi:\Z^{Q_0}\times\Z^{Q_0}\ra\Z$ by
\e
\bar\chi(\bs d,\bs e)=\ts\sum_{e\in Q_1}\bigl(\bs d(h(e))\bs
e(t(e))-\bs d(t(e))\bs e(h(e))\bigr).
\label{dt7eq9}
\e
Then for any $D,E\in\modKQI$ we have
\e
\begin{split}
\bar\chi\bigl(\bdim D,\bdim E\bigr)=
\,&\bigl(\dim\Hom(D,E)-\dim\Ext^1(D,E)\bigr)-\\
&\bigl(\dim\Hom(E,D)-\dim\Ext^1(E,D)\bigr).
\end{split}
\label{dt7eq10}
\e
\label{dt7thm1}
\end{thm}

\begin{proof} Write $D=(X_v:v\in Q_0$, $\rho_e:e\in Q_1)$ and
$E=(Y_v:v\in Q_0$, $\si_e:e\in Q_1)$. Define a sequence of
$\K$-vector spaces and linear maps
\e
\begin{gathered}
\xymatrix@C=17pt@R=5pt{ 0 \ar[r] & \bigop_{v\in Q_0}X_v^*\ot
Y_v\ar[rr]^{\rd_1}
&&\bigop_{e\in Q_1}X_{t(e)}^*\ot Y_{h(e)} \ar[r]^(0.8){\rd_2} &\\
& \bigop_{e\in Q_1}X_{h(e)}^*\ot Y_{t(e)} \ar[rr]^{\rd_3} &&
\bigop_{v\in Q_0}X_v^*\ot Y_v \ar[r] & 0,}
\end{gathered}
\label{dt7eq11}
\e
where $\rd_1,\rd_2,\rd_3$ are given by
\ea
&\rd_1:\bigl(\phi_v\bigr){}_{v\in Q_0} \longmapsto \bigl(
\phi_{h(e)}\ci \rho_e-\si_e\ci\phi_{t(v)}\bigr){}_{e\in Q_1},
\label{dt7eq12}\\
\begin{split}
&\rd_2:\bigl(\psi_e\bigr){}_{e\in Q_1} \longmapsto \bigl(
\ts\sum_{e\in Q_1}L^{W,D,E}_{e,f}(\psi_e)\bigr){}_{f\in
Q_1},\quad\text{where}\\
&L^{W,D,E}_{e,f}(\psi_e)=\sum_{\begin{subarray}{l}\text{terms
$c\,\bigl(\mathop{\bu}\limits^{\sst
t(f)}\mathop{\longra}\limits^f\mathop{\bu}\limits^{\sst h(f)}
\mathop{\longra}\limits^{g_1}\bu\cdots\bu
\mathop{\longra}\limits^{g_k}\mathop{\bu}\limits^{\sst
t(e)}\mathop{\longra}\limits^e\mathop{\bu}\limits^{\sst h(e)}
\mathop{\longra}\limits^{h_1}\bu\cdots\bu\mathop{\longra}\limits^{h_l}
\mathop{\bu}\limits^{\sst t(f)}\bigr)$}\\
\text{in $W$ up to cyclic permutation, $c\in\K$}\end{subarray}
\!\!\!\!\!\!\!\!\!\!\!\!\!\!\!\!\!\!\!\!\!\!\!\!\!\!\!\!\!\!\!\!\!\!\!
\!\!\!\!\!\!\!\!\!\!\!\!\!\!\!\!\!\!\!\!\!\!\!\!\!\!\!\!\!\!\!\!\!\!\!
\!\!\!\!\!\!\!\!\!\!\!\!\!\!\!\!\!\!\!\!\!\!\!\!\!\!\!\!\!\!\!\!\!\!\!
\!\!\!\!\!\!\!\!\!\!\!\!\!\!\!\!\!\!\!\!\!\!\!\!\!\!\!\!\!\!\!\!\!\!\!
\!\!\!\!\!\! } c\,
\si_{h_l}\ci\cdots\ci\si_{h_1}\ci\psi_e\ci\rho_{g_k}
\ci\cdots\ci\rho_{g_1},
\end{split}
\label{dt7eq13}\\
&\rd_3:\bigl(\xi_e\bigr){}_{e\in Q_1} \longmapsto
\bigl(\ts\sum_{e\in Q_1:\,t(e)=v}\xi_e\ci\rho_e - \ts\sum_{e\in
Q_1:\,h(e)=v}\si_e\ci\xi_e\bigr){}_{v\in Q_0}.
\label{dt7eq14}
\ea

Observe that the dual sequence of \eq{dt7eq11}, namely
\e
\begin{gathered}
\xymatrix@C=17pt@R=5pt{ 0 \ar[r] & \bigop_{v\in Q_0}Y_v^*\ot X_v
\ar[rr]^{\rd_3^*}
&&\bigop_{e\in Q_1}Y_{t(e)}^*\ot X_{h(e)} \ar[r]^(0.8){\rd_2^*} &\\
& \bigop_{e\in Q_1}Y_{h(e)}^*\ot X_{t(e)} \ar[rr]^{\rd_1^*} &&
\bigop_{v\in Q_0}Y_v^*\ot X_v \ar[r] & 0,}
\end{gathered}
\label{dt7eq15}
\e
is \eq{dt7eq11} with $D$ and $E$ exchanged. That $\rd_3^*,\rd_1^*$
correspond to $\rd_1,\rd_3$ with $D,E$ exchanged is immediate from
\eq{dt7eq12} and \eq{dt7eq14}; for $\rd_2^*$, we find from
\eq{dt7eq13} that $\bigl(L^{W,D,E}_{e,f}\bigr){}^*=
L^{W,E,D}_{f,e}$, by cyclically permuting the term
$\mathop{\bu}\limits^{\sst t(f)}\mathop{\longra}\limits^f
\mathop{\bu}\limits^{\sst h(f)}\cdots \bu\mathop{\longra}
\limits^{h_l} \mathop{\bu}\limits^{\sst t(f)}$ in \eq{dt7eq13} to
obtain $\mathop{\bu}\limits^{\sst t(e)}\mathop{\longra}\limits^e
\mathop{\bu}\limits^{\sst h(e)}\cdots\bu\mathop{\longra}
\limits^{g_k}\mathop{\bu}\limits^{\sst t(e)}$.

We claim that \eq{dt7eq11}, and hence \eq{dt7eq15}, are {\it
complexes}, that is, $\rd_2\ci\rd_1=0$ and $\rd_3\ci\rd_2=0$. To
show $\rd_2\ci\rd_1=0$, for $\bigl(\phi_v\bigr){}_{v\in Q_0}$ in
$\bigop_{v\in Q_0}X_v^*\ot Y_v$ we have
\e
\begin{split}
\rd_2&\ci\rd_1\bigl((\phi_v){}_{v\in Q_0}\bigr)\\
&=\raisebox{-7pt}{\begin{Large}$\displaystyle\Bigl($\end{Large}}
\sum_{\text{$c\,\bigl(\mathop{\bu}\limits^{\sst
t(f)}\mathop{\longra}\limits^f\mathop{\bu}\limits^{\sst h(f)}
\mathop{\longra}\limits^{g_1}\bu\cdots\bu
\mathop{\longra}\limits^{g_k}\mathop{\bu}\limits^{v}
\mathop{\longra}\limits^{h_1}\bu\cdots\bu\mathop{\longra}\limits^{h_l}
\mathop{\bu}\limits^{\sst t(f)}\bigr)$ in $W$: $k\ge 1$, $l\ge 0$}
\!\!\!\!\!\!\!\!\!\!\!\!\!\!\!\!\!\!\!\!\!\!\!\!\!\!\!\!\!\!\!\!\!\!\!
\!\!\!\!\!\!\!\!\!\!\!\!\!\!\!\!\!\!\!\!\!\!\!\!\!\!\!\!\!\!\!\!\!\!\!
\!\!\!\!\!\!\!\!\!\!\!\!\!\!\!\!\!\!\!\!\!\!\!\!\!\!\!\!\!\!\!\!\!\!\!
\!\!\!\!\!\!} c\,
\si_{h_l}\ci\cdots\ci\si_{h_1}\ci\phi_v\ci\rho_{g_k}
\ci\cdots\ci\rho_{g_1}
\raisebox{-7pt}{\begin{Large}$\displaystyle\Bigr)$\end{Large}}
_{f\in Q_1}\\
&-\raisebox{-7pt}{\begin{Large}$\displaystyle\Bigl($\end{Large}}
\sum_{\text{$c\,\bigl(\mathop{\bu}\limits^{\sst
t(f)}\mathop{\longra}\limits^f\mathop{\bu}\limits^{\sst h(f)}
\mathop{\longra}\limits^{g_1}\bu\cdots\bu
\mathop{\longra}\limits^{g_k}\mathop{\bu}\limits^{v}
\mathop{\longra}\limits^{h_1}\bu\cdots\bu\mathop{\longra}\limits^{h_l}
\mathop{\bu}\limits^{\sst t(f)}\bigr)$ in $W$: $k\ge 0$, $l\ge 1$}
\!\!\!\!\!\!\!\!\!\!\!\!\!\!\!\!\!\!\!\!\!\!\!\!\!\!\!\!\!\!\!\!\!\!\!
\!\!\!\!\!\!\!\!\!\!\!\!\!\!\!\!\!\!\!\!\!\!\!\!\!\!\!\!\!\!\!\!\!\!\!
\!\!\!\!\!\!\!\!\!\!\!\!\!\!\!\!\!\!\!\!\!\!\!\!\!\!\!\!\!\!\!\!\!\!\!
\!\!\!\!\!\!} c\,\si_{h_l}\ci\cdots\ci\si_{h_1}
\ci\phi_v\ci\rho_{g_k} \ci\cdots\ci\rho_{g_1}
\raisebox{-7pt}{\begin{Large}$\displaystyle\Bigr)$\end{Large}}
_{f\in Q_1}\\
&=\raisebox{-7pt}{\begin{Large}$\displaystyle\Bigl($\end{Large}}
\sum_{\text{$c\,\bigl(\mathop{\bu}\limits^{\sst
t(f)}\mathop{\longra}\limits^f\mathop{\bu}\limits^{\sst h(f)}
\mathop{\longra}\limits^{g_1}\bu\cdots\bu
\mathop{\longra}\limits^{g_k}\mathop{\bu}\limits^{\sst t(f)}\bigr)$
in $W$}
\!\!\!\!\!\!\!\!\!\!\!\!\!\!\!\!\!\!\!\!\!\!\!\!\!\!\!\!\!\!\!\!\!\!\!
\!\!\!\!\!\!\!\!\!\!\!\!\!\!\!\!\!\!\!\!\!\!\!\!\!\!\!\!\!\!\! }
c\,\phi_{t(f)}\ci\rho_{g_k} \ci\cdots\ci\rho_{g_1}
\raisebox{-7pt}{\begin{Large}$\displaystyle\Bigr)$\end{Large}}
_{f\in Q_1}\\
&-\raisebox{-7pt}{\begin{Large}$\displaystyle\Bigl($\end{Large}}
\sum_{\text{$c\,\bigl(\mathop{\bu}\limits^{\sst
t(f)}\mathop{\longra}\limits^f\mathop{\bu}\limits^{\sst h(f)}
\mathop{\longra}\limits^{h_1}\bu\cdots\bu\mathop{\longra}\limits^{h_l}
\mathop{\bu}\limits^{\sst t(f)}\bigr)$ in $W$}
\!\!\!\!\!\!\!\!\!\!\!\!\!\!\!\!\!\!\!\!\!\!\!\!\!\!\!\!\!\!\!\!\!\!\!
\!\!\!\!\!\!\!\!\!\!\!\!\!\!\!\!\!\!\!\!\!\!\!\!\!\!\!\!\!\!\! }
c\,\si_{h_l}\ci\cdots\ci\si_{h_1}\ci \phi_{h(f)}
\raisebox{-7pt}{\begin{Large}$\displaystyle\Bigr)$\end{Large}}
_{f\in Q_1}=0.
\end{split}
\label{dt7eq16}
\e
Here the second line of \eq{dt7eq16} comes from the first term
$\phi_{h(e)}\ci\rho_e$ on the r.h.s.\ of \eq{dt7eq12}, and we have
included $\rho_e$ as $\rho_{g_k}$ in
$\rho_{g_k}\ci\cdots\ci\rho_{g_1}$ by replacing $k$ by $k+1$, which
is why we have the condition $k\ge 1$. The third line of
\eq{dt7eq16} comes from the second term $-\si_e\ci\phi_{t(v)}$ on
the r.h.s.\ of \eq{dt7eq12}, and we have included $\si_e$ as
$\si_{h_1}$ in $\si_{h_l}\ci\cdots\ci\si_{h_1}$ replacing $l$ by
$l+1$, which is why we have $l\ge 1$. The fourth and fifth lines of
\eq{dt7eq16} cancel the terms $k\ge 1$, $l\ge 1$ in the second and
third lines. Finally, we note that the sums on the fourth and fifth
lines vanish as they are the compositions of
$\phi_{t(f)},\phi_{h(f)}$ with the relations satisfied by
$\bigl(\rho_e\bigr){}_{e\in Q_1}$ and $\bigl(\si_e\bigr){}_{e\in
Q_1}$ coming from the cyclic derivative $\pd_fW$. Thus
$\rd_2\ci\rd_1=0$. Since \eq{dt7eq15} is \eq{dt7eq11} with $D$ and
$E$ exchanged, the same proof shows that $\rd_2^*\ci\rd_3^*=0$, and
hence $\rd_3\ci\rd_2=0$. Therefore \eq{dt7eq11}, \eq{dt7eq15} are
complexes.

Thus we can form the cohomology of \eq{dt7eq11}. We will show that
it satisfies
\ea
\Ker\rd_1&\cong\Hom(D,E), &\Ker\rd_2/\Im\rd_1&\cong\!\Ext^1(D,E),
\label{dt7eq17}\\
\Ker\rd_3/\Im\rd_2&\cong\Ext^1(E,D)^*, & \!\!\bigl(\ts\bigop_{v\in
Q_0}X_v^*\!\ot\!Y_v\bigr)/\Im\rd_3&\cong\!\Hom(E,D)^*.
\label{dt7eq18}
\ea
For the first equation of \eq{dt7eq17}, observe that
$\rd_1\bigl((\phi_v){}_{v\in Q_0}\bigr)=0$ is equivalent to
$\phi_{h(e)}\ci\rho_e=\si_e\ci\phi_{t(v)}$ for all $e\in Q_1$, which
is the condition for $(\phi_v){}_{v\in Q_0}$ to define a morphism of
representations $\phi:(X,\rho)\ra(Y,\si)$ in
Definition~\ref{dt7def2}.

For the second equation of \eq{dt7eq17}, note that elements of
$\Ext^1(D,E)$ correspond to isomorphism classes of exact sequences
$0\ra E\,{\buildrel\al\over\longra}\,
F\,{\buildrel\be\over\longra}\, D\ra 0$ in $\modKQI$. Write
$F=(Z_v:v\in Q_0$, $\tau_e:e\in Q_1)$. Then for all $v\in Q_0$ we
have exact sequences of $\K$-vector spaces
\e
\smash{\xymatrix{ 0 \ar[r] & Y_v \ar[r]^{\al_v} & Z_v \ar[r]^{\be_v}
& X_v \ar[r] & 0.}}
\label{dt7eq19}
\e
Choose isomorphisms $Z_v\cong Y_v\op X_v$ for all $v\in Q_0$
compatible with \eq{dt7eq19}. Then for each $e\in Q_1$, we have
linear maps $\tau_e:Y_{t(e)}\op X_{t(e)}\ra Y_{h(e)}\op X_{h(e)}$.
As $\al,\be$ are morphisms of representations, we see that in matrix
notation
\e
\tau_e=\begin{pmatrix} \rho_e & \psi_e \\ 0 & \si_e \end{pmatrix}.
\label{dt7eq20}
\e
Thus $\bigl(\psi_e\bigr){}_{e\in Q_1}$ lies in $\bigop_{e\in
Q_1}X_{t(e)}^*\ot Y_{h(e)}$, the second space in~\eq{dt7eq11}.

Given that $\bigl(\rho_e\bigr){}_{e\in Q_1}$ and
$\bigl(\si_e\bigr){}_{e\in Q_1}$ satisfy the relations in $\modKQI$,
which come from the cyclic derivatives $\pd_fW$ for $f\in Q_1$, it
is not difficult to show that $\bigl(\tau_e\bigr){}_{e\in Q_1}$ of
the form \eq{dt7eq20} satisfy the relations in $\modKQI$ if and only
if $\rd_2\bigl(\psi_e\bigr){}_{e\in Q_1}=0$. Therefore exact
sequences $0\ra E\ra F\ra D\ra 0$ in $\modKQI$ together with choices
of isomorphisms $Z_v\cong Y_v\op X_v$ for $v\in Q_0$ splitting
\eq{dt7eq19} correspond to elements $\bigl(\psi_e\bigr){}_{e\in
Q_1}$ in $\Ker\rd_2$. The freedom to choose splittings of
\eq{dt7eq19} is $X_v^*\ot Y_v$. Summing this over all $v\in Q_0$
gives the first space in \eq{dt7eq11}, and quotienting by this
freedom corresponds to quotienting $\Ker\rd_2$ by $\Im\rd_1$. This
proves the second equation of~\eq{dt7eq17}.

Equation \eq{dt7eq18} follows from \eq{dt7eq17} and the fact that
the dual complex \eq{dt7eq15} of \eq{dt7eq11} is \eq{dt7eq11} with
$D,E$ exchanged, so that the dual of the cohomology of \eq{dt7eq11}
is the cohomology of \eq{dt7eq11} with $D,E$ exchanged. Taking the
Euler characteristic of \eq{dt7eq11} and using \eq{dt7eq9} and
\eq{dt7eq17}--\eq{dt7eq18} then yields~\eq{dt7eq10}.
\end{proof}

Equation \eq{dt7eq8} and Theorem \ref{dt7thm1} are analogues for
categories $\modKQI$ coming from quivers with superpotentials of
(a),(b) at the beginning of \S\ref{dt7}. Now (a),(b) for $\coh(X)$
depend crucially on $X$ being a {\it Calabi--Yau\/ $3$-fold}. We now
discuss two senses in which $\modKQI$ can be like a Calabi--Yau
3-fold.

\begin{dfn} A $\K$-linear abelian category $\A$ is called 3-{\it
Calabi--Yau\/}\index{abelian category!3-Calabi--Yau} if for all
$D,E\in\A$ we have $\Ext^i(D,E)=0$ for $i>3$, and there are choices
of isomorphisms $\Ext^i(D,E)\cong \Ext^{3-i}(E,D)^*$ for
$i=0,\ldots,3$, which are functorial in an appropriate way. That is,
$\A$ has Serre duality\index{Serre duality} in dimension 3. When $X$ is
a Calabi--Yau 3-fold over $\K$, the coherent sheaves $\coh(X)$ are
3-Calabi--Yau. For more details, see Ginzburg \cite{Ginz}, Bocklandt
\cite{Bock}, and Segal~\cite{Sega}.
\label{dt7def4}
\end{dfn}

An interesting problem in this field is to find examples of
3-Calabi--Yau abelian categories. A lot of work has been done on
this. It has become clear that categories $\modKQI$ coming from a
superpotential $W$ on $Q$ are often, but not always, 3-Calabi--Yau.
Here are two classes of examples.

\begin{ex} Let $G$ be a finite subgroup of $\SL(3,\C)$. The {\it McKay
quiver\/}\index{quiver!McKay quiver} $Q_G$ of $G$ is defined as
follows. Let the vertex set of $Q_G$ be the set of isomorphism
classes of irreducible representations of $G$. If vertices $i,j$
correspond to $G$-representations $V_i,V_j$, let the number of edges
$\mathop{\bu} \limits^{\sst i}\ra\mathop{\bu}\limits^{\sst j}$ be
$\dim\Hom_G(V_i,V_j\ot\C^3)$, where $\C^3$ has the natural
representation of $G\subset\SL(3,\C)$. Identify these edges with a
basis for $\Hom_G(V_i,V_j\ot\C^3)$.

Following Ginzburg \cite[\S 4.4]{Ginz}, define a cubic
superpotential $W_G$ for $Q_G$ by
\begin{equation*}
W_G=\sum_{\text{triangles } \mathop{\bu} \limits^{\sst
i}\mathop{\ra}\limits^{\sst e} \mathop{\bu} \limits^{\sst
j}\mathop{\ra}\limits^{\sst f} \mathop{\bu}\limits^{\sst
k}\mathop{\ra}\limits^{\sst g} \mathop{\bu}\limits^{\sst i}\text{ in
$Q_G$}\!\!\!\!\!\!\!\!\!\!\!\!\!\!\!\!\!\!\!\!\!\!\!\!\!\!\!\!\!\!\!\!
\!\!\!\!\!\!\!\!\!\!\!\!\!\!\!\!\!\!\!\!\!\!\!\!\!\!\!\!\!\!\!\!\!\!\!}
\Tr\bigl(V_i\,{\buildrel \sst e\over \longra}\, V_j\ot\C^3
\,{\buildrel \sst f\ot\id\over \longra}\, V_k\ot(\C^3)^{\ot^2}
\,{\buildrel\sst g\ot\id\over \longra}\, V_i\ot(\C^3)^{\ot^3}
\,{\buildrel \sst \id\ot\Om\over\longra}\, V_i\bigr)\,gfe,
\end{equation*}
where $\Om:(\C^3)^{\ot^3}\ra\C$ is induced by the holomorphic volume
form $\rd z_1\w\rd z_2\w\rd z_3$ on $\C^3$. Let $I_G$ be the
relations on $Q_G$ defined using $W_G$. Then Ginzburg
\cite[Th.~4.4.6]{Ginz} shows that mod-$\C Q_G/I_G$ is a
3-Calabi--Yau category, which is equivalent to the abelian category
of $G$-equivariant compactly-supported coherent sheaves on $\C^3$.
Using Bridgeland, King and Reid \cite{BKR}, he deduces
\cite[Cor.~4.4.8]{Ginz} that if $X$ is any crepant resolution of
$\C^3/G$, then the derived categories $D^b(\coh_\cs(X))$ and
$D^b(\text{mod-}\C Q_G/I_G)$ are equivalent, where $\coh_\cs(X)$ is
the abelian category of compactly-supported coherent sheaves on~$X$.
\label{dt7ex2}
\end{ex}

\begin{ex} A {\it brane tiling\/}\index{brane tiling} is a bipartite
graph drawn on the 2-torus $T^2$, dividing $T^2$ into
simply-connected polygons. From such a graph one can write down a
quiver $Q$ and superpotential $W$, yielding a quiver with relations
$(Q,I)$. If the brane tiling satisfies certain {\it consistency
conditions}, $\modCQI$ is a 3-Calabi--Yau category. For some
noncompact toric Calabi--Yau 3-fold $X$ constructed from the brane
tiling, the derived categories $D^b(\text{mod-}\C Q_G/I_G)$ and
$D^b(\coh_\cs(X))$ are equivalent. This class of examples arose in
String Theory,\index{String Theory} where they are known as `quiver
gauge theories' or `dimer models', and appear in the work of Hanany
and others, see for instance \cite{FHKVW,HHV,HaKe,HaVe}. Some
mathematical references are Ishii and Ueda \cite[\S 2]{IsUe} and
Mozgovoy and Reineke~\cite[\S 3]{MoRe}.
\label{dt7ex3}
\end{ex}

The abelian categories $\modKQI$ are only 3-Calabi--Yau for some
special quivers $Q$ and superpotentials $W$. For instance, if
$Q\ne\es$ and $W\equiv 0$, so that $\modKQI=\modKQ$, then $\modKQ$
is never 3-Calabi--Yau, since $\Hom(*,*)$, $\Ext^1(*,*)$ in $\modKQ$
are nonzero but $\Ext^2(*,*),\Ext^3(*,*)$ are zero. We now describe
a way to embed any $\modKQI$ coming from a minimal superpotential
$W$ in a 3-Calabi--Yau {\it triangulated\/}
category.\index{triangulated category!3-Calabi--Yau} The first
author is grateful to Alastair King and Bernhard Keller for
explaining this to him.

\begin{rem} By analogy with Definition \ref{dt7def4}, there is also
a notion of when a $\K$-linear triangulated category $\cal T$ is
3-{\it Calabi--Yau}, discussed in Keller \cite{Kell}. Let $Q$ be a
quiver with relations $I$ coming from a minimal superpotential $W$
for $Q$ over $\K$. Then there is a natural way to construct a
$\K$-linear, 3-Calabi--Yau triangulated category $\cal T$, and a
t-structure\index{t-structure} $\cal F$ on $\cal T$ whose heart ${\cal
A}={\cal F}\cap{\cal F}^\perp[1]$ is equivalent to $\modKQI$. This
is briefly discussed in Keller~\cite[\S 5]{Kell}.

Given $Q,W$, Ginzburg \cite{Ginz} constructs a DG-algebra ${\cal
D}(\K Q,W)$ (we want the non-complete version). Then $\cal T$ is the
full triangulated subcategory of the derived category of DG-modules
of ${\cal D}(\K Q,W)$ whose objects are DG-modules with homology of
finite total dimension. The standard t-structure on $\cal T$ has
heart $\cal A$ the DG-modules $M^\bu$ with $H^0(M^\bu)$
finite-dimensional and $H^i(M^\bu)=0$ for $i\ne 0$. Here
$H^0(M^\bu)$ is a representation of $H^0\bigl({\cal D}(\K
Q,W)\bigr)=\K Q/I$. Thus $M^\bu\mapsto H^0(M^\bu)$ induces a functor
${\cal A}\mapsto\modKQI$, which is an equivalence. Inverting this
induces a functor $D^b(\modKQI)\ra{\cal T}$. If this is an
equivalence then $\modKQI$ is 3-Calabi--Yau.

Kontsevich and Soibelman \cite[Th.~9, \S 8.1]{KoSo1} prove a related
result, giving a 1-1 correspondence between $\K$-linear
3-Calabi--Yau triangulated categories $\hat{\cal T}$ satisfying
certain conditions, and quivers $Q$ with minimal superpotential $W$
over $\K$. But their set-up is slightly different: in effect they
use Ginzburg's completed DG-algebra $\hat{\cal D}(\K Q,W)$ instead
of ${\cal D}(\K Q,W)$, they allow $W$ to be a formal power series
rather than just a finite sum, and the heart $\hat{\cal A}$ of the
t-structure on $\hat{\cal T}$ is nil-$\K Q/I$,\nomenclature[nilKQI]{nil-$\K
Q/I$}{abelian category of nilpotent representations of $(Q,I)$} the
abelian category of {\it nilpotent\/} representations
of~$(Q,I)$.\index{quiver!nilpotent representation}

Identify $\modKQI$ with the heart $\cal A$ in $\cal T$. Then for
$E,F\in\modKQI$, we can compute the Ext groups $\Ext^i(E,F)$ in
either $\modKQI$ or $\cal T$. We have $\Ext^i_{\modKQI}(E,F)\cong
\Ext^i_{\cal T}(E,F)$ for $i=0,1$, as $\modKQI$ is the heart of a
t-structure, but if $\modKQI$ is not 3-Calabi--Yau then in general
we have $\Ext^i_{\modKQI}(E,F)\ab\not\cong\Ext^i_{\cal T}(E,F)$ for
$i>1$. The cohomology of the complex \eq{dt7eq11} is $\Ext^*_{\cal
T}(E,F)$, and $\bar\chi$ in \eq{dt7eq9} is the Euler form of $\cal
T$, which may not be the same as the Euler form of $\modKQI$, if
this exists.

In the style of Kontsevich and Soibelman \cite{KoSo1}, we can regard
the Donaldson--Thomas type invariants $\bar{DT}{}^{\bs
d}_{Q,I}(\mu),\bar{DT}{}^{\bs d}_Q(\mu),\hat{DT}{}^{\bs
d}_{Q,I}(\mu),\hat{DT}{}^{\bs d}_Q(\mu)$ of \S\ref{dt73} below as
counting $Z$-semistable objects in the 3-Calabi--Yau category $\cal
T$, where $(Z,{\cal P})$ is the Bridgeland stability condition
\cite{Brid1} on $\cal T$ constructed from the t-structure $\cal F$
on $\cal T$ and the slope stability condition $(\mu,\R,\le)$ on the
heart of~$\cal F$.

From this point of view, the question of whether or not $\modKQI$ is
3-Calabi--Yau seems less important, as we always have a natural
3-Calabi--Yau triangulated category $\cal T$ containing $\modKQI$ to
work in.
\label{dt7rem2}
\end{rem}

\subsection[Behrend functions, Lie algebra morphisms, and
D--T invariants]{Behrend function identities, Lie algebra morphisms,
and Donaldson--Thomas type invariants}
\label{dt73}

We now develop analogues of \S\ref{dt52}, \S\ref{dt53} and
\S\ref{dt62} for quivers. Let $Q$ be a quiver with relations $I$
coming from a minimal superpotential $W$ on $Q$ over $\C$. Write
$\fM_{Q,I}$ for the moduli stack of objects in $\modCQI$, an Artin
$\C$-stack locally of finite type, and $\fM^{\bs d}_{Q,I}$ for the
open substack of objects with dimension vector $\bs d$, which is of
finite type.

The proof of Theorem \ref{dt5thm4} in \S\ref{dt10} depends on two
things: the description of $\fM$ in terms of $\Crit(f)$ in Theorem
\ref{dt5thm3}, and equation \eq{dt3eq14}. For $\modCQI$ equation
\eq{dt7eq8} provides an analogue of Theorem \ref{dt5thm3}, and
Theorem \ref{dt7thm1} an analogue of \eq{dt3eq14}. Thus, the proof
of Theorem \ref{dt5thm4} also yields:

\begin{thm} In the situation above, with\/ $\fM_{Q,I}$ the moduli
stack of objects in a category\/ $\modCQI$ coming from a quiver $Q$
with minimal superpotential\/ $W,$ and\/ $\bar\chi$ defined in
{\rm\eq{dt7eq9},} the Behrend function $\nu_{\fM_{Q,I}}$ of\/
$\fM_{Q,I}$ satisfies the identities \eq{dt5eq2}--\eq{dt5eq3} for
all\/~$E_1,E_2\in\modCQI$.\index{Behrend function!identities}
\label{dt7thm2}
\end{thm}

Since the description of $\fM_{Q,I}$ in terms of $\Crit(W^{\bs d})$
in \eq{dt7eq8} is algebraic rather than complex analytic, and holds
over any field $\K$, we ask:

\begin{quest} Can you prove Theorem {\rm\ref{dt7thm2}} over an
arbitrary algebraically closed field\/ $\K$ of characteristic zero,
using the ideas of\/~{\rm\S\ref{dt42}?}\index{field $\K$!characteristic
zero}
\label{dt7quest1}
\end{quest}

Here is the analogue of Definition~\ref{dt5def1}.

\begin{dfn} Define a Lie algebra $\ti L(Q)$\nomenclature[L(Q)]{$\ti L(Q)$}{Lie
algebra depending on a quiver $Q$} to be the $\Q$-vector space with
basis of symbols $\ti\la^{\bs d}$ for $\bs d\in\Z^{Q_0}$, with Lie
bracket
\begin{equation*}
[\ti\la^{\bs d},\ti \la^{\bs e}]=(-1)^{\bar\chi(\bs d,\bs e)}
\bar\chi(\bs d,\bs e)\ti\la^{\bs d+\bs e},
\end{equation*}
as for \eq{dt5eq4}. This makes $\ti L(Q)$ into an
infinite-dimensional Lie algebra over $\Q$. Define $\Q$-linear maps
$\ti\Psi^{\chi,\Q}_{Q,I}:\oSFai (\fM_{Q,I},\chi,\Q)\ra\ti
L(Q)$\nomenclature[\Psi e]{$\ti\Psi_{Q,I}$}{Lie algebra morphism
$\SFai(\fM_{Q,I})\ra\ti L(Q)$}\nomenclature[\Psi
f]{$\ti\Psi^{\chi,\Q}_{Q,I}$}{Lie algebra morphism
$\oSFai(\fM_{Q,I},\chi,\Q)\ra \ti L(Q)$} and
$\ti\Psi_{Q,I}:\SFai(\fM_{Q,I})\ab\ra\ti L(Q)$ exactly as for
$\ti\Psi^{\chi,\Q},\ti\Psi$ in Definition~\ref{dt5def1}.
\label{dt7def5}
\end{dfn}

The proof of Theorem \ref{dt5thm5} in \S\ref{dt11} has two
ingredients: equation \eq{dt3eq14} and Theorem \ref{dt5thm4}.
Theorems \ref{dt7thm1} and \ref{dt7thm2} are analogues of these in
the quiver case. So the proof of Theorem \ref{dt5thm5} also yields:

\begin{thm} $\ti\Psi_{Q,I}:\SFai(\fM_{Q,I})\ra\ti L(Q)$ and\/
$\ti\Psi^{\chi,\Q}_{Q,I}:\oSFai(\fM_{Q,I},\chi,\Q)\ab\ra\ti L(Q)$
are Lie algebra morphisms.
\label{dt7thm3}
\end{thm}

Here is the analogue of Definitions \ref{dt5def2} and~\ref{dt6def1}.

\begin{dfn} Let $(\mu,\R,\le)$ be a slope stability condition on
$\modCQI$ as in Example \ref{dt7ex1}. It is permissible, as in
\cite[Ex.~4.14]{Joyc7}. So as in \S\ref{dt32} we have elements
$\bdss^{\bs d}(\mu)\in\SFa(\fM_{Q,I})$ and $\bep^{\bs
d}(\mu)\in\SFai(\fM_{Q,I})$ for all $\bs d\in C(\modCQI)=\Z_{\sst\ge
0}^{Q_0}\sm\{0\}\subset\Z^{Q_0}$. As in \eq{dt5eq7}, define {\it
quiver generalized Donaldson--Thomas
invariants\/}\index{Donaldson--Thomas invariants!for quivers|(}
$\bar{DT}{}^{\bs d}_{Q,I}(\mu)\in\Q$ for all $\bs d\in C(\modCQI)$
by\nomenclature[DTQIb]{$\bar{DT}{}^{\bs
d}_{Q,I}(\mu)$}{Donaldson--Thomas type invariants for $(Q,I)$}
\begin{equation*}
\ti\Psi_{Q,I}\bigl(\bep^{\bs d}(\mu)\bigr)=-\bar{DT}{}^{\bs
d}_{Q,I}(\mu)\ti\la^{\bs d}.
\end{equation*}

As in \eq{dt6eq15}, define {\it quiver BPS
invariants\/}\nomenclature[DTQIc]{$\hat{DT}{}^{\bs d}_{Q,I}(\mu)$}{BPS
invariants for $(Q,I)$}\index{BPS invariants!for quivers}
$\hat{DT}{}^{\bs d}_{Q,I}(\mu)\in\Q$ by
\e
\hat{DT}{}^{\bs d}_{Q,I}(\mu)=\sum_{m\ge 1,\; m\mid\bs
d}\frac{\Mo(m)}{m^2}\, \bar{DT}{}^{\bs d/m}_{Q,I}(\mu),
\label{dt7eq21}
\e
where $\Mo:\N\ra\Q$ is the M\"obius function.\index{Mobius
function@M\"obius function} As for \eq{dt6eq14}, the inverse of
\eq{dt7eq21} is
\e
\bar{DT}{}^{\bs d}_{Q,I}(\mu)=\sum_{m\ge 1,\; m\mid\bs
d}\frac{1}{m^2}\, \hat{DT}{}^{\bs d/m}_{Q,I}(\mu).
\label{dt7eq22}
\e

If $W\equiv 0$, so that $\modCQI=\modCQ$, we write $\bar{DT}{}^{\bs
d}_Q(\mu),\hat{DT}{}^{\bs d}_Q(\mu)$ for $\smash{\bar{DT}{}^{\bs
d}_{Q,I}(\mu),\hat{DT}{}^{\bs
d}_{Q,I}(\mu)}$.\nomenclature[DTQb]{$\bar{DT}{}^{\bs
d}_Q(\mu)$}{Donaldson--Thomas type invariants for a quiver
$Q$}\nomenclature[DTQc]{$\hat{DT}{}^{\bs d}_Q(\mu)$}{BPS invariants for a
quiver $Q$} Note that $\mu\equiv 0$ is allowed as a slope stability
condition, with every object in $\modCQI$ 0-semistable, and this is
in many ways the most natural choice. So we have invariants
$\bar{DT}{}^{\bs d}_{Q,I}(0), \hat{DT}{}^{\bs d}_{Q,I}(0)$ and
$\smash{\bar{DT}{}^{\bs d}_Q(0), \hat{DT}{}^{\bs d}_Q(0)}$. We
cannot do this in the coherent sheaf case; the difference is that
for quivers $\fM_{Q,I}^{\bs d}$ is of finite type for all $\bs d\in
C(\modCQI)$, so $(0,\R,\le)$ is permissible on $\modCQI$, but for
coherent sheaves $\fM^\al$ is generally not of finite type for
$\al\in C(\coh(X))$ with $\dim\al>0$, so $(0,\R,\le)$ is not
permissible.
\label{dt7def6}
\end{dfn}

Here is the analogue of the integrality conjecture, Conjecture
\ref{dt6conj1}. We will prove the conjecture in \S\ref{dt76} for the
invariants $\hat{DT}{}^{\bs d}_Q(\mu)$, that is, the case~$W\equiv
0$.\index{Donaldson--Thomas invariants!integrality
properties}\index{stability condition!generic}

\begin{conj} Call\/ $(\mu,\R,\le)$
\begin{bfseries}generic\end{bfseries} if for all\/ $\bs d,\bs e\in
C(\modCQI)$ with\/ $\mu(\bs d)=\mu(\bs e)$ we have\/~$\bar\chi(\bs
d,\bs e)=0$. If\/ $(\mu,\R,\le)$ is generic, then $\hat{DT}{}^{\bs
d}_{Q,I}(\mu)\in\Z$ for all\/~$\bs d\in C(\modCQI)$.
\label{dt7conj1}
\end{conj}

If the maps $c:Q_0\ra\R$ and $r:Q_0\ra(0,\iy)$ defining $\mu$ in
Example \ref{dt7ex1} are generic, it is easy to see that $\mu(\bs
d)=\mu(\bs e)$ only if $\bs d,\bs e$ are linearly dependent over
$\Q$ in $\Z^{Q_0}$, so that $\bar\chi(\bs d,\bs e)=0$ by
antisymmetry of $\bar\chi$, and $(\mu,\R,\le)$ is generic in the
sense of Conjecture \ref{dt7conj1}. Thus, there exist generic
stability conditions $(\mu,\R,\le)$ on any~$\modCQI$.

Let $(\mu,\R,\le)$, $(\ti\mu,\R,\le)$ be slope stability conditions
on $\modCQI$. Then $(0,\R,\le)$ dominates\index{stability
condition!$(\tilde\tau,\tilde T,\leqslant)$ dominates
$(\tau,T,\leqslant)$} both, so Theorem \ref{dt3thm2} with
$(\mu,\R,\le)$, $(\ti\mu,\R,\le)$, $(0,\R,\le)$ in place of
$(\tau,T,\le)$, $(\ti\tau,\ti T,\le)$, $(\hat\tau,\hat T,\le)$
writes $\bep^{\bs d}(\ti\mu)$ in terms of the $\bep^{\bs e}(\mu)$ in
\eq{dt3eq10}. Applying $\ti\Psi_{Q,I}$, which is a Lie algebra
morphism by Theorem \ref{dt7thm3}, to this identity gives an
analogue of Theorem~\ref{dt5thm6}:

\begin{thm} Let\/ $(\mu,\R,\le)$ and\/ $(\ti\mu,\R,\le)$ be any two
slope stability conditions on $\modCQI,$ and\/ $\bar\chi$ be as in
\eq{dt7eq9}. Then for all\/ $\bs d\in C(\modCQI)$ we
have\index{wall-crossing formula}
\ea
&\bar{DT}{}^{\bs d}_{Q,I}(\ti\mu)=
\label{dt7eq23}\\
&\!\!\!\sum_{\substack{\text{iso.}\\ \text{classes}\\
\text{of finite}\\ \text{sets $I$}}}\,\,
\sum_{\substack{\ka:I\ra C(\modCQI):\\ \sum_{i\in I}\ka(i)=\bs d}}\,\,
\sum_{\begin{subarray}{l} \text{connected,}\\
\text{simply-}\\ \text{connected}\\ \text{digraphs $\Ga$,}\\
\text{vertices $I$}\end{subarray}}
\begin{aligned}[t]
(-1)^{\md{I}-1} V(I,\Ga,\ka;\mu,\ti\mu)\cdot
\prod\nolimits_{i\in I} \bar{DT}{}^{\ka(i)}_{Q,I}(\mu)&\\
\cdot (-1)^{\frac{1}{2}\sum_{i,j\in
I}\md{\bar\chi(\ka(i),\ka(j))}}\cdot\! \prod\limits_{\text{edges
\smash{$\mathop{\bu}\limits^{\sst i}\ra\mathop{\bu}\limits^{\sst
j}$} in $\Ga$}\!\!\!\!\!\!\!\!\!\!\!\!\!\!\!\!\!\!\!\!\!\!\!
\!\!\!\!\!\!\!\!} \bar\chi(\ka(i),\ka(j))&,
\end{aligned}
\nonumber
\ea
with only finitely many nonzero terms.
\label{dt7thm4}
\end{thm}

The form $\bar\chi$ in \eq{dt7eq9} is zero if and only if for all
vertices $i,j$ in $Q$, there are the same number of edges $i\ra j$
and $j\ra i$ in $Q$. Then \eq{dt7eq23} gives:

\begin{cor} Suppose that\/ $\bar\chi$ in \eq{dt7eq9} is
zero. Then for any slope stability conditions $(\mu,\R,\le)$ and\/
$(\ti\mu,\R,\le)$ on $\modCQI$ and all\/ $\bs d$ in $C(\modCQI)$ we
have $\bar{DT}{}^{\bs d}_{Q,I}(\ti\mu)=\bar{DT}{}^{\bs
d}_{Q,I}(\mu)$ and\/~$\hat{DT}{}^{\bs
d}_{Q,I}(\ti\mu)=\hat{DT}{}^{\bs d}_{Q,I}(\mu)$.
\label{dt7cor1}
\end{cor}

Here is a case in which we can evaluate the invariants very easily.

\begin{ex} Let $Q$ be a quiver without oriented cycles. Choose a
slope stability condition $(\mu,\R,\le)$ on $\modCQ$ such that
$\mu(\de_v)>\mu(\de_w)$ for all edges $v\ra w$ in $Q$. This is
possible as $Q$ has no oriented cycles. Then up to isomorphism the
only $\mu$-stable objects in $\modCQ$ are the simple representations
$S^v$ for $v\in Q_0$ and the only $\mu$-semistables are $kS^v$ for
$v\in Q_0$ and $k\ge 1$. Here $S^v=(X^v,\rho^v)$, where $X^v_w=\C$
if $v=w$ and $X^v_w=0$ if $v\ne w\in Q_0$, and $\rho^v_e=0$ for
$e\in Q_1$. Examples \ref{dt6ex1}--\ref{dt6ex2} and equations
\eq{dt7eq21}--\eq{dt7eq22} now imply that
\begin{equation*}
\bar{DT}{}^{\bs d}_Q(\mu)=\begin{cases} \displaystyle \frac{1}{l^2},
& \bs d=l\de_v, \; l\ge 1, \; v\in Q_0, \\
0, & \text{otherwise,}\end{cases} \quad \hat{DT}{}^{\bs
d}_Q(\mu)=\begin{cases} 1, & \bs d=\de_v, \; v\in Q_0, \\
0, & \text{otherwise.}\end{cases}
\end{equation*}
\label{dt7ex4}
\end{ex}

\subsection{Pair invariants for quivers}
\label{dt74}\index{stable pair invariants $PI^{\al,n}(\tau')$!analogue
for quivers|(}

We now discuss analogues for quivers of the moduli spaces of stable
pairs $\M_\stp^{\al,n}(\tau')$ and stable pair invariants
$PI^{\al,n}(\tau')$ in \S\ref{dt54}, and the identity \eq{dt5eq17}
in Theorem \ref{dt5thm10} relating $PI^{\al,n}(\tau')$ and the
$\bar{DT}{}^\be(\tau)$. Here are the basic definitions. These quiver
analogues of $\M_\stp^{\al,n}(\tau'),PI^{\al,n}(\tau')$ are not new,
similar things have been studied in quiver theory by Nakajima,
Reineke, Szendr\H oi and other authors for some years
\cite{EnRe,Naga,NaNa,Naka,MoRe,Rein1,Rein2,Szen}. We explain the
relations between our definitions and the literature after
Definition~\ref{dt7def8}.

\begin{dfn} Let $Q$ be a quiver with relations $I$ coming from a
superpotential $W$ on $Q$ over an algebraically closed field $\K$ of
characteristic zero. Suppose $(\mu,\R,\le)$ is a slope stability
condition on $\modKQI$, as in Example~\ref{dt7ex1}.

Let $\bs d,\bs e\in\Z^{Q_0}_{\sst\ge 0}$ be dimension vectors. A
{\it framed representation\/\index{quiver!framed representation}
$(X,\rho,\si)$ of\/ $(Q,I)$ of type\/} $(\bs d,\bs e)$ consists of a
representation $(X,\rho)=(X_v:v\in Q_0$, $\rho_e:e\in Q_1)$ of
$(Q,I)$ over $\K$ with $\dim X_v=\bs d(v)$ for all $v\in Q_0$,
together with linear maps $\si_v:\K^{\bs e(v)}\ra X_v$ for all $v\in
Q_0$. An {\it isomorphism\/} between framed representations
$(X,\rho,\si),(X',\rho',\si')$ consists of isomorphisms $i_v:X_v\ra
X_v'$ for all $v\in Q_0$ such that $i_{h(e)}\ci\rho_e=\rho'_e\ci
i_{t(e)}$ for all $e\in Q_1$ and $i_v\ci\si_v=\si_v'$ for all $v\in
Q_0$. We call a framed representation $(X,\rho,\si)$ {\it stable\/}
if
\begin{itemize}
\setlength{\itemsep}{0pt}
\setlength{\parsep}{0pt}
\item[(i)] $\mu([(X',\rho')])\le\mu([(X,\rho)])$ for all nonzero
subobjects $(X',\rho')\subset(X,\rho)$ in $\modKQI$ or $\modKQ$;
and
\item[(ii)] If also $\si$ factors through $(X',\rho')$, that is,
$\si_v(\C^{e(v)})\subseteq X_v'\subseteq X_v$ for all $v\in Q_0$,
then~$\mu([(X',\rho')])<\mu([(X,\rho)])$.
\end{itemize}
We will use $\mu'$ {\it to denote stability of framed
representations}, defined using~$\mu$.
\label{dt7def7}
\end{dfn}

Following Engel and Reineke \cite[\S 3]{EnRe} or Szendr\H oi
\cite[\S 1.2]{Szen}, we can in a standard way define moduli problems
for all framed representations, and for stable framed
representations. When $W\equiv 0$, so that $\modKQI=\modKQ$, the
moduli space of all framed representations of type $(\bs d,\bs e)$
is an Artin $\K$-stack $\fM_{\fr\,Q}^{\bs d,\bs e}$. By analogy with
\eq{dt7eq2} we have\nomenclature[MdQ]{$\fM_{\fr\,Q}^{\bs d,\bs e}$}{moduli
stack of framed representations of a quiver $Q$}
\e
\fM_{\fr\,Q}^{\bs d,\bs e}\cong\left[\frac{\ts\prod_{e\in
Q_1}\Hom(\K^{\bs d(t(e))},\K^{\bs d(h(e))})\times \prod_{v\in
Q_0}\Hom(\K^{\bs e(v)},\K^{\bs d(v)})}{\ts\prod_{v\in Q_0}\GL(\bs
d(v))}\right].
\label{dt7eq24}
\e
The moduli space of stable framed representations of type $(\bs
d,\bs e)$ is a fine moduli $\K$-scheme\index{fine moduli
scheme}\index{moduli scheme!fine} $\M_{\stf\,Q}^{\bs d,\bs
e}(\mu')$,\nomenclature[MdQst]{$\M_{\stf\,Q}^{\bs d,\bs e}(\mu')$}{fine moduli
scheme of stable framed representations of $Q$} an open
$\K$-substack of $\fM_{\fr\,Q}^{\bs d,\bs e}$, with
\e
\M_{\stf\,Q}^{\bs d,\bs e}(\mu')\cong U^{\bs d,\bs
e}_{\stf,Q}(\mu')/\ts\prod_{v\in Q_0}\GL(\bs d(v)),
\label{dt7eq25}
\e
where $U^{\bs d,\bs e}_{\stf,Q}(\mu')$ is open in
$\prod_{e}\Hom(\K^{\bs d(t(e))},\K^{\bs d(h(e))})\times
\prod_{v}\Hom(\K^{\bs e(v)},\K^{\bs d(v)})$, and $\prod_{v}\GL(\bs
d(v))$ acts freely on $U^{\bs d,\bs e}_{\stf,Q}(\mu')$, and
\eq{dt7eq25} may be written as a GIT quotient for an appropriate
linearization. From \eq{dt7eq24}--\eq{dt7eq25} we see that
$\fM_{\fr\,Q}^{\bs d,\bs e},\M_{\stf\,Q}^{\bs d,\bs e}(\mu')$ are
both smooth with dimension
\e
\dim\fM_{\fr\,Q}^{\bs d,\bs e}=\M_{\stf\,Q}^{\bs d,\bs e}(\mu')=
\hat\chi(\bs d,\bs d)+\ts\sum_{v\in Q_0}\bs e(v)\bs d(v).
\label{dt7eq26}
\e

Similarly, for general $W$, the moduli space of all framed
representations of type $(\bs d,\bs e)$ is an Artin $\K$-stack
$\fM_{\fr\,Q,I}^{\bs d,\bs e}$. By analogy with \eq{dt7eq8} we
have\nomenclature[MdQI]{$\fM_{\fr\,Q,I}^{\bs d,\bs e}$}{moduli stack of framed
representations of $(Q,I)$}
\begin{equation*}
\fM_{\fr\,Q,I}^{\bs d,\bs e}\cong\ts\bigl[\Crit(W^{\bs d})\times \prod_{v\in Q_0}\Hom(\K^{\bs e(v)},\K^{\bs d(v)})/\prod_{v\in
Q_0}\GL(\bs d(v))\bigr],
\end{equation*}
where $\Crit(W^{\bs d})\subseteq\prod_{e\in Q_1}\Hom(\K^{\bs
d(t(e))},\K^{\bs d(h(e))})$ is as in \eq{dt7eq8}, and the moduli
space of stable framed representations of type $(\bs d,\bs e)$ is a
fine moduli $\K$-scheme $\M_{\stf\,Q,I}^{\bs d,\bs
e}(\mu')$,\nomenclature[MdQstI]{$\M_{\stf\,Q,I}^{\bs d,\bs e}(\mu')$}{fine
moduli scheme of stable framed representations of $(Q,I)$} an open
$\K$-substack of $\fM_{\fr\,Q,I}^{\bs d,\bs e}$, with
\begin{equation*}
\M_{\stf\,Q,I}^{\bs d,\bs e}(\mu')\cong \frac{\bigl(\Crit(W^{\bs
d})\times\ts\prod_{v\in Q_0}\Hom(\K^{\bs e(v)},\K^{\bs d(v)})\bigr)\cap
U^{\bs d,\bs e}_{\stf,Q}(\mu')}{\ts\prod_{v\in Q_0}\GL(\bs d(v))}\,.
\end{equation*}

We can now define our analogues of invariants $PI^{\al,n}(\tau')$
for quivers, which following Szendr\H oi \cite{Szen} we call {\it
noncommutative Donaldson--Thomas invariants}.

\begin{dfn} In the situation above, define\nomenclature[NDTQI]{$NDT_{Q,I}^{\bs d,\bs
e}(\mu')$}{noncommutative Donaldson--Thomas invariants for
$(Q,I)$}\nomenclature[NDTQ]{$NDT_Q^{\bs d,\bs e}(\mu')$}{noncommutative
Donaldson--Thomas invariants for $Q$}\index{Donaldson--Thomas
invariants!noncommutative|(}
\ea
NDT_{Q,I}^{\bs d,\bs e}(\mu')&=\chi\bigl(\M_{\stf\,Q,I}^{\bs d,\bs
e}(\mu'),\nu_{\M_{\stf\,Q,I}^{\bs d,\bs e}(\mu')}\bigr),
\label{dt7eq27}\\
\begin{split}
NDT_{Q}^{\bs d,\bs e}(\mu')&=\chi\bigl(\M_{\stf\,Q}^{\bs d,\bs
e}(\mu'),\nu_{\M_{\stf\,Q}^{\bs d,\bs e}(\mu')}\bigr)\\
&=(-1)^{\hat\chi(\bs d,\bs d)+\sum_{v\in Q_0}\bs e(v)\bs d(v)}
\chi\bigl(\M_{\stf\,Q}^{\bs d,\bs e}(\mu')\bigr),
\end{split}
\label{dt7eq28}
\ea
where the second line in \eq{dt7eq28} holds as $\M_{\stf\,Q}^{\bs
d,\bs e}(\mu')$ is smooth of dimension \eq{dt7eq26}, so
$\nu_{\M_{\stf\,Q}^{\bs d,\bs e}(\mu')}\equiv (-1)^{ \hat\chi(\bs
d,\bs d)+\sum_{v\in Q_0}\bs e(v)\bs d(v)}$ by
Theorem~\ref{dt4thm1}(i).
\label{dt7def8}
\end{dfn}

Here is how Definitions \ref{dt7def7} and \ref{dt7def8} relate to
the literature. We first discuss the case of quivers without
relations.
\begin{itemize}
\setlength{\itemsep}{0pt}
\setlength{\parsep}{0pt}
\item `Framed' moduli spaces of quivers\index{quiver!framed
representation} appear in the work of Nakajima, see for instance
\cite[\S 3]{Naka}. His framed moduli schemes ${\mathfrak
R}_\theta(\bs d,\bs e)$ are similar to our moduli schemes
$\smash{\M_{\stf\,Q}^{\bs d,\bs e}(\mu')}$, with one difference:
rather than framing $(X_v:v\in Q_0$, $\rho_e:e\in Q_1)$ using
linear maps $\si_v:\K^{e(v)}\ra X_v$ for $v\in Q_0$, as we do,
he uses linear maps $\smash{\si_v:X_v\ra\K^{e(v)}}$ going the
other way.

Here is a natural way to relate framings of his type to framings of
our type. Given a quiver $Q=(Q_0,Q_1,h,t)$, let $Q^{\rm op}$ be $Q$
with directions of edges reversed, that is, $Q^{\rm
op}=(Q_0,Q_1,t,h)$. If $(X_v:v\in Q_0$, $\rho_e:e\in Q_1)$ is a
representation of $Q$ then $(X_v^*:v\in Q_0$, $\rho_e^*:e\in Q_1)$
is a representation of $Q^{\rm op}$, and this identifies
$\text{mod-}\K Q^{\rm op}$ with the opposite category of $\modKQ$.
Then Nakajima-style framings in $\modKQ$ correspond to our framings
in $\text{mod-}\K Q^{\rm op}$, and vice versa.
\item Let $Q_m$ be the quiver $Q$ with one vertex $v$ and $m$ edges
$v\ra v$, and consider the trivial stability condition
$(0,\R,\le)$ on $\modKQ_m$. Reineke \cite{Rein1} studied
`noncommutative Hilbert schemes'\index{Hilbert
scheme!noncommutative} $\smash{H^{(m)}_{d,e}}$ for $d,e\in\N$,
and determines their Poincar\'e polynomials. In our notation we
have $H^{(m)}_{d,e}=\M_{\stf\,Q_m}^{d,e}(0')$, and Reineke's
calculations and \eq{dt7eq28} yield a formula for
$NDT_{Q_m}^{d}(0')$. In \cite{Rein1} Reineke uses framings as in
Definition \ref{dt7def7}, not following Nakajima.
\item Let $Q$ be a quiver, and $\bs d,\bs e$ be dimension vectors.
Reineke \cite{Rein2} defined `framed quiver moduli' $\M_{\bs
d,\bs e}(Q)$. These are the same as Nakajima's moduli spaces
${\mathfrak R}_0(\bs d,\bs e)$ with trivial stability condition
$\theta=0$, and correspond to our moduli spaces
$\smash{\M_{\stf\,Q}^{\bs d,\bs e}(0')}$, except that the
framing uses maps~$\si_v:X_v\ra\K^{e(v)}$.

Reineke studies $\M_{\bs d,\bs e}(Q)$ for $Q$ without oriented
cycles. This yields the Euler characteristic of
$\smash{\M_{\stf\,Q}^{\bs d,\bs e}(0')}$, and so gives $NDT_{Q}^{\bs
d,\bs e}(0')$ in~\eq{dt7eq28}.
\item Engel and Reineke \cite{EnRe} study `smooth models of quiver
moduli' $M^\Th_{\bs d,\bs e}(Q)$, which agree with our
$\smash{\M_{\stf\,Q}^{\bs d,\bs e}(\mu')}$ for a slope stability
condition $(\mu,\R,\le)$ on $\modKQ$ defined using a map
$\Th:Q_0\ra\Q$, with framing as in Definition \ref{dt7def7}. They
give combinatorial formulae for the Poincar\'e polynomials of
$\smash{M^\Th_{\bs d,\bs e}(Q)}$, allowing us to compute
$NDT_{Q}^{\bs d,\bs e}(\mu')$ in~\eq{dt7eq28}.
\end{itemize}
Next we consider quivers with relations coming from a
superpotential:
\begin{itemize}
\setlength{\itemsep}{0pt}
\setlength{\parsep}{0pt}
\item Let $\modCQI$ come from a minimal superpotential $W$ over
$\C$ on a quiver $Q$. Fix a vertex $v\in Q_0$ of $Q$. Let
$(X,\rho)\in\modCQI$. We say that $(X,\rho)$ is {\it
cyclic},\index{quiver!cyclic representation} and {\it generated by a
vector\/} $x\in X_v$ if $X=\C Q/I\cdot x$. That is, there is no
subobject $(X',\rho')\subset(X,\rho)$ in $\modCQI$ with
$(X',\rho')\ne(X,\rho)$ and~$x\in X'_v\subseteq X_v$.

Szendr\H oi \cite[\S 1.2]{Szen} calls the pair
$\bigl((X,\rho),x\bigr)$ a {\it framed cyclic module\/} for
$(Q,I)$, and defines a moduli space $\M_{v,\bs d}$ of framed
cyclic modules with $\bdim(X,\rho)=\bs d$. Szendr\H oi defines
the {\it noncommutative Donaldson--Thomas invariant\/} $Z_{v,\bs
d}$ to be $\chi\bigl(\M_{v,\bs d},\nu_{\M_{v,\bs d}}\bigr)$. He
computes the $Z_{v,\bs d}$ in an example, the `noncommutative
conifold',\index{conifold!noncommutative} and shows the
generating function of the $Z_{v,\bs d}$ may be written
explicitly as an infinite product. In our notation $\M_{v,\bs
d}$ is $\M_{\stf\,Q,I}^{\bs d,\de_v}(0')$, where the framing
dimension vector $\bs e$ is $\de_v$, that is, $\de_v(w)=1$ for
$v=w$ and 0 for $v\ne w\in Q_0$, and the stability condition
$(\mu,\R,\le)$ on $\modCQI$ is zero. Thus by \eq{dt7eq27},
Szendr\H oi's invariants are our~$NDT_{Q,I}^{\bs d,\de_v}(0')$.

For the conifold,\index{conifold} Nagao and Nakajima \cite{NaNa}
prove relationships between Szendr\H oi's invariants,
Donaldson--Thomas invariants, and Pand\-hari\-pande--Thomas
invariants,\index{Pandharipande--Thomas invariants} via
wall-crossing for stability conditions on the derived category.
Nagao \cite{Naga} generalizes this to other toric Calabi--Yau
3-folds.
\item Let $G$ be a finite subgroup of $\SL(3,\C)$. Young and Bryan
\cite[\S A]{YoBr} discuss {\it Donaldson--Thomas invariants\/
$N^{\bs d}(\C^3/G)$ of the orbifold\/} $[\C^3/G]$. By this they
mean invariants counting ideal sheaves of compactly-supported
$G$-equivariant sheaves on $\C^3$. In two cases
$G=\Z_2\times\Z_2$ and $G=\Z_n$, they show that the generating
function of $N^{\bs d}(\C^3/G)$ can be written explicitly as an
infinite product, in a similar way to the conifold
case~\cite{Szen}.

As in Example \ref{dt7ex2}, Ginzburg defines a quiver $Q_G$ with
superpotential $W_G$ such that mod-$\C Q_G/I_G$ is 3-Calabi--Yau and
equivalent to the category of $G$-equivariant compactly-supported
coherent sheaves on $\C^3$. The definitions imply that Bryan and
Young's $N^{\bs d}(\C^3/G)$ is Szendr\H oi's $Z_{v,\bs d}$ for
$(Q_G,I_G)$, where the vertex $v$ in $Q_G$ corresponds to the
trivial representation $\C$ of $G$. Thus in our notation,~$N^{\bs
d}(\C^3/G)=NDT_{Q_G,I_G}^{\bs d,\de_v}(0')$.
\item Let $Q,W,I$ come from a consistent brane tiling, as in
Example \ref{dt7ex3}. Then Mozgovoy and Reineke \cite{MoRe} write
Szendr\H oi's invariants $Z_{v,\bs d}$ for $Q,I$ as combinatorial
sums, allowing evaluation of them on a computer.
\end{itemize}

In \S\ref{dt75}--\S\ref{dt76} we will use the results of
\cite{Rein1,Szen,YoBr} to write down the values of
$\smash{NDT_{Q,I}^{\bs d,\bs e}(\mu')}$ and $\smash{NDT_{Q}^{\bs
d,\bs e}(\mu')}$ in some of these examples. Then we will use Theorem
\ref{dt7thm5} below to compute $\bar{DT}{}^{\bs d}_{Q,I}(\mu)$ and
$\bar{DT}{}^{\bs d}_{Q}(\mu)$, and equation \eq{dt7eq21} to find
$\smash{\hat{DT}{}^{\bs d}_{Q,I}(\mu)}$ and~$\smash{\hat{DT}{}^{\bs
d}_{Q}(\mu)}$.

\begin{rem}{\bf(a)} Definitions \ref{dt7def7} and \ref{dt7def8} are
fairly direct analogues of Definitions \ref{dt5def3} and
\ref{dt5def5}, with $\coh(X)$ and $(\tau,T,\le)$ replaced by
$\modKQI$ and $(\mu,\R,\le)$. Note that the moduli spaces
$\M_{\stf\,Q,I}^{\bs d,\bs e}(\mu'),\M_{\stf\,Q}^{\bs d,\bs
e}(\mu')$ will in general not be proper. So we cannot define virtual
classes\index{virtual class} for $\smash{\M_{\stf\,Q,I}^{\bs d,\bs
e}(\mu')}$, $\M_{\stf\,Q}^{\bs d,\bs e}(\mu')$, and we have no
analogue of \eq{dt5eq15}; we are forced to define the invariants as
weighted Euler characteristics, following~\eq{dt5eq16}.
\smallskip

\noindent{\bf(b)} Here is why the framing data $\si$ for
$(X,\rho)\in\modKQI$ or $\modKQ$ in Definition \ref{dt7def7} is a
good analogue of the framing $s:\cO(-n)\ra E$ for $E\in\coh(X)$ when
$n\gg 0$ in Definition~\ref{dt5def3}.

In a well-behaved abelian category $\A$, an object $P\in\A$ is
called {\it projective\/} if $\Ext^i(P,E)=0$ for all $E\in\A$ and
$i>0$. Therefore $\dim\Hom(P,E)=\bar\chi([P],[E])$, where $\bar\chi$
is the Euler form of $\A$. If $X$ is a Calabi--Yau 3-fold, there
will generally be no nonzero projectives in $\coh(X)$. However, for
any bounded family $\cal F$ of sheaves in $\coh(X)$, for $n\gg 0$ we
have $\Ext^i(\cO(-n),E)=0$ for all $E$ in $\cal F$ and $i>0$. Thus
$\cO(-n)$ for $n\gg 0$ {\it acts like\/} a projective object in
$\coh(X)$, and this is what is important in \S\ref{dt54}. Thus, a
good generalization of stable pairs in $\coh(X)$ to an abelian
category $\A$ is to consider morphisms $s:P\ra E$ in $\A$, where $P$
is some fixed projective object in $\A$, and~$E\in\A$.

Now when $Q$ has oriented cycles, $\modKQI$ or $\modKQ$ (which
consist of {\it finite-dimensional\/} representations) generally do
not contain enough projective objects for this to be a good
definition. However, if we allow {\it infinite-dimensional\/}
representations $P$ of $\K Q/I$ or $\K Q$, then we can define
projective representations.\index{quiver!projective representation} Let
$\bs e$ be a dimension vector, and define
\begin{equation*}
\ts P^{\bs e}=\bigop_{v\in Q_0} \bigl((\K Q/I)\cdot i_v\bigr)\ot
\K^{\bs e(v)}\quad\text{or}\quad P^{\bs e}=\bigop_{v\in Q_0}
\bigl(\K Q\cdot i_v\bigr)\ot \K^{\bs e(v)},
\end{equation*}
where the idempotent $i_v$ in the algebra $\K Q/I$ or $\K Q$ is the
path of length zero at $v$, so that $\K Q\cdot i_v$ has basis the
set of oriented paths in $Q$ starting at $v$.

Then $P^{\bs e}$ is a left representation of $\K Q/I$ or $\K Q$,
which may be infinite-dimensional if $Q$ has oriented cycles. In the
abelian category of possibly infinite-dimensional representations of
$\K Q/I$ or $\K Q$, it is projective. If $(X,\rho)$ lies in
$\modKQI$ or $\modKQ$ with $\bdim(X,\rho)=\bs d$ then
\e
\ts\Hom\bigl(P^{\bs e},(X,\rho)\bigr)\cong \bigop_{v\in
Q_0}\Hom(\K^{\bs e(v)},X_v),
\label{dt7eq29}
\e
so that $\dim\Hom\bigl(P^{\bs e},(X,\rho)\bigr)=\sum_{v\in Q_0}\bs
e(v)\bs d(v)$. (Note that this is not $\bar\chi(\bs e,\bs d)$.)
Equation \eq{dt7eq29} implies that morphisms of representations
$P^{\bs e}\ra(X,\rho)$ are the same as choices of $\si$ in
Definition \ref{dt7def7}. Thus, framed representations
$(X,\rho,\si)$ in Definition \ref{dt7def7} are equivalent to
morphisms $\si:P^{\bs e}\ra(X,\rho)$, where $P^{\bs e}$ is a fixed
projective. The comparison with $s:\cO(-n)\ra E$ in \S\ref{dt54} is
clear.
\smallskip

\noindent{\bf(c)} Here is another interpretation of framed
representations, following Reineke \cite[\S 3.1]{Rein2}. Given
$(Q,I)$ or $Q$, $\bs d,\bs e$ as above, define another quiver $\ti
Q$ to be $Q$ together with an extra vertex $\iy$, so that $\ti
Q_0=Q_0\amalg\{\iy\}$, and with $\bs e(v)$ extra edges $\iy\ra v$
for each $v\in Q_0$. Let the relations $\ti I$ for $\ti Q$ be the
lift of $I$ to $\K\ti Q$, with no extra relations. Define $\bs{\ti
d}:\ti Q_0\ra\Z_{\sst\ge 0}$ by $\bs{\ti d}(v)=\bs d(v)$ for $v\in
Q_0$ and $\bs{\ti d}(\iy)=1$. It is then easy to show that framed
representations of $(Q,I)$ or $Q$ of type $(\bs d,\bs e)$ correspond
naturally to representations of $(\ti Q,\ti I)$ or $\ti Q$ of type
$\bs{\ti d}$, and one can define a stability condition
$(\ti\mu,\R,\le)$ on $\text{mod-}\K\ti Q/\ti I$ or $\text{mod-}\K\ti
Q$ such that $\mu'$-stable framed representations correspond to
$\ti\mu$-stable representations.
\label{dt7rem3}
\end{rem}

We now prove the analogue of Theorem \ref{dt5thm10} for quivers.

\begin{thm} Suppose\/ $Q$ is a quiver with relations\/ $I$ coming
from a minimal superpotential\/ $W$ on $Q$ over $\C$. Let\/
$(\mu,\R,\le)$ be a slope stability condition on\/ $\modCQI,$ as in
Example {\rm\ref{dt7ex1},} and\/ $\bar\chi$ be as in \eq{dt7eq9}.
Then for all\/ $\bs d,\bs e$ in\/ $C(\modCQI)=\Z_{\sst\ge
0}^{Q_0}\sm\{0\}\subset \Z^{Q_0},$ we have
\e
NDT^{\bs d,\bs e}_{Q,I}(\mu')=\!\!\!\!\!\!\!\!\!\!\!\!\!\!\!
\sum_{\begin{subarray}{l} \bs d_1,\ldots,\bs d_l\in
C(\modCQI),\\ l\ge 1:\; \bs d_1
+\cdots+\bs d_l=\bs d,\\
\mu(\bs d_i)=\mu(\bs d),\text{ all\/ $i$}
\end{subarray} \!\!\!\!\!\!\!\!\! }
\begin{aligned}[t] \frac{(-1)^l}{l!} &\prod_{i=1}^{l}\bigl[
(-1)^{\bs e\cdot\bs d_i-\bar\chi(\bs d_1+\cdots+\bs
d_{i-1},\bs d_i)} \\
&\bigl(\bs e\cdot\bs d_i-\bar\chi(\bs d_1\!+\!\cdots\!+\!\bs
d_{i-1},\bs d_i)\bigr)\bar{DT}{}^{\bs
d_i}_{Q,I}(\mu)\bigr],\!\!\!\!\!\!\!\!\!\!\!\!\!\!
\end{aligned}
\label{dt7eq30}
\e
with\/ $\bs e\cdot\bs d_i=\sum_{v\in Q_0}\bs e(v)\bs d_i(v),$ and\/
$\bar{DT}{}^{\bs d_i}_{Q,I}(\mu),NDT^{\bs d,\bs e}_{Q,I}(\mu')$ as
in Definitions\/ {\rm\ref{dt7def6}, \rm\ref{dt7def8}}. When $W\equiv
0,$ the same equation holds for $\smash{NDT^{\bs d,\bs
e}_Q(\mu'),\bar{DT}{}^{\bs d}_Q(\mu)}$.
\label{dt7thm5}
\end{thm}

\begin{proof} The proof follows that of Theorem \ref{dt5thm10} in
\S\ref{dt13} closely. We need to explain the analogues of the
abelian categories $\A_p,\B_p$ in \S\ref{dt131}. When $\mu\equiv 0$,
we have $\A_p=\modCQI$ and $\B_p=\text{mod-}\C\ti Q/\ti I$, where
$(\ti Q,\ti I)$ is as in Remark \ref{dt7rem3}(c). For general $\mu$,
with $\bs d$ fixed, we take $\A_p$ to be the abelian subcategory of
objects $(X,\rho)$ in $\modCQI$ with $\mu\bigl([(X,\rho)]\bigr)
=\mu(\bs d)$, together with $0$, and $\B_p$ to be the abelian
subcategory of objects $(\ti X,\ti\rho)$ in $\text{mod-}\C\ti Q/\ti
I$ with $\ti\mu\bigl([(\ti X,\ti\rho)]\bigr)=\ti\mu(\bs{\ti d})$,
together with $0$, for $\bs{\ti d},(\ti\mu,\R,\le)$ as in
Remark~\ref{dt7rem3}(c).

Then we have $K(\A_p)\subseteq K(\modCQI)=\Z^{Q_0}$, and
$K(\B_p)=K(\A_p)\op\Z$, as in \S\ref{dt131}, and
$\bar\chi^{\A_p}=\bar\chi\vert_{K(\A_p)}$. The analogue of
\eq{dt13eq5} giving the `Euler form' $\bar\chi{}^{\smash{\B_p}}$ on
$K(\B_p)$ is
\e
\bar\chi{}^{\smash{\B_p}}\bigl((\bs d,k),(\bs d',k')\bigr)
=\bar\chi(\bs d,\bs d')-k\,\bs e\cdot\bs d'+k'\,\bs e\cdot\bs d.
\label{dt7eq31}
\e
The analogue of Proposition \ref{dt13prop1} then holds for all pairs
of elements in $\B_p$, without the restrictions that $\dim V+\dim
W\le 1$ and $k,l\le N$. The point here is that $\cO_X(-n)$ is not
actually a projective object in $\coh(X)$ for fixed $n\gg 0$, so we
have to restrict to a bounded part of the category $\A_p$ in which
it acts as a projective. But as in Remark \ref{dt7rem3}(b), in the
quiver case we are in effect dealing with genuine projectives, so no
boundedness assumptions are necessary.

The rest of the proof in \S\ref{dt13} goes through without
significant changes. Using \eq{dt7eq31} rather than \eq{dt13eq5}
eventually yields equation~\eq{dt7eq30}.
\end{proof}

The proof of Proposition \ref{dt5prop3} now yields:

\begin{cor} In the situation above, suppose $c\in\R$ with\/
$\bar\chi(\bs d,\bs d')=0$ for all\/ $\bs d,\bs d'$ in\/
$C(\modCQI)$ with\/ $\mu(\bs d)=\mu(\bs d')=c$. Then for any $\bs e$
in $C(\modCQI),$ in formal power series we have
\e
1+\!\!\sum_{\bs d\in C(\modCQI):\;\mu(\bs d)=c
\!\!\!\!\!\!\!\!\!\!\!\!\!\!\!\!\!\!\!\!\!\!\!\!\!\!
\!\!\!\!\!\!\!\!\!\!\!\!\!\!\!\!\!\!\!\!\!\!} NDT^{\bs d,\bs
e}_{Q,I}(\mu')q^{\bs d}=
\exp\raisebox{-4pt}{\begin{Large}$\displaystyle\Bigl[$\end{Large}}
-\!\!\!\!\sum_{\bs d\in C(\modCQI):\;\mu(\bs d)=c
\!\!\!\!\!\!\!\!\!\!\!\!\!\!\!\!\!\!\!\!\!\!\!\!\!\!\!\!\!\!\!\!\!\!
\!\!\!\!\!\!\!\!\!\!\!\!\!\!\!\!} (-1)^{\bs e\cdot\bs d}(\bs
e\cdot\bs d)\bar{DT}{}^{\bs d}_{Q,I}(\mu)q^{\bs d}
\raisebox{-4pt}{\begin{Large}$\displaystyle\Bigr]$\end{Large}},
\label{dt7eq32}
\e
where $q^{\bs d}$ for $\bs d\in C(\modCQI)$ are formal symbols
satisfying\/ $q^{\bs d}\cdot q^{\bs d'}=q^{\bs d+\bs d'}$. When
$W\equiv 0,$ the same equation holds for $NDT^{\bs d,\bs
e}_Q(\mu'),\bar{DT}{}^{\bs d}_Q(\mu)$.
\label{dt7cor2}
\end{cor}

\begin{rem} In the coherent sheaf case of \S\ref{dt5}--\S\ref{dt6},
we regarded the generalized Donaldson--Thomas invariants
$\bar{DT}{}^\al(\tau)$, or equivalently the BPS invariants
$\hat{DT}{}^\al(\tau)$, as being the central objects of interest.
The pair invariants $PI^{\al,n}(\tau')$ appeared as auxiliary
invariants, not of that much interest in themselves, but useful for
computing the $\bar{DT}{}^\al(\tau),\hat{DT}{}^\al(\tau)$ and
proving their deformation-invariance.

In contrast, in the quiver literature to date, so far as the authors
know, the invariants $\bar{DT}{}^{\bs d}_{Q,I}(\mu),\ab
\bar{DT}{}^{\bs d}_Q(\mu)$ and $\hat{DT}{}^{\bs
d}_{Q,I}(\mu),\hat{DT}{}^{\bs d}_Q(\mu)$ have not been seriously
considered even in the stable$\,=\,$semistable case, and the
analogues $\smash{NDT_{Q,I}^{\bs d,\bs e}(\mu')}$, $NDT_{Q}^{\bs
d,\bs e}(\mu')$ of pair invariants $PI^{\al,n}(\tau')$ have been the
central object of study.

We wish to argue that the invariants $\smash{\bar{DT}{}^{\bs
d}_{Q,I}(\mu),\ldots,\hat{DT}{}^{\bs d}_Q(\mu)}$ should actually be
regarded as more fundamental and more interesting than the
$\smash{NDT_{Q,I}^{\bs d,\bs e}(\mu')}$, $NDT_{Q}^{\bs d,\bs
e}(\mu')$. We offer two reasons for this. Firstly, as Theorem
\ref{dt7thm5} shows, the $NDT_{\smash{Q,I}}^{\bs d,\bs
e}(\mu'),NDT_{\smash{Q}}^{\bs d,\bs e}(\mu')$ can be written in
terms of the $\bar{DT}{}^{\bs d}_{\smash{Q,I}}(\mu),\bar{DT}{}^{\bs
d}_{\smash{Q}}(\mu)$, and hence by \eq{dt7eq21} in terms of the
$\hat{DT}{}^{\bs d}_{Q,I}(\mu),\hat{DT}{}^{\bs d}_Q(\mu)$, so the
pair invariants contain no more information. The $\bar{DT}{}^{\bs
d}_{Q,I}(\mu),\bar{DT}{}^{\bs d}_Q(\mu)$ are {\it simpler\/} than
the $\smash{NDT_{Q,I}^{\bs d,\bs e}(\mu'),NDT_{Q}^{\bs d,\bs
e}(\mu')}$ as they depend only on $\bs d$ rather than on $\bs d,\bs
e$, and in examples in \S\ref{dt75}--\S\ref{dt76} we will see that
the values of $\smash{\bar{DT}{}^{\bs d}_{Q,I}(\mu), \bar{DT}{}^{\bs
d}_Q(\mu)}$ and especially of $\hat{DT}{}^{\bs
d}_{Q,I}(\mu),\hat{DT}{}^{\bs d}_Q(\mu)$ may be much simpler and
more illuminating than the values of the~$NDT_{Q,I}^{\bs d,\bs
e}(\mu'),NDT_{Q}^{\bs d,\bs e}(\mu')$.

Secondly, the case in \cite{Rein1,Szen,YoBr} for regarding
$NDT_{Q,I}^{\bs d,\bs e}(\mu'),NDT_{Q}^{\bs d,\bs e}(\mu')$ as
analogues of rank 1 Donaldson--Thomas invariants counting ideal
sheaves, that is, of counting surjective morphisms $s:\cO_X\ra E$,
is in some ways misleading. The $NDT_{Q,I}^{\bs d,\bs
e}(\mu'),NDT_{Q}^{\bs d,\bs e}(\mu')$ are closer to our invariants
$PI^{\al,n}(\tau')$ counting $s:\cO_X(-n)\ra E$ for $n\gg 0$ than
they are to counting morphisms $s:\cO_X\ra E$. The difference is
that $\cO_X$ is {\it not a projective object\/} in $\coh(X)$, but
$\cO_X(-n)$ for $n\gg 0$ is effectively a projective object in
$\coh(X)$, as in (b) above.

To see the difference between counting morphisms $s:\cO_X\ra E$ and
counting morphisms $s:\cO_X(-n)\ra E$ for $n\gg 0$, consider the
case where $E$ is a dimension 1 sheaf on a Calabi--Yau 3-fold $X$.
Then the MNOP Conjecture\index{MNOP Conjecture} \cite{MNOP1,MNOP2}
predicts that invariants $DT^{(1,0,\be,m)}(\tau)$ counting morphisms
$s:\cO_X\ra E$ encode the Gopakumar--Vafa invariants $GV_g(\be)$ of
$X$ for all genera $g\ge 0$. But Theorem \ref{dt5thm10} and
Conjecture \ref{dt6conj3} in \S\ref{dt64} imply that invariants
$PI^{(0,0,\be,m),n}(\tau)$ counting morphisms $s:\cO_X(-n)\ra E$ for
$n\gg 0$ encode only the Gopakumar--Vafa
invariants\index{Gopakumar--Vafa invariants} $GV_0(\be)$ of $X$ for
genus~$g=0$.

The point is that since $\cO_X$ is not a projective, counting
morphisms $s:\cO_X\ra E$ gives you information not just about
counting sheaves $E$, but also extra information about how $\cO_X$
and $E$ interact. But as $\cO_X(-n)$ for $n\gg 0$ is effectively a
projective, counting morphisms $s:\cO_X(-n)\ra E$ gives you
information only about counting sheaves $E$, so we might as well
just count sheaves $E$ directly using (generalized)
Donaldson--Thomas invariants.
\label{dt7rem4}
\end{rem}\index{stable pair invariants $PI^{\al,n}(\tau')$!analogue for
quivers|)}\index{Donaldson--Thomas invariants!noncommutative|)}

\subsection{Computing $\bar{DT}{}^{\bs d}_{Q,I}(\mu),
\hat{DT}{}^{\bs d}_{Q,I}(\mu)$ in examples}
\label{dt75}\index{Donaldson--Thomas invariants!computation in
examples|(}

We now use calculations of noncommutative Donaldson--Thomas
invariants in examples by Szendr\H oi \cite{Szen} and Young and
Bryan \cite{YoBr} to write down generating functions for
$\smash{NDT^{\bs d,\bs e}_{Q,I}(\mu')}$, and then apply \eq{dt7eq32}
to deduce values of $\smash{\bar{DT}{}^{\bs d}_{Q,I}(\mu)}$, and
\eq{dt7eq22} to deduce values of $\smash{\hat{DT}{}^{\bs
d}_{Q,I}(\mu)}$. These values of $\hat{DT}{}^{\bs d}_{Q,I}(\mu)$
turn out to be much simpler than those of the $\smash{NDT^{\bs d,\bs
e}_{Q,I}(\mu')}$, and explain the MacMahon function\index{MacMahon
function} product form of the generating functions in
\cite{Szen,YoBr}. The translation between the notation of
\cite{Szen,YoBr} and our notation was explained after Definition
\ref{dt7def8}, and we assume it below.

\subsubsection{Coherent sheaves on $\C^3$}
\label{dt751}

As in Szendr\H oi \cite[\S 1.5]{Szen}, let $Q=(Q_0,Q_1,\ab h,\ab t)$
have one vertex $Q_0=\{v\}$ and three edges $Q_1=\{e_1,e_2,e_3\}$,
so that $h(e_j)=t(e_j)=v$ for $j=1,2,3$. Define a superpotential $W$
on $Q$ by $W=e_1e_2e_3-e_1e_3e_2$. Then the ideal $I$ in $\C Q$ is
generated by $e_2e_3-e_3e_2$, $e_3e_1-e_1e_3$, $e_1e_2-e_2e_1$, and
is $[\C Q,\C Q]$, so $\C Q/I$ is the commutative polynomial algebra
$\C[e_1,e_2,e_3]$, the coordinate ring of the noncompact Calabi--Yau
3-fold $\C^3$, and $\modCQI$ is isomorphic to the abelian
category~$\coh_\cs(\C^3)$.

We have $C(\modCQI)=\N$, so taking $\bs d=d\in\N$, $\bs e=1$, and
$(\mu,\R,\le)$ to be the trivial stability condition $(0,\R,\le)$ on
$\modCQI$, we form invariants $NDT^{d,1}_{Q,I}(0')\in\Z$. Then as in
\cite[\S 1.5]{Szen}, by torus localization one can show that
\e
\ts 1+\sum_{d\ge 1} NDT^{d,1}_{Q,I}(0')q^d=\prod_{k\ge
1}\bigl(1-(-q)^k)\bigr)^{-k},
\label{dt7eq33}
\e
which is Theorem \ref{dt6thm1} for the noncompact Calabi--Yau 3-fold
$X=\C^3$. Taking logs of \eq{dt7eq33} and using \eq{dt7eq32}, which
holds as $\bar\chi\equiv 0$, gives
\begin{equation*}
-\sum_{d\ge 1}(-1)^dd\,\bar{DT}{}^{d}_{Q,I}(0)q^d=
\sum_{k\ge 1} (-k)\log\bigl(1\!-\!(-q)^k\bigr)=
\sum_{k,l\ge 1}\frac{k}{l}(-q)^{kl}.
\end{equation*}
Equating coefficients of $q^d$ yields
\begin{equation*}
\bar{DT}{}^d_{Q,I}(0)=-\sum_{l\ge 1, \; l \mid d}\frac{1}{l^2}.
\end{equation*}
So from \eq{dt7eq22} we deduce that
\e
\hat{DT}{}^d_{Q,I}(0)=-1,\quad\text{all $d\ge 1$.}
\label{dt7eq34}
\e
This is \eq{dt6eq20} for the noncompact Calabi--Yau 3-fold $X=\C^3$,
as in~\S\ref{dt67}.\index{Calabi--Yau 3-fold!noncompact}

\subsubsection{The noncommutative conifold, following Szendr\H oi}
\label{dt752}
\index{conifold!noncommutative|(}

As in Szen\-dr\H oi \cite[\S 2.1]{Szen}, let $Q=(Q_0,Q_1,h,t)$ have
two vertices $Q_0=\{v_0,v_1\}$ and edges $e_1,e_2:v_0\ra v_1$ and
$f_1,f_2:v_1\ra v_0$, as below:
\e
\xymatrix@C=40pt{ \mathop{\bu} \limits_{v_0} \ar@/^/@<2ex>[r]_{e_2}
\ar@/^/@<3ex>[r]^{e_1} & \mathop{\bu}\limits_{v_1}.
\ar@/^/@<0ex>[l]_{f_1} \ar@/^/@<1ex>[l]^{f_2} }
\label{dt7eq35}
\e
Define a superpotential $W$ on $Q$ by $W=e_1f_1e_2f_2-e_1f_2e_2f_1$,
and let $I$ be the associated relations. Then $\modCQI$ is a
3-Calabi--Yau category.\index{abelian category!3-Calabi--Yau} Theorem
\ref{dt7thm1} shows that the Euler form\index{Euler form} $\bar\chi$ on
$\modCQI$ is zero.

We have equivalences of derived categories\index{triangulated
category!equivalence}
\e
D^b(\modCQI)\sim D^b(\coh_\cs(X))\sim D^b(\coh_\cs(X_+)),
\label{dt7eq36}
\e
where $\pi:X\ra Y$ and $\pi_+:X_+\ra Y$ are the two crepant
resolutions of the conifold $Y=\bigl\{(z_1,z_2,z_3,z_4)\in
\C^4:z_1^2+\cdots+z_4^2=0\bigr\}$, and $X,X_+$ are related by a
flop. Here $X,X_+$ are regarded as `commutative' crepant resolutions
of $Y$, and $\modCQI$ as a `noncommutative' resolution of $Y$, in
the sense that $\modCQI$ can be regarded as the coherent sheaves on
the `noncommutative scheme' $\Spec(\C Q/I)$ constructed from the
noncommutative $\C$-algebra~$\C Q/I$.

Szendr\H oi \cite[Th.~2.7.1]{Szen} computed the noncommutative
Donaldson--Thomas invariants $NDT^{\smash{\bs
d,\de_{v_0}}}_{Q,I}(0')$ for $\modCQI$ with $\bs e=\de_{v_0}$, as
combinatorial sums, and using work of Young \cite{Youn} wrote the
generating function of the $NDT^{\smash{\bs d,\de_{v_0}}}_{Q,I}(0')$
as a product \cite[Th.~2.7.2]{Szen}, giving
\e
\begin{split}
1&+\sum_{\bs d\in C(\modCQI)}NDT^{\bs d,\de_{v_0}}_{Q,I}(0')q_0^{\bs
d(v_0)}q_1^{\bs d(v_1)}\\
&=\prod_{k\ge 1}\bigl(1-(-q_0q_1)^k)\bigr)^{-2k}\bigl(1-(-q_0)^k
q_1^{k-1}\bigr)^k\bigl(1-(-q_0)^kq_1^{k+1}\bigr)^k.
\end{split}
\label{dt7eq37}
\e
Taking logs of \eq{dt7eq37} and using \eq{dt7eq32} gives
\ea
&-\sum_{\bs d\in C(\modCQI)} (-1)^{\bs d(v_0)}\bs
d(v_0)\bar{DT}{}^{\bs d}_{Q,I}(0) q_0^{\bs d(v_0)}q_1^{\bs d(v_1)}
\nonumber\\
&=\sum_{k\ge 1}\begin{aligned}[t]\bigl[-2k\log\bigl(1-(-q_0q_1)^k)
+k\log\bigl(1-(-q_0)^kq_1^{k-1}\bigr)&\\
+k\log\bigl(1-(-q_0)^kq_1^{k+1}\bigr)&\bigr]
\end{aligned}
\label{dt7eq38}\\
&=\sum_{k,l\ge 1}\Bigl[\frac{2k}{l}\,(-q_0q_1)^{kl}
-\frac{k}{l}\,(-q_0)^{kl}q_1^{(k-1)l}
-\frac{k}{l}\,(-q_0)^{kl}q_1^{(k+1)l}\Bigr]
\nonumber\\
&=-\sum_{k,l\ge
1}(-1)^{kl}kl\cdot\Bigl[-\frac{2}{l^2}\,q_0^{kl}q_1^{kl}
+\frac{1}{l^2}\,q_0^{kl}q_1^{(k-1)l}+\frac{1}{l^2}\,q_0^{kl}
q_1^{(k+1)l}\Bigr].
\nonumber
\ea

Writing $\bs d\!=\!(d_0,d_1)$ with $d_j\!=\!\bs d(v_j)$ and equating
coefficients of $q_0^{d_0}q_1^{d_1}$ yields
\e
\bar{DT}{}^{(d_0,d_1)}_{Q,I}(0)=\begin{cases} \displaystyle
-2\sum_{l\ge 1, \; l \mid d}\frac{1}{l^2},\!\! &
d_0=d_1=d\ge 1, \\
\displaystyle \frac{1}{l^2}, & d_0=kl, \; d_1=(k-1)l,\;
k,l\ge 1, \\[7pt]
\displaystyle \frac{1}{l^2}, & d_0=kl, \; d_1=(k+1)l,\;
k\ge 0,\; l\ge 1,\!\! \\
0, & \text{otherwise.}
\end{cases}
\label{dt7eq39}
\e
Actually we have cheated a bit here: because of the factor $\bs
d(v_0)$ on the first line, equation \eq{dt7eq38} only determines
$\smash{\bar{DT}{}^{(d_0,d_1)}_{Q,I}(0)}$ when $d_0>0$. But by
symmetry between $v_0$ and $v_1$ in \eq{dt7eq35} we have
$\bar{DT}{}^{(d_0,d_1)}_{Q,I}(0)=\bar{DT}{}^{(d_1,d_0)}_{Q,I}(0)$,
so we can deduce the answer for $d_0=0$, $d_1>0$ from that for
$d_0>0$, $d_1=0$. This is why we included the case $k=0$, $l\ge 1$
on the third line of~\eq{dt7eq39}.

Combining \eq{dt7eq22} and \eq{dt7eq39} we see that
\e
\hat{DT}{}^{(d_0,d_1)}_{Q,I}(0)=\begin{cases} -2, &
(d_0,d_1)=(k,k),\; k\ge 1,\\
\phantom{-}1, & (d_0,d_1)=(k,k-1),\; k\ge 1,\\
\phantom{-}1, &  (d_0,d_1)=(k-1,k),\; k\ge 1,\\
\phantom{-}0, & \text{otherwise.}
\end{cases}
\label{dt7eq40}
\e
Note that the values of the
$\smash{\hat{DT}{}^{(d_0,d_1)}_{Q,I}(0)}$ in \eq{dt7eq40} lie in
$\Z$, as in Conjecture \ref{dt7conj1}, and are far simpler than
those of the $NDT^{\bs d,\de_{v_0}}_{Q,I}(0')$ in \eq{dt7eq37}.
Also, \eq{dt7eq40} restores the symmetry between $v_0,v_1$ in
\eq{dt7eq35}, which is broken in \eq{dt7eq37} by choosing the vertex
$v_0$ in~$\bs e=\de_{v_0}$.

As $\bar\chi\equiv 0$ on $\modCQI$, by Corollary \ref{dt7cor1}
equations \eq{dt7eq39}--\eq{dt7eq40} also give
$\bar{DT}{}^{(d_0,d_1)}_{Q,I}(\mu),\hat{DT}{}^{(d_0,d_1)}_{Q,I}(\mu)$
for any stability condition $(\mu,\R,\le)$ on $\modCQI$. It should
not be difficult to prove \eq{dt7eq39}--\eq{dt7eq40} directly,
without going via pair invariants. If $(\mu,\R,\le)$ is a nontrivial
slope stability condition on $\modCQI$, then Nagao and Nakajima
\cite[\S 3.2]{NaNa} prove that every $\mu$-stable object in
$\modCQI$ lies in class $(k,k)$ or $(k,k-1)$ or $(k-1,k)$ in
$K(\modCQI)$ for $k\ge 1$, and the $\mu$-stable objects in classes
$(k,k-1)$ and $(k-1,k)$ are unique up to isomorphism. The bottom
three lines of \eq{dt7eq40} can be deduced from this.

Szendr\H oi \cite[\S 2.9]{Szen} relates noncommutative
Donaldson--Thomas invariants cou\-nting (framed) objects in
$\modCQI$ with Donaldson--Thomas invariants counting (ideal sheaves
of) objects in $\coh_\cs(X)$, $\coh_\cs(X_+)$, under the
equivalences \eq{dt7eq36}. He ends up with the generating functions
\ea
Z_{\modCQI}(q,z)&=\prod_{k\ge
1}\bigl(1-(-q)^k)\bigr)^{-2k}\bigl(1+(-q)^kz\bigr)^k
\bigl(1+(-q)^kz^{-1}\bigr)^k,
\label{dt7eq41}\\
Z_{\coh_\cs(X)}(q,z)&=\prod_{k\ge
1}\bigl(1-(-q)^k)\bigr)^{-2k}\bigl(1+(-q)^kz\bigr)^k,
\label{dt7eq42}\\
Z_{\coh_\cs(X_+)}(q,z)&=\prod_{k\ge 1}\bigl(1-(-q)^k)\bigr)^{-2k}
\bigl(1+(-q)^kz^{-1}\bigr)^k,
\label{dt7eq43}
\ea
where \eq{dt7eq41} is \eq{dt7eq37} with the variable change
$q=q_0q_1$, $z=q_1$, and \eq{dt7eq42}--\eq{dt7eq43} encode counting
invariants in $\coh_\cs(X)$ and $\coh_\cs(X_+)$ in a similar way.
Nagao and Nakajima \cite{NaNa} explain the relationship between
\eq{dt7eq41}--\eq{dt7eq43} in terms of stability conditions and
wall-crossing on the triangulated
category~\eq{dt7eq36}.\index{triangulated category!wall-crossing}

We can offer a much simpler explanation for the relationship between
our invariants $\hat{DT}{}^{(d_0,d_1)}_{Q,I}(\mu)$ counting
(unframed) objects in $\modCQI$, and the analogous invariants
counting objects (not ideal sheaves) in $\coh_\cs(X),\coh_\cs(X_+)$.
We base it on the following conjecture:

\begin{conj} Let\/ $\cal T$ be a $\C$-linear\/ $3$-Calabi--Yau
triangulated category,\index{triangulated category!3-Calabi--Yau|(} and
abelian categories\/ $\A,\B\subset{\cal T}$ be the hearts of
t-structures on $\cal T$. Suppose the Euler form $\bar\chi$ of\/
$\cal T$ is zero. Let\/ $K({\cal T})$ be a quotient of\/ $K_0({\cal
T}),$ and\/ $K(\A),K(\B)$ the corresponding quotients of\/
$K_0(\A),K_0(\B)$ under\/~$K_0(\A)\cong K_0({\cal T})\cong K_0(\B)$.

Suppose we can define Donaldson--Thomas type invariants
$\bar{DT}{}_\A^\al(\tau),\hat{DT}{}_\A^\al(\tau)$ counting objects
in $\A,$ for $\al\in K(\A)$ and\/ $(\tau,T,\le)$ a stability
condition on $\A,$ and\/
$\bar{DT}{}_\B^\be(\ti\tau),\hat{DT}{}_\B^\be(\ti\tau)$ counting
objects in $\B,$ for $\be\in K(\B)$ and\/ $(\ti\tau,\ti T,\le)$ a
stability condition on $\B,$ as for $\A=\coh(X)$ in
{\rm\S\ref{dt5}--\S\ref{dt6}} and\/ $\A=\modCQI$
in~{\rm\S\ref{dt73}}.

Define $\bar{DT}_\A,\hat{DT}_\A:K(\A)\ra\Q$ and\/
$\bar{DT}_\B,\hat{DT}_\B:K(\B)\ra\Q$ by
\begin{equation*}
\bar{DT}_\A(\al)=\begin{cases} \bar{DT}{}_\A^\al(\tau), & \al\in C(\A), \\
\bar{DT}{}_\A^{-\al}(\tau), & -\al\in C(\A), \\ 0, &
\text{otherwise,}\end{cases} \;\> \hat{DT}_\A(\al)=\begin{cases}
\hat{DT}{}_\A^\al(\tau), & \al\in C(\A), \\
\hat{DT}{}_\A^{-\al}(\tau), & -\al\in C(\A), \\ 0, &
\text{otherwise,}\end{cases}
\end{equation*}
and similarly for $\bar{DT}_\B,\hat{DT}_\B$. Then (possibly under
some extra conditions), under\/ $K(\A)\cong K(\B)$ we have\/
$\bar{DT}_\A\equiv \bar{DT}_\B,$ or
equivalently\/~$\hat{DT}_\A\equiv\hat{DT}_\B$.
\label{dt7conj2}
\end{conj}

Here is why we believe this. We expect that there should be some
extension of Donaldson--Thomas theory from abelian categories to
3-Calabi--Yau triangulated categories $\cal T$, in the style of
Kontsevich--Soibelman \cite{KoSo1}, using Bridgeland stability
conditions on triangulated categories \cite{Brid1}. Invariants
$\bar{DT}{}_\A^\al(\tau)$ for an abelian category $\A$ embedded as
the heart of a t-structure\index{t-structure} in $\cal T$ should be a
special case of triangulated category invariants on $\cal T$, in
which the Bridgeland stability condition $(Z,{\cal P})$ on $D^b(\A)$
is constructed from $(\tau,T,\le)$ on $\A$. If $\A$ is a
3-Calabi--Yau abelian category\index{abelian category!3-Calabi--Yau}
then we take~${\cal T}=D^b(\A)$.\index{triangulated
category!3-Calabi--Yau|)}

Now the $Z$-(semi)stable objects in $\cal T$ should be shifts $E[k]$
for $k\in\Z$ and $E\in\A$ $\tau$-(semistable). The class $[E[k]]$ of
$E[k]$ in $K({\cal T})\cong K(\A)$ is $(-1)^k[E]$. Thus, invariants
$\bar{DT}{}^\al_{\cal T}(Z)$ for $\al\in K(\A)$ should have
contributions $\bar{DT}{}_\A^\al(\tau)$ for $\al\in C(\A)$ counting
$E[2k]$ for $E\in\A$ $\tau$-(semi)stable and $k\in\Z$, and
$\bar{DT}{}_\A^{-\al}(\tau)$ for $\al\in C(\A)$ counting $E[2k+1]$
for $E\in\A$ $\tau$-(semi)stable and $k\in\Z$. This explains the
definitions of $\bar{DT}_\A,\hat{DT}_\A:K(\A)\ra\Q$ above.

As in Corollary \ref{dt7cor1}, if the form $\bar\chi$ on $\A$ (which
is the Euler form of $\cal T$) is zero then \eq{dt5eq14},
\eq{dt7eq23} imply that invariants $\bar{DT}{}_\A^\al(\tau),
\hat{DT}{}_\A^\al(\tau)$ are independent of the choice of stability
condition $(\tau,T,\le)$ on $\A$, since the changes when we cross a
wall always include factors $\bar\chi(\be,\ga)$. The point of
Conjecture \ref{dt7conj2} is that we expect this to be true for
triangulated categories too, so computing invariants in $\cal T$
either in $\A$ or $\B$ should give the same answers,
i.e.~$\bar{DT}_\A\equiv\bar{DT}_\B$.

In the noncommutative conifold example above, from \eq{dt7eq40} we
have
\e
\hat{DT}_{Q,I}(d_0,d_1)=\begin{cases} -2, &
(d_0,d_1)=(k,k),\; 0\ne k\in\Z,\\
\phantom{-}1, & (d_0,d_1)=(k,k-1),\; k\in\Z,\\
\phantom{-}1, &  (d_0,d_1)=(k-1,k),\; k\in\Z,\\
\phantom{-}0, & \text{otherwise.}
\end{cases}
\label{dt7eq44}
\e

The Donaldson--Thomas invariants for $\coh_\cs(X)\cong
\coh_\cs(X_+)$ were computed in Example \ref{dt6ex10}, and from
\eq{dt6eq40}--\eq{dt6eq42} we have
\e
\hat{DT}_{\coh_\cs(X)}(a_2,a_3)=\hat{DT}_{\coh_\cs(X_+)}(a_2,a_3)=
\begin{cases} -2, &
a_2\!=\!0,\; 0\!\ne\! a_3\!\in\!\Z, \\
\phantom{-}1, & a_2=\pm 1,\\
\phantom{-}0, & \text{otherwise.}
\end{cases}
\label{dt7eq45}
\e
As in Szendr\H oi \cite[\S 2.8--\S 2.9]{Szen}, the identification
$K(\modCQI)\!\ra\! K(\coh_\cs(X))$ induced by $D^b(\modCQI)\sim
D^b(\coh_\cs(X))$ in \eq{dt7eq36} is $(d_0,d_1)\mapsto
(-d_0+d_1,d_0)=(a_2,a_3)$, and under this identification we have
$\hat{DT}_{Q,I}\equiv\hat{DT}_{\coh_\cs(X)}$ by
\eq{dt7eq44}--\eq{dt7eq45}. Similarly, the identification
$K(\modCQI)\ra K(\coh_\cs(X_+))$ induced by $D^b(\modCQI)\sim
D^b(\coh_\cs(X_+))$ in \eq{dt7eq36} is $(d_0,d_1)\mapsto
(d_0-d_1,d_0)=(a_2,a_3)$, and again we have
$\hat{DT}_{Q,I}\equiv\hat{DT}_{\coh_\cs(X_+)}$.

Thus $\hat{DT}_{Q,I}\equiv\hat{DT}_{\coh_\cs(X)}\equiv
\hat{DT}_{\coh_\cs(X_+)}$, verifying Conjecture \ref{dt7conj2} for
the equivalences \eq{dt7eq36}. This seems a much simpler way of
relating enumerative invariants in $\modCQI,\coh_\cs(X)$ and
$\coh_\cs(X_+)$ than those
in~\cite{Szen,NaNa}.\index{conifold!noncommutative|)}

\subsubsection{Coherent sheaves on $\C^3/\Z_2^2,$ following Young}
\label{dt753}

Let $G$ be the subgroup $\Z_2^2$ in $\SL(3,\C)$ generated by
$(z_1,z_2,z_3)\mapsto(-z_1,-z_2,z_3)$ and
$(z_1,z_2,z_3)\mapsto(z_1,-z_2,-z_3)$. Then the Ginzburg
construction in Example \ref{dt7ex2} yields a quiver $Q_G$ and a
cubic superpotential $W_G$ giving relations $I_G$ such that mod-$\C
Q_G/I_G$ is 3-Calabi--Yau and equivalent to the abelian category of
$G$-equivariant compactly-supported coherent sheaves on $\C^3$.
Write $(Q,I)$ for $(Q_G,I_G)$. Then $Q$ has 4 vertices
$v_0,\ldots,v_3$ corresponding to the irreducible representations of
$\Z_2^2$, with $v_0$ the trivial representation, and 12 edges, as
below:
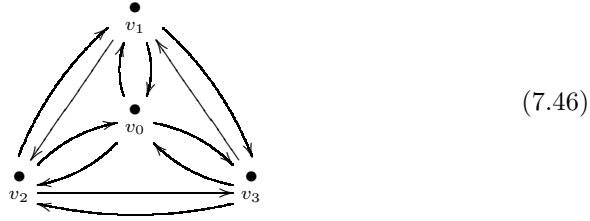
\e
\begin{gathered}
\xymatrix@C=30pt@R=7pt{ & \mathop{\bu} \limits_{v_1}
\ar@<-.5ex>[dddl] \ar@/^/@<1ex>[dddr]
\ar@/^/[dd] \\ \\
& \mathop{\bu} \limits_{v_0} \ar@/^/[dl] \ar@/^/[dr] \ar@/^/[uu] \\
\mathop{\bu} \limits_{v_2} \ar@/^/@<1ex>[uuur] \ar@/^/[ur]
\ar@<-.5ex>[rr] && \mathop{\bu} \limits_{v_3} \ar@<-.5ex>[uuul]
\ar@/^/[ul] \ar@/^/@<1ex>[ll] }
\end{gathered}
\label{dt7eq46}
\e
Theorem \ref{dt7thm1} implies that the Euler form $\bar\chi$ of
$\modCQI$ is zero.

As for \eq{dt7eq37}, Young and Bryan \cite[Th.s 1.5 \& 1.6]{YoBr}
prove that
\e
\begin{split}
1&+\sum_{\bs d\in C(\modCQI)}NDT^{\bs d,\de_{v_0}}_{Q,I}(0')q_0^{\bs
d(v_0)}q_1^{\bs d(v_1)}q_2^{\bs d(v_2)}q_3^{\bs d(v_3)}\\
&=\prod_{k\ge 1}\begin{aligned}[t]
&\bigl(1-(-q_0q_1q_2q_3)^k)\bigr)^{-4k}\\
&\bigl(1-(-q_0q_1)^k(q_2q_3)^{k+1}\bigr)^{-k}
\bigl(1-(-q_0q_1)^k(q_2q_3)^{k-1}\bigr)^{-k}\\
&\bigl(1-(-q_0q_2)^k(q_3q_1)^{k+1}\bigr)^{-k}
\bigl(1-(-q_0q_2)^k(q_3q_1)^{k-1}\bigr)^{-k}\\
&\bigl(1-(-q_0q_3)^k(q_1q_2)^{k+1}\bigr)^{-k}
\bigl(1-(-q_0q_3)^k(q_1q_2)^{k-1}\bigr)^{-k}\\
&\bigl(1-(-q_0q_2q_3)^kq_1^{k+1}\bigr)^k
\bigl(1-(-q_0q_2q_3)^kq_1^{k-1}\bigr)^k\\
&\bigl(1-(-q_0q_3q_1)^kq_2^{k+1}\bigr)^k
\bigl(1-(-q_0q_3q_1)^kq_2^{k-1}\bigr)^k\\
&\bigl(1-(-q_0q_1q_2)^kq_3^{k+1}\bigr)^k
\bigl(1-(-q_0q_1q_2)^kq_3^{k-1}\bigr)^k\\
&\bigl(1-(-q_0)^k(q_1q_2q_3)^{k+1}\bigr)^k
\bigl(1-(-q_0)^k(q_1q_2q_3)^{k-1}\bigr)^k.
\end{aligned}
\end{split}
\label{dt7eq47}
\e

Arguing as for \eq{dt7eq38}--\eq{dt7eq40} and writing $\bs
d\!=\!(d_0,\ldots,d_3)$ with $d_j\!=\!\bs d(v_j)$ yields
\e
\hat{DT}{}^{(d_0,\ldots,d_3)}_{Q,I}(0)=\begin{cases} -4, &
\text{$d_j=k$ for all $j$, $k\ge 1$,}\\
-1, & \text{$d_j=k$ for two $j$, $d_j=k-1$ for two $j$,
$k\ge 1$,}\\
\phantom{-}1, &  \text{$d_j=k$ for three $j$, $d_j\!=\!k\!-\!1$
for one $j$, $k\!\ge\!1$,}\\
\phantom{-}0, & \text{otherwise.}
\end{cases}
\label{dt7eq48}
\e
This is clearly much simpler than \eq{dt7eq47}, and restores the
symmetry between $v_0,\ldots,v_3$ in \eq{dt7eq46} which is lost in
\eq{dt7eq47} by selecting the vertex~$v_0$.

If $X$ is any crepant resolution of $\C^3/G$ then by Ginzburg
\cite[Cor.~4.4.8]{Ginz} we have $D^b(\modCQI)\sim D^b(\coh_\cs(X))$.
As the Euler forms $\bar\chi$ are zero on $\modCQI,\coh_\cs(X)$,
using Conjecture \ref{dt7conj2} we can read off a prediction for the
invariants $\hat{DT}{}^\al_{\coh_\cs(X)}(\tau)$. The first line of
\eq{dt7eq48} corresponds to \eq{dt6eq20} for $X$, as one can show
that~$\chi(X)=4$.

\subsubsection{Coherent sheaves on $\C^3/\Z_n,$ following Young}
\label{dt754}

Let $G$ be the subgroup $\Z_n$ in $\SL(3,\C)$ generated by
$(z_1,z_2,z_3)\mapsto(e^{2\pi i/n}z_1,\ab z_2,\ab e^{-2\pi
i/n}z_3)$. Then the Ginzburg construction in Example \ref{dt7ex2}
gives a quiver $Q$ and a cubic superpotential $W$ giving relations
$I$ such that $\modCQI$ is 3-Calabi--Yau and equivalent to the
abelian category of $G$-equivariant compactly-supported coherent
sheaves on $\C^3$. Then $Q$ has vertices $v_0,\ldots,v_{n-1}$, with
$v_0$ the trivial representation. We take $v_i$ to be indexed by
$i\in\Z_n$, so that $v_i=v_j$ if $i\equiv j \mod n$. With this
convention, $Q$ has edges $v_i\ra v_{i+1}$, $v_i\ra v_i$, $v_i\ra
v_{i-1}$ for $i=0,\ldots,n-1$. The case $n=3$ is shown below:\vskip
5pt
\e
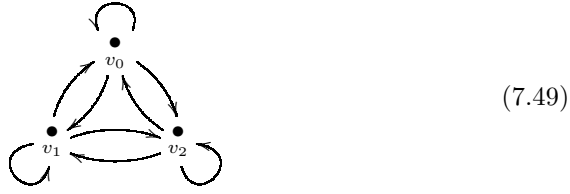
\begin{gathered}
\xymatrix@C=10pt@R=15pt{ & \mathop{\bu} \limits_{v_0}
\ar@/^/@<-.3ex>[dl] \ar@/^/@<.3ex>[dr] \ar@(ur,ul)[]
\\\mathop{\bu} \limits_{v_1} \ar@/^/@<.3ex>[ur] \ar@/^/@<-.3ex>[rr]
\ar@(l,d)[] && \mathop{\bu} \limits_{v_2} \ar@/^/@<-.3ex>[ul]
\ar@/^/@<.3ex>[ll] \ar@(d,r)[] }
\end{gathered}
\label{dt7eq49}
\e
\vskip 11pt

\noindent Theorem \ref{dt7thm1} implies that the Euler form
$\bar\chi$ of $\modCQI$ is zero.

As for \eq{dt7eq37} and \eq{dt7eq47}, Young and Bryan \cite[Th.s 1.4
\& 1.6]{YoBr} prove that
\ea
1+\sum_{\bs d\in C(\modCQI)\!\!\!\!\!\!\!\!\!\!\!\!\!\!
\!\!\!\!\!\!\!\!\!\!\!\!\!\!\!\! }&NDT^{\bs
d,\de_{v_0}}_{Q,I}(0')q_0^{\bs d(v_0)}q_1^{\bs d(v_1)}\cdots
q_{n-1}^{\bs d(v_{n-1})}\!=\!\prod_{k\ge 1}\bigl(1-(-q_0\cdots
q_{n-1})^k)\bigr)^{-n}\cdot
\nonumber\\
&\prod_{0<a<b\le n}\prod_{k\ge 1}\begin{aligned}[t]&
\bigl(1-(-q_0\cdots q_{n-1})^k(q_aq_{a+1}\cdots q_{b-1})\bigr)^{-k}\\
&\bigl(1-(-q_0\cdots q_{n-1})^k(q_aq_{a+1}\cdots
q_{b-1})^{-1}\bigr)^{-k}.
\end{aligned}
\label{dt7eq50}
\ea
Arguing as for \eq{dt7eq38}--\eq{dt7eq40} and writing $\bs
d\!=\!(d_0,\ldots,d_{n-1})$ with $d_j\!=\!\bs d(v_j)$ yields
\e
\hat{DT}{}^{(d_0,\ldots,d_{n-1})}_{Q,I}(0)=\begin{cases} -n, &
\text{$d_i=k$ for all $i$, $k\ge 1$,}\\
-1, & \left\{\begin{aligned}[h]&\text{$d_i=k$ for $i=a,\ldots,b-1$
and}\\[-3pt]
&\text{$d_i=k-1$ for $i=b,\ldots,a+n-1$,}\\[-3pt]
&\text{$0\le a<n,$ $a<b<a+n$, $k\ge 1$,}
\end{aligned}\right.\\
 \phantom{-}0, & \text{otherwise.}
\end{cases}
\label{dt7eq51}
\e
This is simpler than \eq{dt7eq50}, and restores the dihedral
symmetry group of \eq{dt7eq49}, which is lost in \eq{dt7eq50} by
selecting the vertex~$v_0$.

\subsubsection{Conclusions}
\label{dt755}

In each of our four examples, the noncommutative Don\-aldson--Thomas
invariants\index{Donaldson--Thomas invariants!noncommutative}
$NDT^{\bs d,\de_v}_{Q,I}(0')$ can be written in a generating
function as an explicit infinite product involving MacMahon
type\index{MacMahon function} factors \eq{dt7eq33}, \eq{dt7eq37},
\eq{dt7eq47}, \eq{dt7eq50}. In each case, this product form held
because the Euler form $\bar\chi$ of $\modCQI$ was zero, so that the
generating function for $NDT^{\bs d,\de_v}_{Q,I}(0')$ has an
exponential expression \eq{dt7eq32} in terms of the
$\smash{\bar{DT}{}^{\bs d}_{Q,I}(0)}$, and because of simple
explicit formulae \eq{dt7eq34}, \eq{dt7eq40}, \eq{dt7eq48},
\eq{dt7eq51} for the BPS invariants~$\hat{DT}{}^{\bs
d}_{Q,I}(0)$.\index{BPS invariants!for quivers}

In these examples, the BPS invariants $\hat{DT}{}^{\bs
d}_{Q,I}(\mu)$ seem to be a simpler and more illuminating invariant
than the noncommutative Donaldson--Thomas invariants $NDT^{\bs
d,\de_v}_{Q,I}(\mu')$. That $\hat{DT}{}^{\bs d}_{Q,I}(\mu)$ has such
a simple form probably says something interesting about the
representation theory of $\C Q/I$, which may be worth pursuing.
Also, when we pass from the invariants $\smash{\hat{DT}{}^{\bs
d}_{Q,I}(\mu)}$ for the abelian category $\modCQI$ to
$\hat{DT}_{Q,I}:K(\modCQI)\ra\Z$ for the derived category
$D^b(\modCQI)$ as in Conjecture \ref{dt7conj2}, in
\S\ref{dt752}--\S\ref{dt754} things actually become simpler, in that
pairs of entries parametrized by $k\ge 1$ combine to give one entry
parametrized by $k\in\Z$. So maybe these phenomena will be best
understood in the derived category.

We can also ask whether there are other categories $\modCQI$ which
admit the same kind of explicit computation of invariants. For the
programme above to work we need the Euler form $\bar\chi$ to be
zero, which by Theorem \ref{dt7thm1} means that for all vertices
$i,j$ in $Q$ there must be the same number of edges $i\ra j$ as
edges $j\ra i$. Suppose $\modCQI$ comes from a finite subgroup
$G\subset\SL(3,\C)$ as in Example \ref{dt7ex2}, and let
$\pi:X\ra\C^3/G$ be a crepant resolution. Then as $D^b(\modCQI)\sim
D^b(\coh_\cs(X))$, the Euler form of $\modCQI$ is zero if and only
if that of $\coh_\cs(X)$ is zero.

The Euler form\index{Euler form} of $\coh_\cs(X)$ is zero if and only if
$\pi:X\ra\C^3/G$ is {\it semismall},\index{semismall resolution} that
is, no divisors in $X$ lie over points in $\C^3/G$. This is
equivalent to the `hard Lefschetz condition' for $\C^3/G$, and by
Bryan and Gholampour \cite[Lem.~3.4.1]{BrGh} holds if and only if
$G$ is conjugate to a subgroup of either $\SO(3)\subset\SL(3,\C)$ or
$\SU(2)\subset \SL(3,\C)$; in \S\ref{dt753} we have
$\Z_2^2\subset\SO(3)\subset \SL(3,\C)$, and in \S\ref{dt754} we have
$\Z_n\subset\SU(2)\subset \SL(3,\C)$. Following discussion in Bryan
and Gholampour \cite[\S 1.2.1]{BrGh}, Young and Bryan
\cite[Conj.~A.6 \& Rem.~A.9]{YoBr}, and Szendr\H oi \cite[\S
2.12]{Szen}, it seems likely that formulae similar to \eq{dt7eq47}
and \eq{dt7eq50} hold for all finite $G$ in $\SO(3)\subset\SL(3,\C)$
or $\SU(2)\subset\SL(3,\C)$, so that the $\hat{DT}{}^{\bs
d}_{Q,I}(0)$ have a simple form. But note as in
\cite[Rem.~A.10]{YoBr} that the Gromov--Witten
invariants\index{Gromov--Witten invariants} of $X$ computed in
\cite{BrGh} are not always the right ones for computing
Donaldson--Thomas invariants, because of the way they count curves
going to infinity.

\subsection{Integrality of $\hat{DT}{}^{\bs d}_Q(\mu)$ for generic
$(\mu,\R,\le)$}
\label{dt76}\index{Donaldson--Thomas invariants!integrality properties|(}

We now prove Conjecture \ref{dt7conj1} when $W\equiv 0$, that for
$Q$ a quiver and $(\mu,\R,\le)$ generic we have $\hat{DT}{}^{\bs
d}_Q(\mu)\in\Z$. We first compute the invariants when $Q$ has only
one vertex and verify their integrality, using Reineke
\cite{Rein1,Rein3}. This example is also discussed by Kontsevich and
Soibelman~\cite[\S 7.5]{KoSo1}.

\begin{ex} Let $Q_m$ be the quiver with one vertex $v$ and $m$ edges
$v\ra v$, for $m\ge 0$. Then $K(\modCQ_m)=\Z$ and $C(\modCQ_m)=\N$.
Consider the trivial stability condition $(0,\R,\le)$ on
$\smash{\modCQ_m}$. Then our framed moduli space
$\M_{\stf\,Q_m}^{d,e}(0')$ is $H^{(m)}_{d,e}$ in Reineke's notation
\cite{Rein1}. Reineke \cite[Th.~1.4]{Rein1} proves that
\begin{equation*}
\chi\bigl(\M_{\stf\,Q_m}^{d,e}(0')\bigr)=\frac{e}{(m-1)d+1}\,
\binom{md+e-1}{d},
\end{equation*}
so by \eq{dt7eq28} we have
\begin{equation*}
NDT_{Q_m}^{d,e}(0')=(-1)^{d(1-m)+ed}\frac{e}{(m-1)d+1}\,
\binom{md+e-1}{d}.
\end{equation*}
Fixing $e=1$, we see as in \cite[\S 7.5]{KoSo1} that
\e
\begin{split}
1+\sum_{d\ge 1}NDT^{d,1}_{Q_m}(0')q^d&=\sum_{d\ge 0}
\frac{(-1)^{md}}{(m-1)d+1}\,\binom{md}{d}q^d\\[-3pt]
&=\exp\raisebox{-4pt}{\begin{Large}$\displaystyle\Bigl[$\end{Large}}
\sum_{d\ge 1} \frac{(-1)^{md}}{md}\,\binom{md}{d}q^d
\raisebox{-4pt}{\begin{Large}$\displaystyle\Bigr]$\end{Large}}.
\end{split}
\label{dt7eq52}
\e
Taking logs of \eq{dt7eq52} and using \eq{dt7eq32} yields
\begin{equation*}
\bar{DT}{}^d_{Q_m}(0)=\frac{(-1)^{(m+1)d+1}}{md^2}\,\binom{md}{d},
\end{equation*}
so by \eq{dt7eq21} we have
\e
\hat{DT}{}^{d}_{Q_m}(0)=\frac{1}{md^2}\sum_{e\ge 1,\; e\mid
d}\Mo(d/e)(-1)^{(m+1)e+1}\binom{me}{e}.
\label{dt7eq53}
\e
By Reineke \cite[Th.~5.9]{Rein3} applied with $N=m$, $b_1=1$ and
$b_i=0$ for $i>1$ as in \cite[Ex., \S 5]{Rein3}, so that the r.h.s.\
of \eq{dt7eq53} is $-a_d$ in Reineke's notation, we have
$\hat{DT}{}^{d}_{Q_m}(0)\in\Z$ for $d\ge 1$. This will be important
in the proof of Theorem~\ref{dt7thm6}.
\label{dt7ex5}
\end{ex}\index{Donaldson--Thomas invariants!computation in
examples|)}

Now let $Q$ be an arbitrary quiver without relations, and
$(\mu,\R,\le)$ a slope stability condition on $\modCQ$ which is
generic in the sense of Conjecture \ref{dt7conj1}. As in
\S\ref{dt62}, define a 1-morphism $P_m:\fM_Q\ra\fM_Q$ for $m\ge 1$
by $P_m:[E]\mapsto [mE]$ for $E\in\modCQ$. Then as for \eq{dt5eq9}
and \eq{dt6eq16}, for all $\bs d\in C(\modCQ)$ we have
\e
\begin{gathered}
\hat{DT}{}^{\bs d}_Q(\mu)=\chi\bigl(\M_\rss^{\bs
d}(\mu),F^{\bs d}_Q(\mu)\bigr), \qquad\text{where}\\
F^{\bs d}_Q(\mu)=-\sum_{\!\!\!\!m\ge 1,\; m\mid\bs d\!\!\!}
\frac{\Mo(m)}{m^2}\begin{aligned}[h] \CF^\na(\pi)\bigl[&
\CF^\na(P_m)\ci\Pi_{\CF}\ci\\
&\,\bar\Pi^{\chi,\Q}_{\fM_Q}(\bep^{\bs
d/m}(\mu))\cdot\nu_{\fM_Q}\bigr].
\end{aligned}
\end{gathered}
\label{dt7eq54}
\e
Here $\fM_\rss^{\bs d}(\mu)$ is the moduli stack of $\mu$-semistable
objects of class $\bs d$ in $\modCQ$, an open substack of $\fM_Q$,
and $\M_\rss^{\bs d}(\mu)$ is the quasiprojective coarse moduli
scheme\index{coarse moduli scheme}\index{moduli scheme!coarse} of
$\mu$-semistable objects of class $\bs d$ in $\modCQ$, and
$\pi:\fM_\rss^{\bs d}(\mu)\ra\M_\rss^{\bs d}(\mu)$ is the natural
projection 1-morphism.

An object $E$ in $\modCQ$ is called $\mu$-{\it
polystable\/}\index{polystable@$\tau$-polystable} if it is
$\mu$-semistable and a direct sum of $\mu$-stable objects. That is,
$E$ is $\mu$-polystable if and only if $E\cong a_1E_1\op\cdots \op
a_kE_k$, where $E_1,\ldots,E_k$ are pairwise nonisomorphic
$\mu$-stables in $\modCQ$ with $\mu([E_1])=\cdots=\mu([E_k])$ and
$a_1,\ldots,a_k\ge 1$, and $E$ determines $E_1,\ldots,E_k$ and
$a_1,\ldots,a_k$ up to order and isomorphism. Since $\mu$ is a
stability condition, each $\C$-point of $\M_\rss^{\bs d}(\mu)$ is
represented uniquely up to isomorphism by a $\mu$-polystable. That
is, if $E'$ is $\mu$-semistable then $E'$ admits a Jordan--H\"older
filtration with $\mu$-stable factors $E_1,\ldots,E_k$ of
multiplicities $a_1,\ldots,a_k$, and $E=a_1E_1\op\cdots\op a_kE_k$
is the $\mu$-polystable representing~$[E']\in\M_\rss^{\bs
d}(\mu)(\C)$.

Here is a useful expression for $F^{\bs d}_Q(\mu)$ in \eq{dt7eq54}
at a $\mu$-polystable~$E$:

\begin{prop} Let\/ $Q$ be a quiver, $(\mu,\R,\le)$ a slope
stability condition on $\modCQ,$ and\/ $E=a_1E_1\op\cdots\op a_kE_k$
a $\mu$-polystable representing a $\C$-point $[E]$ in $\M_\rss^{\bs
d}(\mu)$ for $\bs d\in C(\modCQ),$ where $E_1,\ldots,E_k$ are
pairwise nonisomorphic $\mu$-stables in $\modCQ$ with\/
$\mu([E_1])=\cdots=\mu([E_k])$ and\/~$a_1,\ldots,a_k\ge 1$.

Define the \begin{bfseries}Ext quiver\end{bfseries}\index{quiver!Ext
quiver} $Q_E$ of\/ $E$ to have vertices $\{1,2,\ldots,k\}$ and\/
$d_{ij}=\dim\Ext^1(E_i,E_j)$ edges $i\ra j$ for all\/
$i,j=1,\ldots,k,$ and define a dimension vector $\bs a$ in
$C(\modCQ_E)$ by $\bs a(i)=a_i$ for $i=1,\ldots,k$. For\/
$i=1,\ldots,k,$ define $\hat E_i\in\modCQ_E$ to have vector spaces
$X_v=\C$ for vertex $v=i$ and\/ $X_v=0$ for vertices $v\ne i$ in
$Q_E,$ and linear maps $\rho_e=0$ for all edges $e$ in $Q_E$. Set\/
$\hat E=a_1\hat E_1\op\cdots\op a_k\hat E_k$ in $\modCQ_E$. Then for
$F^{\bs d}_Q(\mu)$ as in {\rm\eq{dt7eq54},} we have
\e
F^{\bs d}_Q(\mu)\bigl([E]\bigr)=F^{\bs a}_{Q_E}(0)\bigl([\hat
E]\bigr)=\hat{DT}{}^{\bs a}_{Q_E}(0).
\label{dt7eq55}
\e
\label{dt7prop3}
\end{prop}

\begin{proof} Write $(\modCQ_E)_{\hat E_1,\ldots,\hat E_k}$ for
the full subcategory of $F$ in $\modCQ_E$ generated by $\smash{\hat
E_1,\ldots,\hat E_k}$ by repeated extensions. Then $(X,\rho)$ in
$\modCQ_E$ lies in $(\modCQ_E)_{\hat E_1,\ldots,\hat E_k}$ if and
only if it is {\it nilpotent},\index{quiver!nilpotent
representation} that is, $\rho(\C Q_{E(n)})=0$ for some $n\ge 0$,
where the ideal $\C Q_{E(n)}$ of paths of length at least $n$ in $\C
Q_E$ is as in Definition \ref{dt7def1}. Similarly, write
$(\modCQ)_{E_1,\ldots,E_k}$ for the full subcategory of objects in
$\modCQ$ generated by $E_1,\ldots,E_k$ by repeated extensions. Both
are $\C$-linear abelian subcategories.

In $\modCQ_E$ we have $\Hom(\hat E_i,\hat E_j)=\C$ for $i=j$ and
$\Hom(\hat E_i,\hat E_j)=0$ for $i\ne j$, and $\Ext^1(\hat E_i,\hat
E_j)\cong\C^{d_{ij}}$ for all $i,j$. In $\modCQ$ we have
$\Hom(E_i,E_j)=\C$ for $i=j$ and $\Hom(E_i,E_j)=0$ for $i\ne j$, and
$\Ext^1(E_i,E_j) \cong\C^{d_{ij}}$ for all $i,j$.  Choose
isomorphisms $\Ext^1(\hat E_i,\hat E_j)\cong\Ext^1(E_i,E_j)$ for all
$i,j$. It is then easy to construct an equivalence of $\C$-linear
abelian categories
\e
G:(\modCQ_E)_{\hat E_1,\ldots,\hat E_k}\longra
(\modCQ)_{E_1,\ldots,E_k}
\label{dt7eq56}
\e
using linear algebra, such that $G(\hat E_i)=E_i$ for
$i=1,\ldots,k$, and $G$ induces the chosen isomorphisms~$\Ext^1(\hat
E_i,\hat E_j)\ra\Ext^1(E_i,E_j)$.

Write $(\fM_{Q_E})_{\hat E_1,\ldots, \hat E_k},(\fM_Q)_{E_1,
\ldots,E_k}$ for the locally closed $\C$-substacks of objects in
$(\modCQ_E)_{\hat E_1,\ldots,\hat E_k},(\modCQ)_{E_1,\ldots,E_k}$ in
the moduli stacks $\fM_{Q_E},\fM_Q$ of $\modCQ_E,\modCQ$. Then $G$
induces a 1-isomorphism of Artin $\C$-stacks
\begin{equation*}
\dot G:(\fM_{Q_E})_{\hat E_1,\ldots,\hat E_k}\longra (\fM_Q)_{E_1,
\ldots,E_k}.
\end{equation*}
As $G$ identifies $\Hom(\hat E_i,\hat E_j),\Ext^1(\hat E_i,\hat
E_j)$ with $\Hom(E_i,E_j),\Ext^1(E_i,E_j)$, it follows that $G$
takes the restriction to $(\modCQ_E)_{\hat E_1,\ldots,\hat E_k}$ of
the Euler form $\hat\chi_{Q_E}$ on $\modCQ_E$ to the restriction to
$(\modCQ)_{E_1,\ldots,E_k}$ of the Euler form $\hat\chi_Q$ on
$\modCQ$. By \eq{dt7eq2} $\fM_Q^{\bs d}$ is smooth of dimension
$-\hat\chi_Q(\bs d,\bs d)$, so the Behrend function
$\nu_{\smash{\fM_Q^{\bs d}}}\equiv (-1)^{-\hat\chi_Q(\bs d,\bs d)}$
by Corollary \ref{dt4cor1}, and similarly
$\nu_{\smash{\fM_{Q_E}^{\bs a}}}\equiv (-1)^{-\hat\chi_{Q_E}(\bs
a,\bs a)}$. As $\dot G$ takes $\hat\chi_{Q_E}$ to $\hat\chi_Q$, it
follows that
\e
\dot G_*\bigl(\nu_{\fM_{Q_E}}\vert_{(\fM_{Q_E})_{\hat E_1,\ldots,
\hat E_k}}\bigr)=\nu_{\fM_Q}\vert_{(\fM_Q)_{E_1,\ldots,E_k}}.
\label{dt7eq57}
\e

Since all objects in $(\modCQ_E)_{\hat E_1,\ldots,\hat E_k}$ are
0-semistable, and all objects in $(\modCQ)_{E_1,\ldots,E_k}$ are
$\mu$-semistable, and $G$ is a 1-isomorphism, we see that
\begin{equation*}
\dot G_*\bigl(\bde_{\rss\, Q_E}^{{\bs
a}/m}(0)\vert_{(\fM_{Q_E})_{\hat E_1,\ldots, \hat
E_k}}\bigr)=\bde_{\rss\, Q}^{{\bs d}/m}(\mu)\vert_{(\fM_Q)_{E_1,
\ldots,E_k}}
\end{equation*}
for $m\ge 1$ with $m\mid{\bs d}$. So from \eq{dt3eq4} we deduce that
\e
\dot G_*\bigl(\bep^{{\bs a}/m}_{Q_E}(0)\vert_{(\fM_{Q_E})_{\hat
E_1,\ldots, \hat E_k}}\bigr)=\bep^{{\bs d}/m}_Q(\mu)
\vert_{(\fM_{Q_E})_{\hat E_1,\ldots,\hat E_k}}.
\label{dt7eq58}
\e
Equations \eq{dt7eq57} and \eq{dt7eq58} imply that
\begin{equation*}
\dot G_*\bigl(F^{\bs a}_{Q_E}(0)\vert_{(\fM_{Q_E})_{\hat E_1,\ldots,
\hat E_k}(\C)}\bigr)=F^{\bs d}_Q(\mu)\vert_{(\fM_Q)_{E_1,
\ldots,E_k}(\C)},
\end{equation*}
since $\dot G_*$ identifies \eq{dt7eq54} for $\modCQ_E$ on
$(\fM_{Q_E})_{\hat E_1,\ldots,\hat E_k}$ term-by-term with
\eq{dt7eq54} for $\modCQ$ on $(\fM_Q)_{E_1, \ldots,E_k}$. As $\dot
G([\hat E])=[E]$, this implies the first equality of~\eq{dt7eq55}.

Now consider the $\bG_m$-action on $\modCQ_E$ acting by
$\la:(X,\rho)\mapsto(X,\la\rho)$ for $\la\in\bG_m$ and
$(X,\rho)\in\modCQ_E$. This induces $\bG_m$-actions on the moduli
stack $\fM_{Q_E}^{\bs a}$ and the coarse moduli space $\M_\rss^{\bs
a}(0)$. By \eq{dt7eq2} we have $\fM_{Q_E}^{\bs a}\cong[V/H]$ where
$V$ is a vector space and $H=\prod_{i=1}^k\GL(a_i,\C)$, and the
$\bG_m$-action on $\fM_{Q_E}^{\bs a}$ is induced by multiplication
by $\bG_m$ in $V$. Let $A_E$ be the $\C$-algebra of $H$-invariant
polynomials on $V$. Then $\M_\rss^{\bs a}(0)=\Spec A_E$ by GIT.

This $A_E$ is graded by homogeneous polynomials of degree
$d=0,1,\ldots$ on $V$, and $\bG_m$ acts on homogeneous polynomial
$f$ of degree $d$ by $\la:f\mapsto\la^df$. Thus there is exactly one
point in $\M_\rss^{\bs a}(0)$ fixed by the $\bG_m$-action, the ideal
of polynomials in $A_E$ which vanish at $0\in V$. Since $0\in V$
corresponds to $[\hat E]$ in $\fM_{Q_E}^{\bs a}\cong[V/H]$, we see
that there is a $\bG_m$-action on $\M_\rss^{\bs a}(0)$ with unique
fixed point $[\hat E]$. By \eq{dt7eq54} we have $\hat{DT}{}^{\bs
a}_{Q_E}(0)=\chi\bigl(\M_\rss^{\bs d}(0),F^{\bs a}_{Q_E}(0)\bigr)$.
The $\bG_m$-action on $\M_\rss^{\bs d}(0)$ preserves $F^{\bs
a}_{Q_E}(0)$. The second equality of \eq{dt7eq55} follows by the
usual torus localization argument, as all $\bG_m$-orbits other than
$\smash{[\hat E]}$ are copies of $\bG_m$, and contribute 0 to the
weighted Euler characteristic.
\end{proof}

Here is our integrality result, which proves Conjecture
\ref{dt7conj1} for $\modCQ$. We give two different proofs of it. The
first proof works by first proving the analogue of Conjecture
\ref{dt6conj2} in our quiver context, and so is evidence for
Conjecture \ref{dt6conj2}. It relies on the integrality of the
invariants $\hat{DT}{}^{d}_{Q_m}(0)$ in equation \eq{dt7eq53} of
Example \ref{dt7ex5}, which we proved using
Reineke~\cite[Th.~5.9]{Rein3}.

Reineke \cite{Rein3} proves an integrality conjecture of Kontsevich
and Soibelman \cite[Conj.~1]{KoSo1}. The authors believe
\cite[Conj.~1]{KoSo1} concerns integrality of {\it transformation
laws}, rather than of invariants themselves. That is, if
$\mu,\ti\mu$ are generic stability conditions on $\modCQI$, then
translated into our framework \cite[Conj.~1]{KoSo1} should imply
that $\hat{DT}{}^{\bs d}_{Q,I}(\ti\mu)\in\Z$ for all $\bs d$ if and
only if $\hat{DT}{}^{\bs d'}_{Q,I}(\mu)\in\Z$ for all $\bs d'$,
where $\hat{DT}{}^{\bs d}_{Q,I}(\ti\mu),\hat{DT}{}^{\bs
d'}_{Q,I}(\mu)$ are related using~\eq{dt7eq21}--\eq{dt7eq23}.

Despite this, in our second proof of Theorem \ref{dt7thm6} we prove
integrality of the $\hat{DT}{}^{\bs d}_Q(\mu)$ using Reineke
\cite{Rein3} in a more direct way than the first proof. We include
this second proof to try to clarify the relationship between our
work and Reineke's. Also, the second proof implicitly expresses the
$\hat{DT}{}^{\bs d}_Q(\mu)$ in terms of Euler characteristics
$\smash{\chi\bigl(\M_{\st\,Q}^{\bs d'}(\mu)\bigr)}$ of $\mu$-stable
moduli schemes $\M_{\st\,Q}^{\bs d'}(\mu)$, such that integrality of
the $\hat{DT}{}^{\bs d}_Q(\mu)$ follows from integrality of the
$\smash{\chi\bigl(\M_{\st\,Q}^{\bs d'}(\mu)\bigr)}$. Since in our
set up we never use stable moduli schemes, and one of our major
themes is counting strictly semistables correctly, to involve
invariants counting only stables and ignoring strictly semistables
seems curious.

\begin{thm} Let\/ $Q$ be a quiver, and write $\hat\chi_Q:
K(\modCQ)\times K(\modCQ)\ab\ra\Z$ for the Euler form of\/ $Q$ and
$\bar\chi_Q: K(\modCQ)\times K(\modCQ)\ra\Z$ for its
antisymmetrization, as in \eq{dt7eq4}--\eq{dt7eq6}. Let\/
$(\mu,\R,\le)$ be a
\begin{bfseries}generic\end{bfseries}\index{stability condition!generic}
slope stability condition on $\modCQ,$ that is, for all\/ $\bs d,\bs
e\in C(\modCQ)$ with\/ $\mu(\bs d)=\mu(\bs e)$ we have\/
$\bar\chi_Q(\bs d,\bs e)=0$. Then for all\/ $\bs d\in C(\modCQ)$ the
constructible function\index{constructible function} $F^{\bs
d}_Q(\mu)$ on $\M_\rss^{\bs d}(\mu)$ in \eq{dt7eq54} is $\Z$-valued,
so that\/~$\hat{DT}{}^{\bs d}_Q(\mu)\in\Z$.
\label{dt7thm6}
\end{thm}

\begin{proof}[First proof] For $Q,(\mu,\R,\le),\bs d$ as in the
theorem, let a $\C$-point in $\M_\rss^{\bs d}(\mu)$ be represented
by a $\mu$-polystable $E=a_1E_1\op\cdots \op a_kE_k$, where
$E_1,\ldots,E_k$ are pairwise nonisomorphic $\mu$-stables in
$\modCQ$ with $\mu([E_1])=\cdots=\mu([E_k])$, and $a_1,\ldots,a_k\ge
1$. Use the notation of Proposition \ref{dt7prop3}. As
$(\mu,\R,\le)$ is generic and $\mu([E_i])=\mu([E_j])$ we have
$\bar\chi_Q([E_i],[E_j])=0$ for all $i,j$. But $G$ in \eq{dt7eq56}
takes $\hat\chi_{Q_E}$ to $\hat\chi_Q$ and $\bar\chi_{Q_E}$ to
$\bar\chi_Q$, so $\bar\chi_{Q_E}([\hat E_i],[\hat E_j])=0$ for all
$i,j$. Since the $[\hat E_i]$ for $i=1,\ldots,k$ span $K(\modCQ_E)$,
this implies that $\bar\chi_{Q_E}\equiv 0$. We must show that
$F^{\bs d}_Q(\mu)([E])\in\Z$, which by Proposition \ref{dt7prop3} is
equivalent to~$\hat{DT}{}^{\bs a}_{Q_E}(0)\in\Z$.

Thus, replacing $Q_E,\bs a$ by $Q,\bs d$, it is enough to show that
for all quivers $Q$ with $\bar\chi_Q\equiv 0$ and all $\bs d\in
C(\modCQ)$ we have $\hat{DT}{}^{\bs d}_Q(0)\in\Z$. Note that as in
Corollary \ref{dt7cor1}, $\bar\chi_Q\equiv 0$ implies that
$\smash{\hat{DT}{}^{\bs d}_Q(\mu)}$ is independent of the choice of
stability condition $(\mu,\R,\le)$, so $\hat{DT}{}^{\bs
d}_Q(0)\in\Z$ is equivalent to $\hat{DT}{}^{\bs d}_Q(\mu)\in\Z$ for
any $(\mu,\R,\le)$ on $\modCQ$. Write $\md{\bs d}$ for the {\it
total dimension\/} $\sum_{v\in Q_0}\bs d(v)$ of $\bs d$. We will
prove the theorem by induction on~$\md{\bs d}$.

Let $N\ge 0$. Suppose by induction that for all quivers $Q$ with
$\bar\chi_Q\equiv 0$ and all $\bs d\in C(\modCQ)$ with $\md{\bs
d}\le N$ we have $\hat{DT}{}^{\bs d}_Q(0)\in\Z$. (The first step
$N=0$ is vacuous.) Let $Q$ be a quiver with $\bar\chi_Q\equiv 0$ and
$\bs d\in C(\modCQ)$ with $\md{\bs d}=N+1$. We divide into two
cases:
\begin{itemize}
\setlength{\itemsep}{0pt}
\setlength{\parsep}{0pt}
\item[(a)] $\bs d(v)=N+1$ for some $v\in Q_0$, and $\bs d(w)=0$ for
$v\ne w\in Q_0$; and
\item[(b)] there are $v\ne w$ in $Q_0$ with $\bs d(v),\bs d(w)>0$.
\end{itemize}

In case (a), the vertices $w$ in $Q$ with $w\ne v$, and the edges
joined to them make no difference to $\smash{\hat{DT}{}^{\bs
d}_Q(0)}$, as in $(X,\rho)$ with $[(X,\rho)]=\bs d$ in $C(\modCQ)$
the vector spaces $X_w$ are zero for $w\ne v$. Thus $\hat{DT}{}^{\bs
d}_Q(0)= \hat{DT}{}^{N+1}_{Q_m}(0)$, where $m$ is the number of
edges $v\ra v$ in $Q$, and $Q_m$ is the quiver with one vertex $v$
and $m$ edges $v\ra v$. Example \ref{dt7ex5} then shows that
$\hat{DT}{}^{\bs d}_Q(0)\in\Z$, as we want.

In case (b), choose a stability condition $(\mu,\R,\le)$ on $\modCQ$
with $\mu(\de_v)\ne\mu(\de_w)$. Then $\hat{DT}{}^{\bs
d}_Q(0)=\hat{DT}{}^{\bs d}_Q(\mu)$ by Corollary \ref{dt7cor1}. So
\eq{dt7eq54}--\eq{dt7eq55} give
\e
\hat{DT}{}^{\bs d}_Q(0)\!=\!\hat{DT}{}^{\bs d}_Q(\mu)\!=\!
\chi\bigl(\M_\rss^{\bs d}(\mu),F^{\bs d}_Q(\mu)\bigr)\!=\!
\int_{\begin{subarray}{l}\,\,[E]\in \M_\rss^{\bs d}(\mu):\\
E=a_1E_1\op\cdots\op a_kE_k,\\ \text{$E$
$\mu$-polystable}\end{subarray}}\!\!\!\!\!\!\!\!\! \hat{DT}{}^{\bs
a}_{Q_E}(0)\,\rd\chi.
\label{dt7eq59}
\e
Let $E=a_1E_1\op\cdots\op a_kE_k$ be as in \eq{dt7eq59}. Then
$\sum_{i=1}^ka_i[E_i]=\bs d$, so $\sum_{i=1}^ka_i\md{[E_i]}\ab
=\md{\bs d}=N+1$. Suppose for a contradiction that $\md{[E_i]}=1$
for all $i=1,\ldots,k$. Then each $E_i$ is 1-dimensional, and
located at some vertex $u\in Q_0$, so $[E_i]=\de_u$ in $C(\modCQ)$,
and $\mu([E_i])=\mu(u)$. For each $u\in Q_0$, we have $\sum_{i:
[E_i]=\de_u}a_i=\bs d(u)$. As $\bs d(v),\bs d(w)>0$, this implies
there exist $i,j=1,\ldots,k$ with $[E_i]=\de_v$ and $[E_j]=\de_w$.
But then $\mu([E_i])=\mu(\de_v)\ne\mu(\de_w)=\mu([E_j])$, which
contradicts $\mu([E_1])=\cdots=\mu([E_k])$ as $E$ is
$\mu$-polystable.

Therefore $\md{[E_i]}\ge 1$ for all $i=1,\ldots,k$, and
$\md{[E_i]}>1$ for some $i$. As $\sum_{i=1}^ka_k\md{[E_i]}=N+1$ we
see that $\md{\bs a}=\sum_{i=1}^ka_i\le N$, so $\hat{DT}{}^{\bs
a}_{Q_E}(0)\in\Z$ by the inductive hypothesis. As this holds for all
$E$ in \eq{dt7eq59}, $\hat{DT}{}^{\bs d}_Q(0)$ is the Euler
characteristic integral of a $\Z$-valued constructible
function,\index{constructible function} so $\smash{\hat{DT}{}^{\bs
d}_Q(0)}\in\Z$. This completes the inductive step, and the first
proof of Theorem~\ref{dt7thm6}.
\end{proof}

\begin{proof}[Second proof of Theorem \ref{dt7thm6}] Let
$Q,\bar\chi_Q,\mu$ be as in the theorem. Then by perturbing $\mu$
slightly we can find a second slope stability condition
$(\ti\mu,\R,\le)$ on $\modCQ$ such that $\mu$ dominates $\ti\mu$ in
the sense of Definition \ref{dt3def6}, and if $\bs d,\bs e\in
C(\modCQ)$ then $\ti\mu(\bs d)=\ti\mu(\bs e)$ if and only if $\bs
d,\bs e$ are proportional.

Let $\bs d\in C(\modCQ)$ be primitive, and fix $\bs e\in C(\modCQ)$
with $\bs e\cdot\bs d=p>0$. Write $N=\hat\chi_Q(\bs d,\bs d)$. Then
we have
\ea
F^{\bs d,\bs e}(t):&=1+\!\!\sum_{n\ge 1} \chi\bigl(\M_{\stf\,Q}^{\bs
d,\bs e}(\ti\mu')\bigr)t^n=1+\!\!\sum_{n\ge 1}\bigl(
(-1)^{n^2p+nN}NDT^{n\bs d,\bs e}_Q(\ti\mu')\bigr)t^n
\nonumber\\
&=\exp\raisebox{-4pt}{\begin{Large}$\displaystyle\Bigl[$\end{Large}}
-\sum_{n\ge 1} (-1)^{np}np\cdot\bar{DT}{}^{n\bs d}_Q(\ti\mu)\cdot
(-1)^{np+nN}t^n
\raisebox{-4pt}{\begin{Large}$\displaystyle\Bigr]$\end{Large}}
\nonumber\\
&=\exp\raisebox{-4pt}{\begin{Large}$\displaystyle\Bigl[$\end{Large}}
-\sum_{n\ge 1} np\sum_{m\ge 1,\; m\mid n}\frac{1}{m^2}\,
\hat{DT}{}^{n\bs d/m}_Q(\ti\mu)((-1)^Nt)^n
\raisebox{-4pt}{\begin{Large}$\displaystyle\Bigr]$\end{Large}}
\label{dt7eq60}\\
&=\exp\raisebox{-4pt}{\begin{Large}$\displaystyle\Bigl[$\end{Large}}
-\sum_{i\ge 1} ip\cdot\hat{DT}{}^{i\bs d}_Q(\ti\mu)
\sum_{m\ge 1}\frac{((-1)^Nt)^{im}}{m}
\raisebox{-4pt}{\begin{Large}$\displaystyle\Bigr]$\end{Large}}
\nonumber\\
&=\prod_{i\ge 1} \bigl(1-((-1)^Nt)^i\bigr){}^{ip\cdot
\hat{DT}{}^{i\bs d}_Q(\ti\mu)}, \nonumber
\ea
in formal power series in $t$, where the first step uses
\eq{dt7eq28}, the second step is Corollary \ref{dt7cor2} with
$W\equiv 0$, $c=\ti\mu(\bs d)$ and $(-1)^{p+N}t$ in place of $q^{\bs
d}$, the third substitutes in \eq{dt7eq22}, the fourth sets $i=n/m$,
and the fifth uses
\begin{equation*}
1-x=\exp\bigl[\log(1-x)\bigr]=\exp\raisebox{-4pt}{
\begin{Large}$\displaystyle\Bigl[$\end{Large}}
-\sum_{m\ge 1}\frac{x^m}{m}
\raisebox{-4pt}{\begin{Large}$\displaystyle\Bigr]$\end{Large}}.
\end{equation*}

Set $S^{\bs d}(t)=F^{\bs d,\bs e}(t)^{1/p}$, so that in the notation
of Reineke \cite[p.~10]{Rein3} we have $S^{\bs d}_{\ti\mu}(t)=S^{\bs
d}(t)^N$. Then \eq{dt7eq60} gives
\e
S^{\bs d}(t)=\prod_{i\ge 1}\bigl(1-((-1)^Nt)^{i}\bigr){}^{-ia_i},
\quad\text{with}\quad a_i=-\hat{DT}{}^{i\bs d}_Q(\ti\mu).
\label{dt7eq61}
\e
Note that $S^{\bs d}(t)$ is independent of the choice of $\bs e$,
which also follows from \cite[Th.~4.2]{Rein3}. By Reineke
\cite[Cor.~4.3]{Rein3}, $S^{\bs d}(t)$ satisfies the functional
equation
\begin{equation*}
S^{\bs d}(t)^N=\prod_{i\ge 1}\bigl(1-t^iS^{\bs
d}(t)^{iN}\bigr)^{-iN\cdot\chi(\M_{\st\,Q}^{i\bs d}(\ti\mu))},
\end{equation*}
where $\chi\bigl(\M_{\st\,Q}^{i\bs d}(\ti\mu)\bigr)$ is the Euler
characteristic of the moduli scheme $\M_{\st\,Q}^{i\bs d}(\ti\mu)$
of $\ti\mu$-stable $E$ in $\modCQ$ with $\bdim E=i\bs d$. If $N\ne
0$, taking $N^{\rm th}$ roots gives
\e
S^{\bs d}(t)=\prod_{i\ge 1}\bigl(1-t^iS^{\bs
d}(t)^{iN}\bigr)^{-i\cdot\chi(\M_{\st\,Q}^{i\bs d}(\ti\mu))},
\label{dt7eq62}
\e
so that
\e
S^{\bs d}(t)=\prod_{i\ge 1}\bigl(1-(tS^{\bs d}(t)^N)^i\bigr)^{ib_i},
\quad\text{where}\quad b_i=-\chi(\M_{\st\,Q}^{i\bs d}(\ti\mu)).
\label{dt7eq63}
\e
When $N=0$, equations \eq{dt7eq62}--\eq{dt7eq63} follow directly
from~\cite[Th.~6.2]{Rein3}.

Now Reineke \cite[Th.~5.9]{Rein3} shows that if a formal power
series $S^{\bs d}(t)$ satisfies equations \eq{dt7eq61} and
\eq{dt7eq63} for $N\in\Z$ and $a_i,b_i\in\Q$, then $a_i\in\Z$ for
all $i\ge 1$ if and only if $b_i\in\Z$ for all $i\ge 1$. In our case
$b_i\in\Z$ is immediate, so $a_i\in\Z$, and thus $\hat{DT}{}^{i\bs
d}_Q(\ti\mu)\in\Z$. As this holds for all primitive $\bs d\in
C(\modCQ)$, we have $\hat{DT}{}^{\bs d}_Q(\ti\mu)\in\Z$ for all $\bs
d\in C(\modCQ)$. Also, comparing \eq{dt7eq61} and \eq{dt7eq63} shows
that we could compute the $\hat{DT}{}^{i\bs d}_Q(\ti\mu)$ from the
$\chi(\M_{\st\,Q}^{j\bs d}(\ti\mu))$ for~$j=1,\ldots,i$.

Theorem \ref{dt7thm4} now writes $\bar{DT}{}^{\bs d}_Q(\mu)$ in
terms of the $\bar{DT}{}^{\bs e}_Q(\ti\mu)$, in the form
\begin{equation*}
\bar{DT}{}^{\bs d}_Q(\mu)=
\bar{DT}{}^{\bs d}_Q(\ti\mu)+\text{higher order terms.}
\end{equation*}
Each higher order term involves a finite set $I$ with $\md{I}>1$, a
splitting $\bs d=\sum_{i\in I}\ka(i)$ for $\ka(i)\in C(\modCQ)$, a
combinatorial coefficient $V(I,\Ga,\ka;\ti\mu,\mu)\in\Q$, a product
of $\md{I}-1$ terms $\bar\chi(\ka(i),\ka(j))$, and the product of
the $\bar{DT}{}^{ka(i)}_Q(\ti\mu)$. Now as $\mu$ dominates $\ti\mu$
it follows that $V(I,\Ga,\ka;\ti\mu,\mu)=0$ unless
$\mu(\ka(i))=\mu(\bs d)$ for all $i\in I$, as in \cite{Joyc6}. But
then $\mu$ generic implies that $\bar\chi(\ka(i),\ka(j))=0$. Hence
all the higher order terms are zero, and $\bar{DT}{}^{\bs d}_Q(\mu)=
\bar{DT}{}^{\bs d}_Q(\ti\mu)$ for all $\bs d\in C(\modCQ)$.
Therefore $\hat{DT}{}^{\bs d}_Q(\mu)=\hat{DT}{}^{\bs d}_Q(\ti\mu)$
for all $\bs d\in C(\modCQ)$, so $\hat{DT}{}^{\bs d}_Q(\mu)\in\Z$,
as we have to prove.
\end{proof}

As for Question \ref{dt6quest1}, we can ask:

\begin{quest} In the situation of Theorem {\rm\ref{dt7thm6},} does
there exist a natural perverse sheaf\/\index{perverse sheaf} $\cal
Q$ on $\M_\rss^{\bs d}(\mu)$ with\/ $\chi_{\M_\rss^{\bs
d}(\mu)}({\cal Q})\equiv F^{\bs d}_Q(\mu)?$
\label{dt7quest2}
\end{quest}

One can ask the same question about Saito's mixed Hodge
modules\index{mixed Hodge module} \cite{Sait}. These questions should be
amenable to study in explicit examples.

\begin{rem} The proof of Theorem \ref{dt7thm6} also holds without
change for arbitrary generic stability conditions\index{stability
condition!generic} $(\tau,T,\le)$ on $\modCQ$ in the sense of
\S\ref{dt32} with $K(\modCQ)=\Z^{Q_0}$, not just for slope stability
conditions~$(\mu,\R,\le)$.
\label{dt7rem5}
\end{rem}\index{quiver!with superpotential|)}\index{Donaldson--Thomas
invariants!for quivers|)}\index{Donaldson--Thomas invariants!integrality
properties|)}\index{quiver|)}

\section[The proof of Theorem $\text{\ref{dt5thm1}}$]{The proof of
Theorem \ref{dt5thm1}}
\label{dt8}\index{coherent sheaf|(}\index{vector
bundle!algebraic|(}

Let $X$ be a projective Calabi--Yau $m$-fold over an algebraically
closed field $\K$\index{field $\K$} with a very ample line bundle
$\cO_X(1)$. Our definition of Calabi--Yau $m$-fold requires that $X$
should be smooth, the canonical bundle $K_X$ should be trivial, and
that $H^i(\cO_X)=0$ for $0<i<m$. Let $\fM$ and $\fVect$ be the
moduli stacks of coherent sheaves and algebraic vector bundles on
$X$, respectively. Then $\fM,\fVect$ are both Artin $\K$-stacks,
locally of finite type. This section proves Theorem \ref{dt5thm1},
which says that $\fM$ is locally isomorphic to $\fVect$, in the
Zariski topology.\index{Zariski topology}

Recall the following definition of {\it Seidel--Thomas twist},
\cite[Ex.~3.3]{SeTh}:

\begin{dfn} For each $n\in\Z$, the {\it Seidel--Thomas
twist\/}\index{Seidel--Thomas twist|(} $T_{\cO_X(-n)}$ by
$\cO_X(-n)$ is the Fourier--Mukai transform\index{Fourier--Mukai
transform} from $D(X)$ to $D(X)$ with kernel
\begin{equation*}
K=\cone\bigl(\cO_X(n)\boxtimes\cO_X(-n)\longra\cO_{\De}\bigr).
\end{equation*}
A good book on derived categories and Fourier--Mukai transforms is
Huybrechts \cite{Huyb}. Since $X$ is Calabi--Yau, which includes the
assumption that $H^i(\cO_X)=0$ for $0<i<m$, we see that
$\Hom^i_{D(X)}(\cO_X(-n),\cO_X(-n))$ is $\C$ for $i=0,m$ and zero
otherwise, so $\cO_X(-n)$ is a {\it spherical
object\/}\index{spherical object} in the sense of
\cite[Def.~1.1]{SeTh}, and by \cite[Th.~1.2]{SeTh} the
Seidel--Thomas twist $T_{\cO_X(-n)}$ is an
autoequivalence\index{triangulated category!equivalence} of $D(X)$.
Define $T_n=T_{\cO_X(-n)}[-1]$, the composition of $T_{\cO_X(-n)}$
and the shift $[-1]$. Then $T_n$ is also an autoequivalence
of~$D(X)$.
\label{dt8def1}
\end{dfn}

The functors $T_n$ do not preserve the subcategory $\coh(X)$ in
$D(X)$, that is, in general they take sheaves to complexes of
sheaves. However, given any bounded family of sheaves $E_U$ on $X$
we can choose $n\gg 0$ such that $T_n$ takes the sheaves in $E_U$ to
sheaves, rather than complexes.

\begin{lem} Let\/ $U$ be a finite type\/ $\K$-scheme and\/ $E_U$ a
coherent sheaf on\/ $X\times U$ flat over $U,$ that is, a $U$-family
of coherent sheaves on $X$. Then for $n\gg 0,$ $F_U=T_n(U)$ is also
a $U$-family of coherent sheaves on\/~$X$.
\label{dt8lem1}
\end{lem}

\begin{proof} Since $U$ is of finite type, the family of coherent
sheaves $E_U$ is bounded, so there exists $n\gg 0$ such that
$H^i(E_u(n))=0$ vanishes for all $u\in U$ and $i>0$, and $E_u(n)$ is
globally generated. Then we have
\e
\begin{split}
T_n(E_u)&=\cone\bigl(\ts\bigop_{i\ge 0}
\Ext^i(\cO_X(-n),E_u)\ot\cO_X(-n)[-i]\longra E_u\bigr)[-1] \\
&=\cone\bigl(\Hom(\cO_X(-n),E_u)
\ot\cO_X(-n)\longra E_u\bigr)[-1] \\
&=\Ker\bigl(\Hom(\cO_X(-n),E_u)\ot \cO_X(-n)\longra E_u\bigr),
\end{split}
\label{dt8eq1}
\e
where the first line is from the definition, the second as
$H^i(E_u(n))=0$ for $i>0$, and the third as $\Hom(\cO_X(-n),E_u)
\ot\cO_X(-n)\ra E_u$ is surjective in $\coh(X)$ as $E_u(n)$ is
globally generated. Thus $F_u=T_n(E_u)$ is a sheaf, rather than a
complex, for all $u\in U$. In sheaves on $X\times U$, we have an
exact sequence:
\begin{equation*}
\xymatrix{ 0 \ar[r] &F_U \ar[r] & p_X^*p_{X, *}(E_U(n)) \ot
p_X^*(\cO_X(-n)) \ar[r] &E_U \ar[r] &0, }
\end{equation*}
and $F_U$ is flat over $U$ as $E_U$ and $p_X^*p_{X,
*}(E_U(n))\ot p_X^*(\cO_X(-n))$ are.
\end{proof}

We recall the notion of {\it homological dimension\/} of a sheaf, as
in Hartshorne \cite[p.~238]{Hart2} and Huybrechts and
Lehn~\cite[p.~4]{HuLe2}:

\begin{dfn} For a nonzero sheaf $E$ in $\coh(X)$, the {\it homological
dimension\/}\index{sheaf!homological dimension|(}\index{coherent
sheaf!homological dimension|(}
$\hd(E)$\nomenclature[hd(E)]{$\hd(E)$}{homological dimension of a coherent
sheaf $E$} is the smallest $n\ge 0$ for which there exists an exact
sequence in $\coh(X)$
\begin{equation*}
0 \ra V_n \ra V_{n-1} \ra\cdots \ra V_0\ra E\ra 0,
\end{equation*}
with $V_0,\ldots,V_n$ vector bundles (locally free sheaves). Clearly
$\hd(E)=0$ if and only if $E$ is a vector bundle. Equivalently,
$\hd(E)$ is the largest $n\ge 0$ such that $\Ext^n(E,\cO_x)=0$ for
some $x\in X$, where $\cO_x$ is the skyscraper sheaf at $x$. Since
$X$ is smooth of dimension $m$ we have $\hd(E)\le m$ for all
$E\in\coh(X)$.
\label{dt8def2}
\end{dfn}

The operators $T_n$ above decrease $\hd(E)$ by 1 when $n\gg 0$,
unless $\hd(E)=0$ when they fix~$\hd(E)$.

\begin{lem} Let\/ $U,E_U$ and\/ $n\gg 0$ be as in Lemma\/
{\rm\ref{dt8lem1}}. Then for all\/ $u\in U$ we have\/
$\hd(T_n(E_u))=\max(\hd(E_u)-1,0)$.
\label{dt8lem2}
\end{lem}

\begin{proof} As in \eq{dt8eq1} we have an exact sequence
\begin{equation*}
0\longra T_n(E_u)\ra H^0(E_u(n))\ot\cO_X(-n)\longra E_u\longra 0.
\end{equation*}
Applying $\Ext(-,\cO_x)$ to this sequence for $x\in X$ gives a long
exact sequence
\begin{align*}
\cdots &\longra H^0(E_u(n))\ot \Ext^i(\cO_X(-n),\cO_x) \longra
\Ext^i(T_n(E_u),\cO_x) \\
&\longra \Ext^{i+1}(E_u,\cO_x)\longra
H^0(E_u(n))\ot \Ext^{i+1}(\cO_X(-n),\cO_x) \longra\cdots.
\end{align*}
Thus $\Ext^i(T_n(E_u),\cO_x)\cong\Ext^{i+1}(E_u,\cO_x)$ for $i>0$,
as $\Ext^i(\cO_X(-n),\cO_x)=0$ for $i>0$. Since $\hd(E_u)$ is the
largest $n\ge 0$ such that $\Ext^n(E,\cO_x)=0$ for some $x\in X$,
this implies that $\hd(T_n(E_u))=\hd(E_u)-1$ if $\hd(E_u)>0$, and
$\hd(T_n(E_u))=0$ if~$\hd(E_u)=0$.
\end{proof}

\begin{cor} Let\/ $U$ be a finite type\/ $\K$-scheme and\/ $E_U$ a
$U$-family of coherent sheaves on $X$. Then there exist\/
$n_1,\ldots,n_m\gg 0$ such that\/ $T_{n_m}\ci
T_{n_{m-1}}\ci\cdots\ci T_{n_1}(E_U)$ is a $U$-family of vector
bundles on~$X$.
\label{dt8cor1}
\end{cor}

\begin{proof} Apply Lemma \ref{dt8lem1} $m$ times to $E_U$, where
by induction on $i=1,\ldots,m$, $n_i$ is the $n$ in Lemma
\ref{dt8lem1} applied to the $U$-family of sheaves
$T_{n_{i-1}}\ci\cdots\ci T_{n_1}(E_U)$. For each $u\in U$ we have
$\hd(E_u)\le m$, since $X$ is smooth of dimension $m$. So Lemma
\ref{dt8lem2} implies that $\hd\bigl(T_{n_1}(E_u)\bigr)\le m-1$, and
by induction $\hd\bigl(T_{n_i}\ci\cdots\ci T_{n_1}(E_u)\bigr)\le
m-i$ for $i=1,\ldots,m$. Hence $\hd\bigl(T_{n_m}\ci\cdots\ci
T_{n_1}(E_u)\bigr)=0$, so that $T_{n_m}\ci\cdots\ci T_{n_1}(E_u)$ is
a vector bundle on $X$ for all~$u\in U$.
\end{proof}\index{sheaf!homological dimension|)}\index{coherent
sheaf!homological dimension|)}

We can now prove Theorem \ref{dt5thm1}. Let $\fU$ be an open, finite
type substack of $\fM$. Then $\fU$ admits an atlas $\pi:U\ra\fU$,
with $U$ a finite type $K$-scheme. Let $E_U$ be the corresponding
$U$-family of coherent sheaves on $X$. Then $E_U$ is a versal family
of coherent sheaves on $X$ which parametrizes all the sheaves
represented by points in $\fU$, up to isomorphism. Let
$n_1,\ldots,n_m$ be as in Corollary \ref{dt8cor1} for these $U,E_U$.
Then by Lemma \ref{dt8lem1} applied $m$ times and Corollary
\ref{dt8cor1}, $F_U=T_{n_m}\ci\cdots\ci T_{n_1}(E_U)$ is a
$U$-family of vector bundles. Thus $F_U$ corresponds to a 1-morphism
$\pi':U\ra\fVect$.

Since Fourier--Mukai transforms\index{Fourier--Mukai transform} and
shifts preserve moduli families of (complexes of) sheaves, and $E_U$
is a versal family, $F_U$ is a versal family.\index{versal family} Hence
$\pi'$ is an atlas\index{Artin stack!atlas} for an open substack $\fV$
in $\fVect$. If $S$ is a $\K$-scheme then $\Hom(S,\fU)$ is the
category of $S$-families $E_S$ of coherent sheaves on $X$ which lie
in $\fU$ for all $s\in S$, and $\Hom(S,\fV)$ the category of
$S$-families $F_S$ of vector bundles on $X$ which lie in $\fV$ for
all $s\in S$. Then $E_S\mapsto T_{n_m}\ci\cdots\ci T_{n_1}(E_S)$
defines an equivalence of categories $\Hom(S,\fU)\ra\Hom(S,\fV)$.
Hence $T_{n_m}\ci\cdots\ci T_{n_1}$ induces a 1-isomorphism
$\vp:\fU\ra\fV$, proving the first part of Theorem \ref{dt5thm1}.
The second part follows by passing to coarse moduli
spaces.\index{Seidel--Thomas twist|)}\index{coherent sheaf|)}\index{vector
bundle!algebraic|)}

\section[The proofs of Theorems $\text{\ref{dt5thm2}}$ and
$\text{\ref{dt5thm3}}$]{The proofs of Theorems \ref{dt5thm2} and
\ref{dt5thm3}}
\label{dt9}\index{complex analytic space|(}

To prove Theorem \ref{dt5thm2} we will need a local description of
the complex analytic space $\Vect_\rsi(\C)$ underlying the coarse
moduli space $\Vect_\rsi$ of simple algebraic vector bundles on a
projective Calabi--Yau 3-fold $X$, in terms of gauge theory\index{gauge
theory} on a complex vector bundle $E\ra X$, and
infinite-dimensional Sobolev spaces of sections of
$\End(E)\ot\La^{0,q}T^*X$. For Theorem \ref{dt5thm3} we will need a
similar local description for the moduli stack $\fVect$ of algebraic
vector bundles on $X$. Fortunately, there is already a substantial
literature on this subject, mostly aimed at proving the {\it
Hitchin--Kobayashi correspondence},\index{Hitchin--Kobayashi
correspondence} so we will be able to quote many of the results we
need.

Some background references are Hartshorne \cite[App.~B]{Hart2} on
complex analytic spaces (in finite dimensions) and the functor to
them from $\C$-schemes, Laumon and Moret-Bailly \cite{LaMo} on Artin
stacks, and Lang \cite{Lang} on Banach manifolds.\index{Banach manifold}
The general theory of analytic functions on infinite-dimensional
spaces, and (possibly infinite-dimensional) complex analytic spaces
is developed in Douady \cite{Doua1,Doua2}, and summarized in
\cite[\S 4.1.3]{FrMo} and \cite[\S 7.5]{LuTe}. Some books covering
much of \S\ref{dt91}--\S\ref{dt95} are Kobayashi \cite[\S
VII.3]{Koba}, L\"ubke and Teleman \cite[\S 4.1 \& \S 4.3]{LuTe}, and
Friedman and Morgan \cite[\S 4.1--\S 4.2]{FrMo}. Our main reference
is Miyajima \cite{Miya}, who proves that the complex-algebraic and
gauge-theoretic descriptions of $\Vect_\rsi(\C)$ are isomorphic as
complex analytic spaces.

Let $X$ be a projective complex algebraic manifold of dimension $m$.
Then Miyajima considers three different moduli problems:
\begin{itemize}
\setlength{\itemsep}{0pt}
\setlength{\parsep}{0pt}
\item The moduli of {\it holomorphic structures}\index{vector
bundle!holomorphic structure} on a fixed $C^\iy$ complex vector
bundle $E\ra X$. For simple holomorphic structures we form the
coarse moduli space $\Hol_\rsi(E)=\bigl\{\db_E\in\sAs:
\db_E^2=0\bigr\}/\sG$, a complex analytic space.
\item The moduli of {\it complex analytic vector
bundles\/}\index{vector bundle!analytic} over $X$. For simple vector
bundles we form a coarse moduli space $\Vect_\rsi^\an$, a
complex analytic space.
\item The moduli of {\it complex algebraic vector
bundles\/}\index{vector bundle!algebraic} over $X$. For simple
vector bundles we form a coarse moduli space $\Vect_\rsi$, a
complex algebraic space. For all vector bundles we form a moduli
stack $\fVect$, an Artin $\C$-stack.
\end{itemize}

Miyajima \cite[\S 3]{Miya} proves that $\Hol_\rsi(E)\cong
\Vect_\rsi^\an \cong\Vect_\rsi(\C)$ locally as complex analytic
spaces. Presumably one can also prove analogous results for moduli
stacks of all vector bundles, working in some class of analytic
$\C$-stacks, but the authors have not found references on this in
the literature. Instead, to prove what we need about the moduli
stack $\fVect$, we will express our results in terms of {\it versal
families\/}\index{versal family} of objects.

Sections \ref{dt91}, \ref{dt92}, \ref{dt94} and \ref{dt95} explain
moduli spaces of holomorphic structures, of analytic vector bundles,
and of algebraic vector bundles, respectively, and the isomorphisms
between them. Section \ref{dt93} is an aside on existence of local
atlases for $\fM$ with group-invariance properties. All of
\S\S\ref{dt91}, \ref{dt92} and \ref{dt94} is from Miyajima
\cite{Miya} and other sources, or is easily deduced from them.
Sections \ref{dt96}--\ref{dt98} prove Theorems \ref{dt5thm2}
and~\ref{dt5thm3}.

\subsection{Holomorphic structures on a complex vector bundle}
\label{dt91}\index{vector bundle!holomorphic structure|(}

Let $X$ be a compact complex manifold of complex dimension $m$. Fix
a nonzero $C^\iy$ complex vector bundle $E\ra X$ of rank $l>0$. That
is, $E$ is a smooth vector bundle whose fibres have the structure of
complex vector spaces isomorphic to $\C^l$, but $E$ does not (yet)
have the structure of a holomorphic vector bundle. Here are some
basic definitions.\index{gauge theory|(}

\begin{dfn}  A (smooth) {\it
semiconnection\/}\index{semiconnection|(}\index{d-operator@$\db$-operator|(}
(or $\db$-{\it operator}) is a first order differential operator
$\db_E:C^\iy(E)\ra C^\iy(E\ot_\C\La^{0,1}T^*X)$
\nomenclature[dE]{$\db_E$}{semiconnection on complex vector bundle
$E$}satisfying the Leibnitz rule $\db_E(f\cdot e)=e\ot(\db
f)+f\cdot\db_Ee$ for all smooth $f:X\ra\C$ and $e\in C^\iy(E)$,
where $\db$ is the usual operator on complex functions. They are
called semiconnections since they arise as the projections to the
$(0,1)$-forms $\La^{0,1}T^*X$ of connections $\nabla:C^\iy(E)\ra
C^\iy\bigl(E\ot_\C(T^*X\ot_\R\C)\bigr)$, so they are half of an
ordinary connection.

Any semiconnection $\db_E:C^\iy(E)\ra C^\iy(E\ot_\C\La^{0,1}T^*X)$
extends uniquely to operators $\db_E:C^\iy(E\ot_\C\La^{p,q}T^*X)\ra
C^\iy(E\ot_\C\La^{p,q+1}T^*X)$ for all $0\le p\le m$ and $0\le q<m$
satisfying $\db_E(e\wedge \al)=(-1)^{r+s}e\ot\db\al+(\db_Ee)\w\al$
for all smooth $e\in C^\iy(E\ot_\C\La^{r,s}T^*X)$ and $\al\in
C^\iy(\La^{p-r,q-s}T^*X)$ with $0\le r\le p$, $0\le s\le q$. In
particular we can consider the composition
\begin{equation*}
\xymatrix{ C^\iy(E) \ar[r]^(0.35){\db_E} &
C^\iy(E\ot_\C\La^{0,1}T^*X) \ar[r]^{\db_E} &
C^\iy(E\ot_\C\La^{0,2}T^*X).}
\end{equation*}
The composition $\db_E^2$ can be regarded as a section of
$C^\iy\bigl(\End(E)\ot_\C\La^{0,2}T^*X\bigr)$ called the
$(0,2)$-{\it curvature}, analogous to the curvature of a connection.

The semiconnection $\db_E$ defines a {\it holomorphic
structure}\index{vector bundle!holomorphic structure} on $E$ if
$\db_E^2=0$. That is, if $U$ is an open set in $X$ (in the complex
analytic topology) we can define ${\cal E}(U)=\bigl\{e\in
C^\iy(E\vert_U):\db_Ee=0\bigr\}$, and if $V\subseteq U\subseteq X$
are open we have a restriction map $\rho_{UV}:{\cal E}(U)\ra{\cal
E}(V)$. Then $\cal E$ is a {\it complex analytic coherent
sheaf\/}\index{coherent sheaf!complex analytic} on $X$. Since
$\db_E^2=0$ we can use the Newlander--Nirenberg
Theorem\index{Newlander--Nirenberg Theorem} to show that near each $x\in
X$ there exists a basis of holomorphic sections for $E$, and thus
$\cal E$ is locally free, that is, it is an {\it analytic vector
bundle}.\index{vector bundle!analytic} Conversely, given a locally free
complex analytic coherent sheaf $\cal E$ on $X$, we can write it as
a subsheaf of the sheaf of smooth sections of a complex vector
bundle $E\ra X$, and then there is a unique semiconnection $\db_E$
on $E$ with $\db_E^2=0$ such that if $U\subseteq X$ is open and
$e\in C^\iy(U)$ then $e\in{\cal E}(U)$ if and only if~$\db_Ee=0$.

Fix a semiconnection $\db_E$ with $\db_E^2=0$. Then any other
semiconnection $\db_E'$ may be written uniquely as $\db_E+A$ for
$A\in C^\iy\bigl(\End(E)\ot_\C\La^{0,1}T^*X\bigr)$. Thus the set
$\sA$\nomenclature[Aa]{$\sA$}{affine space of smooth semiconnections on a
vector bundle} of smooth semiconnections on $E$ is an
infinite-dimensional affine space. The $(0,2)$-curvature of
$\db_E'=\db_E+A$ is
\begin{equation*}
F_A^{0,2}=\db_EA+A\w A.
\end{equation*}
Here to form $\db_EA$ we extend the action of $\db_E$ on $E$ to
$\End(E)\ot_\C\La^{0,1}T^*X=E\ot_\C E^*\ot_\C\La^{0,1}T^*X$ in the
natural way, and $A\w A$ combines the Lie bracket on $\End(E)$ with
the wedge
product~$\w:\La^{0,1}T^*X\times\La^{0,1}T^*X\ra\La^{0,2}T^*X$.

Write $\Aut(E)$ for the subbundle of invertible elements in
$\End(E)$. It is a smooth bundle of complex Lie groups over $X$,
with fibre $\GL(l,\C)$. Define the {\it gauge group\/}\index{gauge
group} $\sG=C^\iy\bigl(\Aut(E)\bigr)$\nomenclature[Ga]{$\sG$}{gauge group of
smooth gauge transformations of a vector bundle} to be the space of
smooth sections of $\Aut(E)$. It is an infinite-dimensional Lie
group, with Lie algebra $\g=C^\iy(\End(E))$. It acts on the right on
$\sA$ by $\ga:\db_E'\mapsto\db_E'{}^\ga=\ga^{-1}\ci\db_E'\ci\ga$.
That is, $\db_E'{}^\ga$ is the first order differential operator
$C^\iy(E)\ra C^\iy(E\ot_\C\La^{0,1}T^*X)$ acting by $e\mapsto
\ga^{-1}\cdot\bigl(\db_E'(\ga\cdot e)\bigr)$. One can show that
$\db_E'{}^\ga$ satisfies the Leibnitz rule, so that
$\db_E'{}^\ga\in\sA$, and this defines an action of $\sG$ on $\sA$.
Writing $\db_E'=\db_E+A$ we have
\e
(\db_E+A)^\ga=\db_E+(\ga^{-1}\ci A\ci\ga+\ga^{-1}\db\ga).
\label{dt9eq1}
\e

Write $\Stab_\sG(\db_E')$ for the stabilizer group of $\db_E'\in\sA$
in $\sG$. It is a complex Lie group with Lie algebra
\begin{align*}
\stab_\sG(\db_E')&=\Ker\bigl(\db_E':C^\iy(\End(E))\ra C^\iy(\End(E)
\ot_\C\La^{0,1})\bigr)\\
&=\Ker\bigl((\db_E')^*\db_E':C^\iy(\End(E))\ra C^\iy(\End(E))\bigr),
\end{align*}
which is the kernel of an elliptic operator on a compact manifold,
and so is finite-dimensional. In fact $\stab_\sG(\db_E')$ is a
finite-dimensional $\C$-algebra, and $\Stab_\sG(\db_E')$ is the
group of invertible elements in $\stab_\sG(\db_E')$. If $\db_E'$ is
a holomorphic structure then $\stab_\sG(\db_E')$ is the sheaf
cohomology group~$H^0(\End(E,\db_E'))$.

The multiples of the identity $\bG_m\cdot\id_E$ in $\sG$ act
trivially on $\sA$, so $\bG_m\cdot\id_E\subseteq\Stab_\sG(\db_E')$
for all $\db_E'\in\sA$. Call a semiconnection $\db_E'$ {\it simple}
if $\Stab_\sG(\db_E')=\bG_m\cdot\id_E$. Write
$\sAs$\nomenclature[Ab]{$\sAs$}{open subset of simple semiconnections in $\sA$}
for the subset of simple $\db_E'$ in $\sA$. It is a $\sG$-invariant
open subset of $\sA$, in the natural topology.

Now $\sA,\sAs,\sG$ have the disadvantage that they are not Banach
manifolds. Choose Hermitian metrics $h_X$ on $X$ and $h_E$ on the
fibres of $E$. As in Miyajima \cite[\S 1]{Miya}, fix an integer
$k>2m+1$, and write $\sA^{2,k}, \sA_\rsi^{2,k}$
\nomenclature[Ac]{$\sA^{2,k}$}{affine space of $L^2_k$ semiconnections on a
vector bundle}\nomenclature[Ad]{$\sA_\rsi^{2,k}$}{open subset of simple
semiconnections in $\sA^{2,k}$} for the completions of $\sA,\sAs$ in
the Sobolev norm $L^2_k$, and
$\sG^{2,k+1}$\nomenclature[Gb]{$\sG^{2,k+1}$}{gauge group of $L^2_{k+1}$ gauge
transformations of a vector bundle} for the completion of $\sG$ in
the Sobolev norm $L^2_{k+1}$, defining norms using $h_X,h_E$. Then
\e
\sA^{2,k}=\bigl\{\db_E+A:A\in L^2_k(\End(E)
\ot_\C\La^{0,1}T^*X)\bigr\},
\label{dt9eq2}
\e
Also $\sA^{2,k},\sA_\rsi^{2,k}$ are complex Banach manifolds, and
$\sG^{2,k+1}$ is a complex Banach Lie group acting holomorphically
on $\sA^{2,k},\sA_\rsi^{2,k}$ by~\eq{dt9eq1}.

Define $P_k:\sA^{2,k}\ra L^2_{k-1}\bigl(\End(E)
\ot_\C\La^{0,1}T^*X\bigr)$ by
\e
P_k:\db_E+A\longmapsto F_A^{0,2}=\db_EA+A\w A.
\label{dt9eq3}
\e
Using the Sobolev Embedding Theorem\index{Sobolev Embedding Theorem} we
see that $P_k$ is a well-defined, holomorphic map between complex
Banach manifolds.\index{Banach manifold}
\label{dt9def1}
\end{dfn}

\begin{dfn} A {\it family of holomorphic structures\/} $(T,\tau)$
on $E$ is a (finite-dimensional) complex analytic space $T$ and a
complex analytic map of complex analytic spaces $\tau:T\ra
P_k^{-1}(0)$, where $P_k^{-1}(0)\subset\sA^{2,k}$ as above. Two
families $(T,\tau),(T,\tau')$ with the same base $T$ are {\it
equivalent\/} if there exists a complex analytic map
$\si:T\ra\sG^{2,k+1}$ such that $\tau'\equiv \si\cdot\tau$, using
the product $\cdot:\sG^{2,k+1}\times\sA^{2,k}\ra\sA^{2,k}$ which
restricts to $\cdot:\sG^{2,k+1}\times P_k^{-1}(0)\ra P_k^{-1}(0)$.

A family $(T,\tau)$ is called {\it versal at\/}\index{versal
family|(}\index{universal family|(} $t\in T$ if whenever $(T',\tau')$ is
a family of holomorphic structures on $E$ and $t'\in T'$ with
$\tau'(t')=\tau(t)$, there exists an open neighbourhood $U'$ of $t'$
in $T'$ and complex analytic maps $\up:U'\ra T$ and
$\si:U'\ra\sG^{2,k+1}$ such that $\up(t')=t$, $\si(t')=\id_E$, and
$\tau\ci\up\equiv \si\cdot\tau'\vert_{U'}$ as complex analytic maps
$U'\ra P^{-1}(0)$. We call $(T,\tau)$ {\it universal at\/} $t\in T$
if in addition the map $\up:U'\ra T$ is unique, provided the
neighbourhood $U'$ is sufficiently small. (Note that we do not
require $\si$ to be unique. Thus, this notion of universal is
appropriate for defining a coarse moduli space, not a fine moduli
space or moduli stack.) The family $(T,\tau)$ is called {\it
versal\/} (or {\it universal\/}) if it is versal (or universal) at
every~$t\in T$.
\label{dt9def2}
\end{dfn}

Fix a smooth holomorphic structure $\db_E$ on $E$, as above. In
\cite[Th.~1]{Miya}, Miyajima constructs a versal family of
holomorphic structures $(T,\tau)$ containing $\db_E$. We now explain
his construction. Write $\db_E^*$ for the formal adjoint of $\db_E$
computed using the Hermitian metrics $h_X$ on $X$ and $h_E$ on the
fibres of $E$. Then $\db_E^*:C^\iy (E\ot_\C \La^{p,q+1}T^*X)\ra
C^\iy(E\ot_\C \La^{p,q}T^*X)$ for all $p,q$ is a first order
differential operator such that $\langle\db_E
e,e'\rangle_{L^2}=\langle e,\db_E^*e' \rangle_{L^2}$ for all $e\in
C^\iy(E\ot_\C\La^{p,q}T^*X)$ and $e'\in
C^\iy(E\ot_\C\La^{p,q+1}T^*X)$, where $\langle\,,\,\rangle_{L^2}$ is
the $L^2$ inner product defined using $h_X,h_E$. Also $\db_E^*$
extends to Sobolev spaces~$L^2_k$.

Using Hodge theory for $\bigl(C^\iy(\End(E)\ot_\C \La^{0,*}T^*X),
\db_E\bigr)$, we give expressions for the Ext groups of the
holomorphic vector bundle $(E,\db_E)$ with itself:
\begin{align*}
&\Ext^q\bigl((E,\db_E),(E,\db_E)\bigr)\\
&\;\cong\frac{\Ker\bigl(\db_E:C^\iy(\End(E)\ot_\C\La^{0,q}T^*X)\ra
C^\iy(\End(E)\ot_\C\La^{0,q+1}T^*X)\bigr)}
{\Im\bigl(\db_E:C^\iy(\End(E)\ot_\C\La^{0,q-1}T^*X)\ra
C^\iy(\End(E)\ot_\C\La^{0,q}T^*X)\bigr)}\\
&\;\cong\bigl\{e\in C^\iy(\End(E)\ot_\C\La^{0,q}T^*X):
\db_E^{\phantom{*}}e=\db_E^*e=0\bigr\}\\
&\;=\bigl\{e\in C^\iy(\End(E)\ot_\C\La^{0,q}T^*X):
(\db_E^{\phantom{*}}\db_E^*+\db_E^*\db_E^{\phantom{*}})e=0\bigr\}.
\end{align*}
Hence the finite-dimensional complex vector space
\begin{equation*}
\sE^q=\bigl\{e\in C^\iy(\End(E)\ot_\C\La^{0,q}T^*X):
(\db_E^{\phantom{*}}\db_E^*+\db_E^*\db_E^{\phantom{*}})e=0\bigr\}
\end{equation*}
is isomorphic to $\Ext^q\bigl((E,\db_E),(E,\db_E)\bigr)$. Miyajima
\cite[\S 1]{Miya} proves:

\begin{prop}{\bf(a)} In the situation above, for sufficiently
small\/ $\ep>0,$
\e
\begin{split}
Q_\ep=\bigl\{\db_E+A:\,& A\in L^2_k(\End(E)\ot_\C\La^{0,1}T^*X),
\quad \nm{A}_{L^2_k}<\ep,\\
&\db_E^*A=0, \quad \db_E^*(\db_E^{\phantom{*}}A+A\w A)=0\bigr\}
\end{split}
\label{dt9eq4}
\e
is \kern -.2em a \kern -.2em finite-dimensional \kern -.2em complex
\kern -.2em submanifold \kern -.2em of\/ \kern -.2em $\sA^{2,k},$
\kern -.2em of \kern -.2em complex \kern -.2em dimension \kern -.2em
$\dim\Ext^1\ab\bigl((E,\db_E),(E,\db_E)\bigr),$ such that\/
$\db_E\in Q_\ep$ and\/ $T_{\db_E}Q_\ep=\sE^1$. Furthermore,
$Q_\ep\subset\sA\subset\sA^{2,k},$ that is, if\/ $\db_E+A\in Q_\ep$
then $A$ is smooth.
\smallskip

\noindent{\bf(b)} Now define $\pi:Q_\ep\ra\sE^2$ by
$\pi:\db_E+A\mapsto\pi_{\sE^2}(\db_EA+A\w A),$ where
$\pi_{\sE^2}:L^2_{k-1}(\End(E)\ot_\C\La^{0,2}T^*X)\ra\sE^2$ is
orthogonal projection using the $L^2$ inner product. Then $\pi$ is a
holomorphic map of finite-dimensional complex manifolds. Let\/
$T=\pi^{-1}(0),$ as a complex analytic subspace of\/ $Q_\ep$. Then
$T=Q_\ep\cap P_k^{-1}(0)$, as an intersection of complex analytic
subspaces in\/ $\smash{\sA^{2,k},}$ so $T$ is a complex analytic
subspace of\/ $P_k^{-1}(0)$. Also $t=\db_E\in T,$ with\/
$\tau(t)=\db_E,$ and the Zariski tangent space $T_tT$ is
$\sE^1\cong\Ext^1\bigl((E,\db_E),(E,\db_E)\bigr)$.
\smallskip

\noindent{\bf(c)} Making $\ep$ smaller if necessary, $(T,\tau)$ is a
\begin{bfseries}versal\end{bfseries} family of smooth holomorphic
structures on $E,$ which includes $\db_E$. If\/ $\db_E$ is simple,
then $(T,\tau)$ is a \begin{bfseries}universal\end{bfseries} family
of smooth, simple holomorphic structures on~$E$.
\label{dt9prop1}
\end{prop}

This gives the standard Kuranishi picture: there exists a versal
family of deformations of $\db_E$, with base space the zeroes of a
holomorphic map from $\Ext^1\bigl((E,\db_E),\ab(E,\db_E)\bigr)$ to
$\Ext^2\bigl((E,\db_E),(E,\db_E)\bigr)$. Here is a sketch of the
proof.

For (a), we consider the nonlinear elliptic operator
$F:L^2_k(\End(E)\ot_\C\La^{0,1}T^*X)\ab\ra L^2_{k-2}(\End(E)\ot_\C
\La^{0,1}T^*X)$ mapping $F:A\mapsto (\db_E^{\phantom{*}}\db_E^*
+\db_E^*\db_E^{\phantom{*}})A +\db_E^*(A\w A)$. The image of $F$
lies in the orthogonal subspace $(\sE^1)^\perp$ to $\sE^1$ in
$L^2_{k-2} (\End(E)\ot_\C\La^{0,1}T^*X)$, using the $L^2$ inner
product. So we can consider $F$ as mapping
$F:L^2_k(\End(E)\ot_\C\La^{0,1}T^*X)\ra (\sE^1)^\perp$. The
linearization of $F$ at $A=0$ is then surjective, with kernel
$\sE^1$. Part (a) then follows from the Implicit Function Theorem
for Banach spaces, together with elliptic regularity for $F$ to
deduce smoothness in the last part.

For (b), one must show that $(P_k\vert_{Q_\ep})^{-1}(0)$ and
$\pi^{-1}(0)$ coincide as complex analytic subspaces of $Q_\ep$.
Since $\pi$ factors through $P_k$ we have
$(P_k\vert_{Q_\ep})^{-1}(0)\subseteq\pi^{-1}(0)$ as complex analytic
subspaces. It is enough to show that any local holomorphic function
$Q_\ep\ra\C$ of the form $f\ci P_k$ for a local holomorphic function
$f:L^2_{k-1}\bigl(\End(E)\ot_\C\La^{0,2}T^*X\bigr)\ra\C$ may also be
written in the form $\ti f\ci\pi$ for a local holomorphic function
$\ti f:\sE^2\ra\C$.

For (c), the main point is that the condition $\db_E^*A=0$ is a
`slice' to the action of $\sG^{2,k+1}$ on $\sA^{2,k}$ at $\db_E$.
That is, the Hilbert submanifold $\{\db_E+A:\db_E^*A=0\}$ in
$\sA^{2,k}$ intersects the orbit $\sG^{2,k+1}\cdot\db_E$
transversely, and it also intersects every nearby orbit of
$\sG^{2,k+1}$ in $\sA^{2,k}$. The complex analytic space $T$ is
exactly the intersection (as Douady complex analytic subspaces of
$\sA^{2,k}$) of $P_k^{-1}(0)$, the slice $\{\db_E+A:\db_E^*A=0\}$,
and the ball of radius $\ep$ around $\db_E$ in $\sA^{2,k}$. The
point of introducing $Q_\ep,\sE^1,\sE^2,\pi$ is to describe this
complex analytic space $T$ in strictly finite-dimensional
terms.\index{vector bundle!holomorphic
structure|)}\index{semiconnection|)}\index{d-operator@$\db$-operator|)}\index{gauge
theory|)}

\subsection{Moduli spaces of analytic vector bundles on $X$}
\label{dt92}\index{vector bundle!analytic|(}

Let $X$ be a compact complex manifold. Here is the analogue of
Definition \ref{dt9def2} for analytic vector bundles.

\begin{dfn} A {\it family of analytic vector bundles\/} $(T,{\cal
F})$ on $X$ is a (finite-dimensional) complex analytic space $T$ and
a complex analytic vector bundle ${\cal F}$ over $X\times T$ which
is flat over $T$. For each $t\in T$, the {\it fibre\/} ${\cal F}_t$
of the family is ${\cal F}\vert_{X\times\{t\}}$, regarded as a
complex analytic vector bundle over~$X\cong X\times\{t\}$.

A family $(T,{\cal F})$ is called {\it versal at\/} $t\in T$ if
whenever $(T',{\cal F}')$ is a family of analytic vector bundles on
$X$ and $t'\in T'$ with ${\cal F}_t\cong{\cal F}'_{t'}$ as analytic
vector bundles on $X$, there exists an open neighbourhood $U'$ of
$t'$ in $T'$, a complex analytic map $\up:U'\ra T$ with $\up(t')=t$
and an isomorphism $\up^*({\cal F})\cong{\cal F}'\vert_{X\times U'}$
as vector bundles over~$X\times U'$.

It is called {\it universal at\/} $t\in T$ if in addition the map
$\up:U'\ra T$ is unique, provided the neighbourhood $U'$ is
sufficiently small. (Note that we do not require the isomorphism
$\up^*({\cal F})\cong{\cal F}'\vert_{X\times U'}$ to be unique.) The
family $(T,{\cal F})$ is called {\it versal\/} (or {\it
universal\/}) if it is versal (or universal) at every~$t\in T$.
\label{dt9def3}
\end{dfn}

In a parallel result to Proposition \ref{dt9prop1}(c), Forster and
Knorr \cite{FoKn} prove that any analytic vector bundle on $X$ can
be extended to a versal family of analytic vector bundles. Then
Miyajima \cite[\S 2]{Miya} proves:

\begin{prop} Let\/ $X$ be a compact complex manifold, $E\ra X$ a\/
$C^\iy$ complex vector bundle, and\/ $\db_E$ a holomorphic structure
on $E,$ so that $(E,\db_E)$ is an analytic vector bundle over $X$.
Let\/ $(T,\tau)$ be the versal family of holomorphic structures on
$E$ containing $\db_E$ constructed in
Proposition\/~{\rm\ref{dt9prop1}}.

Then there exists a versal family of analytic vector bundles
$(T,{\cal F})$ over $X,$ and an isomorphism ${\cal F}\ra E\times T$
of\/ $C^\iy$ complex vector bundles over $X\times T$ which induces
the family of holomorphic structures $(T,\tau)$. If\/ $(E,\db_E)$ is
simple then $(T,{\cal F})$ is a universal family of simple analytic
vector bundles.
\label{dt9prop2}
\end{prop}

Here is an idea of the proof. Let $(T,{\cal F})$ be a family of
analytic vector bundles over $X$, let $t\in T$, and let $E\ra X$ be
the complex vector bundle underlying the analytic vector bundle
${\cal F}_t\ra X$. Then for some small open neighbourhood $U$ of $t$
in $T$, we can identify ${\cal F}\vert_{X\times U}$ with $(E\times
U)\ra(X\times U)$ as complex vector bundles, where $(E\times
U)\ra(X\times U)$ is the pullback of $E$ from $X$ to $X\times U$.

Thus, the analytic vector bundle structure on ${\cal
F}\vert_{X\times U}$ induces an analytic vector bundle structure on
$(E\times U)\ra(X\times U)$. We can regard this as a first order
differential operator $\db_{E,U}:C^\iy(E)\ra
C^\iy\bigl(E\ot\La^{0,1}T^*X\op E\ot\La^{0,1}T^*U\bigr)$ on bundles
over $X\times U$. Thus, $\db_{E,U}$ has two components, a
$\db$-operator in the $X$ directions and a $\db$-operator in the $U$
directions in $X\times U$. The first of these components is a {\it
family of holomorphic structures\/} $(U,\tau)$ on~$E$.

Therefore, by choosing a (local) trivialization in the
$T$-directions, a family $(T,{\cal F})$ of analytic vector bundles
induces a family $(T,\tau)$ of holomorphic structures on $E$, by
forgetting part of the structure. Conversely, given a family
$(T,\tau)$ of holomorphic structures on $E$, we can try to add extra
structure, a $\db$-operator in the $T$ directions in $X\times T$, to
make $(T,\tau)$ into a family of analytic vector bundles $(T,{\cal
F})$. Miyajima proves that this can be done, and that the local
deformation functors are isomorphic. Hence the (uni)versal family in
Proposition \ref{dt9prop1} lifts to a (uni)versal family of analytic
vector bundles.\index{complex analytic space|)}\index{versal
family|)}\index{universal family|)}\index{vector bundle!analytic|)}

\subsection{Constructing a good local atlas $S$ for $\fM$ near
$[E]$}
\label{dt93}\index{Artin stack!atlas|(}

We divert briefly from our main argument to prove the first part of
the second paragraph of Theorem \ref{dt5thm3}, the existence of a
1-morphism $\Phi:[S/G^{\sst\C}]\ra\fM$ for $\fM$ satisfying various
conditions. Let $X,\fM,E,\Aut(E)$ and $G^{\sst\C}$ be as in
Theorem~\ref{dt5thm3}.

\begin{prop} There exists a finite type open $\C$-substack\/
$\mathfrak Q$ in $\fM$ with\/ $[E]\in{\mathfrak Q}(\C)\subset
\fM(\C)$ and a $1$-isomorphism ${\mathfrak Q}\cong [Q/H],$ where $Q$
is a finite type $\C$-scheme and\/ $H$ an algebraic $\C$-group
acting on $Q$. Furthermore we may realize $Q$ as a quasiprojective
$\C$-scheme, a locally closed\/ $\C$-subscheme of a complex
projective space ${\mathbb P}(W),$ such that the action of\/ $H$ on
$Q$ is the restriction of an action of\/ $H$ on ${\mathbb P}(W)$
coming from a representation of\/ $H$ on\/~$W$.
\label{dt9prop3}
\end{prop}

\begin{proof} We follow the standard method for constructing coarse
moduli schemes\index{coarse moduli scheme}\index{moduli
scheme!coarse} of semistable coherent sheaves in Huybrechts and Lehn
\cite{HuLe2}, adapting it for Artin stacks. Choose an ample line
bundle $\cO_X(1)$ on $X$, let $P$ be the Hilbert
polynomial\index{Hilbert polynomial} of $E$ and fix $n\in\Z$.
Consider Grothendieck's Quot Scheme\index{Quot scheme}
$\Quot_X\bigl(\C^{P(n)}\ot\cO_X(-n),P\bigr)$, explained in \cite[\S
2.2]{HuLe2}, which parametrizes quotients
$\C^{P(n)}\ot\cO_X(-n)\ab\twoheadrightarrow E'$, where $E'$ has
Hilbert polynomial $P$. Then the $\C$-group $H=\GL(P(n),\C)$ acts
on~$\Quot_X\bigl(\C^{P(n)}\ot\cO_X(-n),P\bigr)$.

Let $Q$ be the $H$-invariant open $\C$-subscheme of
$\Quot_X\bigl(\C^{P(n)}\ot\cO_X(-n),P\bigr)$ parametrizing
$\phi:\C^{P(n)}\ot\cO_X(-n)\twoheadrightarrow E'$ for which
$H^i(E'(n))=0$ for $i>0$ and $\phi_*:\C^{P(n)}\ra H^0(E'(n))$ is an
isomorphism. Then the projection $[Q/H]\ra\fM$ taking the
$GL(P(n),\C)$-orbit of $\phi:\C^{P(n)}\ot\cO_X(-n)\twoheadrightarrow
E'$ to $E'$ is a 1-isomorphism with an open $\C$-substack
${\mathfrak Q}$ of $\fM$. If $n\gg 0$ then $[E]\in{\mathfrak
Q}(\C)$. For the last part,
$\Quot_X\bigl(\C^{P(n)}\ot\cO_X(-n),P\bigr)$ is projective as in
\cite[Th.~2.2.4]{HuLe2}, so it comes embedded as a closed
$\C$-subscheme in some ${\mathbb P}(W)$, and by construction the
$H$-action on $\Quot_X\bigl(\C^{P(n)}\ot\cO_X(-n),P\bigr)$ is the
restriction of an $H$-action on ${\mathbb P}(W)$ from a
representation of $H$ on~$W$.
\end{proof}

The proof of the next proposition is similar to parts of that of
Luna's Etale Slice Theorem~\cite[\S III]{Luna}.\index{Luna's Etale Slice
Theorem}

\begin{prop} In the situation of Proposition {\rm\ref{dt9prop3},}
let\/ $s\in Q(\C)$ project to the point\/ $sH$ in ${\mathfrak
Q}(\C)$ identified with\/ $[E]\in\fM(\C)$ under the
$1$-isomorphism\/ ${\mathfrak Q}\cong [Q/H]$. This $1$-isomorphism
identifies the stabilizer groups $\Iso_{\fM}([E])=\Aut(E)$ and\/
$\Iso_{[Q/H]}(sH)=\Stab_H(s),$ and the Zariski tangent spaces
$T_{[E]}\fM\cong\Ext^1(E,E)$ and\/ $T_{sH}[Q/H]\cong T_sQ/T_s(sH),$
so we have natural isomorphisms $\Aut(E)\cong\Stab_H(s)$
and\/~$\Ext^1(E,E)\cong T_sQ/T_s(sH)$.

Let\/ $K$ be the $\C$-subgroup of\/ $\Stab_H(s)\subseteq H$
identified with\/ $G^{\sst\C}$ in $\Aut(E)$. Then there exists a\/
$K$-invariant, locally closed\/ $\C$-subscheme $S$ in $Q$ with\/
$s\in S(\C),$ such that\/ $T_sQ=T_sS\op T_s(sH),$ and the morphism
$\mu:S\times H\ra Q$ induced by the inclusion $S\hookra Q$ and the
$H$-action on $Q$ is smooth of relative dimension\/~$\dim\Aut(E)$.
\label{dt9prop4}
\end{prop}

\begin{proof} We have $s\in Q(\C)\subset{\mathbb P}(W)$, so $s$
corresponds to a 1-dimensional vector subspace $L$ of $W$. As
$K\subseteq\Stab_H(s)\subseteq H$, this $L$ is preserved by $K$,
that is, $L$ is a $K$-subrepresentation of $W$. Since $K\cong
G^{\sst\C}$ is reductive, we can decompose $W$ into a direct sum of
$K$-representations $W=L\op W'$. Then there is a natural
identification of $K$-representations $T_{s\,}{\mathbb P}(W)=W'\ot
L^*$.

We have $s\in (sH)(\C) \subseteq Q(\C)\subseteq {\mathbb P}(W)$,
where $sH$ and $Q$ are both $K$-invariant $\C$-subschemes of
${\mathbb P}(W)$. Hence we have inclusions of Zariski tangent
spaces, which are $K$-representations
\begin{equation*}
T_s(sH)\subseteq T_sQ\subseteq T_{s\,}{\mathbb P}(W)=W'\ot L^*.
\end{equation*}
As $K$ is reductive, we may choose $K$-subrepresentations $W'',W'''$
of $T_{s\,}{\mathbb P}(W)$ such that $T_sQ=T_s(sH)\op W''$ and
$T_{s\,}{\mathbb P}(W)=T_sQ\op W'''$. Then $W'\ot L^*=T_s(sH)\op
W''\op W'''$. Tensoring by $L$ and using $W=L\op W'$ gives a direct
sum of $K$-representations
\e
W=L\op \bigl(T_s(s H)\ot L\bigr)\op\bigl(W''\ot L\bigr)\op
\bigl(W'''\ot L\bigr).
\label{dt9eq5}
\e

Define $S'=Q\cap{\mathbb P}\bigl(L\op(W''\ot L)\op(W'''\ot
L)\bigr)$, where we have omitted the factor $T_s(sH)\ot L$ in
\eq{dt9eq5} in the projective space. Then $S'$ is a locally closed
$\C$-subscheme of ${\mathbb P}(W)$, and is $K$-invariant as it is
the intersection of two $K$-invariant subschemes, with $s\in
S'(\C)$. In the decomposition $T_{s\,}{\mathbb P}(W)=T_s(sH)\op
W''\op W'''$ we have $T_sQ=T_s(sH)\op W''$ and $T_{s\,}{\mathbb
P}\bigl(L\op(W''\ot L)\op(W'''\ot L)\bigr)=W''\op W'''$, which
intersect transversely in $W''$. Hence $Q$ and ${\mathbb
P}\bigl(L\op(W''\ot L)\op(W'''\ot L)\bigr)$ intersect transversely
at $s\in S'(\C)$, and $T_sS'=W''$, so that~$T_sQ=T_sS'\op T_s(sH)$.

Since $T_sQ=T_sS'\op T_s(sH)$, we see that $S'$ intersects the
$H$-orbit $sH$ transversely at $s$. It follows that the morphism
$\mu':S'\times H\ra Q$ induced by the inclusion $S'\hookra Q$ and
the $H$-action on $Q$ is smooth at $(s,1)$, of relative dimension
$\dim\Stab_H(s)=\dim\Aut(E)$. Let $S$ be the $\C$-subscheme of
points $s\in S'(\C)$ such that $\mu'$ is smooth of relative
dimension $\dim\Aut(E)$ at $(s,1)$. This is an open condition, so
$S$ is open in $S'$, and contains $s$. Therefore $T_sQ=T_sS\op
T_s(sH)$, as we want. Write $\mu=\mu'\vert_{S\times H}$. Then
$\mu:S\times H\ra Q$ is smooth of relative dimension $\Aut(E)$ along
$S\times\{1\}$, and so is smooth of relative dimension $\Aut(E)$ on
all of $S\times H$, since $\mu$ is equivariant under the action of
$H$ on $S\times H$ acting trivially on $S$ and on the right on $H$,
and the right action of $H$ on~$Q$.
\end{proof}

Since $S$ is invariant under the $\C$-subgroup $K$ of the $\C$-group
$H$ acting on $Q$, the inclusion $i:S\hookra Q$ induces a
representable 1-morphism of quotient stacks $i_*:[S/K]\ra [Q/H]$. We
claim that $i_*$ is smooth of relative dimension $\dim\Aut(E)-\dim
G^{\sst\C}$. Since the projection $\pi:Q\ra[Q/H]$ is an atlas for
$[Q/H]$, this is true if and only if the projection from the fibre
product $\pi_2:[S/K]\times_{i_*,[Q/H],\pi}Q\ra Q$ is a smooth
morphism of $\C$-schemes of relative dimension $\dim\Aut(E)-\dim
G^{\sst\C}$. But $[S/K]\times_{i_*,[Q/H],\pi}Q\cong (S\times H)/K$,
where $K$ acts in the given way on $S$ and on the left on $H$. The
projection $\pi_2:(S\times H)/K\ra Q$ fits into the commutative
diagram
\begin{equation*}
\xymatrix@C=40pt@R=5pt{ & (S\times H)/K \ar[dr]^{\pi_2} \\
S\times H \ar[ur]^\pi \ar[rr]^\mu && Q, }
\end{equation*}
where $\pi:S\times H\ra (S\times H)/K$ is the projection to the
quotient. As $\pi$ is smooth of relative dimension $\dim K=\dim
G^{\sst\C}$ and surjective, and $\mu$ is smooth of relative
dimension $\dim\Aut(E)$ by Proposition \ref{dt9prop4}, it follows
that $\pi_2$ and hence $i_*$ are smooth of relative dimension
$\dim\Aut(E)-\dim G^{\sst\C}$.

Combining the isomorphism $K\cong G^{\sst\C}$, the 1-morphism
$i_*:[S/K]\ra [Q/H]$, the 1-isomorphism ${\mathfrak Q}\cong [Q/H]$,
and the open inclusion ${\mathfrak Q}\hookra\fM$, yields a
1-morphism $\Phi:[S/G^{\sst\C}]\ra\fM$, as in Theorem \ref{dt5thm3}.
This $\Phi$ is smooth of relative dimension $\dim\Aut(E)-\dim
G^{\sst\C}$, as $i_*$ is. If $\Aut(E)$ is reductive, so that
$G^{\sst\C}=\Aut(E)$, then $\Phi$ is smooth of dimension 0, that is,
$\Phi$ is \'etale.\index{etale morphism@\'etale morphism}

The conditions $\Phi(s\,G^{\sst\C})=[E]$ and
$\Phi_*:\Iso_{[S/G^{\sst\C}]}(s\,G^{\sst\C})\ra\Iso_{\fM}([E])$ is
the natural $G^{\sst\C}\hookra\Aut(E) \cong\Iso_{\fM}([E])$ in
Theorem \ref{dt5thm3} are immediate from the construction. That
$\rd\Phi\vert_{s\,G^{\sst\C}}:T_sS\cong T_{s\,G^{\sst\C}}
[S/G^{\sst\C}]\ra T_{[E]}\fM\cong \Ext^1(E,E)$ is an isomorphism
follows from $T_{[E]}\fM\cong T_{sH}[Q/H]\cong T_sQ/T_s(sH)$ and
$T_sQ=T_sS\op T_s(sH)$ in Proposition~\ref{dt9prop4}.\index{Artin
stack!atlas|)}

\subsection{Moduli spaces of algebraic vector bundles on $X$}
\label{dt94}

We can now discuss results in algebraic geometry corresponding to
\S\ref{dt91}--\S\ref{dt92}. Let $X$ be a projective complex
algebraic manifold.

\begin{dfn}  A {\it family of algebraic vector bundles\/} $(T,{\cal
F})$ on $X$ is a $\C$-scheme $T$, locally of finite type, and an
algebraic vector bundle $\cal F$ over $X\times T$. For each $t\in
T(\C)$, the {\it fibre\/} ${\cal F}_t$ of the family is ${\cal
F}\vert_{X\times\{t\}}$, regarded as an algebraic vector bundle
over~$X\cong X\times\{t\}$.

A family $(T,{\cal F})$ is called {\it formally versal
at\/}\index{versal family!formally versal|(}\index{universal
family!formally universal|(} $t\in T$ if whenever $T'$ is a
$\C$-scheme of finite length with exactly one $\C$-point $t'$, and
$(T',{\cal F}')$ is a family of algebraic vector bundles on $X$ with
${\cal F}_t\cong{\cal F}'_{t'}$ as algebraic vector bundles on $X$,
there exists a morphism $\up:T'\ra T$ with $\up(t')=t$, and an
isomorphism $\up^*({\cal F})\cong{\cal F}'$ as vector bundles over
$X\times T'$. It is called {\it formally universal at\/} $t\in T$ if
in addition the morphism $\up:T'\ra T$ is unique. The family
$(T,{\cal F})$ is called {\it formally versal\/} (or {\it formally
universal\/}) if it is formally versal (or formally universal) at
every~$t\in T$.
\label{dt9def4}
\end{dfn}

By work of Grothendieck and others, as in Laumon and Moret-Bailly
\cite[Th.~4.6.2.1]{LaMo} for instance, we have:

\begin{prop} The moduli functor\/ $\mathbb{VB}^{\text{\rm\'et}}_{
\text{\rm si}}:(\C\text{\rm -schemes})\ra(\text{\rm sets})$ of
isomorphism classes of families of simple algebraic vector bundles
on $X,$ sheafified in the \'etale topology,\index{etale topology@\'etale
topology} is represented by a complex algebraic space $\Vect_\rsi$
locally of finite type, the moduli space of simple algebraic vector
bundles on~$X$.

The moduli functor\/ $\mathbb{VB}:(\C\text{\rm
-schemes})\ra(\text{\rm groupoids})$ of families of algebraic vector
bundles on $X$ is represented by an Artin $\C$-stack\/ $\fVect$
locally of finite type, the moduli stack of algebraic vector bundles
on~$X$.
\label{dt9prop5}
\end{prop}

As in Miyajima \cite[\S 3]{Miya}, the existence of $\Vect_\rsi$ as a
complex algebraic space\index{complex algebraic space} implies the
existence \'etale locally of formally universal families of simple
vector bundles on $X$. Similarly, the existence of $\fVect$ as an
Artin $\C$-stack implies that a smooth 1-morphism $\phi:S\ra\fVect$
from a scheme $S$ corresponds naturally to a formally versal family
$(S,{\cal D})$ of vector bundles on $X$.

\begin{prop}{\bf(a)} Let\/ ${\cal E}$ be a simple algebraic
vector bundle on $X$. Then there exists an affine $\C$-scheme $S,$ a
$\C$-point\/ $s\in S,$ and a formally universal family of simple
algebraic vector bundles $(S,{\cal D})$ on $X$ with\/ ${\cal
D}_s\cong{\cal E}$. This family $(S,{\cal D})$ induces an \'etale
map of complex algebraic spaces $\pi:S\ra\Vect_\rsi$ with\/
$\pi(s)=[{\cal E}]$. There is a natural isomorphism between the
Zariski tangent space $T_sS$ and\/~$\Ext^1({\cal E},{\cal E}\bigr)$.
\smallskip

\noindent{\bf(b)} Let\/ ${\cal E}$ be an algebraic vector bundle on
$X,$ and\/ $G^{\sst\C}$ a maximal reductive subgroup of\/
$\Aut({\cal E})$. Then \S{\rm\ref{dt93}} constructs a
quasiprojective $\C$-scheme $S,$ an action of\/ $G^{\sst\C}$ on $S,$
a $G^{\sst\C}$-invariant point\/ $s\in S(\C),$ and a $1$-morphism
$\Phi:[S/G^{\sst\C}]\ra\fVect$ smooth of relative dimension
$\dim\Aut({\cal E})-\dim G^{\sst\C},$ such that\/
$\Phi(s\,G^{\sst\C})=[{\cal E}],$
$\Phi_*:\Iso_{[S/G^{\sst\C}]}(s\,G^{\sst\C})\ra\Iso_{\fVect}([{\cal
E}])$ is the inclusion $G^{\sst\C}\hookra\Aut({\cal E}),$ and\/
$\rd\Phi\vert_{s\,G^{\sst\C}}:T_sS\cong T_{s\,G^{\sst\C}}
[S/G^{\sst\C}]\ra T_{[{\cal E}]}\fVect\cong \Ext^1({\cal E},{\cal
E})$ is an isomorphism.

Let\/ $(S,{\cal D})$ be the family of algebraic vector bundles on
$X$ corresponding to the $1$-morphism $\Phi\ci\pi:S\ra\fVect,$ where
$\pi:S\ra[S/G]$ is the projection. Then $(S,{\cal D})$ is formally
versal and\/ $G^{\sst\C}$ equivariant, with\/ ${\cal D}_s\cong{\cal
E}$.
\label{dt9prop6}
\end{prop}

Together with \S\ref{dt93}, part (b) completes the proof of the
second paragraph of Theorem \ref{dt5thm3} for $\fVect$. The proof
for $\fM$ follows from Theorem~\ref{dt5thm1}.\index{versal
family!formally versal|)}\index{universal family!formally universal|)}

\subsection[Identifying versal families of vector bundles]{Identifying
versal families of holomorphic structures \\ and algebraic vector
bundles}
\label{dt95}\index{complex analytic space|(}\index{vector bundle!holomorphic
structure|(}\index{versal family|(}\index{universal family|(}

Let $\cal E$ be an algebraic vector bundle on $X$. Write $E\ra X$
for the underlying $C^\iy$ complex vector bundle, and $\db_E$ for
the induced holomorphic structure on $E$. Then $(E,\db_E)$ is the
analytic vector bundle associated to $\cal E$. By Serre \cite{Serr}
we have $\Ext^1({\cal E},{\cal E}\bigr)\cong\Ext^1
\bigl((E,\db_E),(E,\db_E)\bigr)$, that is, Ext groups computed in
the complex algebraic or complex analytic categories are the same.

Then Proposition \ref{dt9prop1} constructs a {\it versal family
$(T,\tau)$ of holomorphic structures on\/} $E$, with $\tau(t)=\db_E$
for $t\in T$ and $T_tT\cong\Ext^1\bigl((E,\db_E),(E,\db_E)\bigr)$.
If $(E,\db_E)$ is simple then $(T,\tau)$ is a {\it universal\/}
family of simple holomorphic structures. Proposition \ref{dt9prop2}
shows that we may lift $(T,\tau)$ to a versal family $(T,{\cal F})$
of analytic vector bundles over $X$, with isomorphism $\bigl({\cal
F}\ra (X\times T)\bigr)\cong\bigl((E\times T)\ra(X\times T)\bigr)$
as $C^\iy$ complex vector bundles inducing $(T,\tau)$. If\/
$(E,\db_E)$ is simple then $(T,{\cal F})$ is a universal family of
simple analytic vector bundles.

On the other hand, using algebraic geometry, Proposition
\ref{dt9prop6} gives a formally versal family of algebraic vector
bundles $(S,{\cal D})$ on $X$ with ${\cal D}_s\cong{\cal E}$ and
$T_sS\cong\Ext^1({\cal E},{\cal E}\bigr)$, and if $\cal E$ is simple
then $(S,{\cal D})$ is a formally universal family of simple
algebraic vector bundles. Now Miyajima \cite[\S 3]{Miya} quotes
Serre \cite{Serr} and Schuster \cite{Schus} to say that if $(S,{\cal
D})$ is a formally versal (or formally universal) family of
algebraic vector bundles on $X$, then the induced family of complex
analytic vector bundles $(S_\an,{\cal D}_\an)$ is versal (or
universal) in the sense of Definition~\ref{dt9def3}.

Hence we have two versal families of complex analytic vector
bundles: $(T,{\cal F})$ from Propositions \ref{dt9prop1} and
\ref{dt9prop2}, with ${\cal F}_t\cong(E,\db_E)$, and $(S_\an,{\cal
D}_\an)$ from Proposition \ref{dt9prop6}, with ${\cal
D}_s\cong(E,\db_E)$. We will prove these two families are locally
isomorphic near $s,t$. In the universal case this is obvious, as in
Miyajima \cite[\S 3]{Miya}. In the versal case we use the
isomorphisms~$T_tT\cong\Ext^1\bigl((E,\db_E),(E,\db_E) \bigr)\cong
T_sS$.

\begin{prop} Let\/ $\cal E$ be an algebraic vector bundle on $X,$
with underlying complex vector bundle $E$ and holomorphic structure
$\db_E$. Let\/ $(T,\tau),(T,{\cal F}),\ab(S,{\cal D})$ be the versal
families of holomorphic structures, analytic vector bundles, and
algebraic vector bundles from Propositions {\rm\ref{dt9prop1},
\ref{dt9prop2}, \ref{dt9prop6},} so $t\in T,$ $s\in S(\C)$ with\/
$\tau(t)\!=\!\db_E,$ ${\cal F}_t\!\cong\!(E,\db_E),$ ${\cal
D}_s\!\cong\!{\cal E}$ and\/ $T_tT\!\cong\!
\Ext^1\bigl((E,\db_E),(E,\db_E)\bigr)\!\cong\! T_sS$. Write
$(S_\an,{\cal D}_\an)$ for the family of analytic vector bundles
underlying\/~$(S,{\cal D})$.

Then there exist open neighbourhoods $T'$ of\/ $t$ in\/ $T$ and\/
$S'_\an$ of\/ $s$ in\/ $S_\an$ and an isomorphism of complex
analytic spaces $\vp:T'\ra S'_\an$ such that\/ $\vp(t)=s$ and\/
$\vp^*({\cal D}_\an)\cong{\cal F}\vert_{X\times T'}$ as analytic
vector bundles over~$X\times T'$.
\label{dt9prop7}
\end{prop}

\begin{proof} From above, $(T,{\cal F})$ and $(S_\an,{\cal
D}_\an)$ are both versal families of analytic vector bundles on $X$
with ${\cal F}_t\cong(E,\db_E)\cong({\cal D}_\an)_s$. By Definition
\ref{dt9def3}, since $(S_\an,{\cal D}_\an)$ is versal, there exists
an open neighbourhood $\ti T$ of $t$ in $T$ and a morphism of
complex analytic spaces $\ti\vp:\ti T\ra S_\an$ such that
$\ti\vp(t)=s$ and $\ti\vp^*({\cal D}_\an)\cong{\cal F}\vert_{\ti
T}$. Similarly, since $(T,{\cal F})$ is versal, there exists an open
neighbourhood $\ti S_\an$ of $s$ in $S_\an$ and a morphism of
complex analytic spaces $\ti\psi:\ti S_\an\ra T$ such that
$\ti\psi(s)=t$ and $\ti\psi^*({\cal F})\cong{\cal D}_\an\vert_{\ti
S_\an}$.

Restricting the isomorphism $\ti\vp^*({\cal D}_\an)\cong{\cal
F}\vert_{\ti T}$ to the fibres at $t$ gives an isomorphism ${\cal
D}_s\cong{\cal F}_t$. We are also given isomorphisms ${\cal
F}_t\cong (E,\db_E)$ and $(E,\db_E)\cong{\cal D}_s$. Composing these
three $(E,\db_E)\cong{\cal D}_s\cong{\cal F}_t\cong (E,\db_E)$ gives
an automorphism $\ga$ of $(E,\db_E)$. Differentiating $\ti\vp$ at
$t$ gives a $\C$-linear map $\rd\ti\vp\vert_t:T_t\ti T\ra T_s\ti
S_\an$. We also have isomorphisms
$T_tT\cong\Ext^1\bigl((E,\db_E),(E,\db_E)\bigr)\cong T_sS$. Using
the interpretation of $\Ext^1\bigl((E,\db_E), (E,\db_E)\bigr)$ as
infinitesimal deformations of $(E,\db_E)$, one can show that under
these identifications $T_tT\cong\Ext^1\bigl(
(E,\db_E),(E,\db_E)\bigr)\cong T_sS$, the map
$\rd\ti\vp\vert_t:T_t\ti T\ra T_s\ti S_\an$ corresponds to
conjugation by $\ga\in\Aut(E,\db_E)$ in $\Ext^1\bigl((E,\db_E),
(E,\db_E)\bigr)$. This implies that $\rd\ti\vp\vert_t:T_t\ti T\ra
T_s\ti S_\an$ is an isomorphism. Similarly,
$\rd\ti\psi\vert_s:T_s\ti S_\an\ra T_t\ti T$ is an isomorphism.

Suppose first that $\cal E$ is simple. Then $(T,{\cal F})$,
$(S_\an,{\cal D}_\an)$ are universal families, so $\ti\vp,\ti\psi$
above are unique. Also by universality of $(T,{\cal F})$ we see that
$\ti\psi\ci\ti\vp\cong\id_T$ on $\ti T\cap\ti\vp^{-1}(\ti S_\an)$,
and similarly $\ti\vp\ci\ti\psi\cong\id_{S_\an}$ on $\ti
S_\an\cap\ti\psi^{-1}(\ti T)$. Hence the restrictions of $\ti\vp$ to
$\ti T\cap\ti\vp^{-1}(\ti S_\an)$ and $\ti\psi$ to $\ti
S_\an\cap\ti\psi^{-1}(\ti T)$ are inverse, and setting $T'=\ti
T\cap\ti\vp^{-1}(\ti S_\an)$ and $\vp'=\ti\vp\vert_{T'}$ gives the
result. This argument was used by Miyajima~\cite[\S 3]{Miya}.

For the general case, $\ti\psi\ci\ti\vp:\ti T\cap\ti\vp^{-1}(\ti
S_\an)\ra T$ is a morphism of complex analytic spaces with
$\ti\psi\ci\ti\vp(t)=t$ and $\rd(\ti\psi\ci\ti\vp) \vert_t:T_tT\ra
T_tT$ an isomorphism. We will show that this implies
$\ti\psi\ci\ti\vp$ is an isomorphism of complex analytic spaces near
$t$. A similar result in algebraic geometry is Eisenbud
\cite[Cor.~7.17]{Eise}. Write $\cO_{t,T}$ for the algebra of germs
of analytic functions on $T$ defined near $t$. Write $\m_{t,T}$ for
the maximal ideal of $f$ in $\cO_{t,T}$ with~$f(t)=0$.

In complex analytic geometry, the operations on $\cO_{t,T}$ are not
just addition and multiplication. We can also apply holomorphic
functions of several variables: if $W$ is an open neighbourhood of
$0$ in $\C^l$ and $F:W\ra\C$ is holomorphic, then there is an
operation $F_*:\m_{t,T}^{\op^l}\ra\cO_{t,T}$ mapping
$F_*:(f_1,\ldots,f_l)\mapsto F(f_1,\ldots,f_l)$. Let $N=\dim T_tT$,
and choose $g_1,\ldots,g_N\in\m_{t,T}$ such that
$g_1+\m_{t,T}^2,\ldots, g_N+\m_{t,T}^2$ are a basis for
$\m_{t,T}/\m_{t,T}^2\cong T_t^*T$. Then $g_1,\ldots,g_N$ generate
$\cO_{t,T}$ over such operations~$F_*$.

Let $f\in\cO_{t,T}$. Since $(\ti\psi\ci\ti\vp)^*:\m_{t,T}/\m_{t,T}^2
\ra\m_{t,T}/\m_{t,T}^2$ is an isomorphism, we see that
$(\ti\psi\ci\ti\vp)^*(g_1),\ldots, (\ti\psi\ci\ti\vp)^*(g_N)$
project to a basis for $\m_{t,T}/\m_{t,T}^2$, so they generate
$\cO_{t,T}$ over operations $F_*$. Thus there exists a holomorphic
function $F$ defined near $0$ in $\C^N$ with
$f=F\bigl((\ti\psi\ci\ti\vp)^* (g_1),\ldots,
(\ti\psi\ci\ti\vp)^*(g_N)\bigr)$. Hence
$f=(\ti\psi\ci\ti\vp)^*\bigl(F(g_1,\ldots,g_N)\bigr)$. Thus
$(\ti\psi\ci\ti\vp)^*:\cO_{t,T}\ra \cO_{t,T}$ is surjective. But
$\cO_{t,T}$ is noetherian as in Griffiths and Harris
\cite[p.~679]{GrHa}, and a surjective endomorphism of a noetherian
ring is an isomorphism. Therefore $(\ti\psi\ci\ti\vp)^*:\cO_{t,T}\ra
\cO_{t,T}$ is an isomorphism of local algebras.

Since $\cO_{t,T}$ determines $(T,t)$ as a germ of complex analytic
spaces, this implies $\ti\psi\ci\ti\vp$ is an isomorphism of complex
analytic spaces near $t$, as we claimed above. Similarly,
$\ti\vp\ci\ti\psi$ is an isomorphism of complex analytic spaces near
$s$. Thus $\ti\vp$ and $\ti\psi$ are isomorphisms of complex
analytic spaces near $s,t$. So we can choose an open neighbourhood
$T'$ of $t$ in $\ti T\cap\ti\vp^{-1}(\ti S_\an)$ such that
$\vp=\ti\vp\vert_{T'}:T'\ra S'_\an=\ti\vp(T')$ is an isomorphism of
complex analytic spaces. The conditions $\vp(t)=s$ and $\vp^*({\cal
D}_\an)\cong{\cal F}\vert_{X\times T'}$ are immediate.
\end{proof}\index{vector bundle!holomorphic structure|)}\index{versal
family|)}\index{universal family|)}

\subsection{Writing the moduli space as $\Crit(f)$}
\label{dt96}

We now return to the situation of \S\ref{dt91}, and suppose $X$ is a
Calabi--Yau 3-fold. Let $X$ be a compact complex 3-manifold with
trivial canonical bundle $K_X$, and pick a nonzero section of $K_X$,
that is, a nonvanishing closed $(3,0)$-form $\Om$ on $X$.

Fix a $C^\iy$ complex vector bundle $E\ra X$ on $X$, and choose a
holomorphic structure $\db_E$ on $E$. Then $\sA^{2,k}$ is given by
\eq{dt9eq2} as in \S\ref{dt91}. Following Thomas \cite[\S 3]{Thom},
define the {\it holomorphic Chern--Simons
functional\/}\index{holomorphic Chern--Simons functional|(}
$CS:\sA^{2,k}\ra\C$\nomenclature[CS]{$CS$}{holomorphic Chern--Simons
functional} by
\e
CS:\db_E+A\longmapsto \frac{1}{4\pi^2}\int_X\Tr\bigl(\ts\frac{1}{2}
(\db_EA)\w A+\frac{1}{3}A\w A\w A\bigr)\w\Om.
\label{dt9eq6}
\e
Here $A\in L^2_k(\End(E) \ot_\C\La^{0,1}T^*X)$ and $\db_EA\in
L^2_{k-1}(\End(E)\ot_\C\La^{0,2}T^*X)$. To form $(\db_EA)\w A$ and
$A\w A\w A$ we take the exterior product of the $\La^{0,q}T^*X$
factors, and multiply the $\End(E)$ factors. So $\frac{1}{2}
(\db_EA)\w A+\frac{1}{3}A\w A\w A$ is a section of $\End(E)
\ot_\C\La^{0,3}T^*X$. We then apply the trace $\Tr:\End(E)\ra\C$ to
get a $(0,3)$-form, wedge with $\Om$, and integrate over $X$. As
$k\ge 1$, Sobolev Embedding\index{Sobolev Embedding Theorem} and $X$
compact with $\dim X=6$ imply $A$ is $L^3$ and $\db_EA$ is
$L^{3/2}$, so the integrand in \eq{dt9eq6} is $L^1$ by H\"older's
inequality, and $CS$ is well-defined.

This $CS$ is a cubic polynomial on the infinite-dimensional affine
space $\sA^{2,k}$. It is a well-defined analytic function on
$\sA^{2,k}$ in the sense of Douady \cite{Doua1,Doua2}. An easy
calculation shows that for all $A,a\in L^2_k(\End(E)
\ot_\C\La^{0,1}T^*X)$ we have
\e
\frac{\rd}{\rd t}\bigl[CS(\db_E+A+ta)\bigr]\big\vert_{t=0}=
\frac{1}{4\pi^2}\int_X\Tr\bigl(a\w(\db_EA+A\w A)\bigr)\w\Om,
\label{dt9eq7}
\e
where $\db_EA+A\w A=F_A^{0,2}=P_k(\db_E+A)$ as in \eq{dt9eq3}.
Essentially, equation \eq{dt9eq7} says that the 1-form $\rd CS$ on
the affine space $\sA^{2,k}$ is given at $\db_E+A$ by the
$(0,2)$-curvature $F_A^{0,2}$ of $\db_E+A$.

\begin{prop} Suppose\/ $X$ is a compact complex\/ $3$-manifold with
trivial canonical bundle, $E\ra X$ a\/ $C^\iy$ complex vector bundle
on $X,$ and\/ $\db_E$ a holomorphic structure on $E$. Define
$CS:\sA^{2,k}\ra\C$ by \eq{dt9eq6}. Let\/ $Q_\ep,T$ be as in
Proposition {\rm\ref{dt9prop1}}. Then for sufficiently small\/
$\ep>0,$ as a complex analytic subspace of the finite-dimensional
complex submanifold $Q_\ep,$ $T$ is the critical locus of the
holomorphic function $CS\vert_{Q_\ep}:Q_\ep\ra\C$.
\label{dt9prop8}
\end{prop}

\begin{proof} Following \cite[\S 1]{Miya}, define
$R_\ep\subset Q_\ep\times L^2_{k-1}(\End(E)\ot_\C\La^{0,2}T^*X)$ by
\e
\begin{split}
R_\ep=\bigl\{(\db_E+A,B)&\in Q_\ep\times L^2_{k-1}(\End(E)\ot_\C\La^{0,2}T^*X):\\
&\db_E^*B=0,\;\> \db_E^*(\db_EB-B\w A+A\w B)=0\bigr\}.
\end{split}
\label{dt9eq8}
\e
Then Miyajima \cite[Lem.~1.5]{Miya} shows that for sufficiently
small $\ep>0$, $R_\ep$ is a complex submanifold of $Q_\ep\times
L^2_{k-1}(\End(E)\ot_\C\La^{0,2}T^*X)$, and in the notation of
\S\ref{dt91}, the projection $\id\times\pi_{\sE^2}:R_\ep\ra
Q_\ep\times\sE^2$ is a biholomorphism. Thus the projection
$\pi_{Q_\ep}:R_\ep\ra Q_\ep$ makes $R_\ep$ into a holomorphic vector
bundle over $Q_\ep$, with fibre $\sE^2\cong
\Ext^2\bigl((E,\db_E),(E,\db_E)\bigr)$. Note from \eq{dt9eq8} that
the fibres of $\pi_{Q_\ep}$ are vector subspaces of
$L^2_{k-1}(\End(E)\ot_\C\La^{0,2}T^*X)$, so $R_\ep$ is a vector
subbundle of the infinite-dimensional vector bundle $Q_\ep\times
L^2_{k-1}(\End(E)\ot_\C\La^{0,2}T^*X)\ra Q_\ep$.

Let $\db_E+A\in Q_\ep$, and set $B=P_k(\db_E+A)=F_A^{0,2}=\db_EA+A\w
A$. Then $\db_E^*B=0$ by the definition \eq{dt9eq4} of $Q_\ep$, and
$\db_EB-B\w A+A\w B=0$ by the Bianchi identity. So
$(\db_E+A,P_k(\db_E+A))\in R_\ep$. Thus $P_k\vert_{Q_\ep}$ is
actually a holomorphic section of the holomorphic vector bundle
$R_\ep\ra Q_\ep$. The complex analytic subspace $T$ in $Q_\ep$ is
$T=(P_k\vert_{Q_\ep})^{-1}(0)$. So we can regard $T$ as the zeroes
of the holomorphic section $P_k\vert_{Q_\ep}$ of the holomorphic
vector bundle~$R_\ep\ra Q_\ep$.

Define a holomorphic map $\Xi:R_\ep\ra T^*Q_\ep$ by
$\Xi:(\db_E+A,B)\mapsto (\db_E+A,\al_B)$, where $\al_B\in
T^*_{\db_E+A}Q_\ep$ is defined by
\e
\al_B(a)=\frac{1}{4\pi^2}\int_X\Tr\bigl(a\w B\w\Om\bigr)
\label{dt9eq9}
\e
for all $a\in T_{\db_E+A}Q_\ep\subset L^2_k(\End(E)\ot_\C
\La^{0,1}T^*X)$. Then $\Xi$ is linear between the fibres of
$R_\ep,T^*Q_\ep$, so it is a morphism of holomorphic vector bundles
over $Q_\ep$. Comparing \eq{dt9eq7} and \eq{dt9eq9} we see that when
$B=P_k(\db_E+A)=\db_EA+A\w A$ we have $\al_B=\rd(CS
\vert_{Q_\ep})\vert_{\db_E+A}$. Hence $\Xi\ci P_k\vert_{Q_\ep}\equiv
\rd(CS\vert_{Q_\ep})$, that is, $\Xi$ takes the holomorphic section
$P_k$ of $R_\ep$ to the holomorphic section $\rd(CS\vert_{Q_\ep})$
of~$T^*Q_\ep$.

Now consider the fibres of $R_\ep$ and $T^*Q_\ep$ at $\db_E\in
Q_\ep$. As in \cite[\S 1]{Miya} we have $T_{\db_E}
Q_\ep\!=\!\sE^1\!\cong\!\Ext^1\bigl((E,\db_E),(E,\db_E)\bigr)$ and
$R_\ep\vert_{\db_E}\!=\!\sE^2\!\cong\!\Ext^2\bigl((E,\db_E),(E,\db_E)
\bigr)$. But $X$ is a Calabi--Yau 3-fold, so by Serre duality we
have an isomorphism $\Ext^2\bigl((E,\db_E),\ab(E,\db_E)\bigr)\cong
\Ext^1\bigl((E,\db_E),(E,\db_E)\bigr){}^*$. The linear map
$\Xi\vert_{\db_E}:\ab R_\ep\vert_{\db_E}\ab\ra T_{\db_E}^*Q_\ep$ is
a multiple of this isomorphism, so $\Xi\vert_{\db_E}$ is an
isomorphism. This is an open condition, so by making $\ep>0$ smaller
if necessary we can suppose that $\Xi:R_\ep\ra T^*Q_\ep$ is an
isomorphism of holomorphic bundles. Since $\Xi\ci
P_k\vert_{Q_\ep}\equiv\rd(CS\vert_{Q_\ep})$, it follows that
$T=(P_k\vert_{Q_\ep})^{-1}(0)$ coincides with
$(\rd(CS\vert_{Q_\ep}))^{-1}(0)$ as a complex analytic subspace of
$Q_\ep$, as we have to prove.
\end{proof}\index{holomorphic Chern--Simons functional|)}

\subsection[The proof of Theorem $\text{\ref{dt5thm2}}$]{The proof of
Theorem \ref{dt5thm2}}
\label{dt97}

We can now prove Theorem \ref{dt5thm2}. The second part of Theorem
\ref{dt5thm1} shows that it is enough to prove Theorem \ref{dt5thm2}
with $\Vect_\rsi$ in place of $\M_\rsi$. Let $X$ be a projective
Calabi--Yau 3-fold over $\C$, and ${\cal E}$ a simple algebraic
vector bundle\index{vector bundle!algebraic} on $X$, with underlying
$C^\iy$ complex vector bundle $E\ra X$ and holomorphic structure
$\db_E$.\index{vector bundle!holomorphic structure} Then Proposition
\ref{dt9prop1} gives a complex analytic space $T$, a point $t\in T$
with $T_tT\cong\Ext^1({\cal E},{\cal E})$, and a universal
family\index{universal family} $(T,\tau)$ of simple holomorphic
structures on $E$ with~$\tau(t)=\db_E$.

Proposition \ref{dt9prop2} shows that $(T,\tau)$ extends to a
universal family $(T,{\cal F})$ of simple analytic vector bundles.
Then Proposition \ref{dt9prop6}(a) gives an affine $\C$-scheme $S$,
a point $s\in S_\an$, a formally universal family of simple
algebraic vector bundles $(S,{\cal D})$ on $X$ with ${\cal
D}_s\cong{\cal E}$, and an \'etale map of complex algebraic spaces
$\pi:S\ra\Vect_\rsi$ with $\pi(s)=[{\cal E}]$. Write $(S_\an,{\cal
D}_\an)$ for the underlying family of simple analytic vector
bundles. Proposition \ref{dt9prop7} gives an isomorphism of complex
analytic spaces $\vp:T'\ra S'_\an$ between open neighbourhoods $T'$
of $t$ in $T$ and $S'_\an$ of $s$ in $S_\an$, with $\vp(t)=s$ and
$\vp^*({\cal D}_\an)\cong{\cal F}\vert_{X\times T'}$. Proposition
\ref{dt9prop8} shows that we may write $T$ as the critical locus of
$CS\vert_{Q_\ep}:Q_\ep\ra\C$, where $Q_\ep$ is a complex manifold
with $T_tQ_\ep\cong\Ext^1({\cal E},{\cal E})$.

Since $Q_\ep$ is a complex manifold with $T_tQ_\ep\cong\Ext^1({\cal
E},{\cal E})$, we may identify $Q_\ep$ near $t$ with an open
neighbourhood $U$ of $u=0$ in $\Ext^1({\cal E},{\cal E})$. A natural
way to do this is to map $Q_\ep\ra\sE^1$ by $\db_E+A\mapsto
\pi_{\sE^1}(A)$, and then use the isomorphism
$\sE^1\cong\Ext^1({\cal E},{\cal E})$. Let $f:U\ra\C$ be the
holomorphic function identified with $CS\vert_{Q_\ep}:Q_\ep\ra\C$.
Since \'etale maps of complex algebraic spaces induce local
isomorphisms of the underlying complex analytic spaces, putting all
this together yields an isomorphism of complex analytic spaces
between $\M_\rsi(\C)$ near $[E]$ and $\Crit(f)$ near 0, as we want.

\subsection[The proof of Theorem $\text{\ref{dt5thm3}}$]{The proof of
Theorem \ref{dt5thm3}}
\label{dt98}

The first part of Theorem \ref{dt5thm1} shows that it is enough to
prove Theorem \ref{dt5thm3} with $\fVect$ in place of $\fM$. Let $X$
be a projective Calabi--Yau 3-fold over $\C$, and ${\cal E}$ an
algebraic vector bundle on $X$, with underlying $C^\iy$ complex
vector bundle $E\ra X$ and holomorphic structure $\db_E$. Let $G$ be
a maximal compact subgroup\index{maximal compact subgroup} in
$\Aut(E)$, and its complexification $G^{\sst\C}$ a maximal reductive
subgroup\index{maximal reductive subgroup} in $\Aut(E)$. Then
Propositions \ref{dt9prop1} and \ref{dt9prop2} give a complex
analytic space $T$, a point $t\in T$ with $T_tT\cong\Ext^1({\cal
E},{\cal E})$, and versal families $(T,\tau)$ of holomorphic
structures on $E$ and $(T,{\cal F})$ of analytic vector bundles on
$E$, with $\tau(t)=\db_E$ and~${\cal
F}_t\cong(E,\db_E)$.\index{versal family}

Proposition \ref{dt9prop6}(b) gives a quasiprojective $\C$-scheme
$S,$ an action of $G^{\sst\C}$ on $S,$ a 1-morphism
$\Phi:[S/G^{\sst\C}]\ra\fVect$ smooth of relative dimension
$\dim\Aut({\cal E})-\dim G^{\sst\C}$, a $G^{\sst\C}$-invariant point
$s\in S(\C)$ with $\Phi(s\,G^{\sst\C})=[{\cal E}]$ and
$T_sS\cong\Ext^1({\cal E},{\cal E})$, and a $G^{\sst\C}$-invariant,
formally versal family of algebraic vector bundles $(S,{\cal D})$ on
$X$ with ${\cal D}_s\cong{\cal E}$. By Serre \cite{Serr} we have
$\Aut({\cal E})=\Aut(E,\db_E)$, that is, the automorphisms of ${\cal
E}$ as an algebraic vector bundle coincide with the automorphisms of
$(E,\db_E)$ as an analytic vector bundle.

Proposition \ref{dt9prop7} gives a local isomorphism of complex
analytic spaces between $T$ near $t$ and $S_\an$ near $s$, and
Proposition \ref{dt9prop8} gives an open neighbourhood $U$ of $0$ in
$\Ext^1({\cal E},{\cal E})$ and a holomorphic function $f:U\ra\C$,
where $U\cong Q_\ep$ and $f\cong CS\vert_{Q_\ep}$, and an
isomorphism of complex analytic spaces between $T$ and $\Crit(f)$
identifying $t$ with 0. Putting these two isomorphisms together
yields an open neighbourhood $V$ of $s$ in $S_\an,$ and an
isomorphism of complex analytic spaces $\Xi:\Crit(f)\ra V$
with~$\Xi(0)=s$.

Consider $\rd\Xi\vert_0:T_0\Crit(f)\ra T_sV$. We have
$T_0\Crit(f)\cong\Ext^1({\cal E},{\cal E})\cong T_sV$ by
Propositions \ref{dt9prop1} and \ref{dt9prop6}(b). The isomorphism
$T_0\Crit(f)\cong\Ext^1({\cal E},{\cal E})$ is determined by a
choice of isomorphism of analytic vector bundles
$\eta_1:(E,\db_E)\ra{\cal F}_t$. The isomorphism $\Ext^1({\cal
E},{\cal E})\cong T_sV$ is determined by a choice of isomorphism of
analytic vector bundles $\eta_2:({\cal D}_\an)_s\ra(E,\db_E)$. The
map $\Xi$ is determined by a choice of local isomorphism of versal
families of analytic vector bundles $\eta_3$ from $(T,{\cal F})$
near $t$ to $(S_\an,{\cal D}_\an)$ near $s$. Composing gives an
isomorphism $\eta_2\ci\eta_3\vert_t\ci\eta_1:(E,\db_E)\ra
(E,\db_E)$, so that $\eta_2\ci\eta_3\vert_t\ci\eta_1$ lies
in~$\Aut(E,\db_E)$.

Following the definitions through we find that $\rd\Xi\vert_0:
\Ext^1({\cal E},{\cal E})\!\ra\!\Ext^1({\cal E},{\cal E})$ is
conjugation by $\ga=\eta_2\ci\eta_3\vert_t\ci\eta_1$ in
$\Aut(E,\db_E)=\Aut({\cal E})$. But in constructing $\eta_3$ we were
free to choose the isomorphism $\eta_3\vert_t:{\cal F}_t\ra({\cal
D}_\an)_s$, and we choose it to make $\ga=\id_{\cal E}$, so that
$\rd\Xi\vert_0$ is the identity on $\Ext^1({\cal E},{\cal E})$. This
proves the first part of the third paragraph of Theorem
\ref{dt5thm3}. It remains to prove the final part, that if $G$ is a
maximal compact subgroup of $\Aut({\cal E})$ then we can take $U,f$
to be $G^{\sst\C}$-invariant, and $\Xi$ to be
$G^{\sst\C}$-equivariant.

First we show that we can take $U,f$ to be $G$-invariant. Now
$\Aut(E,\db_E)$ acts on $\sA^{2,k}$ fixing $\db_E$ by
$\ga:\db_E+A\mapsto \db_E+\ga^{-1}\ci A\ci\ga$, as in \eq{dt9eq1},
since $\db_E\ga=0$ for $\ga\in\Aut(E,\db_E)$. However, the
construction of $(T,\tau)$ in \S\ref{dt91} involves a choice of
metric $h_E$ on the fibres of $E$, which is used to define
$\db_E^*$, and the norm in the condition $\nm{A}_{L^2_k}<\ep$ in
\eq{dt9eq4}. By averaging $h_E$ over the action of $G$, which is
compact, we can choose $h_E$ to be $G$-invariant. Then $\db_E^*$ is
$G$-equivariant, and $\nm{\,\cdot\,}_{L^2_k}$ is $G$-invariant, so
$Q_\ep$ in \eq{dt9eq4} is $G$-invariant, and as $P_k$ is
$G$-equivariant the analytic subspace $T=(P_k\vert_{Q_\ep})^{-1}(0)$
in $Q_\ep$ is also $G$-invariant.

In \S\ref{dt96}, since $\rd CS$ maps $\db_E'$ to its
$(0,2)$-curvature by \eq{dt9eq7}, it is equivariant under the gauge
group $\sG$, so its first integral $CS:\sA^{2,k}\ra\C$ is invariant
under the subgroup $\Aut(E,\db_E)$ of $\sG$ fixing the point $\db_E$
in $\sA^{2,k}$, and $CS\vert_{Q_\ep}$ is invariant under the maximal
compact subgroup $G$ of $\Aut(E,\db_E)$. We choose the
identification of $Q_\ep$ with an open subset $U$ of $\Ext^1({\cal
E},{\cal E})$ to be the composition of the map $Q_\ep\ra\sE^1$
taking $\db_E+A\mapsto\pi_{\sE^1}(A)$ with the isomorphism
$\sE^1\cong\Ext^1 ({\cal E},{\cal E})$. As both of these are
$G$-equivariant, we see that $U\subset\Ext^1 ({\cal E},{\cal E})$
and $f:U\ra\C$ are both $G$-invariant.

Then in Proposition \ref{dt9prop7}, each of $(T,\tau),(T,{\cal
F}),(S,{\cal D})$ is equivariant under an action of $G$, which fixes
$t,0$ and acts on $T_tT\cong\Ext^1({\cal E},{\cal E})\cong T_0S$
through the action of $\Aut({\cal E})$ on $\Ext^1({\cal E},{\cal
E})$. We can choose the isomorphism of versal families of analytic
vector bundles in Proposition \ref{dt9prop7} to be $G$-equivariant,
since the proofs of the versality property extend readily to
equivariant versality under a compact Lie group. This then implies
that $\Xi:\Crit(f)\ra V$ is $G$-equivariant.

Next we modify $U,f,\Xi$ to make them $G^{\sst\C}$-invariant or
$G^{\sst\C}$-equivariant. Let $U'$ be a $G$-invariant connected open
neighbourhood of $0$ in $U\subseteq\Ext^1({\cal E},{\cal E})$.
Define $V'=\Xi(U')\subset S_\an$. Define $U^{\sst\C}=G^{\sst\C}\cdot
U'$ in $\Ext^1({\cal E},{\cal E})$ and $V^{\sst\C}=G^{\sst\C}\cdot
V'$ in $S_\an$. Then $U^{\sst\C},V^{\sst\C}$ are
$G^{\sst\C}$-invariant, and are open in $\Ext^1({\cal E},{\cal E}),
S_\an$, as they are unions of open sets $\ga\cdot U,\ga\cdot V$ over
all~$\ga\in G^{\sst\C}$.

We wish to define $f^{\sst\C}:U^{\sst\C}\ra\C$ by
$f^{\sst\C}(\ga\cdot u)=f(u)$ for $\ga\in G^{\sst\C}$ and $u\in U'$,
and $\Xi^{\sst\C}:\Crit(f^{\sst\C})\ra V^{\sst\C}$ by $\Xi(\ga\cdot
u)=\ga\cdot\Xi(u)$ for $\ga\in G^{\sst\C}$ and
$u\in\Crit(f\vert_{U'})$. Clearly $f^{\sst\C}$ is
$G^{\sst\C}$-invariant, and $\Xi^{\sst\C}$ is
$G^{\sst\C}$-equivariant, provided they are well-defined. To show
they are, we must prove that if $\ga_1,\ga_2\in G^{\sst\C}$ and
$u_1,u_2\in U'$ with $\ga_1\cdot u_1=\ga_2\cdot u_2$ then
$f(u_1)=f(u_2)$, and $\ga_1\cdot\Xi(u_1)=\ga_2\cdot\Xi(u_2)$.

The $G^{\sst\C}$-orbit $G^{\sst\C}\cdot u_1\!=\!G^{\sst\C}\cdot u_2$
is a $G$-invariant complex submanifold of $\Ext^1({\cal E},{\cal
E})$, so $(G^{\sst\C}\cdot u_1)\cap U$ is a $G$-invariant complex
submanifold of $U$. Since $f$ is $G$-invariant, it is constant on
each $G$-orbit in $(G^{\sst\C}\cdot u_1)\cap U$, so as $f$ is
holomorphic it is constant on each connected component of
$(G^{\sst\C}\cdot u_1)\cap U$. We require that the $G$-invariant
open neighbourhood $U'$ of $0$ in $U$ should satisfy the following
condition: whenever $u_1,u_2\in U'$ with $G^{\sst\C}\cdot
u_1=G^{\sst\C}\cdot u_2$, then the connected component of
$(G^{\sst\C}\cdot u_1)\cap U$ containing $u_1$ should intersect
$G\cdot u_2$. This is true provided $U'$ is sufficiently small.

Suppose this condition holds. Then $f$ is constant on the connected
component of $(G^{\sst\C}\cdot u_1)\cap U$ containing $u_1$, with
value $f(u_1)$. This component intersects $G\cdot u_2$, so it
contains $\ga\cdot u_2$ for $\ga\in G$. Hence $f(u_1)=f(\ga\cdot
u_2)=f(u_2)$ by $G$-invariance of $f$, and $f^{\sst\C}$ is
well-defined. To show $\Xi^{\sst\C}$ is well-defined we use a
similar argument, based on the fact that if $\ga\in G^{\sst\C}$ and
$u,\ga\cdot u$ lie in the same connected component of
$(G^{\sst\C}\cdot u)\cap U$ then $\Xi(\ga\cdot u)=\ga\cdot\Xi(u)$,
since this holds for $\ga\in G$ and $\Xi$ is holomorphic. Then
$U^{\sst\C},f^{\sst\C},V^{\sst\C},\Xi^{\sst\C}$ satisfy the last
part of Theorem \ref{dt5thm3}, completing the proof.\index{complex
analytic space|)}

\section[The proof of Theorem $\text{\ref{dt5thm4}}$]{The proof of
Theorem \ref{dt5thm4}}
\label{dt10}\index{Behrend function!identities|(}

Next we prove Theorem \ref{dt5thm4}. Sections \ref{dt101} and
\ref{dt102} prove equations \eq{dt5eq2} and \eq{dt5eq3}. The authors
got an important idea in the proof, that of proving
\eq{dt5eq2}--\eq{dt5eq3} by localizing at the fixed points of the
action of $\bigl\{\id_{E_1}+\la\id_{E_2}: \la\in\U(1)\bigr\}$ on
$\Ext^1(E_1\op E_2,E_1\op E_2)$, from Kontsevich and
Soibelman~\cite[\S 4.4 \& \S 6.3]{KoSo1}.

\subsection[Proof of equation $\text{\eq{dt5eq2}}$]{Proof of equation
\eq{dt5eq2}}
\label{dt101}

We now prove equation \eq{dt5eq2} of Theorem \ref{dt5thm4}. Let $X$
be a Calabi--Yau $3$-fold over $\C,$ $\fM$ the moduli stack of
coherent sheaves on $X$, and $E_1,E_2$ be coherent sheaves on $X$.
Set $E=E_1\op E_2$. Choose a maximal compact subgroup\index{maximal
compact subgroup} $G$ of $\Aut(E)$ which contains the
$\U(1)$-subgroup $T=\bigl\{\id_{E_1}+\la\id_{E_2} :\la\in\U(1)\}$.
Apply Theorem \ref{dt5thm3} with these $E$ and $G$. This gives an
$\Aut(E)$-invariant $\C$-subscheme $S$ in $\Ext^1(E,E)$ with $0\in
S$ and $T_0S=\Ext^1(E,E)$, an \'etale 1-morphism
$\Phi:[S/\Aut(E)]\ra\fM$ with $\Phi([0])=[E]$, a
$G^{\sst\C}$-invariant open neighbourhood $U$ of 0 in $\Ext^1(E,E)$
in the analytic topology, a $G^{\sst\C}$-invariant holomorphic
function $f:U\ra\C$ with $f(0)=\rd f\vert_0=0$, a
$G^{\sst\C}$-invariant open neighbourhood $V$ of 0 in $S_\an$, and a
$G^{\sst\C}$-equivariant isomorphism of complex analytic spaces
$\Xi:\Crit(f)\ra V$ with $\Xi(0)=0$ and\/ $\rd\Xi\vert_0$ the
identity map on $\Ext^1(E,E)$.

Then the Behrend function $\nu_\fM$ at $[E]=[E_1\op E_2]$ satisfies
\e
\begin{split}
\nu_{\fM}(E_1\op E_2)&=\nu_{[S/\Aut(E)]}(0)=(-1)^{\dim\Aut(E)}
\nu_S(0)\\
&=(-1)^{\dim\Aut(E)+\dim\Ext^1(E,E)}\bigl(1-\chi(MF_f(0))\bigr),
\end{split}
\label{dt10eq10}
\e
where in the first step we use that as $\Phi$ is \'etale it is
smooth of relative dimension 0, Theorem \ref{dt4thm1}(ii), and
Corollary \ref{dt4cor1}, in the second step Proposition
\ref{dt4prop2}, and in the third Theorem~\ref{dt4thm2}.

To define the Milnor fibre\index{Milnor fibre} $MF_f(0)$ of $f$ we use a
Hermitian metric on $\Ext^1(E,E)$ invariant under the action of the
compact Lie group $G$. Since $U,f$ are $G$-invariant, it follows
that $\Phi_{f,0}$ and its domain is $G$-invariant, so each fibre
$\Phi_{f,0}^{-1}(z)$ for $0<\md{z}<\ep$ is $G$-invariant. Thus $G$,
and its $\U(1)$-subgroup $T$, acts on the Milnor fibre $MF_f(0)$.
Now $MF_f(0)$ is a manifold, the interior of a compact manifold with
boundary $\overline{MF_f(0)}$, and $T$ acts smoothly on $MF_f(0)$
and $\overline{MF_f(0)}$. Each orbit of $T$ on $MF_f(0)$ is either a
single point, a fixed point of $T$, or a circle ${\cal S}^1$. The
circle orbits contribute zero to $\chi(MF_f(0))$, as $\chi({\cal
S}^1)=0$, so
\e
\chi\bigl(MF_f(0)\bigr)=\chi\bigl(MF_f(0)^T\bigr),
\label{dt10eq11}
\e
where $MF_f(0)^T$ is the fixed point set of $T$ in~$MF_f(0)$.

Consider how $T=\bigl\{\id_{E_1}+\la\id_{E_2}:\la\in\U(1)\}$ acts on
\e
\Ext^1(E,E)\!=\!\Ext^1(E_1,E_1)\!\times\!\Ext^1(E_2,E_2)\!\times\!\Ext^1(E_1,E_2)
\!\times\!\Ext^1(E_2,E_1).
\label{dt10eq12}
\e
As in Theorem \ref{dt5thm3}, $\ga\in T$ acts on $\ep\in\Ext^1(E,E)$
by $\ga:\ep\mapsto\ga\ci\ep\ci\ga^{-1}$. So $\id_{E_1}+\la\id_{E_2}$
fixes the first two factors on the r.h.s.\ of \eq{dt10eq12},
multiplies the third by $\la^{-1}$ and the fourth by $\la$.
Therefore
\e
\Ext^1(E,E)^T=\Ext^1(E_1,E_1)\times\Ext^1(E_2,E_2)\times\{0\}\times\{0\}.
\label{dt10eq13}
\e
Now $MF_f(0)^T=MF_f(0)\cap\Ext^1(E,E)^T=
MF_{f\vert_{\Ext^1(E,E)^T}}(0)$. But $\Crit(f)^T=
\Crit(f\vert_{\Ext^1(E,E)^T})$. Also as $\Xi$ is $T$-equivariant, it
induces a local isomorphism of complex analytic spaces between
$S_\an^T$ near 0 and $\Crit(f)^T$ near 0. Hence
\e
\begin{split}
\nu_{S^T}(0)&=(-1)^{\dim\Ext^1(E_1,E_1)+\dim\Ext^1(E_2,E_2)}
\bigl(1-\chi(MF_{f\vert_{\Ext^1(E,E)^T}}(0)\bigr)\\
&=(-1)^{\dim\Ext^1(E_1,E_1)+\dim\Ext^1(E_2,E_2)}
\bigl(1-\chi(MF_f(0)^T)\bigr)\\
&=(-1)^{\dim\Ext^1(E_1,E_1)+\dim\Ext^1(E_2,E_2)}
\bigl(1-\chi(MF_f(0))\bigr),
\end{split}
\label{dt10eq14}
\e
using Theorem \ref{dt4thm2} and equations \eq{dt10eq11}
and~\eq{dt10eq13}.

Let $s'\in S^T(\C)\subseteq S(\C)$, and set $[E']=\Phi_*(s')$ in
$\fM(\C)$, so that $E'\in\coh(X)$. As $\Phi$ is \'etale,\index{etale
morphism@\'etale morphism} it induces isomorphisms of stabilizer
groups.\index{Artin stack!stabilizer group} But $\Iso_{[S/\Aut(E)]}(s')=
\Stab_{\Aut(E)}(s')$, and $\Iso_{\fM}([E'])=\Aut(E')$, so we have an
isomorphism of complex Lie groups $\Phi_*:\Stab_{\Aut(E)}(s')
\ra\Aut(E')$. As $s'\in S^T(\C)$ we have $T\subset\Stab_{\Aut(E)}
(s')$, so $\Phi_*\vert_T:T\ra\Aut(E')$ is an injective morphism of
Lie groups. Let $R$ be the $\C$-subscheme of points $s'$ in $S^T$
for which $\Phi_*\vert_T(\id_{E_1}+\la\id_{E_2})=
\id_{E_1'}+\la\id_{E_2'}$ for some splitting $E'\cong E_1'\op E_2'$
and all $\la\in\U(1)$. Taking $E_1'=E_1$, $E_2'=E_2$ shows
that~$0\in R(\C)$.

We claim $R$ is open and closed in $S^T$. To see this, note that
$\Phi_T$ is of the form $\Phi_*\vert_T(\id_{E_1}+\la\id_{E_2})=
\la^{a_1}\id_{F_1}+\cdots+ \la^{a_k}\id_{F_k}$, for some splitting
$E'=F_1\op\cdots\op F_k$ with $F_1,\ldots,F_k$ indecomposable and
$a_1,\ldots,a_k\in\Z$. Then $R$ is the subset of $s'$ with
$\{a_1,\ldots,a_k\}=\{0,1\}$. Therefore we see that
\e
E'\cong\ts\bigoplus_{a\in\Z}\Ker\bigl(\la^a\id_{E'}-\Phi_*
\vert_T(\id_{E_1}+\la\id_{E_2})\bigr)
\label{dt10eq15}
\e
for $\la\in\U(1)$ not of finite order, with only finitely many
nonzero terms. Now the Hilbert polynomial\index{Hilbert polynomial} at
$n\gg 0$ of each term on the r.h.s.\ of \eq{dt10eq15} is upper
semicontinuous in $S^T$, and of the l.h.s.\ is locally constant in
$S^T$. Hence the Hilbert polynomials of each term in \eq{dt10eq15}
are locally constant in $S^T$, and in particular, whether
$\Ker\bigl(\la^a\id_{E'} -\Phi_*
\vert_T(\id_{E_1}+\la\id_{E_2})\bigr)\ne 0$ is locally constant in
$S^T$. As $R$ is the subset of $s'$ with $\Ker\bigl(\la^a\id_{E'}
-\Phi_* \vert_T(\id_{E_1}+\la\id_{E_2})\bigr)\ne 0$ if and only if
$a=0,1$, we see $R$ is open and closed in~$S^T$.

The subgroup $\Aut(E_1)\times\Aut(E_2)$ of $\Aut(E)$ commutes with
$T$. Hence the action of $\Aut(E_1)\times\Aut(E_2)$ on $S$ induced
by the action of $\Aut(E)$ on $S$ preserves $S^T$. The action of
$\Aut(E_1)\times\Aut(E_2)$ on $s'\in S^T(\C)$ does not change $E'$
or $\Phi_*\vert_T:T\ra\Aut(E')$ above up to isomorphism, so
$\Aut(E_1)\times\Aut(E_2)$ also preserves $R$. Hence we can form the
quotient stack $[R/\Aut(E_1)\times\Aut(E_2)]$. The inclusions
$R\hookra S$, $\Aut(E_1)\times\Aut(E_2)\hookra\Aut(E)$ induce a
1-morphism of quotient stacks
$\io:[R/\Aut(E_1)\times\Aut(E_2)]\ra[S/\Aut(E)]$. The family of
coherent sheaves parametrized by $S$, $E_S$, pulls back to a family
of coherent sheaves, $E_R$, parametrized by $R$. By definition of
$R$ we have a global splitting $E_R \cong E_{R,1} \oplus E_{R,2}$,
where $E_{R,1},E_{R,2}$ are the eigensubsheaves of $\Phi_*
\vert_T(\id_{E_1}+\la\id_{E_2})$ in $E_R$ with eigenvalues $1,\la$.
These $E_{R,1},E_{R,2}$ induce a 1-morphism $\Psi$ from
$[R/\Aut(E_1) \times \Aut(E_2)]$ to~$\fM\times\fM$.

Then we have a commutative diagram of 1-morphisms of Artin
$\C$-stacks
\e
\begin{gathered}
\xymatrix@C=60pt@R=15pt{
[R/\Aut(E_1)\times\Aut(E_2)]\ar[r]_(0.55)\io \ar[d]^\Psi &
[S/\Aut(E)] \ar[d]_\Phi \\ \fM\times\fM \ar[r]^(0.55)\La & \fM,}
\end{gathered}
\label{dt10eq16}
\e
where $\La:\fM\times\fM\ra\fM$ is the 1-morphism acting on points as
$\La:(E_1',E_2')\mapsto E_1'\op E_2'$, such that $\Psi$ maps $[0]$
to $[(E_1,E_2)]$, with $\Psi_*:\Iso_{[R/\Aut(E_1)\times\Aut
(E_2)]}(0)\ab\ra \Iso_{\fM\times\fM}(E_1,E_2)$ the identity map on
$\Aut(E_1)\times\Aut(E_2)$. Furthermore, we will show that
\eq{dt10eq16} is {\it locally\/ $2$-Cartesian}, in the sense that
$[R/\Aut(E_1)\times \Aut(E_2)]$ is 1-isomorphic to an open substack
$\fN$ of the fibre product
$(\fM\times\fM)\times_{\La,\fM,\Phi}[S/\Aut(E)]$. Since the diagram
\eq{dt10eq16} commutes, there exists a 1-morphism
$\chi:[R/\Aut(E_1)\times\Aut(E_2)]\ra(\fM\times\fM)\times_{\La,\fM,
\Phi}[S/\Aut(E)]$. It is sufficient to construct a local inverse
for~$\chi$.

The reason it may not be globally 2-Cartesian is that there might be
points $s'\in S$ with $\Phi_*([s'])=[E_1'\op E_2']$, so that
$\Phi_*:\Stab_{\Aut(E)}(s')\ra\Aut(E_1'\op E_2')$ is an isomorphism,
but such that the $\U(1)$-subgroup
$\Phi_*^{-1}\bigl(\{\id_{E_1'}+\la\id_{E_2'}:\la\in\U(1)\}\bigr)$ in
$\Aut(E)$ is not conjugate to $T$ in $\Aut(E)$. Then $s',E_1',E_2'$
would yield a point in
$(\fM\times\fM)\times_{\La,\fM,\Phi}[S/\Aut(E)]$ not corresponding
to a point of $[R/\Aut(E_1)\times\Aut(E_2)]$. However, since
$\U(1)$-subgroups of $\Aut(E)$ up to conjugation are discrete
objects, the condition that $\Phi_*^{-1}\bigl(\{\id_{E_1'}+
\la\id_{E_2'}:\la\in\U(1)\}\bigr)$ is conjugate to $T$ in $\Aut(E)$
is open in $(\fM\times\fM)\times_{\La,\fM,\Phi} [S/\Aut(E)]$. Write
$\fN$ for this open substack of
$(\fM\times\fM)\times_{\La,\fM,\Phi}[S/\Aut(E)]$. Then $\chi$
maps~$[R/\Aut(E_1)\times\Aut(E_2)]\ra\fN$.

Let $B$ be a base $\C$-scheme and $\theta:B\ra\fN$ a 1-morphism.
Then $(B,\theta)$ parametrizes the following objects: a principal
$\Aut(E)$-torsor $\eta:P\ra B$; an $\Aut(E)$-equivariant morphism
$\ze:P\ra S$; a $B$-family of coherent sheaves $E_B\cong E_{B,1}\op
E_{B,2}$; and an isomorphism $\ze^*(E_S)\cong\eta^*(E_B)$, where
$E_S$ is the family of coherent sheaves parametrized by $S$. The
open condition on $\fN$ implies that $\ze$ maps $P$ into $R\subset
S^T$. The isomorphism between $\ze^*(E_S)$ and $\eta^*(E_B)$ implies
there exists an $\bigl(\Aut(E_1)\times\Aut(E_2)\bigr)$-subtorsor $Q$
of $P$ over $B$ and the restriction of $\ze$ to $Q$ is
$\bigl(\Aut(E_1)\times \Aut(E_2)\bigr)$-equivariant. Therefore
$\theta$ induces a 1-morphism
$\ka:B\ra[R/\Aut(E_1)\times\Aut(E_2)]$. As this holds functorially
for all $B,\theta$ there is a 1-morphism $\xi:\fN\ra
[R/\Aut(E_1)\times \Aut(E_2)]$ with $\ka$ 2-isomorphic to
$\xi\ci\theta$ for all such $B,\theta$, and $\xi$ is the required
inverse for~$\chi$.

Since \eq{dt10eq16} is locally 2-Cartesian and $\Phi$ is \'etale,
$\Psi$ is \'etale. Thus $\Psi$ is smooth of relative dimension 0,
and Corollary \ref{dt4cor1} and Theorem \ref{dt4thm1}(ii) imply that
$\nu_{[R/\Aut(E_1)\times\Aut(E_2)]}=\Psi^*(\nu_{\fM\times\fM})$.
Hence
\e
\begin{split}
\nu_\fM(E_1)\nu_\fM(E_2)&=\nu_{\fM\times\fM}(E_1,E_2)=
\nu_{[R/\Aut(E_1)\times\Aut(E_2)]}(0)\\
&=(-1)^{\dim\Aut(E_1)+\dim\Aut(E_2)}\nu_R(0)\\
&=(-1)^{\dim\Aut(E_1)+\dim\Aut(E_2)}\nu_{S^T}(0),
\end{split}
\label{dt10eq17}
\e
using Theorem \ref{dt4thm1}(iii) and Corollary \ref{dt4cor1} in the
first step,
$\nu_{[R/\Aut(E_1)\times\Aut(E_2)]}=\Psi^*(\nu_{\fM\times\fM})$ and
$\Psi_*([0])=[(E_1,E_2)]$ in the second, Proposition \ref{dt4prop2}
in the third, and $R$ open in $S^T$ in the fourth.

Combining equations \eq{dt10eq10}, \eq{dt10eq14} and \eq{dt10eq17}
yields
\e
\begin{split}
\nu_{\fM}(E_1\op E_2)=\, &(-1)^{\dim\Aut(E)+\dim\Ext^1(E,E)}\\
&(-1)^{\dim\Ext^1(E_1,E_1)+\dim\Ext^1(E_2,E_2)}\\
&(-1)^{\dim\Aut(E_1)+\dim\Aut(E_2)}\nu_\fM(E_1)\nu_\fM(E_2).
\end{split}
\label{dt10eq18}
\e
To sort out the signs, note that $\Aut(E)$ is open in
\begin{equation*}
\Hom(E,E)=\Hom(E_1,E_1)\op\Hom(E_2,E_2)\op\Hom(E_1,E_2)\op
\Hom(E_2,E_1).
\end{equation*}
Cancelling $(-1)^{\dim\Hom(E_i,E_i)}$, $(-1)^{\dim\Ext^1(E_i,E_i)}$
for $i=1,2$, the sign in \eq{dt10eq18} becomes
$(-1)^{\dim\Hom(E_1,E_2)+\dim\Hom(E_2,E_1)
+\dim\Ext^1(E_1,E_2)+\dim\Ext^1(E_2,E_1)}$. As $X$ is a Calabi--Yau
3-fold, Serre duality gives $\dim\Hom(E_2,E_1)=\dim\Ext^3(E_1,E_2)$
and $\dim\Ext^1(E_2,E_1)=\dim\Ext^2(E_1,E_2)$. Hence the overall
sign in \eq{dt10eq18} is
\begin{equation*}
(-1)^{\dim\Hom(E_1,E_2)-\dim\Ext^1(E_1,E_2)+\dim\Ext^2(E_1,E_2)
-\dim\Ext^3(E_1,E_2)},
\end{equation*}
which is $(-1)^{\bar\chi([E_1],[E_2])}$, proving~\eq{dt5eq2}.

\subsection[Proof of equation $\text{\eq{dt5eq3}}$]{Proof of equation
\eq{dt5eq3}}
\label{dt102}

We continue to use the notation of \S\ref{dt101}. Using the
splitting \eq{dt10eq12}, write elements of $\Ext^1(E,E)$ as
$(\ep_{11},\ep_{22},\ep_{12},\ep_{21})$ with
$\ep_{ij}\in\Ext^1(E_i,E_j)$.

\begin{prop} Let\/ $\ep_{12}\in\Ext^1(E_1,E_2)$ and\/
$\ep_{21}\in\Ext^1(E_2,E_1)$. Then
\begin{itemize}
\setlength{\itemsep}{0pt}
\setlength{\parsep}{0pt}
\item[{\bf(i)}] $(0,0,\ep_{12},0),(0,0,0,\ep_{21})\in\Crit(f)
\subseteq U\subseteq\Ext^1(E,E),$ and\/ $(0,0,\ep_{12},0),\ab
(0,\ab 0,\ab 0,\ab\ep_{21})\in V\subseteq
S(\C)\subseteq\Ext^1(E,E);$
\item[{\bf(ii)}] $\Xi$ maps $(0,0,\ep_{12},0) \mapsto
(0,0,\ep_{12},0)$ and\/
$(0,0,0,\ep_{21})\mapsto(0,0,0,\ep_{21});$ and
\item[{\bf(iii)}] $\Phi_*:[S/\Aut(E)] (\C)\ra\fM(\C),$ the
induced morphism on closed points, maps $[(0,\ab 0,\ab
0,\ab\ep_{21})]\mapsto[F]$ and\/ $[(0,0,\ep_{12},0)]\mapsto
[F'],$ where the exact sequences $0\ra E_1\ra F\ra E_2\ra 0$
and\/ $0\ra E_2\ra F'\ra E_1\ra 0$ in $\coh(X)$ correspond to
$\ep_{21}\in\Ext^1(E_2,E_1)$ and\/ $\ep_{12}\in\Ext^1(E_1,E_2),$
respectively.
\end{itemize}
\label{dt10prop2}
\end{prop}

\begin{proof} For (i) $T^{\sst\C}=\bigl\{\id_{E_1}+\la\id_{E_2}:
\la\in\bG_m\bigr\}$, which acts on $\Ext^1(E,E)$ by
\e
\la:(\ep_{11},\ep_{22},\ep_{12},\ep_{21})
\mapsto(\ep_{11},\ep_{22},\la^{-1}\ep_{12},\la\ep_{21}).
\label{dt10eq19}
\e
Since $U$ is an open neighbourhood of 0 in $\Ext^1(E,E)$ in the
analytic topology, we see that $(0,0,\la^{-1}\ep_{12},0)\in U$  for
$\md{\la}\gg 1$ and $(0,0,0,\la\ep_{21})\in U$ for $0<\md{\la}\ll
1$. Hence $(0,0,\ep_{12},0),(0,0,0,\ep_{21})\in U$ as $U$ is
$G^{\sst\C}$-invariant, and so $T^{\sst\C}$-invariant.

As $f$ is $T^{\sst\C}$-invariant we have
$f(\ep_{11},\ep_{22},\ep_{12},0)=
f(\ep_{11},\ep_{22},\la^{-1}\ep_{12},0)$, so taking the limit
$\la\ra\iy$ and using continuity of $f$ gives
$f(\ep_{11},\ep_{22},\ep_{12},0)=f(\ep_{11},\ep_{22},0,0)$.
Similarly $f(\ep_{11},\ep_{22},0,\ep_{21})=
f(\ep_{11},\ep_{22},0,0)$. But $f(0,0,0,0)=\rd f\vert_0=0$, so we
see that $f(0,0,\ep_{12},0)=f(0,0,0,\ep_{21})=0$, and
\e
\rd f\vert_{(0,0,\ep_{12},0)}\cdot(\ep_{11}',\ep_{22}',\ep_{12}',0)
=0,\qquad \rd
f\vert_{(0,0,0,\ep_{21})}\cdot(\ep_{11}',\ep_{22}',0,\ep_{21}')=0.
\label{dt10eq20}
\e
Now by \eq{dt10eq19}, $T^{\sst\C}$-invariance of $f$ and linearity
in $\ep_{12}'$ we see that
\begin{equation*}
\rd f\vert_{(0,0,0,\ep_{21})}\cdot(0,0,\ep_{12}',0)=
\la^{-1}\rd f\vert_{(0,0,0,\la\ep_{21})}\cdot(0,0,\ep_{12}',0).
\end{equation*}
Using this and $\rd f\vert_0=0$ to differentiate $\rd
f\cdot(0,0,\ep_{12}',0)$ at 0, we find that
\begin{align*}
(\pd^2 f)&\vert_0\cdot (\ep_{21}\ot \ep_{12}')\\
&=\ts \lim_{\la\ra 0}\la^{-1}\bigl(\rd f\vert_{(0,0,0,\la\ep_{21})}\cdot(0,0,\ep_{12}',0)-\rd f\vert_{(0,0,0,0)}\cdot(0,0,\ep_{12}',0)\bigr)\\
&=\ts \lim_{\la\ra 0}\bigl(\rd f\vert_{(0,0,0,\ep_{21})}\cdot(0,0,\ep_{12}',0)-0\bigr) =\rd f\vert_{(0,0,0,\ep_{21})}\cdot(0,0,\ep_{12}',0).
\end{align*}
But $T_0\Crit(f)=\Ext^1(E,E)$, which implies that $(\pd^2
f)\vert_0=0$, so $\rd f\vert_{(0,0,0,\ep_{21})}
\cdot(0,0,\ep_{12}',0)=0$. Together with \eq{dt10eq20} this gives
$\rd f\vert_{(0,0,0,\ep_{21})}=0$, and similarly $\rd
f\vert_{(0,0,\ep_{12},0)}=0$. Therefore $(0,0,\ep_{12},0),
(0,0,0,\ep_{21})\in\Crit(f)\subseteq U\subseteq\Ext^1(E,E)$, as we
have to prove.

For (ii), let $\Xi(0,0,0,\ep_{21})=(\ep_{11}',\ep_{22}',\ep_{12}',
\ep_{21}')$. As $\Xi$ is $T^{\sst\C}$-equivariant, this gives
$\Xi(0,0,0,\la\ep_{21})=(\ep_{11}',\ep_{22}',\la^{-1}\ep_{12}',
\la\ep_{21}')$. But $\Xi(0)=0$ and $\Xi$ is continuous, so taking
the limit $\la\ra 0$ gives $\Xi(0,0,0,\ep_{21})=(0,0,0,\ep_{21}')$.
Thus $\Xi(0,0,0,\la\ep_{21})=(0,0,0,\la\ep_{21}')$. But
$\rd\Xi\vert_0$ is the identity on $\Ext^1(E,E)$, which forces
$\ep_{21}'=\ep_{21}$. Hence $\Xi(0,0,0,\ep_{21})=(0,0,0,\ep_{21})$,
so that $(0,0,0,\ep_{21})\in V$, and similarly
$\Xi(0,0,\ep_{12},0)=(0,0,\ep_{12},0)$ with $(0,0,\ep_{12},0)\in V$,
as we want.

Part (iii) is trivial when $\ep_{21}=\ep_{12}=0$ and $F=F'=E$, so
suppose $\ep_{21},\ep_{12}\ne 0$. Then $[F]$ is the unique point in
$\fM(\C)$, with its nonseparated topology, which is distinct from
$[E]$ but infinitesimally close to $[E]$ in direction
$(0,0,0,\ep_{21})$ in $T_{[E]}\fM=\Ext^1(E,E)$. Similarly,
$[(0,0,\ep_{12},0)]$ is the unique point in $[S/\Aut(E)]$, with its
nonseparated topology, which is distinct from $[0]$ but
infinitesimally close to $[0]$ in direction $(0,0,\ep_{12},0)$ in
$T_{[0]}[S/\Aut(E)]=\Ext^1(E,E)$. But $\Phi_*$ maps $[0]\mapsto[E]$,
and $\rd\Phi_*:T_{[0]}[S/\Aut(E)]\ra T_{[E]}\fM$ is the identity on
$\Ext^1(E,E)$. It follows that $\Phi_*$ maps
$[(0,0,0,\ep_{21})]\mapsto[F]$, and similarly $\Phi_*$
maps~$[(0,0,\ep_{12},0)]\mapsto[F']$.
\end{proof}

Let $0\ne\ep_{21}\in\Ext^1(E_2,E_1)$ correspond to the short exact
sequence $0\ra E_1\ra F\ra E_2\ra 0$ in $\coh(X)$. Then\index{Milnor
fibre|(}
\e
\begin{split}
\nu_{\fM}(F)&=\nu_{[S/\Aut(E)]}(0,0,0,\ep_{21})
=(-1)^{\dim\Aut(E)}\nu_S(0,0,0,\ep_{21})\\
&=(-1)^{\dim\Aut(E)+\dim\Ext^1(E,E)}\bigl(1-
\chi(MF_f(0,0,0,\ep_{21}))\bigr),
\end{split}
\label{dt10eq21}
\e
using $\Phi_*:[(0,0,0,\ep_{21})]\mapsto[F]$ from Proposition
\ref{dt10prop2}, $\Phi$ smooth of relative dimension 0, Corollary
\ref{dt4cor1} and Theorem \ref{dt4thm1}(ii) in the first step,
Proposition \ref{dt4prop2} in the second, and
$\Xi:(0,0,0,\ep_{21})\mapsto(0,0,0,\ep_{21})$ from Proposition
\ref{dt10prop2} and Theorem \ref{dt4thm2} in the last step.

Substituting \eq{dt10eq21} and its analogue for $F'$ into
\eq{dt5eq3}, using equation \eq{dt10eq10} and
$\chi(MF_f(0))\ab=\chi(MF_{f \vert_{\Ext^1(E,E)^T}}(0))$ from
\S\ref{dt101} to substitute for $\nu_{\fM}(E_1\op E_2)$, and
cancelling factors of $(-1)^{\dim\Aut(E)+\dim\Ext^1(E,E)}$, we see
that \eq{dt5eq3} is equivalent to
\e
\begin{split}
&\int_{[\ep_{21}]\in\mathbb{P}(\Ext^1(E_2,E_1))} \!\!\!\!\!\!
\!\!\!\!\!\!\!\!\!\!\!\!\!\!\!\!\!\!\!\!\!\!\!\!\!\!\!\!\!\!
(1-\chi(MF_f(0,0,0,\ep_{21})))\,\rd\chi -
\int_{[\ep_{12}]\in\mathbb{P}(\Ext^1(E_1,E_2))} \!\!\!\!\!\!
\!\!\!\!\!\!\!\!\!\!\!\!\!\!\!\!\!\!\!\!\!\!\!\!\!\!\!\!\!\!
(1-\chi(MF_f(0,0,\ep_{12},0)))\,\rd\chi\\
&\;\>=\bigl(\dim\Ext^1(E_2,E_1)-\dim\Ext^1(E_1,E_2)\bigr)
\bigl(1-\chi(MF_{f\vert_{\Ext^1(E,E)^T}}(0))\bigr).
\end{split}
\label{dt10eq22}
\e
Here $\chi(MF_f(0,0,0,\ep_{21}))$ is independent of the choice of
$\ep_{21}$ representing the point $[\ep_{21}]\in\mathbb{P}
(\Ext^1(E_2,E_1))$, and is a constructible function\index{constructible
function} of $[\ep_{21}]$, so the integrals in \eq{dt10eq22} are
well-defined.

Set $U'=\bigl\{(\ep_{11},\ep_{22},\ep_{12},\ep_{21})\in
U:\ep_{21}\ne 0\bigr\}$, an open set in $U$, and write $V'$ for the
submanifold of $(\ep_{11},\ep_{22},\ep_{12},\ep_{21})\in U'$ with
$\ep_{12}=0$. Let $\ti U'$ be the blowup of $U'$ along $V'$, with
projection $\pi':\ti U'\ra U'$. Points of $\ti U'$ may be written
$(\ep_{11},\ep_{22},[\ep_{12}],\la\ep_{12},\ep_{21})$, where
$[\ep_{12}]\in\mathbb{P} (\Ext^1(E_1,E_2))$, and $\la\in\C$, and
$\ep_{21}\ne 0$. Write $f'=f\vert_{U'}$ and $\ti f'=f'\ci\pi'$. Then
applying Theorem \ref{dt4thm4} to $U',V',f',\ti U',\pi',\ti f'$ at
the point $(0,0,0,\ep_{21})\in U'$ gives
\e
\begin{split}
\chi\bigl(MF_f(0,0,0&,\ep_{21})\bigr)=\int_{[\ep_{12}]
\in\mathbb{P}(\Ext^1(E_1,E_2))}\!\!\!\!\!\!\chi
\bigl(MF_{\ti f'}(0,0,[\ep_{12}],0,\ep_{21})\bigr)\,\rd\chi\\
&+\bigl(1-\dim\Ext^1(E_1,E_2)\bigr)
\chi\bigl(MF_{f\vert_{V'}}(0,0,0,\ep_{21})\bigr).
\end{split}
\label{dt10eq23}
\e

Let $L_{12}\ra\mathbb{P}(\Ext^1(E_1,E_2))$ and $L_{21}\ra\mathbb{P}
(\Ext^1(E_2,E_1))$ be the tautological line bundles, so that the
fibre of $L_{12}$ over a point $[\ep_{12}]$ in $\mathbb{P}
(\Ext^1(E_1,E_2))$ is the 1-dimensional subspace $\{\la\,\ep_{12}:
\la\in\C\}$ in $\Ext^1(E_1,E_2)$. Consider the line bundle
$L_{12}\ot L_{21}\ra\mathbb{P}(\Ext^1(E_1,E_2))\times \mathbb{P}
(\Ext^1(E_2,E_1))$. The fibre of $L_{12}\ot L_{21}$ over
$([\ep_{12}],[\ep_{21}])$ is $\{\la\,\ep_{12}\ot\ep_{21}:
\la\in\C\}$.

Write points of the total space of $L_{12}\ot L_{21}$ as
$\bigl([\ep_{12}],[\ep_{21}],\la\,\ep_{12}\ot\ep_{21}\bigr)$. Define
$W\subseteq\Ext^1(E_1,E_1)\times \Ext^1(E_2,E_2)\times (L_{12}\ot
L_{21})$ to be the open subset of points
$\bigl(\ep_{11},\ep_{22},[\ep_{12}],
[\ep_{21}],\la\,\ep_{12}\ot\ep_{21}\bigr)$ for which
$(\ep_{21},\ep_{22},\la\,\ep_{12},\ep_{21})$ lies in $U$. Since $U$
is $T^{\sst\C}$-invariant, this definition is independent of the
choice of representatives $\ep_{12},\ep_{21}$ for $[\ep_{12}],
[\ep_{21}]$, since any other choice would replace
$(\ep_{11},\ep_{22},\la\,\ep_{12},\ep_{21})$ by $(\ep_{11},\ep_{22},
\la\mu\,\ep_{12},\mu^{-1}\ep_{21})$ for some $\mu\in\bG_m$. Define a
holomorphic function $h:W\ra\C$ by $h\bigl(\ep_{11},\ep_{22},
[\ep_{12}],[\ep_{21}],\la\,\ep_{12}\ot\ep_{21}\bigr)=f(\ep_{11},
\ep_{22},\la\,\ep_{12},\ep_{21})$. As $f$ is
$T^{\sst\C}$-invariant, the same argument shows $h$ is well-defined.

Define a projection $\Pi:\ti U'\ra W$ by $\Pi:(\ep_{11},\ep_{22},
[\ep_{12}],\la\ep_{12},\ep_{21})\mapsto(\ep_{11},\ep_{22},\ab
[\ep_{12}],\ab[\ep_{21}],\la\ep_{12}\ot\ep_{21})$. Then $\Pi$ is a
smooth holomorphic submersion, with fibre $\bG_m$. Furthermore, we
have $\ti f'\equiv h\ci\Pi$. It follows that the Milnor fibre of
$\ti f'$ at $(\ep_{11},\ep_{22}, [\ep_{12}],\la\ep_{12},\ep_{21})$
is the product of the Milnor fibre of $h$ at $(\ep_{11},\ep_{22},
[\ep_{12}],[\ep_{21}],\ab\la\ep_{12}\ot\ep_{21})$ with a small ball
in $\C$, so they have the same Euler characteristic. That is,
\e
\chi \bigl(MF_{\ti f'}(0,0,[\ep_{12}],0,\ep_{21})\bigr)=\chi
\bigl(MF_h(0,0,[\ep_{12}],[\ep_{21}],0)\bigr).
\label{dt10eq24}
\e
Also, we have $f(\ep_{11},\ep_{22},0,\ep_{21})=f(\ep_{11},
\ep_{22},0,0)$ as in the proof of Proposition \ref{dt10prop2}, so the
Milnor fibre of $f\vert_{V'}$ at $(0,0,0,\ep_{21})$ is the product
of the Milnor fibre of $f\vert_{\Ext^1(E,E)^T}$ at 0 with a small
ball in $\Ext^1(E_2,E_1)$, and they have the same Euler
characteristic. That is,
\e
\chi\bigl(MF_{f\vert_{V'}}(0,0,0,\ep_{21})\bigr) =
\chi\bigl(MF_{f\vert_{\Ext^1(E,E)^T}}(0)\bigr).
\label{dt10eq25}
\e
Substituting \eq{dt10eq24} and \eq{dt10eq25} into \eq{dt10eq23}
gives
\begin{align*}
\chi\bigl(MF_f(0,0,0,&\ep_{21})\bigr)=\int_{[\ep_{12}]\in\mathbb{P}
(\Ext^1(E_1,E_2))}\chi\bigl(MF_h(0,0,[\ep_{12}],[\ep_{21}],0)
\bigr)\,\rd\chi\\
&+\bigl(1-\dim\Ext^1(E_1,E_2)\bigr)\chi\bigl(MF_{f\vert_{\Ext^1
(E,E)^T}}(0)\bigr).
\end{align*}

Integrating this over $[\ep_{21}]\in\mathbb{P}(\Ext^1(E_2,E_1))$
yields
\e
\begin{split}
\int_{[\ep_{21}]\in\mathbb{P}(\Ext^1(E_2,E_1))} \!\!\!\!\!\!\!\!\!\!
\!\!\!\!\!\!\!\!\!\!\!\!\!\!\!\!\!\!\!\!\!\!\!\!\!\!\!\!\!\!
\chi\bigl(MF_f(0,0,0,\ep_{21})\bigr)\,\rd\chi=
\int_{([\ep_{12}],[\ep_{21}])\in\mathbb{P}(\Ext^1(E_1,E_2))\times
\mathbb{P}(\Ext^1(E_2,E_1))}\!\!\!\!
\!\!\!\!\!\!\!\!\!\!\!\!\!\!\!\!\!\!\!\!\!\!\!\!\!\!\!\!\!\!
\!\!\!\!\!\!\!\!\!\!\!\!\!\!\!\!\!\!\!\!\!\!\!\!\!\!\!\!\!\!
\!\!\!\!\!\!\!\!\!\!\! \chi\bigl(MF_h(0,0,[\ep_{12}],[\ep_{21}],0)
\bigr)\,\rd\chi&\\
+\bigl(1-\dim\Ext^1(E_1,E_2)\bigr) \dim\Ext^1(E_2,E_1)\cdot
\chi\bigl(MF_{f\vert_{\Ext^1(E,E)^T}}(0)\bigr)&,
\end{split}
\label{dt10eq26}
\e
since $\chi\bigl(\mathbb{P}(\Ext^1(E_2,E_1))\bigr)=\dim
\Ext^1(E_2,E_1)$. Similarly we have
\e
\begin{split}
\int_{[\ep_{12}]\in\mathbb{P}(\Ext^1(E_1,E_2))} \!\!\!\!\!\!\!\!\!\!
\!\!\!\!\!\!\!\!\!\!\!\!\!\!\!\!\!\!\!\!\!\!\!\!\!\!\!\!\!\!
\chi\bigl(MF_f(0,0,\ep_{12},0)\bigr)\,\rd\chi=
\int_{([\ep_{12}],[\ep_{21}])\in\mathbb{P}(\Ext^1(E_1,E_2))\times
\mathbb{P}(\Ext^1(E_2,E_1))}\!\!\!\!
\!\!\!\!\!\!\!\!\!\!\!\!\!\!\!\!\!\!\!\!\!\!\!\!\!\!\!\!\!\!
\!\!\!\!\!\!\!\!\!\!\!\!\!\!\!\!\!\!\!\!\!\!\!\!\!\!\!\!\!\!
\!\!\!\!\!\!\!\!\!\!\! \chi\bigl(MF_h(0,0,[\ep_{12}],[\ep_{21}],0)
\bigr)\,\rd\chi&\\
+\bigl(1-\dim\Ext^1(E_2,E_1)\bigr) \dim\Ext^1(E_1,E_2)\cdot
\chi\bigl(MF_{f\vert_{\Ext^1(E,E)^T}}(0)\bigr)&.
\end{split}
\label{dt10eq27}
\e
Equation \eq{dt10eq22} now follows from \eq{dt10eq27} minus
\eq{dt10eq26}. This completes the proof of~\eq{dt5eq3}.\index{Behrend
function!identities|)}\index{Milnor fibre|)}

\section[The proof of Theorem $\text{\ref{dt5thm5}}$]{The proof of
Theorem \ref{dt5thm5}}
\label{dt11}

We use the notation of \S\ref{dt2}--\S\ref{dt4} and Theorem
\ref{dt5thm5}. It is sufficient to prove that $\ti\Psi^{\chi,\Q}$ is
a Lie algebra morphism, as $\ti\Psi=\ti\Psi^{\chi,\Q}\ci
\bar\Pi^{\chi,\Q}_\fM$ and $\bar\Pi^{\chi,\Q}_\fM:\SFai(\fM)\ra
\oSFai(\fM,\chi,\Q)$ is a Lie algebra morphism as in \S\ref{dt31}.
The rough idea is to insert Behrend functions\index{Behrend function}
$\nu_\fM$ as weights in the proof of Theorem \ref{dt3thm4} in
\cite[\S 6.4]{Joyc4}, and use the identities\index{Behrend
function!identities} \eq{dt5eq2}--\eq{dt5eq3}. However, \cite[\S
6.4]{Joyc4} involved lifting from Euler characteristics to virtual
Poincar\'e polynomials;\index{virtual Poincar\'e polynomial} here we
give an alternative proof involving only Euler characteristics, and
also change some methods in the proof.

We must show $\smash{\ti\Psi^{\chi,\Q}\bigl([f,g]\bigr)=
\bigl[\ti\Psi^{\chi,\Q}(f),\ti\Psi^{\chi,\Q}(g)\bigr]}$ for
$f,g\in\oSFai (\fM,\chi,\Q)$. It is enough to do this for $f,g$
supported on $\smash{\fM^\al,\fM^\be}$ respectively, for $\al,\be\in
C(\coh(X))\cup\{0\}$. Choose finite type, open $\C$-substacks $\fU$
in $\fM^\al$ and $\fV$ in $\fM^\be$ such that $f,g$ are supported on
$\fU,\fV$. This is possible as $f,g$ are supported on constructible
sets and $\fM^\al,\fM^\be$ are locally of finite type. As $\fU,\fV$
are of finite type the families of sheaves they parametrize are
bounded, so by Serre vanishing\index{Serre vanishing}
\cite[Lem.~1.7.6]{HuLe2} we can choose $n\gg 0$ such that for all
$[E_1]\in\fU(\C)$ and $[E_2]\in\fV(\C)$ we have $H^i(E_j(n))=0$ for
all $i>0$ and $j=1,2$. Hence $\dim
H^0(E_1(n))=\bar\chi\bigl([\cO_X(-n)],\al\bigr)=P_\al(n)$ and $\dim
H^0(E_2(n))=\bar\chi\bigl([\cO_X(-n)],\be\bigr)=P_\be(n)$, where
$P_\al,P_\be$ are the Hilbert polynomials of~$\al,\be$.

Consider Grothendieck's Quot Scheme\index{Quot scheme}
$\Quot_X\bigl( U\ot\cO_X(-n),P_\al\bigr)$, explained in \cite[\S
2.2]{HuLe2}, which parametrizes quotients
$U\ot\cO_X(-n)\twoheadrightarrow E$ of the fixed coherent sheaf
$U\ot\cO_X(-n)$ over $X$, such that $E$ has fixed Hilbert polynomial
$P_\al$. By \cite[Th.~2.2.4]{HuLe2},
$\Quot_X\bigl(U\ot\cO_X(-n),P_\al\bigr)$ is a projective $\C$-scheme
representing the moduli functor $\underline{\Quot}{}_X\bigl(
U\ot\cO_X(-n),P_\al\bigr)$ of such quotients.

Define $Q_{\fU,n}$ to be the subscheme of $\Quot_X\bigl(
U\ot\cO_X(-n),P_\al\bigr)$ representing quotients
$U\ot\cO_X(-n)\twoheadrightarrow E_1$ such that $[E_1]\in\fU(\C)$,
and the morphism $U\ot\cO_X(-n)\twoheadrightarrow E_1$ is induced by
an isomorphism $\phi:U\ra H^0(E_1(n))$, noting that $[E_1]
\in\fU(\C)$ implies that $\dim H^0(E_1(n))=P_\al(n)=\dim U$. This is
an open condition on $U\ot\cO_X(-n)\twoheadrightarrow E_1$, as $\fU$
is open in $\fM^\al$, so $Q_{\fU,n}$ is open in $\Quot_X\bigl(
U\ot\cO_X(-n),P_\al\bigr)$, and is a quasiprojective $\C$-scheme,
with
\e
\begin{split}
Q_{\fU,n}(\C)\cong\bigl\{&\text{isomorphism classes
$[(E_1,\phi_1)]$ of pairs $(E_1,\phi_1)$:}\\
&\text{$[E_1]\in\fU(\C)$, $\phi_1:U\ra H^0(E_1(n))$ is an
isomorphism}\bigr\}.
\end{split}
\label{dt11eq1}
\e
The algebraic $\C$-group $\GL(U)\cong\GL(P_\al(n),\C)$ acts on the
right on $Q_{\fU,n}$, on points as $\ga:[(E_1,\phi_1)]\mapsto
[(E_1,\phi_1\ci\ga)]$ in the representation \eq{dt11eq1}. Similarly,
we define an open subscheme $Q_{\fV,n}$ in
$\Quot_X\bigl(V\ot\cO_X(-n),P_\be\bigr)$ with a right action of
$\GL(V)$. In the usual way we have 1-isomorphisms of Artin
$\C$-stacks
\e
\smash{\fU\cong [Q_{\fU,n}/\GL(U)],\qquad
\fV\cong[Q_{\fV,n}/\GL(V)],}
\label{dt11eq2}
\e
which write $\fU,\fV$ as global quotient stacks.

The definition of the Ringel--Hall\index{Ringel--Hall algebra}
multiplication $*$ on $\SFa(\fM)$ in \S\ref{dt31} involves the
moduli stack $\fExact$ of short exact sequences $0\ra E_1\ra F\ra
E_2\ra 0$ in $\coh(X)$, and 1-morphisms
$\pi_1,\pi_2,\pi_3:\fExact\ra\fM$ mapping $0\ra E_1\ra F\ra E_2\ra
0$ to $E_1,F,E_2$ respectively. Thus we have a 1-morphism
$\pi_1\times\pi_3:\fExact\ra\fM\times\fM$. We wish to describe
$\fExact$ and $\pi_1\times\pi_3$ over $\fU\times\fV$ in
$\fM\times\fM$. Suppose $[0\ra E_1\ra F\ra E_2\ra 0]$ is a point in
$\fExact(\C)$ which is mapped to $(\fU\times\fV)(\C)$ by
$\pi_1\times\pi_3$. Then $[E_1]\in\fU(\C)$ and $[E_2]\in\fV(\C)$, so
$E_1,E_2$ have Hilbert polynomials $P_\al,P_\be$, and thus $F$ has
Hilbert polynomial\index{Hilbert polynomial} $P_{\al+\be}$. Also
$H^i(E_j(n))=0$ for all $i>0$ and $j=1,2$ and $\dim
H^0(E_1(n))=P_\al(n)$, $\dim H^0(E_2(n))=P_\be(n)$. Applying
$\Hom(\cO_X(-n),*)$ to $0\ra E_1\ra F\ra E_2\ra 0$ shows that
\begin{equation*}
\xymatrix{0 \ar[r] & H^0(E_1(n)) \ar[r] & H^0(F(n)) \ar[r] &
H^0(E_2(n)) \ar[r] & 0}
\end{equation*}
is exact, so that $\dim H^0(F(n))=P_{\al+\be}(n)$, and $H^i(F(n))=0$
for~$i>0$.

By a similar argument to the construction of the Quot
scheme\index{Quot scheme} in \cite[\S 2.2]{HuLe2}, one can construct
a `Quot scheme for exact sequences' $0\ra E_1\ra F\ra E_2\ra 0$,
which are quotients of the natural split short exact sequence of
coherent sheaves $0\ra U\ot\cO_X(-n)\ra (U\op V)\ot\cO_X(-n)\ra
V\ot\cO_X(-n)\ra 0$. There is an open subscheme $Q_{\fU,\fV,n}$ of
this Quot scheme for exact sequences such that, in a similar way to
\eq{dt11eq1}, there is a natural identification between
$Q_{\fU,\fV,n}(\C)$ and the set of isomorphism classes of data
$(0\ra E_1\ra F\ra E_2\ra 0,\phi_1,\phi,\phi_2)$ where $\phi_1:U\ra
H^0(E_1(n))$, $\phi:U\op V\ra H^0(F(n))$ and $\phi_2:V\ra
H^0(E_2(n))$ are isomorphisms, and the following diagram commutes:
\begin{equation*}
\xymatrix@R=10pt{ 0 \ar[r] & U \ar[r] \ar[d]^{\phi_1}_\cong & U\op V
\ar[r] \ar[d]^{\phi}_\cong & V \ar[r] \ar[d]^{\phi_2}_\cong & 0 \\
0 \ar[r] & H^0(E_1(n)) \ar[r] & H^0(F(n)) \ar[r] & H^0(E_2(n))
\ar[r] & 0.}
\end{equation*}

The automorphism group of the sequence $0\ra U\ra U\op V\ra V\ra 0$
is the algebraic $\C$-group $(\GL(U)\times\GL(V))\lt\Hom(V,U)$, with
multiplication
\begin{equation*}
(\ga,\de,\ep)\cdot(\ga',\de',\ep')=(\ga\ci\ga',\de\ci\de',
\ga\ci\ep'+\ep\ci\de')
\end{equation*}
for $\ga,\ga'\in\GL(U)$, $\de,\de'\in\GL(V)$,
$\ep,\ep'\in\Hom(V,U)$. It is the subgroup of elements
$\bigl(\begin{smallmatrix} \ga & \ep \\ 0 & \de
\end{smallmatrix}\bigr)$ in $\GL(U\op V)$. Then $(\GL(U)\times\GL(V))\lt
\Hom(V,U)$ acts naturally on the right on $Q_{\fU,\fV,n}$. On
points in the representation above it acts by
\begin{align*}
(\ga,\de,\ep):(0&\ra E_1\ra F\ra E_2\ra
0,\phi_1,\phi,\phi_2)\longmapsto\\
&\bigl(0\ra E_1\ra F\ra E_2\ra
0,\phi_1\ci\ga,\phi\ci\bigl(\begin{smallmatrix} \ga &
\ep \\ 0 & \de \end{smallmatrix}\bigr),\phi_2\ci\de\bigr).
\end{align*}
As for \eq{dt11eq2}, we have a 1-isomorphism
\e
\!(\fU\times\fV)_{\io_\fU\times
\io_\fV,\fM\times\fM,\pi_1\times\pi_3}\fExact\!
\cong\!\bigl[Q_{\fU,\fV,n}/(\GL(U)\!\times\!\GL(V))\lt\Hom(V,U)\bigr],\!\!
\label{dt11eq3}
\e
where $\io_\fU:\fU\ra\fM$, $\io_\fV:\fV\ra\fM$ are the inclusions,
and the l.h.s.\ of \eq{dt11eq3} is the open $\C$-substack of
$\fExact$ taken to $\fU\times\fV$ in $\fM\times\fM$
by~$\pi_1\times\pi_3$.

There are projections $\Pi_\fU:Q_{\fU,\fV,n}\ra Q_{\fU,n}$,
$\Pi_\fV:Q_{\fU,\fV,n}\ra Q_{\fV,n}$ acting by
\begin{equation*}
\Pi_\fU,\Pi_\fV:\bigl[(0\ra E_1\ra F\ra E_2\ra
0,\phi_1,\phi,\phi_2)\bigr]\longmapsto\bigl[(E_1,\phi_1)\bigr],
\bigl[(E_2,\phi_2)\bigr].
\end{equation*}
Combining $\Pi_\fU,\Pi_\fV$ with the natural projections of
algebraic $\C$-groups
$(\GL(U)\times\GL(V))\lt\Hom(V,U)\ra\GL(U),\GL(V)$ gives 1-morphisms
\e
\begin{split}
\Pi_\fU':\bigl[Q_{\fU,\fV,n}/(\GL(U)\times\GL(V))\lt\Hom(V,U)\bigr]
&\longra[Q_{\fU,n}/\GL(U)],\\
\Pi_\fV':\bigl[Q_{\fU,\fV,n}/(\GL(U)\times\GL(V))\lt\Hom(V,U)\bigr]
&\longra[Q_{\fV,n}/\GL(V)],
\end{split}
\label{dt11eq4}
\e
which are 2-isomorphic to $\pi_1,\pi_3$ under the 1-isomorphisms
\eq{dt11eq2}, \eq{dt11eq3}. There is a morphism $z:Q_{\fU,n}\times
Q_{\fV,n}\ra Q_{\fU,\fV,n}$ which embeds $Q_{\fU,n}\times Q_{\fV,n}$
as a closed subscheme of $Q_{\fU,\fV,n}$, given on points by
\begin{equation*}
z:\bigl([(E_1,\phi_1)],[(E_2,\phi_2)]\bigr)\mapsto \bigl[(0\!\ra\!
E_1\!\ra\! E_1\!\op E_2\!\ra\!E_2\!\ra\!
0,\phi_1,\phi_1\op\phi_2,\phi_2)\bigr].
\end{equation*}
Write $Q_{\fU,\fV,n}'=Q_{\fU,\fV,n}\sm z(Q_{\fU,n}\times
Q_{\fV,n})$, an open subscheme of $Q_{\fU,\fV,n}$.

Let $q_1\in Q_{\fU,n}(\C)$ correspond to $[(E_1,\phi_1)]$ under
\eq{dt11eq1}, and $q_2\in Q_{\fV,n}(\C)$ correspond to
$[(E_2,\phi_2)]$. Then the fibre $(\Pi_\fU\times\Pi_\fV)^*(q_1,q_2)$
of $\Pi_\fU\times\Pi_\fV$ over $(q_1,q_2)$ is a subscheme of
$Q_{\fU,\fV,n}$ of points $\bigl[(0\ra E_1\ra F\ra E_2\ra
0,\phi_1,\phi,\phi_2) \bigr]$ with $E_1,\phi_1,E_2,\phi_2$ fixed. By
the usual correspondence between extensions and vector spaces
$\Ext^1(\,,\,)$ we find $(\Pi_\fU\times\Pi_\fV)^*(q_1,q_2)$ is a
$\C$-vector space, which we write as $W_{\fU,\fV,n}^{q_1,q_2}$,
where $0\in W_{\fU,\fV,n}^{q_1,q_2}$ is $z(q_1,q_2)$. The subgroup
$\Hom(V,U)$ of $(\GL(U)\times\GL(V))\lt \Hom(V,U)$ acts on
$(\Pi_\fU\times\Pi_\fV)^*(q_1,q_2)\cong W_{\fU,\fV,n}^{q_1,q_2}$ by
translations. Write this action as a linear map
$L_{\fU,\fV,n}^{q_1,q_2}:\Hom(V,U)\ra W_{\fU,\fV,n}^{q_1,q_2}$.

We claim this fits into an exact sequence
\e
\smash{\xymatrix@C=5.5pt{ \!\!0 \ar[r] & \Hom(E_2,E_1) \ar[r] &
\Hom(V,U) \ar[rrr]^(0.53){L_{\fU,\fV,n}^{q_1,q_2}} &&&
W_{\fU,\fV,n}^{q_1,q_2} \ar[rrr]^(0.4){\pi_{E_2,E_1}} &&&
\Ext^1(E_2,E_1) \ar[r] & 0.\!\!}}
\label{dt11eq5}
\e
To see this, note that the fibre of $\Pi_\fU'\times\Pi_\fV'$ over
$(q_1,q_2)$ is the quotient stack $[W_{\fU,\fV,n}^{q_1,q_2}/\!
\Hom(V,U)]$, where $\Hom(V,U)$ acts on $W_{\fU,\fV,n}^{q_1,q_2}$ by
$\ep:w\!\mapsto\! w\!+\!\ab L_{\fU,\fV,n}^{q_1,q_2}(\ep)$, whereas
the fibre of $\pi_1\times\pi_3:\fExact\ra\fM\times\fM$ over
$(E_1,E_2)$ is the quotient stack
$\smash{[\Ext^1(E_2,E_1)/\Hom(E_2,E_1)]}$, where $\Hom(E_2,E_1)$
acts trivially on $\Ext^1(E_2,E_1)$. The 1-isomorphisms \eq{dt11eq2}
and \eq{dt11eq3} induce a 1-isomorphism
$[W_{\fU,\fV,n}^{q_1,q_2}/\Hom(V,U)]\ab\cong[\Ext^1(E_2,E_1)/\Hom(E_2,
E_1)]$, which gives~\eq{dt11eq5}.

We can repeat all the above material on $Q_{\fU,\fV,n}$ with
$\fU,\fV$ exchanged. We use the corresponding notation with accents
`${}\,\,{}\ti{}\,\,{}$'. We obtain a quasiprojective $\C$-scheme
$\ti Q_{\fV,\fU,n}$ whose $\C$-points are isomorphism classes of
data $(0\ra E_2\ra\ti F\ra E_1\ra 0,\phi_2,\ti\phi,\phi_1)$ where
$[E_2]\in\fV(\C)$, $[E_1]\in\fU(\C)$, $\phi_2:V\ra H^0(E_2(n))$,
$\ti\phi:V\op U\ra H^0(\ti F(n))$ and $\phi_1:U\ra H^0(E_1(n))$ are
isomorphisms, and the following diagram commutes:
\begin{gather*}
\xymatrix@R=10pt{ 0 \ar[r] & V \ar[r] \ar[d]^{\phi_2}_\cong & V\op U
\ar[r] \ar[d]^{\ti\phi}_\cong & U \ar[r] \ar[d]^{\phi_1}_\cong & 0 \\
0 \ar[r] & H^0(E_2(n)) \ar[r] & H^0(\ti F(n)) \ar[r] & H^0(E_1(n))
\ar[r] & 0.}
\end{gather*}
There is a closed embedding $\ti z:Q_{\fV,n}\times Q_{\fU,n}\ra\ti
Q_{\fV,\fU,n}$, and we write $\ti Q_{\fV,\fU,n}'=\ti
Q_{\fV,\fU,n}\sm\ti z(Q_{\fV,n}\times Q_{\fU,n})$.

The algebraic $\C$-group $(\GL(V)\times\GL(U))\lt\Hom(U,V)$ acts on
$\ti Q_{\fV,\fU,n}$ with
\begin{equation*}
(\fV\!\times\!\fU)_{\io_\fV\times \io_\fU,\fM\times\fM,\pi_1\times\pi_3}\fExact\!
\cong\!\bigl[\ti Q_{\fV,\fU,n}/(\GL(V)\!\times\!\GL(U))\!\lt\!\Hom(U,V)
\bigr].
\end{equation*}
There are natural projections $\ti\Pi_\fV,\ti\Pi_\fU:\ti
Q_{\fU,\fV,n}\ra Q_{\fV,n},Q_{\fU,n}$ and $\ti\Pi_\fV',\ti\Pi_\fU'$
from $[\ti Q_{\fV,\fU,n}/(\GL(V)\times\GL(U))\lt\Hom(U,V)]$ to
$[Q_{\fV,n}/\GL(V)],[Q_{\fU,n}/\GL(U)]$. If $q_1\in Q_{\fU,n}(\C)$
and $q_2\in Q_{\fV,n}(\C)$ correspond to $[(E_1,\phi_1)]$ and
$[(E_2,\phi_2)]$ then $(\ti\Pi_\fV\times\ti\Pi_\fU)^*(q_2,q_1)$ in
$\ti Q_{\fV,\fU,n}$ is a $\C$-vector space $\ti
W_{\fV,\fU,n}^{q_2,q_1}$ with an exact sequence
\e
\smash{\xymatrix@C=5.5pt{ \!\!0 \ar[r] & \Hom(E_1,E_2) \ar[r] &
\Hom(U,V) \ar[rrr]^(0.53){\ti L_{\fV,\fU,n}^{q_2,q_1}} &&& \ti
W_{\fV,\fU,n}^{q_2,q_1} \ar[rrr]^(0.4){\ti\pi_{E_1,E_2}} &&&
\Ext^1(E_1,E_2) \ar[r] & 0.\!\!}}
\label{dt11eq6}
\e

Now consider the stack function\index{stack function}
$f\in\oSFai(\fM,\chi,\Q)$. Since $f$ is supported on $\fU$, by
Proposition \ref{dt3prop1} we may write $f$ in the form
\e
f=\ts\sum_{i=1}^n\de_i[(Z_i\times[\Spec\C/\bG_m],\io_\fU\ci\rho_i)],
\label{dt11eq7}
\e
where $\de_i\in\Q$, $Z_i$ is a quasiprojective $\C$-variety, and
$\rho_i:Z_i\times[\Spec\C/\bG_m]\ra\fU$ is representable for
$i=1,\ldots,n$, and $\io_\fU:\fU\ra\fM$ is the inclusion, and each
term in \eq{dt11eq7} has algebra stabilizers.\index{stack
function!with algebra stabilizers} Consider the fibre product
$P_i=Z_i\times_{\rho_i,\fU,\pi_\fU}Q_{\fU,n}$, where
$\pi_\fU:Q_{\fU,n}\ra\fU$ is the projection induced by \eq{dt11eq2}.
As $\pi_\fU$ is a principal $\GL(U)$-bundle of Artin $\C$-stacks,
$\pi_1:P_i\ra Z_i$ is a principal $\GL(U)$-bundle of $\C$-schemes,
and so is Zariski locally trivial as $\GL(U)$ is special. Thus by
cutting the $Z_i$ into smaller pieces using relation Definition
\ref{dt2def10}(i), we can suppose the fibrations $\pi_1:P_i\ra Z_i$
are trivial, with trivializations $P_i\cong Z_i\times\GL(U)$.
Composing the morphisms $Z_i\hookra Z_i\times\{1\}\subset
P_i\,{\buildrel \pi_2\over\longra}\,Q_{\fU,n}$ gives a morphism
$\xi_i:Z_i\ra Q_{\fU,n}$.

The algebra stabilizers condition implies that if $z\in Z_i(\C)$ and
$(\io_\fU\ci\rho_i)_*(z)$ is a point $[E]\in\fM(\C)$ then on
stabilizer groups $(\io_\fU\ci\rho_i)_*:\bG_m\ra\Aut(E)$ must map
$\la\mapsto\la\id_E$. If $q\in Q_{\fU,n}(\C)$ with
$(\pi_\fU)_*(q)=[E]$ then $(\pi_\fU)_*:\Stab_{\GL(U)}(q)\ra\Aut(E)$
is an isomorphism, and from the construction it follows that
$(\pi_\fU)_*$ maps $\la\id_U\ra \la\id_E$ for $\la\in\bG_m$. Hence
the 1-morphism $\rho_i:Z_i\times[\Spec\C/\bG_m]\ra[Q_{\fU,n}/
\GL(U)]\cong\fU$ acts on stabilizer groups as $(\rho_i)_*:\la\mapsto
\la\id_U$ for $\la\in\bG_m$, for all $z\in Z_i(\C)$. It is now easy
to see that the 1-morphism $\rho_i: Z_i\times[\Spec\C/\bG_m]\ra\fU$,
regarded as a morphism of global quotient stacks
$\rho_i:[Z_i/\bG_m]\ra[Q_{\fU,n}/\GL(U)]$ where $\bG_m$ acts
trivially on $Z_i$, is induced by the morphisms $\xi_i:Z_i\ra
Q_{\fU,n}$ of $\C$-schemes and $I_U:\bG_m\ra\GL(U)$ of algebraic
$\C$-groups mapping~$I_U:\la\mapsto\la\id_U$.

Thus we may write $f$ in the form \eq{dt11eq7}, where each $Z_i$ is
a quasiprojective $\C$-variety and each $\rho_i:Z_i\times
[\Spec\C/\bG_m]\ra[Q_{\fU,n}/\GL(U)]\cong \fU$ is induced by
$\xi_i:Z_i\ra Q_{\fU,n}$ and $I_U:\bG_m\ra\GL(U)$,
$I_U:\la\mapsto\la\id_U$. Similarly, we may write
\e
g=\ts\sum_{j=1}^{\hat n}\hat\de_j[(\hat
Z_j\times[\Spec\C/\bG_m],\io_\fV\ci\hat\rho_j)],
\label{dt11eq8}
\e
where $\hat Z_j$ is quasiprojective and $\hat\rho_j:\hat
Z_j\times[\Spec\C/\bG_m]\ra[Q_{\fV,n}/\GL(V)]\cong\fV$ is induced by
$\hat\xi_j:\hat Z_j\ra Q_{\fV,n}$ and $I_V:\bG_m\ra\GL(V)$,
$I_V:\la\mapsto\la\id_V$.

Combining \eq{dt11eq7}--\eq{dt11eq8} gives an expression for $f\ot
g$ in $\oSF(\fM\times\fM,\chi,\Q)$:
\e
f\ot g=\ts\sum_{i=1}^n\sum_{j=1}^{\hat
n}\de_i\hat\de_j\bigl[\bigl(Z_i\times\hat
Z_j\times[\Spec\C/\bG_m^2],(\io_\fU\times \io_\fV)\ci
(\rho_i\times\hat\rho_j)\bigr)\bigr].
\label{dt11eq9}
\e
Using the 1-isomorphisms \eq{dt11eq2}, \eq{dt11eq3} and the
correspondence between the 1-morphisms $\pi_1,\pi_3$ and
$\Pi_\fU',\Pi_\fV'$ in \eq{dt11eq4} and $\Pi_\fU,\Pi_\fV$, we obtain
1-isomorphisms
\ea
\bigl(Z_i&\times\hat Z_j\times[\Spec\C/\bG_m^2]\bigr)\times_{
(\io_\fU\times \io_\fV)\ci
(\rho_i\times\hat\rho_j),\fM\times\fM,\pi_1\times\pi_3}\fExact
\nonumber\\
&\cong \bigl(Z_i\!\times\!\hat
Z_j\!\times\![\Spec\C/\bG_m^2]\bigr)\! \times_{\begin{subarray}{l}
{} \\ {} \\ \rho_i\times\hat\rho_j,
[Q_{\fU,n}/\GL(U)]\times[Q_{\fV,n}/\GL(V)],\Pi_\fU'\times\Pi_\fV'
\end{subarray}
\!\!\!\!\!\!\!\!\!\!\!\!\!\!\!\!\!\!\!\!\!\!\!\!\!\!\!\!\!\!\!
\!\!\!\!\!\!\!\!\!\!\!\!\!\!\!\!\!\!\!\!\!\!\!\!\!\!\!\!\!\!\!
\!\!\!\!\!\!\!\!\!\!\!\!\!\!\!\!\!\!\!\!\!\!\!\!\!\!\!\!\!\!\!
\!\!\!\!\!\!\!\!\!\!\!\!\!\!\!\!\!\!\!\!
}\bigl[Q_{\fU,\fV,n}/(\GL(U)\!\times\!\GL(V))\!\lt\!\Hom(V,U) \bigr]
\nonumber\\
&\cong\! \bigl[\bigl((Z_i\!\times\!\hat
Z_j)\times_{\xi_i\times\hat\xi_j, Q_{\fU,n}\times
Q_{\fV,n},\Pi_\fU\times\Pi_\fV}Q_{\fU,\fV,n}\bigr)/\bG_m^2
\!\lt\!\Hom(V,U)\bigr],
\label{dt11eq10}
\ea
where in the last line, the multiplication in $\bG_m^2\lt\Hom(V,U)$
is $(\la,\mu,\ep)\cdot(\la',\mu',\ep')\ab=(\la\la',\mu\mu',\la\ep'
+\mu'\ep)$ for $\la,\la',\mu,\mu'\in\bG_m$ and
$\ep,\ep'\in\Hom(V,U)$, and $\bG_m^2\lt\Hom(V,U)$ acts on
$(Z_i\times\hat Z_j)\times_{\cdots}Q_{\fU,\fV,n}$ by the composition
of the morphism
$\bG_m^2\lt\Hom(V,U)\ra(\GL(U)\times\GL(V))\lt\Hom(V,U)$ mapping
$(\la,\mu,\ep)\mapsto(\la\id_U,\ab\mu\id_V,\ep)$ and the action of
$(\GL(U)\times\GL(V))\lt\Hom(V,U)$ on $Q_{\fU,\fV,n}$, with the
trivial action on~$Z_i\times\hat Z_j$.

Now $f*g=(\pi_2)_*\bigl((\pi_1\times\pi_3)^*(f\ot g)\bigr)$ by
\eq{dt3eq3}. Applying $(\pi_1\times\pi_3)^*$ to each term in
\eq{dt11eq9} involves the fibre product in the first line of
\eq{dt11eq10}. So from \eq{dt3eq3}, \eq{dt11eq9} and \eq{dt11eq10}
we see that
\e
f*g\!=\!\sum\limits_{i=1}^n\sum\limits_{j=1}^{\hat
n}\de_i\hat\de_j\bigl[\bigl(\bigl[(Z_i\!\times\!\hat
Z_j)\times_{Q_{\fU,n}\!\times\! Q_{\fV,n}}\!Q_{\fU,\fV,n}/
\bG_m^2\!\lt\!\Hom(V,U)\bigr], \psi_{ij}\bigr)\bigr],
\label{dt11eq11}
\e
for 1-morphisms $\psi_{ij}:\bigl[(Z_i\!\times\!\hat
Z_j)\times_{\cdots}
Q_{\fU,\fV,n}/\bG_m^2\!\lt\!\Hom(V,U)\bigr]\!\ra\!\fM^{\al+\be}$.
Similarly
\e
g*f\!=\!\sum\limits_{i=1}^n\sum\limits_{j=1}^{\hat
n}\de_i\hat\de_j\bigl[\bigl(\bigl[(\hat Z_j\!\times\!
Z_i)\times_{Q_{\fV,n}\!\times\! Q_{\fU,n}}\!\ti Q_{\fV,\fU,n}/
\bG_m^2\!\lt\!\Hom(U,V)\bigr],\ti\psi_{ji}\bigr)\bigr].
\label{dt11eq12}
\e

Next we use relations Definition \ref{dt2def10}(i)--(iii) in
$\oSF(\fM,\chi,\Q)$ to write \eq{dt11eq11}--\eq{dt11eq12} in a more
useful form. When $G=\bG_m^2\lt\Hom(V,U)$ and
$T^G=\bG_m^2\times\{0\}\subset\bG_m^2\lt\Hom(V,U)$, we find that
$\cQ(G,T^G)=\bigl\{T^G,\{(\la,\la):\la\in\bG_m\}\bigr\}=
\{\bG_m^2,\bG_m\}$, in the notation of Definition \ref{dt2def9}. We
need to compute the coefficients $F(G,T^G,Q)$ in Definition
\ref{dt2def10}(iii) for $Q=\bG_m^2,\bG_m$. Let $X$ be the
homogeneous space $\bG_m^2\backslash G\cong \Hom(V,U)$, considered
as a $\C$-scheme, with a right action of $G$. Then $\bG_m^2\subset
G$ acts on $X\cong\Hom(V,U)$ by $(\la,\mu):\ep\mapsto\la\mu^{-1}\ep$
and $\bG_m\subset G$ acts trivially on $X$. Then in
$\uoSF(\Spec\C,\chi,\Q)$ we have
\ea
&\bigl[[\Spec\C/\bG_m^2]\bigr]\!=\!\bigl[[X/G]\bigr]
\!=\!F(G,T^G,\bG_m^2)\bigl[[X/\bG_m^2]\bigr]\!+\!
F(G,T^G,\bG_m)\bigl[[X/\bG_m]\bigr]
\nonumber\\
&\;=F(G,T^G,\bG_m^2)\bigl(\bigl[[\Spec\C/\bG_m^2]\bigr]
\!+\!\bigl[P(\Hom(V,U))\!\times\![\Spec\C/\bG_m]\bigr]\bigr)
\nonumber\\
&\quad+ F(G,T^G,\bG_m)\bigl[\Hom(V,U)\times[\Spec\C/\bG_m]\bigr]
\label{dt11eq13}\\
&\;=F(G,T^G,\bG_m^2)\bigl[[\Spec\C/\bG_m^2]\bigr]
\nonumber\\
&\quad+\bigl(F(G,T^G,\bG_m^2)\dim\Hom(V,U)\!+\!F(G,T^G,\bG_m)\bigr)
\bigl[[\Spec\C/\bG_m]\bigr],
\nonumber
\ea
where in the second step we use Definition \ref{dt2def10}(iii), in
the third we divide $[X/\bG_m^2]$ into
$[(\Hom(V,U)\sm\{0\})/\bG_m^2]\cong \ab P(\Hom(V,U))\times
[\Spec\C/\bG_m]$ and $[\{0\}/\bG_m^2]\ab \cong[\Spec\C/\bG_m^2]$ and
use Definition \ref{dt2def10}(i), and in the fourth we use
Definition \ref{dt2def10}(ii) and
$\chi\bigl(P(\Hom(V,U))\bigr)=\dim\Hom(V,U)$,
$\chi\bigl(\Hom(V,U)\bigr)=1$.

As $\bigl[[\Spec\C/\bG_m]\bigr], \bigl[[\Spec\C/\bG_m^2]\bigr]$ are
independent in $\uoSF(\Spec\C,\chi,\Q)$ by Pro\-position
\ref{dt2prop3}, equating coefficients in \eq{dt11eq13} gives
$F(G,T^G,\bG_m^2)=1$ and $F(G,\ab T^G,\ab\bG_m^2)=-\dim U\dim V$.
Therefore Definition \ref{dt2def10}(iii) gives
\e
\begin{split}
\bigl[\bigl(\bigl[&(Z_i\times\hat Z_j)\times_{Q_{\fU,n}\times
Q_{\fV,n}}Q_{\fU,\fV,n}/\bG_m^2\lt\Hom(V,U)\bigr],
\psi_{ij}\bigr)\bigr]=\\
&\bigl[\bigl(\bigl[(Z_i\times\hat Z_j)\times_{Q_{\fU,n}\times
Q_{\fV,n}}Q_{\fU,\fV,n}/\bG_m^2\bigr],
\psi_{ij}\ci\io^{\bG_m^2}\bigr)\bigr]\\ &-\dim U\dim
V\bigl[\bigl(\bigl[(Z_i\times\hat Z_j)\times_{Q_{\fU,n}\times
Q_{\fV,n}}Q_{\fU,\fV,n}/\bG_m\bigr],
\psi_{ij}\ci\io^{\bG_m}\bigr)\bigr].
\end{split}
\label{dt11eq14}
\e

Split $Q_{\fU,\fV,n}$ into $z(Q_{\fU,n}\times Q_{\fV,n})\cong
Q_{\fU,n}\times Q_{\fV,n}$ and $Q_{\fU,\fV,n}'$. In the second line
of \eq{dt11eq14}, the action of $\bG_m^2$ is trivial on
$Z_i\times\hat Z_j$ and on $z(Q_{\fU,n}\times Q_{\fV,n})$. On
$Q_{\fU,\fV,n}'$, the element $(\la,\mu)$ in $\bG_m^2$ acts by
dilation by $\la\mu^{-1}$ in the fibres
$W_{\fU,\fV,n}^{q_1,q_2}\sm\{0\}$. Thus we can write $\bG_m^2$ as a
product of the diagonal $\bG_m$ factor $\{(\la,\la):\la\in\bG_m\}$
which acts trivially, and a $\bG_m$ factor $\{(\la,1):\la\in\bG_m\}$
which acts freely  on $Q_{\fU,\fV,n}'$. Hence Definition
\ref{dt2def10}(i) gives
\ea
\bigl[\bigl(\bigl[&(Z_i\times\hat Z_j)\times_{Q_{\fU,n}\times
Q_{\fV,n}}Q_{\fU,\fV,n}/\bG_m^2\bigr],
\psi_{ij}\ci\io^{\bG_m^2}\bigr)\bigr]=
\nonumber\\
&\bigl[\bigl(Z_i\times\hat Z_j\times[\Spec\C/\bG_m^2],
\psi_{ij}\ci\io^{\bG_m^2}\ci z\bigr)\bigr]
\label{dt11eq15}\\
&+\bigl[\bigl((Z_i\!\times\!\hat Z_j)\times_{Q_{\fU,n}\times
Q_{\fV,n}}(Q'_{\fU,\fV,n}/\bG_m)
\!\times\![\Spec\C/\bG_m],\psi_{ij}\!\ci\!\io^{\bG_m^2}\bigr)\bigr],
\nonumber\\
\bigl[\bigl(\bigl[&(Z_i\times\hat Z_j)\times_{Q_{\fU,n}\times
Q_{\fV,n}}Q_{\fU,\fV,n}/\bG_m\bigr],
\psi_{ij}\ci\io^{\bG_m}\bigr)\bigr]=
\nonumber\\
&\bigl[\bigl(Z_i\times\hat Z_j\times[\Spec\C/\bG_m],
\psi_{ij}\ci\io^{\bG_m}\ci z\bigr)\bigr]
\label{dt11eq16}\\
&+\bigl[\bigl((Z_i\times\hat Z_j)\times_{Q_{\fU,n}\times
Q_{\fV,n}}Q'_{\fU,\fV,n}
\times[\Spec\C/\bG_m],\psi_{ij}\ci\io^{\bG_m}\bigr)\bigr], \nonumber
\ea
since $(Z_i\times\hat Z_j)\times_{Q_{\fU,n}\times
Q_{\fV,n}}z(Q_{\fU,n}\times Q_{\fV,n}) \cong Z_i\times\hat Z_j$.
Here $Q'_{\fU,\fV,n}/\bG_m$ is a quasiprojective $\C$-variety, with
projection $\Pi_\fU\times\Pi_\fV: Q'_{\fU,\fV,n}\ra Q_{\fU,n}\times
Q_{\fV,n}$ with fibre ${\mathbb P}(W_{\fU,\fV,n}^{q_1,q_2})$ over
$(q_1,q_2)\in (Q_{\fU,n}\times Q_{\fV,n})(\C)$. The action of
$\bG_m$ on $Q'_{\fU,\fV,n}$ is given on points by $\la:\bigl[(0\ra
E_1\ra F\ra E_2\ra 0,\phi_1,\phi,\phi_2)\bigr]\mapsto \bigl[(0\ra
E_1\ra F\ra E_2\ra 0,\la\phi_1,\phi\ci\bigl(\begin{smallmatrix} \la
& 0 \\ 0 & 1
\end{smallmatrix}\bigr),\phi_2)\bigr]$, for~$\la\in\bG_m$.

In the final term in \eq{dt11eq16}, the 1-morphism
$\psi_{ij}\ci\io^{\bG_m}$ factors via the projection
$Q'_{\fU,\fV,n}\ra Q'_{\fU,\fV,n}/\bG_m$, since $\bigl[(0\ra E_1\ra
F\ra E_2\ra 0,\la\phi_1,\phi\ci\bigl(\begin{smallmatrix} \la & 0 \\
0 & 1 \end{smallmatrix}\bigr),\phi_2)\bigr]$ maps to $[F]$ for all
$\la\in\bG_m$. The projection $(Z_i\times\hat
Z_j)\times_{Q_{\fU,n}\times Q_{\fV,n}}Q'_{\fU,\fV,n}\ra
(Z_i\times\hat Z_j)\times_{Q_{\fU,n}\times
Q_{\fV,n}}(Q'_{\fU,\fV,n}/\bG_m)$ is a principal bundle with fibre
$\bG_m$, and so is Zariski locally trivial as $\bG_m$ is
special.\index{algebraic $\K$-group!special} Therefore cutting
$(Z_i\times\hat Z_j)\times_{Q_{\fU,n}\times
Q_{\fV,n}}(Q'_{\fU,\fV,n}/\bG_m)$ into disjoint pieces over which
the fibration is trivial and using relations Definition
\ref{dt2def10}(i),(ii) and $\chi(\bG_m)=0$ shows that
\e
\bigl[\bigl((Z_i\times\hat Z_j)\times_{Q_{\fU,n}\times
Q_{\fV,n}}Q'_{\fU,\fV,n}
\times[\Spec\C/\bG_m],\psi_{ij}\ci\io^{\bG_m}\bigr)\bigr]=0.
\label{dt11eq17}
\e

Combining equations \eq{dt11eq11} and \eq{dt11eq14}--\eq{dt11eq17}
now gives
\ea
&f*g=\ts\sum_{i=1}^n\sum_{j=1}^{\hat n}\de_i\hat\de_j
\bigl[\bigl(Z_i\times\hat Z_j\times[\Spec\C/\bG_m^2],
\psi_{ij}\ci\io^{\bG_m^2}\ci z\bigr)\bigr]
\nonumber\\
\begin{split}
&\qquad+\ts\sum_{i=1}^n\sum_{j=1}^{\hat n}\de_i\hat\de_j
\bigl[\bigl((Z_i\times\hat Z_j)\times_{Q_{\fU,n}\times Q_{\fV,n}}(Q'_{\fU,\fV,n}/\bG_m)\\
&\qquad\qquad\qquad\qquad\qquad\qquad \times[\Spec\C/\bG_m],
\psi_{ij}\ci\io^{\bG_m^2}\bigr)\bigr]
\end{split}
\label{dt11eq18}\\
&-\dim U\dim V\ts\sum_{i=1}^n\sum_{j=1}^{\hat
n}\de_i\hat\de_j\bigl[\bigl(Z_i\!\times\!\hat
Z_j\!\times\![\Spec\C/\bG_m],
\psi_{ij}\!\ci\!\io^{\bG_m}\!\ci\!z\bigr)\bigr]. \nonumber
\ea
Similarly, from equation \eq{dt11eq12} we deduce that
\ea
&g*f=\ts\sum_{i=1}^n\sum_{j=1}^{\hat n}\de_i\hat\de_j
\bigl[\bigl(\hat Z_j\times Z_i\times[\Spec\C/\bG_m^2],
\ti\psi_{ji}\ci\io^{\bG_m^2}\ci\ti z\bigr)\bigr]
\nonumber\\
\begin{split}
&\qquad+\ts\sum_{i=1}^n\sum_{j=1}^{\hat n}\de_i\hat\de_j
\bigl[\bigl((\hat Z_j\times Z_i)\times_{Q_{\fV,n}\times Q_{\fU,n}}(\ti Q'_{\fV,\fU,n}/\bG_m)\\
&\qquad\qquad\qquad\qquad\qquad\qquad \times[\Spec\C/\bG_m],
\ti\psi_{ji}\ci\io^{\bG_m^2}\bigr)\bigr]
\end{split}
\label{dt11eq19}\\
&-\dim U\dim V\ts\sum_{i=1}^n\sum_{j=1}^{\hat
n}\de_i\hat\de_j\bigl[\bigl(\hat Z_j\!\times\!\hat
Z_i\!\times\![\Spec\C/\bG_m],
\ti\psi_{ji}\!\ci\!\io^{\bG_m}\!\ci\!\ti z\bigr)\bigr]. \nonumber
\ea

Subtracting \eq{dt11eq19} from \eq{dt11eq18} gives an expression for
the Lie bracket $[f,g]$. Now the first terms on the right hand sides
of \eq{dt11eq18} and \eq{dt11eq19} are equal, as over points $z_1\in
Z_i(\C)$ and $\hat z_2\in\hat Z_j(\C)$ projecting to
$[E_1]\in\fU(\C)$ and $[E_2]\in\fV(\C)$ they correspond to exact
sequences $[0\ra E_1\ra E_1\op E_2\ra E_2\ra 0]$ and $[0\ra E_2\ra
E_2\op E_1\ra E_1\ra 0]$ respectively, and so project to the same
point $[E_1\op E_2]$ in $\fM$. Similarly, the final terms on the
right hand sides of \eq{dt11eq18} and \eq{dt11eq19} are equal. Hence
\e
\begin{split}
[f,&g]=\ts\sum_{i=1}^n\sum_{j=1}^{\hat n}\de_i\hat\de_j\,\cdot \\
\bigl\{&[((Z_i\times\hat Z_j)\times_{Q_{\fU,n}\times
Q_{\fV,n}}(Q'_{\fU,\fV,n}/\bG_m)\times[\Spec\C/\bG_m],
\psi_{ij}\ci\io^{\bG_m^2})]\\
-&[((\hat Z_j\times Z_i)\times_{Q_{\fV,n}\times Q_{\fU,n}}(\ti
Q'_{\fV,\fU,n}/\bG_m)\times[\Spec\C/\bG_m],
\ti\psi_{ji}\ci\io^{\bG_m^2})]\bigr\}.
\end{split}
\label{dt11eq20}
\e
Note that \eq{dt11eq20} writes $[f,g]\in\oSFai(\fM,\chi,\Q)$ as a
$\Q$-linear combination of $[(U\times[\Spec\C/\bG_m],\rho)]$ for $U$
a quasiprojective $\C$-variety, as in Proposition~\ref{dt3prop1}.

We now apply the $\Q$-linear map $\ti\Psi^{\chi,\Q}$ to $f,g$ and
$[f,g]$. Since $f,g$ are supported on $\fM^\al,\fM^\be$, Definition
\ref{dt5def1} and equations \eq{dt11eq7} and \eq{dt11eq8} yield
\e
\ti\Psi^{\chi,\Q}(f)=\ga\, \ti \la^\al\quad\text{and}\quad
\ti\Psi^{\chi,\Q}(g)=\hat\ga\, \ti \la^\be,
\label{dt11eq21}
\e
where $\ga,\hat\ga\in\Q$ are given by
\e
\ga=\ts\sum_{i=1}^n\de_i\,\chi\bigl(Z_i,(\io_\fU\ci\rho_i)^*(\nu_\fM)
\bigr), \quad\hat\ga=\ts\sum_{j=1}^{\hat n}\hat\de_j\,
\chi\bigl(\hat Z_j,(\io_\fV\ci\hat\rho_j)^*(\nu_\fM)\bigr).
\label{dt11eq22}
\e
Using Theorem \ref{dt4thm1}(iii) and Corollary \ref{dt4cor1} we have
\begin{align*}
\chi\bigl(Z_i,(\io_\fU\ci\rho_i)^*&(\nu_\fM)\bigr)\chi\bigl(\hat
Z_j,(\io_\fV\ci\hat\rho_j)^*(\nu_\fM)\bigr)\\
&=\chi\bigl(Z_i\times\hat Z_j,(\io_\fU\ci\rho_i)^*(\nu_\fM)\boxdot
(\io_\fV\ci\hat\rho_j)^*(\nu_\fM)\bigr)\\
&=\chi\bigl(Z_i\times\hat Z_j,(\io_\fU\ci\rho_i\times\io_\fV
\ci\hat\rho_j)^*(\nu_{\fM\times\fM})\bigr).
\end{align*}
Thus multiplying the two equations of \eq{dt11eq22} together gives
\e
\ga\hat\ga=\ts\sum_{i=1}^n\sum_{j=1}^{\hat n}\de_i\hat\de_j\,
\chi\bigl(Z_i\times\hat Z_j,(\io_\fU\ci\rho_i\times\io_\fV
\ci\hat\rho_j)^*(\nu_{\fM\times\fM})\bigr).
\label{dt11eq23}
\e

In the same way, since $[f,g]$ is supported on $\fM^{\al+\be}$,
using \eq{dt11eq20} we have
\e
\begin{gathered}
\ti\Psi^{\chi,\Q}\bigl([f,g]\bigr)=\ze\,\ti
\la^{\al+\be},\qquad\text{where}\\
\begin{aligned}
\ze= &\ts\sum_{i=1}^n\sum_{j=1}^{\hat n}\de_i\hat\de_j\, \chi\bigl(
(Z_i\times\hat Z_j)\times_{Q_{\fU,n}\times
Q_{\fV,n}}(Q'_{\fU,\fV,n}/\bG_m),
\psi_{ij}^*(\nu_\fM)\bigr)\\
-&\ts\sum_{i=1}^n\sum_{j=1}^{\hat n}\de_i\hat\de_j\,\chi\bigl( (\hat
Z_j\times Z_i)\times_{Q_{\fV,n}\times Q_{\fU,n}}(\ti
Q'_{\fV,\fU,n}/\bG_m), \ti\psi_{ji}^*(\nu_\fM)\bigr).
\end{aligned}
\end{gathered}
\label{dt11eq24}
\e
Write $\pi_{ij}:(Z_i\times\hat Z_j)\times_{Q_{\fU,n}\times
Q_{\fV,n}}(Q'_{\fU,\fV,n}/\bG_m)\ra Z_i\times\hat Z_j$ for the
projection, and $\ti\pi_{ji}$ for its analogue with $\fU,\fV$
exchanged. Then from \cite{Joyc1}, we have
\begin{equation*}
\chi\bigl((Z_i\!\times\!\hat Z_j)\times_{Q_{\fU,n}\times Q_{\fV,n}}(Q'_{\fU,\fV,n}/\bG_m),\psi_{ij}^*(\nu_\fM)\bigr)\!=\!
\chi\bigl(Z_i\!\times\!\hat
Z_j,\CF(\pi_{ij})(\psi_{ij}^*(\nu_\fM))\bigr),
\end{equation*}
where $\CF(\pi_{ij})$ is the pushforward of constructible
functions.\index{constructible function} Substituting this and its
analogue for $\ti\pi_{ji}$ into \eq{dt11eq24} and identifying
$Z_i\times\hat Z_j\cong \hat Z_j\times Z_i$ yields
\e
\begin{gathered}
\ze=\ts\sum_{i=1}^n\sum_{j=1}^{\hat n}\de_i\hat\de_j\,\chi\bigl(
Z_i\times\hat Z_j,F_{ij}\bigr), \qquad\text{where}\\
F_{ij}\!=\!\CF(\pi_{ij})(\psi_{ij}^*(\nu_\fM))-
\CF(\ti\pi_{ji})(\ti\psi_{ji}^*(\nu_\fM))\qquad\text{in
$\CF(Z_i\times \hat Z_j)$.}
\end{gathered}
\label{dt11eq25}
\e

Let $z_1\in Z_i(\C)$ for some $i=1,\ldots,n$, and $\hat z_2\in\hat
Z_j(\C)$ for some $j=1,\ldots,\hat n$. Set $q_1=(\xi_i)_*(z_1)$ in
$Q_{\fU,n}(\C)$ and $\smash{q_2=(\hat\xi_j)_*(\hat z_2)}$ in
$Q_{\fV,n}(\C)$, and let $q_1,q_2$ correspond to isomorphism classes
$[(E_1,\phi_1)],[(E_2,\phi_2)]$ with $[E_1]\in\fU(\C)$ and
$[E_2]\in\fV(\C)$. We will compute an expression for
$F_{ij}(z_1,\hat z_2)$ in terms of $E_1,E_2$. The fibre of
$\pi_{ij}:(Z_i\times\hat Z_j)\times_{Q_{\fU,n}\times
Q_{\fV,n}}(Q'_{\fU,\fV,n}/\bG_m)\ra Z_i\times\hat Z_j$ over
$(z_1,\hat z_2)$ is the fibre of $\Pi_\fU\times\Pi_\fV:
Q'_{\fU,\fV,n}/\bG_m\ra Q_{\fU,n}\times Q_{\fV,n}$ over $(q_1,q_2)$,
which is the projective space ${\mathbb
P}(W_{\fU,\fV,n}^{q_1,q_2})$. Thus the definition of $\CF(\pi_{ij})$
in \S\ref{dt21} implies that
\e
\bigl(\CF(\pi_{ij})(\psi_{ij}^*(\nu_\fM))\bigr)(z_1,\hat z_2)=
\chi\bigl({\mathbb P}(W_{\fU,\fV,n}^{q_1,q_2}),
\psi_{ij}^*(\nu_\fM)\bigr).
\label{dt11eq26}
\e

To understand the constructible function\index{constructible function}
$\psi_{ij}^*(\nu_\fM)$ on ${\mathbb P}(W_{\fU,\fV,n}^{q_1,q_2})$,
consider the linear map
$\smash{\pi_{E_2,E_1}:W_{\fU,\fV,n}^{q_1,q_2}\ra\Ext^1(E_2,E_1)}$ in
\eq{dt11eq5}. The kernel $\Ker\pi_{E_2,E_1}$ is a subspace of
$W_{\fU,\fV,n}^{q_1,q_2}$, so ${\mathbb
P}(\Ker\pi_{E_2,E_1})\subseteq {\mathbb
P}(W_{\fU,\fV,n}^{q_1,q_2})$. The induced map
\e
\smash{(\pi_{E_2,E_1})_*:{\mathbb
P}\bigl(W_{\fU,\fV,n}^{q_1,q_2}\bigr)\sm {\mathbb
P}\bigl(\Ker\pi_{E_2,E_1}\bigr)\longra{\mathbb
P}\bigl(\Ext^1(E_2,E_1)\bigr)}
\label{dt11eq27}
\e
is surjective as $\pi_{E_2,E_1}$ is, and has fibre
$\Ker\pi_{E_2,E_1}$. Let $[w]\in{\mathbb P}(W_{\fU,\fV,n}^{q_1,
q_2})$. If $[w]\notin {\mathbb P}(\Ker\pi_{E_2,E_1})$, write
$(\pi_{E_2,E_1})_*([w])=[\la]$ for $0\ne\la\in\Ext^1(E_2,E_1)$, and
then $(\psi_{ij})_*([w])=[F]$ in $\fM(\C)$ where the exact sequence
$0\ra E_1\ra F\ra E_2\ra 0$ corresponds to $\la\in\Ext^1(E_2,E_1)$,
and $(\psi_{ij}^*(\nu_\fM))([w])=\nu_{\fM}(F)$. If $[w]\in{\mathbb
P}(\Ker\pi_{E_2,E_1})$ then $(\psi_{ij})_*([w])=[E_1\op E_2]$ in
$\fM(\C)$, so $(\psi_{ij}^*(\nu_\fM))([w])=\nu_{\fM}(E_1\op E_2)$.
Therefore
\e
\begin{split}
\chi\bigl({\mathbb P}(W_{\fU,\fV,n}^{q_1,q_2}),
\psi_{ij}^*(\nu_\fM)\bigr)&=
\int_{\begin{subarray}{l}[\la]\in\mathbb{P}(\Ext^1(E_2,E_1)):\\
\la\; \Leftrightarrow\; 0\ra E_1\ra F\ra E_2\ra
0\end{subarray}}\!\!\!\!\!\!\nu_{\fM}(F)\,\rd\chi\\
&+\dim\Ker\pi_{E_2,E_1} \cdot \nu_{\fM}(E_1\op E_2),
\end{split}
\label{dt11eq28}
\e
since the fibres $\Ker\pi_{E_2,E_1}$ of $(\pi_{E_2,E_1})_*$ in
\eq{dt11eq27} have Euler characteristic 1, and $\chi\bigl({\mathbb
P}(\Ker\pi_{E_2,E_1})\bigr)=\dim\Ker\pi_{E_2,E_1}$.

Combining \eq{dt11eq26} and \eq{dt11eq28} with their analogues with
$\fU,\fV$ exchanged and substituting into \eq{dt11eq25} yields
\ea
F_{ij}(z_1,\hat z_2)&=
\int_{\begin{subarray}{l}[\la]\in\mathbb{P}(\Ext^1(E_2,E_1)):\\
\la\; \Leftrightarrow\; 0\ra E_1\ra F\ra E_2\ra
0\end{subarray}}\!\!\!\!\!\!\nu_{\fM}(F)\,\rd\chi
-\int_{\begin{subarray}{l}[\ti\la]\in\mathbb{P}(\Ext^1(E_1,E_2)):\\
\ti\la\; \Leftrightarrow\; 0\ra E_2\ra\ti F\ra E_1\ra
0\end{subarray}}\!\!\!\!\!\!\nu_{\fM}(\ti F)\,\rd\chi
\nonumber\\
&+\bigl(\dim\Ker\pi_{E_2,E_1}-\dim\Ker\ti\pi_{E_1,E_2}\bigr)\nu_{\fM}(E_1\op
E_2).
\label{dt11eq29}
\ea
From the exact sequences \eq{dt11eq5}--\eq{dt11eq6} we see that
\begin{align*}
&\dim\Ker\pi_{E_2,E_1}-\dim\Ker\ti\pi_{E_1,E_2}\\
&=\!\bigl(\dim\Hom(V,U)\!-\!\dim\Hom(E_2,E_1)\bigr)
\!-\!\bigl(\dim\Hom(U,V)\!-\!\dim\Hom(E_1,E_2)\bigr)\\
&=\dim\Hom(E_1,E_2)-\dim\Hom(E_2,E_1).
\end{align*}
Substituting this into \eq{dt11eq29} and using \eq{dt3eq14},
\eq{dt5eq2} and \eq{dt5eq3} gives
\begin{align*}
F_{ij}(z_1,\hat z_2)&=\bigl(\dim\Ext^1(E_2,E_1)-\dim\Ext^1(E_1,E_2)\\
&\qquad+\dim\Hom(E_1,E_2)-\dim\Hom(E_2,E_1)\bigr)\nu_\fM(E_1\op E_2)\\
&=(-1)^{\bar\chi(\al,\be)}\bar\chi(\al,\be)\nu_{\fM\times\fM}(E_1,E_2)\\
&=(-1)^{\bar\chi(\al,\be)}\bar\chi(\al,\be)
(\io_\fU\ci\rho_i\times\io_\fV \ci\hat\rho_j)^*(\nu_{\fM\times\fM})(z_1,\hat
z_2).
\end{align*}
Hence $F_{ij}\equiv (-1)^{\bar\chi(\al,\be)}\bar\chi(\al,\be)
(\io_\fU\ci\rho_i\times\io_\fV
\ci\hat\rho_j)^*(\nu_{\fM\times\fM})$. So \eq{dt11eq23},
\eq{dt11eq25} give
\begin{align*}
\ze&=\ts\sum_{i=1}^n\sum_{j=1}^{\hat n}\de_i\hat\de_j\,
\chi\bigl(Z_i\times\hat Z_j,(-1)^{\bar\chi(\al,\be)}\bar\chi(\al,\be)
(\io_\fU\ci\rho_i\times\io_\fV \ci\hat\rho_j)^*(\nu_{\fM\times\fM})\bigr)\\
&=(-1)^{\bar\chi(\al,\be)}\bar\chi(\al,\be)\ga\hat\ga.
\end{align*}

From equations \eq{dt11eq21} and \eq{dt11eq24} we now have
\begin{equation*}
\ti\Psi^{\chi,\Q}(f)=\ga\,\ti \la^\al,\;\>
\ti\Psi^{\chi,\Q}(g)=\hat\ga\, \ti \la^\be,\;\>
\ti\Psi^{\chi,\Q}\bigl([f,g]\bigr)=
(-1)^{\bar\chi(\al,\be)}\bar\chi(\al,\be)\ga\hat\ga\,\ti
\la^{\al+\be},
\end{equation*}
so $\ti\Psi^{\chi,\Q}\bigl([f,g]\bigr)=\bigl[\ti\Psi^{\chi,\Q}(f),
\ti\Psi^{\chi,\Q}(g)\bigr]$ by \eq{dt5eq4}, and $\ti\Psi^{\chi,\Q}$
is a Lie algebra morphism. This completes the proof of
Theorem~\ref{dt5thm5}.

\section[The proofs of Theorems $\text{\ref{dt5thm7},
\ref{dt5thm8}}$ and $\text{\ref{dt5thm9}}$]{The proofs of Theorems
\ref{dt5thm7}, \ref{dt5thm8} and \ref{dt5thm9}}
\label{dt12}\index{stable pair|(}

This section will prove Theorems \ref{dt5thm7}, \ref{dt5thm8} and
\ref{dt5thm9}, which say that the moduli space of {\it stable
pairs\/} introduced in \S\ref{dt54} is a projective $\K$-scheme
$\M_\stp^{\al,n}(\tau')$ with a symmetric obstruction
theory,\index{symmetric obstruction theory}\index{obstruction
theory!symmetric} and the corresponding invariants
$PI^{\al,n}(\tau')= \int_{[\M_\stp^{\al,n}(\tau')]^\vir}1$ are
unchanged under deformations of $X$. Throughout $\K$ is an arbitrary
algebraically closed field,\index{field $\K$} and when we consider
Calabi--Yau 3-folds $X$ over $\K$, we do not assume $H^1(\cO_X)=0$.
A good reference for the material we use on derived categories of
coherent sheaves\index{derived category} is Huybrechts~\cite{Huyb}.

\subsection{The moduli scheme of stable pairs
$\M_\stp^{\al,n}(\tau')$}
\label{dt121}

To prove deformation-invariance in Theorem \ref{dt5thm9} we will
need to work not with a single Calabi--Yau 3-fold $X$ over $\K$, but
with a {\it family\/} of Calabi--Yau 3-folds
$X\stackrel{\vp}{\longra}U$ over a base $\K$-scheme $U$. Taking
$U=\Spec\K$ recovers the case of one Calabi--Yau 3-fold. Here are
our assumptions and notation for such families.

\begin{dfn} Let $\K$ be an algebraically closed field, and
$X\stackrel{\vp}{\longra}U$ be a smooth projective morphism of
algebraic $\K$-varieties $X,U$, with $U$ connected. Let $\cO_X(1)$
be a relative very ample line bundle for
$X\stackrel{\vp}{\longra}U$. For each $u\in U(\K)$, write $X_u$ for
the fibre $X\times_{\vp,U,u}\Spec\K$ of $\vp$ over $u$, and
$\cO_{X_u}(1)$ for $\cO_X(1)\vert_{X_u}$. Suppose that $X_u$ is a
smooth Calabi--Yau 3-fold over $\K$ for all $u\in U(\K)$, which may
have $H^1(\cO_{X_u})\ne 0$. The Calabi--Yau condition implies that
the dualizing complex $\om_{\vp}$ of $\vp$ is a line bundle trivial
on the fibres of~$\vp$.

The hypotheses of Theorem \ref{dt5thm9} require the
$K^\num(\coh(X_u))$ to be canonically isomorphic {\it locally in\/}
$U(\K)$. But by Theorem \ref{dt4thm8}, we can pass to a finite cover
$\ti U$ of $U$, so that the $K^\num(\coh(\ti X_{\ti u}))$ are
canonically isomorphic {\it globally in\/} $\ti U(\K)$. So,
replacing $X,U$ by $\ti X,\ti U$, we will assume from here until
Theorem \ref{dt12thm7} that the numerical Grothendieck
groups\index{Grothendieck group!numerical} $K^\num(\coh(X_u))$ for $u\in
U(\K)$ are all canonically isomorphic {\it globally in\/} $U(\K)$,
and we write $K(\coh(X))$ for this group $K^\num(\coh(X_u))$ up to
canonical isomorphism. We return to the locally isomorphic case
after Theorem~\ref{dt12thm7}.

Let $E$ be a coherent sheaf on $X$ which is flat over $U$. Then the
fibre $E_u$ over $u\in U(\K)$ is a coherent sheaf on $X_u$, and as
$E$ is flat over $U$ and $U(\K)$ is connected, the class $[E_u]\in
K^\num(\coh(X_u))\cong K(\coh(X))$ is independent of $u\in U(\K)$.
We will write $[E]\in K(\coh(X))$ for this class~$[E_u]$.

For any $\al\in K(\coh(X))$, write $P_\al$ for the {\it Hilbert
polynomial\/} of $\al$ with respect to $\cO_X$. Then for any $u\in
U(\K)$, if $E_u\in\coh(X_u)$ with $[E_u]=\al$ in
$K^\num(\coh(X_u))\cong K(\coh(X))$, the Hilbert polynomial
$P_{E_u}$ of $E_u$ w.r.t.\ $\cO_{X_u}(1)$ is $P_\al$. Define
$\tau:C(\coh(X))\ra G$ by $\tau(\al)=P_\al/r_\al$ as in Example
\ref{dt3ex1}, where $r_\al$ is the leading coefficient of $P_\al$.
Then $(\tau,G,\le)$ is Gieseker stability on $\coh(X_u)$, for
each~$u\in U(\K)$.

Later, we will fix $\al\in K(\coh(X))$, and we will fix an integer
$n\gg 0$, such that every Gieseker semistable coherent sheaf $E$
over any fibre $X_u$ of $X\ra U$ with $[E]=\al\in K^\num(\coh(X_u))
\cong K(\coh(X))$ is $n$-regular. This is possible as $U$ is of
finite type. We follow the convention in \cite{BeFa1} of taking
$D(X)$\nomenclature[D(X)]{$D(X)$}{derived category of quasi-coherent sheaves on
$X$} to be the derived category of complexes of {\it
quasi-coherent\/} sheaves\index{sheaf!quasi-coherent} on $X$, even
though complexes in this book will always have coherent cohomology.
\label{dt12def1}
\end{dfn}

We generalize Definitions \ref{dt5def3} and \ref{dt5def4} to the
families case:

\begin{dfn} Let $\K$, $X\stackrel{\vp}{\longra}U,\cO_X(1)$ be as above.
Fix $n\gg 0$ in $\Z$. A {\it pair\/} is a nonzero morphism of
sheaves $s:\cO_{X}(-n)\ra E$, where $E$ is a nonzero sheaf on $X$,
flat over $U$. A {\it morphism\/} between two pairs
$s:\cO_{X}(-n)\ra E$ and $t:\cO_{X}(-n)\ra F$ is a morphism of
$\cO_X$-modules $f:E\ra F$, with $f\ci s=t$. A pair $s:\cO_X(-n)\ra
E$ is called {\it stable\/} if:
\begin{itemize}
\setlength{\itemsep}{0pt}
\setlength{\parsep}{0pt}
\item[(i)] $\tau([E'])\le\tau([E])$ for all subsheaves $E'$ of
$E$ with $0\neq E'\neq E$; and
\item[(ii)] if also $s$ factors through $E'$,
then~$\tau([E'])<\tau([E])$.
\end{itemize}
The {\it class\/} of a pair $s:\cO_{X}(-n)\ra E$ is the numerical
class $[E]$ in $K(\coh(X))$. We will use $\tau'$ {\it to denote
stability of pairs}, defined using $\cO_X(1)$.

Let $T$ be a $U$-scheme, that is, a morphism of $\K$-schemes
$\psi:T\ra U$. Let $\pi:X_T\ra T$ be the pullback of $X$ to $T$,
that is, $X_T=X\times_{\vp,U,\psi}T$. A $T$-{\it family of stable
pairs with class\/} $\al$ is a morphism of $\cO_{X_T}$-modules
$s:\cO_{X_T}(-n)\ra E$, where $E$ is flat over $T$, and when
restricting to $U$-points $t$ in $T$, $s_t:\cO_{X_t}(-n)\ra E_t$ is
a stable pair, and $[E_t]=\al$ in $K(\coh(X))$. As in Definition
\ref{dt5def4} we define the {\it moduli functor of stable pairs with
class\/} $\al$:
\begin{equation*}
\xymatrix{ {\mathbb M}_\stp^{\al,n}(\tau'): \Sch_U \ar[r]
&{\mathop{\bf Sets}}.}
\end{equation*}
\label{dt12def2}
\end{dfn}

Pairs, or framed modules, have been studied extensively for the last
twenty years, especially on curves. Some references are Bradlow et
al.\ \cite{BDGW}, Huybrechts and Lehn \cite{HuLe1} and Le Potier
\cite{LePo}. Note that stable pairs can have no automorphisms. Le
Potier gives the construction of the moduli spaces in our generality
in \cite[Th.~4.11]{LePo}. It follows directly from his construction
that in our case of stable pairs, with no strictly semistables, we
always get a fine moduli scheme.\index{fine moduli scheme}\index{moduli
scheme!fine} Theorem \ref{dt5thm7} follows when $U=\Spec\K$. Later
in the section we will abbreviate $\M_\stp^{\al,n}(\tau')$ to $\M$,
especially in subscripts~$X_\M$.

\begin{thm} Let\/ $\al\in K(\coh(X))$ and\/ $n\in\Z$. Then the
moduli functor ${\mathbb M}_\stp^{\al,n}(\tau')$ is represented by a
projective $U$-scheme~$\M_\stp^{\al,n}(\tau')$.
\label{dt12thm1}
\end{thm}

\begin{proof} This follows from a more general result of Le
Potier \cite[Th.~4.11]{LePo}, which we now explain. Let $X$ be a
smooth projective $U$-scheme of dimension $m$, with very ample line
bundle $\cO_X(1)$. In \cite[\S 4]{LePo}, Le Potier defines a {\it
coherent system\/}\index{coherent system} to be a pair $(\Ga,E)$,
where $E\in\coh(X)$ and $\Ga\subseteq H^0(E)$ is a vector subspace.
A {\it morphism of coherent systems\/} $f:(\Ga,E)\ra(\Ga',E')$ is a
morphism $f:E\ra E'$ in $\coh(X)$ with~$f(\Ga)\subseteq\Ga'$.

Let $q\in\Q[t]$ be a polynomial with positive leading coefficient.
For a coherent system $(\Ga,E)$ with $E\ne 0$, define a polynomial
$p_{(\Ga,E)}$ in $\Q[t]$ by
\begin{equation*}
p_{(\Ga,F)}(t)=\frac{P_E(t)+\dim\Ga\cdot q(t)}{r_E}\,,
\end{equation*}
where $P_E$ is the Hilbert polynomial of $E$ and $r_E>0$ its leading
coefficient. Then we call $(\Ga,E)$ $q$-{\it semistable\/}
(respectively $q$-{\it stable\/}) if $E$ is pure and whenever
$E'\subset E$ is a subsheaf with $E'\ne 0,E$ and $\Ga'=\Ga\cap
H^0(E')\subset H^0(E)$ we have $p_{(\Ga',E')}\le p_{(\Ga,E)}$
(respectively $p_{(\Ga',E')}< p_{(\Ga,E)}$), using the total order
$\le$ on polynomials $\Q[t]$ in Example~\ref{dt3ex1}.

Then Le Potier \cite[Th.~4.11]{LePo} shows that if $\al\in
K^\num(\coh(X))$, then the moduli functor $\underline{\rm
Sys}^\al(q)$ of $q$-semistable coherent systems $(\Ga,E)$ with
$[E]=\al$ is represented by a projective moduli $U$-scheme
$\mathop{\rm Sys}^\al(q)$, such that $U$-points of $\mathop{\rm
Sys}^\al(q)$ correspond to S-equivalence\index{S-equivalence} classes of
coherent systems $(\Ga,E)$. The method is to fix $N\gg 0$ and a
vector space $V$ of dimension $P_\al(N)$, and to define a projective
`Quot scheme'\index{Quot scheme} $\mathop{\rm Quot}^{\al,N}(q)$ of pairs
$\bigl((\Ga,E),\vp\bigr)$, where $(\Ga,E)$ is a coherent system with
$[E]=\al$ and $\vp:V\ra H^0(E(N))$ is an isomorphism, and $\GL(V)$
acts on $\mathop{\rm Quot}^{\al,N}(q)$. Le Potier shows that there
exists a linearization $\cal L$ for this action of $\GL(V)$,
depending on $q$, such that GIT (semi)stability of
$\bigl((\Ga,E),\vp\bigr)$ coincides with $q$-(semi)stability of
$(\Ga,E)$. Then $\mathop{\rm Sys}^\al(q)$ is the GIT
quotient~$\mathop{\rm Quot}^{\al,N}(q)/\!/_{\cal L}\GL(V)$.

Here is how to relate this to our situation. Fix $\al\in K(\coh(X))$
and $n\in\Z$. To a pair $s:\cO_{X}(-n)\ra E$ in the sense of
Definition \ref{dt12def2} we associate the coherent system $(\langle
s\rangle,E(n))$, with sheaf $E(n)=E\ot\cO_X(n)$ and 1-dimensional
subspace $\Ga\subset H^0(E(n))$ spanned by $0\ne s\in H^0(E(n))$. We
take $q\in\Q[t]$ to have degree 0, so that $q\in\Q_{>0}$, and to be
sufficiently small (in fact $0<q\le 1/d!\,r_\al$ is enough, where
$d=\dim\al$ and $r_\al$ is the leading coefficient of the Hilbert
polynomial $P_\al$). Then it is easy to show that $s:\cO_X(-n)\ra E$
is stable if and only if $(\langle s\rangle,E(n))$ is $q$-stable if
and only if $(\langle s\rangle,E(n))$ is $q$-semistable.

Hence \cite[Th.~4.11]{LePo} gives a projective coarse moduli
$U$-scheme\index{coarse moduli scheme}\index{moduli scheme!coarse}
$\M_\stp^{\al,n}(\tau')$. Since there are no strictly $q$-semistable
$(\langle s\rangle,E(n))$ in this moduli space, $U$-points of
$\M_\stp^{\al,n}(\tau')$ correspond to isomorphism classes of pairs
$s:\cO_{X}(-n)\ra E$, not just S-equivalence classes. Also, as
stable pairs $s:\cO_{X}(-n)\ra E$ have no automorphisms,
$\M_\stp^{\al,n}(\tau')$ is actually a fine moduli
scheme.\index{fine moduli scheme}\index{moduli scheme!fine}
\end{proof}

\subsection{Pairs as objects of the derived category}
\label{dt122}\index{derived category|(}

We can consider a stable pair $s:\cO_X(-n)\ra E$ on $X$ as a complex
$I$ in the derived category $D(X)$, with $\cO_X(-n)$ in degree $-1$
and $E$ in degree 0. We evaluate some $\Ext$ groups for such a
complex~$I$.

\begin{prop} Let\/ $s:\cO_X(-n)\ra E$ be a stable pair, and suppose
$n$ is large enough that $H^i(E(n))=0$ for $i>0$. Write $I$ for
$\cO_X(-n)\,{\buildrel s\over\longra}\,E$ considered as an object
of\/ $D(X),$ with\/ $E$ in degree $0$. Then:
\begin{itemize}
\setlength{\itemsep}{0pt}
\setlength{\parsep}{0pt}
\item[{\bf (a)}] $\Ext_{\sst D(X)}^i(I,E)=0$ for\/ $i<1,$ $i>3$.
\item[{\bf (b)}] $\Ext_{\sst D(X)}^i(I,\cO_X(-n))=0$ for\/
$i<1,$ $i>3,$ and\/ $\Ext_{\sst D(X)}^1(I,\cO_X(-n))\cong\K$.
\item[{\bf (c)}] $\Ext_{\sst D(X)}^i(I,I)=0$ for\/ $i<0,$ $i>3,$
and\/ $\Ext_{\sst D(X)}^i(I,I)\cong\K$ for\/~$i=0,3$.
\end{itemize}
\label{dt12prop1}
\end{prop}

\begin{proof} We have a distinguished triangle $I[-1]\ra
\cO_X(-n)\,{\buildrel s\over\longra}\,E\ra I$ in $D(X)$. Taking
$\Ext_{\sst D(X)}^*$ of this with $E,\cO_X(-n)$ and $I$ gives long
exact sequences:
\e
\begin{gathered}
\!\!\!\!\!\text{\begin{footnotesize}$\displaystyle
\xymatrix@C=13pt@R=3pt{ \Ext^{i-1}(\cO(-\!n),E) \ar[r] & \Ext^i(I,E)
\ar[r] & \Ext^i(E,E) \ar[r]^(0.42){\ci s} & \Ext^i(\cO(-\!n),E),
\\
\!\Ext^{i-1}(\cO(-\!n),\cO(-\!n))\! \ar[r] & \!\Ext^i(I,\cO(-\!n))\!
\ar[r] & \!\Ext^i(E,\cO(-\!n))\! \ar[r]^(0.42){\ci s} &
\!\Ext^i(\cO(-n),\cO(-n)),\!
\\
\Ext^{i-1}(\cO(-\!n),I) \ar[r] & \Ext^i(I,I) \ar[r] & \Ext^i(E,I)
\ar[r]^(0.42){\ci s} & \Ext^i(\cO(-\!n),I).
}\!\!\!$\end{footnotesize}}
\end{gathered}
\label{dt12eq1}
\e

In the first row of \eq{dt12eq1}, since $\Ext_{\sst
D(X)}^i(\cO_X(-n),E)=H^i(E(n))=0$ for $i\ne 0$ and $\Ext_{\sst
D(X)}^i(E,E)=0$ for $i<0$ and $i>3$, this gives $\Ext_{\sst
D(X)}^i(I,E)=0$ for $i<0$ and $i>3$, and
\e
\Hom_{\sst D(X)}(I,E)\cong\Ker\bigl(\ci s:\Hom(E,E)\longra
\Hom(\cO_X(-n),E)\bigr).
\label{dt12eq2}
\e
Write $\pi:E\ra F$ for the cokernel of $s:\cO_X(-n)\ra E$. Suppose
$0\ne\be\in\Hom(E,E)$ with $\be\ci s=0$ in $\Hom(\cO_X(-n),E)$. Both
$\Ker(\be)$ and $\Im(\be)$ are in fact subsheaves of $E$ with
$[E]=[\Ker\be]+[\Im\be]$, and $\Ker\be\ne 0$ as $\be\ci s=0$ with
$s\ne 0$, and $\Im\be\ne 0$ as $\be\ne 0$. Since $E$ is
$\tau$-semistable, the seesaw inequalities imply that
$\tau([\Ker\be])=\tau([\Im\be])=\tau([E])$. But as $s$ factors
through $\Ker\be$, stability of the pair implies that
$\tau([\Ker\be])<\tau([E])$, a contradiction. So $\Hom_{\sst
D(X)}(I,E)=0$ by \eq{dt12eq2}, proving~(a).

We have $\Ext_{\sst D(X)}^i(\cO_X(-n),\cO_X(-n))=H^i(\cO_X)=0$ for
$i<0$ or $i>3$ and is $\K$ for $i=0$, and $\Ext_{\sst
D(X)}^i(E,\cO_X(-n))\cong\Ext_{\sst D(X)}^{3-i}(\cO_X(-n),E)^* \cong
H^{3-i}(E(n))^*$ by Serre duality, which is zero unless $i=3$. Part
(b) follows from the second row of \eq{dt12eq1} and the fact that
$\ci s:\Ext^3(E,\cO_X(-n))\ra\Ext^3(\cO_X(-n),\cO_X(-n))\cong\K$ is
nonzero, as this is Serre dual to the morphism
$\Hom(\cO_X(-n),\cO_X(-n))\ra\Hom(\cO_X(-n),E)$ taking~$1\mapsto
s\ne 0$.

For (c), Serre duality and part (a) gives $\Ext_{\sst
D(X)}^i(E,I)=0$ for $i<0$ and $i>2$. Thus the third row of
\eq{dt12eq1} yields $\Ext_{\sst D(X)}^i(I,I)\cong \Ext_{\sst
D(X)}^{i-1}(\cO_X(-n),I)\cong \Ext_{\sst D(X)}^{4-i}(I,\cO_X(-n))^*$
for $i<0$ and $i>3$. So $\Ext_{\sst D(X)}^i(I,I)=0$ for $i<0$ and\/
$i>3$ by (b). Also $\K\cong \Ext_{\sst
D(X)}^2(\cO_X(-n),I)\ra\Ext_{\sst D(X)}^3(I,I)\ra 0$ is exact, and
$\Ext_{\sst D(X)}^3(I,I)\cong\Hom_{\sst D(X)}(I,I)^*$ cannot be zero
as $I$ is nonzero in $D(X)$, so $\Ext_{\sst D(X)}^3(I,I)\cong\K$ and
hence $\Ext_{\sst D(X)}^0(I,I)\cong\K$ by Serre duality, giving~(c).
\end{proof}

\begin{cor} In the situation of Proposition\/ {\rm\ref{dt12prop1},}
the object\/ $I$ in $D(X)$ up to quasi-isomorphism and the integer\/
$n$ determine the stable pair $s:\cO_X(-n)\ra E$ up to isomorphism.
\label{dt12cor1}
\end{cor}

\begin{proof} In the distinguished triangle $I[-1]\,{\buildrel\pi\over
\longra}\,\cO_X(-n)\,{\buildrel s\over\longra}\,E\ra I$, the
morphism $\pi$ is nonzero since otherwise $E\cong\cO_X(-n)\op I$
which is not a sheaf. But $\Ext^1_{\sst D(X)}(I,\cO_X(-n))\cong\K$
by Proposition \ref{dt12prop1}(b), so $\Ext^1_{\sst
D(X)}(I,\cO_X(-n))=\K\cdot\pi$. Thus the morphism $\pi$ is
determined by $I,n$ up to $\bG_m$ rescalings, so $E\cong\cone(\pi)$
and $s$ are determined by $I,n$ up to isomorphism.
\end{proof}

We have $U$-schemes $X$ and $\M_\stp^{\al,n}(\tau')$, by Theorem
\ref{dt12thm1}, so we can form the fibre product
$X_\M=X\times_U\M_\stp^{\al,n}(\tau')$, which we regard as a family
of Calabi--Yau 3-folds $X_\M\,{\buildrel\pi\over\longra}\,
\M_\stp^{\al,n}(\tau')$ over the base $U$-scheme
$\M_\stp^{\al,n}(\tau')$. Since $\M_\stp^{\al,n}(\tau')$ is a fine
moduli scheme for pairs\index{fine moduli scheme}\index{moduli
scheme!fine} on $X$, on $X_\M$ we have a universal pair, which we
denote by ${\mathbb S}: \cO_{X_\M}(-n)\ra{\mathbb E}$. We can regard
this as an object in the derived category $D(X_{\M_\stp^{\al,n}
(\tau')})$, with $\cO_{X_{\M_\stp^{\al,n} (\tau')}}(-n)$ in degree
$-1$ and $\mathbb E$ in degree 0, and we will denote this object
by~${\mathbb I}=\cone({\mathbb S})$.\index{stable
pair|)}\index{derived category|)}

\subsection{Cotangent complexes and obstruction theories}
\label{dt123}\index{cotangent complex|(}\index{obstruction theory|(}

Suppose $X,Y$ are schemes over some base $\K$-scheme $U$, and
$X\,{\buildrel\phi\over\longra}\,Y$ is a morphism of $U$-schemes.
Then one can define the {\it cotangent sheaf\/} (or {\it sheaf of
relative differentials\/}) $\Om_{X/Y}$ in $\coh(X)$, as in
Hartshorne \cite[\S II.8]{Hart2}. This generalizes cotangent bundles
of smooth schemes: if $X$ is a smooth $\K$-scheme then
$\Om_{X/\Spec\K}$ is the cotangent bundle $T^*X$, a locally free
sheaf of rank $\dim X$ on $X$. If $X\,{\buildrel\phi\over\longra}\,
Y\, {\buildrel\psi\over\longra}\,Z$ are morphisms of $U$-schemes
then there is an exact sequence
\e
\xymatrix{ \phi^*(\Om_{Y/Z}) \ar[r] & \Om_{X/Z} \ar[r] & \Om_{X/Y}
\ar[r] & 0 }
\label{dt12eq3}
\e
in $\coh(X)$. Note that the morphism $\phi^*(\Om_{Y/Z})\ra\Om_{X/Z}$
need not be injective, that is, \eq{dt12eq3} may not be a short
exact sequence. Morally speaking, this says that $\phi\mapsto
\Om_{X/Y}$ is a right exact functor, but may not be left exact.

Cotangent complexes are derived versions of cotangent sheaves, for
which \eq{dt12eq3} is replaced by a distinguished triangle
\eq{dt12eq4}, making it fully exact. The {\it cotangent complex\/}
$L_{X/Y}$ of a morphism $X\,{\buildrel\phi\over\longra}\,Y$ is an
object in the derived category $D(X)$, constructed by Illusie
\cite{Illu1}; a helpful review is given in Illusie \cite[\S
1]{Illu2}. It has $h^0(L_{X/Y})\cong\Om_{X/Y}$. If $\phi$ is smooth
then $L_{X/Y}=\Om_{X/Y}$. Here are some properties of cotangent
complexes:
\begin{itemize}
\setlength{\itemsep}{0pt}
\setlength{\parsep}{0pt}
\item[(a)] Suppose $X\,{\buildrel\phi\over\longra}\,Y\,
{\buildrel\psi\over\longra}\,Z$ are morphisms of $U$-schemes.
Then there is a distinguished triangle in $D(X)$, \cite[\S
2.1]{Illu1}, \cite[\S 1.2]{Illu2}:
\e
\xymatrix{ L\phi^*(L_{Y/Z}) \ar[r] & L_{X/Z} \ar[r] & L_{X/Y}
\ar[r] & L\phi^*(L_{Y/Z})[1]. }
\label{dt12eq4}
\e
This is called the {\it distinguished triangle of transitivity}.
\item[(b)] Suppose we have a commutative diagram of morphisms of
$U$-schemes:
\begin{equation*}
\xymatrix@R=10pt{ R \ar[r]_\rho \ar[d]^\ep & S \ar[r]_\si \ar[d]
& T \ar[d] \\ X \ar[r]^\phi & Y \ar[r]^\psi & Z.}
\end{equation*}
Then we get a commutative diagram in $D(R)$, \cite[\S
2.1]{Illu1}:
\begin{equation*}
\xymatrix@R=13pt@C=10pt{ L\rho^*(L_{S/T}) \ar[r] & L_{R/T}
\ar[r] & L_{R/S} \ar[r] & L\phi^*(L_{S/T})[1] \\
L\ep^*\bigl(L\phi^*(L_{Y/Z})\bigr) \ar[r] \ar[u] &
L\ep^*(L_{X/Z}) \ar[r] \ar[u] & L\ep^*(L_{X/Z}) \ar[r] \ar[u] &
L\ep^*(\bigl(L\phi^*(L_{Y/Z})[1]\bigr), \ar[u]}
\end{equation*}
where the rows come from the distinguished triangles of
transitivity for $R\ra S\ra T$ and~$X\ra Y\ra Z$.
\item[(c)] Suppose we have a Cartesian diagram of $U$-schemes:
\begin{equation*}
\xymatrix@R=10pt{ X\times_ZY \ar[r]_{\pi_Y} \ar[d]_{\pi_X} & Y
\ar[d]^\psi \\ X \ar[r]^\phi & Z.}
\end{equation*}
If $\phi$ or $\psi$ is flat then we have {\it base change
isomorphisms\/} \cite[\S 2.2]{Illu1}, \cite[\S 1.3]{Illu2}:
\e
\begin{split}
L_{X\times_ZY/Y}\cong L\pi_X^*(L_{X/Z}),\quad
L_{X\times_ZY/X}\cong L\pi_Y^*(L_{Y/Z}),\\
\text{and}\qquad L_{X\times_ZY/Z}\cong L\pi_X^*(L_{X/Z})\op
L\pi_Y^*(L_{Y/Z}).
\end{split}
\label{dt12eq5}
\e
\end{itemize}

Recall the following definitions from Behrend and
Fantechi~\cite{Behr,BeFa1,BeFa2}:

\begin{dfn} Let $Y$ be a $\K$-scheme, and $D(Y)$ the derived
category of quasicoherent sheaves on $Y$.
\begin{itemize}
\setlength{\itemsep}{0pt}
\setlength{\parsep}{0pt}
\item[{\bf(a)}]  A complex $E^\bu\in D(Y)$ is {\it perfect of perfect
amplitude contained in\/} $[a,b]$, if \'{e}tale locally on $Y$,
$E^\bu$ is quasi-isomorphic to a complex of locally free sheaves
of finite rank in degrees $a,a+1,\ldots,b$.
\item[{\bf(b)}] We say that a complex $E^\bu\in D(Y)$ {\it satisfies
condition\/} $(*)$ if
\begin{itemize}
\setlength{\itemsep}{0pt}
\setlength{\parsep}{0pt}
\item[(i)] $h^i(E^\bu)=0$ for all $i>0$,
\item[(ii)] $h^i(E^\bu)$ is coherent for $i=0,-1$.
\end{itemize}
\item[{\bf(c)}] An {\it obstruction theory\/}\index{obstruction
theory!definition} for $Y$ is a morphism $\phi:E^\bu\ra L_Y$ in
$D(Y)$, where $L_Y=L_{Y/\Spec\K}$ is the cotangent complex of
$Y$, and $E$ satisfies condition $(*)$, and $h^0(\phi)$ is an
isomorphism, and $h^{-1}(\phi)$ is an epimorphism.
\item[{\bf(d)}] An obstruction theory $\phi:E^\bu\ra L_Y$ is called
{\it perfect\/}\index{obstruction theory!perfect} if $E^\bu$ is
perfect of perfect amplitude contained in $[-1,0]$.
\item[{\bf(e)}] A perfect obstruction theory $\phi:E^\bu\ra L_Y$ on
$Y$ is called {\it symmetric\/}\index{obstruction
theory!symmetric}\index{symmetric obstruction theory!definition}
if there exists an isomorphism $\theta:E^\bu\ra E^{\bu\vee}[1]$,
such that $\theta^{\vee}[1]=\theta$. Here
$E^{\bu\vee}\!=\!R\SHom(E^\bu,\cO_Y)$ is the {\it dual\/} of
$E^\bu$, and $\theta^\vee$ the dual morphism of~$\theta$.
\end{itemize}
If instead $Y\stackrel{\psi}{\longra}U$ is a morphism of
$\K$-schemes, so $Y$ is a $U$-scheme, we define {\it relative
perfect obstruction theories\/} $\phi:E^\bu\ra L_{Y/U}$ in the
obvious way.
\label{dt12def3}
\end{dfn}

A closed immersion of $\K$-schemes $j:T\ra\ov{T}$ is called a {\it
square zero extension\/}\index{square zero extension} with {\it ideal
sheaf\/}\index{ideal sheaf} $J$ if $J$ is the ideal sheaf of $T$ in
$\ov{T}$ and $J^2=0$, so that we have an exact sequence
in~$\coh(\ov{T})$:
\e
\xymatrix{ 0 \ar[r] & J \ar[r] & \cO_{\ov{T}} \ar[r] & \cO_{T}
\ar[r] & 0.}
\label{dt12eq6}
\e
We will always take $T,\ov T$ to be {\it affine\/} schemes.

The deformation theory of a $\K$-scheme $Y$ is largely governed by
its cotangent complex $L_Y \in D(Y)$, in the following sense.
Suppose that we are given a square-zero extension $\ov{T}$ of $T$
with ideal sheaf $J$, with $T,\ov T$ affine, and a morphism $g:T \ra
Y$. Then the theory of cotangent complexes gives a canonical
morphism
\begin{equation*}
\xymatrix{ g^*(L_Y) \ar[r] &L_T \ar[r] &J[1]}
\end{equation*}
in $D(T)$. This morphism, $\om(g)\in\Ext^1(g^*L_Y,J)$, is equal to
zero if and only if there exists an extension $\ov{g}:\ov{T}\ra Y$
of $g$. Moreover, when $\om(g)=0$, the set of isomorphism extensions
form a torsor under~$\Hom(g^*L_Y,J)$.

Behrend and Fantechi prove the following theorem, which both
explains the term obstruction theory and provides a criterion for
verification in practice:

\begin{thm}[{Behrend and Fantechi \cite[Th.~4.5]{BeFa1}}] The
following two conditions are equivalent for $E^\bu \in D(Y)$
satisfying condition $(*)$.
\begin{itemize}
\setlength{\itemsep}{0pt}
\setlength{\parsep}{0pt}
\item[{\bf(a)}] The morphism $\phi: E^\bu \ra L_Y$ is an
obstruction theory.
\item[{\bf(b)}] Suppose we are given a setup $(T,\ov{T},J,g)$ as
above. The morphism $\phi$ induces an element $\phi^*(\om(g))\in
\Ext^1(g^*E^\bu, J)$ from $\om(g)\in\Ext^1(g^*L_Y, J)$ by
composition. Then $\phi^*(\om(g))$ vanishes if and only if there
exists an extension $\ov{g}$ of\/ $g$. If it vanishes, then the
set of extensions form a torsor under~$\Hom(g^*E^\bu,J)$.
\end{itemize}
The analogue also holds for relative obstruction theories.
\label{dt12thm2}
\end{thm}\index{cotangent complex|)}\index{obstruction theory|)}

\subsection{Deformation theory for pairs}
\label{dt124}

Let $X,T,\ov{T}$ be $U$-schemes with $T,\ov T$ affine, and
$T\ra\ov{T}$ a square zero extension. Write $X_T=X\times_UT$ and
$X_{\ov{T}}=X\times_U\ov{T}$, which are $U$-schemes with projections
$X_T\ra T$, $X_{\ov{T}}\ra \ov{T}$ and $X_T\ra X_{\ov{T}}$. We have
a Cartesian diagram:
\begin{equation*}
\xymatrix@R=1pt@C=30pt{
&X_{\ov{T}} \ar[dd]^(0.3){\pi} \ar[rd] \\
X_T \ar[ru]^{j_*} \ar[rr]_(0.3){\tau} \ar[dd]^{\pi} & &X \ar[dd]^{\vp} \\
&\ov{T} \ar[rd] \\
T \ar[ru]^j \ar[rr] & &U, }
\end{equation*}
and an exact sequence in~$\coh(X_{\ov{T}})$:
\e
\xymatrix{ 0 \ar[r] & \pi^*J \ar[r] & \cO_{X_{\ov{T}}} \ar[r] &
\cO_{X_T} \ar[r] & 0.}
\label{dt12eq7}
\e

Let $s:\cO_{X_T}(-n)\ra E$ be a $T$-family of stable pairs. The
deformation theory of stable pairs involves studying
$\ov{T}$-families of stable pairs $\ov{s}:\cO_{X_{\ov{T}}}(-n)\ra
\ov{E}$ extending $s:\cO_{X_T}(-n)\ra E$, that is, with
$(j_*)^*\bigl((\ov{E},\ov{s})\bigr)\cong(E,s)$. Tensoring
\eq{dt12eq7} with $\ov{s}:\cO_{X_{\ov{T}}}(-n)\ra \ov{E}$ gives a
commutative diagram in $\coh(X_{\ov{T}})$, with exact rows:
\e
\begin{gathered}
\xymatrix@R=15pt@C=8pt{0 \ar[r] &\pi^*(J)\smash{
\ot_{\cO_{X_{\ov{T}}}}}\cO_{X_{\ov{T}}}(-n) \ar[r] \ar[d] &
\cO_{X_{\ov{T}}}(-n) \ar[r] \ar@{.>}[d]^{\ov s} &
\cO_{X_{\ov{T}}}(-n)\smash{\ot_{\cO_{X_{\ov{T}}}}}\cO_{X_T}
\ar[r] \ar[d]^{s\ot\id_{\cO_{X_T}}} &0 \\
0 \ar[r] &\pi^*(J)\smash{ \ot_{\cO_{X_{\ov{T}}}}}\ov{E} \ar@{.>}[r]
& \ov{E} \ar@{.>}[r] & \ov{E}\smash{\ot_{\cO_{X_{\ov{T}}}}}\cO_{X_T}
\ar[r] &0.}\!\!\!\!\!\!
\end{gathered}
\label{dt12eq8}
\e
But $\ov{E}\ot_{\cO_{X_{\ov{T}}}}\cO_{X_T}\cong(j_*)^*(\ov{E})\cong
E$, and $\cO_{X_{\ov{T}}}(-n)\ot_{\cO_{X_{\ov{T}}}}\cO_{X_T}\cong
\cO_{X_T}(-n)$, and as $J^2=0$ we have
$J\ot_{\cO_{\ov{T}}}\cO_T\cong J$, so $\pi^*(J)
\ot_{\cO_{X_{\ov{T}}}}\cO_{X_T}\cong\pi^*(J)$, and thus
\begin{equation*}
\pi^*(J)\ot_{\cO_{X_{\ov{T}}}}\ov{E}\!\cong\!
\pi^*(J)\ot_{\cO_{X_{\ov{T}}}}\cO_{X_T}
\ot_{\cO_{X_{\ov{T}}}}\ov{E}\!\cong\!\pi^*(J)
\ot_{\cO_{X_{\ov{T}}}}E\!\cong\!\pi^*(J)\ot_{\cO_{X_T}}E.
\end{equation*}
Hence \eq{dt12eq8} is equivalent to the commutative diagram
in~$\coh(X_{\ov{T}})$:
\e
\begin{gathered}
\xymatrix@R=15pt@C=18pt{0 \ar[r] &\pi^* J \ot_{\cO_{X_T}}
\cO_{X_T}(-n) \ar[r] \ar[d] & \cO_{X_{\ov{T}}}(-n) \ar[r]
\ar@{.>}[d]^{\ov s} & \cO_{X_T}(-n)
\ar[r] \ar[d]^{s} &0 \\
0 \ar[r] &\pi^* J \smash{\ot_{\cO_{X_T}}} E \ar@{.>}[r] & \ov E
\ar@{.>}[r] &E \ar[r] &0.}
\end{gathered}
\label{dt12eq9}
\e
Both rows are exact sequences of $\cO_{X_{\ov{T}}}$-modules. Since
$E$ is flat over $T$, such $\ov E$, if it exists, is necessarily
flat over $\ov{T}$. Now Illusie \cite[\S IV.3]{Illu1} studies the
problem of completing a diagram of the form \eq{dt12eq9}, and
proves:

\begin{thm}[Illusie {\cite[Prop.~IV.3.2.12]{Illu1}}] There exists an
element\/ $ob$ in $\Ext^2_{D(X_T)}\bigl(\cone(s),\pi^*J\ot E\bigr),$
whose vanishing is necessary and sufficient to complete the diagram
\eq{dt12eq9}. If\/ $ob=0$ then the set of isomorphism classes of
deformations forms a torsor
under~$\Ext^1_{D(X_T)}(\cone(s),\pi^*J\ot E)$.
\label{dt12thm3}
\end{thm}

Illusie also shows in \cite[\S IV.3.2.14]{Illu1} that this element
$ob$ can be written as the composition of three morphisms in
$D(X_T)$, in the commutative diagram:
\e
\begin{gathered}
\xymatrix@R=18pt@C=47pt{ \cone(s) \ar[d]^{ob}
\ar[r]_(0.3){\At_{\cO_{X_T}/\cO_X}(s)} &
k^1\bigl(L^{gr}_{(\cO_{X_T}\op \cO_{X_T}(-n))/\cO_{X}} \ot
(\cO_{X_T} \op E)\bigr)[1]
\ar[d]_(0.55){k^1(e(\cO_{X_{\ov{T}}}\op\cO_{X_{\ov{T}}}(-n))\ot
\id_{\cO_{X_T} \op E})} \\
\pi^* J \ot E [2] & k^1(\pi^*J \ot (\cO_{X_T} \op \cO_{X_T}(-n)) \ot
(\cO_{X_T} \op E))[2]. \ar[l]_(0.7){\Pi_1}}\!\!\!
\end{gathered}
\label{dt12eq10}
\e
Here in Illusie's set up, we regard $\cO_{X_T}\op \cO_{X_T}(-n)$ and
$\cO_{X_T} \op E$ as sheaves of {\it graded\/} algebras on $X_T$,
with $\cO_{X_T}$ in degree 0, and $\cO_{X_T}(-n),E$ in degree 1; we
also regard $\cO_{X}$ as a sheaf of graded algebras concentrated in
degree 0. Then $L^{\smash{gr}}_{(\cO_{X_T}\op
\cO_{X_T}(-n))/\cO_{X}}$ is the cotangent complex for sheaves of
graded algebras. The notation $k^1(\cdots)$ means take the degree 1
part of `$\cdots$' in the
grading.\nomenclature[kz1()]{$k^1(\cdots)$}{degree 1 part of graded
object `$\cdots$'}

The morphism $\At_{\cO_{X_T}/\cO_X}(s)$ in \eq{dt12eq10} is called
the {\it Atiyah class\/}\index{Atiyah class|(} of $s$, as in Illusie
\cite[\S IV.2.3]{Illu1}, and the morphism
\e
e\bigl(\cO_{X_{\ov{T}}}\!\op\!\cO_{X_{\ov{T}}}(-n)\bigr):L^{gr}_{(\cO_{X_T}\op
\cO_{X_T}(-n))/\cO_{X}}\!\ra\!\pi^*J\!\ot\!(\cO_{X_T}
\!\op\!\cO_{X_T}(-n))[1]
\label{dt12eq11}
\e
corresponds as in \cite[\S IV.2.4]{Illu1} to the following extension
of graded $\cO_X$-algebras:
\e
0 \!\ra\! \pi^*J\!\ot\! (\cO_{X_T}\!\op\!\cO_{X_T}(-n) ) \!\ra\!
\cO_{X_{\ov{T}}}\!\op\!\cO_{X_{\ov{T}}}(-n)\!\ra\!\cO_{X_T}\!\op\!
\cO_{X_T}(-n)\!\ra\! 0,
\label{dt12eq12}
\e
and the morphism $\Pi_1$ in \eq{dt12eq10} is projection to the first
factor on the right in
\begin{equation*}
k^1\bigl(\pi^*J \!\ot\! (\cO_{X_T}\!\op\!\cO_{X_T}(-n))\!\ot\!(\cO_{X_T}
\!\op\!E)\bigr)\!=\!(\pi^* J\!\ot\!E)\!\op\!(\pi^* J\!\ot\!\cO_{X_T}(-n)).
\end{equation*}

We will factorize \eq{dt12eq10} further. We have a cocartesian
diagram
\begin{equation*}
\xymatrix@R=10pt{\cO_{X_T} \ar[r] &\cO_{X_T} \op \cO_{X_T}(-n) \\
\cO_X \ar[u] \ar[r] &\cO_X \op \cO_X(-n) \ar[u] }
\end{equation*}
of sheaves of graded algebras. Since $\cO_X(-n)$ is a flat
$\cO_X$-module, $\cO_X \op \cO_X(-n)$ is a flat graded
$\cO_X$-algebra. Therefore, by \eq{dt12eq5} we have an isomorphism:
\e
\begin{split}
L_{(\cO_{X_T}/\cO_X)}\!\ot\!(\cO_{X_T}\!\op\!\cO_{X_T}(-n))\!\op\!
L^{gr}_{(\cO_{X}\op\cO_{X}(-n))/\cO_X}\!\ot\! (\cO_{X_T}\!\op\!
\cO_{X_T}(-n))&\!\!\! \\
{\buildrel\cong\over\longra}\, L^{gr}_{(\cO_{X_T}\op
\cO_{X_T}(-n))/\cO_{X}}.&\!\!\!
\end{split}
\label{dt12eq13}
\e

Since $\vp$ is flat in the following Cartesian diagram:
\begin{equation*}
\xymatrix@R=7pt{
X_T \ar[r] \ar[d]_\pi & X \ar[d]^\vp \\
T \ar[r] & U, }
\end{equation*}
equation \eq{dt12eq5} gives
\e
L_{\cO_{X_T}/\cO_X}\cong\pi^* L_{\cO_T/\cO_U}.
\label{dt12eq14}
\e
Let $e(\cO_{\ov{T}}) \in \Ext^1(L_{\cO_T/\cO_U},J)$ and
$e(\cO_{X_{\ov{T}}}) \in \Ext^1(L_{\cO_{X_T}/\cO_X}, \pi^* J)$
correspond to algebra extensions \eq{dt12eq6} and \eq{dt12eq7} as in
\cite[\S IV.2.4]{Illu1}. Since \eq{dt12eq7} is the pullback of
\eq{dt12eq6} by $\pi$ we have
\e
e(\cO_{X_{\ov{T}}})=\pi^*e(\cO_{\ov{T}})\in
\Ext^1(L_{\cO_{X_T}/\cO_X}, \pi^* J)\cong
\Ext^1(\pi^*L_{\cO_T/\cO_U},\pi^*J),
\label{dt12eq15}
\e
using \eq{dt12eq14}. The extension \eq{dt12eq12} is
$\ot_{\cO_{X}}(\cO_X \op \cO_X(-n))$ applied to \eq{dt12eq7}. Thus
the following diagram commutes:
\e
\begin{gathered}
\xymatrix@R=7pt@C=40pt{ L^{gr}_{\cO_{X_T} \op \cO_{X_T}(-n)/\cO_X}
\ar[d]^{pr_1}
\ar[rd]^{e(\cO_{X_{\ov{T}}}\op\cO_{X_{\ov{T}}}(-n))} \\
L_{\cO_{X_T}/\cO_X}\!\ot\!(\cO_{X_T}\!\op\!
\cO_{X_T}(-n))\ar[r]_{e(\cO_{X_{\ov{T}}})\ot\id_{\cO_{X_T} \op
\cO_{X_T}(-n)}} &\pi^* J\!\ot\!(\cO_{X_T}\!\op\!\cO_{X_T}(-n))[1],
}\!\!\!\!\!\!\!\!\!
\end{gathered}
\label{dt12eq16}
\e
where $pr_1$ is projection to the first factor in~\eq{dt12eq13}.

Combining equations \eq{dt12eq10}--\eq{dt12eq16} gives a commutative
diagram:
\begin{small}
\begin{equation}
\begin{gathered}
\xymatrix@R=15pt{
\cone(s) \ar[d]^(0.55){\At_{\cO_{X_T}/\cO_X}(s)}
\ar@{.>}[dddddr]^(0.55){ob}
\\
k^1\bigl(L^{gr}_{(\cO_{X_T}\op \cO_{X_T}(-n))/\cO_{X}} \ot (\cO_{X_T} \op
E)\bigr)[1]\!\!\!\!\!\!\!\!\!\!\!\!
\ar[d]^(0.4){pr_1} \ar@<4pt>[dr]^(0.65){{}\qquad\qquad k^1(e(\cO_{X_{\ov{T}}}
\op\cO_{X_{\ov{T}}}(-n))\ot \id_{\cO_{X_T} \op
E})}
\\
*\txt{$k^1\bigl(L_{(\cO_{X_T}/\cO_X)}\!\ot\!(\cO_{X_T}\!\op\!\cO_{X_T}(-n))$ \\
$\ot (\cO_{X_T}\!\op\!E)\bigr)[1]$}
\ar[r]_(0.55){k^1(e(\cO_{X_{\ov{T}}})\ot\id_{\ldots})} \ar[d]^{\text{project}}
&*\txt{$k^1\bigl(\pi^* J \ot (\cO_{X_T} \op \cO_{X_T}(-n))$ \\
$\ot (\cO_{X_T} \op E)\bigr) [2]$}\ar[d]_{\text{project}}
\ar@<7ex>@/^/@{.>}[ddd]^{\Pi_1}
\\
k^1\bigl(L_{(\cO_{X_T}/\cO_X)} \ot (\cO_{X_T} \op E)\bigr)[1]
\ar[r]_(0.55){k^1(e(\cO_{X_{\ov{T}}})\ot\id_{\ldots})}
\ar[d]^= & k^1\bigl(\pi^*J \ot (\cO_{X_T} \op
E)\bigr)[2] \ar[d]_=
\\
L_{\cO_{X_T}/\cO_X} \ot E[1] \ar[r]_{e(\cO_{X_{\ov{T}}})\ot\id_E}
\ar[d]^{\text{$\cong$ by \eq{dt12eq14}}}
&\pi^* J \ot E [2] \ar[d]_=
\\
\pi^*L_{\cO_T/\cO_U} \ot E[1] \ar[r]^{\pi^*e(\cO_{\ov T})\ot\id_E} &\pi^* J
\ot E[2]. }
\end{gathered}
\!\!\!\!\!\!\!\!\!\!\!\!\!\!\!\!
\label{dt12eq17}
\end{equation}
\end{small}Let $\At(E,s)$\nomenclature[At(E,s)]{$\At(E,s)$}{Atiyah class of a
family of stable pairs $s:\cO_{X_T}(-n)\ra E$} denote the
composition of all the morphisms in the left column of diagram
\eq{dt12eq17}. We call this the {\it Atiyah class\/} of the family
of pairs $s:\cO_X(-n)\ra E$ over $X_T$. Thus, we may restate
Illusie's results as:

\begin{thm} In Theorem {\rm\ref{dt12thm3},} the obstruction
morphism $ob$ factorizes as
\e
\xymatrix@C=55pt{\cone(s)\ar[r]^(0.45){\At(E,s)}
&\pi^*L_{\cO_T/\cO_U}\ot E[1]\ar[r]^(0.55){\pi^*e(\cO_{\ov
T})\ot\id_E} &\pi^*J\ot E[2]. }
\label{dt12eq18}
\e
\label{dt12thm4}
\end{thm}

Note that in \eq{dt12eq18}, $\At(E,s)$ is independent of the choice
of $J,\ov{T}$, and $\pi^*e(\cO_{\ov T})$ depends on the square zero
extension $J,\ov{T}$ but is independent of the choice of pair $E,s$.
A very similar picture is explained by Huybrechts and Thomas
\cite{HuTh}, when they show that obstruction class $ob$ for
deforming a complex $E^\bu$ in $D^b(\coh(X))$ is the composition of
an {\it Atiyah class\/} depending on $E^\bu$, and a {\it
Kodaira--Spencer class\/}\index{Kodaira--Spencer class} depending on
the square-zero extension.\index{Atiyah class|)}

\subsection{A non-perfect obstruction theory for
$\M_\stp^{\al,n}(\tau')/U$}\index{obstruction theory|(}
\label{dt125}

Now suppose $\al\in K(\coh(X))$, and $n\gg 0$ is large enough that
$H^i(E(n))=0$ for all $i>0$ and all $\tau$-semistable sheaves $E$ of
class $\al$. We will construct the natural relative obstruction
theory $\phi:B^\bu\ra L_{\M_\stp^{\al,n}(\tau')/U}$ for
$\M_\stp^{\al,n}(\tau')$, which unfortunately is neither perfect nor
symmetric. Sections \ref{dt126} and \ref{dt127} explain how to
modify $\phi$ to a perfect, symmetric obstruction theory, firstly in
the case $\rank{\mathbb I}\ne 0$, and then a more complicated
construction for the general case.

\begin{rem} Here is an informal sketch of what is going on in
\S\ref{dt125}--\S\ref{dt127}. Consider a single stable pair
$s:\cO_X(-n)\ra E$ on $X$, and let $I$ be the pair considered as an
object of $D(X)$. Think of $(E,s)$ as a point of the moduli scheme
$\M_\stp^{\al,n}(\tau')$. It turns out that {\it deformations\/} of
$(E,s)$ are given by $\Ext^1_{\sst D(X)}(I,E)$, so that the tangent
space $T_{(E,s)}\M_\stp^{\al,n}(\tau')\cong\Ext^1_{\sst D(X)}(I,E)$,
and the {\it obstruction space\/} to deforming $(E,s)$ is
$\Ext^2_{\sst D(X)}(I,E)$, so we may informally
write~$O_{(E,s)}\M_\stp^{\al,n}(\tau')\cong\Ext^2_{\sst D(X)}(I,E)$.

Now obstruction theories concern cotangent not tangent spaces, as
they map to the cotangent complex $L_{\M_\stp^{\al,n}(\tau')},$ so
we are interested in the dual spaces
\begin{align*}
T_{(E,s)}^*\M_\stp^{\al,n}(\tau')&\cong\Ext^1_{\sst D(X)}(I,E)^*
\cong\Ext^2_{\sst D(X)}(E,I\ot K_X)\cong H^0(B^\bu),\\
O_{(E,s)}^*\M_\stp^{\al,n}(\tau')&\cong\Ext^2_{\sst D(X)}(I,E)^*
\cong\Ext^1_{\sst D(X)}(E,I\ot K_X)\cong H^{-1}(B^\bu),
\end{align*}
using Serre duality in the second steps, and where $B^\bu$ in the
third steps is defined in \eq{dt12eq21} below, and $\om_\pi$ in
\eq{dt12eq21} plays the r\^ole of~$K_X$.

Thus, the complex $B^\bu$ in $D(X)$ encodes the (dual of the)
deformations of $s:\cO_X(-n)\ra E$ in its degree 0 cohomology, and
the (dual of the) obstructions in its degree $-1$ cohomology. This
is what we want from an obstruction theory, and we will show in
Proposition \ref{dt12prop2} that $B^\bu$ can indeed be made into an
obstruction theory for $\M_\stp^{\al,n}(\tau')$. However, there are
two problems with it. Firstly, $H^{-2}(B^\bu)$ may be nonzero, so
$B^\bu$ is not concentrated in degrees $[-1,0]$, that is, it is not
a {\it perfect\/} obstruction theory. Secondly, as $\Ext^1_{\sst
D(X)}(I,E)$ and $\Ext^2_{\sst D(X)}(I,E)$ are not dual spaces,
$B^\bu$ is not {\it symmetric}.

Here is how we fix these problems. There are natural {\it
identity\/} and {\it trace morphisms\/} $\id_I:H^i(\cO_X)\ra
\Ext^i_{\sst D(X)}(I,I)$ and $\tr_I:\Ext^i_{\sst D(X)}(I,I) \ra
H^i(\cO_X)$, with $\tr_I\ci\id_I =\rank I\cdot 1_{H^i(\cO_X)}$.
Suppose for the moment that $\rank I\ne 0$. Then $\Ext^i_{\sst
D(X)}(I,I)\cong \Ext^i_{\sst D(X)}(I,I)_0\op H^i(\cO_X)$, where
$\Ext^i_{\sst D(X)}(I,I)_0=\Ker\tr_I$ is the trace-free part of
$\Ext^i(I,I)$. We have natural morphisms
\begin{equation*}
\xymatrix@C=15pt{\Ext^i_{\sst D(X)}(I,E)
\ar[r]^{T\ci } & \Ext^i_{\sst D(X)}(I,I) \ar[r] &
\Ext^i_{\sst D(X)}(I,I)/H^i(\cO_X)\cong
\Ext^i_{\sst D(X)}(I,I)_0,}
\end{equation*}
where $T:E\ra I$ is the natural morphism in $D(X)$. Now
$\Ext^1_{\sst D(X)}(I,E)\ra\Ext^1_{\sst D(X)}(I,I)_0$ is an
isomorphism, and $\Ext^2_{\sst D(X)}(I,E)\ra\Ext^2_{\sst
D(X)}(I,I)_0$ is injective. The idea of \S\ref{dt126} is to replace
$\Ext^\bu_{\sst D(X)}(I,E)$ by $\Ext^\bu_{\sst D(X)}(I,I)_0$. We
construct a complex $G^\bu$ with $H^i(G^\bu)\cong \Ext^{1-i}_{\sst
D(X)}(I,I)_0^*$, and this is our symmetric perfect obstruction
theory for~$\M_\stp^{\al,n}(\tau')$.

If $\rank I=0$ then this construction fails, since
$\tr_I\ci\id_I=0$, and we no longer have a canonical splitting
$\Ext^i_{\sst D(X)}(I,I)\cong \Ext^i_{\sst D(X)}(I,I)_0\op
H^i(\cO_X)$. We deal with this in \S\ref{dt127} in a peculiar way.
The basic idea is to replace $\Ext^i_{\sst D(X)}(I,E)$ by
$\Ext^i_{\sst D(X)}(I,I)/\id_I\bigl(H^i(\cO_X)\bigr)$ when $i=0,1$,
and by $\Ker\bigl(\tr I:\Ext^i_{\sst D(X)}(I,I)\ra H^i(\cO_X)\bigr)$
when $i=2,3$. This yields groups which are zero in degrees 0 and 3
and dual in degrees 1 and 2, so again they give us a symmetric
perfect obstruction theory for~$\M_\stp^{\al,n}(\tau')$.
\label{dt12rem}
\end{rem}

As in \S\ref{dt122}, write $X_\M= X\times_U\M_\stp^{\al,n}(\tau')$
with projection $\pi:X_\M\ra \M_\stp^{\al,n}(\tau')$, and write
${\mathbb S}:\cO_{X_\M}(-n)\ra{\mathbb E}$ for the universal stable
pair on $X_\M$. We will also regard this as an object ${\mathbb
I}=\cone({\mathbb S})$ in $D(X_{\M})$ with $\cO_{X_\M}(-n)$ in
degree $-1$ and $\mathbb E$ in degree 0. Thus we have a
distinguished triangle in~$D(X_\M)$:
\e
\xymatrix{{\mathbb I}[-1] \ar[r] &
\cO_{\smash{X_\M}}(-n)\ar[r]^(0.6){\mathbb S} & {\mathbb E}
\ar[r]^{\mathbb T} & {\mathbb I}.}
\label{dt12eq19}
\e

Define objects in
$D(\M_\stp^{\al,n}(\tau'))$:\nomenclature[A,B,C]{$A^\bu,B^\bu,\ldots$}{complexes
in the derived category $D(\M_\stp^{\al,n}(\tau'))$}
\ea
A^\bu&=R\pi_*\bigl(R\SHom({\mathbb I},{\mathbb
I})\ot\om_{\pi}\bigr)[2],
\label{dt12eq20}\\
B^\bu&=R\pi_*\bigl(R\SHom({\mathbb E},{\mathbb
I})\ot\om_{\pi}\bigr)[2],
\label{dt12eq21}\\
\check B^\bu&=R\pi_*\bigl(R\SHom({\mathbb I},{\mathbb
E})\ot\om_{\pi}\bigr)[2],
\label{dt12eq22}\\
C^\bu&=R\pi_*\bigl(R\SHom(\cO_{X_\M}(-n),{\mathbb I})
\ot\om_\pi\bigr)[2],
\label{dt12eq23}\\
\check C^\bu&=R\pi_*\bigl(R\SHom({\mathbb I},\cO_{X_\M}(-n))
\ot\om_\pi\bigr)[2],
\label{dt12eq24}\\
D^\bu&=R\pi_*\bigl(R\SHom({\mathbb E},{\mathbb E})
\ot\om_\pi\bigr)[2],
\label{dt12eq25}\\
E^\bu&=R\pi_*\bigl(R\SHom(\cO_{X_\M}(-n),{\mathbb E})
\ot\om_\pi\bigr)[2],
\label{dt12eq26}\\
\check E^\bu&=R\pi_*\bigl(R\SHom({\mathbb E},\cO_{X_\M}(-n))
\ot\om_\pi\bigr)[2],
\label{dt12eq27}\\
F^\bu&=R\pi_*(\om_{\pi})[2]\cong R\pi_*\bigl(R\SHom(\cO_{X_\M}(-n),
\cO_{X_\M}(-n))\ot\om_\pi\bigr)[2].
\label{dt12eq28}
\ea

Applying $R\pi_*\bigl(R\SHom(\text{---},*)\ot\om_\pi\bigr)[2]$,
$R\pi_*\bigl(R\SHom(*,\text{---}) \ot\om_\pi\bigr)[2]$ to
\eq{dt12eq19} for $*={\mathbb I},{\mathbb E},\cO_{X_\M}$ gives six
distinguished triangles in $D(\M_\stp^{\al,n}(\tau'))$, which we
write as a commutative diagram with rows and columns distinguished
triangles:
\e
\begin{gathered}
\xymatrix@C=30pt@R=12pt{A^\bu \ar[r]^\be & B^\bu \ar[r]^\ga & C^\bu
\ar[r]^\de & A^\bu[1] \\
\check B^\bu \ar[r]^\ep \ar[u]^{\check\be} & D^\bu \ar[r]^\ze
\ar[u]^{\check\ep} & E^\bu
\ar[r]^\eta \ar[u]^{\check\io} & \check B^\bu[1] \ar[u]^{\check\be} \\
\check C^\bu \ar[r]^\io \ar[u]^{\check\ga} & \check E^\bu \ar[r]^\ka
\ar[u]^{\check\ze} & F^\bu \ar[r]^\la \ar[u]^{\check\ka} &
\check C^\bu[1] \ar[u]^{\check\ga} \\
A^\bu[-1] \ar[r]^\be \ar[u]^{\check\de}  & B^\bu[-1] \ar[r]^\ga
\ar[u]^{\check\eta} & C^\bu[-1] \ar[r]^\de \ar[u]^{\check\la} &
A^\bu. \ar[u]^{\check\de}}
\end{gathered}
\label{dt12eq29}
\e
We take \eq{dt12eq29} to be the definition of the
morphisms~$\be,\ldots,\check\la$.

Here is the point of the notation with accents
`$\check{\phantom{m}}$': the Calabi--Yau condition on $\vp:X\ra U$
implies that the dualizing complex $\om_\vp$ is a line bundle on $X$
trivial on the fibres of $\vp$, so it is $\vp^*(L)$ for some line
bundle $L$ on $U$. Suppose $U$ is affine. Then any line bundle on
$U$ is trivial, so we may choose an isomorphism $L\cong\cO_U$, and
then $\om_\vp\cong\cO_X$, and on $X_\M$ we have
$\om_\pi\cong\pi_X^*(\om_\vp)\cong\pi_X^*(\cO_X)\cong\cO_{X_\M}$.
Using this isomorphism $\om_\pi\cong\cO_{X_\M}$ we get isomorphisms
\e
\begin{gathered}
A^{\bu\vee}[1]\!\cong\!A^\bu,\; B^{\bu\vee}[1]\!\cong\!\check
B^\bu,\; \check B^{\bu\vee}[1]\!\cong\!B^\bu,\;C^{\bu\vee}[1]
\!\cong\!\check C^\bu,\; \check C^{\bu\vee}[1]\!\cong \! C^\bu,\\
D^{\bu\vee}[1]\cong D^\bu,\;\> E^{\bu\vee}[1]\cong\check E^\bu,\;\>
\check E^{\bu\vee}[1]\cong E^\bu,\;\> F^{\bu\vee}[1]\cong F^\bu,
\end{gathered}
\label{dt12eq30}
\e
where $A^{\bu\vee},\ldots,F^{\bu\vee}$ are the {\it duals\/} of
$A^\bu,\ldots,F^\bu$, dualizing in \eq{dt12eq29} corresponds to
taking the transpose along the diagonal, and
$\check\be,\ldots,\check\la$ would be the dual morphisms of
$\be,\ldots,\la$, respectively.

\begin{lem} The complex $B^\bu$ in \eq{dt12eq21} satisfies
condition~$(*)$.
\label{dt12lem1}
\end{lem}

\begin{proof} As $B^\bu$ is obtained by applying standard
derived functors to a complex of quasi-coherent sheaves with
coherent cohomology, the general theory guarantees that $B^\bu$ is
also a complex of quasi-coherent sheaves with coherent cohomology.
Thus we only have to check that $h^i(B^\bu)=0$ for $i>0$. But the
fibre of $h^i(B^\bu)$ at a $U$-point $p$ of $\M_\stp^{\al,n}(\tau')$
corresponding to a stable pair $I_p=\cO_X(-n)\,{\buildrel
s_p\over\longra}\,E_p$ is $\Ext^{2+i}_{D(X)}(E_p,I_p\ot\om_\pi)$,
which is dual to $\Ext^{1-i}_{D(X)}(I_p,E_p)$ by Serre duality, and
so vanishes when $i>0$ by Proposition~\ref{dt12prop1}(a).
\end{proof}

As in \S\ref{dt124}, the {\it Atiyah class\/}\index{Atiyah class} of
the universal pair ${\mathbb S}:\cO_{X_M}(-n)\ra{\mathbb E}$ is
\e
\At({\mathbb E},{\mathbb S}):{\mathbb I} \longra
\pi^*(L_{\M_\stp^{\al,n}(\tau')/U}) \ot {\mathbb E}[1].
\label{dt12eq31}
\e
We have isomorphisms
\e
\begin{split}
\Ext^1&({\mathbb I}, \pi^*(L_{\M_\stp^{\al,n}(\tau')/U})
\ot {\mathbb E})\\
&\cong\Ext^1(R\SHom({\mathbb E},{\mathbb I}),
\pi^*(L_{\M_\stp^{\al,n}(\tau')/U}))\\
&\cong \Ext^1(R\SHom({\mathbb E},{\mathbb I})\!\ot\!\om_{\pi}[3],
\pi^*(L_{\M_\stp^{\al,n}(\tau')/U})\!\ot\!\om_{\pi}[3])\!\!\!\\
&\cong\Ext^1(R\SHom({\mathbb E},{\mathbb I}) \ot \om_{\pi}[3],
\pi^{!}(L_{\M_\stp^{\al,n}(\tau')/U})) \\
&\cong \Ext^1(R\pi_*(R\SHom({\mathbb E},{\mathbb I}) \ot
\om_{\pi}[3]),L_{\M_\stp^{\al,n}(\tau')/U})\\
&\cong \Hom(B^\bu, L_{\M_\stp^{\al,n}(\tau')/U}),
\end{split}
\label{dt12eq32}
\e
using $\pi^!(K^\bu)=\pi^*(K^\bu)\ot\om_{\pi}[3]$ for all $K^\bu\in
D(\M_\stp^{\al,n}(\tau'))$ as $\pi$ is smooth of dimension $3$ in
the third isomorphism, the adjoint pair $(R\pi_*,\pi^!)$ in the
fourth, and \eq{dt12eq21} in the fifth. Define
\e
\phi:B^\bu\longra L_{\M_\stp^{\al,n}(\tau')/U}
\label{dt12eq33}
\e
to correspond to $\At({\mathbb E},{\mathbb S})$ in \eq{dt12eq31}
under the isomorphisms~\eq{dt12eq32}.

\begin{prop} The morphism $\phi$ in \eq{dt12eq33} makes $B^\bu$
into a (generally not perfect) relative obstruction theory
for~$\M_\stp^{\al,n}(\tau')/U$.
\label{dt12prop2}
\end{prop}

\begin{proof} Suppose we are given an affine $U$-scheme $T$, a
square-zero extension $\ov{T}$ of $T$ with ideal sheaf $J$ as in
\S\ref{dt123}, and a morphism of $U$-schemes $g:T \ra
\M_\stp^{\al,n}(\tau')$. Set $X_T=X\times_UT$. Then we have a
Cartesian diagram:
\begin{equation*}
\xymatrix@R=17pt{ X_T \ar[r]^(0.4)f \ar[d]^{\xi} \ar[rrd]
&X_\M \ar[d]|(0.5)\hole_(0.3){\pi} \ar[rd] \\
T \ar[r]^(0.4){g} \ar[rrd] &\M_\stp^{\al,n}(\tau') \ar[rd] &X \ar[d] \\
& &U. }
\end{equation*}
The universal stable pair ${\mathbb S}:\cO_{X_\M}(-n)\ra{\mathbb E}$
on $X_\M$ pulls back under $f$ to a $T$-family ${\mathbb
S}_T:\cO_{X_T}(-n)\ra{\mathbb E}_T$ of stable pairs, which we write
as ${\mathbb I}_T$ when considered as an object of $D(X_T)$. Since
$\M_\stp^{\al,n}(\tau')$ is a fine moduli space, extensions $\ov
g:\ov T \ra \M_\stp^{\al,n}(\tau')$ of $g$ to $\ov T$ are equivalent
to extensions $\ov{\mathbb S}_T:\cO_{X_{\ov T}}(-n)\ra\ov{\mathbb
E}_T$ of ${\mathbb S}_T:\cO_{X_T}(-n)\ra{\mathbb E}_T$
$s:\cO_{X_T}(-n)\ra E$ to $\ov T$, which is exactly the problem
considered in \S\ref{dt124}. For $i\in\Z$ we have the following
isomorphisms:
\e
\begin{aligned}
&\Ext^i(g^*B^\bu, J)\cong \Ext^i(g^*(R \pi_*(R\SHom({\mathbb E},
{\mathbb I}) \ot
\om_{\pi}[2])), J) \\
&\;\cong \Ext^i(R \pi_*(R\SHom({\mathbb E}, {\mathbb I}) \ot
\om_{\pi}[2]), Rg_*J)\\
&\;\cong \Ext^i(R\SHom({\mathbb E}, {\mathbb
I}) \ot \om_{\pi}[2], \pi^*(Rg_*J) \ot
\om_{\pi}[3]) \\
&\;\cong \Ext^{i+1}(R\SHom\!({\mathbb E}, {\mathbb I}), \pi^*(Rg_*J))
\!\cong\!\Ext^{i+1}(R\SHom\!({\mathbb E}, {\mathbb I}), Rf_*(\xi^*J))
\!\!\! \\
&\;\cong\Ext^{i+1}(Lf^*(R\SHom\!({\mathbb E}, {\mathbb I})),\xi^* J)
\!\cong\! \Ext^{i+1}(R\SHom\!(Lf^*{\mathbb E}, Lf^*{\mathbb I}), \xi^* J)
\!\!\!\!\! \\
&\;\cong \Ext^{i+1}(Lf^* {\mathbb I}, \xi^* J \ot Lf^* {\mathbb E})
\cong \Ext^{i+1}(f^* {\mathbb I}, \xi^* J \ot f^* {\mathbb E}),
\end{aligned}
\label{dt12eq34}
\e
using \eq{dt12eq21} in the first step, the adjoint pair $(g^*,
Rg_*)$ in the second, the adjoint pair $(R\pi_*, \pi^!)$ and
$\pi^!(A)=\pi^*(A)\ot\om_{\pi}[3]$ in the third, base change for the
flat morphism $\pi$ in the fifth, and the adjoint pair $(Lf^*,
Rf_*)$ in the sixth. In the final step, as $\cO_{X_\M}(-n)$ and
${\mathbb E}$ are flat over $\M_\stp^{\al,n}(\tau')$ we have $Lf^*
{\mathbb E}\cong f^* {\mathbb E}$, and $Lf^*({\mathbb I})$ is
quasi-isomorphic to $\cO_{X_T}(-n) \ra {\mathbb E}_T$, which we
denote as $f^*{\mathbb I}={\mathbb I}_T$.

In a similar way to isomorphisms \eq{dt12eq32}, the composition
\begin{equation*}
\xymatrix@C=37pt{g^*(B^\bu) \ar[r]^(0.4){g^* \phi}
& g^*(L_{\M_\stp^{\al,n}(\tau')/U}) \ar[r] &L_{T/U}}
\end{equation*}
lifts to
\begin{equation*}
\At({\mathbb E}_T,{\mathbb S}_T):{\mathbb I}_T \longra
\xi^*(L_{T/U}) \ot {\mathbb E}_T[1],
\end{equation*}
the Atiyah class\index{Atiyah class} of the family ${\mathbb
S}_T:\cO_{X_T}(-n) \ra{\mathbb E}_T$. Thus the composition
\begin{equation*}
\xymatrix@C=65pt{ {\mathbb I}_T \ar[r]^(0.4){\At({\mathbb E}_T,
{\mathbb S}_T)} &\xi^*L_{T/U}\ot{\mathbb E}_T[1]
\ar[r]^{\xi^*e(\cO_{\ov T})\ot\id_{{\mathbb E}_T}}
&\xi^*J\ot{\mathbb E}_T[2] }
\end{equation*}
is the element $\phi^*(\om(g))$ under the isomorphism \eq{dt12eq34}
for $i=1$, by Theorems \ref{dt12thm3} and \ref{dt12thm4}. The
morphism $g:T\ra\M_\stp^{\al,n}(\tau')$ extends to $\ov{g}:
\ov{T}\ra\M_\stp^{\al,n}(\tau')$ if and only if the family of pairs
extend from $T$ to $\ov{T}$. Therefore, by Theorems
\ref{dt12thm2}--\ref{dt12thm4}, Lemma \ref{dt12lem1}, and equation
\eq{dt12eq34} for $i=0$ we conclude that $\phi$ is an obstruction
theory for~$\M_\stp^{\al,n}(\tau')$.
\end{proof}

As in the proof of Lemma \ref{dt12lem1}, the fibre of $h^i(B^\bu)$
at a $U$-point $p$ is $\Ext^{1-i}(I_p,\ab E_p)^*$, so that
$h^i(B^\bu)=0$ unless $i=-2,-1,0$ by Proposition \ref{dt12prop1}(a).
The first row of \eq{dt12eq1} then shows the fibre of
$h^{-2}(B^\bu)$ is $\Ext^3(E_p,E_p)^*$, which is isomorphic to
$\Hom(E_p,E_p)$ if $\om_\vp\cong\cO_X$, and so is never zero as
$\al=[E_p]\ne 0$. Thus, $B^\bu$ is perfect of perfect amplitude
contained in $[-2,0]$ but not in $[-1,0]$, and $\phi$ is not a
perfect obstruction theory.

\subsection{A perfect obstruction theory when $\rank\al\ne 1$}
\label{dt126}\index{obstruction theory!perfect|(}

We now modify $\phi$ to get a perfect obstruction theory
$\psi:G^\bu\ra L_{\M_\stp^{\al,n}(\tau')/U}$ which is symmetric when
$U$ is affine. Parts of the construction fail when $\rank\al=1$, and
we explain how to fix this in \S\ref{dt127}. The identity and trace
morphisms $\cO_{X_\M}\ra R\SHom({\mathbb I},{\mathbb I})$ and
$R\SHom({\mathbb I},{\mathbb I})\ra\cO_{X_\M}$ induce morphisms
$\id_{\mathbb I}:F^\bu\ra A^\bu$ and $\tr_{\mathbb I}:A^\bu\ra
F^\bu$ in $D(\M_\stp^{\al,n}(\tau'))$. Since $\rank{\mathbb
E}=\rank\al$ and $\rank{\mathbb I}=\rank\al-1$, as in \cite[\S
4]{HuTh} we see that
\e
\tr_{\mathbb I}\ci\id_{\mathbb I}=(\rank\al-1)\,1_{F^\bu},
\label{dt12eq35}
\e
where $1_{F^\bu}:F^\bu\ra F^\bu$ is the identity map. Because of
\eq{dt12eq35} we must treat the $\rank\al\ne 1$ and $\rank\al=1$
cases differently. For $\de,\check\de,\la,\check\la$ as in
\eq{dt12eq29} we have natural identities
\e
\check\de\ci\id_{\mathbb I}=\la:F^\bu\ra \check C^\bu[1]
\quad\text{and}\quad \tr_{\mathbb I}\ci\de=\check\la:C^\bu[-1]\ra
F^\bu.
\label{dt12eq36}
\e

Define objects $G^\bu,\check G^\bu$ in $D(\M_\stp^{\al,n}(\tau'))$
by
\e
G^\bu=\cone(\tr_{\mathbb I})[-1] \qquad\text{and}\qquad \check
G^\bu=\cone(\id_{\mathbb I}),
\label{dt12eq37}
\e
so that we have distinguished triangles:
\e
\begin{gathered}
\xymatrix@R=7pt{ G^\bu \ar[r]^\nu & A^\bu \ar[r]^{\tr_{\mathbb I}} &
F^\bu
\ar[r]^(0.45)\mu & G^\bu[1], \\
\check G^\bu[-1] \ar[r]^{\check\mu} & F^\bu \ar[r]^{\id_{\mathbb I}}
& A^\bu \ar[r]^(0.45){\check\nu} & \check G^\bu. }
\end{gathered}
\label{dt12eq38}
\e
We take \eq{dt12eq38} to be the definition of $\mu,\nu,
\check\mu,\check\nu$. Again, if we were given an isomorphism
$\om_\pi\cong\cO_{X_\M}$ then we would have isomorphisms
$G^{\bu\vee}[1]\cong \check G^\bu$ and $\check G^{\bu\vee}[1]\cong
G^\bu$ as in \eq{dt12eq30}, and $\check\mu,\check\nu$ would be the
dual morphisms of $\mu,\nu$, and dualizing would exchange the two
lines of~\eq{dt12eq38}.

Applying the {\it octahedral axiom\/}\index{octahedral
axiom}\index{triangulated category!octahedral axiom} in the triangulated
category $D(\M_\stp^{\al,n}(\tau'))$, as in Gelfand and Manin
\cite[\S IV.1]{GeMa}, gives diagrams:
\begin{equation*}
\xymatrix@!0@C=20pt@R=16pt{ C^\bu \ar[ddrr]^{[1]}_\de
\ar[dddd]_{[1]}^{\check\la} &&&& B^\bu \ar[llll]^\ga &&& C^\bu
\ar[dddd]_{[1]}^{\check\la} &&&& B^\bu \ar[llll]^\ga
\ar@{.>}[ddll]^\rho \\
&& \star &&&&&&& \circlearrowleft
\\
& \circlearrowleft & A^\bu \ar[uurr]^\be
\ar[ddll]_(0.5){\tr_{\mathbb I}} & \circlearrowleft &&&&& \star &
E^\bu \ar[uull]^{\check\io} \ar@{.>}[ddrr]^{[1]}_\si & \star
\\
&& \star &&&&&&& \circlearrowleft
\\
F^\bu \ar[rrrr]_{[1]}^{\mu} &&&& G^\bu \ar[uull]^\nu
\ar[uuuu]_{\be\ci\nu} &&& F^\bu \ar[rrrr]_{[1]}^{\mu}
\ar[uurr]^{\check\ka} &&&& G^\bu, \ar[uuuu]_{\be\ci\nu} }
\end{equation*}
where `$\star$' indicates a distinguished triangle, and
`$\circlearrowleft$' a commutative triangle, and we have used the
first row and third column of \eq{dt12eq29}, equation $\tr_{\mathbb
I}\ci\de=\check\la$ in \eq{dt12eq36}, and the first row of
\eq{dt12eq38}. Thus, the octahedral axiom\index{octahedral
axiom}\index{triangulated category!octahedral axiom} gives morphisms
$\rho,\si$ in a distinguished triangle:
\e
\xymatrix{G^\bu \ar[r]^{\be\ci\nu} & B^\bu \ar[r]^\rho & E^\bu
\ar[r]^\si & G^\bu[1]. }
\label{dt12eq39}
\e

The complex $E^\bu$ in $D(\M_\stp^{\al,n}(\tau'))$ has cohomology
$h^i(E^\bu)\in\coh(\M_\stp^{\al,n}(\tau'))$ for $i\in\Z$. At a
$U$-point $p$ of $\M_\stp^{\al,n}(\tau')$ corresponding to a stable
pair $s_p:\cO_X(-n)\ra E_p$, the fibre of $h^i(E^\bu)$ is
$\Ext^{i-2}(\cO_X(-n),E_p)$ by \eq{dt12eq26}, and this is zero for
$i\ne -2$ by choice of $n\gg 0$. Therefore we have:

\begin{lem} The cohomology sheaves\/ $h^i(E^\bu)$ satisfy
$h^i(E^\bu)=0$ for\/~$i\ne -2$.
\label{dt12lem2}
\end{lem}

Now define a morphism
\e
\psi=\phi\ci\be\ci\nu:G^\bu\longra L_{\M_\stp^{\al,n}(\tau')/U}.
\label{dt12eq40}
\e

\begin{prop} The morphism $\psi:G^\bu\longra
L_{\M_\stp^{\al,n}(\tau')/U}$ in \eq{dt12eq40} makes $G^\bu$ into a
relative obstruction theory for~$\M_\stp^{\al,n}(\tau')/U$.
\label{dt12prop3}
\end{prop}

\begin{proof} Taking cohomology sheaves of the distinguished
triangle \eq{dt12eq39} gives a long exact sequence in
$\coh(\M_\stp^{\al,n}(\tau'))$:
\e
\xymatrix{ \cdots \ar[r] & h^i(G^\bu) \ar[r]^{h^i(\be\ci\nu)} &
h^i(B^\bu) \ar[r]^{h^i(\rho)} & h^i(E^\bu) \ar[r]^{h^i(\si)} &
h^{i+1}(G^\bu) \ar[r] & \cdots.}
\label{dt12eq41}
\e
Lemma \ref{dt12lem2} then implies that $h^i(\be\ci\nu):h^i(G^\bu)\ra
h^i(B^\bu)$ is an isomorphism for $i=0$ and an epimorphism for
$i=-1$. Also $h^0(\phi)$ is an isomorphism and $h^{-1}(\phi)$ is an
epimorphism, since $\phi$ is an obstruction theory by Proposition
\ref{dt12prop2}. Thus $h^0(\psi)=h^0(\phi)\ci h^0(\be\ci\nu)$ is an
isomorphism and $h^{-1}(\psi)=h^{-1}(\phi)\ci h^{-1}(\be\ci\nu)$ is
an epimorphism. We can also see from Lemma \ref{dt12lem1} and
\eq{dt12eq41} that $G^\bu$ satisfies condition $(*)$. Hence $\psi$
is an obstruction theory by Definition~\ref{dt12def3}(c).
\end{proof}

Suppose now that $\rank\al\ne 1$, and also if $\K$ has positive
characteristic\index{field $\K$!positive characteristic} that $\cha\K$
does not divide $\rank\al-1$. This implies that the morphism
$\tr_{\mathbb I}\ci\id_{\mathbb I}$ in \eq{dt12eq35} is invertible.

\begin{lem} Suppose $\rank\al\ne 1\mod\cha\K$. Then
$\theta=\check\nu\ci\nu:G^\bu\ra\check G^\bu$ is an isomorphism.
Also\/~$A^\bu\cong F^\bu\op G^\bu\cong F^\bu\op\check G^\bu$.
\label{dt12lem3}
\end{lem}

\begin{proof} Use the octahedral axiom\index{octahedral
axiom}\index{triangulated category!octahedral axiom} \cite[\S
IV.1]{GeMa} and \eq{dt12eq35}, \eq{dt12eq38} to form diagrams:
\e
\begin{gathered}
\xymatrix@!0@C=20pt@R=16pt{ G^\bu[1] \ar[ddrr]^{[1]}_\nu
\ar[dddd]_{[1]}^{\theta} &&&& F^\bu \ar[llll]^\mu &&& G^\bu[1]
\ar[dddd]_{[1]}^{\theta} &&&& F^\bu \ar[llll]^\mu
\ar@{.>}[ddll] \\
&& \star &&&&&&& \circlearrowleft
\\
& \circlearrowleft & A^\bu \ar[uurr]^{\tr_{\mathbb I}}
\ar[ddll]_(0.5){\check\nu} & \circlearrowleft &&&&& \star & 0
\ar@{.>}[uull] \ar@{.>}[ddrr]^(0.55){[1]} & \star
\\
&& \star &&&&&&& \circlearrowleft
\\
\check G^\bu \ar[rrrr]_{[1]}^{\check\mu} &&&& F^\bu
\ar[uull]^{\id_{\mathbb I}} \ar[uuuu]_{\begin{subarray}{l}
(\rank\al-1) \\ {}\;\>  1_{F_\bu}\end{subarray}} &&& \check G^\bu
\ar[rrrr]_{[1]}^{\check\mu} \ar@{.>}[uurr] &&&& F^\bu.
\ar[uuuu]_{\begin{subarray}{l} (\rank\al-1) \\
{}\;\> 1_{F_\bu}\end{subarray}} }
\end{gathered}
\label{dt12eq42}
\e
Since $(\rank\al-1)1_{F_\bu}$ is an isomorphism, the cone on it is
zero, so the central object in the right hand square is zero. This
in turn implies that the cone on $\theta$ is zero, so $\theta$ is an
isomorphism. Also in the second square the morphisms $\mu,\check\mu$
factor through zero, so $\mu=\check\mu=0$, and as $A^\bu$ is the
cone on $\mu,\check\mu$ this gives~$A^\bu\cong F^\bu\op G^\bu\cong
F^\bu\op\check G^\bu$.
\end{proof}

The splitting $A^\bu\cong F^\bu\op G^\bu$ corresponds to
\begin{equation*}
R\pi_*\bigl(R\SHom({\mathbb I},{\mathbb I})\ot\om_{\pi}\bigr)[2]\cong
R\pi_*(\om_{\pi})[2]\op
R\pi_*\bigl(R\SHom({\mathbb I},{\mathbb
I})_0\ot\om_{\pi}\bigr)[2],
\end{equation*}
where $R\SHom({\mathbb I},{\mathbb I})_0$ are the trace-free
automorphisms of $\mathbb I$.

\begin{lem} If\/ $\rank\al\ne 1\mod\cha\K$ then\/ $h^i(G^\bu)=0$
for\/ $i\ne 0,-1$.
\label{dt12lem4}
\end{lem}

\begin{proof} Let $p$ be a $U$-point of $\M_\stp^{\al,n}(\tau')$
corresponding to a stable pair $s_p:\cO_X(-n)\ra E_p$. Specializing
the first row of \eq{dt12eq38} and \eq{dt12eq39} at $p$ to get
distinguished triangles in $D^b(U)$, taking long exact sequences in
cohomology, and using \eq{dt12eq20}--\eq{dt12eq26} to substitute for
$A^\bu_p,B^\bu_p,E^\bu_p,F^\bu_p$ gives long exact sequences:
\begin{gather*}
\cdots \ra \Ext^{i+1}(I_p,I_p) \ra H^{i+1}(\cO_X) \ra H^i(G^\bu_p)
\ra \Ext^{i+2}(I_p,I_p)\ra\cdots,\\
\cdots\!\ra\!H^i(G^\bu_p)\!\ra\!\Ext^{i+3}(E_p,I_p)\!\ra\!
\Ext^{i+2}(\cO_X(-n),E_p)\!\ra\! H^{i+1}(G^\bu_p)\!\ra\!\cdots.
\end{gather*}

In the first line, by Proposition \ref{dt12prop1}(c) we have
$\Ext^i(I_p,I_p),H^i(\cO_X)=0$ for $i<0$, $i>3$ and
$\Ext^i(I_p,I_p)\ra H^i(\cO_X)$ is a morphism $\K\ra\K$ for $i=0,3$.
The morphism $\Ext^3(I_p,I_p)\ra H^3(\cO_X)$ is dual to the identity
morphism $H^0(\cO_X)\ra\Ext^0(I_p,I_p)$, and so is an isomorphism
without conditions on $\rank\al$. The composition
$H^0(\cO_X)\ra\Ext^0(I_p,I_p)\ra H^0(\cO_X)$ is multiplication by
$\rank\al-1$ by \eq{dt12eq35}, which is nonzero by assumption, so
$\Ext^0(I_p,I_p)\ra H^0(\cO_X)$ is also an isomorphism. Hence so
$H^i(G^\bu_p)=0$ for $i<-1$ and $i>1$. In the second line,
$\Ext^i(E_p,I_p)=0$ for $i<1$ and $i>3$ by Proposition
\ref{dt12prop1}(a) and Serre duality, and $\Ext^i(\cO_X(-n),E_p)=0$
for $i\ne 0$ by choice of $n\gg 0$, so $H^i(G^\bu_p)=0$ for $i<-2$
and $i>0$. The lemma follows.
\end{proof}

Lemma \ref{dt12lem4} and the proofs of \cite[Lem.~2.10]{PaTh} or
\cite[Lem.~4.2]{HuTh} imply:

\begin{lem} If\/ $\rank\al\ne 1\mod\cha\K$ then\/ $G^\bu$ in
\eq{dt12eq37} is perfect of perfect amplitude contained in~$[-1,
0]$.
\label{dt12lem5}
\end{lem}

Putting all this together gives:

\begin{thm} If\/ $\rank\al\ne 1\mod\cha\K$ then\/ $\psi$ is a
perfect relative obstruction theory for $\M_\stp^{\al,n}(\tau')/U$.
If\/ $U$ is affine it is symmetric.\index{obstruction
theory!symmetric}\index{symmetric obstruction theory}
\label{dt12thm5}
\end{thm}

\begin{proof} By Proposition \ref{dt12prop3} $\psi$ is a relative
obstruction theory, which is perfect by Lemma \ref{dt12lem5}. When
$U$ is affine, as in \S\ref{dt125} we may choose an isomorphism
$\om_\pi\cong\cO_{X_\M}$, and this induces isomorphisms
\eq{dt12eq30} and $\check G^\bu\cong G^{\bu\vee}[1]$. Thus Lemma
\ref{dt12lem3} gives an isomorphism $\theta:G^\bu\ra
G^{\bu\vee}[1]$. Since $\theta=\check\nu\ci\nu$ and
$\check\nu=\nu^\vee$ we see that $\theta^\vee[1]=\theta$. Hence
$\psi$ is a symmetric obstruction theory.
\end{proof}

This proves Theorem \ref{dt5thm8} in the case that $\rank\al\ne
1\mod\cha\K$.

\subsection{An alternative construction for all $\rank\al$}
\label{dt127}

If either $\rank\al=1$, or $\cha\K>0$ divides $\rank\al-1$,\index{field
$\K$!positive characteristic} then $\tr_{\mathbb I}\ci\id_{\mathbb
I}=0$ in \eq{dt12eq35}. The proofs in \S\ref{dt126} then fail in two
ways:
\begin{itemize}
\setlength{\itemsep}{0pt}
\setlength{\parsep}{0pt}
\item In Lemma \ref{dt12lem3}, equation \eq{dt12eq42} must be
replaced by the diagrams:
\begin{equation*}
\xymatrix@!0@C=25pt@R=16pt{ G^\bu[1] \ar[ddrr]^{[1]}_\nu
\ar[dddd]_{[1]}^{\theta} &&&& F^\bu \ar[llll]^\mu && G^\bu[1]
\ar[dddd]_{[1]}^{\theta} &&&& F^\bu \ar[llll]^\mu
\ar@{.>}[ddll] \\
&& \star &&&&&& \circlearrowleft
\\
& \circlearrowleft & A^\bu \ar[uurr]^{\tr_{\mathbb I}}
\ar[ddll]_(0.5){\check\nu} & \circlearrowleft &&&& \star\,\,\,\,\,\,{} &
F^\bu\!\op\! F^\bu[1] \ar@{.>}[uull] \ar@{.>}[ddrr]^(0.55){[1]}
& {}\,\,\,\,\,\,\star
\\
&& \star &&&&&& \circlearrowleft
\\
\check G^\bu \ar[rrrr]_{[1]}^{\check\mu} &&&& F^\bu
\ar[uull]^{\id_{\mathbb I}} \ar[uuuu]_{0} && \check
G^\bu \ar[rrrr]_{[1]}^{\check\mu} \ar@{.>}[uurr] &&&& F^\bu.
\ar[uuuu]_{0} }
\end{equation*}
As the central object of the right hand square is no longer
zero, $\theta:G^\bu\ra\check G^\bu$ is not an isomorphism, so we
cannot show $\psi$ is symmetric for $U=\Spec\K$ in Theorem
\ref{dt12thm5}. Also we cannot conclude that $\mu=\check\mu=0$,
so we do not have~$A^\bu\cong F^\bu\op G^\bu\cong F^\bu\op\check
G^\bu$.
\item In the proof of Lemma \ref{dt12lem4} the morphism
$\Ext^0(I_p,I_p)\ra H^0(\cO_X)$ is zero. This implies that
$H^{-2}(G^\bu_p)\cong\K$, so $G^\bu$ is perfect of amplitude
contained in $[-2,0]$ rather than $[-1,0]$, and $\psi$ is not
perfect in Theorem~\ref{dt12thm5}.
\end{itemize}

The same problem occurs the construction of obstruction theories for
moduli schemes of simple complexes $\mathbb I$ in $D^b(X)$ in
Huybrechts and Thomas \cite[\S 4]{HuTh}. When $\rank{\mathbb I}\ne 0
\mod\cha\K$, so that $\tr_{\mathbb I}\ci\id_{\mathbb I}\ne 0$, they
consider complexes $\mathbb I$ with fixed determinant \cite[\S
4.2]{HuTh}, and obtain a perfect obstruction theory similar to our
$\psi$ in \S\ref{dt126}. When $\rank{\mathbb I}=0\mod\cha\K$, so
that $\tr_{\mathbb I}\ci\id_{\mathbb I}=0$, they instead consider
complexes $\mathbb I$ without fixed determinant \cite[\S 4.4]{HuTh},
and they modify their obstruction theory using truncation functors.

We now present an alternative, more complex construction of a
perfect obstruction theory for $L_{\M_\stp^{\al,n}(\tau')/U}$ which
works for all $\al$, and in fact is isomorphic to $\psi$ in
\S\ref{dt126} when $\rank\al\ne 1 \mod\cha\K$. Applying the
truncation functors $\tau^{\le -1}$, $\tau^{\ge 0}$ to $F^\bu$ gives
a distinguished triangle
\e
\xymatrix@C=30pt{\tau^{\le -1}F^\bu \ar[r]^(0.6){\tau^{\le -1}} &
F^\bu \ar[r]^{\tau^{\ge 0}} & \tau^{\ge 0}F^\bu \ar[r] & (\tau^{\le
-1}F^\bu)[1]. }
\label{dt12eq43}
\e

\begin{prop} The following composition of morphisms in
$D(\M_\stp^{\al,n}(\tau'))$ is zero:
\e
\xymatrix@C=30pt{\tau^{\le -1}F^\bu \ar[r]^(0.6){\tau^{\le-1}} &
F^\bu \ar[r]^{\id_{\mathbb I}} & A^\bu \ar[r]^{\be} & B^\bu
\ar[r]^(0.3)\phi & L_{\M_\stp^{\al,n}(\tau')/U}. }
\label{dt12eq44}
\e
\label{dt12prop4}
\end{prop}

\begin{proof} Let $u\in U(\K)$, with Calabi--Yau 3-fold $X_u$, and
let $s_u:\cO_{X_u}(-n)\ra E_u$ be a stable pair on $X_u$, written as
$I_u$ when considered as an object in $D^b(X_u)$. The {\it
determinant\/} $\det I_u$ of $I_u$ is a line bundle $L$ over $X_u$,
with first Chern class~$c_1(L)\ab=\ch_1(\al-[\cO_X(-n)])$.

Write ${\mathfrak J}^\al$ for the relative moduli $U$-stack of line
bundles $L$ over $X_u$ for $u\in U(\K)$ with first Chern class
$\ch_1(\al-[\cO_X(-n)])$. Then ${\mathfrak J}^\al$ is an Artin
$\K$-stack with a 1-morphism ${\mathfrak J}^\al\ra U$, whose fibre
${\mathfrak J}^\al_u$ over $u\in U(\K)$ is the moduli stack of line
bundles over $X_u$ with first Chern class $\ch_1(\al-[\cO_X(-n)])$.
Each $\K$-point in ${\mathfrak J}^\al_u$ has stabilizer group
$\K^\times$, since line bundles are simple sheaves, and the coarse
moduli scheme\index{coarse moduli scheme}\index{moduli
scheme!coarse} of ${\mathfrak J}^\al_u$ is the usual Jacobian of
line bundles on~$X_u$.

There is a natural 1-morphism $\Pi_{\mathfrak J}:
\M_\stp^{\al,n}(\tau')\ra{\mathfrak J}^\al$ taking
$s_u:\cO_{X_u}(-n)\ra E_u$ to $\det I_u$. Thus, from the sequence of
1-morphisms of $\K$-stacks $\M_\stp^{\al,n}(\tau')\ra {\mathfrak
J}^\al\ra U$, by \eq{dt12eq4} we get a distinguished triangle in
$D(\M_\stp^{\al,n}(\tau'))$:
\begin{equation*}
\xymatrix@C=27pt{ L\Pi_{\mathfrak J}^*(L_{{\mathfrak J}^\al/U})
\ar[r]^{\rd\Pi_{\mathfrak J}} & L_{\M_\stp^{\al,n}(\tau')/U} \ar[r] &
L_{\M_\stp^{\al,n}(\tau')/{\mathfrak J}^\al} \ar[r] &
L\Pi_{\mathfrak J}^*(L_{{\mathfrak J}^\al/U}). }\!\!\!
\end{equation*}
Now in \eq{dt12eq44}, we may think of $B^\bu$ as the obstruction
theory of pairs $s:\cO_X(-n)\ra E$, the morphism $\be:A^\bu\ra
B^\bu$ as the dual of the morphism taking a stable pair
$s_u:\cO_{X_u}(-n)\ra E_u$ to the associated complex $I_u$, and the
morphism $\id_{\mathbb I}:F^\bu\ra A^\bu$ as the dual of the
morphism taking $I_u$ to $\det I_u$ in ${\mathfrak J}^\al_u$. So
$\be\ci\id_{\mathbb I}:F^\bu\ra B^\bu$ is, on the level of
obstruction theories, the dual of $\Pi_{\mathfrak J}$ taking
$s_u:\cO_{X_u}(-n)\ra E_u$ to $\det I_u$. Therefore there exists a
morphism $\up$ in a commutative diagram:
\e
\begin{gathered}
\xymatrix@C=30pt@R=15pt{F^\bu \ar[r]_{\be\ci\id_{\mathbb I}}
\ar@{.>}[d]_(0.4)\up & B^\bu \ar[d]^(0.4)\phi \\
L_{\Pi_{\mathfrak J}}^*(L_{{\mathfrak J}^\al/U})
\ar[r]^{\rd\Pi_{\mathfrak J}} & L_{\M_\stp^{\al,n}(\tau')/U}. }
\end{gathered}
\label{dt12eq45}
\e

By assumption the numerical Grothendieck groups\index{Grothendieck
group!numerical} $K^\num(\coh(X_u))$ are all canonically isomorphic
for $u\in U(\K)$. If $\M_\stp^{\al,n}(\tau')\ne\es$, this implies
that line bundles with first Chern class $\ch_1(\al-[\cO_X(-n)])$
exist for all $u\in U(\K)$, since otherwise $K^\num(\coh(X_u))$
would depend on $u$. As $U$ is an algebraic $\K$-variety, and so
reduced, it follows that deformations of line bundles with first
Chern class $\ch_1(\al-[\cO_X(-n)])$ on $X_u$ are unobstructed in
the family of Calabi--Yau 3-folds $\vp:X\ra U$. Therefore
${\mathfrak J}^\al\ra U$ is a {\it smooth\/} 1-morphism of Artin
$\K$-stacks. It is clear that ${\mathfrak J}^\al_u$ is a smooth
Artin $\K$-stack for each $u\in U(\K)$, since Jacobians are smooth
abelian varieties. We are here saying a bit more, that the
1-morphism ${\mathfrak J}^\al\ra U$ is smooth.

For a smooth morphism of $\K$-schemes $\phi:X\ra Y$ the cotangent
complex $L_{X/Y}$ is concentrated in degree 0. But for the smooth
1-morphism ${\mathfrak J}^\al\ra U$, as ${\mathfrak J}^\al$ is an
Artin $\K$-stack rather than a $\K$-scheme, the cotangent complex
$L_{{\mathfrak J}^\al/U}$ is concentrated in degrees 0 and 1. It
follows that the morphism $\up\ci\tau^{\le -1}:\tau^{\le -1}F^\bu\ra
L_{\Pi_{\mathfrak J}}^*(L_{{\mathfrak J}^\al/U})$ is zero, since
$\tau^{\le -1}F^\bu$ lives in degree $\le -1$ and $L_{\Pi_{\mathfrak
J}}^*(L_{{\mathfrak J}^\al/U})$ in degree $\ge 0$. Hence in
\eq{dt12eq44} we have $\phi\ci\be\ci\id_{\mathbb I}\ci\tau^{\le
-1}=\rd\Pi_{\mathfrak J}\ci \up\ci\tau^{\le -1}=0$, by
\eq{dt12eq45}, which proves the proposition.

In fact in \eq{dt12eq45} we have an isomorphism $L_{\Pi_{\mathfrak
J}}^* (L_{{\mathfrak J}^\al/U})\cong \tau^{\ge 0}F^\bu$, which
identifies $\up$ with the projection $\tau^{\ge 0}:F^\bu\ra\tau^{\ge
0}F^\bu$ in~\eq{dt12eq43}.
\end{proof}

Set $H^\bu\!=\!\cone(\tau^{\ge 0}\ci\tr_{\mathbb I})[-1],\check
H^\bu\!=\!\cone(\id_{\mathbb I}\ci\tau^{\le -1})$, giving triangles
\e
\begin{gathered}
\xymatrix@C=35pt@R=4pt{H^\bu \ar[r]^a & A^\bu
\ar[r]^(0.45){\tau^{\ge 0}\ci\tr_{\mathbb I}} & \tau^{\ge 0}F^\bu
\ar[r]^b & H^\bu[1],\\
\check H^\bu[-1] \ar[r]^{\check b} & \tau^{\le -1}F^\bu
\ar[r]^(0.55){\id_{\mathbb I}\ci\tau^{\le -1}} & A^\bu
\ar[r]^{\check a} & \check H^\bu.}
\end{gathered}
\label{dt12eq46}
\e
We have $(\tau^{\ge 0}\ci\tr_{\mathbb I})\ci(\id_{\mathbb
I}\ci\tau^{\le -1})=(\rank\al-1)\tau^{\ge 0}\ci\tau^{\le -1}=0$ by
\eq{dt12eq35}, so by the first line of \eq{dt12eq46} there exists
$c:\tau^{\le-1}F^\bu\ra H^\bu$ with $a\ci c=\id_{\mathbb
I}\ci\tau^{\le -1}$. Define $K^\bu=\cone(c)$, in a distinguished
triangle
\e
\xymatrix{\tau^{\le -1}F^\bu \ar[r]^(0.6){c} & H^\bu \ar[r]^{d} &
K^\bu \ar[r]^(0.35){e} & (\tau^{\le -1}F^\bu)[1]. }
\label{dt12eq47}
\e
Then by the octahedral axiom\index{octahedral axiom}\index{triangulated
category!octahedral axiom} we have diagrams
\begin{equation*}
\xymatrix@!0@C=20pt@R=16pt{ \tau^{\ge 0}F^\bu \ar[ddrr]^{[1]}_{b}
\ar[dddd]^{[1]}_{\begin{subarray}{l}d\ci b\\ =\check
e\end{subarray}} &&&& A^\bu \ar[llll]_{\tau^{\ge 0}\ci\tr_{\mathbb
I}} &&& \tau^{\ge 0}F^\bu \ar[dddd]^{[1]}_{\begin{subarray}{l}d\ci b\\
=\check e\end{subarray}} &&&& A^\bu \ar[llll]_{\tau^{\ge
0}\ci\tr_{\mathbb I}}
\ar[ddll]^{\check a} \\
&& \star &&&&&&& \circlearrowleft
\\
& \circlearrowleft & H^\bu \ar[uurr]^{a} \ar[ddll]_(0.5){d} &
\circlearrowleft &&&&& \star & \check H^\bu \ar@{.>}[uull]^{\check
c} \ar[ddrr]^{[1]}_{\check b} & \star
\\
&& \star &&&&&&& \circlearrowleft
\\
K^\bu \ar[rrrr]_{[1]}^e &&&& \tau^{\le -1}F^\bu \ar[uull]^c
\ar[uuuu]_{\id_{\mathbb I}\ci\tau^{\le -1}} &&& K^\bu
\ar[rrrr]_{[1]}^e \ar@{.>}[uurr]^{\check d} &&&& \tau^{\le -1}F^\bu.
\ar[uuuu]_{\id_{\mathbb I}\ci\tau^{\le -1}} }
\end{equation*}
Thus we get a morphism $\check c:\check H^\bu\ra\tau^{\ge 0}F^\bu$
with $\check c\ci\check a=\tau^{\ge 0}\ci\tr_{\mathbb I}$ and
$K^\bu\cong\cone(\check c)[-1]$, and a distinguished triangle with
$\check e=d\ci b$:
\e
\xymatrix{(\tau^{\ge 0}F^\bu)[-1] \ar[r]^(0.6){\check e} & K^\bu
\ar[r]^{\check d} & \check H^\bu \ar[r]^(0.4){\check c} & \tau^{\ge
0}F^\bu. }
\label{dt12eq48}
\e

In morphisms $\tau^{\le -1}F^\bu\ra L_{\M_\stp^{\al,n}(\tau')/U}$ we
have
\begin{equation*}
\phi\ci\be\ci a\ci c=\phi\ci\be\ci \id_{\mathbb I}\ci\tau^{\le -1}=0,
\end{equation*}
by $a\ci c=\id_{\mathbb I}\ci\tau^{\le-1}$ and Proposition
\ref{dt12prop4}. Thus as \eq{dt12eq47} is distinguished there exists
a morphism $\xi:K^\bu\ra L_{\M_\stp^{\al,n}(\tau')/U}$ with~$\xi\ci
d=\phi\ci\be\ci a$.

\begin{thm} This\/ $\xi:K^\bu\ra L_{\M_\stp^{\al,n}(\tau')/U}$ is a
perfect relative obstruction theory for $\M_\stp^{\al,n}(\tau')/U$.
If\/ $U$ is affine it is symmetric.\index{obstruction
theory!symmetric}\index{symmetric obstruction theory}
\label{dt12thm6}
\end{thm}

This proves Theorem \ref{dt5thm8}. We will sketch the proof of
Theorem \ref{dt12thm6} in four steps, leaving the details to the
reader:
\begin{itemize}
\setlength{\itemsep}{0pt}
\setlength{\parsep}{0pt}
\item[(a)] Show that $h^i(K^\bu)=0$ for $i\ne 0,-1$.
\item[(b)] Show that $K^\bu$ is perfect of perfect amplitude
contained in~$[-1, 0]$.
\item[(c)] When $U$ is affine, construct $\theta:K^\bu\,{\buildrel\cong
\over\longra}\,K^{\bu\vee}[1]$ with $\theta^\vee[1]=\theta$.
\item[(d)] Show that $h^0(\xi)$ is an isomorphism and $h^{-1}(\xi)$
an epimorphism.
\end{itemize}
For (a), we have $h^i(A^\bu)=h^i(F^\bu)=0$ for $i>1$ and
$h^1(A^\bu)\cong h^1(F^\bu)\cong\cO_\M$, where
$h^1(A^\bu)\cong\cO_\M$ follows from $\Ext^3(I,I)\cong\K$ in
Proposition \ref{dt12prop1}(c). Since $h^i(E^\bu)=0$ for $i\ne -2$
by Serre vanishing,\index{Serre vanishing} taking the long exact
sequence $h^*(\text{---})$ in the third column of \eq{dt12eq29}
implies that $h^i(\check\la):h^{i-1}(C^\bu)\ra h^i(F^\bu)$ is an
isomorphism for $i=0,1$. But \eq{dt12eq36} yields $h^i(\check\la)=
h^i(\tr_{\mathbb I})\ci h^i(\de)$. Therefore $h^i(\tr_{\mathbb
I}):h^i(A^\bu)\ra h^i(F^\bu)$ is surjective for $i=0,1$. As
$h^1(A^\bu)\cong h^1(F^\bu)\cong\cO_\M$ we see that
$h^1(\tr_{\mathbb I})$ is an isomorphism.

Taking the long exact sequence $h^*(\text{---})$ in the first line
of \eq{dt12eq46} and using $h^1(\tr_{\mathbb I})$ an isomorphism and
$h^0(\tr_{\mathbb I})$ surjective then gives $h^i(H^\bu)=0$ for
$i>0$. Then $h^i(\tau^{\le -1}F^\bu)=0$ for $i\ge 0$ and
\eq{dt12eq47} imply that $h^i(K^\bu)=0$ for $i>0$. Similarly, from
the third row of \eq{dt12eq29}, the equation
$h^i(\la)=h^i(\check\de)\ci h^i(\id_{\mathbb I})$ from \eq{dt12eq36}
and the second line of \eq{dt12eq46} we get $h^i(\check H^\bu)=0$
for $i<-1$, and then $h^i(K^\bu)=0$ for $i<-1$ by \eq{dt12eq48}.
Step (b) then follows as for Lemma~\ref{dt12lem5}.

For (c), if $U$ is affine then choosing an isomorphism
$\om_\pi\cong\cO_{X_\M}$ induces isomorphisms $A^{\bu\vee}[1]\cong
A^\bu$ and $F^{\bu\vee}[1]\cong F^\bu$ as in \eq{dt12eq30}. We then
have $(\tau^{\ge 0}F^\bu)^\vee[1]\cong \tau^{\le -1}F^\bu$ and
$(\tau^{\le -1}F^\bu)^\vee[1]\cong \tau^{\ge 0}F^\bu$. Under these
identifications $\tau^{\ge 0}\ci\tr_{\mathbb I}$ and $\id_{\mathbb
I}\ci\tau^{\le -1}$ are dual morphisms. Hence the two distinguished
triangles \eq{dt12eq46} are dual, and we get isomorphisms
$H^{\bu\vee}[1]\cong \check H^\bu$, $\check H^{\bu\vee}[1]\cong
H^\bu$. In \eq{dt12eq46}--\eq{dt12eq48} $\check a,\ldots,\check e$
are dual to $a,\ldots,e$, and $K^{\bu\vee}[1]\cong K^\bu$ as we
want.

For (d), as $\xi\ci d=\phi\ci\be\ci a$ we have commutative diagrams
in $\coh(\M_\stp^{\al,n}(\tau'))$
\begin{equation*}
\xymatrix@R=7pt@C=40pt{ h^i(H^\bu) \ar[d]_{h^i(d)} \ar[r]^{h^i(\be\ci a)} &
h^i(B^\bu) \ar[r]^(0.4){h^i(\phi)} & h^i(L_{\M_\stp^{\al,n}(\tau')/U}) \\
h^i(K^\bu) \ar[urr]_{h^i(\xi)}}
\end{equation*}
for each $i\in\Z$. We know $h^0(\phi)$ is an isomorphism and
$h^{-1}(\phi)$ an epimorphism by Proposition \ref{dt12prop2}. By
similar arguments to (a), we show that $h^0(\be\ci a)$ is an
isomorphism and $h^{-1}(\be\ci a)$ an epimorphism, and from the
distinguished triangle \eq{dt12eq47} and $h^i(\tau^{\le -1}F^\bu)=0$
for $i>-1$ we see that $h^0(d)$ is an isomorphism and $h^{-1}(d)$ an
epimorphism. Step (d) follows.\index{obstruction theory!perfect|)}

\subsection{Deformation-invariance of the $PI^{\al,n}(\tau')$}
\label{dt128}

We now prove Theorem \ref{dt5thm9}. In \S\ref{dt121}--\S\ref{dt127}
we assumed that the numerical Grothendieck groups\index{Grothendieck
group!numerical} $K^\num(\coh(X_u))$ for $u\in U(\K)$ are all
canonically isomorphic {\it globally in\/} $U(\K)$, and we wrote
$K(\coh(X))$ for $K^\num(\coh(X_u))$ up to canonical isomorphism. We
first prove Theorem \ref{dt5thm9} under this assumption.

As in Definition \ref{dt12def1} we have a family of Calabi--Yau
3-folds $X\stackrel{\vp}{\longra}U$ with $X,U$ algebraic
$\K$-varieties and $U$ connected, and a relative very ample line
bundle $\cO_X(1)$ for $X\stackrel{\vp}{\longra}U$, and we suppose
$K^\num(\coh(X_u))$ for $u\in U(\K)$ are all globally canonically
isomorphic to $K(\coh(X))$. Then for $\al\in K(\coh(X))$, as in
\S\ref{dt125} we choose $n\gg 0$ large enough that $H^i(E_u(n))=0$
for all $i>0$ and all $\tau$-semistable sheaves $E_u$ on $X_u$ of
class $\al\in K^\num(\coh(X_u))$ for any $u\in U(\K)$.

Then \S\ref{dt121} constructs a projective $U$-scheme
$\M_\stp^{\al,n}(\tau')$, and \S\ref{dt125}--\S\ref{dt127} construct
perfect relative obstruction theories $\psi:G^\bu\ra
L_{\M_\stp^{\al,n}(\tau')/U}$ when $\rank\al\ne 1\mod\cha\K$ and
$\xi:K^\bu\ra L_{\M_\stp^{\al,n}(\tau')/U}$ for all $\al$. For each
$u\in U(\K)$ the fibre of $\M_\stp^{\al,n}(\tau')\ra U$ at $u$ is a
projective $\K$-scheme $\M_\stp^{\al,n}(\tau')_u$, and $\psi,\xi$
specialize to perfect obstruction theories $\psi_u,\xi_u$ for
$\M_\stp^{\al,n}(\tau')_u$, which are symmetric by the case
$U=\Spec\K$ in Theorems \ref{dt12thm5} and~\ref{dt12thm6}.

Using these perfect obstruction theories, Behrend and Fantechi
\cite{BeFa1} construct virtual classes\index{virtual class}
$[\M_\stp^{\al,n}(\tau')]^\vir$ in the relative Chow homology\index{Chow
homology} $A_0(\M_\stp^{\al,n}(\tau')\ra U)$, and
$[\M_\stp^{\al,n}(\tau')_u]^\vir$ in the absolute Chow homology
$A_0(\M_\stp^{\al,n}(\tau')_u)$. In \eq{dt5eq15} we define
$PI^{\al,n}(\tau')_u=\int_{[\M_\stp^{\al,n}(\tau')_u]^\vir}1$. Since
$\psi_u,\xi_u$ are the specializations of $\psi,\xi$ at $u\in U(\K)$
we have
\begin{equation*}
[\M_\stp^{\al,n}(\tau')_u]^\vir=u^*\bigl([\M_\stp^{\al,n}(\tau')]^\vir
\bigr).
\end{equation*}
By `conservation of number', as in \cite[Prop.~10.2]{Fult} for
instance, $[\M_\stp^{\al,n}(\tau')_u]^\vir$ has the same degree for
all $u\in U(\K)$, as $U$ is connected, which proves:

\begin{thm} Let\/ $\K$ be an algebraically closed field, $U$ a
connected algebraic $\K$-variety, $X\ra U$ a family of Calabi--Yau
$3$-folds $X_u$ over $\K,$ which may have $H^1(\cO_{X_u})\ne 0,$ and
$\cO_X(1)$ a relative very ample line on $X$. Suppose
$K^\num(\coh(X_u))$ is globally canonically isomorphic to
$K(\coh(X))$ for\/ $u\in U(\K)$. Then for all\/ $\al\in K(\coh(X))$
and\/ $n\gg 0$ the invariants\/ $PI^{\al,n}(\tau')_{u}$ of stable
pairs on each fibre $X_u$ of\/ $X\ra U,$ computed using the ample
line bundle\/ $\cO_{X_u}(1)$ on\/ $X_u,$ are independent of\/~$u\in
U(\K)$.
\label{dt12thm7}
\end{thm}

This is the first part of Theorem \ref{dt5thm9}. If instead the
$K^\num(\coh(X_u))$ are only canonically isomorphic {\it locally
in\/} $U(\K)$, then by Theorem \ref{dt4thm8} we can pass to a finite
\'etale cover $\pi:\ti U\ra U$, such that the induced family of
Calabi--Yau 3-folds $\ti\vp:\ti X\ra\ti U$ has $K^\num(\coh(\ti
X_{\ti u}))$ globally canonically isomorphic for $\ti u\in\ti
U(\K)$, where $\ti u$ is of the form $(u,\io)$ for $u\in U(\K)$ and
$\io:K^\num(\coh(X_u))\,{\buildrel\cong\over\longra}\,K(\coh(X))$,
and $\ti X_{\ti u}=X_u$. Theorem \ref{dt12thm7} for this family
shows that $PI^{\al,n}(\tau')_{(u,\io)}$ is independent of
$(u,\io)\in\ti U(\K)$. But $PI^{\al,n}(\tau')_{(u,\io)}=
PI^{\al,n}(\tau')_u$ as $\ti X_{\ti u}=X_u$, and the proof of
Theorem \ref{dt5thm9} is complete.\index{obstruction theory|)}

\section[The proof of Theorem $\text{\ref{dt5thm10}}$]{The proof of
Theorem \ref{dt5thm10}}
\label{dt13}

In this section we will prove Theorem \ref{dt5thm10}, which says
that the invariants $PI^{\al,n}(\tau')$ counting stable pairs,
defined in \S\ref{dt54}, can be written in terms of the generalized
Donaldson--Thomas invariants $\bar{DT}{}^\be(\tau)$ in \S\ref{dt53}
by
\e
PI^{\al,n}(\tau')=\!\!\!\!\!\!\!\!\!\!\!\!\!\!\!
\sum_{\begin{subarray}{l} \al_1,\ldots,\al_l\in
C(\coh(X)),\\ l\ge 1:\; \al_1
+\cdots+\al_l=\al,\\
\tau(\al_i)=\tau(\al),\text{ all $i$}
\end{subarray} \!\!\!\!\!\!\!\!\! }
\begin{aligned}[t] \frac{(-1)^l}{l!} &\prod_{i=1}^{l}\bigl[
(-1)^{\bar\chi([\cO_X(-n)]-\al_1-\cdots-\al_{i-1},\al_i)} \\
&\bar\chi\bigl([\cO_{X}(-n)]\!-\!\al_1\!-\!\cdots\!-\!\al_{i-1},\al_i\bigr)
\bar{DT}{}^{\al_i}(\tau)\bigr],\!\!\!\!\!\!\!\!\!\!\!\!\!\!
\end{aligned}
\label{dt13eq1}
\e
for $n\gg 0$. As the $PI^{\al,n}(\tau')$ are deformation-invariant
by Theorem \ref{dt12thm7}, it follows by induction in Corollary
\ref{dt5cor4} that the $\bar{DT}{}^\al(\tau)$ are
deformation-invariant. Equation \eq{dt13eq1} is also useful for
computing the $\bar{DT}{}^\al(\tau)$ in examples.

\subsection{Auxiliary abelian categories $\A_p,\B_p$}
\label{dt131}

In order to relate the invariants of stable pairs and the
generalized Donaldson-Thomas invariants, we will introduce auxiliary
abelian categories $\A_p,\B_p$ and apply wall-crossing formulae in
$\B_p$ to obtain equation~\eq{dt13eq1}.

\begin{dfn} We continue to use the notation of
\S\ref{dt3}--\S\ref{dt5}, so that $X$ is a Calabi--Yau 3-fold with
ample line bundle $\cO_X(1)$, $\tau$ is Gieseker
stability\index{Gieseker stability} on the abelian category
$\coh(X)$ of coherent sheaves on $X$, and so on.

Fix some nonzero $\al\in K(\coh(X))$ with $\M_\rss^\al(\tau)\ne 0$,
for which we will prove \eq{dt13eq1}. Then $\al$ has Hilbert
polynomial $P_\al(t)$ with leading coefficient $r_\al$. Write
$p(t)=P_\al(t)/r_\al$ for the reduced Hilbert
polynomial\index{Hilbert polynomial} of $\al$. Let $d=\dim\al$. Then
$d=1,2$ or 3, and $p(t)=t^d+a_{d-1} t^{d-1}+\cdots+a_0$,
for~$a_0,\ldots,a_{d-1}\in\Q$.

Define $\A_p$\nomenclature[Ap]{$\A_p$}{an abelian subcategory of
$\tau$-semistable sheaves in $\coh(X)$} to be the subcategory of
$\coh(X)$ whose objects are zero sheaves and nonzero
$\tau$-semistable sheaves $E\in\coh(X)$ with $\tau([E])=p$, that is,
$E$ has reduced Hilbert polynomial $p$, and such that
$\Hom_{\A_p}(E,F)=\Hom(E,F)$ for all $E,F\in\A_p$. Then $\A_p$ is a
full and faithful abelian subcategory\index{abelian category}
of~$\coh(X)$.

If $E\in\A_p$ then the Hilbert polynomial $P_E$ of $E$ is a rational
multiple of $p(t)$. Since $P_E:\Z\ra\Z$ and $P_E(l)\ge 0$ for $l\gg
0$, we see that $P_E(t)\equiv\frac{k}{d!}p(t)$ for some
$k\in\Z_{\sst\ge 0}$. Let $P_\al(t)=\frac{N}{d!}p(t)$ for some
$N>0$. It will turn out that to prove \eq{dt13eq1}, we need only
consider sheaves $E\in\A_p$ with $P_E(t)\equiv\frac{k}{d!}p(t)$ for
$k=0,1,\ldots,N$, that is, we need consider only $\tau$-semistable
sheaves with {\it finitely many different Hilbert polynomials}.

By Huybrechts and Lehn \cite[Th.~3.37]{HuLe2}, the family of
$\tau$-semistable sheaves $E$ on $X$ with a fixed Hilbert polynomial
is bounded, so the family of $\tau$-semistable sheaves $E$ on $X$
with Hilbert polynomial $P_E(t)\equiv\frac{k}{d!}p(t)$ for any
$k=0,1,\ldots,N$ is also bounded. Hence by Serre
vanishing\index{Serre vanishing} \cite[Lem.~1.7.6]{HuLe2} we can
choose $n\gg 0$ such that every $\tau$-semistable sheaf $E$ on $X$
with Hilbert polynomial $P_E(t)\equiv\frac{k}{d!}\,p(t)$ for some
$k=0,1,\ldots,N$ has $H^i(E(n))=0$ for all $i>0$. That is,
$\Ext^i\bigl(\cO_X(-n), E\bigr)=0$ for $i>0$, so equation
\eq{dt3eq1} implies that
\e
\dim\Hom\bigl(\cO_X(-n),E\bigr)=\ts\frac{k}{d!}\,p(n)=
\bar\chi\bigl([\cO_X(-n)],[E]\bigr).
\label{dt13eq2}
\e
We use this $n$ to define $\M_\stp^{\al,n}(\tau')$ and
$PI^{\al,n}(\tau')$ in \S\ref{dt54}, and $\B_p$ below.

Now define a category $\B_p$\nomenclature[Bp]{$\B_p$}{abelian
category of coherent sheaves extended by vector spaces} to have
objects triples $(E,V,s)$, where $E$ lies in $\A_p$, $V$ is a
finite-dimensional $\C$-vector space, and
$s:V\ra\Hom\bigl(\cO_X(-n),E\bigr)$ is a $\C$-linear map. Given
objects $(E,V,s),(E',V',s')$ in $\B_p$, define morphisms
$(f,g):(E,V,s)\ra(E',V',s')$ in $\B_p$ to be pairs $(f,g)$, where
$f:E\ra E'$ is a morphism in $\A_p$ and $g:V\ra V'$ is a $\C$-linear
map, such that the following diagram commutes:
\begin{equation*}
\xymatrix@R=10pt@C=50pt{V \ar[r]^(0.4){s} \ar[d]^{{}\,g} &
\Hom\bigl(\cO_X(-n),E\bigr) \ar[d]^{{}\,f \ci} \\
V' \ar[r]^(0.4){s'}        & \Hom\bigl(\cO_X(-n),E'\bigr), }
\end{equation*}
where `$f\ci$' maps $t\mapsto f\ci t$.

Define $K(\A_p)$ to be the image of $K_0(\A_p)$ in
$K(\coh(X))=K^\num(\coh(X))$. Then each $E\in\A_p\subset\coh(X)$ has
numerical class $[E]\in K(\A_p)\subset K(\coh(X))$. Define
$K(\B_p)=K(\A_p)\op\Z$, and for $(E,V,s)$ in $\B_p$ define the
numerical class $[(E,V,s)]$ in $K(\B_p)$ to be $([E],\dim V)$.
\label{dt13def1}
\end{dfn}

For coherent sheaves, the auxiliary category $\B_p$ is a
generalization of the coherent systems introduced by Le Potier
\cite{LePo}. A version of the category $\B_p$ for representations of
quivers was discussed in \S\ref{dt74}. It is now straightforward
using the methods of \cite{Joyc3} to prove:

\begin{lem} The category\/ $\B_p$ is abelian and $\B_p,K(\B_p)$
satisfy Assumption {\rm\ref{dt3ass}} over\/ $\K=\C$. Also $\B_p$ is
noetherian\index{noetherian} and artinian,\index{artinian} and the moduli
stacks $\fM_{\B_p}^{\smash{(\be,d)}}$ are of \begin{bfseries}finite
type\end{bfseries} for all\/~$(\be,d)\in C(\B_p)$.
\label{dt13lem1}
\end{lem}

Here $\fM_{\B_p}^{\smash{(\be,d)}}$ is of finite type as it is built
out of $\tau$-semistable sheaves $E$ in class $\be$ in $K(\coh(X))$,
which form a bounded family by \cite[Th.~3.37]{HuLe2}. Lemma
\ref{dt13lem1} means that we can apply the results of
\cite{Joyc3,Joyc4,Joyc5,Joyc6} to $\B_p$. Note that $\A_p$ embeds as
a full and faithful subcategory in $\B_p$ by $E\mapsto (E,0,0)$.
Every object $(E,V,s)$ in $\B_p$ fits into a short exact sequence
\e
\xymatrix{0 \ar[r] & (E,0,0) \ar[r] & (E,V,s) \ar[r] & (0,V,0)
\ar[r] & 0}
\label{dt13eq3}
\e
in $\B_p$, and $(0,V,0)$ is isomorphic to the direct sum of $\dim V$
copies of the object $(0,\C,0)$ in $\B_p$. Thus, regarding $\A_p$ as
a subcategory of $\B_p$, we see that $\B_p$ is generated over
extensions by $\A_p$ and one extra object~$(0,\C,0)$.

By considering short exact sequences \eq{dt13eq3} with $V=\C$ we see
that
\e
\begin{split}
\Ext^1_{\B_p}\bigl((0,\C,0),(E,0,0)\bigr)&=H^0(E(n))\cong
\Hom\bigl(\cO_X(-n),E\bigr)\\
&\cong\Ext^1_{D(X)}\bigl(\cO_X(-n)[-1],E\bigr),
\end{split}
\label{dt13eq4}
\e
where $\cO_X(-n)[-1]$ is the shift of the sheaf $\cO_X(-n)$ in the
derived category $D(X)$. Thus the extra element $(0,\C,0)$ in $\B_p$
behaves like $\cO_X(-n)[-1]$ in $D(X)$. In fact there is a natural
embedding functor $F:\B_p\ra D(X)$ which takes $(E,V,s)$ in $\B_p$
to the complex $\cdots\ra 0\ra V\ot\cO_X(-n)\,{\buildrel
s\over\longra}\, E\ra 0\ra \cdots$ in $D(X)$, where
$V\ot\cO_X(-n),E$ appear in positions $-1,0$ respectively. Then $F$
takes $\A_p$ to $\A_p\subset\coh(X)\subset D(X)$, and $(0,\C,0)$ to
$\cO_X(-n)[-1]$ in~$D(X)$.

Therefore we can think of $\B_p$ as the abelian subcategory of
$D(X)$ generated by $\A_p$ and $\cO_X(-n)[-1]$. But working in the
derived category would lead to complications about forming moduli
stacks of objects in $D(X)$, classifying objects up to
quasi-isomorphism, and so on, so we prefer just to use the explicit
description of $\B_p$ in Definition~\ref{dt13def1}.

Although $D(X)$ is a 3-Calabi--Yau triangulated
category,\index{triangulated category!3-Calabi--Yau} and $\B_p$ is
embedded in $D(X)$, it does not follow that $\B_p$ is a
3-Calabi--Yau abelian category,\index{abelian category!3-Calabi--Yau}
and we do not claim this. In \S\ref{dt32} we defined the Euler form
$\bar\chi$ of $\coh(X)$, and used the Calabi--Yau 3-fold property to
prove \eq{dt3eq14}, which was the crucial equation in proving the
wall-crossing formulae \eq{dt3eq27}, \eq{dt5eq14} for the invariants
$J^\al(\tau), \bar{DT}{}^\al(\tau)$. We will show that even though
$\B_p$ may not be a 3-Calabi--Yau abelian category, a weakened
version of \eq{dt3eq14} still holds in $\B_p$, which will be enough
to prove wall-crossing formulae for invariants in~$\B_p$.

\begin{dfn} Define $\bar\chi{}^{\smash{\B_p}}:K(\B_p)\times K(\B_p)\ra\Z$ by
\e
\begin{split}
\bar\chi{}^{\smash{\B_p}}\bigl((\be,d),(\ga,e)\bigr)&=
\bar\chi\bigl(\be-d[\cO_X(-n)],\ga-e[\cO_X(-n)]\bigr)\\
&=\bar\chi(\be,\ga)-d\bar\chi\bigl([\cO_X(-n)],\ga\bigr)
+e\bar\chi\bigl([\cO_X(-n)],\be\bigr).
\end{split}
\label{dt13eq5}
\e
This is the natural Euler form\index{Euler form} on $K(\B_p)$ induced by
the functor $F:\B_p\ra D(X)$, since $K^\num(D(X))=K^\num(\coh(X))$,
and
\begin{align*}
\bigl[F(E,V,s)\bigr]=\bigl[V\ot\cO_X(-n)\,{\buildrel s\over\longra}\,
E\bigr]&=\dim V\bigl[\cO_X(-n)[-1]\bigr]+[E]\\
&=[E]-\dim V\bigl[\cO_X(-n)\bigr]
\end{align*}
in $K^\num(D(X))$, and $\coh(X),D(X)$ have the same Euler
form~$\bar\chi$.
\label{dt13def2}
\end{dfn}

\begin{prop} Suppose $(E,V,s),(F,W,t)$ lie in $\B_p$ with\/ $\dim
V\!+\!\dim W\!\le\!1$ and\/ $P_E(t)\equiv\frac{k}{d!}\,p(t),$
$P_F(t)\equiv\frac{l}{d!}\,p(t)$ for some $k,l=0,1,\ldots,N$. Then
\e
\begin{split}
\bar\chi{}^{\smash{\B_p}}&\bigl([(E,V,s)],[(F,W,t)]\bigr)=\\
&\bigl(\dim\Hom_{\B_p}\bigl((E,V,s),(F,W,t)\bigr)
-\dim\Ext^1_{\B_p}\bigl((E,V,s),(F,W,t)\bigr)\bigr)-\\
&\bigl(\dim\Hom_{\B_p}\bigl((F,W,t),(E,V,s)\bigr)
-\dim\Ext^1_{\B_p}((F,W,t),(E,V,s)\bigr)\bigr).
\end{split}
\label{dt13eq6}
\e
\label{dt13prop1}
\end{prop}

\begin{proof} The possibilities for $(\dim V,\dim W)$ are $(0,0),
(1,0)$ or $(0,1)$. For $(0,0)$ we have $V=W=s=t=0$, and then
$\bar\chi{}^{\smash{\B_p}}\bigl([(E,0,0)],[(F,0,0)]\bigr)=
\bar\chi\bigl([E],[F]\bigr)$, $\Hom_{\B_p}\bigl((E,0,0),
(F,0,0)\bigr)=\Hom(E,F)$, and so on, so \eq{dt13eq6} follows from
\eq{dt3eq14}. The cases $(1,0),(0,1)$ are equivalent after
exchanging $(E,V,s),(F,W,t)$, so it is enough to do the $(0,1)$
case. Thus we must verify \eq{dt13eq6} for $(E,0,0)$ and $(F,\C,t)$.

By Definition \ref{dt13def1}, $\Hom_{\B_p}\bigl(
(E,0,0),(F,\C,t)\bigr)$ is the vector space of $(f,0)$ for
$f\in\Hom(E,F)$ such that the following diagram commutes:
\begin{equation*}
\xymatrix@R=10pt@C=60pt{0 \ar[r] \ar[d] & \cO_X(-n)\ar[d]_t \\
E \ar[r]^f &F. }
\end{equation*}
This is no restriction on $f$, so
\e
\Hom_{\B_p}\bigl((E,0,0),(F,\C,t)\bigr)\cong \Hom(E,F).
\label{dt13eq7}
\e
Also $\Ext^1_{\B_p}\bigl((E,0,0),(F,\C,t)\bigr)$ corresponds to the
set of isomorphism classes of commutative diagrams with exact rows:
\e
\begin{gathered}
\xymatrix@R=13pt@C=17pt{ 0 \ar[r] & \C\ot \cO_X(-n)
\ar[rr]_(0.53){g\ot\id_{\cO_X(-n)}} \ar[d]^(0.45)t && Y\ot \cO_X(-n)
\ar[rr]
\ar[d]^(0.45)u && 0 \ar[r] \ar[d] & 0 \\
0 \ar[r] & F \ar[rr]^f && G \ar[rr]^{f'} && E \ar[r] & 0.}
\end{gathered}
\label{dt13eq8}
\e
Here $Y$ is a $\C$-vector space, $g:\C\ra Y$ is linear, $G\in\A_p$,
and $f,f',u$ are morphisms are in $\coh(X)$. By exactness of the top
row, $g$ is an isomorphism, so we can identify $Y=\C$ and
$g=\id_\C$. Then for any exact $0\ra F\,{\buildrel
f\over\longra}\,G\,{\buildrel f'\over\longra}\,E\ra 0$ in $\A_p$ we
define $u=f\ci t$ to complete \eq{dt13eq8}. Hence diagrams
\eq{dt13eq8} correspond up to isomorphisms with exact $0\ra F\ra
G\ra E\ra 0$ in $\A_p$, giving
\e
\Ext^1_{\B_p}\bigl((E,0,0),(F,\C,t)\bigr)\cong \Ext^1(E,F).
\label{dt13eq9}
\e

Similarly, $\Hom_{\B_p}\bigl((F,\C,t),(E,0,0)\bigr)$ is the set of
$(f,0)$ for $f\in\Hom(F,E)$ such that the following diagram
commutes:
\begin{equation*}
\xymatrix@R=10pt@C=60pt{ \cO_X(-n)\ar[d]^t \ar[r] & 0 \ar[d] \\
F \ar[r]^f &E. }
\end{equation*}
That is, we need $f\ci t=0$. So
\e
\begin{split}
\Hom_{\B_p}&\bigl((F,\C,t),(E,0,0)\bigr)\\
&\cong \Ker\bigl( \Hom(F,E)\,{\buildrel \ci t \over\longra}\,
\Hom(\cO_X(-n),E)\bigr).
\end{split}
\label{dt13eq10}
\e
And $\Ext^1_{\B_p}\bigl((F,\C,t),(E,0,0)\bigr)$ corresponds to the
set of isomorphism classes of commutative diagrams with exact rows:
\e
\begin{gathered}
\xymatrix@R=10pt@C=15pt{ 0 \ar[r] & 0 \ar[rr] \ar[d] && Y\ot
\cO_X(-n) \ar[rr]_{g\ot\id_{\cO_X(-n)}} \ar[d]^u && \C\ot \cO_X(-n)
\ar[r] \ar[d]_t & 0 \\
0 \ar[r] & E \ar[rr]^f && G \ar[rr]^{f'} && F \ar[r] & 0.}
\end{gathered}
\label{dt13eq11}
\e
Again, we identify $Y=\C$ and $g=\id_\C$. Then for a given exact
sequence $0\ra E\,{\buildrel f\over\longra}\,G\,{\buildrel
f'\over\longra}\,F\ra 0$ in $\A_p$, we want to know what are the
possibilities for $u$ to complete \eq{dt13eq11}. Applying
$\Hom\bigl(\cO_X(-n),-\bigr)$ to $0\ra E\,{\buildrel
f\over\longra}\,G\,{\buildrel f'\over\longra}\,F\ra 0$ yields an
exact sequence
\e
\begin{gathered}
\xymatrix@R=5pt{0 \ar[r] & \Hom\bigl(\cO_X(-n),E\bigr) \ar[r]_{f\ci}
& \Hom\bigl(\cO_X(-n),G\bigr) \ar[r]_(0.7){f'\ci} & \\
& \Hom\bigl(\cO_X(-n),F\bigr) \ar[r] & \Ext^1
\bigl(\cO_X(-n),E\bigr) \ar[r] &\cdots. }
\end{gathered}
\label{dt13eq12}
\e
But as $P_E(t)\equiv\frac{k}{d!}\,p(t),$ for $k\le N$, by choice of
$n$ in Definition \ref{dt13def1} we have
$\Ext^1\bigl(\cO_X(-n),E\bigr)=0$, so `$f'\ci$' in \eq{dt13eq12} is
surjective, and there exists at least one
$u\in\Hom\bigl(\cO_X(-n),G\bigr)$ with $t=f'\ci u$. If $u,\ti u$ are
possible choices for $u$ then $f'\ci(u-\ti u)=0$, so $u-\ti u$ lies
in the kernel of `$f'\ci$' in \eq{dt13eq12}, which is the image of
`$f\ci$' by exactness, and is isomorphic to
$\Hom\bigl(\cO_X(-n),E\bigr)$.

Na\"\i vely this appears to show that $\Ext^1_{\smash{\B_p}}
\bigl((F,\C,t),(E,0,0)\bigr)$ is the direct sum of $\Ext^1(F,E)$,
which represents the freedom to choose $G,f,f'$ in \eq{dt13eq11} up
to isomorphism, and $\Hom\bigl(\cO_X(-n),E\bigr)$, which
parametrizes the additional freedom to choose $u$ in \eq{dt13eq11}.
However, this is not quite true. Instead,
$\Ext^1_{\B_p}\bigl((F,\C,t),\ab(E,0,0)\bigr)$ parametrizes
isomorphism classes of diagrams \eq{dt13eq11}, up to isomorphisms
which are the identity on the second and fourth columns. Two
different choices $u,u'$ for $u$ in \eq{dt13eq11} might still be
isomorphic in this sense, through an isomorphism $g$ in the
following commutative diagram:
\e
\begin{gathered}
\xymatrix@R=10pt@C=5pt{&&&& \C\ot \cO_X(-n)\ar[dl]_(0.6)u
\ar[ddr]^(0.4){u'} \ar[rrr]_{\id_{\C\ot \cO_X(-n)}} &&& \C\ot
\cO_X(-n) \ar[dl]^(0.4)t \ar[ddr]^t \\
E \ar[drrr]_{\id_E} \ar[rrr]^f &&& G \ar[drr]_(0.4)g
\ar[rrr]_(0.85){f'}|(0.62)\hole &&& F
\ar[drr]_(0.3){\id_F} \\
&&& E \ar[rr]^(0.2)f && G \ar[rrr]^(0.2){f'} &&& F. }
\end{gathered}
\label{dt13eq13}
\e

Reasoning in the abelian category $\coh(X)$, as $f'\ci g=\id_F\ci
f'$ we have $f'\ci(g-\id_G)=0$, so $g-\id_G$ factorizes through the
kernel $f$ of $f'$, that is, $g-\id_G=f\ci h$, where $h:G\ra E$.
Also $g\ci f=f\ci\id_E=f$, so $(g-\id_G)\ci f=0$, and $f\ci h\ci
f=0$. As $f$ is injective this gives $h\ci f=0$. So $h$ factorizes
via the cokernel $f'$ of $f$, and $h=k\ci f'$ for $k:F\ra E$.
Therefore in \eq{dt13eq13} we may write $g=\id_G+f\ci k\ci f'$ for
$k\in\Hom(F,E)$. Hence, for any given choice $u$ in \eq{dt13eq13},
the equivalent choices $u'$ are of the form $u'=u+(f\ci k\ci f')\ci
u=u+f\ci k\ci t$. Thus we must quotient by the vector space of
morphisms $f\ci k\ci t$, for $k\in\Hom(F,E)$. As $f$ is injective,
this is isomorphic to the vector space of morphisms $k\ci t$ in
$\Hom\bigl(\cO_X(-n),E\bigr)$. This proves that there is an exact
sequence
\e
\begin{gathered}
0\ra \Coker\bigl( \Hom(F,E)\,{\buildrel \ci t
\over\longra}\, \Hom(\cO_X(-n),E)\bigr)\\
\longra
\Ext^1_{\B_p}\bigl((F,\C,t),(E,0,0)\bigr)\longra\Ext^1(F,E)\ra 0.
\end{gathered}
\label{dt13eq14}
\e

Now taking dimensions in equations \eq{dt13eq7}, \eq{dt13eq9},
\eq{dt13eq10} and \eq{dt13eq14}, and noting in \eq{dt13eq10} and
\eq{dt13eq14} that if $F:U\ra V$ is a linear map of
finite-dimensional vector spaces then $\dim\Ker F-\dim\Coker F=\dim
U-\dim V$, we see that
\begin{align*}
&\bigl(\dim\Hom_{\B_p}\bigl((E,0,0),(F,\C,t)\bigr)
-\dim\Ext^1_{\B_p}\bigl((E,0,0),(F,\C,t)\bigr)\bigr)-\\
&\bigl(\dim\Hom_{\B_p}\bigl((F,\C,t),(E,0,0)\bigr)
-\dim\Ext^1_{\B_p}((F,\C,t),(E,0,0)\bigr)\bigr)\\
&\;\> =\dim\Hom(E,F)-\dim\Ext^1(E,F)-\dim\Hom(F,E)+\dim\Ext^1(F,E)\\
&\qquad\qquad\qquad\qquad\qquad\qquad+\dim\Hom(\cO_X(-n),E)\\
&\;\> =\bar\chi([E],[F])+\bar\chi\bigl([\cO_X(-n)],[E]\bigr)
=\bar\chi{}^{\smash{\B_p}}\bigl([(E,0,0)],[(F,\C,t)]\bigr),
\end{align*}
using equations \eq{dt3eq14}, \eq{dt13eq2} which holds as
$P_E(t)\equiv\frac{k}{d!}\,p(t)$ for $k\le N$, and \eq{dt13eq5}.
This completes the proof of Proposition~\ref{dt13prop1}.
\end{proof}

\subsection{Three weak stability conditions on $\B_p$}
\label{dt132}\index{stability condition!weak}

\begin{dfn} It is easy to see that the positive cone $C(\B_p)$ of
$\B_p$ is
\begin{equation*}
C(\B_p)=\bigl\{(\be,d):\text{$\be\in C(\A)$ and $d\ge 0$ or $\be=0$
and $d>0$}\bigr\}.
\end{equation*}
Define weak stability conditions $(\dot\tau,\dot T,\le),(\ti\tau,\ti
T,\le),(\hat\tau,\hat T,\le)$ on $\B_p$~by:
\begin{itemize}
\setlength{\itemsep}{0pt}
\setlength{\parsep}{0pt}
\item $\dot T=\{-1,0\}$ with the natural order $-1<0$, and
$\dot\tau(\be,d)=0$ if $d=0$, and $\dot\tau(\be,d)=-1$ if $d>0$; and
\item $\ti T=\{0,1\}$ with the natural order $0<1$, and
$\ti\tau(\be,d)=0$ if $d=0$, and $\ti\tau(\be,d)=1$ if~$d>0$;
\item $\hat T=\{0\}$, and $\hat\tau(\be,d)=0$ for all~$(\be,d)$.
\end{itemize}

Since $\B_p$ is artinian by Lemma \ref{dt13lem1}, it is
$\dot\tau$-artinian, and as $\fM_\rss^{(\be,d)}(\dot\tau)$ is a
substack of $\fM_{\B_p}^{\smash{(\be,d)}}$ which is of finite type
by Lemma \ref{dt13lem1}, $\fM_\rss^{(\be,d)}(\dot\tau)$ is of finite
type for all $(\be,d)\in C(\B_p)$. Therefore $(\dot\tau,\dot T,\le)$
is {\it permissible}\index{stability condition!permissible} by
Definition \ref{dt3def4}, and similarly so are $(\ti\tau,\ti
T,\le),(\hat\tau,\hat T,\le)$. Note too that $(\hat\tau,\hat T,\le)$
{\it dominates\/}\index{stability condition!$(\tilde\tau,\tilde
T,\leqslant)$ dominates $(\tau,T,\leqslant)$} $(\dot\tau,\dot
T,\le), (\ti\tau,\ti T,\le)$, in the sense of
Definition~\ref{dt3def6}.
\label{dt13def3}
\end{dfn}

We can describe some of the moduli spaces
$\fM_\rss^{\smash{(\be,d)}}
(\dot\tau),\fM_\rss^{\smash{(\be,d)}}(\ti\tau)$.

\begin{prop}{\bf(a)} For all\/ $\be\in C(\A_p)$ we have natural stack
isomorphisms $\fM_\rss^{\smash{(\be,0)}}(\dot\tau)\cong
\fM_\rss^\be(\tau)$ identifying\/ $(E,0,0)$ with\/ $E,$ where
$\fM_\rss^\be(\tau)$ is as in\/ {\rm\S\ref{dt32}}. Also
$\fM_\rss^{\smash{(0,1)}}(\dot\tau)\!\cong\![\Spec\C/\bG_m]$ is the
point\/ $(0,\C,0),$ and\/
$\fM_\rss^{\smash{(\be,1)}}(\dot\tau)\!=\!\es$ for\/~$\be\!\ne\! 0$.

\smallskip

\noindent{\bf(b)} Let\/ $\al,n$ be as in Definition
{\rm\ref{dt13def1},} $\M_\stp^{\al,n}(\tau')$ the moduli scheme of
stable pairs $s:\cO_X(-n)\ra E$ from\/ {\rm\S\ref{dt12},} and\/
$\fM_\rss^{(\al,1)}(\ti\tau)$ the moduli stack of\/
$\ti\tau$-semistable objects in class $(\al,1)$ in $\B_p$.
Then\/~$\fM_\rss^{(\al,1)}(\ti\tau)\cong
\M_\stp^{\al,n}(\tau')\times[\Spec\C/\bG_m]$.
\label{dt13prop2}
\end{prop}

\begin{proof} For (a), all objects $(E,0,0)$ in class $(\be,0)$ are
$\dot\tau$-semistable, so $\fM_\rss^{\smash{(\be,0)}}(\dot\tau)\!\ab
=\fM_{\B_p}^{\smash{(\be,0)}}\cong\fM_{\A_p}^{\smash{\be}}\cong
\fM_\rss^{\smash{\be}}(\tau)$. The unique object in class $(0,1)$ in
$\B_p$ up to isomorphism is $(0,\C,0)$, and it has no nontrivial
subobjects, so it is $\dot\tau$-semistable. The automorphism group
of $(0,\C,0)$ in $\B_p$ is $\bG_m$. Therefore
$\fM_\rss^{\smash{(0,1)}}(\dot\tau)\cong [\Spec\C/\bG_m]$ is the
point $(0,\C,0)$. Suppose $(E,V,s)$ lies in class $(\be,1)$ in
$\B_p$ for $\be\ne 0$ in $C(\A_p)$. Consider the short exact
sequence in~$\B_p$:
\e
\begin{gathered}
\xymatrix@R=10pt@C=20pt{ 0 \ar[r] & 0 \ar[r] \ar[d] & V\ot\cO_X(-n)
\ar[d]^s \ar[r]_{\id} & V\ot\cO_X(-n) \ar[d]
\ar[r] & 0\phantom{,} \\
0 \ar[r] & E \ar[r] & E \ar[r] & 0 \ar[r] & 0,}
\end{gathered}
\label{dt13eq15}
\e
that is, $0\ra (E,0,0)\ra (E,V,s)\ra (0,V,0)\ra 0$. We have
$[(E,0,0)]=(\be,0)$ and $[(0,V,0)]=(0,1)$ in $K(\B_p)$, and
$\dot\tau(\be,0)=0>-1=\dot\tau(0,1)$, so \eq{dt13eq15}
$\dot\tau$-destabilizes $(E,V,s)$. Thus any object $(E,V,s)$ in
class $(\be,1)$ in $\B_p$ is $\dot\tau$-unstable, and
$\fM_\rss^{\smash{(\be,1)}}(\dot\tau)=\es$, proving~(a).

For (b), points of $\M_\stp^{\al,n}(\tau')$ are morphisms
$s:\cO_X(-n)\ra E$ with $[E]=\al$, and points of
$\fM_\rss^{\smash{(\al,1)}}(\ti\tau)$ are triples $(E,V,s)$ with
$[E]=\al$, $\dim V=1$ and $s:V\ot\cO_X(-n)\ra E$ a morphism. Define
a 1-morphism $\pi_1:\M_\stp^{\al,n}(\tau')\ra
\fM_\rss^{\smash{(\al,1)}} (\ti\tau)$ by $\pi_1:\bigl(s:\cO_X(-n)\ra
E\bigr)\longmapsto(E,\C,s)$. It is straightforward to check that
$s:\cO_X(-n)\ra E$ is a $\tau'$-stable pair if and only if
$(E,\C,s)$ is $\ti\tau$-semistable in $\B_p$.

Define another 1-morphism $\pi_2:\fM_\rss^{\smash{(\al,1)}}(\ti\tau)
\ra\M_\stp^{\al,n}(\tau')$ by $\pi_2:(E,V,s)\mapsto\bigl(s(v):
\cO_X(-n)\ra E\bigr)$, for some choice of $0\ne v\in V$. If $v,v'$
are possible choices then $v'=\la v$ for some $\la\in\bG_m$, since
$\dim V=1$. The isomorphism $\la\id_E:E\ra E$ is an isomorphism
between the stable pairs $s(v):\cO_X(-n)\ra E$ and
$s(v'):\cO_X(-n)\ra E$, so they have the same isomorphism class, and
define the same point in $\M_\stp^{\al,n}(\tau')$. Thus $\pi_2$ is
well-defined.

On $\C$-points, $\pi_1,\pi_2$ define inverse maps. The scheme
$\M_\stp^{\al,n}(\tau')$ parametrizes isomorphism classes of objects
parametrized by $\smash{\fM^{(\al, 1)}_{ss}(\ti\tau)}$. Therefore,
by \cite[Rem.~3.19]{LaMo}, $\pi_2:\fM^{(\al,1)}_{ss}(\ti\tau)
\ra\M_\stp^{\al,n}(\tau')$ is a gerbe, which has fibre
$[\Spec\C/\bG_m]$. Also $\pi_1$ is a trivializing section of
$\pi_2$, so by \cite[Lem.~3.21]{LaMo}, $\fM^{(\al,1)}_{ss}(\ti\tau)$
is a trivial $\bG_m$-gerbe over $\M_\stp^{\al,n}(\tau')$, that
is,~$\fM_\rss^{\smash{(\al,1)}}(\ti\tau)\cong
\M_\stp^{\al,n}(\tau')\times[\Spec\C/\bG_m]$.
\end{proof}

\subsection{Stack function identities in $\SFa(\fM_{\B_p})$}
\label{dt133}

As in \S\ref{dt31} we have a Ringel--Hall algebra\index{Ringel--Hall
algebra} $\SFa(\fM_{\B_p})$ with multiplication $*$, and a Lie
subalgebra $\SFai(\fM_{\B_p})$. As in \S\ref{dt32}, since
$(\ti\tau,\ti T,\le)$ and $(\dot\tau,\dot T,\le)$ are permissible we
have elements $\bdss^{\smash{(\be,d)}}
(\ti\tau),\bdss^{\smash{(\be,d)}}(\dot\tau)$ in $\SFa(\fM_{\B_p})$
for $(\be,d)\in C(\B_p)$, and we define
$\bep^{\smash{(\be,d)}}(\ti\tau),\bep^{\smash{(\be,d)}}(\dot\tau)$
by \eq{dt3eq4}, which lie in $\SFai(\fM_{\B_p})$ by Theorem
\ref{dt3thm1}. Applying Theorem \ref{dt3thm2} with dominating
permissible stability condition $(\hat\tau,\hat T,\le)$ yields:

\begin{prop} For all\/ $(\be,d)$ in $C(\B_p)$ we have the identity
in $\SFa(\B_p):$
\e
\begin{gathered}
\bep^{(\be,d)}(\ti\tau)= \!\!\!\!\!\!\!\!\!\!\!\!\!
\sum_{\begin{subarray}{l}n\ge 1,\;(\be_1,d_1),\ldots,(\be_n,d_n)\in
C(\B_p):\\ (\be_1,d_1)+\cdots+(\be_n,d_n)=(\be,d)\end{subarray}}
\!\!\!\!\!\!\!\!\!\!\!\!\!\!\!\!\!\!\!\!\!\!\!
\begin{aligned}[t]
U\bigl((\be_1,d_1),&\ldots,(\be_n,d_n);\dot\tau,\ti\tau\bigr)\cdot\\
&\bep^{(\be_1,d_1)}(\dot\tau)*\cdots* \bep^{(\be_n,d_n)}(\dot\tau).
\end{aligned}
\end{gathered}
\label{dt13eq16}
\e
There are only finitely many nonzero terms in\/~\eq{dt13eq16}.
\label{dt13prop3}
\end{prop}

We now take $(\be,d)=(\al,1)$ in \eq{dt13eq16}, where $\al$ is as
fixed in Definition \ref{dt13def1}. Then each term has
$d_1+\cdots+d_n=1$ with $d_i\ge 0$, so we have $d_k=1$ for some
$k=1,\ldots,n$ and $d_i=0$ for $i\ne k$. But
$\bep^{(\be_k,1)}(\dot\tau)$ is supported on
$\fM_\rss^{\smash{(\be_k,1)}}(\dot\tau)$ which is empty for
$\be_k\ne 0$ by Proposition \ref{dt13prop2}(a). Thus the only
nonzero terms in \eq{dt13eq16} have $(\be_i,d_i)=(\be_i,0)$ for
$i\ne k$ and $\be_i\in C(\A_p)$ and $(\be_k,d_k)=(0,1)$. Changing
notation to $\al_i=\be_i$ for $i<k$ and $\al_i=\be_{i+1}$ for $i\ge
k$ gives:
\e
\begin{split}
&\bep^{(\al,1)}(\ti\tau)=\\
&\sum_{\begin{subarray}{l}1\le k\le n,\\
\al_1,\ldots,\al_{n-1}\in C(\A_p):\\
\al_1+\cdots+\al_{n-1}=\al\end{subarray}
\!\!\!\!\!\!\!\!\!\!\!\!\!\!\!\!\!\!\!\!\!\!\!\!\!\!\!\!\! }
\begin{aligned}[t]
U\bigl((\al_1,0),\ldots,(\al_{k-1},0),(0,1),(\al_k,0),\ldots,
(\al_{n-1},0);\dot\tau,\ti\tau\bigr)&\cdot\\
\bep^{(\al_1,0)}(\dot\tau)*\cdots*\bep^{(\al_{k-1},0)}(\dot\tau)*
\bep^{(0,1)}(\dot\tau)&\\
*\,\bep^{(\al_k,0)}(\dot\tau)*\cdots*\bep^{(\al_{n-1},0)}(\dot\tau)&.
\end{aligned}
\end{split}
\label{dt13eq17}
\e

\begin{prop} In equation \eq{dt13eq17} we have
\e
\begin{split}
U\bigl((\al_1,0),\ldots,(\al_{k-1},0),&(0,1),(\al_k,0),\ldots,
(\al_{n-1},0);\dot\tau,\ti\tau\bigr)\\
&=\frac{(-1)^{n-k}}{(k-1)!(n-k)!}.
\end{split}
\label{dt13eq18}
\e
\label{dt13prop4}
\end{prop}

\begin{proof} The coefficient $U(\cdots;\dot\tau,\ti\tau)$ is defined
in equation \eq{dt3eq8}. Consider some choices $l,m$, $a_i,b_i,
\be_i,\ga_i$ in this sum. There are two conditions in \eq{dt3eq8}.
The first, that $\dot\tau(\be_i)=\dot\tau(\al_j)$, $i=1,\ldots,m$,
$a_{i-1}<j\le a_i$, holds if and only if we have $a_{p-1}=k-1$ and
$a_p=k$ for some $p=1,\ldots,m$. The second, that $\ti\tau(\ga_i)
=\ti\tau((\al,1))$, $i=1,\ldots,l$, is equivalent to $l=1$, since if
$l>1$ then one $\ga_i$ is of the form $(\be,1)$, with
$\ti\tau(\ga_i)=1$ and the other $\ga_j$ are of the form $(\be,0)$,
with $\ti\tau(\ga_j)=0$. Thus we may rewrite \eq{dt3eq8} as
\ea
&U\bigl((\al_1,0),\ldots,(\al_{k-1},0),(0,1),(\al_k,0),
\ldots,(\al_{n-1},0);\dot\tau,\ti\tau\bigr)=
\label{dt13eq19}\\
&\sum_{\begin{subarray}{l} \phantom{wiggle}\\
1\le p\le m\le n,\;\> 0=a_0<a_1<\cdots<a_{p-1}=k-1, \;\>
k=a_p<a_{p+1}<\cdots<a_m=n.\\
\text{Define $\be_1,\ldots,\be_m\in C(\B_p)$ by
$\be_i=(\al_{a_{i-1}+1}+\cdots+\al_{a_i},0)$, $i<p$,}\\
\text{$\be_p=(0,1)$, $\be_i=(\al_{a_{i-1}}+\cdots+\al_{a_i-1},0)$,
$i>p$.}\end{subarray}
\!\!\!\!\!\!\!\!\!\!\!\!\!\!\!\!\!\!\!\!\!\!\!\!\!\!\!\!\!\!\!\!\!
\!\!\!\!\!\!\!\!\!\!\!\!\!\!\!\!\!\!\!\!\!\!\!\!\!\!\!\!\!\!\!\!\!
\!\!\!\!\!\!\!\!\!\!\!\!\!\!\!\!\!\!\!\!\!\!\!\!\!\!\!\!\!\!}
S(\be_1,\be_2,\ldots,\be_m;\dot\tau,\ti\tau)
\cdot\prod_{i=1}^m\frac{1}{(a_i-a_{i-1})!}\,. \nonumber
\ea

In \eq{dt13eq19} we have $\be_p=(0,1)$ and $\be_i=(\al'_i,0)$ for
$\al'_i\in C(\A_p)$, $i\ne p$. Using Definition \ref{dt13def3}, in
Definition \ref{dt3def6} we see that $i=1,\ldots,m-1$ satisfies
neither (a) nor (b) if $i<p-1$, satisfies (b) when $i=p-1$, and
satisfies (a) for $i\ge p$. Therefore
\e
S(\be_1,\be_2,\ldots,\be_m;\dot\tau,\ti\tau)=\begin{cases}
(-1)^{m-1}, & p=1, \\ (-1)^{m-2}, & p=2, \\ 0, & p>2. \end{cases}
\label{dt13eq20}
\e
Since $0=a_0<\cdots<a_{p-1}=k-1$, we see that $p=1$ if $k=1$, and
$p>1$ if $k>1$. So we divide into two cases $k=1$ in \eq{dt13eq21}
and $k>1$ in \eq{dt13eq22}, and rewrite \eq{dt13eq19} using
\eq{dt13eq20} in each case:
\ea
&U\bigl((0,1),(\al_1,0),
\ldots,(\al_{n-1},0);\dot\tau,\ti\tau\bigr)=
\sum_{\begin{subarray}{l} \phantom{wiggle}\\
1\le m\le n,\;\> 1=a_1<a_2<\cdots<a_m=n\\
\end{subarray}
\!\!\!\!\!\!\!\!\!\!\!\!\!\!\!\!\!\!\!\!\!\!\!\!\!\!\!\!\!\!\!\!\!
\!\!\!\!\!\!\!\!\!\!\!\!\!\!\!\!\!\!\!\!\!\!\!\!\!\!\!\!\!\!\!\!\! }
(-1)^{m-1}\cdot\prod_{i=2}^m\frac{1}{(a_i-a_{i-1})!}\,,
\label{dt13eq21}\\
&U\bigl((\al_1,0),\ldots,(\al_{k-1},0),(0,1),(\al_k,0),
\ldots,(\al_{n-1},0);\dot\tau,\ti\tau\bigr)=
\label{dt13eq22}\\
&\qquad\qquad\frac{1}{(k-1)!}\cdot\sum_{\begin{subarray}{l}
\phantom{wiggle}\\
2\le m\le n,\;\> k=a_2<a_3<\cdots<a_m=n\end{subarray}
\!\!\!\!\!\!\!\!\!\!\!\!\!\!\!\!\!\!\!\!\!\!\!\!\!\!\!\!\!\!\!\!\!
\!\!\!\!\!\!\!\!\!\!}
(-1)^{m-2}\cdot\prod_{i=3}^m\frac{1}{(a_i-a_{i-1})!}\,. \nonumber
\ea
Here the factor $1/(k-1)!$ in \eq{dt13eq22} is $1/(a_1-a_0)!$ in
\eq{dt13eq19}, since $a_0=0$, $a_1=k-1$, and $a_2=k$. We evaluate a
rewritten version of the sums in \eq{dt13eq21} and~\eq{dt13eq22}:

\begin{lem} For all\/ $l\ge 1$ we have
\e
\sum_{\begin{subarray}{l} \phantom{wiggle}\\
1\le m\le l,\;\> 0=a_0<a_1<\cdots<a_m=l.\\
\end{subarray}
\!\!\!\!\!\!\!\!\!\!\!\!\!\!\!\!\!\!\!\!\!\!\!\!\!\!\!\!\!\!\!\!\!
\!\!\!\!\!\!\!\!\!\!\!\!\!\!\!\!\!\!\!\!\!\!\!\!\!\!\!\!\!\!\!\!\! }
(-1)^m\cdot\prod_{i=1}^m\frac{1}{(a_i-a_{i-1})!}=\frac{(-1)^l}{l!}\,.
\label{dt13eq23}
\e
\label{dt13lem2}
\end{lem}

\begin{proof} Write $T_l$ for the l.h.s.\ of \eq{dt13eq23}.
Then in formal power series we have:
\e
\begin{aligned}
&\sum_{l=1}^\iy T_lt^l = \sum_{l=1}^\infty
\sum_{\begin{subarray}{l} \phantom{wiggle}\\
1\le m\le l,\;\> 0=a_0<a_1<\cdots<a_m=l.\\
\end{subarray}
\!\!\!\!\!\!\!\!\!\!\!\!\!\!\!\!\!\!\!\!\!\!\!\!\!\!\!\!\!\!\!\!\!
\!\!\!\!\!\!\!\!\!\!\!\!\!\!\!\!\!\!\!\!\!\!\!\!\!\!\!\!\!\!\!\!\! }
(-1)^m\cdot\prod_{i=1}^m\frac{t^{a_i-a_{i-1}}}{(a_i-a_{i-1})!}
=\sum_{l=1}^\infty
\sum_{\begin{subarray}{l} \phantom{wiggle}\\
1\le m\le l,\;\> b_1,\ldots, b_m\ge 1, \\
b_1+\cdots+b_m=l\end{subarray}
\!\!\!\!\!\!\!\!\!\!\!\!\!\!\!\!\!\!\!\!\!\!\!\!\!\!\!\!\!\!\!\!\!
\!\!\!\!\!\!\!\! }
(-1)^m\cdot\prod_{i=1}^m\frac{t^{b_i}}{(b_i)!}\\
&\;=\sum_{m=1}^\iy (-1)^m
\raisebox{-4pt}{\begin{Large}$\displaystyle\Bigl[$\end{Large}}
\sum_{j=1}^{\iy}\frac{t^j}{j!}
\raisebox{-4pt}{\begin{Large}$\displaystyle\Bigr]$\end{Large}}^m
=\sum_{m=1}^\iy(-1)^m(e^t-1)^m=\frac{-(e^t-1)}{1+(e^t-1)}=e^{-t}-1,
\end{aligned}
\!\!\!\!\!\!\!\!\!\!\!\!\!\!\!\!\!
\label{dt13eq24}
\e
where in the second step we set $b_i=a_i-a_{i-1}$, and in the third
we regard $l$ as defined by $b_1+\cdots+b_m=l$ and drop the sum over
$l$, and then replace the sum over $b_1,\ldots,b_m$ by an $m^{\rm
th}$ power of a sum over $j$. Equating coefficients of $t^l$ in
\eq{dt13eq24} gives~\eq{dt13eq23}.
\end{proof}

Now the r.h.s.\ of \eq{dt13eq21} agrees with the l.h.s.\ of
\eq{dt13eq23} with $l=n-1$, replacing $n,m,a_1,\ldots,a_m$ in
\eq{dt13eq21} by $l+1,m+1,a_0,\ldots,a_m$ respectively. Thus
\eq{dt13eq21} and Lemma \ref{dt13lem2} prove the case $k=1$ of
\eq{dt13eq18}. Similarly, apart from the factor $1/(k-1)!$, the
r.h.s.\ of \eq{dt13eq22} agrees with the l.h.s.\ of \eq{dt13eq23}
with $l=n-k$, replacing $n,m,a_2,\ldots,a_m$ in \eq{dt13eq21} by
$l+k,m+2,a_0,\ldots,a_m$ respectively. This gives the case $k>1$ of
\eq{dt13eq18}, and completes the proof of Proposition
\ref{dt13prop4}.
\end{proof}

Substituting \eq{dt13eq18} into \eq{dt13eq17} and replacing $n$ by
$l+1$ and $k$ by $k-1$ gives
\e
\begin{gathered}
\bep^{(\al,1)}(\ti\tau)=
\sum_{\begin{subarray}{l}\phantom{wiggle}\\
0\le k\le l,\;\> \al_1,\ldots,\al_l\in C(\A_p):\\
\al_1+\cdots+\al_l=\al\end{subarray}
\!\!\!\!\!\!\!\!\!\!\!\!\!\!\!\!\!\!\!\!\!\!\!\!\!\!\!\!\!
\!\!\!\!\!\!\!\!\!\!\!\!\!\!\!\!\!\! }
\begin{aligned}[t]
\frac{(-1)^{l-k}}{k!(l-k)!} \cdot
\bep^{(\al_1,0)}(\dot\tau)*\cdots*\bep^{(\al_k,0)}(\dot\tau)*
\bep^{(0,1)}(\dot\tau)&\\
*\,\bep^{(\al_{k+1},0)}(\dot\tau)*\cdots*\bep^{(\al_l,0)}(\dot\tau)&.
\end{aligned}
\end{gathered}
\label{dt13eq25}
\e
As in Theorem \ref{dt3thm3}, by \cite[Th.~5.4]{Joyc6} we can rewrite
the wall crossing formula\index{wall-crossing formula} \eq{dt13eq25} in
terms of the Lie bracket $[\,,\,]$ on $\SFai(\fM_{\B_p})$, rather
than the Ringel--Hall\index{Ringel--Hall algebra} multiplication $*$ on
$\SFa(\fM_{\B_p})$. In this case, we can do it explicitly.

\begin{prop} In the situation above we have
\e
\begin{split}
&\bep^{(\al,1)}(\ti\tau)=\\[-3pt]
&\sum_{l\ge 1,\;\> \al_1,\ldots,\al_l\in C(\A_p):
\;\>\al_1+\cdots+\al_l=\al
\!\!\!\!\!\!\!\!\!\!\!\!\!\!\!\!\!\!\!\!\!\!\!\!\!\!\!\!\!
\!\!\!\!\!\!\!\!\!\!\!\!\!\!\!\!\!\!\!\!\!\!\!\!\!\!\!\!\!
\!\!\!\!\!\!\!\!\!\!\!\!\!\!\!\!\!\!\!\!\!\!\!\!\!\!\! }
\frac{(-1)^l}{l!}[[\cdots[[\bep^{(0,1)}(\dot\tau),\bep^{(\al_1,0)}
(\dot\tau)],\bep^{(\al_2,0)}(\dot\tau)],\cdots],\bep^{(\al_l,0)}(\dot\tau)].
\end{split}
\label{dt13eq26}
\e
\label{dt13prop5}
\end{prop}

\begin{proof} The term $[[\cdots[\bep^{(0,1)}(\dot\tau),\bep^{(\al_1,0)}
(\dot\tau)],\cdots],\bep^{(\al_l,0)}(\dot\tau)]$ in \eq{dt13eq26}
has $l$ ne\-sted commutators $[\,,\,]$, and so consists of $2^l$
terms. For each of these $2^l$ terms, let $k$ be the number of the
$l$ commutators in which we reverse the order of multiplication.
Then the sign of this term is $(-1)^k$, and $k$ $\bep^{(\al_i,0)}
(\dot\tau)$'s appear before $\bep^{(0,1)}(\dot\tau)$ in the product.
There are $\binom{l}{k}$ such terms for fixed $k$. Thus we have
\e
\begin{split}
[[\cdots&[[\bep^{(0,1)}(\dot\tau),\bep^{(\al_1,0)}(\dot\tau)],
\bep^{(\al_2,0)}(\dot\tau)],\cdots],\bep^{(\al_l,0)}(\dot\tau)]=\\
&\sum_{k=0}^l \binom{l}{k}\begin{aligned}[t]&\text{ terms of the
form }(-1)^k\bep^{(\al_{i_1},0)}(\dot\tau)*\cdots*
\bep^{(\al_{i_k},0)}(\dot\tau)\\
&\qquad*\bep^{(0,1)}(\dot\tau)*
\bep^{(\al_{i_{k+1}},0)}(\dot\tau)*\cdots*\bep^{(\al_{i_l},0)}(\dot\tau),
\end{aligned}
\end{split}
\label{dt13eq27}
\e
where $\{i_1,\ldots,i_l\}$ is some permutation of~$\{1,\ldots,l\}$.

Let us now sum \eq{dt13eq27} over all permutations of
$\{1,\ldots,l\}$, acting by permuting $\al_1,\ldots,\al_l$. The
permutations $\{i_1,\ldots,i_l\}$ are then also summed over all
permutations of $\{1,\ldots,l\}$, giving
\e
\begin{split}
\sum_{\si\in S_l}[[\cdots&[[\bep^{(0,1)}(\dot\tau),\bep^{(\al_{\si(1)},
0)}(\dot\tau)],\bep^{(\al_{\si(2)},0)}(\dot\tau)],\cdots],
\bep^{(\al_{\si(l)},0)}(\dot\tau)]=\\[-5pt]
&\sum_{\si\in S_l}\sum_{k=0}^l \binom{l}{k}\begin{aligned}[h]
(-1)^k\bep^{(\al_{\si(1)},0)}(\dot\tau)*\cdots*
\bep^{(\al_{\si(k)},0)}(\dot\tau)*\bep^{(0,1)}(\dot\tau)&\\
*\,\bep^{(\al_{\si(k+1)},0)}(\dot\tau)*\cdots*\bep^{(\al_{\si(l)},0)}
(\dot\tau)&,
\end{aligned}
\end{split}
\label{dt13eq28}
\e
where $S_l$ is the symmetric group of permutations
$\si:\{1,\ldots,l\}\ra\{1,\ldots,l\}$.

We now have
\begin{align*}
&\sum_{l\ge 1,\;\> \al_1,\ldots,\al_l\in C(\A_p):
\;\>\al_1+\cdots+\al_l=\al
\!\!\!\!\!\!\!\!\!\!\!\!\!\!\!\!\!\!\!\!\!\!\!\!\!\!\!\!\!
\!\!\!\!\!\!\!\!\!\!\!\!\!\!\!\!\!\!\!\!\!\!\!\!\!\!\!\!\!
\!\!\!\!\!\!\!\!\!\!\!\!\!\!\!\!\!\!\!\!\!\!\!\!\!\!\! }
\frac{(-1)^l}{l!}[[\cdots[[\bep^{(0,1)}(\dot\tau),\bep^{(\al_1,0)}
(\dot\tau)],\bep^{(\al_2,0)}(\dot\tau)],\cdots],\bep^{(\al_l,0)}(\dot\tau)]=\\
&\sum_{\begin{subarray}{l} l\ge 1,\\ \al_1,\ldots,\al_l\in C(\A_p):\\
\al_1+\cdots+\al_l=\al\end{subarray} \!\!\!\!\!\!\!\!\!\!\!\!\!\! }
\frac{(-1)^l}{(l!)^2}\sum_{\si\in S_l}\begin{aligned}[t]
[[\cdots[[\bep^{(0,1)}(\dot\tau),\bep^{(\al_{\si(1)},0)}
(\dot\tau)],\bep^{(\al_{\si(2)},0)}(\dot\tau)],\cdots]&,\\
\bep^{(\al_{\si(l)},0)}(\dot\tau)]&=
\end{aligned}\\
&\sum_{\begin{subarray}{l} l\ge 1,\\ \al_1,\ldots,\al_l\in C(\A_p):\\
\al_1+\cdots+\al_l=\al\end{subarray} \!\!\!\!\!\!\!\!\!\!\!\!\!\! }
\frac{(-1)^l}{(l!)^2}
\sum_{\si\in S_l}\sum_{k=0}^l \!\binom{l}{k}\begin{aligned}[t]
(-1)^k\bep^{(\al_{\si(1)},0)}(\dot\tau)\!*\!\cdots\!*\!
\bep^{(\al_{\si(k)},0)}(\dot\tau)\!*\!\bep^{(0,1)}(\dot\tau)&\\
*\,\bep^{(\al_{\si(k+1)},0)}(\dot\tau)*\cdots*\bep^{(\al_{\si(l)},0)}
(\dot\tau)=\,&
\end{aligned}\\
&\sum_{\begin{subarray}{l} l\ge 1,\\ \al_1,\ldots,\al_l\in C(\A_p):\\
\al_1+\cdots+\al_l=\al\end{subarray} \!\!\!\!\!\!\!\!\!\!\!\!\!\! }
\frac{(-1)^l}{l!}
\sum_{k=0}^l \binom{l}{k}\begin{aligned}[t]
(-1)^k\bep^{(\al_1,0)}(\dot\tau)*\cdots*
\bep^{(\al_k,0)}(\dot\tau)*\bep^{(0,1)}(\dot\tau)&\\
*\,\bep^{(\al_{k+1},0)}(\dot\tau)*\cdots*\bep^{(\al_l,0)}(\dot\tau)&=
\end{aligned}\\
&\sum_{\begin{subarray}{l}\phantom{wiggle}\\
0\le k\le l,\;\> \al_1,\ldots,\al_l\in C(\A_p):\\
\al_1+\cdots+\al_l=\al\end{subarray}
\!\!\!\!\!\!\!\!\!\!\!\!\!\!\!\!\!\!\!\!\!\!\!\!\!\!\!\!\!
\!\!\!\!\!\!\!\!\!\!\!\!\!\!\!\!\!\! }
\begin{aligned}[t]
\frac{(-1)^{l-k}}{k!(l-k)!} \cdot
\bep^{(\al_1,0)}(\dot\tau)*\cdots*\bep^{(\al_k,0)}(\dot\tau)*
\bep^{(0,1)}(\dot\tau)&\\
*\,\bep^{(\al_{k+1},0)}(\dot\tau)*\cdots*\bep^{(\al_l,0)}(\dot\tau)&
=\bep^{(\al,1)}(\ti\tau),
\end{aligned}
\end{align*}
using the fact that the sums over $\al_1,\ldots,\al_l\in C(\A_p)$
with $\al_1+\cdots+\al_l=\al$ are symmetric in permutations of
$\{1,\ldots,l\}$ in the first and third steps, \eq{dt13eq28} in the
second, and \eq{dt13eq25} in the fifth. This proves
equation~\eq{dt13eq26}.
\end{proof}

\subsection{A Lie algebra morphism $\ti\Psi{}^{\B_p}:\SFai(\fM_{\B_p})
\ra\ti L(\B_p)$}
\label{dt134}

We now define a Lie algebra morphism $\ti\Psi{}^{\B_p}:\SFai
(\fM_{\B_p})\ra\ti L(\B_p)$, which is a version of $\ti\Psi:
\SFai(\fM)\ra\ti L(X)$ in \S\ref{dt53} for our auxiliary abelian
category $\B_p$. Since as in \S\ref{dt131} we do not know $\B_p$ is
3-Calabi--Yau,\index{abelian category!3-Calabi--Yau} and also as we will
see below we only have good control of the Behrend
function\index{Behrend function|(} $\nu_{\smash{\fM_{\B_p}}}$ on a
bounded part of $\fM_{\smash{\B_p}}$, we will choose the Lie algebra
$\ti L(\B_p)$ to be small, a finite-dimensional, nilpotent Lie
algebra, and define $\ti\Psi{}^{\B_p}$ to be supported on
$\smash{\fM_{\B_p}^{(\be,d)}}$ for only finitely many~$(\be,d)\in
K(\B_p)$.

\begin{dfn} Define $\cal S$ to be the subset of $(\be,d)$ in
$C(\B_p)\subset K(\B_p)$ such that $P_\be(t)=\frac{k}{d!}p(t)$ for
$k=0,\ldots,N$ and $d=0$ or 1. (These were the conditions on
numerical classes in Proposition \ref{dt13prop1}.) Then $\cal S$ is
a finite set, as \cite[Th.~3.37]{HuLe2} implies that
$\tau$-semistable sheaves $E$ on $X$ with Hilbert polynomials
$\frac{k}{d!}p(t)$ for $k=0,\ldots,N$ can realize only finite many
numerical classes $\be\in K(\A_p)\subset K(\coh(X))$. Define a Lie
algebra $\ti L(\B_p)$ to be the $\Q$-vector space with basis of
symbols $\ti\la^{(\be,d)}$ for $(\be,d)\in \cal S$, with Lie bracket
\e
\begin{split}
&[\ti \la^{(\be,d)},\ti \la^{(\ga,e)}]=\\
&\;\>\begin{cases} (-1)^{\bar\chi^{\B_p}((\be,d),(\ga,e))}
\bar\chi^{\B_p}\bigl((\be,d),(\ga,e)\bigr)\ti \la^{(\be+\ga,d+e)}, &
(\be+\ga,d+e)\in{\cal S}, \\
0, & \text{otherwise,}\end{cases}
\end{split}
\label{dt13eq29}
\e
as in \eq{dt5eq4}. As $\bar\chi^{\B_p}$ is antisymmetric, and ${\cal
S}\subset K(\B_p)$ has the property that if $\ep,\ze,\eta\in \cal S$
and $\ep+\ze+\eta\in \cal S$ then $\ep+\ze,\ep+\eta,\ze+\eta\in \cal
S$, equation \eq{dt13eq29} satisfies the Jacobi identity, and makes
$\ti L(\B_p)$ into a finite-dimensional, nilpotent Lie algebra over
$\Q$. Now define a $\Q$-linear map
$\ti\Psi{}^{\B_p}:\SFai(\fM_{\B_p})\ra \ti L(\B_p)$ exactly as for
$\ti\Psi:\SFai(\fM)\ra\ti L(X)$ in Definition~\ref{dt5def1}.
\label{dt13def4}
\end{dfn}

We shall show that $\ti\Psi{}^{\B_p}$ is a {\it Lie algebra
morphism}, by modifying the proof for $\ti\Psi$ in Theorem
\ref{dt5thm5}. The two key ingredients in the proof of Theorem
\ref{dt5thm5} were, firstly, equation \eq{dt3eq14} writing the Euler
form $\bar\chi$ of $\coh(X)$ in terms of $\dim\Hom$ and $\dim\Ext^1$
in $\coh(X)$, and secondly, the identities \eq{dt5eq2}--\eq{dt5eq3}
for the Behrend function $\nu_{\fM}$ in Theorem \ref{dt5thm4}.
Proposition \ref{dt13prop1} proves the analogue of \eq{dt3eq14} in
the bounded part of $\B_p$ we need it for. Here is an analogue of
Theorem~\ref{dt5thm4}.

\begin{prop}{\bf(a)} If\/ $(\be,0)\in\cal S$ then
$\pi:\smash{\fM_{\B_p}^{(\be,0)}\ra\fM^\be}$ mapping
$(E,0,0)\ab\mapsto E$ is a $1$-isomorphism, and the Behrend
functions
satisfy\/~$\nu_{\smash{\fM_{\B_{\smash{p}}}^{(\be,0)}}}\equiv
\pi^*(\nu_\fM^\be)$.

If\/ $(\be,1)\in\cal S$ then
$\pi:\smash{\fM_{\B_{\smash{p}}}^{(\be,1)}}\ra\fM^\be$ mapping
$(E,V,s)\mapsto E$ is smooth of relative dimension
$\bar\chi\bigl([\cO_X(-n)],\be\bigr)-1,$
and\/~$\nu_{\fM_{\B_p}^{(\be,1)}}\equiv
(-1)^{\bar\chi([\cO_X(-n)],\be)-1}\pi^*(\nu_\fM^\be)$.

\noindent{\bf(b)} An analogue of Theorem {\rm\ref{dt5thm4}} holds in
$\B_p,$ with\/ $E_1,E_2\in\coh(X)$ replaced by
$(E_1,V_1,s_1),(E_2,V_2,s_2)\in\B_p$ such that\/ $[(E_1\op
E_2,V_1\op V_2,s_1\op s_2)]\in\cal S,$ and\/ $\Ext^1$ replaced by
$\Ext^1_{\B_p},$ and\/ $\bar\chi$ replaced by\/~$\bar\chi{}^{\B_p}$.
\label{dt13prop6}
\end{prop}

\begin{proof} The first part of (a) is immediate. For the second,
note that if $(\be,1)\in\cal S$ and $(E,V,s)$ is a point in
$\fM_{\B_{\smash{p}}}^{(\be,1)}$ then $[E]=\be$ in $K(\A_p)$, we may
identify $V\cong\C$, and then $s:\cO_X(-n)\ra E$, that is, $s\in
H^0(E(n))$. But by choice of $E$ and of $n$ in Definition
\ref{dt13def1}, we have $H^i(E(n))=0$ for $i>0$, so $H^0(E(n))$ is a
vector space of fixed dimension
$P_\be(n)=\bar\chi([\cO_X(-n)],\be)$. Furthermore, $E\mapsto
H^0(E(n))$ is a {\it vector bundle\/}\index{vector bundle!on Artin
stack} (in the Artin stack sense) over the stack $\fM^\be$, with
fibre $H^0(E(n))\cong\C^{\bar\chi([\cO_X(-n)],\be)}$ over~$E$.

Now consider the fibre of $\pi:\fM_{\B_p}^{(\be,1)}\ra\fM^\be$ over
$E$. It is a set of pairs $(V,s)$ with $V\cong\C$ and $s:V\ra
H^0(E(n))$ linear, satisfying a stability condition. This stability
condition requires $s\ne 0$, and selects an open set of such $s$.
Dividing out by automorphisms of $V$ turns $H^0(E(n))\sm 0$ into the
projective space $\mathbb{P}\bigl(H^0(E(n))\bigr)$. Hence the fibre
of $\pi$ over $E$ is an open subset of the projective space
$\mathbb{P}\bigl(H^0(E(n))\bigr)$. Since $E\mapsto H^0(E(n))$ is a
vector bundle over $\fM^\be$, $E\mapsto
\mathbb{P}\bigl(H^0(E(n))\bigr)$ is a projective space bundle over
$\fM^\be$. Therefore $\fM_{\B_{\smash{p}}}^{(\be,1)}$ is an open
subset of a smooth fibration over $\fM^\be$ with fibre
$\CP^{\bar\chi([\cO_X(-n)],\be)-1}$. So $\pi$ is smooth of relative
dimension $\bar\chi\bigl([\cO_X(-n)],\be\bigr)-1$. The Behrend
function equation follows from Theorem \ref{dt4thm1}(ii) and
Corollary~\ref{dt4cor1}.

For (b), we can now follow the proof of Theorem \ref{dt5thm4}, using
facts from (a) above. In Theorem \ref{dt5thm3}, we proved that an
atlas for $\fM^\be$ near $E$ may be written locally in the complex
analytic topology as $\Crit(f)$ for holomorphic $f:U\ra\C$, where
$U$ is an open neighbourhood of $0$ in $\Ext^1(E,E)$, and $U,f$ are
invariant under the complexification $G^{\sst\C}$ of a maximal
compact subgroup $G$ of $\Aut(E)$. From the second part of (a), it
follows that an atlas for $\fM_{\B_{\smash{p}}}^{(\be,1)}$ near
$E,V,s$ may be written locally in the complex analytic topology as
$\Crit(f)\times W$, where $W$ is an open set in $H^0(E(n))$. But
$\Crit(f)\times W=\Crit(f\ci\pi_U)$, where $f\ci\pi_U:U\times
W\ra\C$ is a holomorphic function on a smooth complex manifold.

Therefore, just as we can write the moduli stack $\fM$ locally as
$\Crit(f)$, and so use differential-geometric reasoning with the
Milnor fibres\index{Milnor fibre} of $f$ to prove
\eq{dt5eq2}--\eq{dt5eq3} in Theorem \ref{dt5thm4}, so we can write
the moduli stacks $\fM_{\B_{\smash{p}}}^{(\be,d)}$ for
$(\be,d)\in{\cal S}$ locally as $\Crit(f\ci\pi_U)$, and the proof of
Theorem \ref{dt5thm4} extends to give~(b).
\end{proof}

We can now follow the proof of Theorem \ref{dt5thm5} using
Proposition \ref{dt13prop1} in place of \eq{dt3eq14} and Proposition
\ref{dt13prop6}(b) in place of Theorem \ref{dt5thm4} to prove:

\begin{prop} $\ti\Psi{}^{\B_p}:\SFai(\fM_{\B_p})\ra\ti L(\B_p)$ is
a Lie algebra morphism.
\label{dt13prop7}
\end{prop}

\subsection[Proof of Theorem $\text{\ref{dt5thm10}}$]{Proof of Theorem
\ref{dt5thm10}}
\label{dt135}

Finally we prove Theorem \ref{dt5thm10}. We will apply the Lie
algebra morphism $\ti\Psi{}^{\B_p}$ to the Lie algebra equation
\eq{dt13eq26}. Observe that the terms $(\al,1)$, $(1,0)$ and
$(\al_i,0)$ occurring in \eq{dt13eq26} all lie in $\cal S$. We will
prove that
\e
\begin{gathered}
\ti\Psi{}^{\B_p}\bigl(\bep^{(\al,1)}(\ti\tau)\bigr)=
-PI^{\al,n}(\tau')\ti \la^{(\al,1)},\qquad
\ti\Psi{}^{\B_p}\bigl(\bep^{(0,1)}(\dot\tau)\bigr)= -\ti \la^{(0,1)},\\
\text{and}\qquad \ti\Psi{}^{\B_p}\bigl(\bep^{(\al_i,0)}(\dot\tau)
\bigr)=-\bar{DT}{}^{\al_i}(\tau)\ti \la^{(\al_i,0)}.
\end{gathered}
\label{dt13eq31}
\e

For the first equation, there are no strictly $\ti\tau$-semistables
in $\fM_\rss^{(\al,1)}(\ti\tau)$, so $\bep^{(\al,1)}(\ti\tau)=
\bdss^{(\al,1)}(\ti\tau)$, and $\ti\Psi{}^{\B_p}\bigl(
\bep^{(\al,1)}(\ti\tau)\bigr)=\chi^{\rm na}\bigl(\fM_\rss^{(\al,1)}
(\ti\tau),\nu_{\fM_\rss^{(\al,1)}(\ti\tau)}\bigr)\ti \la^{(\al,1)}$
in the notation of Definition \ref{dt2def3}. But
$\fM_\rss^{\smash{(\al,1)}}(\ti\tau)\cong
\M_\stp^{\al,n}(\tau')\times[\Spec\C/\bG_m]$ by Proposition
\ref{dt13prop2}(b), so the projection
$\pi:\fM_\rss^{\smash{(\al,1)}} (\ti\tau)\ra\M_\stp^{\al,n}(\tau')$
is smooth of relative dimension $-1$, and
$\nu_{\fM_\rss^{\smash{(\al,1)}}
(\ti\tau)}=-\pi^*(\nu_{\M_\stp^{\al,n}(\tau')})$ by Theorem
\ref{dt4thm1}(ii) and Corollary \ref{dt4cor1}. Hence
\begin{equation*}
\chi^{\rm na}\bigl(\fM_\rss^{(\al,1)}
(\ti\tau),\nu_{\fM_\rss^{(\al,1)}(\ti\tau)}\bigr)=
-\chi\bigl(\M_\stp^{\al,n}(\tau'),\nu_{\M_\stp^{\al,n}
(\tau')}\bigr)=-PI^{\al,n}(\tau')
\end{equation*}
by \eq{dt5eq16}, proving the first equation of \eq{dt13eq31}. Now
$\fM_\rss^{\smash{(0,1)}}(\dot\tau)\cong[\Spec\C/\bG_m]$ by
Proposition \ref{dt13prop2}(a), so $\bep^{(0,1)}(\dot\tau)$ is just
the stack characteristic function of $[\Spec\C/\bG_m]$. But
$[\Spec\C/\bG_m]$ is a single point with Behrend function $-1$, so
the second equation follows. And the isomorphism
$\fM_\rss^{\smash{(\al_i,0)}}(\dot\tau)\cong\fM_\rss^{\al_i}
(\tau)\subset\fM$ identifies $\bep^{(\al_i,0)}(\dot\tau)$ with
$\bep^{\al_i}(\tau)$, so the third equation of \eq{dt13eq31} follows
from~\eq{dt5eq7}.

Hence, applying $\ti\Psi{}^{\B_p}$ (which is a Lie algebra morphism
by Proposition \ref{dt13prop7}) to \eq{dt13eq26} and substituting in
\eq{dt13eq31} gives an equation in the Lie algebra $\ti L(\B_p)$:
\e
\begin{split}
&-PI^{\al,n}(\tau')\ti \la^{(\al,1)}=\\[-3pt]
&\sum_{l\ge 1,\;\> \al_1,\ldots,\al_l\in C(\A_p):
\;\>\al_1+\cdots+\al_l=\al
\!\!\!\!\!\!\!\!\!\!\!\!\!\!\!\!\!\!\!\!\!\!\!\!\!\!\!\!\!
\!\!\!\!\!\!\!\!\!\!\!\!\!\!\!\!\!\!\!\!\!\!\!\!\!\!\!\!\!
\!\!\!\!\!\!\!\!\!\!\!\!\!\!\!\!\!\!\!\!\!\!\!\!\!\!\! }
\begin{aligned}[h]
\frac{(-1)^l}{l!}[[\cdots[[ -\ti \la^{(0,1)},
-\bar{DT}{}^{\al_1}(\tau)\ti \la^{(\al_1,0)}],-\bar{DT}{}^{\al_2}
(\tau)\ti \la^{(\al_2,0)}],\cdots]&,\\
-\bar{DT}{}^{\al_l}(\tau)\ti \la^{(\al_l,0)}]&.
\end{aligned}
\end{split}
\label{dt13eq32}
\e
Using the definitions \eq{dt13eq5} of $\bar\chi{}^{\B_p}$ and
\eq{dt13eq29} of the Lie bracket in $\ti L(\B_p)$, and noting that
the condition $\al_i\in C(\coh(X))$ with $\tau(\al_i)=\tau(\al)$ in
\eq{dt5eq17} corresponds to $\al_i\in C(\A_p)$ in \eq{dt13eq32}, we
see that \eq{dt13eq32} reduces to \eq{dt5eq17}. There are only
finitely many nonzero terms in each of these equations, as in
Proposition \ref{dt13prop3}. This completes the proof of
Theorem~\ref{dt5thm10}.\index{Behrend function|)}


\clearpage
\printnomenclature[1.3cm]
\clearpage
\addcontentsline{toc}{section}{Index}
\printindex
\end{document}